\documentclass[10pt,reqno]{amsart}
\usepackage{amsfonts,amsrefs,latexsym,amsmath, amssymb, mathrsfs, verbatim}
\usepackage{url,color}
\usepackage{pifont}
\usepackage{upgreek}
\usepackage{fancyhdr}
\usepackage{hyperref}
\usepackage{calligra}
\usepackage{marvosym}
\usepackage[percent]{overpic}
\usepackage{pict2e}
\usepackage[hypcap=false]{caption}
\usepackage{xcolor}
\usepackage{wasysym}
\usepackage{accents}
\usepackage{enumitem}
\setlist[enumerate]{label*=\arabic*.}

\topmargin - .23 in

\newtheorem{theorem}{Theorem}[section]
\newtheorem{proposition}{Proposition}[section]
\newtheorem{lemma}[proposition]{Lemma}
\newtheorem{corollary}[proposition]{Corollary}

\theoremstyle{definition}
\newtheorem{definition}{Definition}[section]
\newtheorem{remark}{Remark}[section]
\newtheorem*{nonumberremark}{Remark}

\DeclareMathAlphabet{\mathcalligra}{T1}{calligra}{m}{n}
\DeclareFontShape{T1}{calligra}{m}{n}{<->s*[2.2]callig15}{}

\newcommand{\Ent}{s}
\newcommand{\GradEnt}{S}

\newcommand{\LogDensity}{\uprho}

\newcommand{\vortrenormalized}{\Omega}
\newcommand{\Speed}{c}

\newcommand{\VortVort}{\mathcal{C}}
\newcommand{\DivGradEnt}{\mathcal{D}}

\newcommand{\dive}{\mbox{\upshape div}}
\newcommand{\curl}{\mbox{\upshape curl}}

\newcommand{\Transport}{\mathbf{B}}

\newcommand{\Tboot}{T_*}
\newcommand{\RescaledTboot}{T_{*;(\uplambda)}}

\newcommand{\RescaledFoliationparameter}{w_{*;(\uplambda)}}

\newcommand{\Sob}{N}

\newcommand{\gensmoothfunction}{\mathrm{f}}

\newcommand{\quadsmoothfunction}{\mathscr{Q}}
\newcommand{\linsmoothfunction}{\mathscr{L}}
\newcommand{\contfunction}{F}

\newcommand{\enmomem}{\mathbf{Q}}
\newcommand{\Ricfour}{\mathbf{Ric}}
\newcommand{\Riemfour}{\mathbf{Riem}}

\newcommand{\spherenormal}{N}

\newcommand{\nulllapse}{b}
\newcommand{\spheresecondfund}{\uptheta}

\newcommand{\Jen}[1]{^{(#1)} \mkern-3mu \mathbf{J}}
\newcommand{\Jenarg}[2]{{^{(#1)} \mkern-3mu \mathbf{J}^{#2}}}
\newcommand{\Jenwithlowerarg}[2]{{^{(#1)} \mkern-3mu \widetilde{\mathbf{J}}^{#2}}}

\newcommand{\deformarg}[3]{{^{(#1)} \mkern-1mu \pmb{\pi}_{#2 #3}}}

\newcommand{\gfour}{\mathbf{g}}

\newcommand{\gsphere}{g \mkern-8.5mu / }

\newcommand{\stgsphere}{e \mkern-9mu / }
\newcommand{\congsphere}{\widetilde{g} \mkern-8.5mu / }

\newcommand{\Chfour}{\pmb{\Gamma}}

\newcommand{\Flatdiv}{\mbox{\upshape div}\mkern 1mu}
\newcommand{\Flatcurl}{\mbox{\upshape curl}\mkern 1mu}

\newcommand{\tvol}{\varpi_g}
\newcommand{\tvolarg}[3]{\varpi_{g(#1,#2,#3)}}
\newcommand{\spherevol}{\varpi_{\gsphere}}

\newcommand{\spherevolarg}[3]{\varpi_{\gsphere(#1,#2,#3)}}
\newcommand{\flatspherevolarg}[1]{\varpi_{\stgsphere(#1)}}

\newcommand{\sphereproject}{{\Pi \mkern-12mu / } \, }
\newcommand{\Sigmatproject}{\underline{\Pi}}

\newcommand{\inhom}{\mathfrak{F}}
\newcommand{\remainder}{\mathfrak{R}}
\newcommand{\anotherremainder}{\mathfrak{R}}

\newcommand{\mytr}{{\mbox{\upshape tr}}}

\newcommand{\controlling}{Q}

\newcommand{\Lunit}{L}
\newcommand{\uLunit}{\underline{L}}

\newcommand{\Dfour}{\mathbf{D}}
\newcommand{\angprojD}{{\mathbf{D} \mkern-13mu / \,}}
\newcommand{\angprojDarg}[1]{{{\mathbf{D} \mkern-13mu / \,}_{#1}}}
\newcommand{\angD}{ {\nabla \mkern-14mu / \,} }
\newcommand{\angDarg}[1]{{\angD_{\mkern-3mu #1}}}

\newcommand{\angDsquaredarg}[2]{ {\angD_{\mkern-3mu #1 #2}^2} }
\newcommand{\angdiv}{\mbox{\upshape{div} $\mkern-17mu /$\,}}
\newcommand{\angcurl}{\mbox{\upshape{curl} $\mkern-17mu /$\,}}
\newcommand{\angLap}{ {\Delta \mkern-12mu / \, } }

\newcommand{\angupmu}{ { {\upmu \mkern-10mu /} \, } }

\newcommand{\rgeo}{\tilde{r}}

\newcommand{\aff}{A}

\newcommand{\Tranchar}{\upgamma}

\newcommand{\weight}{W}
\newcommand{\uweight}{\underline{W}}

\newcommand{\Lie}{\mathcal{L}}
\newcommand{\SigmatLie}{\underline{\mathcal{L}}}
\newcommand{\angLie}{ { \mathcal{L} \mkern-10mu / } }

\newcommand{\volrat}{\upsilon}

\newcommand{\leb}{Q}

\begin{document}
\title{Rough sound waves in $3D$ compressible Euler flow with vorticity}
\author[MMD,CL,GM,JS]{Marcelo M. Disconzi$^{\# *}$, Chenyun Luo$^{**}$, 
Giusy Mazzone$^{\dagger *** }$, Jared Speck$^{\bowtie ****}$}
	
\thanks{$^{\#}$MMD gratefully acknowledges support from NSF grant \# 1812826,
from NSF grant \# 2107701,
from a Sloan Research Fellowship provided by the Alfred P. Sloan foundation,
from a Discovery grant administered by Vanderbilt University, and 
from a Dean's Faculty Fellowship.
}

\thanks{$^{\dagger}$GM  gratefully acknowledges support of the Natural Sciences and Engineering Research Council of Canada (NSERC) 
through a Discovery Grant.
}

\thanks{$^{\bowtie}$JS gratefully acknowledges support from NSF grant \# 2054184,
from NSF CAREER grant \# 1914537,
and from a Sloan Research Fellowship provided by the Alfred P. Sloan foundation.
}

\thanks{$^{*}$Department of Mathematics, Vanderbilt University, Nashville, TN, USA.
\texttt{marcelo.disconzi@vanderbilt.edu}}

\thanks{$^{**}$Department of Mathematics, The Chinese University of Hong Kong, Shatin, NT, Hong Kong.
\texttt{cluo@math.cuhk.edu.hk}}

\thanks{$^{***}$Department of Mathematics  and Statistics, Queen's University, Kingston, ON, Canada.
\texttt{giusy.mazzone@queensu.ca}}

\thanks{$^{****}$Department of Mathematics, Vanderbilt University, Nashville, TN, USA
\texttt{jared.speck@vanderbilt.edu}}

\begin{abstract}
We prove a series of intimately related results tied to the regularity and geometry of solutions to
the $3D$ compressible Euler equations. The results concern ``general'' solutions, which can
have nontrivial vorticity and entropy. 
Our geo-analytic framework exploits and reveals additional virtues of a recent new formulation of the equations, which 
decomposed the flow into a geometric ``(sound) wave-part'' coupled to a ``transport-div-curl-part''
(transport-part for short),
with both parts exhibiting remarkable properties.
Our main result is that the time of existence
can be controlled in terms of the $H^{2^+}(\mathbb{R}^3)$-norm of the wave-part of the initial data
and various Sobolev and H\"{o}lder norms of the transport-part of the initial data,
the latter comprising the initial vorticity and entropy.
The wave-part regularity assumptions are optimal in the scale of Sobolev spaces: 
Lindblad \cite{hL1998} showed that shock singularities 
can instantly form if one only assumes a bound for the 
$H^2(\mathbb{R}^3)$-norm of the wave-part of the initial data.
Our proof relies on the assumption that the transport-part of the initial data is more regular than the wave-part,
and we show that the additional regularity is propagated by the flow, 
even though the transport-part of the flow is deeply coupled
to the rougher wave-part. To implement our approach,
we derive several results of independent interest:
\textbf{i)} sharp estimates for the acoustic geometry, 
which in particular capture how the vorticity and entropy affect the Ricci curvature of the acoustical metric
and therefore, via Raychaudhuri's equation, influence the evolution of the geometry of  acoustic null hypersurfaces, i.e., sound cones;
\textbf{ii)} Strichartz estimates for quasilinear sound waves coupled to vorticity and entropy;
and \textbf{iii)} Schauder estimates for the transport-div-curl-part.
Compared to previous works on low regularity, 
the main new features of the paper are that the quasilinear PDE
systems under study exhibit multiple speeds of propagation
and that elliptic estimates for various components of the fluid 
are needed, both to avoid loss of regularity and to gain space-time integrability.

\bigskip


\bigskip


\end{abstract}

\maketitle

\centerline{\today}

\tableofcontents

\newpage

\section{Introduction and overview of the main results}
\label{S:INTRO}
In this paper, we study the compressible Euler equations in three spatial dimensions:
\begin{subequations}{\label{E:EULERSYSTEMSTANDARD}}
\begin{align} \label{E:TRANSPORTDENSRENORMALIZEDRELATIVETORECTANGULAR}
	\Transport \varrho
	& = - \varrho\, \Flatdiv v,
	&&
		\\
	\Transport v^i 
	& = -\varrho^{-1} \updelta^{ia} \partial_a p,
	\label{E:TRANSPORTVELOCITYRELATIVETORECTANGULAR}
	&
	&
	(i=1,2,3),
		\\
	\Transport \Ent
	& = 0,
	\label{E:ENTROPYTRANSPORT}
\end{align}
\end{subequations}
where $\varrho:\mathbb{R}^{1+3} \rightarrow [0,\infty)$,
$v:\mathbb{R}^{1+3} \rightarrow \mathbb{R}^3$, 
and $\Ent:\mathbb{R}^{1+3} \rightarrow \mathbb{R}$
are the fluid's density, velocity, and entropy, respectively; 
$p$ is the fluid's pressure, which is a given smooth function of $\varrho$ and $\Ent$ known as the \emph{equation of state}
-- whose choice reflects one's assumptions about the properties of the fluid --;
\begin{align} \label{E:MATERIALVECTORVIELDRELATIVETORECTANGULAR}
	\Transport 
	& := 
	\partial_t 
	+  
	v^a \partial_a
\end{align} 
is the material derivative vectorfield;
$\mathbf{X} f := \mathbf{X}^{\alpha} \partial_{\alpha} f$
denotes the derivative of the scalar function $f$ in the direction of the vectorfield $\mathbf{X}$;
$\updelta^{ab}$ is the standard Kronecker delta; and 
$\Flatdiv v := \partial_a v^a$ is the standard (three-dimensional) Euclidean divergence of $v$.
Equations \eqref{E:EULERSYSTEMSTANDARD} are expressed relative to
Cartesian coordinates 
$\lbrace x^{\alpha} \rbrace_{\alpha=0,1,2,3}$ on $\mathbb{R}^{1+3}$,
where here and throughout, 
$\lbrace \partial_{\alpha} \rbrace_{\alpha=0,1,2,3}$ denotes the corresponding partial derivative vectorfields,
$x^0 :=t$ denotes time, $\partial_0 := \partial_t$, 
$\lbrace x^a \rbrace_{a=1,2,3}$ are the spatial coordinates, 
and repeated indices are summed over their relevant ranges, with
lowercase Greek indices ranging from $0$ to $3$ 
and lowercase Latin indices from from $1$ to $3$.
We assume that $\inf_{t=0}\varrho > 0$, which allows us to avoid 
the well-known difficulty that the hyperbolicity of the equations can degenerate 
along fluid-vacuum boundaries.

Our main goal in this paper is to prove a series of intimately related results tied to the regularity and geometry of solutions.
We study ``general\footnote{As we mentioned above, the solutions that we study have strictly positive density, 
i.e., we avoid studying fluid-vacuum boundaries.} solutions,''  
which can have non-vanishing vorticity (i.e., $\curl v \neq 0$) and non-constant entropy. 
We allow for any\footnote{We assume that the equation of state is sufficiently smooth.}
equation of state\footnote{In practice, instead of the density $\varrho$,
we work with the logarithmic density, defined in Subsect.\,\ref{SS:NEWFORMULATIONEULER}.} 
$p = p(\varrho,\Ent)$ with positive sound speed $\Speed := \sqrt{\frac{\partial p(\varrho,\Ent)}{\partial \varrho}}$.
The central theme of the paper is that under low regularity assumptions on the initial data, 
it is possible to avoid, at least for short times, the formation of shocks,
which are singularities caused by sound wave compression.
These issues are fundamental for the Cauchy problem: 
for sufficiently rough initial data, 
\emph{ill-posedness occurs \cites{hL1998,rG2018} due to instantaneous 
shock formation, which is precipitated by the degeneration of the acoustic geometry, including the
intersection of the acoustic characteristics}. 
Shocks are of particular interest because they are the only
singularities that have been shown, through constructive methods
\cites{dC2007,dCsM2014,jLjS2018,jLjS2021,tBsSvV2019a,tBsSvV2020}, to develop for open sets\footnote{We also mention 
here the spectacular
work \cite{fMpRiRjS2019b} on the existence of implosion singularities in spherical symmetry
under an adiabatic equation of state $p = \varrho^{\upgamma}$ with $\upgamma > 1$.
These are singularities in which the density and velocity blow up at the center of symmetry
in finite time. The methods of \cite{fMpRiRjS2019b} suggest that the
implosion singularities might enjoy co-dimension stability without symmetry assumptions,
though perhaps not full stability for an open set of data.} 
of regular initial data. 
This motivates our main result: 
controlling the time of existence under optimal Sobolev regularity assumptions 
on the data of ``the part of the flow that blows up'' in \cites{dC2007,dCsM2014,jLjS2018,jLjS2021}. 
See Theorem~\ref{E:TIMEOFCLASSICALEXISTENCEVERYROUGH} for a heuristic statement of the main result
and Theorem~\ref{T:MAINTHEOREMROUGHVERSION} for the precise version.
The proof relies on a deep analysis of the geometry of solutions
that \emph{exploits hidden structures} in the equations.
We remark that, in the language of the present paper, the formation
of a shock would correspond to the vanishing of the null lapse $\nulllapse$
defined in \eqref{E:NULLLAPSE}, more precisely to the following singular behavior:
$\| \nulllapse^{-1} \|_{L_t^1 L_x^{\infty}} = \infty$. To avoid this singular scenario for short times,
we prove the estimates stated in \eqref{E:L2INTIMEESTIMATESFORNULLLAPSEALONGCONES}.

\begin{nonumberremark}[Remarks on different versions of the article]
This version of our article closely matches the published version \cite{mDcLgMjS2022}, 
which features some corrections and an abbreviated introduction compared to the first version we posted on
arXiv.org on September 5, 2019 \cite{mDcLgMjS2019}. 
As in the published version, here we refer readers to the first arXiv version 
for an extended introduction that features some additional background material. 
\end{nonumberremark}

\begin{theorem}[Control of the time of classical existence (heuristic version)]
\label{E:TIMEOFCLASSICALEXISTENCEVERYROUGH}
The time of classical existence of a solution to the $3D$ compressible Euler equations 
can be controlled in terms of the $H^{2^+}(\Sigma_0)$-norm
of the ``wave-part'' of the data (which is tied to sound waves, i.e., the part of the solution that is prone to shock formation)
and additional Sobolev and H\"{o}lder norms of the ``transport-part'' of the data (which is tied to the transporting of vorticity and entropy),
where $\Sigma_0 := \lbrace 0 \rbrace \times \mathbb{R}^3$ is the initial Cauchy hypersurface. 
\end{theorem}

We now highlight three features of our work:
\begin{itemize}
	\item Our results are optimal in that \emph{$H^{2^+}(\Sigma_0)$ cannot be replaced with $H^2(\Sigma_0)$}. 
	More precisely, even in the irrotational and isentropic case (i.e., $\curl v \equiv 0$ and $\Ent \equiv \mbox{const}$, and thus
	the transport-part of the solution is trivial),
	the works \cites{hL1998,rG2018} imply that
	ill-posedness occurs\footnote{The Cartesian coordinate partial derivatives 
	of the solution blow up, but in principle, it could remain smooth in different coordinates; 
	e.g., Einstein's equations are well-posed in $H^2$ \cite{sKiRjS2015},
	even though they are $H^2$-ill-posed in wave coordinates \cite{bEhL2017}.}
	if one assumes only an $H^2(\Sigma_0)$-bound
	on $v$ and $\varrho$, due to the instantaneous formation of shocks.
	\item Our results appear to be the first of their kind for a quasilinear system featuring
	\emph{multiple characteristic speeds}, i.e., sound waves coupled to the transporting of vorticity and entropy. 
	\item In the irrotational and isentropic case, 
		where the Euler equations reduce to a quasilinear wave equation for a potential,
		Theorem~\ref{E:TIMEOFCLASSICALEXISTENCEVERYROUGH} recovers the low regularity well-posedness results for quasilinear wave equations 
		proved in \cites{hSdT2005,qW2017}.
		However, much like in the works \cites{jLjS2018,jLjS2021} on shocks, the following theme permeates our paper: 
		(especially) at low regularity levels, general compressible Euler solutions
		are \emph{not} ``perturbations of waves;''
		the presence of even the tiniest amount of vorticity or non-trivial entropy is a ``game changer''
		requiring substantial new insights, particularly for controlling the acoustic geometry.
		This is because \underline{the vorticity and entropy are deeply and subtly coupled to the sound waves}.
\end{itemize}

In proving Theorem~\ref{E:TIMEOFCLASSICALEXISTENCEVERYROUGH}, 
we derive several companion results of independent interest, including: 
\begin{itemize}
\item \emph{Control of the acoustic geometry} in the presence of vorticity and entropy.
	By ``acoustic geometry,'' we mean an \emph{acoustical eikonal function} $u$,
	that is, a solution to the acoustical eikonal equation
	$(\gfour^{-1})^{\alpha \beta} \partial_{\alpha} u \partial_{\beta} u =  0$,
	where the \emph{acoustical metric}\footnote{In practice, when constructing $u$,
	we work with a rescaled version of the acoustical metric; see Subsect.\,\ref{SS:EIKONAL}. \label{FN:RESCALEDEIKONAL}} 
	$\gfour = \gfour(\varrho,v,\Ent)$ is a Lorentzian metric 
	(see Def.\,\ref{D:ACOUSTICALMETRIC}) depending on the fluid solution. 
	Acoustical eikonal functions are adapted to the characteristics of the ``(sound) wave-part'' 
	of the solution and are fundamentally connected to shock waves.
	The regularity properties of $u$ are highly (and tensorially) tied to those of the fluid, and the intersection of the level sets of $u$
	would signify the formation of a shock.
\item \emph{Strichartz estimates} for (quasilinear) sound waves coupled to vorticity and entropy. 
\item \emph{Schauder estimates} for the vorticity and entropy, which solve transport-div-curl equations.
\end{itemize}

All aspects of our paper are fundamentally based on a new formulation of the compressible Euler 
equations as a system of wave and transport-div-curl equations,
derived in \cite{jS2019c} and stated in condensed form in 
Prop.\,\ref{P:GEOMETRICWAVETRANSPORT}. This new formulation exhibits remarkable geo-analytic properties 
that are crucial for our results. See also \cite{jLjS2020a} for the case of a barotropic
equation of state and \cite{mDjS2019} for a similar formulation of the relativistic Euler equations.

Standard proofs of local well-posedness for the compressible Euler flow
are based on applying only energy estimates and Sobolev embedding to a first-order formulation of the equations,
such as  \eqref{E:EULERSYSTEMSTANDARD}.
Such proofs
 require $(\varrho - \bar{\varrho},v,\Ent) \in H^{(5/2)^+}(\Sigma_0)$,
where $\bar{\varrho} > 0$ is a fixed constant background density.
Compared to such standard proofs, 
Theorem~\ref{E:TIMEOFCLASSICALEXISTENCEVERYROUGH} reduces the required Sobolev regularity of 
the wave-part of the data (i.e., the data of $\varrho$ and $\Flatdiv v$) by\footnote{Here, when discussing the regularity of $v$, $\Flatdiv v$, and $\curl v$,
we are implicitly referring to the Hodge estimate \eqref{E:STANDARDL2DIVCURLESTIMATES}.} half of a derivative, 
but requires additional smoothness on the transport-part of the data (i.e., of $\curl v$ and $\Ent$); see Theorem~\ref{T:MAINTHEOREMROUGHVERSION}. It is important to point out that one should not think that 
this additional smoothness of the transport-part of the data leads to an oversimplification of the problem.
This is because, to the best of our knowledge, one cannot propagate the extra smoothness using \emph{solely}
equations \eqref{E:EULERSYSTEMSTANDARD} (or other equivalent first-order formulation), i.e., 
without appealing to a non-standard formulation of the equations such as
the one given in Prop.\,\ref{P:GEOMETRICWAVETRANSPORT} and employed here
(see also \cite{mDdE2017} for another type of propagation of extra smoothness for the Euler equations
that also involves reformulating the equations).
Moreover, such propagation of extra
regularity does not hold for general first-order symmetric hyperbolic systems, which is one of the standard
frameworks used in the study of the compressible Euler equations. Furthermore, even when employing
the formulation of Prop.\,\ref{P:GEOMETRICWAVETRANSPORT}, the propagation of extra smoothness
for the transport part of the system is very delicate in that the transport- and wave-parts are
coupled in a highly non-trivial way (in particular through the acoustic geometry).
In this regard, a remarkable aspect of our work is:
\begin{quote}
We propagate the regularity of the ``smoother'' transport-part of the compressible Euler flow,
even though it is deeply coupled to the rougher wave-part.
\end{quote}
To propagate the extra smoothness, we exploit the full nonlinear structure of the aforementioned
new formulation of the equations and carry out a delicate analysis of the interaction of the wave- and transport-
parts of the system as well as the acoustic geometry.\footnote{Readers less familiar with
Strichartz and acoustic geometry estimates can consult the arXiv version of this paper \cite{mDcLgMjS2019},
wherein we provide a longer introduction with further background.}

\subsection{New formulation of the Euler equations\label{SS:NEWFORMULATIONEULER}}
In Section~\ref{SSS:STATEMENTOFEQUATIONSOFMOTION}, we provide the new formulation of the equations
that we use in our analysis. 	We first introduce some notation and define
some additional quantities that play a role in the new formulation.

Recall that we assume that the pressure $p$ is a given smooth function of $\varrho$ and $\Ent$,
and that the speed of sound $\Speed$ is defined by
$
\displaystyle
\Speed := \sqrt{\frac{\partial p}{\partial \varrho}\left|\right._{\Ent}}$, 
where $\displaystyle
\frac{\partial p}{\partial \varrho}\left|\right._{\Ent}$ is the partial derivative 
of $p$ with respect to $\varrho$ at fixed $\Ent$.
From now on, 
we view $p$ and $\Speed$ as smooth functions of the \emph{logarithmic density} 
\begin{align} \label{E:LOGDENSITY}
 \LogDensity 
	& := \ln \left(\frac{\varrho}{\bar{\varrho}} \right) : \mathbb{R}^{1+3} \rightarrow \mathbb{R},
\end{align}
(as opposed to the standard density) and $\Ent$, where we recall that
$\bar{\varrho} > 0$ is a fixed constant background density. That is,
we view $p = p(\LogDensity,\Ent)$ and $\Speed = \Speed(\LogDensity,\Ent)$.
If $f = f(\LogDensity,\Ent)$ is a scalar function, then
	we use the following notation to denote partial differentiation with respect to
	$\LogDensity$ and $\Ent$:
	$
	\displaystyle
	f_{;\LogDensity} 
	:= \frac{\partial f}{\partial \LogDensity}
	$
	and
	$
	\displaystyle
	f_{;\Ent} 
	:= \frac{\partial f}{\partial \Ent}
	$.

\subsubsection{Additional fluid variables}
\label{SS:ADDITIONALFLUIDVARIABLES}
We first recall that the fluid vorticity is the $\Sigma_t$-tangent vectorfield 
$\omega:\mathbb{R}^{1+3} \rightarrow \mathbb{R}^3$, 
where $\Sigma_t := \{ (\tau,x^1,x^2,x^3) \in \mathbb{R}^{1+3} \, | \, \tau = t \}$, 
with the following Cartesian spatial components:
\begin{align} \label{E:VORTICITYDEFINITION}
	\omega^i 
	& := (\Flatcurl v)^i
	:= \upepsilon^{iab} \partial_a v_b,
\end{align}
where throughout, $\upepsilon^{iab}$ denotes the fully antisymmetric symbol normalized by $\upepsilon^{123}=1$.

We will derive estimates for the \emph{specific vorticity} and \emph{entropy gradient}, which are vectorfields
featured in the next definition. These variables solve equations with a favorable structure
and thus play a key role in our analysis.

\begin{definition}[Specific vorticity and entropy gradient]
\label{D:SPECIFICVORTICITYANDENTROPYGRADIENT}
We define the specific vorticity $\vortrenormalized: \mathbb{R}^{1+3} \rightarrow \mathbb{R}^3$
and the entropy gradient $\GradEnt : \mathbb{R}^{1+3} \rightarrow \mathbb{R}^3$
to be the $\Sigma_t$-tangent vectorfields with the following Cartesian components:
\begin{align} \label{E:SPECIFICVORT}
	\vortrenormalized^i
	& := \frac{\omega^i}{(\varrho/\bar{\varrho})}
		= \frac{(\Flatcurl v)^i}{\exp \LogDensity},
	&
	\GradEnt^i
	& := \updelta^{ia} \partial_a \Ent.
\end{align}
\end{definition}

The ``modified'' fluid variables featured in the next definition solve equations with 
remarkable structures. In total, such structures allow us to prove that these variables
exhibit a gain in regularity compared to standard estimates.
We stress that this gain of regularity is crucial for showing that the different solution variables
have enough regularity to be compatible with our approach.

\begin{definition}[Modified fluid variables]
	\label{D:MODIFIEDFLUIDVARIABLES}
	We define the Cartesian components of the $\Sigma_t$-tangent vectorfield $\VortVort$ 
	and the scalar function $\DivGradEnt$ as follows:
	\begin{subequations}
	\begin{align} \label{E:RENORMALIZEDCURLOFSPECIFICVORTICITY}
		 \VortVort^i
		& :=
			\exp(-\LogDensity) (\Flatcurl \vortrenormalized)^i
			+
			\exp(-3\LogDensity) \Speed^{-2} \frac{p_{;\Ent}}{\bar{\varrho}} \GradEnt^a \partial_a v^i
			-
			\exp(-3\LogDensity) \Speed^{-2} \frac{p_{;\Ent}}{\bar{\varrho}} (\partial_a v^a) \GradEnt^i,
				\\
		\DivGradEnt
		& := 
			\exp(-2 \LogDensity) \Flatdiv \GradEnt 
			-
			\exp(-2 \LogDensity) \GradEnt^a \partial_a \LogDensity.
			\label{E:RENORMALIZEDDIVOFENTROPY}
	\end{align}
	\end{subequations}
\end{definition}

The following definitions are primarily for notational convenience.

\begin{definition}[The wave variables]
	\label{D:WAVEVARIABLES}
		We define the wave variables $\Psi_{\iota}$, ($\iota=0,1,2,3,4$), and the 
		array $\vec{\Psi}$ of wave variables, as follows:
		\begin{subequations}
		\begin{align} 
			\Psi_0
			& := \LogDensity,
			&
			\Psi_i 
			& := v^i,
			\qquad
			(i=1,2,3),
			&
			\Psi_4
			& := \Ent,
				\label{E:PSIDEFS} \\
			\vec{\Psi} & := (\Psi_0,\Psi_1,\Psi_2,\Psi_3,\Psi_4).
			&&
			&&
			\label{E:WAVEARRAY}
		\end{align}
		\end{subequations}
\end{definition}

\begin{definition}[Arrays of Cartesian component functions]
	\label{D:ARRAYSOFCARTESIANCOMPONENTFUNCTIONS}
	We define the following arrays:
	\begin{align}
		\vec{v}
		& := (v^1,v^2,v^3),
		&
		\vec{\vortrenormalized}
		& := (\vortrenormalized^1,\vortrenormalized^2,\vortrenormalized^3),
		&
		\vec{\GradEnt}
		& := (\GradEnt^1,\GradEnt^2,\GradEnt^3),
		&
		\vec{\VortVort}
		& := (\VortVort^1,\VortVort^2,\VortVort^3).
		\end{align}
\end{definition}

Throughout, we use the following notation for Cartesian partial derivative operators:
\begin{itemize} 
	\item $\partial$ denotes a spatial derivative with respect to the Cartesian coordinates.
	\item $\pmb{\partial} = (\partial_t,\partial)$ denotes a spacetime derivative with respect to the Cartesian coordinates.
\end{itemize}
Moreover,
$\pmb{\partial} \vec{\Psi}$ denotes the array of scalar functions 
$\pmb{\partial} \vec{\Psi} := (\partial_{\alpha} \Psi_{\iota})_{\alpha = 0,1,2,3, \iota = 0,1,2,3,4}$
(recall that $\partial_0 = \partial_t$),
and  
$\partial \vec{\Psi}$ denotes the array of scalar functions 
$\partial \vec{\Psi} := (\partial_a \Psi_{\iota})_{a = 1,2,3, \iota = 0,1,2,3,4}$.
Arrays such as 
$\pmb{\partial} \vec{v}$,
$\partial \vec{v}$,
$\pmb{\partial} \vec{\vortrenormalized}$, 
$\partial \vec{\vortrenormalized}$, 
$\pmb{\partial}^2 \vec{\Psi}$,
etc., are defined analogously.
Moreover, 
$
\partial^{\leq 1} \vec{\Psi}
$
denotes the array whose entries are those of $\vec{\Psi}$ together with those of $\partial \vec{\Psi}$,
and arrays such as
$\partial^{\leq 1} \vec{\vortrenormalized}$,
$\partial^{\leq 1} \vec{\GradEnt}$, 
etc., are defined analogously.

\subsubsection{Acoustical metric and wave operators}
\label{SSS:ACOUSTICALMETRICANDWAVEOPERATORS}
Our analysis of the wave-part of the system is fundamentally tied to the 
acoustical metric $\gfour$ and related geometric tensors.

\begin{definition}[The acoustical metric and first fundamental form] 
\label{D:ACOUSTICALMETRIC}
We define the \emph{acoustical metric} $\gfour = \gfour(\LogDensity,v,\Ent)$ relative
to the Cartesian coordinates as follows:
\begin{align}
		\gfour
		& := 
		-  dt \otimes dt
			+ 
			\Speed^{-2} \sum_{a=1}^3(dx^a - v^a dt) \otimes (dx^a - v^a dt).
				\label{E:ACOUSTICALMETRIC} 
	\end{align}
We define\footnote{As we describe in Subsubsect.\,\ref{SSS:NULLFRAME},
$g$ can be extended to a $\Sigma_t$-tangent spacetime tensor.
By definition, the extended version of $g$ agrees with the original version when acting on $\Sigma_t$-tangent vectors
and vanishes upon any contraction with $\Transport$.
The extended $g$ satisfies the identity
$g = \Speed^{-2} \sum_{a=1}^3(dx^a - v^a dt) \otimes (dx^a - v^a dt)$.
\label{FN:FIRSTFUNDAMENTALFORMEXTENDEDTOSPACETIMETENSOR}} 
the \emph{first fundamental form} $g = g(\LogDensity,v,\Ent)$ of $\Sigma_t$
and the corresponding \emph{inverse first fundamental form} $g^{-1} = g^{-1}(\LogDensity,v,\Ent)$
relative to the Cartesian coordinates as follows:
\begin{align} \label{E:FIRSTFUNDAMENTALFORM} 
		g
		& := 
			\Speed^{-2} \sum_{a=1}^3 dx^a \otimes dx^a,
		&
		g^{-1} 
		& := 
			\Speed^2 \sum_{a=1}^3 \partial_a \otimes \partial_a.
\end{align}
\end{definition}

It is straightforward to check that relative to the Cartesian coordinates, we have
\begin{align} \label{E:INVERSEACOUSTICALMETRIC}
	\gfour^{-1} 
		& = 
			- \Transport \otimes \Transport
			+ \Speed^2 \sum_{a=1}^3 \partial_a \otimes \partial_a,
		&
	\mbox{\upshape det} \gfour
	& = - \Speed^{-6}.
\end{align}
It is also straightforward to verify the following facts, which we will use throughout:
$\Transport$ is $\gfour$-orthogonal to $\Sigma_t$
and normalized by
\begin{align} \label{E:TRANSPORTISLENGTHONE}
	\gfour(\Transport,\Transport)
	& = -1.
\end{align}

\begin{remark}
	Note that $\gfour_{\alpha \beta} = \gfour_{\alpha \beta}(\vec{\Psi})$ and $\Transport^{\alpha} = \Transport^{\alpha}(\vec{\Psi})$.
	Note also that $(\gfour^{-1})^{00} = - 1$. We will sometimes silently use this basic fact.
\end{remark}

The following wave operators arise in our analysis of solutions.

\begin{definition}[Covariant and reduced wave operators]
$\square_{\gfour}$ denotes the covariant wave operator of $\gfour$, 
which acts on scalar functions $\varphi$ by the coordinate invariant formula
$\square_{\gfour} \varphi := \frac{1}{\sqrt{|\mbox{\upshape det $\gfour$}|}} \partial_\alpha \left(\sqrt{|\mbox{\upshape det} \gfour|} (\gfour^{-1})^{\alpha\beta} 
\partial_\beta \varphi \right)$. 
$\hat{\square}_{\gfour}$ denotes the reduced wave operator of $\gfour$, 
and it acts on scalar functions $\varphi$ by the following formula (relative to Cartesian coordinates):
$\hat{\square}_{\gfour} \varphi := (\gfour^{-1})^{\alpha \beta} \partial_{\alpha} \partial_\beta \varphi$.
\end{definition}

\subsubsection{Statement of the geometric wave-transport formulation of the compressible Euler equations}
\label{SSS:STATEMENTOFEQUATIONSOFMOTION}
We now provide the geometric formulation of the compressible Euler equations that we use
to study the regularity of solutions. Detailed versions of the equations were derived in \cite{jS2019c}*{Theorem~1},
but for our purposes here, it suffices to work with the schematic version stated
in Prop.\,\ref{P:GEOMETRICWAVETRANSPORT}. 

We will use the following schematic notation,
which captures the essential structures that are relevant for our analysis.
Later in the article, we will introduce additional schematic notation.
\begin{itemize}
	\item $\linsmoothfunction(A)[B]$ denotes any scalar-valued function that is linear in $B$ with coefficients that are a (possibly nonlinear) function of $A$,
		i.e., a term of the form $\gensmoothfunction(A) \cdot B$,
		where $\gensmoothfunction$ denotes a generic smooth function 
		that is free to vary from line to line.
	\item $\quadsmoothfunction(A)[B,C]$ denotes any scalar-valued function that 
		is quadratic in $B$ and $C$ 
			with coefficients that are a (possibly nonlinear) function of $A$,
			i.e., a term of the form 
			$\gensmoothfunction(A) \cdot B \cdot C$.
\end{itemize}

\begin{proposition}\cite{jS2019c}*{The geometric wave-transport formulation of the compressible Euler equations}
\label{P:GEOMETRICWAVETRANSPORT}
Smooth solutions to the compressible Euler equations 
\eqref{E:TRANSPORTDENSRENORMALIZEDRELATIVETORECTANGULAR}--\eqref{E:ENTROPYTRANSPORT}
also verify the following system of equations, where all terms on the RHSs
are displayed schematically:\footnote{The precise form of the schematic terms in equation \eqref{E:COVARIANTWAVE}
depends on $\Psi$, but the details are not important for our analysis. Similar remarks apply to the remaining equations. \label{FN:SUPPRESSIONOFCARTESIANINDEX}} 

\medskip

\noindent \underline{Wave equations:}
For $\Psi \in \lbrace \LogDensity, v^1, v^2, v^3, \Ent \rbrace$, we have
\begin{align} \label{E:COVARIANTWAVE}
\hat{\square}_{\gfour(\vec{\Psi})} \Psi 
& = 
	\mathfrak{F}_{(\Psi)}
	:=
	\linsmoothfunction(\vec{\Psi})[\vec{\VortVort},\DivGradEnt]
	+ 
	\quadsmoothfunction(\vec{\Psi})[\pmb{\partial} \vec{\Psi},\pmb{\partial} \vec{\Psi}].
 \end{align}
Moreover, replacing $\hat{\square}_{\gfour(\vec{\Psi})}$ on LHS~\eqref{E:COVARIANTWAVE} 
with the covariant wave operator $\square_{\gfour(\vec{\Psi})}$ leads to a wave equation
whose RHS has the same schematic form as RHS~\eqref{E:COVARIANTWAVE}.

\medskip

\noindent \underline{Transport equations:}
The Cartesian component functions $\lbrace \vortrenormalized^i \rbrace_{i=1,2,3}$ and $\lbrace \GradEnt^i \rbrace_{i=1,2,3}$ verify the following equations:
\begin{align} \label{E:TRANSPORTEQNMAIN}
\Transport \vortrenormalized^i
& 
= \linsmoothfunction(\vec{\Psi},\vec{\vortrenormalized},\vec{\GradEnt})[\pmb{\partial} \vec{\Psi}],
&
\Transport \GradEnt^i
& = \linsmoothfunction(\vec{\Psi},\vec{\GradEnt})[\pmb{\partial} \vec{\Psi}].
\end{align}

\medskip

\noindent \underline{Transport div-curl system for the specific vorticity:}
The scalar function $\dive \vortrenormalized$ 
and the Cartesian component functions $\lbrace \VortVort^i \rbrace_{i=1,2,3}$ verify the following equations:
\begin{subequations}
\begin{align}
\dive \vortrenormalized 
	& = 
		\mathfrak{F}_{(\dive \vortrenormalized)}
		:=
		\linsmoothfunction(\vec{\vortrenormalized})[\pmb{\partial} \vec{\Psi}],
\label{E:DIVVORTICITY}
\\
\Transport \VortVort^i 
	& = 
		\mathfrak{F}_{(\VortVort^i)}
		:=
		\quadsmoothfunction(\vec{\Psi})[\pmb{\partial} \vec{\Psi},\partial \vec{\vortrenormalized}]
		+
		\quadsmoothfunction(\vec{\Psi})[\pmb{\partial} \vec{\Psi},\partial \vec{\GradEnt}]
		+
		\quadsmoothfunction(\vec{\Psi},\vec{\GradEnt})[\pmb{\partial} \vec{\Psi},\pmb{\partial} \vec{\Psi}]
		+
		\linsmoothfunction(\vec{\Psi},\vec{\vortrenormalized},\vec{\GradEnt})[\pmb{\partial} \vec{\Psi}].
\label{E:TRANSPORTVORTICITYVORTICITY}
\end{align}
\end{subequations}

\medskip

\noindent \underline{Transport div-curl system for the entropy gradient:}
The scalar function $\DivGradEnt$ and the Cartesian component functions $\lbrace \GradEnt^i \rbrace_{i=1,2,3}$ verify the following equations:
\begin{subequations}
\begin{align}
	\Transport \DivGradEnt & 
		=
		\mathfrak{F}_{(\DivGradEnt)}
		:=
		\quadsmoothfunction(\vec{\Psi})[\pmb{\partial} \vec{\Psi},\partial \vec{\GradEnt}]
		+
		\quadsmoothfunction(\vec{\Psi},\vec{\GradEnt})[\pmb{\partial} \vec{\Psi},\pmb{\partial} \vec{\Psi}]
		+
		\linsmoothfunction(\vec{\Psi},\vec{\GradEnt})[\partial \vec{\vortrenormalized}],
		\label{E:TRANSPORTDIVGRADENTROPY}
\\
(\curl \GradEnt)^i & = 0.
\label{E:CURLGRADENT}
\end{align}
\end{subequations}

\end{proposition}

\begin{remark}
	We emphasize that for our main results,
	it is crucial that \emph{generic} first derivatives of $\vortrenormalized$ and $\GradEnt$ do \emph{not} appear on
	RHS~\eqref{E:COVARIANTWAVE}; rather, only the special combinations $\vec{\VortVort}$ and $\DivGradEnt$ appear.
\end{remark}

\begin{remark}
	\label{R:SIMPLEWAYTOTHINKOFVORTICITYANDENTROPYGRADIENT}
	In obtaining the form of the equations of Prop.\,\ref{P:GEOMETRICWAVETRANSPORT}
	as a consequence of the equations presented in \cite{jS2019c}, 
	we have used the simple relations
	$\vortrenormalized^i = \linsmoothfunction(\vec{\Psi})[\partial \vec{\Psi}]$
	and
	$\GradEnt^i = \updelta^{ic} \partial_c \Ent = \linsmoothfunction[\partial \vec{\Psi}]$.
\end{remark}

\begin{remark}
	In the equations of \cite{jS2019c}, all derivative-quadratic inhomogeneous terms are null forms.
	However, using Remark~\ref{R:SIMPLEWAYTOTHINKOFVORTICITYANDENTROPYGRADIENT},
	we have rewritten, for example, terms of type $\GradEnt \cdot \GradEnt$,
	as $\quadsmoothfunction[\pmb{\partial} \vec{\Psi},\pmb{\partial} \vec{\Psi}]$,
	where $\quadsmoothfunction[\pmb{\partial} \vec{\Psi},\pmb{\partial} \vec{\Psi}]$ is not necessarily a null form.
	That is, the quadratic terms $\quadsmoothfunction(\cdot)[\cdot,\cdot]$ in Prop.\,\ref{P:GEOMETRICWAVETRANSPORT} 
	are not necessarily null forms. While the presence of null form structures is crucial for the study of the formation of shocks,
	such null form structures are not important for the results of this article.
\end{remark}

Proposition \ref{P:GEOMETRICWAVETRANSPORT} justifies our use of the terminology
``wave-parts'' and ``transport-parts'' to refer to different parts of the system.
In particular, it shows that the Cartesian velocity components $v^i$ 
and $\LogDensity$
satisfy covariant wave equations of the form $\square_{\gfour} (v^i,\LogDensity) = \cdots$,
and we therefore refer to $\varrho$ and $v^i$ as the ``wave-part'' of the compressible Euler flow.
In contrast, $\Ent$, $\partial \Ent$, and the specific vorticity $\vortrenormalized$
satisfy transport equations along the integral curves of the material derivative
vectorfield $\Transport := \partial_t + v^a \partial_a$, and we therefore refer to these as the ``transport-part'' of the compressible Euler flow.
Moreover, the variables $\VortVort$ and $\DivGradEnt$
satisfy transport-div-curl subsystems and, therefore, we also consider these to be part of the ``transport-part'' of the flow.

\subsection{Statement of the main result concerning control of the time of classical existence}
\label{SS:CONTROLOFTIMEOFEXISTENCE}
We now precisely state the theorem on the time of classical existence. 
We recall that $\bar{\varrho} > 0$ is a fixed constant background density.

\begin{theorem}[Control of the time of classical existence under low regularity assumptions on the wave-part of the data]
	\label{T:MAINTHEOREMROUGHVERSION}
	Consider a smooth\footnote{For convenience, in this paper, we will assume that the solutions are as many times differentiable as necessary.
	Thus, ``smooth'' means ``as smooth as necessary for the \emph{qualitative} arguments (such as integration by parts) 
	to go through.'' However, all of our \emph{quantitative} estimates depend only on the Sobolev and H\"{o}lder norms mentioned in
	Theorem~\ref{T:MAINTHEOREMROUGHVERSION}. \label{FN:SMOOTHNESS}} 
	solution to the compressible Euler equations in $3D$
	whose initial data obey the following three assumptions\footnote{
	We note that since assumption 3 implies that $\varrho|_{\Sigma_0}$ is strictly positive,
	we have
	$\| \varrho - \bar{\varrho} \|_{H^{\Sob}(\Sigma_0)} 
	\approx
	\| \LogDensity \|_{H^{\Sob}(\Sigma_0)} 
	$,
	where $\LogDensity$ is the logarithmic density
	defined in \eqref{E:LOGDENSITY};
	this standard estimate can be proved using the
	product estimates of Lemma~\ref{L:PRELIMINARYPRODUCTANDCOMMUTATORESTIMATES}.
	\label{FN:DATAFORLOGDENSITY} }
	for some real numbers\footnote{Similar results can be proved for $\Sob > 5/2$ 
	using only energy estimates and Sobolev embedding.} 
	$2 < \Sob \leq 5/2$, $0 < \upalpha < 1$, $0 \leq D_{\Sob;\upalpha} < \infty$, $0 < c_1 < c_2$, and $0 < c_3$:
	\begin{enumerate}
		\item
			$\| (\varrho - \bar{\varrho},v,\curl v) \|_{H^{\Sob}(\Sigma_0)} 
			+
			\| \Ent \|_{H^{\Sob+1}(\Sigma_0)} \leq D_{\Sob;\upalpha}$,
			where $\bar{\varrho} > 0$ is a constant background density.
	\item The modified fluid variables $\VortVort$ and $\DivGradEnt$
		from Def.\,\ref{D:MODIFIEDFLUIDVARIABLES}
		(which vanish for irrotational and isentropic solutions),
		verify the H\"{o}lder-norm bound
		$\| (\VortVort,\DivGradEnt) \|_{C^{0,\upalpha}(\Sigma_0)} \leq D_{\Sob;\upalpha}$.
	\item Along $\Sigma_0$,
		the data functions are contained in the interior of 
		a compact subset $\mathfrak{K}$ of state-space 
		in which $\varrho \geq c_3$
		and the speed of sound is bounded from below by 
		$c_1$ and above by $c_2$.
	\end{enumerate}
		Then the solution's time of classical existence $T$ depends only on $D_{\Sob;\upalpha}$ and $\mathfrak{K}$,
		i.e., $T = T(D_{\Sob;\upalpha},\mathfrak{K}) > 0$. Moreover, the Sobolev regularity of the data is propagated
		by the solution for $t \in [0,T]$, as is H\"{o}lder regularity.\footnote{Prop.\,\ref{P:TOPORDERENERGYESTIMATES} 
		allows us to propagate all of the Sobolev regularity of the initial data,
		while \eqref{E:HOLDERREGULARITYOFMODIFIEDVARIABLESPROPAGATED} allows us to propagate
		some H\"{o}lder regularity for $(\vec{\VortVort},\DivGradEnt)$; the H\"{o}lder norm that we can control has an exponent 
		that is controllable in terms of $\Sob - 2$, but the exponent is possibly smaller than $\upalpha$.
		Moreover, the norms that we can control are uniformly bounded by functions of $(D_{\Sob;\upalpha},\mathfrak{K})$ for $t \in [0,T]$.}
\end{theorem}

\begin{remark}[Regularity needed for Strichartz estimates and differences from the irrotational and isentropic case]	
	\label{R:REGULRITYNEEDEDFORSTRICHARTZ}
	In Theorem~\ref{T:MAINTHEOREMROUGHVERSION},
	we have assumed additional Sobolev regularity on the transport-part of the flow (specifically $\curl v$ and $\Ent$)
	compared to the classical local well-posedness regime $(\varrho - \bar{\varrho},v,\Ent) \in H^{(5/2)^+}(\Sigma_0)$.
	This is because our approach to controlling
	$\int_0^T
			\| 
				\pmb{\partial} (\varrho,v,\Ent)
			\|_{L^{\infty}(\Sigma_{\uptau})}
		\, d \uptau
	$
	(which, as we mention below \eqref{E:MIXEDSPACETIMEESTIMATENEEDEDFORMAINTHEOREM}, is crucial for the proof of Theorem~\ref{T:MAINTHEOREMROUGHVERSION})
	relies on deriving Strichartz estimates for the nonlinear wave equations of Prop.\,\ref{P:GEOMETRICWAVETRANSPORT},
	which in turn requires the transport-part of the system to be more regular than the wave part. That is,
	at the classical local well-posedness regularity level 
	(which is such that the transport-part does not generically enjoy any relative gain in regularity),
	\emph{the approach of treating the compressible Euler equations as a coupled wave-div-curl-transport system
	fails,\footnote{At the classical local well-posedness level, one can treat the compressible Euler equations
	as a first-order symmetric hyperbolic system and
	obtain control over
	$\int_0^T
			\| 
				\pmb{\partial} (\varrho,v,\Ent)
			\|_{L^{\infty}(\Sigma_{\uptau})}
		\, d \uptau
	$
	as a consequence of Sobolev embedding and symmetric hyperbolic energy estimates.
	However, symmetric hyperbolic formulations of the equations do not exhibit 
	the intricate structures that we exploit in proving Theorem~\ref{T:MAINTHEOREMROUGHVERSION}.}}
	except in the irrotational and isentropic case \cites{hSdT2005,qW2017}
	(where the compressible Euler equations reduce to a quasilinear wave equation for a potential function).
	The failure comes from the wave equation source terms\footnote{See Def.\,\ref{D:ARRAYSOFCARTESIANCOMPONENTFUNCTIONS}	
	regarding the notation ``$\vec{\VortVort}$.''} 
	$\vec{\VortVort}$ and $\DivGradEnt$ on RHS~\eqref{E:COVARIANTWAVE},
	which are the modified fluid variables from Def.\,\ref{D:MODIFIEDFLUIDVARIABLES}.
	For general solutions (i.e., solutions with vorticity and non-trivial entropy), 
	from the point of view of regularity, $\vec{\VortVort}$ and $\DivGradEnt$ scale, 
	in a naive sense, like $\partial^2 v$ and $\partial^2 \Ent$. Therefore, 
	at the classical local well-posedness threshold, 
	$\vec{\VortVort}$ and $\DivGradEnt$ are elements of $H^{(1/2)^+}(\Sigma_t)$.
	This level of source-term regularity is insufficient for using a Duhamel argument to justify the desired Strichartz estimate
	for the nonlinear wave equation \eqref{E:COVARIANTWAVE}; see the proof of
	Theorem~\ref{T:IMPROVEMENTOFSTRICHARTZBOOTSTRAPASSUMPTION} for details on how the source terms
	enter into the proof of Strichartz estimates. 
	This is one key reason why, throughout the paper, we assume the transport-part 
	data regularity $\| \curl v \|_{H^{\Sob}(\Sigma_0)} \leq D_{\Sob;\upalpha}$ and $\| \Ent \|_{H^{\Sob+1}(\Sigma_0)} \leq D_{\Sob;\upalpha}$
	(these inequalities are automatically satisfied in the irrotational and isentropic\footnote{Technically, $\Ent$ could be a non-zero
	constant in the isentropic case, leading to $\| \Ent \|_{L^2(\Sigma_0)} = \infty$. However, this infinite norm would be irrelevant in that $\Ent$
	would be constant throughout the evolution and thus trivial to control.} 
	case).
\end{remark}

Given the estimates we derive in Sects.\,\ref{S:MODELPROBLEM}--\ref{S:ELLIPTICESTIMATESINHOLDERSPACES},
it is known
that Theorem~\ref{T:MAINTHEOREMROUGHVERSION} 
essentially
follows from
the following a priori estimate,
where $\pmb{\partial} f := (\partial_t f, \partial_1 f, \partial_2 f, \partial_3 f)$,
$\Sigma_{\uptau}$ is the standard flat hypersurface of constant time,
and $T$ is as in the statement of the theorem:
\begin{align} \label{E:MIXEDSPACETIMEESTIMATENEEDEDFORMAINTHEOREM}
		\int_0^T
			\| 
				\pmb{\partial} (\varrho,v,\Ent)
			\|_{L^{\infty}(\Sigma_{\uptau})}
		\, d \uptau
		& \lesssim 1.
\end{align}
That is, we will not provide the details 
on how Theorem~\ref{T:MAINTHEOREMROUGHVERSION}
follows from \eqref{E:MIXEDSPACETIMEESTIMATENEEDEDFORMAINTHEOREM}
via a continuity argument and persistence of regularity
 (see, e.g., \cite{aM1984}*{Section~2.2, Corollary 2} or \cite{hR2009b}*{Lemma~9.14} for the main ideas behind the proof), but will instead
focus our efforts on justifying the a priori estimate\footnote{Actually, under our framework, the bound
$
\int_0^T
			\| 
				\pmb{\partial} \Ent
			\|_{L^{\infty}(\Sigma_{\uptau})}
		\, d \uptau
		\lesssim 1
$ will be trivial to justify since we will prove the stronger result $\Ent \in  L^{\infty}\left([0,T],H^{\Sob+1}(\mathbb{R}^3) \right)$.}  
\eqref{E:MIXEDSPACETIMEESTIMATENEEDEDFORMAINTHEOREM}
for $T > 0$ sufficiently small (where the required smallness depends only the norms of the data and the set $\mathfrak{K}$ 
mentioned in Theorem~\ref{T:MAINTHEOREMROUGHVERSION}).
More precisely, our approach requires us to prove a stronger result,
namely Theorem~\ref{T:IMPROVEMENTOFSTRICHARTZBOOTSTRAPASSUMPTION},
whose proof in turn is coupled to all of the other ingredients mentioned above.
We also remark that, as we explain in 
Sects.\,\ref{S:STRICHARTZESTIMATESFORWAVEUPGRADEDTOHOLDER}--\ref{S:REDUCTIONSOFSTRICHARTZ},
most of the arguments needed for the proof of Theorem~\ref{T:IMPROVEMENTOFSTRICHARTZBOOTSTRAPASSUMPTION}, 
including a series of technical-but-known
reductions, are supplied by other papers cited in Sects.\ 
\ref{S:STRICHARTZESTIMATESFORWAVEUPGRADEDTOHOLDER}--\ref{S:REDUCTIONSOFSTRICHARTZ}.
In this paper, our main focus will be showing how to control the vorticity and entropy in norms that 
allow to use the machinery from these other papers.
Our proof relies on norms of the vorticity and entropy
on constant-time hypersurfaces and sound cones,
and the main novelties of our work are:
\textbf{i)} we can propagate
substantial smoothness for the vorticity, entropy, and modified fluid variables
$\VortVort$ and $\DivGradEnt$ from Def.\,\ref{D:MODIFIEDFLUIDVARIABLES}, 
even though these variables are intimately coupled to the rougher wave part of the solution;
\textbf{ii)} we can obtain suitable estimates for the acoustic geometry by exploiting
the precise structure of the new formulation of compressible Euler flow  
provided by Prop.\,\ref{P:GEOMETRICWAVETRANSPORT},
which allows us to show that the main top-order vorticity/entropy-dependent terms driving the evolution
of the acoustic geometry are\footnote{See, for example, the first terms
on RHSs~\eqref{E:MODIFIEDRAYCHAUDHURI} and \eqref{E:ANGDCOMMUTEDMODIFIEDRAYCHAUDHURI}.} 
$\VortVort$ and $\DivGradEnt$ -- as opposed to 
generic first-order derivatives of $\vortrenormalized$ and $\GradEnt$.

\begin{remark}[Remarks on local well-posedness]
	\label{R:UPGRADETHEPROOF}
	Theorem~\ref{T:MAINTHEOREMROUGHVERSION} provides the main ingredient, 
	namely a priori estimates for smooth solutions,
	needed for a full proof of local well-posedness, 
	including existence in the regularity spaces featured in the theorem and uniqueness in related spaces.
	We anticipate that the remaining aspects of a full proof of local well-posedness 
	could be shown by deriving, using the ideas that we use to prove Theorem~\ref{T:MAINTHEOREMROUGHVERSION}, 
	uniform estimates for sequences of smooth solutions and their differences. For ideas on how to proceed,
	readers can consult \cite{hSdT2005}, in which existence and uniqueness were proved at low regularity levels 
	for quasilinear wave equations.
\end{remark}

We now further describe some ingredients of independent interest
that we use in the proof of Theorem~\ref{T:MAINTHEOREMROUGHVERSION}.
\begin{enumerate}
	\renewcommand{\labelenumi}{\textbf{\Roman{enumi})}}
	\item \textbf{Control of the acoustic geometry}.
		For quasilinear wave systems with a single wave operator,
		there has been remarkable progress
		on obtaining control of the acoustic geometry 
		and applications to low regularity local well-posedness, see 
		\cites{sKiR2003,sKiR2005d,sKiR2006a,sKiR2005b,sKiR2005c,sKiR2008,hSdT2005,sKiRjS2015,qW2017}.
		A fundamental new aspect of the present work 
		is that \emph{the vorticity and entropy
		appear as source terms in the acoustic geometry estimates},
		signifying a coupling between the geometry of sound cones and transport phenomena.
		The coupling enters in particular through the 
		Ricci curvature of the acoustical metric $\gfour$
		(see Def.\,\ref{D:ACOUSTICALMETRIC}), which, by virtue of the compressible Euler equations, 
		can be expressed in terms of quantities involving the vorticity and entropy;
		see Lemma~\ref{L:CURVATUREDECOMPOSITIONS}.
		We also exploit some remarkable
		consequences of the compressible Euler formulation provided by Prop.\,\ref{P:GEOMETRICWAVETRANSPORT}, namely,
		through careful geometric decompositions we show that high order derivatives of vorticity and entropy
		\emph{occur only the special combinations $\VortVort$ and $\DivGradEnt$}; 
		see Prop.\,\ref{P:PDESMODIFIEDACOUSTICALQUANTITIES}.
		The point is that the modified fluid variables $\VortVort$ and $\DivGradEnt$
		-- as opposed to generic first-order derivatives of $\vortrenormalized$ and $\GradEnt$ --
		\emph{enjoy good estimates up to top-order along sound cones},
		and such estimates turn out to be crucial for obtaining control of the acoustic geometry.
		This unexpected-but-critical structure should not be taken for granted since
		\emph{generic} high order derivatives of the vorticity and entropy
		can be controlled \emph{only along constant-time hypersurfaces}.
	\item \textbf{Strichartz estimates for the wave-part of solutions}.
			As in the works cited in \textbf{I}, our derivation of Strichartz estimates
			is fundamentally based on having suitable quantitative control of the acoustic geometry; see Sect.\,\ref{S:REDUCTIONSOFSTRICHARTZ}.
			Therefore, in view of the discussion in \textbf{I}, we see that the Strichartz estimates
			are tied to the delicate regularity properties of the vorticity and entropy along sound cones.
	\item \textbf{New Schauder estimates for the transport-div-curl equations} appearing in the compressible Euler formulation;
		see Sect.\,\ref{S:ELLIPTICESTIMATESINHOLDERSPACES}.
		These provide us with mixed spacetime estimates for the transport-part that complement the Strichartz estimates,
		allowing us to control the new (compared to the previously treated case of irrotational and isentropic solutions) kinds of 
		derivative-quadratic terms that we encounter in the energy and elliptic estimates.
\end{enumerate}

\subsection{Some general remarks and connections with prior work}
\label{SS:GENERALREMARKSCONNECTIONSPREVIOUSWORK}
Much of the remarkable progress that has been obtained for quasilinear hyperbolic PDEs
over the last two decades stems from studying specific systems of 
geometric or physical interest (as opposed to ``general systems''), 
where very delicate structural features of the equations can be exploited 
in combination with a precise understanding
of the regularity of the system's characteristics. Moreover, a common theme in these
developments is that the special structural and/or regularity features of the system become visible
only after one rewrites the equations in some novel way, which might involve a coordinate
system adapted to the problem in question and/or a new formulation of the equations
of motion in the spirit of the equations of Prop.\,\ref{P:GEOMETRICWAVETRANSPORT}.

A primary example is Einstein's equations, where 
the following notable results were obtained in recent years:
the formation of trapped surfaces \cite{dC2009},
the stability of the Kerr Cauchy horizon \cite{mDjL2017}, 
stable curvature blowup \cites{iRjS2018,iRjS2018b,jS2018a},
instability of anti de Sitter space \cites{gM2020,gM2018},
and the proof \cite{sKiRjS2015} of the bounded $L^2$ curvature conjecture.
For the compressible Euler equations, we can cite
Christodoulou's breakthrough works 
\cites{dC2007,dCsM2014} on the formation of shocks in the irrotational and isentropic case,
and, more recently, the works \cites{jLjS2018,jLjS2021,tBsSvV2019a,tBsSvV2020} 
on the formation of shocks for solutions with vorticity and entropy.

Regarding the problem of low regularity, in the case
of an irrotational and isentropic flow, where the compressible Euler equations can be written as a system of quasilinear
wave equations with a single wave speed, our result follows directly from the optimal
low regularity local well-posedness by Smith and Tataru \cite{hSdT2005} or also from 
the more recent physical-space approach to the problem by Wang \cite{qW2017}.
This highlights, once more, that the main novelty of our work is to obtain control of the
fluid flow under optimal regularity assumptions on the wave-part of the system \emph{in the presence
of vorticity and entropy.}

In order to highlight the difference between our result and what can be obtained using solely
techniques from quasilinear wave equations, we now discuss an approach that one could take for 
controlling the wave-part of the system
at sub-$H^{(5/2)^+}(\Sigma_0)$ regularity levels\footnote{Recall that 
$H^{(5/2)^+}(\Sigma_0)$ is what is required for standard local well-posedness based on energy
estimates and Sobolev embedding.}, one that is simpler than the approach that we use here,
but less powerful in that it would \emph{not} allow one to reach the $H^{2^+}(\Sigma_0)$ regularity threshold 
for the wave-part.
Specifically, 
one could control the wave-part of the system at a regularity level below $H^{(5/2)^+}(\Sigma_0)$
by invoking the technology of Strichartz estimates for \emph{linear} wave equations with rough coefficients,
based on Fourier integral parametrix representations,
developed in a series of works by Tataru \cites{dT2000,dT2001,dT2002b}, 
which improved the foundational work \cite{hBjyC1999b}
of Bahouri--Chemin; see also the related work \cite{hS1998}. 
By ``linear,'' we mean in particular that the proofs do not exploit any information about the principal coefficients of the
wave operator besides their pre-specified regularity.
In particular, when combined with the bootstrap-type arguments given in 
Sects.\,\ref{S:DATAANDBOOTSTRAPASSUMPTION}--\ref{S:ELLIPTICESTIMATESINHOLDERSPACES},
the methods of \cites{dT2000,dT2002b} (see in particular \cite{dT2000}*{Theorem~6} 
and \cite{dT2002b}*{Theorem~5.1})
would allow one to prove local well-posedness 
assuming that $(\varrho - \bar{\varrho},v) \in H^{(13/6)^+}(\Sigma_0)$
and that the transport-part of the data enjoys the same relative gain in regularity 
that we assume for our results
(e.g., $\Ent \in H^{(19/6)^+}(\Sigma_0)$ and $\partial^2 \Ent \in C^{0,0^+}(\Sigma_0)$);
see Subsubsect.\,\ref{SSS:MODELWAVESTRICHARTZ} for further discussion.
The work \cite{hSdT2002} shows that without further information about the principal coefficients of the wave equation, 
Tataru's linear Strichartz estimates are optimal.
Thus, since our results further lower the Sobolev regularity threshold by $1/6$,
our analysis \emph{necessarily exploits} the specific nonlinear structure of the equations of Prop.\,\ref{P:GEOMETRICWAVETRANSPORT}.
We also refer to the works (some of which we mentioned earlier)
\cites{hSdT2005,sKiR2003,sKiR2005d,sKiRjS2015,qW2017}
for further low regularity results in which the nonlinear structure of the PDE plays a fundamental role.

\subsection{Paper outline}
\label{SS:PAPEROUTLINE}
The remainder of the paper is organized as follows:
\begin{itemize}
	\item In Sect.\,\ref{S:MODELPROBLEM}, we outline the main ideas of our analysis through
		the study of a model problem.
	\item In Sect.\,\ref{S:DATAANDBOOTSTRAPASSUMPTION}, we recall some standard constructions from Littlewood--Paley theory,
		define the norms that we use until Sect.\,\ref{S:SETUPCONSTRUCTIONOFEIKONAL},
		define the parameters that
		play a role in our analysis, state our assumptions on the data, and formulate bootstrap 
		assumptions. The two key bootstrap assumptions are Strichartz estimates for the wave-part of the solution
		and complementary mixed spacetime estimates for the transport-part.
	\item In Sect.\,\ref{S:PRELIMINARYENERGYANDELLIPTICESTIMATES}, we use the bootstrap assumptions
		to derive preliminary below-top-order energy and elliptic estimates, which are useful
		for controlling simple error terms.
	\item In Sect.\,\ref{S:TOPORDERENERGYESTIMATES}, we use the bootstrap assumptions 
		and the results of Sect.\,\ref{S:PRELIMINARYENERGYANDELLIPTICESTIMATES}
		to derive top-order energy and elliptic estimates along constant-time hypersurfaces.
	\item In Sect.\,\ref{S:ENERGYESTIMATESALONGNULLHYPERSURFACES}, we derive
		energy estimates along acoustic null hypersurfaces, which complement the estimates
		from Sect.\,\ref{S:TOPORDERENERGYESTIMATES}.
		We need these estimates along null hypersurfaces in Sect.\,\ref{S:ESTIMATESFOREIKONALFUNCTION},
		when we control the acoustic geometry.
		Compared to prior works, 
		the main contribution of Sect.\,\ref{S:ENERGYESTIMATESALONGNULLHYPERSURFACES}
		is the estimate \eqref{E:SOUNDCONEENERGYESTIMATESFORMODIFIEDVARIABLES},
		which shows that the modified fluid variables
		$(\vec{\VortVort},\DivGradEnt)$
		can be controlled in $L^2$ up to top-order along acoustic null hypersurfaces, i.e., sound cones;
		as we described in Subsect.\,\ref{SS:CONTROLOFTIMEOFEXISTENCE},
		\emph{such control along sound cones is not available 
		for generic top-order derivatives of the vorticity and entropy}.
		\item In Sect.\,\ref{S:STRICHARTZESTIMATESFORWAVEUPGRADEDTOHOLDER}, we 
		prove Theorem~\ref{T:IMPROVEMENTOFSTRICHARTZBOOTSTRAPASSUMPTION}, which yields Strichartz estimates
		for the wave-part of the solution, thereby improving the first key bootstrap assumption
		and justifying the estimate \eqref{E:MIXEDSPACETIMEESTIMATENEEDEDFORMAINTHEOREM}.
		The proof of Theorem~\ref{T:IMPROVEMENTOFSTRICHARTZBOOTSTRAPASSUMPTION} 
		is conditional on Theorem~\ref{T:FREQUENCYLOCALIZEDSTRICHARTZ},
		whose proof in turn relies on the estimates for the acoustic geometry
		that we derive in Sect.\,\ref{S:ESTIMATESFOREIKONALFUNCTION}.
	\item In Sect.\,\ref{S:ELLIPTICESTIMATESINHOLDERSPACES}, we use Schauder estimates 
		to derive mixed spacetime estimates for the transport-part of the solution,
		thereby improving the second key bootstrap assumption.
		At this point in the paper, to close the bootstrap argument and complete the proof of Theorem~\ref{T:MAINTHEOREMROUGHVERSION},
		it only remains for us to prove Theorem~\ref{T:FREQUENCYLOCALIZEDSTRICHARTZ}.
	\item In Sect.\,\ref{S:SETUPCONSTRUCTIONOFEIKONAL}, 
			in service of proving Theorem~\ref{T:FREQUENCYLOCALIZEDSTRICHARTZ},
			we construct the acoustic geometry on spacetime slabs
			corresponding to a partition of the bootstrap time interval; 
			see Subsect.\,\ref{SS:PARTITIONOFBOOTSTRAPINTERVAL}	for the construction of the partition.
			The acoustic geometry is centered around an acoustical eikonal function.
			We also define corresponding geometric norms.
	\item In Sect.\,\ref{S:ESTIMATESFOREIKONALFUNCTION},
		we derive estimates for the acoustic geometry.
		The main result is Prop.\,\ref{P:MAINESTIMATESFOREIKONALFUNCTIONQUANTITIES}.
	\item In Sect.\,\ref{S:REDUCTIONSOFSTRICHARTZ}, we review some results derived in
		\cite{qW2017}, which in total show that the
		results of Sect.\,\ref{S:ESTIMATESFOREIKONALFUNCTION}
		imply Theorem~\ref{T:FREQUENCYLOCALIZEDSTRICHARTZ}.
		This closes the bootstrap argument, justifies the estimate \eqref{E:MIXEDSPACETIMEESTIMATENEEDEDFORMAINTHEOREM}, 
		and completes the proof of Theorem~\ref{T:MAINTHEOREMROUGHVERSION}.
\end{itemize}

\noindent \emph{Note added.}
After the completion of this manuscript, the 
work \cite{qWang2019} became available, 
in which the author considers the compressible Euler equations under a barotropic equation of state $p = p(\varrho)$
(and thus the variable $\Ent$ is absent from the analysis).
In this case, the author was able to lower the regularity of $\curl v|_{\Sigma_0}$
compared to Theorem \ref{T:MAINTHEOREMROUGHVERSION} by eliminating the
H\"{o}lder-norm bound assumption and showing that it suffices to assume
$\curl v \in H^{\Sob'}(\Sigma_0)$, where $2 < \Sob' < \left(\frac{\Sob - 2}{5} \right)^2$.
Moreover, in the wake of \cite{qWang2019}, there also appeared
\cite{hZ2020}, where a $2D$ local well-posedness
result is established in the barotropic case such that the density, velocity, and specific vorticity 
are in $H^2$, and \cite{hZ2021}, which provides an alternative proof of the results of \cite{qWang2019}.

\section{A model problem}
\label{S:MODELPROBLEM}
In this section, we discuss a model problem that serves as a blueprint for the rest of the paper.
The purpose of this section is to provide insight into the analysis and is entirely independent of the rest of the paper.
Readers not interested in a schematic guide to the main ideas of the paper can skip this section.

\subsection{Overview of the analysis via a model problem}
\label{SS:ANALYSISOVERVIEW}
In this subsection, we exhibit some of the main ideas behind our analysis by discussing a model problem.

\subsubsection{Statement of the model system}
\label{SSS:MODELSYSTEM}
We will study the following schematically depicted model system
in the scalar unknown $\Psi$ and the $\Sigma_t$-tangent
unknown vectorfield $W$ on $\mathbb{R}^{1+3}$:
\begin{subequations}
\begin{align}
	\hat{\square}_{\gfour(\Psi)} \Psi
	& = \Flatcurl W + \pmb{\partial} \Psi \cdot \pmb{\partial} \Psi,
		\label{E:MODELWAVE} \\
	\Flatdiv W 
	& = \partial \Psi,
		\label{E:MODELDIV} \\
	\left\lbrace
	\partial_t 
	+
	\Psi
	\partial_1
	\right\rbrace
	\Flatcurl W
	& = \partial \Psi \cdot \partial W.
	\label{E:MODELCURLEVOLUTION}
\end{align}
\end{subequations}
We intend for the system \eqref{E:MODELWAVE}--\eqref{E:MODELCURLEVOLUTION} to be a caricature of the equations of Prop.\,\ref{P:GEOMETRICWAVETRANSPORT}.
Above, $\gfour_{\alpha \beta}(\Psi)$ are given Cartesian component functions
(assumed to depend smoothly on $\Psi$) of the Lorentzian metric $\gfour$,
and $\hat{\square}_{\gfour(\Psi)} := (\gfour^{-1})^{\alpha \beta} \partial_{\alpha} \partial_{\beta}$.
$\Psi$ may be thought of as a model for the wave-part of the compressible Euler equations,
while $W$ may be thought of as a model for the transport-part (e.g., the vorticity and entropy gradient),
with $\partial_t 
	+
	\Psi
	\partial_1
$
a model quasilinear transport operator (the fact that it involves only $\partial_t$ and $\partial_1$ is
not important).
That is, from the point of view of regularity, 
we can think that $\Psi \sim (\LogDensity,v)$ and $W \sim (\curl v,\partial \Ent)$.
We intend for the reader to interpret the inhomogeneous terms schematically
(especially, since, for example, LHS~\eqref{E:MODELWAVE} is a scalar 
while the first term on RHS~\eqref{E:MODELWAVE} appears to be a vector).

We will outline how to control the time of existence  
for solutions to the model system \eqref{E:MODELWAVE}--\eqref{E:MODELCURLEVOLUTION} 
assuming the data-bound
\begin{align*}
\| (\Psi,\partial_t \Psi) \|_{H^{\Sob}(\Sigma_0) \times H^{\Sob-1}(\Sigma_0)}
	+
	\| \partial W \|_{H^{\Sob-1}(\Sigma_0)}
< \infty,
\end{align*}
where $2 < \Sob \leq 5/2$ is a fixed real number.
In Subsubsect.\,\ref{SSS:MODELMIXEDSPACETIMEFORTRANSPORT},
we will find that we need to make the further H\"{o}lder regularity assumption
$\| \curl W \|_{C^{0,\upalpha}(\Sigma_0)} < \infty$ for some $\upalpha > 0$,
much like we did in Theorem~\ref{T:MAINTHEOREMROUGHVERSION}.
In the rest of Subsect.\,\ref{SS:ANALYSISOVERVIEW},
``$\mbox{\upshape data}$''
schematically denotes any quantity depending on 
$
\| (\Psi,\partial_t \Psi) \|_{H^{\Sob}(\Sigma_0) \times H^{\Sob-1}(\Sigma_0)}
	+
	\| \partial W \|_{H^{\Sob-1}(\Sigma_0)}
$.

\subsubsection{A priori energy and elliptic estimates along $\Sigma_t$ for the model system}
\label{SSS:MODELENERGYESTIMATE}
The most fundamental step in controlling the time of existence
is to derive a priori energy and elliptic estimates along $\Sigma_t$.
In the context of the compressible Euler equations, 
we provide the analog of this step in
Prop.\,\ref{P:TOPORDERENERGYESTIMATES} below.
To obtain the desired a priori estimate for the model system, 
we first note that 
equation \eqref{E:MODELDIV} and the standard elliptic Hodge estimate 
\begin{align} \label{E:ELLIPTICHODGE}
	\| \partial W \|_{H^{\Sob-1}(\Sigma_t)} 
	& \lesssim 
	\| \Flatdiv W \|_{H^{\Sob-1}(\Sigma_t)} 
	+ 
	\| \Flatcurl W \|_{H^{\Sob-1}(\Sigma_t)}
\end{align}
together imply the following bound:
\begin{align}\label{E:PARTIALWHODGE}
	\| \partial W \|_{H^{\Sob-1}(\Sigma_t)}
	&
	\lesssim
	 \| \partial \Psi \|_{H^{\Sob-1}(\Sigma_t)}
	+
	\| \Flatcurl W \|_{H^{\Sob-1}(\Sigma_t)}. 
\end{align}
Next, by combining standard estimates for the wave equation \eqref{E:MODELWAVE}, 
based on energy estimates and the Littlewood--Paley calculus,
we deduce (where we ignore all numerical constants ``$C$'') that
\begin{align} \label{E:MODELWAVEEQUATIONBASICENERGYSTIMATE}
	&
	\| (\Psi,\partial_t \Psi) \|_{H^{\Sob}(\Sigma_t) \times H^{\Sob-1}(\Sigma_t)}^2
		\\
	& \leq
		\mbox{\upshape data}
		+
		\int_0^t
			\left\lbrace
				1
				+
				\| \pmb{\partial} \Psi \|_{L_x^{\infty}(\Sigma_{\uptau})}
			\right\rbrace
			\left\lbrace
				\| (\Psi,\partial_t \Psi) \|_{H^{\Sob}(\Sigma_{\uptau}) \times H^{\Sob-1}(\Sigma_{\uptau})}^2
				+
				\| \Flatcurl W \|_{H^{\Sob-1}(\Sigma_{\uptau})}^2
			\right\rbrace
		\, d \uptau.
		\notag
\end{align}
Similarly, 
with the help of the Littlewood--Paley calculus,
we can derive energy estimates for the transport equation \eqref{E:MODELCURLEVOLUTION}
and use \eqref{E:PARTIALWHODGE} to control the top-order derivatives of the factor $\partial W$
on RHS~\eqref{E:MODELCURLEVOLUTION},
thereby obtaining the following bound:
\begin{align} \label{E:MODELTRANSPORTENERGYESTIMATE}
&
\| \Flatcurl W \|_{H^{\Sob-1}(\Sigma_t)}^2
	\\ 
& \leq
	\mbox{\upshape data}
	+
	\int_0^t
			\left\lbrace
				1
				+
				\| \pmb{\partial} \Psi \|_{L^{\infty}(\Sigma_{\uptau})}
				+
				\|
					\partial W
				\|_{L^{\infty}(\Sigma_{\uptau})}
			\right\rbrace
			\left\lbrace
				\| (\Psi,\partial_t \Psi) \|_{H^{\Sob}(\Sigma_{\uptau}) \times H^{\Sob-1}(\Sigma_{\uptau})}^2
				+
				\| \Flatcurl W \|_{H^{\Sob-1}(\Sigma_{\uptau})}^2
			\right\rbrace
		\, d \uptau.
		\notag
\end{align}
Adding \eqref{E:MODELWAVEEQUATIONBASICENERGYSTIMATE} and \eqref{E:MODELTRANSPORTENERGYESTIMATE},
applying Gr\"{o}nwall's inequality,
and finally again using \eqref{E:PARTIALWHODGE},
we obtain (again ignoring all numerical constants ``$C$'')
the following estimate:
\begin{align} \label{E:MODELGRONWALL}
	\| (\Psi,\partial_t \Psi) \|_{H^{\Sob}(\Sigma_t) \times H^{\Sob-1}(\Sigma_t)}
	+
	\| \partial W \|_{H^{\Sob-1}(\Sigma_t)}
	& \leq
		\mbox{\upshape data}
		\times \exp\left(1 + \| \pmb{\partial} \Psi \|_{L^1([0,t])L_x^{\infty}} + \| \partial W \|_{L^1([0,t])L_x^{\infty}} \right).
\end{align}
Thus, \eqref{E:MODELGRONWALL} would immediately imply the desired a priori estimate
if we were able to simultaneously show that for $T > 0$ sufficiently small, we have the following key bounds
for some $\updelta >0 $ and $\updelta_1 >0$ with $0 < \updelta_1 \leq \upalpha$:
\begin{align} \label{E:MODELSPACETIMEBOUNDS}
\| \pmb{\partial} \Psi \|_{L^2([0,T])L_x^{\infty}},
	\,
\| \partial W \|_{L^2([0,T])L_x^{\infty}}
\lesssim 
T^{\updelta}
\mbox{\upshape data}
+
T^{\updelta}
\| \Flatcurl W \|_{C^{0,\updelta_1}(\Sigma_0)}.
\end{align}
The rest of the discussion in Subsect.\,\ref{SS:ANALYSISOVERVIEW} concerns
the proof of \eqref{E:MODELSPACETIMEBOUNDS}.

\subsubsection{Strichartz estimates and acoustic geometry for the model system}
\label{SSS:MODELWAVESTRICHARTZ}
We now discuss how to establish \eqref{E:MODELSPACETIMEBOUNDS} for the 
term $\| \pmb{\partial} \Psi \|_{L^2([0,T])L_x^{\infty}}$ using Strichartz estimates.
In practice, this can be accomplished by first making a bootstrap assumption
that is weaker than \eqref{E:MODELSPACETIMEBOUNDS}, 
then combining it with \eqref{E:MODELGRONWALL} to deduce the energy bound
\begin{align*}
\| (\Psi,\partial_t \Psi) \|_{H^{\Sob}(\Sigma_t) \times H^{\Sob-1}(\Sigma_t)}
	+
	\| \partial W \|_{H^{\Sob-1}(\Sigma_t)}
\leq
		\mbox{\upshape data},
\end{align*}
and then finally proving Strichartz estimates that imply the ``improved'' estimate 
\begin{align*}
\| \pmb{\partial} \Psi \|_{L^2([0,T])L_x^{\infty}} + \| \partial W \|_{L^2([0,T])L_x^{\infty}}
\lesssim 
T^{\updelta}
\mbox{\upshape data}
+
T^{\updelta}
\| \Flatcurl W \|_{C^{0,\updelta_1}(\Sigma_0)}.
\end{align*}
Thus, to illustrate the main ideas, we will assume the energy bound
and sketch how to prove
$\| \pmb{\partial} \Psi \|_{L^2([0,T])L_x^{\infty}} \lesssim 1$,
where, for convenience, we will ignore the small power of $T^{\updelta}$ (which in reality is important for gaining smallness in various estimates)
and also ignore term ``$\mbox{\upshape data}$'' by considering it to be $\lesssim 1$.
At this point in our discussion of the model system,
we will also ignore the following important technical point: 
to close some estimates,
one must achieve control of not only
$\| \pmb{\partial} \Psi \|_{L^2([0,T])L_x^{\infty}}$,
but also
$
\sum_{\upnu \geq 2}
		\upnu^{2 \updelta_1}
		\|
			P_{\upnu} \pmb{\partial} \vec{\Psi}
		\|_{L^2([0,T])L_x^{\infty}}^2
$
and
$
\|
		\pmb{\partial} \vec{\Psi}
	\|_{L^2([0,T])C_x^{0,\updelta_1}}
$,
where $P_{\upnu}$ are standard dyadic Littlewood--Paley projections and $\updelta_1 > 0$ is a small H\"{o}lder exponent;
see Theorem~\ref{T:IMPROVEMENTOFSTRICHARTZBOOTSTRAPASSUMPTION} and Cor.\,\ref{C:HOLDERTYPESTRICHARTZESTIMATEFORWAVEVARIABLES} for the details.
We will elaborate on the importance of controlling 
$
\|
		\pmb{\partial} \vec{\Psi}
	\|_{L^2([0,T])C_x^{0,\updelta_1}}
$
in Subsubsect.\,\ref{SSS:MODELMIXEDSPACETIMEFORTRANSPORT},
when we explain how to control $\| \partial W \|_{L^2([0,T])L_x^{\infty}}$.
As we describe starting two paragraphs below, 
our approach to deriving the Strichartz estimates 
is fundamentally connected to the geometry of $\gfour$-null hypersurfaces,
i.e., hypersurfaces whose normals $V$ verify $\gfour(V,V)=0$,
and in order to control the geometry of null hypersurfaces,
we use arguments that rely on having a bound for
$
\sum_{\upnu \geq 2}
		\upnu^{2 \updelta_1}
		\|
			P_{\upnu} \pmb{\partial} \vec{\Psi}
		\|_{L^2([0,T])L_x^{\infty}}^2
$.

The basic idea behind obtaining the desired bound for $\| \pmb{\partial} \Psi \|_{L^2([0,T])L_x^{\infty}}$
is to establish an appropriate Strichartz estimate for
the wave equation \eqref{E:MODELWAVE}. The analog estimate in the context of the standard flat linear wave equation
$- \partial_t^2 \varphi + \Delta \varphi = 0$ on $\mathbb{R}^{1+3}$ is the well-known Strichartz estimate
$\| \pmb{\partial} \varphi \|_{L_t^2([0,1]) L_x^{\infty}} \lesssim \| \pmb{\partial} \varphi \|_{H^{1+\varepsilon}(\Sigma_0)}$,
valid for any $\varepsilon > 0$.
As we mentioned in Subsect.\,\ref{SS:GENERALREMARKSCONNECTIONSPREVIOUSWORK},
the important work of Tataru \cites{dT2000,dT2002b} 
(see in particular \cite{dT2000}*{Theorem~6} 
and \cite{dT2002b}*{Theorem~5.1}), which
provided Strichartz estimates for linear wave equations with \emph{rough} coefficients, would in fact yield
the desired bound $\| \pmb{\partial} \Psi \|_{L^2([0,T])L_x^{\infty}} \lesssim 1$
\emph{under the \underline{stronger} assumption $\Sob > 13/6$},
\emph{provided one can simultaneously bound RHS~\eqref{E:MODELWAVE} in $\| \cdot \|_{L^{\infty}([0,T])H_x^{\Sob-1}}$},
i.e, provided one can control
$\| \Flatcurl W + \pmb{\partial} \Psi \cdot \pmb{\partial} \Psi\|_{L^{\infty}([0,T])H_x^{\Sob-1}}$.
For the model system, there is no difficulty in extending the estimate \eqref{E:MODELGRONWALL} to the
case $\Sob > 13/6$. Thus, assuming that one can also control the term
$\| \partial W \|_{L^2([0,t])L_x^{\infty}}$
on RHS~\eqref{E:MODELGRONWALL},
we obtain (using Tataru's framework) the desired bound
$\| \pmb{\partial} \Psi \|_{L^2([0,T])L_x^{\infty}} \lesssim 1$
under this stronger assumption $\Sob > 13/6$.
We stress that in the case of the compressible Euler equations, controlling the analog of the term
$\| \Flatcurl W \|_{L^{\infty}([0,T])H_x^{\Sob-1}}$ 
\emph{is} possible (see Prop.\,\ref{P:TOPORDERENERGYESTIMATES}), 
but only by exploiting the special structures of the equations of Prop.\,\ref{P:GEOMETRICWAVETRANSPORT}.
Moreover, it is not possible to achieve such control
at the classical local well-posedness level $(\varrho - \bar{\varrho},v,\Ent) \in H^{(5/2)^+}(\Sigma_0)$;
see Remark~\ref{R:REGULRITYNEEDEDFORSTRICHARTZ}.

It is known \cite{hSdT2002} that without further information about 
the principal coefficients $(\gfour^{-1})^{\alpha \beta}$ of the wave operator $\hat{\square}_{\gfour}$, 
Tataru's linear Strichartz estimates are optimal.
Thus, to achieve the goal of lowering the Sobolev regularity threshold to $\Sob > 2$,
we must exploit the specific structure of the system \eqref{E:MODELWAVE}--\eqref{E:MODELCURLEVOLUTION}.
Over the last two decades, 
a robust framework for achieving this goal
for quasilinear wave systems with a \emph{single wave speed}\footnote{By this, we mean wave equation systems featuring only one Lorentzian metric.} 
has emerged, 
starting with \cite{sKiR2003},
progressing through the results
\cites{sKiR2005d,sKiR2006a,sKiR2005b,sKiR2005c,sKiR2008,hSdT2005,qW2017},
and, in the case of the Einstein-vacuum equations,
culminating in the proof \cite{sKiRjS2015} of the bounded $L^2$ curvature conjecture.
As we will further explain below, the most significant difference between the case of single-speed quasilinear wave systems 
and the model system \eqref{E:MODELWAVE}--\eqref{E:MODELCURLEVOLUTION} is the presence of the terms on RHSs~\eqref{E:MODELWAVE}--\eqref{E:MODELCURLEVOLUTION}
that depend on one derivative of $W$.
Despite the presence of these terms, our approach here allows us to initiate the derivation of Strichartz estimates for the model system 
starting from the same crucial ingredient found in the works cited above on single-speed quasilinear wave systems: 
an outgoing acoustical eikonal function $u$, which is a solution to the following eikonal equation 
(Footnote~\ref{FN:RESCALEDEIKONAL} also applies here, i.e., 
as we describe in Subsect.\,\ref{SS:EIKONAL},
when constructing $u$, we work with a rescaled version of the acoustical metric):
\begin{align} \label{E:EIKONALINTRO}
	(\gfour^{-1})^{\alpha \beta} \partial_{\alpha} u \partial_{\beta} u =  0
\end{align}
such that $\partial_t u > 0$. 

A glaring point is that \emph{the regularity properties of $u$ are tied to those of the solution of \eqref{E:MODELWAVE}--\eqref{E:MODELCURLEVOLUTION}
through the dependence of the coefficients $(\gfour^{-1})^{\alpha \beta}$ of the eikonal equation \eqref{E:EIKONALINTRO} on $\Psi$.} 
Thus, if one studies solutions of \eqref{E:MODELWAVE}--\eqref{E:MODELCURLEVOLUTION}
using arguments that rely on estimates for $u$ and its derivatives,
one must carefully confirm that the regularity of $u$ needed for the arguments 
is compatible with that of $\Psi$.
This serious technical issue, which we further discuss below,
was first handled by Christodoulou--Klainerman \cite{dCsK1993} in their proof of the stability of Minkowski spacetime
as a solution to the Einstein--vacuum equations.
In our study of compressible Euler flow, 
we dedicate the entirety of Sect.\,\ref{S:SETUPCONSTRUCTIONOFEIKONAL}
towards the construction of an appropriate $u$ (where the role of $\gfour$ is played by the acoustical metric of Def.\,\ref{D:ACOUSTICALMETRIC}) 
and related geometric quantities,
while in Sect.\,\ref{S:ESTIMATESFOREIKONALFUNCTION}, we derive the
difficult, tensorial regularity properties of these quantities.

The level sets of $u$, denoted by $\mathcal{C}_u$,
are $\gfour$-null hypersurfaces, and in this paper, we will construct $u$ 
so that the $\mathcal{C}_u$ are outgoing sound cones; see Fig.\,\ref{F:DOMAINS}.
Through a long series of reductions, originating in \cites{dT2001,dT2002b}
and with further insights provided by \cites{sK2001,sKiR2003,hSdT2005,qW2017},
it is known that the desired Strichartz estimate 
$\| \pmb{\partial} \Psi \|_{L^2([0,T])L_x^{\infty}} \lesssim 1$
for solutions to equation \eqref{E:MODELWAVE} can be proved for $\Sob > 2$, 
thanks in part to the availability of the bound \eqref{E:MODELGRONWALL},
\emph{provided one can derive complementary, highly tensorial, Sobolev estimates for the derivatives of $u$ up to top-order, 
both along $\Sigma_t$ and along null hypersurfaces $\mathcal{C}_u$}.
We refer to this task as ``controlling the acoustic geometry,''
and our above remarks make clear that the regularity of the
acoustic geometry depends on that of $\Psi$ and $W$; 
see the discussion surrounding equation \eqref{E:RAYCHINTRO}
for further clarification of this point.
In Sect.\,\ref{S:REDUCTIONSOFSTRICHARTZ}, we review the main ideas behind deriving
the Strichartz estimate as a consequence of control of the acoustic geometry.
The basic chain of logic\footnote{In our detailed proof, we partition $[0,T]$ into appropriate subintervals
and derive estimates on each subinterval; see Subsect.\,\ref{SS:PARTITIONOFBOOTSTRAPINTERVAL}. This strategy is 
part of the series of reductions mentioned above. Here we are ignoring this technical aspect of the proof.} 
is: control over the acoustic geometry
$\implies$
an estimate for an $L^2$-type (weighted) conformal energy for solutions to $\square_{\gfour} \varphi = 0$
$\implies$
dispersive decay estimates for $\varphi$
$\implies$
(via a $\mathcal{T} \mathcal{T}^*$ argument) 
linear Strichartz estimates
$\implies$
(by Duhamel's principle, the energy estimates, and the Schauder estimates for the transport-part of the system
discussed in Subsubsect.\,\ref{SSS:MODELMIXEDSPACETIMEFORTRANSPORT}) Strichartz estimates for the quasilinear wave equation \eqref{E:COVARIANTWAVE}.

The task of controlling the acoustic geometry is quite involved and occupies the second half of the paper;
see Prop.\,\ref{P:MAINESTIMATESFOREIKONALFUNCTIONQUANTITIES} for a lengthy list of estimates that we use to
control the acoustic geometry.
In the case of quasilinear wave equations, 
many of the ideas for how to control $u$ originated in 
\cites{dCsK1993,sKiR2003,sKiR2005d,sKiR2006a,sKiR2005b,sKiR2005c,sKiR2008,qW2017}.
For the model system, the main new difficulty is the presence of
the term $\Flatcurl W$ on the right-hand side of the wave equation \eqref{E:MODELWAVE},
whose regularity properties strongly influence those of $u$; below we will elaborate on this issue.
In this subsubsection, 
we cannot hope to discuss all of the technical difficulties that
arise when controlling $u$, so we will mainly highlight a few key points that are new compared to earlier works.
Readers can consult the introduction to \cite{qW2017} for an overview 
of many of the technical difficulties that arise in the case of quasilinear
wave equations and for how they can be overcome. At the end of this subsubsection,
we will mention some of these difficulties since they occur in the present work as well.

As is standard in the theory of wave equations, 
our analysis relies on a $\gfour$-null frame $\lbrace \Lunit, \uLunit, e_1, e_2 \rbrace$
adapted to $u$, where the vectorfield $\Lunit$ is rescaled version of the gradient vectorfield of $u$,
normalized by $\Lunit t = 1$; see \eqref{E:LUNITISRESCALEDGRADIENTOFEIKONAL}.
Thus, by \eqref{E:EIKONALINTRO}, 
$\Lunit$ is null (i.e., $\gfour(\Lunit,\Lunit) = 0$),
tangent to $\mathcal{C}_u$, and orthogonal to the spheres $S_{t,u} := \mathcal{C}_u \cap \Sigma_t$.
Moreover, 
$\uLunit$ is null, 
transversal to $\mathcal{C}_u$,
orthogonal to $S_{t,u}$,
and normalized by $\uLunit t = 1$,
and $\lbrace e_A \rbrace_{A=1,2}$ are a $\gfour$-orthonormal
frame tangent to $S_{t,u}$; see Fig.\,\ref{F:NULLFRAME},
and see Sect.\,\ref{S:SETUPCONSTRUCTIONOFEIKONAL} for
details on the construction of the objects depicted in the figure.

\begin{figure}[!ht]
\centering
\includegraphics[scale=.25]{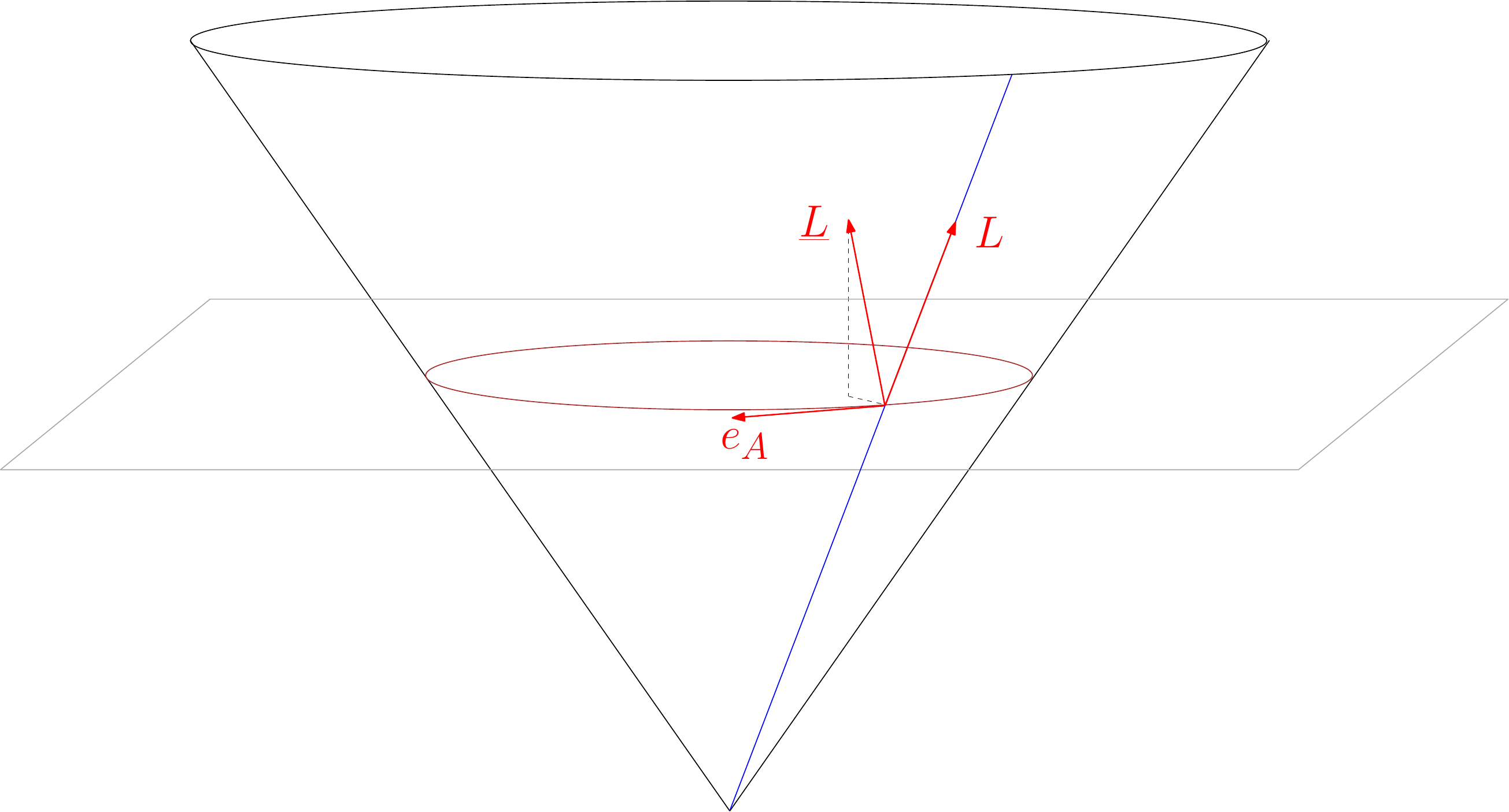}
\caption{The null frame}
\label{F:NULLFRAME}
\end{figure}

Controlling the acoustic geometry means, essentially,
deriving estimates for various connection
coefficients\footnote{These are, roughly, first $\gfour$-covariant derivatives of the frame in the directions of the frame.} 
of the null frame and their derivatives. 
There are many quantities that we need to estimate, but for brevity, 
in our discussion of the model problem, we will discuss only one of them.
Specifically, of primary importance for applications to Strichartz estimates
is the null mean curvature of the level sets of $u$ 
(i.e., of sound cones in the context of compressible Euler flow),
denoted by $\mytr_{\gsphere} \upchi$ and defined by $\mytr_{\gsphere} \upchi = \sum_{A=1}^2 \gfour(\Dfour_{e_A} \Lunit, e_A)$,
with $\Dfour$ the Levi-Civita connection of $\gfour$.
Analytically, $\mytr_{\gsphere} \upchi$ corresponds to a special combination of up-to-second-order
derivatives of $u$ with coefficients that depend, relative to Cartesian coordinates, 
on the up-to-first-order derivatives 
of $\gfour$. To bound $\mytr_{\gsphere} \upchi$, one exploits that it 
verifies \emph{Raychaudhuri's equation} (see \eqref{E:RAYCHAUDHURI} and \eqref{E:MODIFIEDRAYCHAUDHURI}),
which is an evolution equation with source terms depending on the Ricci curvature of $\gfour$.
A careful decomposition of the Ricci curvature (see Lemma~\ref{L:CURVATUREDECOMPOSITIONS})
allows one to express Raychaudhuri's equation in the form\footnote{Some of the terms denoted by ``$\cdots$'' on RHS~\eqref{E:RAYCHINTRO} 
	are important from the point of view of their $L^{\infty}$-size; we are ignoring those terms in the present discussion because we are focusing on issues
	tied to regularity.}
\begin{align} \label{E:RAYCHINTRO}
	\Lunit (\mytr_{\gsphere} \upchi + \Chfour_{\Lunit})
	& = \frac{1}{2} \Lunit^{\alpha} \Lunit^{\beta} \hat{\square}_{\gfour} (\gfour_{\alpha \beta}(\Psi))
		+ 
		\cdots,
\end{align}
where $\Chfour_{\Lunit} := \Lunit^{\alpha} \Chfour_{\alpha}$, and
$\Chfour^{\alpha} \sim (\gfour^{-1})^2 \cdot \pmb{\partial} \gfour$ is a contracted Cartesian Christoffel symbol of $\gfour$.
Here we emphasize that the regularity properties of
$\mytr_{\gsphere} \upchi + \Chfour_{\Lunit}$ are tied to those
of the source terms in the wave equation \eqref{E:MODELWAVE},
since the first term on RHS~\eqref{E:RAYCHINTRO} can be expressed via \eqref{E:MODELWAVE} and the chain rule.
It turns out that in order to obtain enough control of the acoustic geometry to prove the Strichartz estimates, 
one needs to control, among other terms, the $\mathcal{C}_u$-tangential derivatives, namely $\Lunit$ and $\angD$, of $\mytr_{\gsphere} \upchi$
in various norms along $\mathcal{C}_u$,
where $\angD$ is the Levi-Civita connection
of the Riemannian metric $\gsphere$ induced on the spheres $S_{t,u}$ by $\gfour$;
see, for example, the estimate \eqref{E:ONEANGULARDERIVATIVEOFTRCHIMODANDTRFREECHIL2INTIMEESTIMATES}.
This suggests, in view of equation \eqref{E:MODELWAVE} and the presence of the product 
$\frac{1}{2} \Lunit^{\alpha} \Lunit^{\beta} \hat{\square}_{\gfour} (\gfour_{\alpha \beta}(\Psi))$ 
on RHS~\eqref{E:RAYCHINTRO},
that we in particular have to control
$\| \angD \Flatcurl W \|_{L^2(\mathcal{C}_u)}$. In fact, one needs control of a slightly higher Lebesgue exponent
than $2$ in the angular variables to close the proof, though we will downplay this technical issue
in our simplified discussion here. For the compressible Euler equations, 
see Prop.\,\ref{P:ENERGYESTIMATESALONGNULLHYPERSURFACES} for the precise estimates 
that we need for the fluid variables along null hypersurfaces. 
We emphasize that in reality, the needed control of $\angD \Flatcurl W$ is at the top-order level 
(i.e., it relies on the assumption $\| \partial W \|_{H^{\Sob-1}(\Sigma_0)} < \infty$).
To achieve the desired control, we use two crucial structural features 
of the equations.
\begin{enumerate}
	\item $\Flatcurl W$ satisfies the transport equation \eqref{E:MODELCURLEVOLUTION}.
		Therefore, using standard energy estimates for transport equations and the energy estimate \eqref{E:MODELGRONWALL} along $\Sigma_t$
		(which can be used to obtain spacetime control of the source terms in the transport equation), 
		one can control,
		roughly,\footnote{The precise norm that we need to control
		along null hypersurfaces 
		is the one on LHS~\eqref{E:SOUNDCONEENERGYESTIMATESFORMODIFIEDVARIABLES},
		which involves Littlewood--Paley projections adapted to $\Sigma_t$.}
		$\Flatcurl W$ in $\| \cdot \|_{H^{\Sob-1}(\mathcal{H})}$ 
		along \emph{any hypersurface} $\mathcal{H}$
		that is transversal to the transport operator 
		$\partial_t 
	+
	\Psi
	\partial_1$ on LHS~\eqref{E:MODELCURLEVOLUTION}. Note that \emph{the needed estimate along $\mathcal{H}$ 
	would not be available if, instead of
	$\Flatcurl W$ on RHS~\eqref{E:MODELWAVE}, we had a generic spatial derivative $\partial W$};
	we can control \emph{generic} top-order spatial derivatives of $W$ in $L^2$ 
	\emph{only along the hypersurfaces $\Sigma_t$},
	since elliptic Hodge estimates of type \eqref{E:ELLIPTICHODGE} hold only along such hypersurfaces.
	In the compressible Euler equations, this miraculous structural feature is manifested by the fact that the principal
	transport terms on RHS~\eqref{E:COVARIANTWAVE}
	are precisely $\vec{\VortVort}$ and $\DivGradEnt$,
	which satisfy the transport equations
	\eqref{E:TRANSPORTVORTICITYVORTICITY} and \eqref{E:TRANSPORTDIVGRADENTROPY}.
	We again refer to Prop.\,\ref{P:ENERGYESTIMATESALONGNULLHYPERSURFACES} for the precise estimates
	that we derive for $\vec{\VortVort}$ and $\DivGradEnt$ along null hypersurfaces.
	\item To control the acoustic geometry, one must consider the case $\mathcal{H} := \mathcal{C}_u$,
		and thus one needs to know that
		$\partial_t 
	+
	\Psi
	\partial_1$
	is transversal to the sound cones $\mathcal{C}_u$. For the model system, 
	the transversality could be guaranteed only by making 
	assumptions on the structure of the component functions $(\gfour^{-1})^{\alpha \beta}(\Psi)$.
	However, for the compressible Euler equations, the needed transversality is guaranteed by a crucial geometric fact:
	the relevant transport vectorfield operator is $\Transport$, and it enjoys the timelike property $\gfour(\Transport,\Transport) = -1$
	(see \eqref{E:TRANSPORTISLENGTHONE}),
	thus ensuring that $\Transport$ is transversal to \emph{any} $\gfour$-null hypersurface.
\end{enumerate}

We close this subsubsection by highlighting a few key technical issues that were also present in \cite{qW2017} and related works.
\begin{itemize}
	\item
		To close our bootstrap argument, we find it convenient to partition
		the bootstrap interval
		and to work with a rescaled version of the solution adapted 
		to the partition. We define the partitioning in Subsect.\,\ref{SS:PARTITIONOFBOOTSTRAPINTERVAL}
		and the rescaling in Subsect.\,\ref{SS:RESCALEDSOLUTION}.
		Moreover, for each partition and corresponding rescaled solution, we construct
		an eikonal function adapted to that specific partition; we will ignore this 
		technical issue for the rest of this subsubsection.
	\item It turns out that the connection coefficients of the null frame
	do not satisfy PDEs that allow us to derive the desired estimates.
	Thus, one must instead work with
	a collection of ``modified'' connection coefficients that satisfy 
	better PDEs, for which we can derive the desired estimates.
	This is already apparent from equation \eqref{E:RAYCHINTRO},
	which suggests that
	$\mytr_{\gsphere} \upchi + \Chfour_{\Lunit}$
	is the ``correct'' quantity to study from the point of view of PDE analysis.
	We define these modified quantities in Subsect.\,\ref{SS:MODIFIEDQUANTITIES}.
	\item To close the proof, we need to control
		$\| \mytr_{\gsphere} \upchi + \Chfour_{\Lunit} \|_{L_t^{\infty} L_x^{\infty}}$
		via the transport equation \eqref{E:RAYCHINTRO}; see, for example,
		the estimate\footnote{The actual estimates that we need involve $\rgeo$ weights, where $\rgeo$ is defined in \eqref{E:GEOMETRICRADIAL}.
		We also note that in the bulk of the article,
		we denote $\mytr_{\gsphere} \upchi + \Chfour_{\Lunit}$ by $\mytr_{\congsphere} \widetilde{\upchi}$; 
		see \eqref{E:TRACEDRELATIONBETWEENNULLSECONDFUNDANDCONFORMALNULLSECONDFUND}.} \eqref{E:TRCHILINFINITYESTIMATES}.
		However, given the low regularity, it is not automatic that we have quantitative control of the ``data-term''
		$\| \mytr_{\gsphere} \upchi + \Chfour_{\Lunit} \|_{L^{\infty}(\Sigma_0)}$,
		as such control depends on the initial condition for $u$ (which we are free to choose).
		In Prop.\,\ref{P:INITIALFOLIATION}, we recall a result of \cite{qW2017}, 
		which shows that there exists a foliation of $\Sigma_0$ that can be used to define an initial condition for $u$
		with many good properties, leading in particular to the desired quantitative control of 
		$\| \mytr_{\gsphere} \upchi + \Chfour_{\Lunit} \|_{L^{\infty}(\Sigma_0)}$.
	\item In the proof of the conformal energy estimate from \cite{qW2017} 
		(the results of which we quote in our proof of the Strichartz estimate), 
		there is a technical part of the argument in which one needs to work with a conformally rescaled metric $e^{2 \upsigma} \gfour$, 
		constructed such that its null second 
		fundamental form has a trace equal to the quantity $\mytr_{\gsphere} \upchi + \Chfour_{\Lunit}$ highlighted above;
		we refer readers to \cite{qW2017}*{Section~1.4.1} for further discussion on this issue.
		In Subsubsect.\,\ref{SSS:CONFORMALMETRIC}, we construct the conformally rescaled metric.
		To close the conformal energy estimate, we must derive estimates for various geometric derivatives of $\upsigma$
		up to second order; see Prop.\,\ref{P:MAINESTIMATESFOREIKONALFUNCTIONQUANTITIES}.
\end{itemize}	

\subsubsection{Mixed spacetime estimates for the transport variable}
\label{SSS:MODELMIXEDSPACETIMEFORTRANSPORT}
We now discuss how to establish \eqref{E:MODELSPACETIMEBOUNDS} for the term $\| \partial W \|_{L^2([0,T])L_x^{\infty}}$
on the left-hand side.
As in Subsubsect.\,\ref{SSS:MODELWAVESTRICHARTZ}, we will 
assume the energy bound
\begin{align*}
\| (\Psi,\partial_t \Psi) \|_{H^{\Sob}(\Sigma_t) \times H^{\Sob-1}(\Sigma_t)}
	+
	\| \partial W \|_{H^{\Sob-1}(\Sigma_t)}
\lesssim 1,
\end{align*}
we will ignore the small power of $T^{\updelta}$,
and, imagining that we are carrying out a bootstrap argument, we will
assume the results of that subsubsection, i.e., we assume the bound $\| \pmb{\partial} \Psi \|_{L^2([0,T])L_x^{\infty}} \lesssim 1$.
The main idea of controlling $\| \partial W \|_{L^2([0,T])L_x^{\infty}}$
is to in fact control, for some small constant $\updelta_1 > 0$, the stronger norm\footnote{One might be tempted to 
avoid using the H\"{o}lder-based norms $\| \cdot \|_{L^2([0,T]) C_x^{0,\updelta_1}}$ and to instead
use elliptic theory to obtain control of $\| \partial W \|_{L^2([0,T])BMO_x}$.
The difficulty is that control of 
$\| \partial W \|_{L^2([0,T])BMO_x}$ is insufficient for controlling the nonlinear term
$\partial \Psi \cdot \partial W$ on RHS~\eqref{E:MODELCURLEVOLUTION} in the norm $\| \cdot \|_{L_t^2([0,T]) H_x^{\Sob}}$,
which in turn would obstruct closure of the energy estimates.}
$\| \partial W \|_{L^2([0,T]) C_x^{0,\updelta_1}}$
by combining estimates for the transport-div-curl system
\eqref{E:MODELDIV}--\eqref{E:MODELCURLEVOLUTION}
with the following standard elliptic Schauder-type estimate (see Lemma~\ref{L:SCHAUDERESTIMATESFORDIVCURLSYSTEMS}):
\begin{align} \label{E:INTROSCHAUDERESTIMATESFORDIVCURLSYSTEMS}
	\| \partial W \|_{C^{0,\updelta_1}(\mathbb{R}^3)}
	& \lesssim 
		\| \dive W \|_{C^{0,\updelta_1}(\mathbb{R}^3)}
		+
		\| \curl W \|_{C^{0,\updelta_1}(\mathbb{R}^3)}
		+
		\| \partial W \|_{L^2(\mathbb{R}^3)}.
\end{align}
It is well-known that \eqref{E:INTROSCHAUDERESTIMATESFORDIVCURLSYSTEMS} is false
when the space $C^{0,\updelta_1}(\mathbb{R}^3)$ is replaced (on both sides) 
with $L^{\infty}(\mathbb{R}^3)$; this explains our reliance on H\"{o}lder norms. 
To control RHS~\eqref{E:INTROSCHAUDERESTIMATESFORDIVCURLSYSTEMS}, 
we will use the following important fact, mentioned already in the first paragraph of Subsubsect.\,\ref{SSS:MODELWAVESTRICHARTZ}: 
\emph{the Strichartz estimate
$\| \pmb{\partial} \Psi \|_{L^2([0,T])L_x^{\infty}} \lesssim 1$
can be slightly strengthened, under the scope of our approach,
to $\| \pmb{\partial} \Psi \|_{L^2([0,T])C_x^{0,\updelta_1}} \lesssim 1$}; 
see Cor.\,\ref{C:HOLDERTYPESTRICHARTZESTIMATEFORWAVEVARIABLES}.
To proceed, we take the norm $\| \cdot \|_{C^{0,\updelta_1}(\Sigma_t)}$ of the transport equation \eqref{E:MODELCURLEVOLUTION}
and integrate in time,
use \eqref{E:INTROSCHAUDERESTIMATESFORDIVCURLSYSTEMS}
to bound the source term factor $\partial W$ on RHS~\eqref{E:MODELCURLEVOLUTION},
use \eqref{E:MODELDIV} to substitute for the first term on RHS~\eqref{E:INTROSCHAUDERESTIMATESFORDIVCURLSYSTEMS},
and use the strengthened Strichartz estimate $\| \pmb{\partial} \Psi \|_{L^2([0,T])C_x^{0,\updelta_1}} \lesssim 1$
(which in particular, as the arguments of Lemma~\ref{L:ESTIMATESFORTRANSPORTCHARACTERISTICS} show, 
yields control of the integral curves of the transport operator
$
\partial_t 
	+
	\Psi
	\partial_1
$
on LHS~\eqref{E:MODELCURLEVOLUTION})
to obtain the following estimate 
(see Subsect.\,\ref{SS:PROOFOFPROPMIXEDSPACETIMEHOLDERESTIMATESFORVORTICITYANDENTROPY} for the details):
\begin{align} \label{E:MODELHOLDERBOUND}
	\| \Flatcurl W \|_{C^{0,\updelta_1}(\Sigma_t)}
	& \lesssim
		\| \Flatcurl W \|_{C^{0,\updelta_1}(\Sigma_0)}
		+
		\mbox{\upshape data}
		+
		\int_0^t
			\| \pmb{\partial} \Psi \|_{C^{0,\updelta_1}(\Sigma_{\uptau})}
			\| \Flatcurl W \|_{C^{0,\updelta_1}(\Sigma_{\uptau})}
		\, d \uptau.
\end{align}
To control the first term on RHS~\eqref{E:MODELHOLDERBOUND},
we need to assume that
$\| \partial W \|_{C^{0,\upalpha}(\Sigma_0)} < \infty$,
for some $\upalpha  > 0$ (and then $\updelta_1 > 0$ is chosen to be $\leq \upalpha$).
There seems to be no way to avoid this assumption by the method we are using since
transport equation solutions do not gain regularity or satisfy Strichartz estimates (which are tied to dispersion).
From \eqref{E:MODELHOLDERBOUND}, 
Gr\"{o}nwall's inequality,
and the bound
$\| \pmb{\partial} \Psi \|_{L^2([0,T])C_x^{0,\updelta_1}} \lesssim 1$,
we find that
\begin{align*}
\| \Flatcurl W \|_{C^{0,\updelta_1}(\Sigma_t)}
\lesssim
\| \Flatcurl W \|_{C^{0,\updelta_1}(\Sigma_0)}
+ 
\mbox{\upshape data}.
\end{align*}
From this bound,
\eqref{E:INTROSCHAUDERESTIMATESFORDIVCURLSYSTEMS},
equation \eqref{E:MODELDIV} (which we again use to substitute for the first term on RHS~\eqref{E:INTROSCHAUDERESTIMATESFORDIVCURLSYSTEMS}),
and the assumed energy bound,
we find that
\begin{align*}
\| \partial W \|_{C^{0,\updelta_1}(\Sigma_t)}
\lesssim
\| \Flatcurl W \|_{C^{0,\updelta_1}(\Sigma_0)}
+
\mbox{\upshape data}
+
\| \pmb{\partial} \Psi \|_{C^{0,\updelta_1}(\Sigma_t)}.
\end{align*}
Finally, squaring this estimate,
integrating in time, 
and again using the bound $\| \pmb{\partial} \Psi \|_{L^2([0,T])C_x^{0,\updelta_1}} \lesssim 1$,
we obtain the desired bound
$\| \partial W \|_{L^2([0,T])C_x^{0,\updelta_1}} 
\lesssim \| \Flatcurl W \|_{C^{0,\updelta_1}(\Sigma_0)}
+
\mbox{\upshape data}$.

We have therefore sketched how to establish \eqref{E:MODELSPACETIMEBOUNDS} which, in view of \eqref{E:MODELGRONWALL}, 
justifies (for $t$ sufficiently small) the fundamental estimate
$
\| (\Psi,\partial_t \Psi) \|_{H^{\Sob}(\Sigma_t) \times H^{\Sob-1}(\Sigma_t)}
	+
	\| \partial W \|_{H^{\Sob-1}(\Sigma_t)}
\leq
		\mbox{\upshape data}
$.

\section{Littlewood--Paley projections, standard norms, parameters, assumptions on the initial data, bootstrap assumptions, and notation regarding constants}
\label{S:DATAANDBOOTSTRAPASSUMPTION}
In this section, we define the standard Littlewood--Paley projections,
define various norms and parameters that we use in our analysis,
state our assumption on the initial data,
formulate the bootstrap assumptions 
that we use in proving Theorem~\ref{T:MAINTHEOREMROUGHVERSION},
and state our conventions for constants $C$.

\subsection{Littlewood--Paley projections}
\label{SS:LITTLEWOODPALEY}
We fix a smooth function $\upeta:\mathbb{R}^3 \rightarrow [0,1]$ supported on the frequency-space annulus 
$\lbrace \xi \in \mathbb{R}^3 \ | \ 1/2 \leq |\xi| \leq 2 \rbrace$
such that for $\xi \neq 0$, we have $\sum_{k \in \mathbb{Z}} \upeta(2^k \xi) = 1$. For dyadic frequencies
$\uplambda = 2^k$ with $k \in \mathbb{Z}$, we define the standard Littlewood--Paley projection $P_{\uplambda}$, which
acts on scalar functions $F:\mathbb{R}^3 \rightarrow \mathbb{C}$, as follows:
\begin{align} \label{E:LITTLEWOODPALEYPROJECTION}
	P_{\uplambda} F(x)
	& := 
	\frac{1}{(2 \pi)^3}
	\int_{\mathbb{R}^3}
		e^{i x \cdot \xi} \upeta(\uplambda^{-1} \xi) \hat{F}(\xi)
	\, d \xi,
\end{align}
where
$
\hat{F}(\xi)
:=
	\int_{\mathbb{R}^3}
		e^{-i x \cdot \xi} F(x)
	\, d x
$ 
(with $dx := dx^1 dx^2 dx^3$)
is the Fourier transform of $F$.
If $F$ is an array-valued function, then $P_{\uplambda} F$ denotes the array of
projections of its components.
If $I \subset 2^{\mathbb{Z}}$ is an interval of dyadic frequencies, then
$P_I F := \sum_{\nu \in I} P_{\upnu} F$, and 
$P_{\leq \uplambda} F := P_{(-\infty,\uplambda]} F$.

If $F$ is a function on $\Sigma_t$, then 
$P_{\uplambda} F(t,x) := P_{\uplambda} G(x)$, where $G(x) := F(t,x)$,
and similarly for $P_I F(t,x)$ and $P_{\leq \uplambda} F(t,x)$.

\subsection{Norms and seminorms}
\label{SS:NORMS}
In this subsection, we define some standard norms and seminorms
that we will use in the first part of the paper,
before we control the acoustic geometry.
To control the acoustic geometry,
we will use additional norms, defined in
Subsect.\,\ref{SS:GEOMETRICNORMS}.

For scalar- or array-valued functions $F$  
and
$1 \leq q < \infty$,
$
\| F \|_{L^q(\Sigma_t)} 
:= 
\left\lbrace
\int_{\Sigma_t} |F(t,x)|^q \, dx
\right\rbrace^{1/q}
$
and
$
\| F \|_{L^{\infty}(\Sigma_t)}
:=
\mbox{ess sup}_{x \in \mathbb{R}^3}
|F(t,x)|
$
are standard Lebesgue norms of $F$,
where we recall that $\Sigma_t$ is the standard constant-time slice.
Lebesgue norms on subsets $D \subset \Sigma_t$
are defined in an analogous fashion, e.g.,
if $D = \lbrace t \rbrace \times D'$,
then
$
\| F \|_{L^q(D)} 
:= 
\left\lbrace
\int_{D'} |F(t,x)|^q \, dx
\right\rbrace^{1/q}
$.
Similarly, if $\lbrace A_{\uplambda} \rbrace_{\uplambda \in 2^{\mathbb{N}}}$
is a dyadic-indexed sequence of real numbers and $1 \leq q < \infty$, 
then
$
\|
	A_{\upnu}
\|_{\ell_{\upnu}^q}
:=
\left\lbrace
\sum_{\upnu \geq 1}
	A_{\upnu}^q
\right\rbrace^{1/q}
$.

We will rely on the following
family of seminorms, parameterized by real numbers $M$ 
(where we will have $M > 0$ in our applications below):
\begin{align} \label{E:LAMBDAL2NORM}
	\| 
		\Lambda^M F
	\|_{L^2(\Sigma_t)}
	& := 
			\sqrt{
			\sum_{\upnu \geq 2} 
			\upnu^{2M} \| P_{\upnu} F \|_{L^2(\Sigma_t)}^2},
\end{align}
where on RHS~\eqref{E:LAMBDAL2NORM} and throughout, 
sums involving Littlewood--Paley projections are understood to be dyadic sums.

For real numbers $M \geq 0$, we define the following standard Sobolev norm 
for functions $F$ on $\Sigma_t$:
\begin{align}	 \label{E:STANDARDSOB}
		\| F \|_{H^M(\Sigma_t)}
		& :=
			\left\lbrace
			\| P_{\leq 1} F \|_{L^2(\Sigma_t)}^2
		+
		\| 
			\Lambda^M F
		\|_{L^2(\Sigma_t)}^2
		\right\rbrace^{1/2}.
\end{align}
Throughout, we will rely on the standard fact that when $M$ is an integer,
the norm defined in \eqref{E:STANDARDSOB}
is equivalent to
$
\sum_{|\vec{I}| \leq M} \| \partial_{\vec{I}} F \|_{L^2(\Sigma_t)}
$,
where $\vec{I}$ are spatial derivative multi-indices.

If $F$ is a function defined on a subset $D \subset \mathbb{R}^3$ and $\upbeta \geq 0$, 
then we define the H\"{o}lder norm $\| \cdot \|_{C^{0,\upbeta}(D)}$ of $F$ as follows:
\begin{align} \label{E:HOLDERNORMDEF}
	\| F \|_{C^{0,\upbeta}(D)}
	& := 
		\| F \|_{L^{\infty}(D)}
		+
		\sup_{x,y \in D, 0 < |x - y|}
		\frac{|F(x) - F(y)|}{|x-y|^{\upbeta}}.
\end{align}
Similarly, if $F$ is a function defined on a subset $D \subset \Sigma_t$
of the form $D = \lbrace t \rbrace \times D'$, then
$\| F \|_{C^{0,\upbeta}(D)} := \| G \|_{C^{0,\upbeta}(D')}$,
where $G(x) := F(t,x)$.

We will also use the following mixed norms for functions $F$ defined on $\mathbb{R}^{1+3}$, 
where $1 \leq q_1 < \infty$, $1 \leq q_2 \leq \infty$,
and $I$ is an interval of time:
\begin{subequations}
\begin{align}
	\| F \|_{L^{q_1}(I) L_x^{q_2}}
	& := 
	\left\lbrace
		\int_I \| F \|_{L^{q_2}(\Sigma_{\uptau})}^{q_1} \, d \uptau
	\right\rbrace^{1/q_1},
	&
	\| F \|_{L^{\infty}(I) L_x^{q_2}}
	& := 
		\mbox{ess sup}_{\uptau \in I} \| F \|_{L^{q_2}(\Sigma_{\uptau})},
			\label{E:MIXEDLEBESGUENORMS} \\
\| F \|_{L^{q_1}(I) C_x^{0,\upbeta}}
	& := 
	\left\lbrace
		\int_I \| F \|_{C^{0,\upbeta}(\Sigma_{\uptau})}^{q_1} \, d \uptau
	\right\rbrace^{1/q_1},
	&
	\| F \|_{L^{\infty}(I) C_x^{0,\upbeta}}
	& := 
		\mbox{ess sup}_{\uptau \in I} \| F \|_{C^{0,\upbeta}(\Sigma_{\uptau})}.
		\label{E:MIXEDHOLDERNORM}
\end{align}
\end{subequations}
Similarly, if $\lbrace F_{\uplambda} \rbrace_{\uplambda \in 2^{\mathbb{N}}}$
is a dyadic-indexed sequence of functions $F_{\uplambda}$ on $\Sigma_t$,
then
\begin{align} \label{E:MIXEDBIGLITTLEL2NORM}
	\|
		F_{\upnu}
	\|_{\ell_{\upnu}^2 L^2(\Sigma_t)}
	& 
	:=
	\left\lbrace
	\sum_{\upnu \geq 1}
	\|
		F_{\upnu}
	\|_{L^2(\Sigma_t)}^2
	\right\rbrace^{1/2}.
\end{align}

\subsection{Choice of parameters}
\label{SS:PARAMETERS}
In this subsection, we introduce the parameters that will play a role in our analysis.
We recall that $2 < \Sob \leq 5/2$ and $0 < \upalpha < 1$ denote given real numbers
corresponding, respectively, to the assumed Sobolev regularity of the data and
the assumed H\"{o}lder regularity of the transport
part of the data; see \eqref{E:FINITEWAVEDATANORM}--\eqref{E:FINITETRANSPORTDATANORM}.
We then choose positive numbers
$q$, $\upepsilon_0$, $\updelta_0$, $\updelta$, 
and
$\updelta_1$ 
that satisfy the following conditions:
\begin{subequations}
\begin{align}
	2 & < q < \infty,
		\label{E:QCONSTRAINT} \\
	0 & < \upepsilon_0 := \frac{\Sob - 2}{10} < \frac{1}{10},
		\label{E:EPSILON0INEQUALITY} \\
	\updelta_0 & := \min\left\lbrace \upepsilon_0^2, \frac{\upalpha}{10} \right\rbrace,
		\label{E:DELTA0DEF} \\
	0 & < \updelta := \frac{1}{2} - \frac{1}{q}
		< \upepsilon_0,
			\label{E:DELTADEFINITIONANDPOSITIVITY} \\
	\updelta_1 
	& := \min\left\lbrace \Sob - 2 - 4 \upepsilon_0 - \updelta(1 - 8 \upepsilon_0), \upalpha \right\rbrace
		>  
		8 \updelta_0 > 0.
		\label{E:DELTA1DEFINITIONANDPOSITIVITY}
\end{align}
\end{subequations}
More precisely, we consider $\Sob$, $\upalpha$,
$\upepsilon_0$,
and
$\updelta_0$
to be fixed throughout the paper, 
while
in some of our arguments below, we will treat $q$, $\updelta$, and $\updelta_1$ as parameters,
where $q > 2$ will need to be chosen to be sufficiently close to $2$ 
(i.e., $\updelta > 0$ will need to be chosen to be sufficiently small).

\subsection{Assumptions on the initial data}
\label{SS:DATAASSUMPTIONS}
The following definition captures the subset of solution space
in which the compressible Euler equations are hyperbolic
in a non-degenerate sense.

\begin{definition}[Regime of hyperbolicity] \label{D:REGIMEOFHYPERBOLICITY}
	We define $\mathcal{K}$ as follows, where $\Speed(\LogDensity,\Ent)$ is the speed of sound:
	\begin{align} \label{E:REGIMEOFHYPERBOLICITY}
		\mathcal{K}
		& := 
		\left\lbrace
			(\LogDensity,\Ent,\vec{v},\vec{\vortrenormalized},\vec{\GradEnt})
			\in
			\mathbb{R} \times \mathbb{R} \times \mathbb{R}^3 \times \mathbb{R}^3 \times \mathbb{R}^3
			\ | \
			0 < \Speed(\LogDensity,\Ent) < \infty
		\right\rbrace.
	\end{align}
\end{definition}

We set 
\begin{align} \label{E:DATAFUNCTIONS}
	(\mathring{\LogDensity},\mathring{\Ent},\mathring{\vec{v}},\mathring{\vec{\vortrenormalized}},\mathring{\vec{\GradEnt}},\mathring{\vec{\VortVort}},
	\mathring{\DivGradEnt}) 
	& := (\LogDensity,\Ent,\vec{v},\vec{\vortrenormalized},\vec{\GradEnt},\vec{\VortVort},\DivGradEnt)|_{\Sigma_0}.
\end{align}

With $\Sob$ and $\upalpha$ as in Subsect.\,\ref{SS:PARAMETERS},
 we assume that
\begin{subequations}
\begin{align} \label{E:FINITEWAVEDATANORM}
	\| \mathring{\LogDensity} \|_{H^{\Sob}(\Sigma_0)}
	+
	\| \mathring{\vec{v}} \|_{H^{\Sob}(\Sigma_0)}
		< \infty,
		\\
	\| \mathring{\vec{\vortrenormalized}} \|_{H^{\Sob}(\Sigma_0)}
	+
	\| \mathring{\Ent} \|_{H^{\Sob+1}(\Sigma_0)}
	+
	\| \mathring{\vec{\VortVort}} \|_{C^{0,\upalpha}(\Sigma_0)}
	+
	\| \mathring{\DivGradEnt} \|_{C^{0,\upalpha}(\Sigma_0)}
	<
	\infty.
	\label{E:FINITETRANSPORTDATANORM}
\end{align}
\end{subequations}
\eqref{E:FINITEWAVEDATANORM} corresponds to ``rough'' regularity assumptions on the wave-part of the data,
while \eqref{E:FINITETRANSPORTDATANORM}
corresponds to regularity assumptions on the transport-part of the data.

Let $\mbox{\upshape int} U$ denote the interior of the set $U$.
We assume that there are compact subsets $\mathring{\mathfrak{K}}$
and 
$\mathfrak{K}$
of
$\mbox{\upshape int} \mathcal{K}$ such that
\begin{align} \label{E:SETSDESCRIBINGINTETERIOROFREGIMEOFHYPERBOLICITY}
	(\mathring{\LogDensity},\mathring{\Ent}, \mathring{\vec{v}},\mathring{\vec{\vortrenormalized}},\mathring{\vec{\GradEnt}})(\mathbb{R}^3)
	\subset 
	\mbox{\upshape int} \mathring{\mathfrak{K}}
	\subset
	\mathring{\mathfrak{K}}
	\subset
	\mbox{\upshape int} \mathfrak{K}
	\subset
	\mathfrak{K}
	\subset 
	\mbox{\upshape int} \mathcal{K}.
\end{align}

\subsection{Bootstrap assumptions}
\label{SS:BOOT}
For the rest of the article, 
 $0 < \Tboot \ll 1$ denotes a ``bootstrap time''
that we will choose to be sufficiently small in a manner that depends only on
the quantities introduced in Subsect.\,\ref{SS:DATAASSUMPTIONS}.
We assume that
$(\LogDensity,\Ent,\vec{v},\vec{\vortrenormalized},\vec{\GradEnt})$
is a smooth (see Footnote~\ref{FN:SMOOTHNESS}) solution
to the equations of Prop.\,\ref{P:GEOMETRICWAVETRANSPORT}
on the ``bootstrap slab'' $[0,\Tboot] \times \mathbb{R}^3$.

\subsubsection{Bootstrap assumptions tied to $\mathcal{K}$.}
\label{SSS:HYPERBOLICITYBOOTSTRAP}
Let $\mathfrak{K}$ be the subset from Subsect.\,\ref{SS:DATAASSUMPTIONS}.
We assume that 
\begin{align} \label{E:BOOTSOLUTIONDOESNOTESCAPEREGIMEOFHYPERBOLICITY}
	(\LogDensity,\Ent,\vec{v},\vec{\vortrenormalized},\vec{\GradEnt})([0,\Tboot] \times \mathbb{R}^3)
	\subset \mathfrak{K}.
\end{align}
In Cor.\,\ref{C:IMPROVEMENTOFLINFINITYBOOTSTRAP}, we derive a strict improvement of \eqref{E:BOOTSOLUTIONDOESNOTESCAPEREGIMEOFHYPERBOLICITY}.

\begin{remark}[Uniform $L^{\infty}(\Sigma_t)$ bounds]
\label{R:UNIFORMLINFTYBOUNDS}
Note that the bootstrap assumption \eqref{E:BOOTSOLUTIONDOESNOTESCAPEREGIMEOFHYPERBOLICITY} implies, 
in particular,
uniform $L^{\infty}(\Sigma_t)$ bounds, depending on $\mathfrak{K}$, 
for
$\LogDensity$, $\Ent$, $\vec{v}$, $\vec{\vortrenormalized}$,
and $\vec{\GradEnt} \sim \partial \Ent$. Throughout the article,
we will often use these simple $L^{\infty}(\Sigma_t)$ bounds
without explicitly mentioning that we are doing so.
\end{remark}

\subsubsection{Mixed spacetime norm bootstrap assumptions}
\label{SSS:MIXEDSPACETIMENORMBOOTSTRAP}
We assume that the following estimates hold:
\begin{subequations}
\begin{align} \label{E:BOOTSTRICHARTZ}
		\|
			\pmb{\partial} \vec{\Psi}
		\|_{L^2_t([0,\Tboot])L_x^{\infty}}^2
		+
		\sum_{\upnu \geq 2}
		\upnu^{2 \updelta_0}
		\|
			P_{\upnu} \pmb{\partial} \vec{\Psi}
		\|_{L^2_t([0,\Tboot])L_x^{\infty}}^2
	& \leq 1,
	\\
	\|
		\partial (\vec{\vortrenormalized},\vec{\GradEnt})
	\|_{L^2_t([0,\Tboot])L_x^{\infty}}^2
	+
	\sum_{\upnu \geq 2}
	\upnu^{2 \updelta_0}
	\|
		P_{\upnu} \partial (\vec{\vortrenormalized},\vec{\GradEnt})
	\|_{L^2_t([0,\Tboot])L_x^{\infty}}^2
	& \leq 1.
	\label{E:BOOTL2LINFINITYFIRSTDERIVATIVESOFVORTICITYBOOTANDNENTROPYGRADIENT}
\end{align}
\end{subequations}
In Theorem~\ref{T:IMPROVEMENTOFSTRICHARTZBOOTSTRAPASSUMPTION}, 
we derive a strict improvement of \eqref{E:BOOTSTRICHARTZ}.
In Theorem~\ref{T:MIXEDSPACETIMEHOLDERESTIMATESFORVORTICITYANDENTROPY},
we derive a strict improvement of \eqref{E:BOOTL2LINFINITYFIRSTDERIVATIVESOFVORTICITYBOOTANDNENTROPYGRADIENT}.

\begin{remark}
	When deriving the energy estimates, we will only use the bounds for 
	$
	\|
			\pmb{\partial} \vec{\Psi}
		\|_{L^2_t([0,\Tboot])L_x^{\infty}}
	$
	and
	$
	\|
		\partial (\vec{\vortrenormalized},\vec{\GradEnt})
	\|_{L^2_t([0,\Tboot])L_x^{\infty}}
	$.
	We use the bounds for the two sums in \eqref{E:BOOTSTRICHARTZ}--\eqref{E:BOOTL2LINFINITYFIRSTDERIVATIVESOFVORTICITYBOOTANDNENTROPYGRADIENT}
	to obtain control over the acoustic geometry, that is, for proving Prop.\,\ref{P:MAINESTIMATESFOREIKONALFUNCTIONQUANTITIES}. 
In turn, such control over the acoustic geometry will allow us to prove a frequency-localized Strichartz estimate 
(Theorem~\ref{T:FREQUENCYLOCALIZEDSTRICHARTZ}), and then to improve the Strichartz-type assumption for the wave variables 
(i.e., to prove Theorem~\ref{T:IMPROVEMENTOFSTRICHARTZBOOTSTRAPASSUMPTION}). For more details about this strategy, 
we refer to Subsubsect.\,\ref{SSS:MODELWAVESTRICHARTZ}. 
\end{remark}

\subsection{Notation regarding constants}
\label{SS:NOTATION}
In the rest of the paper, $C > 0$ denotes a constant that is free to vary from line to line.
$C$ is allowed to depend on
$\Sob$, $\upalpha$, the parameters from Subsect.\,\ref{SS:PARAMETERS},
the norms of the data from Subsect.\,\ref{SS:DATAASSUMPTIONS},
and the set $\mathfrak{K}$ from Subsect.\,\ref{SS:DATAASSUMPTIONS}.
We often bound explicit functions of $t$ by $\leq C$ since $t \leq \Tboot \ll 1$.
For given quantities $A,B \geq 0$,
write $A \lesssim B$ to mean that there exists a $C > 0$ such that $A \leq C B$.
We write $A \approx B$ to mean that $A \lesssim B$ and $B \lesssim A$.

\section{Preliminary energy and elliptic estimates}
\label{S:PRELIMINARYENERGYANDELLIPTICESTIMATES}
Our main goal in this section is to prove preliminary energy and elliptic
estimates that yield $H^2(\Sigma_t)$-control of the velocity, density, and specific vorticity, and $H^3(\Sigma_t)$-control of the entropy. 
The main result is provided by Prop.\,\ref{P:PRELIMINARYENERGYANDELLIPTICESTIMATES}.
These preliminary below-top-order estimates are useful, 
in the context of controlling the solution's top-order derivatives,
for handling all but the most difficult error terms.
The proof of Prop.\,\ref{P:PRELIMINARYENERGYANDELLIPTICESTIMATES} is located in
Subsect.\,\ref{SS:PROOFOFPRELIMINARYENERGYANDELLIPTICESTIMATES}.
Before proving the proposition, we first provide two standard ingredients:
the geometric energy method for wave equations and transport equations,
and estimates in $L^2(\Sigma_t)$-based spaces for $\dive$-$\curl$ systems.

\begin{proposition}[Preliminary energy and elliptic estimates]
\label{P:PRELIMINARYENERGYANDELLIPTICESTIMATES} 
There exists a continuous strictly increasing function $\contfunction: [0,\infty) \rightarrow [0,\infty)$
such that under the initial data and bootstrap assumptions of Sect.\,\ref{S:DATAANDBOOTSTRAPASSUMPTION}, 
smooth solutions to the compressible Euler equations
satisfy the following estimates for $t \in [0,\Tboot]$:
\begin{align}
	&
	\sum_{k=0}^2
	\| \partial_t^k (\LogDensity,\vec{v},\vec{\vortrenormalized}) \|_{H^{2-k}(\Sigma_t)}
	+
	\sum_{k=0}^2
	\| \partial_t^k \Ent \|_{H^{3-k}(\Sigma_t)}
	+
	\sum_{k=0}^1
	\| \partial_t^k (\vec{\VortVort},\DivGradEnt) \|_{H^{1-k}(\Sigma_t)}
		\label{E:BASICH2ENERGYESTIMATE} \\
	& \leq
	\contfunction\left(
	\| (\LogDensity, \vec{v}, \vec{\vortrenormalized}) \|_{H^2(\Sigma_0)}
	+
	\| \Ent \|_{H^3(\Sigma_0)}
	\right).
	\notag
\end{align}

Moreover, for any $a$ and $b$ with $0 \leq a \leq b \leq \Tboot$, 
solutions $\varphi$ to the inhomogeneous wave equation 
\begin{align} \label{E:INHOMOGENEOUSLINEARWAVE}
\square_{g(\vec{\Psi})} \varphi 
& = \mathfrak{F}
\end{align}
satisfy the following estimate:
\begin{align} \label{E:PRELIMINARYH1DOTHOMOGENEOUSWAVEEQUATIONINHOMOGENEOUSESTIMATE}
	\| \pmb{\partial} \varphi \|_{L^2(\Sigma_b)}
	& \lesssim \| \pmb{\partial} \varphi \|_{L^2(\Sigma_a)}
		+
		\| \mathfrak{F} \|_{L^1([a,b]) L_x^2}.
\end{align}

\end{proposition}

\subsection{The geometric energy method for wave equations}
\label{SS:VECTORFIELDMULTIPLIERMETHOD}
To derive energy estimates for solutions to the wave equations in \eqref{E:COVARIANTWAVE},
we will use the well-known vectorfield multiplier method.
In this subsection, we set up this geometric energy method.
Throughout this subsection, we lower and raise Greek indices with 
the acoustical metric $\gfour=\gfour(\vec{\Psi})$ from Def.\,\ref{D:ACOUSTICALMETRIC} and its inverse.
Moreover, we recall that $\Dfour$ denotes the Levi-Civita connection of $\gfour$
and $\square_{\gfour} := (\gfour^{-1})^{\alpha \beta} \Dfour_{\alpha} \Dfour_{\beta}$
denotes the corresponding covariant wave operator.

\subsubsection{Energy-momentum tensor, energy current, and deformation tensor}
\label{SSS:ENERGYMOMENTUM}
We define the energy-momentum tensor associated to a scalar function $\varphi$
to be the following symmetric type $\binom{0}{2}$ tensorfield:
\begin{align} \label{E:ENMOMENTUMTENSOR}
	\enmomem_{\alpha \beta}[\varphi]
	& := \partial_{\alpha} \varphi \partial_{\beta} \varphi
		- 
		\frac{1}{2} \gfour_{\alpha \beta} (\gfour^{-1})^{\kappa \lambda} \partial_{\kappa} \varphi \partial_{\lambda} \varphi.
\end{align}
Given $\varphi$ and any ``multiplier'' vectorfield $\mathbf{X}$,
we define the corresponding energy current
$\Jenarg{\mathbf{X}}{\alpha}[\varphi]$ vectorfield:
\begin{align} \label{E:MULTIPLIERVECTORFIELD}
	\Jenarg{\mathbf{X}}{\alpha}[\varphi]
	& := \enmomem^{\alpha \beta}[\varphi] 
		\mathbf{X}_{\beta}.
\end{align}
We define the deformation tensor of $\mathbf{X}$ to be the following
symmetric type $\binom{0}{2}$ tensorfield:
\begin{align} \label{E:DEFORMATIONTENSOR}
\deformarg{\mathbf{X}}{\alpha}{\beta}
& := \Dfour_{\alpha} \mathbf{X}_{\beta}	
		+
		\Dfour_{\beta} \mathbf{X}_{\alpha}.
\end{align}

A straightforward computation yields the following
identity, which will form the starting point for our energy estimates for 
the wave equations:
\begin{align} \label{E:DIVERGENCEOFENERGYCURRENT}
	\Dfour_{\kappa} \Jenarg{\mathbf{X}}{\kappa}[\varphi]
	& = (\square_{\gfour} \varphi) \mathbf{X} \varphi
			+
			\frac{1}{2} 
			\enmomem^{\kappa \lambda} \deformarg{\mathbf{X}}{\kappa}{\lambda}.
\end{align}

\subsubsection{The basic energy along $\Sigma_t$}
\label{SSS:ENERGYALONGSIGMAT}
To derive energy estimates for solutions $\varphi$ to wave equations $\square_{\gfour} \varphi = \mathfrak{F}$,
we will rely on the following energy $\mathbb{E}[\varphi](t)$,
where $\Transport = \partial_t + v^a \partial_a$ is the material derivative vectorfield:
\begin{align} \label{E:BASICENERGYDEF}
	\mathbb{E}[\varphi](t)
	& := \int_{\Sigma_t}
					\left\lbrace
						\Jenarg{\Transport}{\kappa}[\varphi] \Transport_{\kappa}
						+ 
						\varphi^2
					\right\rbrace
			 \, d \varpi_g
			= \int_{\Sigma_t}
					\left\lbrace
						\enmomem^{00}[\varphi] 
						+
						\varphi^2
					\right\rbrace
			 \, d \varpi_g.
\end{align}
In \eqref{E:BASICENERGYDEF} and throughout,
$d \varpi_g$ is the volume form induced on $\Sigma_t$ by the first fundamental form $g$
of $\gfour$. A straightforward computation yields that relative to the Cartesian coordinates, 
we have 
\begin{align} \label{E:SPACEVOLUMEFORMRELATIVETOCARTESIAN}
	d \varpi_g 
	& 
	= \sqrt{\mbox{\upshape det} g} dx^1 dx^2 dx^3
	=
	\Speed^{-3} dx^1 dx^2 dx^3.
\end{align}
Also, \eqref{E:TRANSPORTISLENGTHONE} implies that $\Transport$ is timelike with respect to $\gfour$.
This leads to the coercivity of $\mathbb{E}[\varphi](t)$, as we show in the next lemma.

\begin{lemma}[Coerciveness of {$\mathbb{E}[\varphi](t)$}]
	\label{L:CORECIVENESSOFBASICENERGY}
	Under the bootstrap assumptions of Sect.\,\ref{S:DATAANDBOOTSTRAPASSUMPTION}, 
	the following estimate holds for $t \in [0,\Tboot]$:
	\begin{align} \label{E:CORECIVENESSOFBASICENERGY}
		\mathbb{E}[\varphi](t)
		& \approx \| (\varphi, \partial_t \varphi) \|_{H^1(\Sigma_t) \times L^2(\Sigma_t)}^2.
	\end{align}
\end{lemma}

\begin{proof}
	Since the bootstrap assumption \eqref{E:BOOTSOLUTIONDOESNOTESCAPEREGIMEOFHYPERBOLICITY} guarantees that the
	solution is contained in $\mathfrak{K}$, we have $\Speed \approx 1$
	and thus, by \eqref{E:SPACEVOLUMEFORMRELATIVETOCARTESIAN}, 
	$d \varpi_g = \Speed^{-3} dx^1 dx^2 dx^3 \approx dx^1 dx^2 dx^3$.
	Next, using \eqref{E:INVERSEACOUSTICALMETRIC}, 
	\eqref{E:TRANSPORTISLENGTHONE},
	\eqref{E:ENMOMENTUMTENSOR},
	and
	\eqref{E:MULTIPLIERVECTORFIELD},
	we compute that
	$\Jenarg{\Transport}{\kappa}[\varphi] \Transport_{\kappa}
	= \frac{1}{2}(\Transport \varphi)^2 + \frac{1}{2} \Speed^2 \updelta^{ab} \partial_a \varphi \partial_b \varphi 
	$. Using that $\Transport \varphi = \partial_t \varphi + v^a \partial_a \varphi$,
	that $|\vec{v}|$ is uniformly bounded for solutions contained in $\mathfrak{K}$,
	and that $\Speed \approx 1$,
	and applying Young's inequality to the cross term $2 (\partial_t \varphi)(v^a \partial_a \varphi)$
	in $(\Transport \varphi)^2$,
	we deduce 
	$
	\Jenarg{\Transport}{\kappa}[\varphi] \Transport_{\kappa}
	\approx |\pmb{\partial} \varphi|^2
	$.
	From these estimates and definition \eqref{E:BASICENERGYDEF}, 
	the desired estimate \eqref{E:CORECIVENESSOFBASICENERGY} easily follows.
	
\end{proof}

In the next lemma, we provide the basic energy inequality that we will
use when deriving energy estimates for solutions to the wave equations.

\begin{lemma}[Basic energy inequality for the wave equations]
	\label{L:BASICENERGYINEQUALITYFORWAVEEQUATIONS}
	Let $\varphi$ be smooth on $[0,\Tboot] \times \mathbb{R}^3$. 
	Under the bootstrap assumptions of Sect.\,\ref{S:DATAANDBOOTSTRAPASSUMPTION}, 
	the following inequality holds for $t \in [0,\Tboot]$:
	\begin{align} \label{E:BASICENERGYINEQUALITYFORWAVEEQUATIONS}
			\| (\varphi, \partial_t \varphi) \|_{H^1(\Sigma_t) \times L^2(\Sigma_t)}^2
			& \lesssim 
				\| (\varphi, \partial_t \varphi) \|_{H^1(\Sigma_0) \times L^2(\Sigma_0)}^2
				+
				\int_0^t
						\|
							\pmb{\partial} \vec{\Psi} 
						\|_{L^{\infty}(\Sigma_{\uptau})}
					\| (\varphi, \partial_t \varphi) \|_{H^1(\Sigma_{\uptau}) \times L^2(\Sigma_{\uptau})}^2
				\, d \uptau
					\\
			& \ \
				+
				\int_0^t
					\|\hat{\square}_{\gfour} \varphi \|_{L^2(\Sigma_{\uptau})} \|\pmb{\partial} \varphi \|_{L^2(\Sigma_{\uptau})}
				\, d \uptau.
				\notag
	\end{align}

\end{lemma}

\begin{proof}
	Let $\Jenwithlowerarg{\Transport}{\alpha}[\varphi] := \Jenarg{\Transport}{\alpha}[\varphi] - \varphi^2 \Transport^{\alpha}$,
	where $\Jenarg{\Transport}{\alpha}[\varphi]$ is defined by \eqref{E:MULTIPLIERVECTORFIELD}.	
	Note that \eqref{E:TRANSPORTISLENGTHONE} implies
	$\Jenwithlowerarg{\Transport}{\kappa}[\varphi] \Transport_{\kappa}
	= \Jenarg{\Transport}{\kappa}[\varphi] \Transport_{\kappa} 
	+ \varphi^2
	$ 
	and thus
	$\Jenwithlowerarg{\Transport}{\kappa}[\varphi] \Transport_{\kappa}$
	is equal to the integrand in the middle term in \eqref{E:BASICENERGYDEF}.
	Next, taking into account definition \eqref{E:DEFORMATIONTENSOR}, 
	we compute that 
	$\Dfour_{\kappa} \Jenwithlowerarg{\Transport}{\kappa}[\varphi] = \Dfour_{\kappa} \Jenarg{\Transport}{\kappa}[\varphi]
	- 2 \varphi \Transport \varphi - 
	\frac{1}{2}
	\varphi^2 (\gfour^{-1})^{\kappa \lambda} \deformarg{\Transport}{\kappa}{\lambda}$.
	Applying the divergence theorem on the spacetime region $[0,t] \times \mathbb{R}^3$ relative to the volume form
	$d \varpi_{\gfour} = \sqrt{|\mbox{\upshape det} \gfour|}  dx^1 dx^2 dx^3 d \uptau = d \varpi_g d \uptau$
	(where the last equality follows from \eqref{E:FIRSTFUNDAMENTALFORM}--\eqref{E:INVERSEACOUSTICALMETRIC} and 
	\eqref{E:SPACEVOLUMEFORMRELATIVETOCARTESIAN}),
	recalling that $\Transport$ is the future-directed $\gfour$-unit normal to $\Sigma_t$,
	appealing to definition \eqref{E:BASICENERGYDEF}, 
	and using equation \eqref{E:DIVERGENCEOFENERGYCURRENT} with $\mathbf{X} := \Transport$, 
	we deduce
	\begin{align} \label{E:WAVEEQUATIONENERGYIDUSEDINPROOF}
		\mathbb{E}[\varphi](t)
		& = \mathbb{E}[\varphi](0)
			-
			\int_0^t
				\int_{\Sigma_{\uptau}}
					(\square_{\gfour} \varphi) 
					\Transport \varphi
				\, d \varpi_g
			\, d \uptau
			+
			2
			\int_0^t
				\int_{\Sigma_{\uptau}}
					\varphi \Transport \varphi
					\, d \varpi_g
			\, d \uptau
				\\
		& \ \
			+
			\frac{1}{2}
			\int_0^t
				\int_{\Sigma_{\uptau}}
					(\gfour^{-1})^{\kappa \lambda} \deformarg{\Transport}{\kappa}{\lambda}
					\varphi^2 
					\, d \varpi_g
			\, d \uptau
			-
			\frac{1}{2} 
			\int_0^t
				\int_{\Sigma_{\uptau}}
				\enmomem^{\kappa \lambda}[\varphi] \deformarg{\Transport}{\kappa}{\lambda}
				\, d \varpi_g
			\, d \uptau.
			\notag
	\end{align}
	Next, we note that since the bootstrap assumption \eqref{E:BOOTSOLUTIONDOESNOTESCAPEREGIMEOFHYPERBOLICITY} guarantees that the
	compressible Euler solution is contained in $\mathfrak{K}$, 
	we have the following estimates for $\alpha,\beta = 0,1,2,3$:
	$|\Transport^{\alpha}| \lesssim 1$,
	$|\gfour_{\alpha \beta}| \lesssim 1$, 
	$|(\gfour^{-1})^{\alpha \beta}| \lesssim 1$, 
	and $|\pmb{\partial} \gfour_{\alpha \beta}| \lesssim |\pmb{\partial} \vec{\Psi}|$.
	It follows that
	$\square_{\gfour} \varphi = \hat{\square}_{\gfour} \varphi 
	+ 
	\mathcal{O}({|\pmb{\partial} \vec{\Psi}|}) |\pmb{\partial} \varphi|$,
	$|\Transport \varphi| \lesssim |\pmb{\partial} \varphi|$,	
	$\enmomem[\varphi] \lesssim |\pmb{\partial} \varphi|^2$,
	and $|\deformarg{\Transport}{\kappa}{\lambda}| \lesssim |\pmb{\partial} \vec{\Psi}|$.
	From these estimates, 
	the identity \eqref{E:WAVEEQUATIONENERGYIDUSEDINPROOF},
	the coercivity estimate \eqref{E:CORECIVENESSOFBASICENERGY}, 
	and the Cauchy--Schwarz inequality along $\Sigma_{\uptau}$,
	we conclude \eqref{E:BASICENERGYINEQUALITYFORWAVEEQUATIONS}.
\end{proof}

\subsection{The energy method for transport equations}
\label{SS:ENERGYESTIMATESFORTRANSPORT}
In this subsection, we provide a simple lemma that yields 
a basic energy inequality for solutions to transport equations.

\begin{lemma}[Energy estimates for transport equations]
	Let $\varphi$ be smooth on $[0,\Tboot] \times \mathbb{R}^3$. 
	Under the bootstrap assumptions of Sect.\,\ref{S:DATAANDBOOTSTRAPASSUMPTION}, 
	the following inequality holds for $t \in [0,\Tboot]$:
\label{L:BASICENERGYINEQUALITYFORTRANSPORTEQUATIONS}
	\begin{align} \label{E:BASICENERGYINEQUALITYFORTRANSPORTEQUATIONS}
			\| \varphi \|_{L^2(\Sigma_t)}^2
			& \lesssim 
				\| \varphi \|_{L^2(\Sigma_0)}^2
				+
				\int_0^t
					\| \pmb{\partial} \vec{\Psi} \|_{L^{\infty}(\Sigma_{\uptau})} \| \varphi \|_{L^2(\Sigma_{\uptau})}^2
				\, d \uptau
				+
				\int_0^t
					\| \varphi \|_{L^2(\Sigma_{\uptau})} \| \Transport \varphi \|_{L^2(\Sigma_{\uptau})}
				\, d \uptau.
	\end{align}
\end{lemma}

\begin{proof}
	Let $\mathbf{J}^{\alpha} := \varphi^2 \Transport^{\alpha}$.
	Then $\partial_{\alpha} \mathbf{J}^{\alpha} = 2 \varphi \Transport \varphi 
	+ (\partial_a v^a) \varphi^2$.
	Thus, we have
	$|\partial_{\alpha} \mathbf{J}^{\alpha}| \lesssim |\varphi| |\Transport \varphi|
	+ |\pmb{\partial} \vec{\Psi}| \varphi^2$. From this estimate,
	a routine application of the divergence theorem on the spacetime region $[0,t] \times \mathbb{R}^3$ relative to the Cartesian coordinates
	that exploits the positivity of $\mathbf{J}^0 = \varphi^2$,
	and the Cauchy--Schwarz inequality along $\Sigma_{\uptau}$,
	we conclude the desired estimate \eqref{E:BASICENERGYINEQUALITYFORTRANSPORTEQUATIONS}.
\end{proof}

\subsection{\texorpdfstring{The standard elliptic div-curl identity in $L^2$ spaces}{The standard elliptic div-curl identity in L2 spaces}}
\label{SS:STANDARDELLIPTICL2IDENTITY}
To control the top-order spatial derivatives of the specific vorticity and entropy,
we will rely on the following standard elliptic identity.

\begin{lemma}[Elliptic div-curl identity in $L^2$ spaces]
\label{L:ELLIPTICDIVCURLESTIMATES}
For vectorfields $V \in H^1(\mathbb{R}^3;\mathbb{R}^3)$, the following identity holds:
\begin{align} \label{E:STANDARDL2DIVCURLESTIMATES}
	\sum_{a,b=1}^3 \|\partial_a V^b \|_{L^2(\mathbb{R}^3)}^2
	& = 
			\| \dive V \|_{L^2(\mathbb{R}^3)}^2
			+
			\| \curl V \|_{L^2(\mathbb{R}^3)}^2.
\end{align}

\end{lemma}

\begin{proof}
It suffices to prove the desired identity for smooth, compactly supported vectorfields, since
these are dense in $H^1(\mathbb{R}^3;\mathbb{R}^3)$.
For smooth, compactly supported vectorfields, the desired identity follows from
integrating the divergence identity 
$
\sum_{a,b=1}^3
		(\partial_a V_b)^2
 	= 	
	(\dive V)^2
	+
	|\curl V|^2
	+
		\partial_a
		\left\lbrace
			V^b 
			\partial_b V^a
		\right\rbrace
		-
		\partial_a
		\left\lbrace
			V^a 
			\dive V
		\right\rbrace
$
over $\mathbb{R}^3$ with respect to volume form of the standard Euclidean metric on $\mathbb{R}^3$.

\end{proof}

\subsection{\texorpdfstring{Proof of Proposition~\ref{P:PRELIMINARYENERGYANDELLIPTICESTIMATES}}{Proof of Proposition ref P:PRELIMINARYENERGYANDELLIPTICESTIMATES}}
\label{SS:PROOFOFPRELIMINARYENERGYANDELLIPTICESTIMATES}
We first note that the estimates for the terms 
\begin{align}
\notag
\| \partial_t^2 (\LogDensity,\vec{v}) \|_{L^2(\Sigma_t)},
\,
\sum_{k=1}^2
\| \partial_t^k \vec{\vortrenormalized} \|_{H^{2-k}(\Sigma_t)},
\,
\sum_{k=1}^2
	\| \partial_t^k \Ent \|_{H^{3-k}(\Sigma_t)},
\mbox{ and }
	\| \partial_t (\vec{\VortVort},\DivGradEnt) \|_{L^2(\Sigma_t)}
\end{align}
on LHS~\eqref{E:BASICH2ENERGYESTIMATE} follow once we have obtained
the desired estimates for the remaining terms on LHS~\eqref{E:BASICH2ENERGYESTIMATE}.
The reason is that 
these time-derivative-involving terms 
can be bounded by $\lesssim$ the sum of products of the other terms
on LHS~\eqref{E:BASICH2ENERGYESTIMATE}
by using the equations of Prop.\,\ref{P:GEOMETRICWAVETRANSPORT} 
to solve for the relevant time derivatives 
in terms of spatial derivatives and then using standard product estimates
as well as our bootstrap assumption that the compressible Euler solution is contained in $\mathfrak{K}$ 
(i.e., \eqref{E:BOOTSOLUTIONDOESNOTESCAPEREGIMEOFHYPERBOLICITY});
we omit these straightforward details.
Thus, it suffices for us to bound the remaining terms on LHS~\eqref{E:BASICH2ENERGYESTIMATE}.

To proceed, we commute the equations of Prop.\,\ref{P:GEOMETRICWAVETRANSPORT} 
with up to one spatial derivative,
appeal to Def.\,\ref{D:MODIFIEDFLUIDVARIABLES},
consider Remark~\ref{R:SIMPLEWAYTOTHINKOFVORTICITYANDENTROPYGRADIENT},
and use the bootstrap assumption \eqref{E:BOOTSOLUTIONDOESNOTESCAPEREGIMEOFHYPERBOLICITY},
thereby deducing that for $\Psi \in \lbrace \LogDensity, v^1, v^2, v^3, \Ent \rbrace$, 
we have the following pointwise estimates:
\begin{align} \label{E:ONCECOMMUTEDWAVEQUATIONS}
	|\hat{\square}_{\gfour} \partial^{\leq 1} \Psi|
	& \lesssim
			|\partial (\vec{\VortVort},\DivGradEnt)|
			+
			\left\lbrace
				|\pmb{\partial} \vec{\Psi}|
				+
				1
			\right\rbrace
			|\partial \pmb{\partial} \vec{\Psi}|
			+
			\sum_{P=1}^3
			|\pmb{\partial} \vec{\Psi}|^P,
				\\
	|\Transport \partial^{\leq 1} (\vec{\vortrenormalized},\vec{\GradEnt})|
	& \lesssim 
			|\partial \pmb{\partial} \vec{\Psi}|
			+
		\left\lbrace
			|(\pmb{\partial} \vec{\Psi},\partial \vec{\vortrenormalized},\partial \vec{\GradEnt})|
			+
			1
		\right\rbrace
		|\pmb{\partial} \vec{\Psi}|,
			\label{E:POINTWISEBELOWTOPORDERCOMMUTEDVORCITITYANDENTROPY}
			\\
	|\partial(\dive \vortrenormalized,\curl \GradEnt)|
	& \lesssim |\partial \pmb{\partial} \vec{\Psi}|
			+
			|\partial \vec{\vortrenormalized}| |\pmb{\partial} \vec{\Psi}|,
		 \label{E:ONECOMMUTEDNONEVOLUTIONEQUATIONS} \\
|\partial(\curl \vortrenormalized,\dive \GradEnt)|
	& \lesssim
		|\partial (\vec{\VortVort},\DivGradEnt)|
		+
		|(\pmb{\partial} \vec{\Psi},\partial \vec{\vortrenormalized},\partial \vec{\GradEnt})|
		+
		\left\lbrace
			|(\pmb{\partial} \vec{\Psi},\partial \vec{\vortrenormalized},\partial \vec{\GradEnt})|
			+
			1
		\right\rbrace
		|\pmb{\partial} \vec{\Psi}|,
			\label{E:POINTWISEONCECOMMUTEDCURLVORTANDDIVGRADENTINTERMSOFMODVARIABLES}  \\
|\Transport \partial (\vec{\VortVort}, \DivGradEnt)|
	& 
		\lesssim
		\left\lbrace
			|\pmb{\partial} \vec{\Psi}| 
			+
			1
		\right\rbrace
		|\partial^2 (\vec{\vortrenormalized},\vec{\GradEnt})|
		+
		\left\lbrace
			|(\pmb{\partial} \vec{\Psi},\partial \vec{\vortrenormalized},\partial \vec{\GradEnt})|
			+
			1
		\right\rbrace
		|\partial \pmb{\partial} \vec{\Psi}|
		+
		|\pmb{\partial} \vec{\Psi}|^2
		|(\partial \vec{\vortrenormalized},\partial \vec{\GradEnt})|
			\label{E:ONCECOMMUTEDTOPORDERTRANSPORT} 
			\\
	& \ \	
		+
		\sum_{P=1}^3
		|\pmb{\partial} \vec{\Psi}|^P.
		\notag
\end{align}
We clarify that in deriving \eqref{E:POINTWISEONCECOMMUTEDCURLVORTANDDIVGRADENTINTERMSOFMODVARIABLES},
we used Def.\,\ref{D:MODIFIEDFLUIDVARIABLES}
to algebraically solve for $\curl \vortrenormalized$ and $\dive \GradEnt$.

Using the estimates \eqref{E:ONCECOMMUTEDWAVEQUATIONS}--\eqref{E:ONCECOMMUTEDTOPORDERTRANSPORT}, 
we will derive estimates for the ``controlling quantity'' $\controlling_2(t)$ defined by
\begin{align} \label{E:H2CONTROLLINGQUANTITY}
	\controlling_2(t)
	& := 
			\| (\vec{\Psi}, \pmb{\partial} \vec{\Psi}) \|_{H^2(\Sigma_t) \times H^1(\Sigma_t)}^2
			+
			\| \partial (\vec{\VortVort},\DivGradEnt) \|_{L^2(\Sigma_t)}^2
			+
			\| (\vec{\vortrenormalized},\vec{\GradEnt}) \|_{H^1(\Sigma_t)}^2.
\end{align}
We will prove the following two estimates:
\begin{align} \label{E:H2CONTROLLINGQUANTITYGRONWALLREADY}
	\controlling_2(t)
	& \lesssim 
		\controlling_2(0)
		+
		\int_0^t
			\left\lbrace
				\| \pmb{\partial} \vec{\Psi} \|_{L^{\infty}(\Sigma_{\uptau})}^2
				+
				\|
					\partial (\vec{\vortrenormalized}, \vec{\GradEnt})
				\|_{L^{\infty}(\Sigma_{\uptau})}
				+
				1
			\right\rbrace
			\controlling_2(\uptau)
		\, d \uptau,
			\\
	\| \partial^2 (\vec{\vortrenormalized},\vec{\GradEnt}) \|_{L^2(\Sigma_t)}^2
	&  
	\lesssim 
	\controlling_2(t) + \controlling_2^3(t).
	\label{E:NONLINEARH2DERIVATIVESOFVORTICITYANDGRADENTINTERMSOFH2CONTROLLINGQUANTITY}
\end{align}
Then from 
the bootstrap assumptions \eqref{E:BOOTSTRICHARTZ}--\eqref{E:BOOTL2LINFINITYFIRSTDERIVATIVESOFVORTICITYBOOTANDNENTROPYGRADIENT},
\eqref{E:H2CONTROLLINGQUANTITYGRONWALLREADY},
and Gr\"{o}nwall's inequality,
we deduce that for $t \in [0,\Tboot]$, we have 
$\controlling_2(t) \lesssim \controlling_2(0)$.
From this estimate,
\eqref{E:NONLINEARH2DERIVATIVESOFVORTICITYANDGRADENTINTERMSOFH2CONTROLLINGQUANTITY},
and the remarks made at the beginning the proof,
we arrive at the desired estimate \eqref{E:BASICH2ENERGYESTIMATE}.

It remains for us to prove \eqref{E:H2CONTROLLINGQUANTITYGRONWALLREADY}
and \eqref{E:NONLINEARH2DERIVATIVESOFVORTICITYANDGRADENTINTERMSOFH2CONTROLLINGQUANTITY}.
We start with the elliptic estimates needed to control
$\partial^2 \vec{\vortrenormalized}$ and $\partial^2 \vec{\GradEnt}$ in $\| \cdot \|_{L^2(\Sigma_t)}$.
From \eqref{E:STANDARDL2DIVCURLESTIMATES} with $\partial \vortrenormalized$ and $\partial \GradEnt$ in the role of $V$,
\eqref{E:ONECOMMUTEDNONEVOLUTIONEQUATIONS},
and
\eqref{E:POINTWISEONCECOMMUTEDCURLVORTANDDIVGRADENTINTERMSOFMODVARIABLES},
we find that
\begin{align}
	\| \partial^2 (\vec{\vortrenormalized}, \vec{\GradEnt}) \|_{L^2(\Sigma_t)}^2& 
	\lesssim 
	\| \partial (\vec{\VortVort}, \DivGradEnt) \|_{L^2(\Sigma_t)}^2
	+
	\| \partial \pmb{\partial} \vec{\Psi} \|_{L^2(\Sigma_t)}^2
		\\
& \ \
	+
	\left\lbrace
		\| \pmb{\partial} \vec{\Psi} \|_{L^{\infty}(\Sigma_t)}^2
		+
		1
	\right\rbrace
	\left\lbrace
		\| \partial (\vec{\vortrenormalized}, \vec{\GradEnt}) \|_{L^2(\Sigma_t)}^2 
		+
		\| \pmb{\partial} \vec{\Psi} \|_{L^2(\Sigma_t)}^2
	\right\rbrace,
	\notag
\end{align}
which, in view of definition \eqref{E:H2CONTROLLINGQUANTITY}, implies that
\begin{align} \label{E:H2DERIVATIVESOFVORTICITYANDGRADENTINTERMSOFH2CONTROLLINGQUANTITY}
	\| \partial^2 (\vec{\vortrenormalized}, \vec{\GradEnt}) \|_{L^2(\Sigma_t)}^2
	&  
	\lesssim 
	\left\lbrace
		\| \pmb{\partial} \vec{\Psi} \|_{L^{\infty}(\Sigma_t)}^2
		+
		1
	\right\rbrace
	\controlling_2(t).
\end{align}
Moreover, through an argument similar to the one we used to derive 
\eqref{E:H2DERIVATIVESOFVORTICITYANDGRADENTINTERMSOFH2CONTROLLINGQUANTITY}, 
based on \eqref{E:STANDARDL2DIVCURLESTIMATES},
\eqref{E:ONECOMMUTEDNONEVOLUTIONEQUATIONS},
and
\eqref{E:POINTWISEONCECOMMUTEDCURLVORTANDDIVGRADENTINTERMSOFMODVARIABLES},
but modified in that we now use the interpolation-product estimate\footnote{This standard estimate can be obtained by using H\"older's inequality, Sobolev embedding, and interpolation estimates. For a more detailed proof, we refer to the proof of  \eqref{E:INTERPOLATIONPRODUCTOFTWOH1FUNCTIONSINL2}.}
$$
\|G_1 \cdot G_2 \|_{L^2(\Sigma_t)} \lesssim \| G_1 \|_{L^2(\Sigma_t)}^{1/2} \| G_1 \|_{H^1(\Sigma_t)}^{1/2} \| G_2 \|_{H^1(\Sigma_t)},$$
to derive the bound
$$
\|  
	|(\pmb{\partial} \vec{\Psi},\partial \vec{\vortrenormalized},\partial \vec{\GradEnt})|
	|\pmb{\partial} \vec{\Psi}|
\|_{L^2(\Sigma_t)}^2
\lesssim
\|
	\partial (\vec{\vortrenormalized},\vec{\GradEnt})
\|_{L^2(\Sigma_t)}
\|
	\partial (\vec{\vortrenormalized},\vec{\GradEnt})
\|_{H^1(\Sigma_t)}
\|  
	\pmb{\partial} \vec{\Psi}
\|_{H^1(\Sigma_t)}^2
+
\|  
	\pmb{\partial} \vec{\Psi}
\|_{H^1(\Sigma_t)}^4,
$$
we deduce that 
$$
\| \partial^2 (\vec{\vortrenormalized}, \vec{\GradEnt}) \|_{L^2(\Sigma_t)}^2
	\lesssim 
	\| \partial^2 (\vec{\vortrenormalized}, \vec{\GradEnt}) \|_{L^2(\Sigma_t)}
	\controlling_2^{3/2}(t)
	+
	\controlling_2(t)
	+
	\controlling_2^2(t)
	\leq
	\frac{1}{2} \| \partial^2 (\vec{\vortrenormalized}, \vec{\GradEnt}) \|_{L^2(\Sigma_t)}^2
	+
	C \controlling_2(t)
	+
	C \controlling_2^3(t),
$$
from which the desired bound \eqref{E:NONLINEARH2DERIVATIVESOFVORTICITYANDGRADENTINTERMSOFH2CONTROLLINGQUANTITY}
readily follows.

We now derive energy estimates for the evolution equations.
From \eqref{E:BASICENERGYINEQUALITYFORWAVEEQUATIONS} with $\partial^{\leq 1} \vec{\Psi}$ 
in the role of $\varphi$,
\eqref{E:ONCECOMMUTEDWAVEQUATIONS}, 
the Cauchy--Schwarz inequality along $\Sigma_{\uptau}$, Young's inequality,
and definition \eqref{E:H2CONTROLLINGQUANTITY},
we deduce (occasionally using the non-optimal bound $|\pmb{\partial} \vec{\Psi}| \lesssim |\pmb{\partial} \vec{\Psi}|^2 + 1$)
that
	\begin{align} \label{E:WAVEEQUATIONENERGYINEQUALITYREADYTOBEGRONWALLED}
			\| (\vec{\Psi}, \pmb{\partial} \vec{\Psi}) \|_{H^2(\Sigma_t) \times H^1(\Sigma_t)}^2
			& \lesssim 
				\controlling_2(0)
					+
				\int_0^t
					\left\lbrace
						\|
							\pmb{\partial} \vec{\Psi} 
						\|_{L^{\infty}(\Sigma_{\uptau})}^2
						+
						1
					\right\rbrace
					\controlling_2(\uptau)
				\, d \uptau.
\end{align}
Using a similar argument based on
\eqref{E:BASICENERGYINEQUALITYFORTRANSPORTEQUATIONS} with
$\partial^{\leq 1} \vec{\vortrenormalized}$ 
and
$\partial^{\leq 1} \vec{\GradEnt}$
in
the role of $\varphi$
and equation \eqref{E:POINTWISEBELOWTOPORDERCOMMUTEDVORCITITYANDENTROPY},
we deduce 
\begin{align} \label{E:BELOWTOPORDERDERIVATIVESOFVORTICITYANDENTROPYREADYTOBEGRONWALLED}
	\| (\vec{\vortrenormalized},\vec{\GradEnt}) \|_{H^1(\Sigma_t)}^2
	& \lesssim 
		\controlling_2(0)
		+
		\int_0^t
			\left\lbrace
				\| \pmb{\partial} \vec{\Psi} \|_{L^{\infty}(\Sigma_{\uptau})}^2
					+
				1
			\right\rbrace
			\controlling_2(\uptau)
		\, d \uptau.
\end{align}
Using a similar argument based on
\eqref{E:BASICENERGYINEQUALITYFORTRANSPORTEQUATIONS} with
$\partial \vec{\VortVort}$ and $\partial \DivGradEnt$ in
the role of $\varphi$
and equation \eqref{E:ONCECOMMUTEDTOPORDERTRANSPORT},
and using the elliptic estimate \eqref{E:H2DERIVATIVESOFVORTICITYANDGRADENTINTERMSOFH2CONTROLLINGQUANTITY} to control
the norm $\| \cdot \|_{L^2(\Sigma_t)}$ of the (linear) factor of $\partial^2 (\vec{\vortrenormalized}, \vec{\GradEnt})$ 
on RHS~\eqref{E:ONCECOMMUTEDTOPORDERTRANSPORT}, 
we deduce
\begin{align} \label{E:ONCECOMMUTEDMODIFIEDVARIABLESENERGYINEQUALITYREADYTOBEGRONWALLED}
			\| \partial (\vec{\VortVort},\DivGradEnt) \|_{L^2(\Sigma_t)}^2
			& \lesssim 
				\controlling_2(0)
				+
				\int_0^t
					\left\lbrace
						\| \pmb{\partial} \vec{\Psi} \|_{L^{\infty}(\Sigma_{\uptau})}^2
						+
						\|
							\partial (\vec{\vortrenormalized}, \vec{\GradEnt})
						\|_{L^{\infty}(\Sigma_{\uptau})}
						+
						1
					\right\rbrace
					\controlling_2(\uptau)
				\, d \uptau.
	\end{align}
	Adding 
	\eqref{E:WAVEEQUATIONENERGYINEQUALITYREADYTOBEGRONWALLED},
	\eqref{E:BELOWTOPORDERDERIVATIVESOFVORTICITYANDENTROPYREADYTOBEGRONWALLED},
	and
	\eqref{E:ONCECOMMUTEDMODIFIEDVARIABLESENERGYINEQUALITYREADYTOBEGRONWALLED},
	we conclude, in view of definition \eqref{E:H2CONTROLLINGQUANTITY},
	the desired bound \eqref{E:H2CONTROLLINGQUANTITYGRONWALLREADY}.
	\hfill $\qed$

\section{Energy and elliptic estimates along constant-time hypersurfaces up to top-order}
\label{S:TOPORDERENERGYESTIMATES}
Our main goal in this section is to use the bootstrap assumptions
to prove energy and elliptic estimates along $\Sigma_t$ up to top-order. The main result is Prop.\,~\ref{P:TOPORDERENERGYESTIMATES},
which we prove in Subsect.\,\ref{SS:PROOFOFTOPORDERENERGYESTIMATES} after providing some
preliminary technical estimates. 

\begin{proposition}[Energy and elliptic estimates up to top-order]
\label{P:TOPORDERENERGYESTIMATES}
There exists a continuous strictly increasing function $\contfunction : [0,\infty) \rightarrow [0,\infty)$
such that under the initial data and bootstrap assumptions of Sect.\,\ref{S:DATAANDBOOTSTRAPASSUMPTION}, 
the following estimate holds for $t \in [0,\Tboot]$:
\begin{align} \label{E:TOPORDERENERGYESTIMATES}
	&
	\sum_{k=0}^2
	\| \partial_t^k (\LogDensity,\vec{v},\vec{\vortrenormalized}) \|_{H^{\Sob-k}(\Sigma_t)}
	+
	\sum_{k=0}^2
	\| \partial_t^k \Ent \|_{H^{\Sob+1-k}(\Sigma_t)}
	+
	\sum_{k=0}^1
	\| \partial_t^k (\vec{\VortVort},\DivGradEnt) \|_{H^{\Sob-1-k}(\Sigma_t)}
		\\
	& \leq
		\contfunction
		\left(
		\| (\LogDensity,\vec{v}, \vec{\vortrenormalized}) \|_{H^{\Sob}(\Sigma_0)}
		+
		\| \Ent \|_{H^{\Sob+1}(\Sigma_0)}
		\right).
		\notag
\end{align}

\end{proposition}

\subsection{Equations satisfied by the frequency-projected solution variables}
\label{SS:FREQUENCYPROJECTEDEQUATIONS}
In proving Prop.\,\ref{P:TOPORDERENERGYESTIMATES}, we will derive
energy and elliptic estimates for projections of the solution variables onto 
dyadic frequencies $\upnu \in 2^{\mathbb{N}}$. 
In the next lemma, as a preliminary step in deriving these estimates,
we derive the equations satisfied by the frequency-projected solution variables.

\begin{lemma}[Equations satisfied by the frequency-projected solution variables]
\label{L:FREQUENCYPROJECTEDEVOLUTION}	
Let $\upnu \in 2^{\mathbb{N}}$.
For solutions to the equations of Prop.\,\ref{P:GEOMETRICWAVETRANSPORT},
the following equations hold, 
where
$\gfour = \gfour(\vec{\Psi})$,
$\Psi \in \lbrace \LogDensity, v^1, v^2, v^3, \Ent \rbrace$,
and the terms 
$\inhom_{(\Psi)}, \cdots, \mathfrak{F}_{(\DivGradEnt)}$
on RHSs~\eqref{E:REMAINDERTERMFREQUENCYPROJECTEDCOVARIANTWAVEVE}--\eqref{E:REMAINDERTERMFREQUENCYPROJECTEDTRANSPORTDIVGRADENTROPY}
are defined in Prop.\,\ref{P:GEOMETRICWAVETRANSPORT}:
\begin{subequations}
\begin{align} \label{E:FREQUENCYPROJECTEDWAVE}
\hat{\square}_{\gfour} P_{\upnu} \Psi 
& = 
\hat{\remainder}_{(\Psi);\upnu},
	\\
\square_{\gfour} P_{\upnu} \Psi 
& = 
\remainder_{(\Psi);\upnu},
	\label{E:COVARIANTFREQUENCYPROJECTEDWAVE} 
\end{align}
\end{subequations}

\begin{subequations}
\begin{align}
\dive P_{\upnu} \vortrenormalized 
	& = 
		\remainder_{(\dive \vortrenormalized);\upnu},
\label{E:FREQUENCYPROJECTEDDIVVORTICITY}
\\
\Transport P_{\upnu} \VortVort^i 
	& = 
		\remainder_{(\VortVort^i);\upnu},
\label{E:FREQUENCYPROJECTEDTRANSPORTVORTICITYVORTICITY}
\end{align}
\end{subequations}

\begin{subequations}
\begin{align}
	\Transport P_{\upnu} \DivGradEnt 
		& = \remainder_{(\DivGradEnt);\upnu},
		\label{E:FREQUENCYPROJECTEDTRANSPORTDIVGRADENTROPY}
\\
(\curl P_{\upnu} \GradEnt)^i & = 0,
\label{E:FREQUENCYPROJECTEDCURLGRADENT}
\end{align}	
\end{subequations}
where the inhomogeneous terms take the following form:
\begin{subequations}
\begin{align} \label{E:REMAINDERTERMFREQUENCYPROJECTEDCOVARIANTWAVEVE} 
\hat{\remainder}_{(\Psi);\upnu}
& = P_{\upnu} \inhom_{(\Psi)}
		+
		\sum_{(\alpha,\beta) \neq (0,0)}
		\left\lbrace
			(\gfour^{-1})^{\alpha \beta}
			-
			P_{\leq \upnu} (\gfour^{-1})^{\alpha \beta}
		\right\rbrace
		P_{\upnu} \partial_{\alpha} \partial_{\beta} \Psi
			\\
	& \ \
		+
		\sum_{(\alpha,\beta) \neq (0,0)}
		\left\lbrace
		\left(P_{\leq \upnu} (\gfour^{-1})^{\alpha \beta} 
		\right)
		P_{\upnu}
		\partial_{\alpha} \partial_{\beta} \Psi
		-
		P_{\upnu} 
		\left[(\gfour^{-1})^{\alpha \beta} \partial_{\alpha} \partial_{\beta} \Psi
		\right]
		\right\rbrace,
		\notag
			\\
\remainder_{(\Psi);\upnu}
& = \hat{\remainder}_{(\Psi);\upnu}
		-	
		\Chfour^{\alpha} 
		P_{\upnu} \partial_{\alpha} \Psi,
		\label{E:COVARIANTREMAINDERTERMFREQUENCYPROJECTEDCOVARIANTWAVEVE}  
	\end{align}
	\end{subequations}
	$\Chfour^{\alpha} 
	= 
	(\gfour^{-1})^{\alpha \beta} (\gfour^{-1})^{\gamma \delta} \partial_{\gamma} \gfour_{\beta \delta}
	-
	\frac{1}{2}
	(\gfour^{-1})^{\alpha \beta} (\gfour^{-1})^{\gamma \delta} \partial_{\beta} \gfour_{\gamma \delta}
	= \mathscr{L}(\vec{\Psi})[\pmb{\partial} \vec{\Psi}] 
	$
	are the contracted Cartesian Christoffel symbols of $\gfour(\vec{\Psi})$,
	and
	\begin{subequations}
		\begin{align}
				\remainder_{(\dive \vortrenormalized);\upnu}
				& = P_{\upnu} \mathfrak{F}_{(\dive \vortrenormalized)},
					\label{E:REMAINDERTERMFREQUENCYPROJECTEDDIVVORTICITY} \\
		\remainder_{(\VortVort^i);\upnu}
		& = P_{\upnu} \mathfrak{F}_{(\VortVort^i)}
				+ 
				\left\lbrace
					v^a 
					-
					P_{\leq \upnu} v^a
				\right\rbrace
				 P_{\upnu} \partial_a \VortVort^i
				+ 
				\left\lbrace
					(P_{\leq \upnu} v^a)  P_{\upnu} \partial_a \VortVort^i
					-
					P_{\upnu} [v^a \partial_a \VortVort^i]
				\right\rbrace,
				\label{E:REMAINDERTERMFREQUENCYPROJECTEDTRANSPORTVORTICITYVORTICITY}
		\end{align}
	\end{subequations}
	
	\begin{align} \label{E:REMAINDERTERMFREQUENCYPROJECTEDTRANSPORTDIVGRADENTROPY}
		\remainder_{(\DivGradEnt);\upnu}
		& = P_{\upnu} \mathfrak{F}_{(\DivGradEnt)}
				+ 
				\left\lbrace
					v^a 
					-
					P_{\leq \upnu} v^a
				\right\rbrace
				 P_{\upnu} \partial_a \DivGradEnt
				+ 
				\left\lbrace
					(P_{\leq \upnu} v^a)  P_{\upnu} \partial_a \DivGradEnt
					-
					P_{\upnu} [v^a \partial_a \DivGradEnt]
				\right\rbrace.
	\end{align}

Moreover,
\begin{subequations}
\begin{align}
\square_{\gfour} P_{\upnu} \pmb{\partial} \Psi 
& = 
\anotherremainder_{(\pmb{\partial} \Psi);\upnu},
	\label{E:TIMEDERIVATIVECOMMUTEDWAVEFREQUENCYPROJECTED} 
		\\
\Transport P_{\upnu} \pmb{\partial} \VortVort^i
& = \anotherremainder_{(\pmb{\partial} \VortVort^i);\upnu},
	\label{E:TIMEDERIVATIVECOMMUTEDMODIFIEDVORTFREQUENCYPROJECTED} 
	\\
\Transport P_{\upnu} \pmb{\partial} \DivGradEnt 
& =  \anotherremainder_{(\pmb{\partial} \DivGradEnt);\upnu},
\label{E:TIMEDERIVATIVECOMMUTEDMODIFIEDENTFREQUENCYPROJECTED}
\end{align}
\end{subequations}	
where
\begin{subequations}
\begin{align}
	\anotherremainder_{(\pmb{\partial} \Psi);\upnu}
	& = 
	P_{\upnu} \pmb{\partial} \inhom_{(\Psi)}
		-
		\sum_{(\alpha,\beta) \neq (0,0)}
		P_{\upnu}
		\left\lbrace
			\left(\pmb{\partial} (\gfour^{-1})^{\alpha \beta}
			\right)
			\partial_{\alpha} \partial_{\beta} \Psi
		\right\rbrace
		-	
		\Chfour^{\alpha} 
		P_{\upnu} \partial_{\alpha} \pmb{\partial} \Psi
		\label{E:INHOMTERMTIMEDERIVATIVECOMMUTEDWAVEFREQUENCYPROJECTED}  \\
	& \ \
		+
		\sum_{(\alpha,\beta) \neq (0,0)}
		\left\lbrace
			(\gfour^{-1})^{\alpha \beta}
			-
			P_{\leq \upnu} (\gfour^{-1})^{\alpha \beta}
		\right\rbrace
		P_{\upnu} \partial_{\alpha} \partial_{\beta} \pmb{\partial} \Psi
			\notag \\
	& \ \
		+
		\sum_{(\alpha,\beta) \neq (0,0)}
		\left\lbrace
		\left(P_{\leq \upnu} (\gfour^{-1})^{\alpha \beta} 
		\right)
		P_{\upnu}
		\partial_{\alpha} \partial_{\beta} \pmb{\partial} \Psi
		-
		P_{\upnu} 
		\left[(\gfour^{-1})^{\alpha \beta} \partial_{\alpha} \partial_{\beta} \pmb{\partial} \Psi
		\right]
		\right\rbrace,
			\notag \\
	\anotherremainder_{(\pmb{\partial} \VortVort^i);\upnu} 
	& = 	P_{\upnu} \pmb{\partial} \mathfrak{F}_{(\VortVort^i)}
				-
				P_{\upnu}[(\pmb{\partial} v^a) \partial_a \VortVort^i]
			\label{E:INHOMTERMTIMEDERIVATIVECOMMUTEDMODIFIEDVORTFREQUENCYPROJECTED}		
				\\
	& \ \
			+ 
				\left\lbrace
					v^a 
					-
					P_{\leq \upnu} v^a
				\right\rbrace
				 P_{\upnu} \partial_a \pmb{\partial} \VortVort^i
				+ 
				\left\lbrace
					(P_{\leq \upnu} v^a)  P_{\upnu} \partial_a \pmb{\partial} \VortVort^i
					-
					P_{\upnu} [v^a \partial_a \pmb{\partial} \VortVort^i]
				\right\rbrace,	
		\notag \\
	\anotherremainder_{(\pmb{\partial} \DivGradEnt);\upnu}  
	& = P_{\upnu} \pmb{\partial} \mathfrak{F}_{(\DivGradEnt)}
			-
			P_{\upnu} [(\pmb{\partial} v^a) \partial_a \DivGradEnt]
				\label{E:INHOMTERMTIMEDERIVATIVECOMMUTEDMODIFIEDDIVGRADFREQUENCYPROJECTED} \\
		& \ \	
				+ 
				\left\lbrace
					v^a 
					-
					P_{\leq \upnu} v^a
				\right\rbrace
				 P_{\upnu} \partial_a \pmb{\partial} \DivGradEnt
				+ 
				\left\lbrace
					(P_{\leq \upnu} v^a)  P_{\upnu} \partial_a \pmb{\partial} \DivGradEnt
					-
					P_{\upnu} [v^a \partial_a \pmb{\partial} \DivGradEnt]
				\right\rbrace.
	\notag
\end{align}	
\end{subequations}	
	
\end{lemma}

\begin{proof}
	The lemma follows from straightforward computations 
	and the fact that 
	$
	\square_{\gfour} \varphi
	=  
	\hat{\square}_{\gfour} \varphi 
	- 
	\Chfour^{\alpha} \partial_{\alpha} \varphi$ for scalar functions $\varphi$.
	We therefore omit the details.
\end{proof}

\subsection{Product and commutator estimates}
\label{SS:PRODUCTANDCOMMUTATORESTIMATES}
In this subsection, we derive estimates for various norms of the inhomogeneous terms
$\hat{\remainder}_{(\Psi);\upnu},\cdots,\anotherremainder_{(\pmb{\partial} \DivGradEnt);\upnu} $
on RHSs~\eqref{E:REMAINDERTERMFREQUENCYPROJECTEDCOVARIANTWAVEVE}--\eqref{E:REMAINDERTERMFREQUENCYPROJECTEDTRANSPORTDIVGRADENTROPY}.
We provide the main result in Lemma~\ref{L:PRODUCTANDCOMMUTATORFORCOMPRESSIBLEEULER}.

\subsubsection{Preliminary product and commutator estimates}
\label{SSS:PRELIMINARYPRODUCTANDCOMMUTATORESTIMATES}
In the next lemma, we provide some standard product and commutator estimates
that are based on the Littlewood--Paley calculus.

\begin{lemma}[Preliminary product and commutator estimates]
\label{L:PRELIMINARYPRODUCTANDCOMMUTATORESTIMATES}
The following estimates hold, where we assume that $F$, $G_i$, and $\varphi$ are (possibly array-valued) functions
on $\Sigma_t$, that $\gensmoothfunction$ is a smooth function
of its arguments, and that $\gensmoothfunction'$ denotes the derivative of $\gensmoothfunction$ with respect to its arguments.

\medskip

\noindent \underline{\textbf{Product estimates}}:
For any $\varepsilon$ such that $0 < \varepsilon < 1$
(in our forthcoming applications, we will set $\varepsilon := \Sob - 2$), 
the following product estimates hold,
where the implicit constants are allowed to depend on $\varepsilon$,
$\| \gensmoothfunction \circ \varphi \|_{L^{\infty}(\Sigma_t)}$,
and
$\| \gensmoothfunction' \circ \varphi \|_{L^{\infty}(\Sigma_t)}$,
and the projection operators $P_{\upnu}$ on the RHSs of the estimates
are allowed to correspond to a slightly different projection operator, localized at the same 
frequency, than the ones on the LHSs:
\begin{align} \label{E:LAMBDAONEPLUSEPSILONL2NORMEQUIVALENTTOLAMBDAEPSILONL2NORMOFGRADF}
	\| 
		\Lambda^{1 + \varepsilon} F
	\|_{L^2(\Sigma_t)}
	\approx
	\| 
		\Lambda^{\varepsilon} \partial F
	\|_{L^2(\Sigma_t)},
\end{align}

\begin{subequations}
\begin{align}
	\| 
		G_1 \cdot G_2
	\|_{L^2(\Sigma_t)}
	& \lesssim
		\| G_1 \|_{H^1(\Sigma_t)}
		\| G_2 \|_{H^1(\Sigma_t)},
			\label{E:PRODUCTOFTWOH1FUNCTIONSINL2} 
			\\
	\| 
		G_1 \cdot G_2
	\|_{L^2(\Sigma_t)}
	& \lesssim
		\| G_1 \|_{L^2(\Sigma_t)}^{1/2}
		\| G_1 \|_{H^1(\Sigma_t)}^{1/2}
		\| G_2 \|_{H^1(\Sigma_t)},
			\label{E:INTERPOLATIONPRODUCTOFTWOH1FUNCTIONSINL2} 
			\\
	\| 
		G_1 \cdot G_2 \cdot G_3
	\|_{L^2(\Sigma_t)}
	& \lesssim
		\| G_1 \|_{H^1(\Sigma_t)}
		\| G_2 \|_{H^1(\Sigma_t)}
		\| G_3 \|_{H^1(\Sigma_t)}.
		\label{E:PRODUCTOFTHREEH1FUNCTIONSINL2}
\end{align}
\end{subequations}

In addition, for dyadic frequencies $\upnu \geq 1$, we have:
\begin{align} \label{E:FREQUENCYPROJECTEDLINFINITYSMOOTHFUNCTIONPRODUCTESTIMATE} 
	\| 
		P_{\upnu} (\gensmoothfunction \circ \varphi \cdot G)
	\|_{L^{\infty}(\Sigma_t)}
	& \lesssim 
		\upnu^{-1/2}
		\| \partial \varphi \|_{L^{\infty}(\Sigma_t)} \| G \|_{H^1(\Sigma_t)}
		+
		\| P_{\upnu} G \|_{L^{\infty}(\Sigma_t)}.
\end{align}	

Moreover,
\begin{subequations}
\begin{align}
	\| 
		\Lambda^{\varepsilon} (\gensmoothfunction \circ \varphi \cdot G)
	\|_{L^2(\Sigma_t)}
	& \lesssim 
		\| \Lambda^{\varepsilon} G \|_{L^2(\Sigma_t)}
		+
		\| \partial \varphi \|_{H^1(\Sigma_t)}
		\| G \|_{H^{\varepsilon}(\Sigma_t)},
			\label{E:LAMBDAEPSILONFGPRODUCTESTIMATEINVOLVINGNOL2NORMOFF} \\
		\| 
		\Lambda^{\varepsilon} (F \cdot G)
	\|_{L^2(\Sigma_t)}
	& \lesssim 
		\| F \|_{H^{1/2 + \varepsilon}(\Sigma_t)} \| G \|_{H^1(\Sigma_t)}
		+
		\| G \|_{H^{1/2 + \varepsilon}(\Sigma_t)} \| F \|_{H^1(\Sigma_t)},
			\label{E:LAMBDAEPSILONFGPRODUCTESTIMATEINVOLVINGONLYSOBOLEV} \\
	\| 
		\Lambda^{\varepsilon} (F \cdot \partial G)
	\|_{L^2(\Sigma_t)}
	& \lesssim 
		\| F \|_{L^{\infty}(\Sigma_t)} \| \partial G \|_{H^{\varepsilon}(\Sigma_t)}
		+
		\| G \|_{L^{\infty}(\Sigma_t)} \| \partial F \|_{H^{\varepsilon}(\Sigma_t)},
		\label{E:LAMBDAEPSILONFGRADGPRODUCTESTIMATEINVOLVINGLINFINITY}
			\\
	\| 
		\Lambda^{\varepsilon} (G_1 \cdot G_2 \cdot G_3)
	\|_{L^2(\Sigma_t)}
	& \lesssim
		\sum_{j=1}^3
		\| G_j \|_{H^{1 + \varepsilon}(\Sigma_t)} \prod_{k \neq j} \| G_k \|_{H^1(\Sigma_t)}.
		\label{E:LAMBDAEPSILONPRODOFTHREETERMSL2ESTIMATEINTERMSOFH1PLUSEPSILONANDH1NORMS}
	\end{align}
	\end{subequations}

\medskip	
	
\noindent \underline{\textbf{Commutator estimates}}:
The following commutator estimates hold for dyadic frequencies $\upnu \geq 1$:
\begin{subequations}
\begin{align} \label{E:SUBTRACTLOWFREQUENCIESPRODUCTESTIMATE}
	\| [\gensmoothfunction \circ \varphi - P_{\leq \upnu} (\gensmoothfunction \circ \varphi)] \cdot P_{\upnu} G \|_{L^2(\Sigma_t)}
	& \lesssim \upnu^{-1} \| \partial \varphi \|_{L^{\infty}(\Sigma_t)} \| P_{\upnu} G \|_{L^2(\Sigma_t)},
		\\
	\| P_{\upnu} [\gensmoothfunction \circ \varphi \cdot \partial G] 
			- 
		P_{\leq \upnu} (\gensmoothfunction \circ \varphi) \cdot P_{\upnu} \partial G \|_{L^2(\Sigma_t)}
	& \lesssim 
				\| \partial \varphi \|_{L^{\infty}(\Sigma_t)} \| P_{\upnu} G \|_{L^2(\Sigma_t)}
					\label{E:FREQUENCYOFAPRODUCTMINUSLOWTIMESFREQUENCYESTIMATE}
						\\
	& \ \
				+
				\| G \|_{L^{\infty} (\Sigma_t)} 
				\| 
					P_{\upnu} [\gensmoothfunction' \circ \varphi \cdot \partial \varphi] 
				\|_{L^2(\Sigma_t)} 
					\notag \\
	& \ \
				+
				\sum_{\uplambda > \upnu}
				\uplambda^{-1}
				\| \partial \varphi \|_{L^{\infty}(\Sigma_t)} \| P_{\uplambda} \partial G \|_{L^2(\Sigma_t)}.
				\notag
\end{align}
\end{subequations}

\medskip

\noindent \underline{\textbf{Convolution-type estimate for dyadic-indexed sums}}:
If $\lbrace A_{\uplambda} \rbrace_{\uplambda \in 2^{\mathbb{N}}}$
is a dyadic-indexed sequence of real numbers, then 
\begin{align} \label{E:CONVOLUTIONESTIMATE}
	\left\|
	\upnu^{1 + \varepsilon}
	\sum_{\uplambda > \upnu}
	\uplambda^{-1}
	A_{\uplambda}
	\right\|_{\ell_{\upnu}^2}
	\lesssim
	\left\|
	\upnu^{\varepsilon}
	A_{\upnu}
	\right\|_{\ell_{\upnu}^2}.
\end{align}
\end{lemma}

\begin{proof}
\eqref{E:LAMBDAONEPLUSEPSILONL2NORMEQUIVALENTTOLAMBDAEPSILONL2NORMOFGRADF} is a basic result in harmonic analysis; see, e.g., 
\cite{hBjyDrD2011}*{Chapter~2}.
\eqref{E:LAMBDAEPSILONFGPRODUCTESTIMATEINVOLVINGONLYSOBOLEV} is proved in \cite{qW2014b}*{Lemma~17}.
\eqref{E:LAMBDAEPSILONFGRADGPRODUCTESTIMATEINVOLVINGLINFINITY} follows from the proof of
\cite{qW2014b}*{Lemma~19}, which yielded a similar estimate, differing only in the following minor fashion:
the terms
$\| \partial G \|_{H^{\varepsilon}(\Sigma_t)}$
and
$\| \partial F \|_{H^{\varepsilon}(\Sigma_t)}$
on the right-hand side were replaced, respectively, with
$\| G \|_{H^{1+\varepsilon}(\Sigma_t)}$
and
$\| F \|_{H^{1+\varepsilon}(\Sigma_t)}$.
\eqref{E:LAMBDAEPSILONPRODOFTHREETERMSL2ESTIMATEINTERMSOFH1PLUSEPSILONANDH1NORMS} is proved as \cite{qW2014b}*{Lemma~18}.
\eqref{E:FREQUENCYPROJECTEDLINFINITYSMOOTHFUNCTIONPRODUCTESTIMATE} follows from
the proof of \cite{qW2017}*{Equation~(8.2)}
and the standard Sobolev embedding estimate 
$\| G \|_{L^6(\Sigma_t)} \lesssim \| G \|_{H^1(\Sigma_t)}$.
\eqref{E:PRODUCTOFTHREEH1FUNCTIONSINL2}
follows from the H\"{o}lder estimate
$
\| 
		G_1 \cdot G_2 \cdot G_3
	\|_{L^2(\Sigma_t)}
\leq
\| G_1 \|_{L^6(\Sigma_t)}
		\| G_2 \|_{L^6(\Sigma_t)}
		\| G_3 \|_{L^6(\Sigma_t)}
$
and the Sobolev embedding estimate 
$\| G_i \|_{L^6(\Sigma_t)} \lesssim \| G_i \|_{H^1(\Sigma_t)}$,
while \eqref{E:PRODUCTOFTWOH1FUNCTIONSINL2}
follows from the H\"{o}lder estimate
$
\| 
		G_1 \cdot G_2 
	\|_{L^2(\Sigma_t)}
\leq
\| G_1 \|_{L^4(\Sigma_t)}
\| G_2 \|_{L^4(\Sigma_t)}
$
and the Sobolev embedding estimate 
$\| G_i \|_{L^4(\Sigma_t)} \lesssim \| G_i \|_{H^1(\Sigma_t)}$.
Similarly \eqref{E:INTERPOLATIONPRODUCTOFTWOH1FUNCTIONSINL2},
follows from the H\"{o}lder estimate
$
\| 
		G_1 \cdot G_2 
\|_{L^2(\Sigma_t)}
\leq
\| G_1 \|_{L^3(\Sigma_t)}
\| G_2 \|_{L^6(\Sigma_t)}
$,
the Sobolev embedding estimate 
$\| G_2 \|_{L^6(\Sigma_t)} \lesssim \| G_2\|_{H^1(\Sigma_t)}$,
and the Sobolev interpolation estimate
$
\| G_1 \|_{L^3(\Sigma_t)}
\lesssim
\| G_1 \|_{L^2(\Sigma_t)}^{1/2}
\| G_1 \|_{H^1(\Sigma_t)}^{1/2}
$.
With the help of the Sobolev embedding result
$
\| \partial \varphi \|_{L^6(\Sigma_t)}
\lesssim
\| \partial \varphi \|_{H^1(\Sigma_t)}$, 
the estimate \eqref{E:LAMBDAEPSILONFGPRODUCTESTIMATEINVOLVINGNOL2NORMOFF}
follows from a straightforward adaptation of the proof of
\cite{qW2017}*{Equation~(8.1)}, which provided a similar estimate in the case $0 < \varepsilon < 1/2$.
The estimates \eqref{E:SUBTRACTLOWFREQUENCIESPRODUCTESTIMATE} and \eqref{E:FREQUENCYOFAPRODUCTMINUSLOWTIMESFREQUENCYESTIMATE}
follow from the proof of \cite{qW2017}*{Lemma~2.4}. To obtain \eqref{E:CONVOLUTIONESTIMATE}, we first observe that 
	$
	\upnu^{1 + \varepsilon}
	\sum_{\uplambda > \upnu}
	\uplambda^{-1}
	A_{\uplambda}
	=
	\sum_{\uplambda > \upnu}
	\left(\frac{\uplambda}{\upnu}\right)^{-(1 + \varepsilon)}
	\uplambda^{\varepsilon}
	A_{\uplambda}
	=
	(\widetilde{A}*B)_{\upnu}
	$,
	where $\widetilde{A}$ denotes the dyadic sequence $\widetilde{A}_{\uplambda} := \uplambda^{\varepsilon} A_{\uplambda}$,  
	$B$ denotes the dyadic sequence $B_{\uplambda} := \mathbf{1}_{[1,\infty)}(\uplambda) \uplambda^{-(1+\varepsilon)}$,
	$\mathbf{1}_{[1,\infty)}(\uplambda)$ denotes the characteristic function of the dyadic interval $[1,\infty)$,
	and $(\widetilde{A}*B)_{\upnu}$ denotes the convolution of $\widetilde{A}$ and $B$, viewed as a function of $\upnu$.
	Thus, from Young's $L^2 * L^1 \rightarrow L^2$ convolution inequality
	and the bound $\| B_{\uplambda} \|_{\ell_{\uplambda}^1} \lesssim 1$,
	we deduce that
	$\| \widetilde{A} * B\|_{\ell_{\upnu}^2} \lesssim \| \widetilde{A}_{\upnu} \|_{\ell_{\upnu}^2}$,
	which is the desired bound.
\end{proof}

\subsubsection{Product and commutator estimates estimates for the compressible Euler equations}
\label{SSS:PRODUCTANDCOMMUTATORFORCOMPRESSIBLEEULER}
In the next lemma, we derive bounds that are sufficient for controlling 
the error terms in the top-order energy-elliptic estimates of Prop.\,\ref{P:TOPORDERENERGYESTIMATES}
and the top-order energy estimates along null hypersurfaces
of Prop.\,\ref{P:ENERGYESTIMATESALONGNULLHYPERSURFACES}.

\begin{lemma}[Product and commutator estimates estimates for the compressible Euler equations]
	\label{L:PRODUCTANDCOMMUTATORFORCOMPRESSIBLEEULER}
	Under the bootstrap assumptions of Sect.\,\ref{S:DATAANDBOOTSTRAPASSUMPTION}
	and the $H^2(\Sigma_t)$ energy estimates of Prop.\,\ref{P:PRELIMINARYENERGYANDELLIPTICESTIMATES},
	for solutions to the equations of Prop.\,\ref{P:GEOMETRICWAVETRANSPORT},
	the inhomogeneous terms from the equations of Lemma~\ref{L:FREQUENCYPROJECTEDEVOLUTION}	
	verify the following estimates for $t \in [0,\Tboot]$,
	where the implicit constants are allowed to depend in a continuous increasing fashion on
	the data norms
	$
	\| (\LogDensity,\vec{v},\vec{\vortrenormalized}) \|_{H^{\Sob}(\Sigma_0)}
	+
	\| \Ent \|_{H^{\Sob+1}(\Sigma_0)}
	$.
	
	\medskip
	
	\noindent \underline{\textbf{Frequency-summed control of the inhomogeneous terms}}:
	The following estimates hold, where in 
	$
	\|
		\cdot
	\|_{\ell_{\upnu}^2 L^2(\Sigma_t)}
	$,
	the $\ell_{\upnu}^2$-seminorm is taken over dyadic frequencies $\upnu \geq 1$:
	\begin{align} \label{E:PRODUCTANDCOMMUTATORESTIMATESFORWAVEEQUATIONS}
		&
		\|
			\upnu^{\Sob-1} \hat{\remainder}_{(\Psi);\upnu}
		\|_{\ell_{\upnu}^2 L^2(\Sigma_t)},
			\,
		\|
			\upnu^{\Sob-1} \remainder_{(\Psi);\upnu}
		\|_{\ell_{\upnu}^2 L^2(\Sigma_t)},
			\\
		&
		\|
			\upnu^{\Sob-2} \partial \hat{\remainder}_{(\Psi);\upnu}
		\|_{\ell_{\upnu}^2 L^2(\Sigma_t)},
			\,
		\|
			\upnu^{\Sob-2} \partial \remainder_{(\Psi);\upnu}
		\|_{\ell_{\upnu}^2 L^2(\Sigma_t)},
		\,
		\|
			\upnu^{\Sob-2} \anotherremainder_{(\pmb{\partial} \Psi);\upnu}
		\|_{\ell_{\upnu}^2 L^2(\Sigma_t)}
			\notag \\
		& \lesssim 
			\| \partial (\vec{\VortVort},\DivGradEnt) \|_{H^{\Sob-2}(\Sigma_t)}
			+
			\left\lbrace
				\| \pmb{\partial} \vec{\Psi} \|_{L^{\infty}(\Sigma_t)} 
				+
				1
			\right\rbrace
			\| \pmb{\partial} \vec{\Psi} \|_{H^{\Sob-1}(\Sigma_t)}
			+
				\| \pmb{\partial} \vec{\Psi} \|_{L^{\infty}(\Sigma_t)} 
				+
				1,
				\notag
		\end{align}

	\begin{align} \label{E:PRODUCTANDCOMMUTATORESTIMATESFORMODVORTVORTANDMODDIVGRADENT}
		&
		\|
			\upnu^{\Sob-1} \remainder_{(\VortVort^i);\upnu}
		\|_{\ell_{\upnu}^2 L^2(\Sigma_t)},
			\,
		\|
			\upnu^{\Sob-1} \remainder_{(\DivGradEnt);\upnu}
		\|_{\ell_{\upnu}^2 L^2(\Sigma_t)},
			\\
		& \|
			\upnu^{\Sob-2} \partial \remainder_{(\VortVort^i);\upnu}
		\|_{\ell_{\upnu}^2 L^2(\Sigma_t)},
			\,
		\|
			\upnu^{\Sob-2} \partial \remainder_{(\DivGradEnt);\upnu}
		\|_{\ell_{\upnu}^2 L^2(\Sigma_t)},
			\notag \\
		&
		\|
			\upnu^{\Sob-2} \anotherremainder_{(\pmb{\partial} \VortVort^i);\upnu}
		\|_{\ell_{\upnu}^2 L^2(\Sigma_t)},
			\,
		\|
			\upnu^{\Sob-2} \anotherremainder_{(\pmb{\partial} \DivGradEnt);\upnu}
		\|_{\ell_{\upnu}^2 L^2(\Sigma_t)}
			\notag \\
		& \lesssim 
			\left\lbrace
				\| \pmb{\partial} \vec{\Psi} \|_{L^{\infty}(\Sigma_t)} 
				+
				1
			\right\rbrace
			\| \partial (\vec{\vortrenormalized}, \vec{\GradEnt}) \|_{H^{\Sob-1}(\Sigma_t)}
			\notag \\
		& \ \
			+
			\left\lbrace
				\| \pmb{\partial} \vec{\Psi} \|_{L^{\infty}(\Sigma_t)} 
				+
				\| \partial (\vec{\vortrenormalized}, \vec{\GradEnt})\|_{L^{\infty}(\Sigma_t)} 
				+
				1
			\right\rbrace
			\| \pmb{\partial} \vec{\Psi} \|_{H^{\Sob-1}(\Sigma_t)}
				\notag \\
		& \ \	
			+
			\| \pmb{\partial} \vec{\Psi} \|_{L^{\infty}(\Sigma_t)}
			+
			1,
			\notag
	\end{align}

	\begin{align} \label{E:PRODUCTANDCOMMUTATORESTIMATESFORMODDIVVORT}
		\|
			\upnu^{\Sob-1} \remainder_{(\dive \vortrenormalized);\upnu}
		\|_{\ell_{\upnu}^2 L^2(\Sigma_t)},
			\,
		\|
			\upnu^{\Sob-2} \partial \remainder_{(\dive \vortrenormalized);\upnu}
		\|_{\ell_{\upnu}^2 L^2(\Sigma_t)}
		& \lesssim 
			\| \pmb{\partial} \vec{\Psi} \|_{H^{\Sob-1}(\Sigma_t)}
			+
			1.
	\end{align}

\noindent \underline{\textbf{Control of $\curl \vortrenormalized$ and $\dive \GradEnt$ in terms of the modified fluid variables}}:
The following estimates hold, where the modified fluid variables 
$\VortVort$ and $\DivGradEnt$ are as in Def.\,\ref{D:MODIFIEDFLUIDVARIABLES}:
	
	\begin{subequations}
	\begin{align} \label{E:TOPORDERESTIMATEFORCURLVORTINTERMSOFMODVORT}
		\| \Lambda^{\Sob-1} \curl \vortrenormalized \|_{L^2(\Sigma_t)} 
		& \lesssim 
		\| \partial \vec{\VortVort} \|_{H^{\Sob-2}(\Sigma_t)} 
		+
		\| \partial \vec{\Psi} \|_{H^{\Sob-1}(\Sigma_t)} 
		+
		1,	
		\\
		\| \Lambda^{\Sob-1} \dive \GradEnt \|_{L^2(\Sigma_t)} 
		& \lesssim 
		\| \partial \DivGradEnt \|_{H^{\Sob-2}(\Sigma_t)} 
		+
		\| \partial \vec{\Psi} \|_{H^{\Sob-1}(\Sigma_t)} 
		+
		1.
		\label{E:TOPORDERESTIMATEFORDIVGRADENTINTERMSOFMODENT}
	\end{align}
	\end{subequations}
\end{lemma}

\begin{proof}
	All of these estimates are standard consequences of Lemma~\ref{L:PRELIMINARYPRODUCTANDCOMMUTATORESTIMATES}
	and we therefore prove only one representative estimate;
	we refer to \cite{qW2017}*{Lemmas~2.2, 2.3, 2.4, and 2.7} for the proof of very similar
	estimates. Specifically, we will prove \eqref{E:PRODUCTANDCOMMUTATORESTIMATESFORWAVEEQUATIONS}.
	Throughout the proof, we use the convention for implicit constants stated in the lemma.
	We will silently use our bootstrap assumption that the compressible Euler solution is contained in $\mathfrak{K}$ 
	(i.e., \eqref{E:BOOTSOLUTIONDOESNOTESCAPEREGIMEOFHYPERBOLICITY}).
	We will also silently use the estimate \eqref{E:LAMBDAONEPLUSEPSILONL2NORMEQUIVALENTTOLAMBDAEPSILONL2NORMOFGRADF},
	the estimates of Prop.\,\ref{P:PRELIMINARYENERGYANDELLIPTICESTIMATES},
	and simple estimates of the type
	$\| \vec{\Psi} \|_{H^{\Sob}(\Sigma_t)} 
	\lesssim 
	\| \vec{\Psi} \|_{H^2(\Sigma_t)} + \| \partial \vec{\Psi} \|_{H^{\Sob - 1}(\Sigma_t)}
	\lesssim 
	1 + \| \partial \vec{\Psi} \|_{H^{\Sob - 1}(\Sigma_t)}
	$,
	the point being that by Prop.\,\ref{P:PRELIMINARYENERGYANDELLIPTICESTIMATES},
	we have already shown that $\| \vec{\Psi} \|_{H^2(\Sigma_t)} \lesssim 1$
	(and similarly for the variables $\vec{\vortrenormalized}$ and $\vec{\GradEnt}$).
	
	In proving \eqref{E:PRODUCTANDCOMMUTATORESTIMATESFORWAVEEQUATIONS},
	we will show only how to obtain the desired bound for the term
	$
	\|
		\upnu^{\Sob-1} \remainder_{(\Psi);\upnu}
	\|_{\ell_{\upnu}^2 L^2(\Sigma_t)}
	$;
	the remaining terms on LHS~\eqref{E:PRODUCTANDCOMMUTATORESTIMATESFORWAVEEQUATIONS}
	can be bounded using nearly identical arguments.
	To proceed, we start by bounding the first term $P_{\upnu} \inhom_{(\Psi)}$
	on RHS~\eqref{E:REMAINDERTERMFREQUENCYPROJECTEDCOVARIANTWAVEVE}. 
	That is, we must bound 
	$\| \upnu^{\Sob-1} P_{\upnu} \mbox{\upshape RHS}~\eqref{E:COVARIANTWAVE} \|_{\ell_{\upnu}^2 L^2(\Sigma_t)}$.
	We begin by bounding the first product on RHS~\eqref{E:COVARIANTWAVE},
	which is of the form $\mathrm{f}(\vec{\Psi}) (\vec{\VortVort},\DivGradEnt)$.
	Repeatedly using the product estimates of Lemma~\ref{L:PRELIMINARYPRODUCTANDCOMMUTATORESTIMATES}
	and appealing to Def.\,\ref{D:MODIFIEDFLUIDVARIABLES},
	we deduce (where throughout, we allow $\mathrm{f}$ to vary from line to line, 
	in particular denoting the derivatives of $\mathrm{f}$ also by $\mathrm{f}$),
	with $\mathscr{P}$ a polynomial with bounded coefficients 
	that is allowed to vary from line to line,
	that
	\begin{align} \label{E:FIRSTSTEPPRODUCTANDCOMMUTATORESTIMATESFORWAVEEQUATIONS}
	\|
		\Lambda^{\Sob-1} [\mathrm{f}(\vec{\Psi}) (\vec{\VortVort},\DivGradEnt)]
	\|_{L^2(\Sigma_t)}
	& 
	\lesssim
	\|
		\Lambda^{\Sob-2}  \partial [\mathrm{f}(\vec{\Psi}) (\vec{\VortVort},\DivGradEnt)]
	\|_{L^2(\Sigma_t)}
			\\
& 
	\lesssim
	\|
		\Lambda^{\Sob-2} [\mathrm{f}(\vec{\Psi}) \partial (\vec{\VortVort},\DivGradEnt)]
	\|_{L^2(\Sigma_t)}
	+
	\|
		\Lambda^{\Sob-2}  [\mathrm{f}(\vec{\Psi}) \partial \vec{\Psi} \cdot (\vec{\VortVort},\DivGradEnt)]
	\|_{L^2(\Sigma_t)}
		\notag	\\
	\notag	\\
& 	\lesssim	
		\left\lbrace
			\| \partial (\vec{\VortVort},\DivGradEnt) \|_{H^{\Sob-2}(\Sigma_t)}
			+
			1
		\right\rbrace
		\mathscr{P}
			\left( 
			\|
				\partial \vec{\Psi} 
			\|_{H^1(\Sigma_t)},
			\|
				(\vec{\VortVort},\DivGradEnt)
			\|_{H^1(\Sigma_t)}
			\right)	
		\notag
			\\
& 	\lesssim	
		\| \partial (\vec{\VortVort},\DivGradEnt) \|_{H^{\Sob-2}(\Sigma_t)}
		+
		1
		\notag
	\end{align}
	as desired. The second product on RHS~\eqref{E:COVARIANTWAVE}
	is of the form $\mathrm{f}(\vec{\Psi}) \cdot \pmb{\partial} \vec{\Psi} \cdot \pmb{\partial} \vec{\Psi}$.
	Thus, using the product estimates of Lemma~\ref{L:PRELIMINARYPRODUCTANDCOMMUTATORESTIMATES}
	and the bound $\| \vec{\Psi} \|_{H^2(\Sigma_t)} \lesssim 1$,
	we deduce that
		\begin{align} \label{E:SECONDSTEPPRODUCTANDCOMMUTATORESTIMATESFORWAVEEQUATIONS}
	\|
		\Lambda^{\Sob-1} [\mathrm{f}(\vec{\Psi}) \cdot \pmb{\partial} \vec{\Psi} \cdot \pmb{\partial} \vec{\Psi}]
	\|_{L^2(\Sigma_t)}
	& 
	\lesssim
	\|
		\Lambda^{\Sob-2}  \partial [\mathrm{f}(\vec{\Psi}) \cdot \pmb{\partial} \vec{\Psi} \cdot \pmb{\partial} \vec{\Psi}]
	\|_{L^2(\Sigma_t)}
			\\
& 
	\lesssim
	\|
		\Lambda^{\Sob-2} [\mathrm{f}(\vec{\Psi}) \pmb{\partial} \vec{\Psi} \cdot \partial \pmb{\partial} \vec{\Psi}]
	\|_{L^2(\Sigma_t)}
	+
	\|
		\Lambda^{\Sob-2}  [\mathrm{f}(\vec{\Psi}) \partial \vec{\Psi} \cdot \pmb{\partial} \vec{\Psi} \cdot \pmb{\partial} \vec{\Psi}]
	\|_{L^2(\Sigma_t)}
		\notag	\\
& 	\lesssim	
		\|
			\mathrm{f}(\vec{\Psi}) \pmb{\partial} \vec{\Psi}
		\|_{L^{\infty}(\Sigma_t)}
		\|
			\pmb{\partial} \vec{\Psi}
		\|_{H^{\Sob-1}(\Sigma_t)}
		+
		\|
			\pmb{\partial} \vec{\Psi} 
		\|_{L^{\infty}(\Sigma_t)}
		\|
			\partial [\mathrm{f}(\vec{\Psi}) \pmb{\partial} \vec{\Psi}]
		\|_{H^{\Sob-2}(\Sigma_t)}
		\notag
			\\
	& \ \
		+
		\|
			\mathrm{f}(\vec{\Psi}) \partial \vec{\Psi} 
		\|_{H^{\Sob-1}(\Sigma_t)}
		\|
			\pmb{\partial} \vec{\Psi} 
		\|_{H^1(\Sigma_t)}^2
		\notag
		    \\
	& \ \
		+
		\|
			\pmb{\partial} \vec{\Psi}
		\|_{H^{\Sob-1}(\Sigma_t)}
		\|
			\mathrm{f}(\vec{\Psi}) \partial \vec{\Psi} 
		\|_{H^1(\Sigma_t)}
		\|
			\pmb{\partial} \vec{\Psi} 
		\|_{H^1(\Sigma_t)}
		\notag
			\\
	& \lesssim
		\|
			\pmb{\partial} \vec{\Psi} 
		\|_{H^{\Sob-1}(\Sigma_t)}
		\left\lbrace
			\|
				\pmb{\partial} \vec{\Psi}
			\|_{L^{\infty}(\Sigma_t)}
			+
			\mathscr{P}
			\left( 
			\|
				\pmb{\partial} \vec{\Psi} 
			\|_{H^1(\Sigma_t)}
			\right)	
		\right\rbrace
			\notag \\
	& \ \
		+
		\left\lbrace
			\|
				\pmb{\partial} \vec{\Psi}
			\|_{L^{\infty}(\Sigma_t)}
			+
			1
		\right\rbrace
		\mathscr{P}
			\left( 
			\|
				\pmb{\partial} \vec{\Psi}
			\|_{H^1(\Sigma_t)}
			\right)	
		\notag
			\\
	& \lesssim
		\|
			\pmb{\partial} \vec{\Psi}
		\|_{H^{\Sob-1}(\Sigma_t)}
		\left\lbrace
			\|
				\pmb{\partial} \vec{\Psi}
			\|_{L^{\infty}(\Sigma_t)}
			+
			1	
		\right\rbrace
		+
		\|
			\pmb{\partial} \vec{\Psi}
		\|_{L^{\infty}(\Sigma_t)}
		+
		1
		\notag
	\end{align}
	as desired.
	It remains for us to bound the two sums on RHS~\eqref{E:REMAINDERTERMFREQUENCYPROJECTEDCOVARIANTWAVEVE} 
	in the norm $\| \upnu^{\Sob-1} \cdot \|_{\ell_{\upnu}^2 L^2(\Sigma_t)}$.
	To handle the first sum, 
	we use \eqref{E:SUBTRACTLOWFREQUENCIESPRODUCTESTIMATE}
	with $\vec{\Psi}$ in the role of $\varphi$ and $\partial \pmb{\partial} \vec{\Psi}$ in the role of $G$
	to deduce
	\begin{align} \label{E:FOURTHSTEPPRODUCTANDCOMMUTATORESTIMATESFORWAVEEQUATIONS}
		\sum_{(\alpha,\beta) \neq (0,0)}
		\|
		\upnu^{\Sob-1}
		\left\lbrace
			(\gfour^{-1})^{\alpha \beta}
			-
			P_{\leq \upnu} (\gfour^{-1})^{\alpha \beta}
		\right\rbrace
		P_{\upnu} \partial_{\alpha} \partial_{\beta} \Psi
		\|_{\ell_{\upnu}^2 L^2(\Sigma_t)}
		& \lesssim
		\| \partial \vec{\Psi} \|_{L^{\infty}(\Sigma_t)} 
		\| 
			\Lambda^{\Sob-2}
			\partial \pmb{\partial} \vec{\Psi} 
		\|_{L^2(\Sigma_t)}
			\\
		& \lesssim
		\| \partial \vec{\Psi} \|_{L^{\infty}(\Sigma_t)}
		\| 
			\pmb{\partial} \vec{\Psi} 
		\|_{H^{\Sob-1}(\Sigma_t)}
		\notag
		\end{align}	
		as desired.
		To bound the last sum on RHS~\eqref{E:REMAINDERTERMFREQUENCYPROJECTEDCOVARIANTWAVEVE}
	in the norm $\| \upnu^{\Sob-1} \cdot \|_{\ell_{\upnu}^2 L^2(\Sigma_t)}$,
	we use \eqref{E:FREQUENCYOFAPRODUCTMINUSLOWTIMESFREQUENCYESTIMATE}
	with $\vec{\Psi}$ in the role of $\varphi$ and $\pmb{\partial} \vec{\Psi}$ in the role of $G$,
	the bound
	$
	\| 
			\Lambda^{\Sob-1} [\gensmoothfunction(\vec{\Psi}) \cdot \pmb{\partial} \vec{\Psi}] 
		\|_{L^2(\Sigma_t)}
		\lesssim
		\|
			\pmb{\partial} \vec{\Psi} 
		\|_{H^{\Sob-1}(\Sigma_t)}
		+
		1
	$
	(which follows from the product estimates of Lemma~\ref{L:PRELIMINARYPRODUCTANDCOMMUTATORESTIMATES}
	and the bound $\| \vec{\Psi} \|_{H^2(\Sigma_t)} \lesssim 1$),
	and the convolution estimate \eqref{E:CONVOLUTIONESTIMATE}
	with 
	$
	\| P_{\uplambda} \partial \pmb{\partial} \vec{\Psi} \|_{L^2(\Sigma_t)}
	$
	in the role of $A_{\uplambda}$
	to deduce
	\begin{align} \label{E:LASTSTEPPRODUCTANDCOMMUTATORESTIMATESFORWAVEEQUATIONS}
		&
		\sum_{(\alpha,\beta) \neq (0,0)}
		\|
		\upnu^{\Sob-1}
		\left\lbrace
		\left(P_{\leq \upnu} (\gfour^{-1})^{\alpha \beta} 
		\right)
		P_{\upnu}
		\partial_{\alpha} \partial_{\beta} \Psi
		-
		P_{\upnu} 
		\left[(\gfour^{-1})^{\alpha \beta} \partial_{\alpha} \partial_{\beta} \Psi
		\right]
		\right\rbrace
		\|_{\ell_{\upnu}^2 L^2(\Sigma_t)}
			\\
		& \lesssim
		\| \pmb{\partial} \vec{\Psi} \|_{L^{\infty}(\Sigma_t)} 
		\| 
			\Lambda^{\Sob-1} [\gensmoothfunction(\vec{\Psi}) \cdot \pmb{\partial} \vec{\Psi}] 
		\|_{L^2(\Sigma_t)}
		\notag \\
	& \ \
				+
				\| \partial \vec{\Psi} \|_{L^{\infty}(\Sigma_t)} 
				\left\|
				\upnu^{\Sob-1}
				\sum_{\uplambda > \upnu}
				\uplambda^{-1}
				\| P_{\uplambda} \partial \pmb{\partial} \vec{\Psi} \|_{L^2(\Sigma_t)}
				\right\|_{\ell_{\upnu}^2}
				\notag
					\\
	& \lesssim
	\| \pmb{\partial} \vec{\Psi} \|_{L^{\infty} (\Sigma_t)} 
	\| 
		\pmb{\partial} \vec{\Psi} 
	\|_{H^{\Sob-1}(\Sigma_t)}
	+
	\| \pmb{\partial} \vec{\Psi} \|_{L^{\infty} (\Sigma_t)} 
		\notag
		\end{align}	
		as desired.
		
		The remaining estimates in the lemma can be proved using similar arguments,
		and we omit the details.
		We clarify that 
		\textbf{i)}
		to derive some of the estimates in their stated form,
		one must use Def.~\ref{D:MODIFIEDFLUIDVARIABLES}
		to express $\vec{\VortVort}$ and $\DivGradEnt$
		in terms of the other solution variables
		and \textbf{ii)}
		in order to bound the term
		$
		\|
			\upnu^{\Sob-2} \anotherremainder_{(\pmb{\partial} \Psi);\upnu}
		\|_{\ell_{\upnu}^2 L^2(\Sigma_t)}
		$
	on LHS~\eqref{E:PRODUCTANDCOMMUTATORESTIMATESFORWAVEEQUATIONS} 
	and the terms
	$
	\|
			\upnu^{\Sob-2} \anotherremainder_{(\pmb{\partial} \VortVort^i);\upnu}
		\|_{\ell_{\upnu}^2 L^2(\Sigma_t)}	
	$
	and
	$
		\|
			\upnu^{\Sob-2} \anotherremainder_{(\pmb{\partial} \DivGradEnt);\upnu}
		\|_{\ell_{\upnu}^2 L^2(\Sigma_t)}
	$
	on LHS~\eqref{E:PRODUCTANDCOMMUTATORESTIMATESFORMODVORTVORTANDMODDIVGRADENT} using arguments of the type given above,
	one must derive Sobolev estimates for products featuring
	the time-derivative-involving terms
	$
	\partial_t^2 \vec{\Psi}
	$,
	$
	\partial_t \vec{\VortVort}
	$,
	$
	\partial_t \DivGradEnt
	$,
	$
	\partial_t \vec{\vortrenormalized}
	$,
	and
	$
	\partial_t \vec{\GradEnt}
	$.
	These time-derivative-involving terms can be handled 
	by first using the equations of Prop.\,\ref{P:GEOMETRICWAVETRANSPORT} 
	to solve for the relevant time derivatives 
	in terms of spatial derivatives and then using 
	the estimates of Lemma~\ref{L:PRELIMINARYPRODUCTANDCOMMUTATORESTIMATES},
	as we did above.
		
\end{proof}

\subsection{ \texorpdfstring{Proof of Proposition~\ref{P:TOPORDERENERGYESTIMATES}}{Proof of Proposition ref  P:TOPORDERENERGYESTIMATES}}
\label{SS:PROOFOFTOPORDERENERGYESTIMATES}
Throughout the proof, we rely on the remarks made in the first paragraph of the proof of
Lemma~\ref{L:PRODUCTANDCOMMUTATORFORCOMPRESSIBLEEULER}. 
In particular, we silently use the already proven below-top-order estimates \eqref{E:BASICH2ENERGYESTIMATE}.
Moreover, we use the convention that our implicit constants are allowed to depend on
functions $F$ of the norms of the data of the type stated on RHS~\eqref{E:TOPORDERENERGYESTIMATES};
in particular, we consider such functions of the norms of the data to be bounded by $\lesssim 1$.
Finally, whenever convenient, we consider factors of $t$ to be bounded by $\lesssim 1$.

We first note that,
for the same reasons stated at the beginning of the proof of Prop.\,\ref{P:PRELIMINARYENERGYANDELLIPTICESTIMATES}, 
the estimates for the terms 
$
\| \partial_t^2 (\LogDensity,\vec{v}) \|_{H^{\Sob-2}(\Sigma_t)}
$,
$
\sum_{k=1}^2
\| \partial_t^k \vec{\vortrenormalized} \|_{H^{\Sob-k}(\Sigma_t)}
$,
$
\sum_{k=1}^2
	\| \partial_t^k \Ent \|_{H^{\Sob+1-k}(\Sigma_t)}
$,
and
$
\| \partial_t (\vec{\VortVort},\DivGradEnt) \|_{H^{\Sob-2}(\Sigma_t)}
$ 
on LHS~\eqref{E:TOPORDERENERGYESTIMATES} follow from straightforward arguments once we have obtained
the desired estimates for the remaining terms on LHS~\eqref{E:TOPORDERENERGYESTIMATES};
we therefore omit the details for bounding these terms.

To prove the desired estimates for the remaining terms on LHS~\eqref{E:TOPORDERENERGYESTIMATES}, 
we will derive energy and elliptic estimates 
for the solution variables at fixed frequency, which satisfy the equations of Lemma~\ref{L:FREQUENCYPROJECTEDEVOLUTION}.
After summing over dyadic frequencies, this will allow us to obtain estimates for the ``controlling quantity'' $\controlling_{\Sob}(t)$ defined by
\begin{align} \label{E:TOPORDERCONTROLLINGQUANTITY}
	\controlling_{\Sob}(t)
	& := 
			\| \pmb{\partial} \vec{\Psi} \|_{H^{\Sob-1}(\Sigma_t)}^2
			+
			\| \partial (\vec{\VortVort}, \DivGradEnt) \|_{H^{\Sob-2}(\Sigma_t)}^2.
\end{align}
Our assumptions on the initial data imply that $\controlling_{\Sob}(0) \lesssim 1$, and we will use this fact throughout the proof.

The main steps in deriving a bound for $\controlling_{\Sob}(t)$ are proving the following two bounds:
\begin{align} \label{E:TOPORDERCONTROLLINGQUANTITYGRONWALLREADY}
	\controlling_{\Sob}(t)
	& \lesssim 
		1
		+
		\int_0^t
			\left\lbrace
				\| \pmb{\partial} \vec{\Psi} \|_{L^{\infty}(\Sigma_{\uptau})}
				+
				\|
					\partial (\vec{\vortrenormalized},\vec{\GradEnt})
				\|_{L^{\infty}(\Sigma_{\uptau})}
				+
				1
			\right\rbrace
			\controlling_{\Sob}(\uptau)
		\, d \uptau,
			\\
	\| \partial  (\vec{\vortrenormalized},\vec{\GradEnt}) \|_{H^{\Sob-1}(\Sigma_t)}^2
	& \lesssim
			\controlling_{\Sob}(t)
			+
			1.
			\label{E:TOPORDERDELLIPTIC}
\end{align}
Then from the bootstrap assumptions \eqref{E:BOOTSTRICHARTZ}--\eqref{E:BOOTL2LINFINITYFIRSTDERIVATIVESOFVORTICITYBOOTANDNENTROPYGRADIENT},
\eqref{E:TOPORDERCONTROLLINGQUANTITYGRONWALLREADY},
and Gr\"{o}nwall's inequality,
we deduce that for $t \in [0,\Tboot]$, we have 
$\controlling_{\Sob}(t) \lesssim 1$.
From this estimate, \eqref{E:TOPORDERDELLIPTIC},
and the below-top-order energy estimates \eqref{E:BASICH2ENERGYESTIMATE},
we conclude, in view of the remarks made above, the desired bound \eqref{E:TOPORDERENERGYESTIMATES}.

It remains for us to prove \eqref{E:TOPORDERCONTROLLINGQUANTITYGRONWALLREADY} and \eqref{E:TOPORDERDELLIPTIC}.
To prove \eqref{E:TOPORDERDELLIPTIC}, we first use the elliptic identity \eqref{E:STANDARDL2DIVCURLESTIMATES}
with $P_{\upnu} \vec{\vortrenormalized}$ and $P_{\upnu} \vec{\GradEnt}$ in the role of $V$
and equations \eqref{E:FREQUENCYPROJECTEDDIVVORTICITY} and \eqref{E:FREQUENCYPROJECTEDCURLGRADENT}
to deduce, after multiplying by $\upnu^{2(\Sob-1)}$ and summing over $\upnu \geq 1$, that
\begin{align} \label{E:FREQUENCYSUMMEDELLIPTIC}
		\| \Lambda^{\Sob-1} \partial (\vec{\vortrenormalized},\vec{\GradEnt}) \|_{L^2(\Sigma_t)}^2
		& =
			\| \upnu^{\Sob-1} \remainder_{(\dive \vortrenormalized);\upnu} \|_{\ell_{\upnu}^2 L^2(\Sigma_t)}^2
			+
			\| \Lambda^{\Sob-1} (\curl \vortrenormalized, \dive \GradEnt) \|_{L^2(\Sigma_t)}^2.
	\end{align}
	Using \eqref{E:PRODUCTANDCOMMUTATORESTIMATESFORMODDIVVORT},
	\eqref{E:TOPORDERESTIMATEFORCURLVORTINTERMSOFMODVORT},
	and \eqref{E:TOPORDERESTIMATEFORDIVGRADENTINTERMSOFMODENT},
	and appealing to definition \eqref{E:TOPORDERCONTROLLINGQUANTITY},
	we find that
	$\mbox{RHS~\eqref{E:FREQUENCYSUMMEDELLIPTIC}} \lesssim \mbox{RHS~\eqref{E:TOPORDERDELLIPTIC}}$.
	Also using Prop.\,\ref{P:PRELIMINARYENERGYANDELLIPTICESTIMATES} to deduce that
	$
	\| P_{\leq 1} \partial (\vec{\vortrenormalized},\vec{\GradEnt}) \|_{L^2(\Sigma_t)}^2 \lesssim 1$,
	we conclude the desired estimate \eqref{E:TOPORDERDELLIPTIC}.

We now derive energy estimates for the evolution equations.
To proceed, we first use equation \eqref{E:FREQUENCYPROJECTEDWAVE}
and \eqref{E:BASICENERGYINEQUALITYFORWAVEEQUATIONS} with $P_{\upnu} \vec{\Psi}$ in the role of $\varphi$
to deduce that
\begin{align} \label{E:WAVEEQUATIONFREQUENCYPROJECTEDENERGYESTIMATE}
			\| (P_{\upnu} \vec{\Psi},P_{\upnu} \partial_t \vec{\Psi}) \|_{H^1(\Sigma_t) \times L^2(\Sigma_t)}^2
			& \lesssim 
				\| (P_{\upnu} \vec{\Psi},P_{\upnu} \partial_t \vec{\Psi}) \|_{H^1(\Sigma_0) \times L^2(\Sigma_0)}^2
					\\
			& \ \
				+
				\int_0^t
						\|
							\pmb{\partial} \vec{\Psi} 
						\|_{L^{\infty}(\Sigma_{\uptau})}
						\| (P_{\upnu} \vec{\Psi},P_{\upnu} \partial_t \vec{\Psi}) \|_{H^1(\Sigma_{\uptau}) \times L^2(\Sigma_{\uptau})}^2
				\, d \uptau
				\notag	\\
			& \ \
				+
				\sum_{\iota=0}^4
				\int_0^t
					\| \hat{\remainder}_{(\Psi_{\iota});\upnu} \|_{L^2(\Sigma_{\uptau})} 
					\|\pmb{\partial} P_{\upnu} \vec{\Psi} \|_{L^2(\Sigma_{\uptau})}
				\, d \uptau.
				\notag
	\end{align}
	Multiplying \eqref{E:WAVEEQUATIONFREQUENCYPROJECTEDENERGYESTIMATE} by $\upnu^{2(\Sob-1)}$, summing over dyadic frequencies $\upnu \geq 1$, 
	using the Cauchy--Schwarz inequality for
	$\ell_{\upnu}^2$, using \eqref{E:PRODUCTANDCOMMUTATORESTIMATESFORWAVEEQUATIONS}, 
	and using Young's inequality,
	we deduce, in view of definition \eqref{E:TOPORDERCONTROLLINGQUANTITY}, that
	\begin{align} \label{E:WAVEEQUATIONFREQUENCYSUMMEDENERGYESTIMATE}
			\| \pmb{\partial} \vec{\Psi} \|_{H^{\Sob-1}(\Sigma_t)}^2
			& \lesssim 
				1 
				+
				\int_0^t
					\|
						\pmb{\partial} \vec{\Psi} 
					\|_{L^{\infty}(\Sigma_{\uptau})}
				\, d \uptau
				+
				\int_0^t
					\left\lbrace
						\|
							\pmb{\partial} \vec{\Psi} 
						\|_{L^{\infty}(\Sigma_{\uptau})}
						+
						1
					\right\rbrace
					\controlling_{\Sob}(\uptau) \, d\uptau.
		\end{align}
Similarly, using equations \eqref{E:FREQUENCYPROJECTEDTRANSPORTVORTICITYVORTICITY}, \eqref{E:FREQUENCYPROJECTEDTRANSPORTDIVGRADENTROPY},
and \eqref{E:BASICENERGYINEQUALITYFORTRANSPORTEQUATIONS} with $P_{\upnu} \vec{\VortVort}$ and $P_{\upnu} \DivGradEnt$
in the role of $\varphi$, we deduce that
\begin{align} \label{E:ENERGYINEQUALITYFORFREQUENCYPROJECTEDTRANSPORTEQUATIONS}
			\| (P_{\upnu} \vec{\VortVort},P_{\upnu} \DivGradEnt) \|_{L^2(\Sigma_t)}^2
			& \lesssim 
				\| (P_{\upnu} \vec{\VortVort},P_{\upnu} \DivGradEnt) \|_{L^2(\Sigma_0)}^2
					\\
			& \ \
				+
				\int_0^t
					\| \pmb{\partial} \vec{\Psi} \|_{L^{\infty}(\Sigma_{\uptau})} 
					\| (P_{\upnu} \vec{\VortVort},P_{\upnu} \DivGradEnt) \|_{L^2(\Sigma_{\uptau})}^2
				\, d \uptau
					\notag \\
			& \ \
				+
				\sum_{i=1}^3
				\int_0^t
					\|  P_{\upnu} \VortVort^i \|_{L^2(\Sigma_{\uptau})} \| \remainder_{(\VortVort^i);\upnu}\|_{L^2(\Sigma_{\uptau})}
				\, d \uptau
					\notag \\
			& \ \
				+
				\int_0^t
					\|  P_{\upnu} \DivGradEnt \|_{L^2(\Sigma_{\uptau})} \| \remainder_{(\DivGradEnt);\upnu}\|_{L^2(\Sigma_{\uptau})}
				\, d \uptau.
				\notag
	\end{align}
	Multiplying \eqref{E:ENERGYINEQUALITYFORFREQUENCYPROJECTEDTRANSPORTEQUATIONS} by $\upnu^{2(\Sob-1)}$, 
	summing over dyadic frequencies $\upnu \geq 1$,
	using the Cauchy--Schwarz inequality for
	$\ell_{\upnu}^2$, using \eqref{E:PRODUCTANDCOMMUTATORESTIMATESFORMODVORTVORTANDMODDIVGRADENT}, 
	using \eqref{E:TOPORDERDELLIPTIC} to bound the factor
	$
	\| \partial  (\vec{\vortrenormalized},\vec{\GradEnt}) \|_{H^{\Sob-1}(\Sigma_t)}
	$
	on RHS~\eqref{E:PRODUCTANDCOMMUTATORESTIMATESFORMODVORTVORTANDMODDIVGRADENT},
	and using Young's inequality,
	we deduce, in view of definition \eqref{E:TOPORDERCONTROLLINGQUANTITY}, that
	\begin{align} \label{E:TOPORDERENERGYINEQUALITYFORMODVARS}
			\| \partial (\vec{\VortVort}, \DivGradEnt) \|_{H^{\Sob-2}(\Sigma_t)}^2
			& \lesssim 
				1
				+
				\int_0^t	
					\| \pmb{\partial} \vec{\Psi} \|_{L^{\infty}(\Sigma_{\uptau})}
				\, d \uptau \\
		& \ \
				+
				\int_0^t
			\left\lbrace
				\| \pmb{\partial} \vec{\Psi} \|_{L^{\infty}(\Sigma_{\uptau})}
				+
				\|
					\partial (\vec{\vortrenormalized}, \vec{\GradEnt})
				\|_{L^{\infty}(\Sigma_{\uptau})}
				+
				1
			\right\rbrace
			\controlling_{\Sob}(\uptau)
			\, d \uptau.
			\notag
	\end{align}
	Finally, adding \eqref{E:WAVEEQUATIONFREQUENCYSUMMEDENERGYESTIMATE} and \eqref{E:TOPORDERENERGYINEQUALITYFORMODVARS},
	and controlling the second term on RHS~\eqref{E:TOPORDERENERGYINEQUALITYFORMODVARS} by
	using the bootstrap assumption \eqref{E:BOOTSTRICHARTZ} to infer that
	$
	\int_0^t	
					\| \pmb{\partial} \vec{\Psi} \|_{L^{\infty}(\Sigma_{\uptau})}
				\, d \uptau
	\lesssim 1
	$,
	we conclude \eqref{E:TOPORDERCONTROLLINGQUANTITYGRONWALLREADY}. 
	We have therefore proved the proposition.
	
	\hfill $\qed$

\section{Energy estimates along acoustic null hypersurfaces}
\label{S:ENERGYESTIMATESALONGNULLHYPERSURFACES}
Our main goal in this section is to derive energy estimates
for the fluid variables along acoustic null hypersurfaces 
(which we sometimes refer to as ``$\gfour$-null hypersurfaces'' to clarify their tie to the acoustical metric,
 or simply ``null hypersurfaces'' for short). 
We will use these estimates in Sect.\,\ref{S:ESTIMATESFOREIKONALFUNCTION},
when we derive quantitative control of the acoustic geometry
(for example, in the proof of Prop.\,\ref{P:ESTIMATESFORFLUIDVARIABLES}).
Compared to prior works, the main contribution of the present section is the estimate
\eqref{E:SOUNDCONEENERGYESTIMATESFORMODIFIEDVARIABLES},
which shows that the modified fluid variables
$(\vec{\VortVort},\DivGradEnt)$
can be controlled in $L^2$ up to top-order along acoustic null hypersurfaces;
as we described in point \textbf{I} of Subsect.\,\ref{SS:CONTROLOFTIMEOFEXISTENCE},
\emph{such control along acoustic null hypersurfaces is not available 
for generic top-order derivatives of the vorticity and entropy}.

\subsection{Geometric ingredients}
\label{SS:GEOMETRICINGREDIENTS}
We assume that in some subset of $[0,\Tboot] \times \mathbb{R}^3$ equal to the closure of an open set, 
$U$ is an acoustical eikonal function. More precisely, we assume that 
$U$ is a solution to the eikonal equation $(\gfour^{-1})^{\alpha \beta} \partial_{\alpha} U \partial_{\beta} U = 0$
such that $\partial_t U > 0$ and such that $U$ is smooth and non-degenerate (i.e. $|\pmb{\partial} U| \neq 0$) 
away from the integral curve of $\Transport$ emanating
from a point ${\bf{z}} \in \Sigma_T$ for some $T \in [0,\Tboot]$; see Subsect.~\ref{SS:EIKONAL} and Subsubsect.~\ref{SSS:EIKONALINTERIOR} for discussion of our choice of ${\bf{z}}$ and the integral curve.

In Sect.\,\ref{S:SETUPCONSTRUCTIONOFEIKONAL}, 
we will construct a related eikonal function, one that is equivalent to the eikonal functions considered here,
differing only in that we work with rescaled solution
variables starting in Sect.\,\ref{S:SETUPCONSTRUCTIONOFEIKONAL} (see Subsect.\,\ref{SS:RESCALEDSOLUTION} for their definition).
We let $l :=  \frac{-1}{(\gfour^{-1})^{\alpha \beta} \partial_{\alpha} U \partial_{\beta} t} > 0$
denote the null lapse\footnote{We use the symbol ``$\nulllapse$'' to denote the null lapse of the eikonal function
constructed in Sect.\,\ref{S:SETUPCONSTRUCTIONOFEIKONAL}. Moreover, starting in Sect.\,\ref{S:SETUPCONSTRUCTIONOFEIKONAL}, 
we use the symbol ``$\Lunit$'' to denote the analog of the vectorfield denoted by ``$V$'' in the present subsection. \label{FN:CHANGEOFNOTATION}}, 
and we define $V^{\alpha} := - l (\gfour^{-1})^{\alpha \beta} \partial_{\beta} U$.
Thus, $\gfour(V,V)= 0$ and $V t = 1$. We assume that the hypersurface $\mathcal{N}$ is equal to some portion of a level set of $U$.
Note that $V$ is normal to $\mathcal{N}$ and thus $\mathcal{N}$ is a $\gfour$-null hypersurface. 
We define the two-dimensional spacelike surfaces $\mathcal{S}_t := \Sigma_t \cap \mathcal{N}$.
We let $\gsphere$ denote the Riemannian metric induced by $\gfour$ on $\mathcal{S}_t$, 
we let $\angD$ denote the corresponding Levi-Civita connection, and we let
$d \varpi_{\gsphere}$ denote the volume form on $\mathcal{S}_t$ induced by $\gsphere$.

We now define acoustic null fluxes along $\mathcal{N}$.

\begin{definition}[Acoustic null fluxes]
\label{D:SOUNDCONEFLUX}
For scalar functions $\varphi$ defined on $\mathcal{N}$, 
we define the acoustic null fluxes $\mathbb{F}_{(Wave)}[\varphi;\mathcal{N}]$ 
and $\mathbb{F}_{(Transport)}[\varphi;\mathcal{N}]$
as follows, where relative to 
arbitrary coordinates on $\mathcal{S}_t$, $|\angD \varphi|_{\gsphere}^2 := (\gsphere^{-1})^{AB} \angDarg{A} \varphi \angDarg{B} \varphi$:
\begin{align} \label{E:SOUNDCONEFLUXES}
	\mathbb{F}_{(Wave)}[\varphi;\mathcal{N}]
	& := 
		\int_{\mathcal{N}}
			\left\lbrace
				(V \varphi)^2
				+
				|\angD \varphi|_{\gsphere}^2
			\right\rbrace
			\, d \varpi_{\gsphere} d t,
	&
	\mathbb{F}_{(Transport)}[\varphi;\mathcal{N}]
	& := 
		\int_{\mathcal{N}}
			\varphi^2
		 \, d \varpi_{\gsphere} d t.
\end{align}
\end{definition}

\subsection{Energy estimates along acoustic null hypersurfaces}
\label{SS:ENERGYESTIMATESALONGNULLL}
In this subsection, we establish the main energy estimate for the fluid solution variables along null hypersurfaces.
As we mentioned at the start of Sect.\,\ref{S:ENERGYESTIMATESALONGNULLHYPERSURFACES},
the main new ingredient of interest is \eqref{E:SOUNDCONEENERGYESTIMATESFORMODIFIEDVARIABLES},
whose proof relies on the special structure of the equations of Prop.\,\ref{P:GEOMETRICWAVETRANSPORT}.
In Sect.\,\ref{S:ESTIMATESFOREIKONALFUNCTION}, we will apply
Prop.\,\ref{P:ENERGYESTIMATESALONGNULLHYPERSURFACES} 
along a family of null hypersurfaces that are equal to the level sets of an acoustical eikonal function
that we construct in Subsect.\,\ref{SS:EIKONAL} 
(we denote the acoustical eikonal function by ``$u$'' starting in Sect.\,\ref{S:SETUPCONSTRUCTIONOFEIKONAL}).

\begin{proposition}[Energy estimates along acoustic null hypersurfaces]
	\label{P:ENERGYESTIMATESALONGNULLHYPERSURFACES}	
		Let $\mathcal{N}$ be any of the null hypersurface portions from Subsect.\,\ref{SS:GEOMETRICINGREDIENTS}.
		Assume that for some pair of times $0 \leq t_I < t_F \leq \Tboot$, $\mathcal{N}$ and some subsets of
		$\Sigma_{t_I}$ and $\Sigma_{t_F}$ collectively form the boundary a compact subset of $[0,\Tboot] \times \mathbb{R}^3$.
		Then under the initial data and bootstrap assumptions of Sect.\,\ref{S:DATAANDBOOTSTRAPASSUMPTION}
		and the conclusions of Prop.\,\ref{P:TOPORDERENERGYESTIMATES}, 
		the following estimates hold for $\Psi \in \lbrace \LogDensity, v^1, v^2, v^3, \Ent \rbrace$:
		\begin{align}
			\mathbb{F}_{(Wave)}[\pmb{\partial} \Psi;\mathcal{N}]
			+
			\sum_{\upnu > 1}
			\upnu^{2(\Sob-2)}
			\mathbb{F}_{(Wave)}[P_{\upnu} \pmb{\partial} \Psi;\mathcal{N}]
			& \lesssim 1.
			\label{E:SOUNDCONEENERGYESTIMATESFORWAVEVARIABLES}
		\end{align}
		
		Moreover, 
		\begin{align} \label{E:SOUNDCONEENERGYESTIMATESFORMODIFIEDVARIABLES}
			\mathbb{F}_{(Transport)}[\pmb{\partial}(\vec{\VortVort},\DivGradEnt);\mathcal{N}]
			+
			\sum_{\upnu > 1}
			\upnu^{2(\Sob-2)}
			\mathbb{F}_{(Transport)}[P_{\upnu} \pmb{\partial} (\vec{\VortVort},\DivGradEnt);\mathcal{N}]
			& \lesssim 1.
		\end{align}
\end{proposition}

\begin{proof}
	We first prove \eqref{E:SOUNDCONEENERGYESTIMATESFORMODIFIEDVARIABLES} for $\pmb{\partial} \VortVort^i$.
	We set $\mathbf{J}^{\alpha} := |\pmb{\partial} \VortVort^i|^2 \Transport^{\alpha}$ and compute, 
	relative to the Cartesian coordinates, that 
	$\Dfour_{\alpha} \mathbf{J}^{\alpha} = 2 (\pmb{\partial} \VortVort^i) \cdot \Transport \pmb{\partial} \VortVort^i 
	+ 
	(\partial_a v^a) |\pmb{\partial} \VortVort^i|^2
	+ 
	\Chfour_{\alpha \ \beta}^{\ \alpha} \Transport^{\beta} |\pmb{\partial} \VortVort^i|^2
	$,
	where 
	$
	\Chfour_{\alpha \ \beta}^{\ \gamma}
	=
	\frac{1}{2} (\gfour^{-1})^{\gamma \sigma}
	\left\lbrace\partial_{\alpha} \gfour_{\sigma \beta} 
	+ \partial_{\beta} \gfour_{\alpha \sigma}
	-
	\partial_{\sigma} \gfour_{\alpha \beta}
	\right\rbrace
	$
	are the Cartesian Christoffel symbols of $\gfour$.
	From the constructions carried out Subsect.\,\ref{SS:GEOMETRICINGREDIENTS}, we find that
	$\gfour(\Transport,V)
	=
	- Vt
	=
	- 1
	$
	and thus
	$\gfour(\mathbf{J},V)
	=
	- 	 |\pmb{\partial} \VortVort^i|^2 
	$.
	Note also that since $\gfour(\Transport,\Transport) = -1$, we have
	$\gfour(\mathbf{J},\Transport)
	=
	- 	 |\pmb{\partial} \VortVort^i|^2 
	$.
	We now apply the divergence theorem (where the Riemannian volume forms are induced by $\gfour$) using the vectorfield $\mathbf{J}^{\alpha}$
	on the compact spacetime region 
	bounded by $\Sigma_{t_I}$, $\Sigma_{t_F}$, and $\mathcal{N}$.
	Considering also the fact that $\Chfour_{\alpha \ \beta}^{\ \alpha} = \gensmoothfunction(\vec{\Psi}) \pmb{\partial} \vec{\Psi}$,
	we arrive at the following inequality for $\pmb{\partial} \VortVort^i$:
	\begin{align}
		\int_{\mathcal{N}}
			|\pmb{\partial} \VortVort^i|^2
		\, d \varpi_{\gsphere} d t
		& =
		-
		\int_{\mathcal{N}}
			\gfour(\mathbf{J},V)
		\, d \varpi_{\gsphere} d t
			\label{E:ENERGYESTIMATEALONGSOUNDCONES} \\
		& \lesssim
		\int_{\Sigma_{t_I}}
			|\gfour(\mathbf{J},\Transport)|
		\, d \varpi_g
		+
		\int_{\Sigma_{t_F}}
			|\gfour(\mathbf{J},\Transport)|
		\, d \varpi_g
			\notag \\
		& \ \
		+
		\int_{t_{I}}^{t_F}
		\int_{\Sigma_{\uptau}}
			|\pmb{\partial} \VortVort^i| 
			|\Transport \pmb{\partial} \VortVort^i|
		\, d \varpi_g
		\, d \uptau
		+
		\int_{t_{I}}^{t_F}
		\int_{\Sigma_{\uptau}}
			\|\pmb{\partial} \vec{\Psi} \|_{L^{\infty}(\Sigma_{\uptau})}
			|\pmb{\partial} \VortVort^i |^2
		\, d \varpi_g
		\, d \uptau,
		\notag
	\end{align}
	where $d \varpi_g$ is the volume form induced on constant-time hypersurfaces by their first fundamental form $g$.
	Here we clarify that the normalization condition $V t = 1$ has the following virtue:
	it guarantees that the volume element on $\mathcal{N}$ appearing in the divergence theorem
	is precisely $d \varpi_{\gsphere} d t$.
	From the energy estimates of Prop.\,\ref{P:TOPORDERENERGYESTIMATES},
	we deduce that the two integrals $\int_{\Sigma_{t_I}} \cdots$ and $\int_{\Sigma_{t_F}} \cdots$ on RHS~\eqref{E:ENERGYESTIMATEALONGSOUNDCONES}
	are $\lesssim 1$. 
	Next, commuting the evolution equation \eqref{E:TRANSPORTVORTICITYVORTICITY} with $\pmb{\partial}$,
	using the resulting expression to substitute for the factor $\Transport \pmb{\partial} \VortVort^i$ on RHS~\eqref{E:ENERGYESTIMATEALONGSOUNDCONES},
	using the bootstrap assumptions
	and the energy estimates of Prop.\,\ref{P:TOPORDERENERGYESTIMATES},
	and using the Cauchy--Schwarz and Young's inequalities,
	we deduce that the two integrals 
	$\int_{\Sigma_{\uptau}}
			\cdots
	$
	on $\mbox{\upshape RHS~\eqref{E:ENERGYESTIMATEALONGSOUNDCONES}}$
	are 
	$\lesssim 
		1 + \| \pmb{\partial} \vec{\Psi} \|_{L^{\infty}(\Sigma_{\uptau})}^2
		+
		 \|\partial (\vec{\vortrenormalized},\vec{\GradEnt}) \|_{L^{\infty}(\Sigma_{\uptau})}^2
	$.
	Also using the bootstrap assumptions \eqref{E:BOOTSTRICHARTZ}--\eqref{E:BOOTL2LINFINITYFIRSTDERIVATIVESOFVORTICITYBOOTANDNENTROPYGRADIENT}, 
	we see that $\mbox{\upshape RHS~\eqref{E:ENERGYESTIMATEALONGSOUNDCONES}} \lesssim 1$,
	which, in view of definition \eqref{E:SOUNDCONEFLUXES}, 
	yields the desired bound
	$
	\mathbb{F}_{(Transport)}[\pmb{\partial} \vec{\VortVort};\mathcal{N}]
	\lesssim 1
	$.
	
	To obtain the desired bound for the sum 
	on LHS~\eqref{E:SOUNDCONEENERGYESTIMATESFORMODIFIEDVARIABLES}
	involving the terms $P_{\upnu} \pmb{\partial} \VortVort$,
	we repeat the above argument with $P_{\upnu} \pmb{\partial} \VortVort^i$ in the role of $\pmb{\partial} \VortVort^i$.
	Considering also the evolution equation \eqref{E:TIMEDERIVATIVECOMMUTEDMODIFIEDVORTFREQUENCYPROJECTED},
	we obtain the following bound:
	\begin{align}
		\mathbb{F}_{(Transport)}[P_{\upnu} \pmb{\partial} \VortVort^i;\mathcal{N}]
		& \lesssim
		\int_{\Sigma_{t_I}}
			|P_{\upnu} \pmb{\partial}  \VortVort^i|^2
		\, d \varpi_g
		+
		\int_{\Sigma_{t_F}}
			|P_{\upnu} \pmb{\partial} \VortVort^i|^2
		\, d \varpi_g
			\label{E:FIXEDFREQUENCYENERGYESTIMATEALONGSOUNDCONES} \\
		& \ \
		+
		\int_{t_{I}}^{t_F}
		\int_{\Sigma_{\uptau}}
			|P_{\upnu} \pmb{\partial} \VortVort^i| 
			|\remainder_{(\pmb{\partial} \VortVort^i);\upnu}|
		\, d \varpi_g
		\, d \uptau
		+
		\int_{t_{I}}^{t_F}
		\int_{\Sigma_{\uptau}}
			\|
				\pmb{\partial} \vec{\Psi} 
			\|_{L^{\infty}(\Sigma_{\uptau})}
			|P_{\upnu} \pmb{\partial}  \VortVort^i |^2
		\, d \varpi_g
		\, d \uptau.
		\notag
	\end{align}
	Multiplying \eqref{E:FIXEDFREQUENCYENERGYESTIMATEALONGSOUNDCONES} by $\upnu^{2(\Sob-2)}$,
	summing over $\upnu > 1$,
	using the estimate \eqref{E:PRODUCTANDCOMMUTATORESTIMATESFORMODVORTVORTANDMODDIVGRADENT}
	and the Cauchy--Schwarz inequality for $L^2(\Sigma_{\uptau})$ and $\ell_{\upnu}^2$,
	and using the energy estimates of Prop.\,\ref{P:TOPORDERENERGYESTIMATES}
	and the bootstrap assumptions \eqref{E:BOOTSTRICHARTZ}--\eqref{E:BOOTL2LINFINITYFIRSTDERIVATIVESOFVORTICITYBOOTANDNENTROPYGRADIENT},
	we conclude that $\mbox{RHS}~\eqref{E:FIXEDFREQUENCYENERGYESTIMATEALONGSOUNDCONES} \lesssim 1$ as desired.
	
	The estimate \eqref{E:SOUNDCONEENERGYESTIMATESFORMODIFIEDVARIABLES} for the terms involving $\DivGradEnt$
	can be obtained in a similar fashion
	with the help of the evolution equations \eqref{E:TRANSPORTDIVGRADENTROPY} and \eqref{E:TIMEDERIVATIVECOMMUTEDMODIFIEDENTFREQUENCYPROJECTED},
	and we omit the details.
	
	The estimate \eqref{E:SOUNDCONEENERGYESTIMATESFORWAVEVARIABLES} can be obtained using similar arguments,
	with a few minor adjustments that we now describe.
	To bound the first term on LHS~\eqref{E:SOUNDCONEENERGYESTIMATESFORWAVEVARIABLES},
	we apply the divergence theorem with the vectorfield $\Jenarg{\Transport}{\alpha}[\pmb{\partial} \Psi]$ defined by \eqref{E:MULTIPLIERVECTORFIELD}.
	The integrand appearing on the analog of LHS~\eqref{E:ENERGYESTIMATEALONGSOUNDCONES} is $\gfour(\Jen{\Transport},V)$,
	which through standard arguments (for example, using a null frame as in Subsubsect.\,\ref{SSS:NULLFRAME}) can be shown to be equal to
	$
	\frac{1}{2}
	\left\lbrace
	|V \pmb{\partial} \Psi|^2
	+
	|\angD \pmb{\partial} \Psi|_{\gsphere}^2
	\right\rbrace$,
	that is, equal to the integrand in the definition \eqref{E:SOUNDCONEFLUXES}
	of $\mathbb{F}_{(Wave)}[\pmb{\partial} \Psi;\mathcal{N}]$ (aside from the factor of $1/2$).
	The spacetime error integrals appearing on the analog of RHS~\eqref{E:ENERGYESTIMATEALONGSOUNDCONES}
	have integrands equal to RHS~\eqref{E:DIVERGENCEOFENERGYCURRENT} 
	(with $\Transport$ in the role of $\mathbf{X}$),
	where one commutes the wave equation \eqref{E:COVARIANTWAVE}
	with $\pmb{\partial}$ to obtain algebraic expressions for
	$\square_{\gfour} \pmb{\partial} \Psi$. One can then argue as we did above to show that the error integrals are $\lesssim 1$ as desired.
	To bound the sum on LHS~\eqref{E:SOUNDCONEENERGYESTIMATESFORWAVEVARIABLES},
	we can use a similar argument based on the wave equation
	\eqref{E:TIMEDERIVATIVECOMMUTEDWAVEFREQUENCYPROJECTED}
	and the estimate
	\eqref{E:PRODUCTANDCOMMUTATORESTIMATESFORWAVEEQUATIONS}.

\end{proof}

\section{Strichartz estimates for the wave equation and control of H\"{o}lder norms of the wave variables}
\label{S:STRICHARTZESTIMATESFORWAVEUPGRADEDTOHOLDER}
The main results of this section are Theorem~\ref{T:IMPROVEMENTOFSTRICHARTZBOOTSTRAPASSUMPTION},
which yields a strict improvement of the Strichartz-type bootstrap assumption \eqref{E:BOOTSTRICHARTZ} for the wave variables,
and Corollary~\ref{C:HOLDERTYPESTRICHARTZESTIMATEFORWAVEVARIABLES}.
Our proof of Theorem~\ref{T:IMPROVEMENTOFSTRICHARTZBOOTSTRAPASSUMPTION}
relies on a frequency-localized Strichartz estimate provided by 
Theorem~\ref{T:FREQUENCYLOCALIZEDSTRICHARTZ}.
We outline the proof of Theorem~\ref{T:FREQUENCYLOCALIZEDSTRICHARTZ} in
Sect.\,\ref{S:REDUCTIONSOFSTRICHARTZ}; given the estimates for the acoustic geometry that we derive in Sect.\,\ref{S:ESTIMATESFOREIKONALFUNCTION},
the proof of Theorem~\ref{T:FREQUENCYLOCALIZEDSTRICHARTZ} 
is essentially the same as the proof of an analogous frequency-localized Strichartz estimate 
featured in \cite{qW2017}.

\begin{remark}[Reminder concerning the various parameters]
	Our analysis in this section extensively refers to the collection of parameters from Subsect.\,\ref{SS:PARAMETERS}.
\end{remark}

\subsection{\texorpdfstring{Statement of Theorem~\ref{T:IMPROVEMENTOFSTRICHARTZBOOTSTRAPASSUMPTION} and proof of Corollary~\ref{C:HOLDERTYPESTRICHARTZESTIMATEFORWAVEVARIABLES}}{Statement of Theorem ref T:IMPROVEMENTOFSTRICHARTZBOOTSTRAPASSUMPTION and proof of Corollary ref C:HOLDERTYPESTRICHARTZESTIMATEFORWAVEVARIABLES}}
We now provide the main results of Sect.\,\ref{S:STRICHARTZESTIMATESFORWAVEUPGRADEDTOHOLDER}, 
starting with Theorem~\ref{T:IMPROVEMENTOFSTRICHARTZBOOTSTRAPASSUMPTION}.
The proof of the theorem is located in Subsect.\,\ref{SS:PROOFOFMAINSTRICHARTZWAVEESTIMATES}.

\begin{theorem}[Improvement of the Strichartz-type bootstrap assumption for the wave variables]
\label{T:IMPROVEMENTOFSTRICHARTZBOOTSTRAPASSUMPTION}
If $\updelta > 0$ is sufficiently small, then
under the initial data and bootstrap assumptions of Sect.\,\ref{S:DATAANDBOOTSTRAPASSUMPTION}, 
the following estimate for the wave variables
$\vec{\Psi} = (\LogDensity,v^1,v^2,v^3,\Ent)$
holds, where $\updelta_1$ is defined by \eqref{E:DELTA1DEFINITIONANDPOSITIVITY}:
\begin{align} \label{E:IMPROVEMENTOFSTRICHARTZBOOTSTRAPASSUMPTION}
		\|
			\pmb{\partial} \vec{\Psi}
		\|_{L^2([0,\Tboot]) L_x^{\infty}}^2
		+
		\sum_{\upnu \geq 2}
		\upnu^{2 \updelta_1}
		\|
			P_{\upnu} \pmb{\partial} \vec{\Psi}
		\|_{L^2([0,\Tboot])L_x^{\infty}}^2
	& \lesssim
	\Tboot^{2 \updelta}.
\end{align}

\end{theorem}

The second main result of this section is the following corollary, 
which is a simple consequence of Theorem~\ref{T:IMPROVEMENTOFSTRICHARTZBOOTSTRAPASSUMPTION}.
It plays a fundamental role in Sect.\,\ref{S:ELLIPTICESTIMATESINHOLDERSPACES},
when we derive Schauder estimates for $\vortrenormalized$ and $\GradEnt$.

\begin{corollary}[Strichartz-type estimate with a H\"{o}lder spatial norm for the wave variables] 
	\label{C:HOLDERTYPESTRICHARTZESTIMATEFORWAVEVARIABLES}
	Under the assumptions and conclusions of Theorem~\ref{T:IMPROVEMENTOFSTRICHARTZBOOTSTRAPASSUMPTION},
	the following estimate holds for the wave variable array
	$\vec{\Psi} = (\LogDensity,v^1,v^2,v^3,\Ent)$:
	\begin{align} \label{E:HOLDERTYPESTRICHARTZESTIMATEFORWAVEVARIABLES}
	\|
		\pmb{\partial} \vec{\Psi}
	\|_{L^2([0,\Tboot])C_x^{0,\updelta_1}}^2
		+
		\sum_{\upnu \geq 2}
		\|
			P_{\upnu} \pmb{\partial} \vec{\Psi}
		\|_{L^2([0,\Tboot])C_x^{0,\updelta_1}}^2
	& \lesssim
	\Tboot^{2 \updelta}.
\end{align}
\end{corollary}

\begin{proof}[Discussion of proof]
Given Theorem~\ref{T:IMPROVEMENTOFSTRICHARTZBOOTSTRAPASSUMPTION}, 
Cor.\,\ref{C:HOLDERTYPESTRICHARTZESTIMATEFORWAVEVARIABLES} 
follows from standard results in harmonic analysis;
see, for example,
\cite{MT1991}*{Equation~(A.1.5)} and the discussion surrounding it.
\end{proof}

\subsection{Partitioning of the bootstrap time interval}
\label{SS:PARTITIONOFBOOTSTRAPINTERVAL}
In proving Theorem~\ref{T:IMPROVEMENTOFSTRICHARTZBOOTSTRAPASSUMPTION},
we will follow the strategy of \cite{qW2017} by constructing an appropriate partition of the bootstrap time interval $[0,\Tboot]$.
The partition refers to a parameter $\Lambda_0$,
where in the rest of the paper, $\Lambda_0 \gg 1$ denotes a dyadic frequency that 
is chosen to be sufficiently large (we adjust the largeness of $\Lambda_0$ as needed throughout the course of the analysis).
In view of the bootstrap assumptions 
\eqref{E:BOOTSTRICHARTZ}--\eqref{E:BOOTL2LINFINITYFIRSTDERIVATIVESOFVORTICITYBOOTANDNENTROPYGRADIENT},
it is straightforward to see
that for $\uplambda \geq \Lambda_0$,
we can partition $[0,\Tboot]$ into intervals $[t_k,t_{k+1}]$ of length
$|t_{k+1} - t_k| \leq \uplambda^{- 8 \upepsilon_0} \Tboot$
such that the total number of intervals is $\approx \uplambda^{8 \upepsilon_0}$
and such that
\begin{subequations}
\begin{align} \label{E:PARTITIONEDBOOTSTRICHARTZ}
		\|
			\pmb{\partial} \vec{\Psi}
		\|_{L^2([t_k,t_{k+1}])L_x^{\infty}}^2
		+
		\sum_{\upnu \geq 2}
		\upnu^{2 \updelta_0}
		\|
			P_{\upnu} \pmb{\partial} \vec{\Psi}
		\|_{L^2([t_k,t_{k+1}])L_x^{\infty}}^2
	& \leq \uplambda^{- 8 \upepsilon_0},
		\\
	\|
		(\partial \vec{\vortrenormalized},\partial \vec{\GradEnt})
	\|_{L^2([t_k,t_{k+1}])L_x^{\infty}}^2
	+
	\sum_{\upnu \geq 2}
	\upnu^{2 \updelta_0}
	\|
		(P_{\upnu} \partial \vec{\vortrenormalized},P_{\upnu} \partial \vec{\GradEnt})
	\|_{L^2([t_k,t_{k+1}])L_x^{\infty}}^2
	& \leq \uplambda^{- 8 \upepsilon_0}.
	\label{E:PARTITIONEDBOOTL2LINFINITYFIRSTDERIVATIVESOFVORTICITYBOOTANDNENTROPYGRADIENT}
\end{align}
\end{subequations}
We refer readers to
\cite{sKiR2003}*{Remark~1.3} for more details on the construction of a partition 
of $[0,\Tboot]$ such that
\eqref{E:PARTITIONEDBOOTSTRICHARTZ}--\eqref{E:PARTITIONEDBOOTL2LINFINITYFIRSTDERIVATIVESOFVORTICITYBOOTANDNENTROPYGRADIENT}
hold.

\subsection{Frequency-localized Strichartz estimate}
\label{SS:FREQUENCYLOCALIZEDSTRICHARTZ}
The main step in the proof of Theorem~\ref{T:IMPROVEMENTOFSTRICHARTZBOOTSTRAPASSUMPTION}
is proving a frequency-localized version, specifically Theorem~\ref{T:FREQUENCYLOCALIZEDSTRICHARTZ};
see Sect.\,\ref{S:REDUCTIONSOFSTRICHARTZ} for an outline of its proof, which relies on
estimates for the acoustic geometry that we derive in Sect.\,\ref{S:ESTIMATESFOREIKONALFUNCTION}.

\begin{theorem}[Frequency-localized Strichartz estimate]
\label{T:FREQUENCYLOCALIZEDSTRICHARTZ}
Fix $\uplambda \geq \Lambda_0$, and
let $\varphi$ be a solution to the following covariant linear wave equation
on the slab $[t_k,t_{k+1}] \times \mathbb{R}^3$,
where $\lbrace [t_k,t_{k+1}] \rbrace_{k=1,\cdots}$ denotes the finite collection of
time intervals constructed in Subsect.\,\ref{SS:PARTITIONOFBOOTSTRAPINTERVAL}:
\begin{align} \label{E:LINEARWAVEFORFREQUENCYLOCALIZEDSTRICHARTZ}
	\square_{\gfour(\vec{\Psi})} \varphi
	& = 0.
\end{align}
Under the initial data and bootstrap assumptions of Sect.\,\ref{S:DATAANDBOOTSTRAPASSUMPTION}, 
if $\Lambda_0$ is sufficiently large, then for any $q > 2$ sufficiently close to $2$
and any $\uptau \in [t_k,t_{k+1}]$, 
we have the following estimate:
\begin{align} \label{E:FREQUENCYLOCALIZEDSTRICHARTZ}
		\|
			P_{\uplambda} \pmb{\partial} \varphi
		\|_{L^q([\uptau,t_{k+1}]) L_x^{\infty}}
	& \lesssim
		\uplambda^{\frac{3}{2} - \frac{1}{q}}
		\| \pmb{\partial} \varphi \|_{L^2(\Sigma_{\uptau})}.
\end{align}
\end{theorem}

\subsection{Proof of Theorem~\ref{T:IMPROVEMENTOFSTRICHARTZBOOTSTRAPASSUMPTION} given Theorem~\ref{T:FREQUENCYLOCALIZEDSTRICHARTZ}}
\label{SS:PROOFOFMAINSTRICHARTZWAVEESTIMATES}
In this proof, we often suppress the $x$-dependence of functions, 
and we use the remarks made in the first paragraph of Subsect.\,\ref{SS:PROOFOFTOPORDERENERGYESTIMATES}.
Let $W(t,\uptau)[f,f_0]$ be the solution at time $t$ to the covariant linear wave equation
$\square_{\gfour(\vec{\Psi})} \left(W(t,\uptau)[f,f_0] \right) = 0$ whose data at time $\uptau$ are 
$W(\uptau,\uptau)[f,f_0] := f$ and $\partial_t W(\uptau,\uptau)[f,f_0] := f_0$.
We assume that $\uplambda \geq \Lambda_0$, as in Theorem~\ref{T:FREQUENCYLOCALIZEDSTRICHARTZ}.
Let $\widetilde{P}_{\uplambda} := \sum_{1/2 \leq \frac{\upmu}{\uplambda} \leq 2} P_{\upmu}$,
so that in particular, $P_{\uplambda} = \widetilde{P}_{\uplambda} P_{\uplambda}$.
Then from equation \eqref{E:COVARIANTFREQUENCYPROJECTEDWAVE} and Duhamel's principle, 
for $\Psi \in \lbrace \LogDensity, v^1, v^2, v^3, \Ent \rbrace$
and $t \in [t_k,t_{k+1}]$, 
we have 
\begin{align} \label{E:DUHAMELFORFREQUENCYPROJECTEDWAVEEQUATION}
	P_{\uplambda} \Psi(t)
	& = W(t,t_k)[P_{\uplambda} \Psi(t_k),P_{\uplambda} \partial_t \Psi(t_k)]
		+
		\int_{t_k}^t
			W(t,\uptau)[0,\remainder_{(\Psi);\uplambda}(\uptau)]
		\, d \uptau.
\end{align}
Differentiating \eqref{E:DUHAMELFORFREQUENCYPROJECTEDWAVEEQUATION} with
$\pmb{\partial}$ and applying $\widetilde{P}_{\uplambda}$,
and letting $\mathbf{1}_{[t_k,t]}(\cdot)$ denote the characteristic function of the interval $[t_k,t]$,
we find that
\begin{align} \label{E:DIFFERENTIATEDANDPROJECTEDDUHAMELFORFREQUENCYPROJECTEDWAVEEQUATION}
	P_{\uplambda} \pmb{\partial} \Psi(t)
	& = 
		\widetilde{P}_{\uplambda}
		\left\lbrace
			\pmb{\partial} W(t,t_k)[P_{\uplambda} \Psi(t_k),P_{\uplambda} \partial_t \Psi(t_k)]
		\right\rbrace
		+
		\int_{t_k}^{t_{k+1}}
			\mathbf{1}_{[t_k,t]}(\uptau) \widetilde{P}_{\uplambda} \pmb{\partial} W(t,\uptau)[0,\remainder_{(\Psi);\uplambda}(\uptau)]
		\, d \uptau
			\\
	& := I_{\uplambda}(t) + II_{\uplambda}(t).
	\notag
\end{align}

We now recall that $\updelta = \frac{1}{2} - \frac{1}{q} > 0$ (see \eqref{E:DELTADEFINITIONANDPOSITIVITY}),
where $q > 2$ is any number for which Theorem~\ref{T:FREQUENCYLOCALIZEDSTRICHARTZ} holds.
Then from \eqref{E:FREQUENCYLOCALIZEDSTRICHARTZ} with $\widetilde{P}_{\uplambda}$ in the role of $P_{\uplambda}$,
H\"{o}lder's inequality, 
the covariant wave equation \eqref{E:COVARIANTFREQUENCYPROJECTEDWAVE} satisfied by $P_{\uplambda} \Psi$,
and the energy estimate \eqref{E:PRELIMINARYH1DOTHOMOGENEOUSWAVEEQUATIONINHOMOGENEOUSESTIMATE},
we find that
\begin{align} \label{E:BOUNDFORHOMOGENEOUSTERMINDIFFERENTIATEDANDPROJECTEDDUHAMELFORFREQUENCYPROJECTEDWAVEEQUATION}
		\|
			I_{\uplambda}
		\|_{L^2([t_k,t_{k+1}])L_x^{\infty}}
		& \lesssim 
			|t_{k+1} - t_k|^{\updelta}
			\|
				I_{\uplambda}
			\|_{L^q([t_k,t_{k+1}])L_x^{\infty}}
				\\
	& \lesssim 
		|t_{k+1} - t_k|^{\updelta}
		\uplambda^{\frac{3}{2} - \frac{1}{q}}
		\| \pmb{\partial} P_{\uplambda} \Psi \|_{L^2(\Sigma_{t_k})}
		\notag 
		\\
	& \lesssim 
		|t_{k+1} - t_k|^{\updelta}
		\uplambda^{1 + \updelta}
		\left\lbrace
			\| \pmb{\partial} P_{\uplambda} \Psi \|_{L^2(\Sigma_0)}
			+
			\| \remainder_{(\Psi);\uplambda} \|_{L^1([0,\Tboot])L_x^2}
		\right\rbrace.
		\notag
\end{align}

Similarly, using \eqref{E:FREQUENCYLOCALIZEDSTRICHARTZ} (again with $\widetilde{P}_{\uplambda}$ in the role of $P_{\uplambda}$) 
and Minkowski's inequality for integrals, 
we find that
\begin{align} \label{E:BOUNDFORINHOMOGENEOUSTERMINDIFFERENTIATEDANDPROJECTEDDUHAMELFORFREQUENCYPROJECTEDWAVEEQUATION}
		\|
			II_{\uplambda}
		\|_{L^2([t_k,t_{k+1}])L_x^{\infty}}
		& \lesssim 
			\int_{t_k}^{t_{k+1}}
				\|
					\mathbf{1}_{[t_k,t]}(\uptau) P_{\uplambda} \pmb{\partial} W(t,\uptau)[0,\remainder_{(\Psi);\uplambda}(\uptau)]
				\|_{L_t^2([\uptau,t_{k+1}])L_x^{\infty}}
			\, d \uptau
				\\
		& \lesssim 
			\int_{t_k}^{t_{k+1}}
				|t_{k+1} - \uptau|^{\updelta}
				\|
					P_{\uplambda} \pmb{\partial} W(t,\uptau)[0,\remainder_{(\Psi);\uplambda}(\uptau)]
				\|_{L_t^q([\uptau,t_{k+1}])L_x^{\infty}}
			\, d \uptau
			\notag \\
		& \lesssim 
			|t_{k+1} - t_k|^{\updelta}
			\uplambda^{1 + \updelta}
			\| \remainder_{(\Psi);\uplambda} \|_{L^1([t_k,t_{k+1}])L_x^2}.
				\notag
\end{align}
Using \eqref{E:DIFFERENTIATEDANDPROJECTEDDUHAMELFORFREQUENCYPROJECTEDWAVEEQUATION},
\eqref{E:BOUNDFORHOMOGENEOUSTERMINDIFFERENTIATEDANDPROJECTEDDUHAMELFORFREQUENCYPROJECTEDWAVEEQUATION},
and \eqref{E:BOUNDFORINHOMOGENEOUSTERMINDIFFERENTIATEDANDPROJECTEDDUHAMELFORFREQUENCYPROJECTEDWAVEEQUATION},
and recalling that $|t_{k+1} - t_k| \lesssim \uplambda^{-8 \upepsilon_0} \Tboot$,
we find that
\begin{align} \label{E:FREQUENCYDOUBLEPROJECTEDKEYSTRICHARTESTIMATE}
	\|
		P_{\uplambda} \pmb{\partial} \Psi
	\|_{L^2([t_k,t_{k+1}])L_x^{\infty}}
	& \lesssim 
		\uplambda^{1 + \updelta(1 - 8 \upepsilon_0)}
		\Tboot^{\updelta}
		\left\lbrace
			\| \pmb{\partial} P_{\uplambda} \Psi \|_{L^2(\Sigma_0)}
			+
			\| \remainder_{(\Psi);\uplambda} \|_{L^1([0,\Tboot])L_x^2}
		\right\rbrace.
\end{align}

Next, we square \eqref{E:FREQUENCYDOUBLEPROJECTEDKEYSTRICHARTESTIMATE},
sum over all intervals $[t_k,t_{k+1}]$,
recall that there are $\lesssim \uplambda^{8 \upepsilon_0}$ such intervals,
and multiply the resulting inequality by $\uplambda^{2 \updelta_1}$
(where $\updelta_1 > 0$ is defined in \eqref{E:DELTA1DEFINITIONANDPOSITIVITY}),
thereby obtaining:
\begin{align} \label{E:SECONDFREQUENCYDOUBLEPROJECTEDKEYSTRICHARTESTIMATE}
	\uplambda^{2 \updelta_1}
	\|
		P_{\uplambda} \pmb{\partial} \Psi
	\|_{L^2([0,\Tboot])L_x^{\infty}}^2
	& \lesssim 
		\uplambda^{2 \updelta_1}
		\uplambda^{8 \upepsilon_0}
		\uplambda^{2 + 2 \updelta(1 - 8 \upepsilon_0)}
		\Tboot^{2 \updelta}
		\left\lbrace
			\| \pmb{\partial} P_{\uplambda} \Psi \|_{L^2(\Sigma_0)}^2
			+
			\| \remainder_{(\Psi);\uplambda} \|_{L^1([0,\Tboot])L_x^2}^2
		\right\rbrace
			\\
		&
		\lesssim
		\Tboot^{2 \updelta}
		\left\lbrace
			\| \uplambda^{\Sob-1} \pmb{\partial} P_{\uplambda} \Psi \|_{L^2(\Sigma_0)}^2
			+
			\| \uplambda^{\Sob-1} \remainder_{(\Psi);\uplambda} \|_{L^1([0,\Tboot])L_x^2}^2
		\right\rbrace.
		\notag
\end{align}

We now sum \eqref{E:SECONDFREQUENCYDOUBLEPROJECTEDKEYSTRICHARTESTIMATE} over 
dyadic frequencies $\uplambda \geq \Lambda_0$
and use the H\"{o}lder-in-time estimate
\begin{align}
\notag
\| \uplambda^{\Sob-1} \remainder_{(\Psi);\uplambda} \|_{L^1([0,\Tboot])L_x^2}^2
\lesssim
\Tboot
\| \uplambda^{\Sob-1} \remainder_{(\Psi);\uplambda} \|_{L^2([0,\Tboot])L_x^2}^2
\end{align}
to deduce that
\begin{align} \label{E:SUMMEDFREQUENCYDOUBLEPROJECTEDKEYSTRICHARTESTIMATE}
	&
	\sum_{\upnu \geq \Lambda_0} 
	\upnu^{2 \updelta_1}
	\|
		P_{\upnu} \pmb{\partial} \Psi
	\|_{L^2([0,\Tboot])L_x^{\infty}}^2
		\\
& \lesssim 
		\Tboot^{2 \updelta}
		\left\lbrace
			\| (\Psi, \partial_t \Psi) \|_{H^{\Sob}(\Sigma_0) \times H^{\Sob-1}(\Sigma_0)}^2
			+
			\Tboot
			\| \upnu^{\Sob-1} \remainder_{(\Psi);\upnu} \|_{L^2([0,\Tboot])\ell_{\upnu}^2 L_x^2}^2
		\right\rbrace.
		\notag
\end{align}
Using the estimate \eqref{E:PRODUCTANDCOMMUTATORESTIMATESFORWAVEEQUATIONS},
the Strichartz-type bootstrap assumption \eqref{E:BOOTSTRICHARTZ},
and the top-order energy estimate \eqref{E:TOPORDERENERGYESTIMATES},
we deduce that 
$\| \upnu^{\Sob-1} \remainder_{(\Psi);\upnu} \|_{L^2([0,\Tboot])\ell_{\upnu}^2 L_x^2}^2 \lesssim 1$.
Inserting this estimate and the trivial 
bound $\| (\Psi, \partial_t \Psi) \|_{H^{\Sob}(\Sigma_0) \times H^{\Sob-1}(\Sigma_0)}^2 \lesssim 1$
into RHS~\eqref{E:SUMMEDFREQUENCYDOUBLEPROJECTEDKEYSTRICHARTESTIMATE},
we find that
\begin{align} \label{E:SECONDSUMMEDFREQUENCYDOUBLEPROJECTEDKEYSTRICHARTESTIMATE}
	\sum_{\upnu \geq \Lambda_0} 
	\upnu^{2 \updelta_1}
	\|
		P_{\upnu} \pmb{\partial} \Psi
	\|_{L^2([0,\Tboot])L_x^{\infty}}^2
	\lesssim 
		\Tboot^{2 \updelta}.
\end{align}

Next, we note that Sobolev embedding and the energy estimate \eqref{E:TOPORDERENERGYESTIMATES}
yield that
$
\|
	P_{\leq \Lambda_0} \pmb{\partial} \Psi
\|_{L_x^{\infty}(\Sigma_t)}
\lesssim 
\|
	P_{\leq \Lambda_0} \pmb{\partial} \Psi
\|_{H^2(\Sigma_t)}
\lesssim
\|
	\pmb{\partial} \Psi
\|_{L^2(\Sigma_t)}
\lesssim
1
$
(where the implicit constants are allowed to depend on $\Lambda_0$)
and thus
\begin{align} \label{E:LOWFREQUENCIESSUMMEDFREQUENCYDOUBLEPROJECTEDKEYSTRICHARTESTIMATE}
\|
		P_{\leq \Lambda_0} \pmb{\partial} \Psi
	\|_{L^2([0,\Tboot])L_x^{\infty}}^2
\lesssim \Tboot 
\lesssim
\Tboot^{2 \updelta}.
\end{align}

We are now ready to bound the term 
$
\|
			\pmb{\partial} \vec{\Psi}
		\|_{L^2([0,\Tboot]) L_x^{\infty}}^2
$
on LHS~\eqref{E:IMPROVEMENTOFSTRICHARTZBOOTSTRAPASSUMPTION}.
To proceed, we use the triangle inequality, the Cauchy--Schwarz inequality, 
and the fact that $\sum_{\upnu \geq \Lambda_0} \upnu^{-2 \updelta_1} < \infty$
to deduce that 
\begin{align} \label{E:MAINWAVESTRICHARTZESTIMATEPROOF}
\| \pmb{\partial} \Psi \|_{L^{\infty}(\Sigma_t)}
& \lesssim 
\| P_{\leq \Lambda_0} \pmb{\partial} \Psi \|_{L^{\infty}(\Sigma_t)}
+
\sum_{\upnu \geq \Lambda_0}  
\upnu^{- \updelta_1}
\| \upnu^{\updelta_1} P_{\upnu} \pmb{\partial} \Psi \|_{L^{\infty}(\Sigma_t)}
	\\
& 
\lesssim 
\| P_{\leq \Lambda_0} \pmb{\partial} \Psi \|_{L^{\infty}(\Sigma_t)}
+
\sqrt{
\sum_{\upnu \geq \Lambda_0}  
\upnu^{2 \updelta_1}
\| P_{\upnu} \pmb{\partial} \Psi \|_{L^{\infty}(\Sigma_t)}^2}.
\notag
\end{align}
Squaring \eqref{E:MAINWAVESTRICHARTZESTIMATEPROOF},
integrating the resulting inequality over the interval $[0,\Tboot]$, 
and using \eqref{E:SECONDSUMMEDFREQUENCYDOUBLEPROJECTEDKEYSTRICHARTESTIMATE}
and \eqref{E:LOWFREQUENCIESSUMMEDFREQUENCYDOUBLEPROJECTEDKEYSTRICHARTESTIMATE},
we conclude the desired bound
for the term 
$
\|
			\pmb{\partial} \vec{\Psi}
		\|_{L^2([0,\Tboot]) L_x^{\infty}}^2
$
on LHS~\eqref{E:IMPROVEMENTOFSTRICHARTZBOOTSTRAPASSUMPTION}.
From this bound, 
\eqref{E:SECONDSUMMEDFREQUENCYDOUBLEPROJECTEDKEYSTRICHARTESTIMATE},
and the basic inequality
$
\|
		P_{\upnu} \pmb{\partial} \Psi
	\|_{L^{\infty}(\Sigma_t)}
\lesssim
\|
	\pmb{\partial} \Psi
\|_{L^{\infty}(\Sigma_t)}
$,
the desired bound for the sum on LHS~\eqref{E:IMPROVEMENTOFSTRICHARTZBOOTSTRAPASSUMPTION}
readily follows. This completes the proof of Theorem~\ref{T:IMPROVEMENTOFSTRICHARTZBOOTSTRAPASSUMPTION}.

\hfill $\qed$

\section{Schauder-transport estimates in H\"{o}lder spaces for the first derivatives of the specific vorticity and the second derivatives of the entropy}
\label{S:ELLIPTICESTIMATESINHOLDERSPACES}
	Our main goal in this section is to derive improvements of the mixed spacetime norm bootstrap assumptions
	\eqref{E:BOOTL2LINFINITYFIRSTDERIVATIVESOFVORTICITYBOOTANDNENTROPYGRADIENT}
	for $\partial \vec{\vortrenormalized}$ and $\partial \vec{\GradEnt}$.
	The main result is Theorem~\ref{T:MIXEDSPACETIMEHOLDERESTIMATESFORVORTICITYANDENTROPY}.
	We also derive a strict improvement of the bootstrap assumption \eqref{E:BOOTSOLUTIONDOESNOTESCAPEREGIMEOFHYPERBOLICITY}.
	Before proving the theorem, we first derive two fundamentally important precursor results: 
	\textbf{i)} Schauder estimates for div-curl systems;
	\textbf{ii)} Estimates that yield control of the characteristics of the transport operator $\Transport$
	(i.e., over the integral curves of $\Transport$);
	and 
	\textbf{ii)'} With the help of \textbf{ii}, we derive a priori estimates in H\"{o}lder
	spaces for solutions $\varphi$ to transport equations $\Transport \varphi = \mathfrak{F}$
	with $\mathfrak{F} \in L_t^1 C_x^{0,\updelta_1}$
	(see Lemma~\ref{L:TRANSPORTESTIMATESINHOLDERSPACES}).
	Thanks to these three preliminary ingredients,
	Theorem~\ref{T:MIXEDSPACETIMEHOLDERESTIMATESFORVORTICITYANDENTROPY} 
	will follow from a Gr\"{o}nwall inequality estimate.
	
	\subsection{\texorpdfstring{Statement of Theorem~\ref{T:MIXEDSPACETIMEHOLDERESTIMATESFORVORTICITYANDENTROPY} and proof of an improvement of the basic $L^{\infty}$-type bootstrap assumption}{Statement of Theorem  ref T:MIXEDSPACETIMEHOLDERESTIMATESFORVORTICITYANDENTROPY and proof of an improvement of the basic Linfty-type bootstrap assumption}}
	We now state the main theorem of this section. Its proof is located in Subsect.\,\ref{SS:PROOFOFPROPMIXEDSPACETIMEHOLDERESTIMATESFORVORTICITYANDENTROPY}.
	
	\begin{theorem}[Lebesgue-H\"{o}lder norm estimates for the specific vorticity and entropy gradient and improvements of the bootstrap assumptions]
	\label{T:MIXEDSPACETIMEHOLDERESTIMATESFORVORTICITYANDENTROPY}
	Under the initial data and bootstrap assumptions of Sect.\,\ref{S:DATAANDBOOTSTRAPASSUMPTION}, 
	the following estimates hold:
	\begin{align} 
		\|
			(\vec{\VortVort},\DivGradEnt)
		\|_{L^{\infty}([0,\Tboot])C_x^{0,\updelta_1}}
		& \lesssim 1,
			\label{E:HOLDERREGULARITYOFMODIFIEDVARIABLESPROPAGATED} \\
		\|
			\partial (\vec{\vortrenormalized},\vec{\GradEnt})
		\|_{L^2([0,\Tboot])C_x^{0,\updelta_1}}^2
		& \lesssim \Tboot^{2 \updelta}.
		\label{E:MIXEDSPACETIMEHOLDERESTIMATESFORVORTICITYANDENTROPY}
	\end{align}
	Moreover,
	\begin{align} \label{E:SUPPEDMIXEDSPACETIMEHOLDERESTIMATESFORVORTICITYANDENTROPY}
	\sum_{\upnu \geq 1}
	\upnu^{\updelta_1}
	\|
		 P_{\upnu} \partial (\vec{\vortrenormalized},\vec{\GradEnt})
	\|_{L^2([0,\Tboot]) L_x^{\infty}}^2
	& \lesssim \Tboot^{2 \updelta}.
	\end{align}
	
	\end{theorem}

Before initiating the proof of Theorem~\ref{T:MIXEDSPACETIMEHOLDERESTIMATESFORVORTICITYANDENTROPY}, 
we first use it as an ingredient in deriving a strict improvement of
the bootstrap assumption \eqref{E:BOOTSOLUTIONDOESNOTESCAPEREGIMEOFHYPERBOLICITY}.

\begin{corollary}[Improvement of the basic $L^{\infty}$-type bootstrap assumption]
	\label{C:IMPROVEMENTOFLINFINITYBOOTSTRAP}
	Let $\mathfrak{K}$ be the compact set appearing in the bootstrap assumption \eqref{E:BOOTSOLUTIONDOESNOTESCAPEREGIMEOFHYPERBOLICITY}.
	Under the initial data and bootstrap assumptions of Sect.\,\ref{S:DATAANDBOOTSTRAPASSUMPTION},  
	the following containment holds
	whenever $\Tboot$ is sufficiently small:
	\begin{align} \label{E:IMPROVEDBOOTSOLUTIONDOESNOTESCAPEREGIMEOFHYPERBOLICITY}
	(\LogDensity,\Ent,\vec{v},\vec{\vortrenormalized},\vec{\GradEnt})([0,\Tboot] \times \mathbb{R}^3)
	\subset \mbox{\upshape int} \mathfrak{K}.
\end{align}
\end{corollary}	

\begin{proof}
Let $\vec{\varphi}$ denote the following array of scalar functions:
$\vec{\varphi} := (\LogDensity,\Ent,\vec{v},\vec{\vortrenormalized},\vec{\GradEnt})$.
Using \eqref{E:TRANSPORTEQNMAIN}
and the bootstrap assumption
\eqref{E:BOOTSOLUTIONDOESNOTESCAPEREGIMEOFHYPERBOLICITY},
we deduce that
$|\partial_t \vec{\varphi}| \lesssim |\pmb{\partial} \vec{\Psi}| + |\partial \vec{\vortrenormalized}| + |\partial \vec{\GradEnt}| + 1$.
Hence, from the fundamental theorem of calculus, 
the estimates \eqref{E:IMPROVEMENTOFSTRICHARTZBOOTSTRAPASSUMPTION} and \eqref{E:MIXEDSPACETIMEHOLDERESTIMATESFORVORTICITYANDENTROPY},
and the Cauchy--Schwarz inequality with respect to $t$, we deduce that the following estimate holds for $t \in [0,\Tboot]$:
$|
	\vec{\varphi}(t,x)
	-
	\vec{\varphi}(0,x)
|
\lesssim 
\|\pmb{\partial} \vec{\Psi} \|_{L^1([0,t]) L_x^{\infty}} 
+ 
\|\partial \vec{\vortrenormalized} \|_{L^1([0,t]) L_x^{\infty}} 
+ 
\|\partial \vec{\GradEnt} \|_{L^1([0,t]) L_x^{\infty}}
+
t
\lesssim
\Tboot^{1/2 + \updelta}
$.
It follows that we can guarantee that $\vec{\varphi}(t,x)$ 
is arbitrarily close to $\vec{\varphi}(0,x)$
by choosing $\Tboot$ to be sufficiently small.
From this fact and \eqref{E:SETSDESCRIBINGINTETERIOROFREGIMEOFHYPERBOLICITY},
we conclude \eqref{E:IMPROVEDBOOTSOLUTIONDOESNOTESCAPEREGIMEOFHYPERBOLICITY}.
	
\end{proof}

\subsection{Schauder estimates for div-curl systems}
\label{SS:SCHAUDERFORDIVCURL}
In the next lemma, we provide a standard Schauder estimate for div-curl systems on Euclidean space $\mathbb{R}^3$. 

\begin{lemma}[Schauder estimates for div-curl systems]
\label{L:SCHAUDERESTIMATESFORDIVCURLSYSTEMS}
Let $V$ be a vectorfield on $\mathbb{R}^3$ such that $V \in C^2(\mathbb{R}^3) \cap H^2(\mathbb{R}^3)$,
and let $\updelta_1 > 0$ be the parameter from \eqref{E:DELTA1DEFINITIONANDPOSITIVITY}.
Then the following estimate holds:\footnote{Our proof of the estimate \eqref{E:SCHAUDERESTIMATESFORDIVCURLSYSTEMS} goes through
for $\updelta_1 \in (0,1/2)$, but in practice, we need the estimate only for the value of $\updelta_1$ specified in \eqref{E:DELTA1DEFINITIONANDPOSITIVITY}.}
\begin{align} \label{E:SCHAUDERESTIMATESFORDIVCURLSYSTEMS}
	\| \partial V \|_{C^{0,\updelta_1}(\mathbb{R}^3)}
	& \lesssim 
		\| \dive V \|_{C^{0,\updelta_1}(\mathbb{R}^3)}
		+
		\| \curl V \|_{C^{0,\updelta_1}(\mathbb{R}^3)}
		+
		\| V \|_{H^2(\mathbb{R}^3)}.
\end{align}
\end{lemma}

\begin{proof}
	Let $z \in \mathbb{R}^3$ and let $B_2(z)$ be the ball of Euclidean radius $2$ centered at $z$.
	As a first step, we will show that if $W \in C^2(\mathbb{R}^3) \cap H^2(\mathbb{R}^3)$
	is a vectorfield on $\mathbb{R}^3$ that is supported in $B_2(z)$,
	then we have (with implicit constants that are independent of $z$):
\begin{align} \label{E:COMPACTCASESCHAUDERESTIMATESFORDIVCURLSYSTEMS}
	\| \partial W \|_{C^{0,\updelta_1}(B_2(z))}
	& \lesssim 
		\| \dive W \|_{C^{0,\updelta_1}(B_2(z))}
		+
		\| \curl W \|_{C^{0,\updelta_1}(B_2(z))}.
\end{align}
	To prove \eqref{E:COMPACTCASESCHAUDERESTIMATESFORDIVCURLSYSTEMS},
	we let $\Phi(x) :=  \frac{-1}{4\pi |x|}$ denote the fundamental solution of the Euclidean Laplacian on $\mathbb{R}^3$.
	The standard Helmholtz decomposition yields the following identity, where $\upepsilon^{ijk}$ is the fully antisymmetric
	symbol normalized by $\upepsilon^{123} = 1$:
	\begin{align} \label{E:STANDARDHELMHOLTZDECOMP}
		W^j
		& = \dive W * \updelta^{jc} \partial_c \Phi
				-
				\upepsilon^{jcd}\updelta_{ca} (\curl W)^a * \partial_d \Phi.
	\end{align}	
The desired estimate \eqref{E:COMPACTCASESCHAUDERESTIMATESFORDIVCURLSYSTEMS} now follows from
standard estimates for the first derivatives of the convolutions 
on RHS~\eqref{E:STANDARDHELMHOLTZDECOMP}; see, for example, the proofs of \cite{dGnT2001}*{Lemma~4.2} and \cite{dGnT2001}*{Lemma~4.4}.

To prove \eqref{E:SCHAUDERESTIMATESFORDIVCURLSYSTEMS},
let $B_1(z) \subset \mathbb{R}^3$ be the Euclidean ball with radius $1$ centered at $z$.
Let $\chi \geq 0$ be a $C^{\infty}$ spherically symmetric cut-off function on $\mathbb{R}^3$ 
with $\chi(x) = 1$ for $|x| \leq 1$ and $\chi(x) = 0$ for $|x| \geq 2$,
and let $\chi_z(x) := \chi(x-z)$. 
It follows that $\chi_z(x) = 1$ for $x \in B_1(z)$ and thus
$
\| \partial V \|_{C^{0,\updelta_1}(B_1(z))}
=
\| \partial (\chi_z V) \|_{C^{0,\updelta_1}(B_1(z))}
\leq
\| \partial (\chi_z V) \|_{C^{0,\updelta_1}(B_2(z))}
$.
From this estimate,
\eqref{E:COMPACTCASESCHAUDERESTIMATESFORDIVCURLSYSTEMS} with $\chi_z V$ in the role of $W$
(this estimate is valid since $\chi_z V$ is compactly supported in $B_2(z)$), 
the standard estimate
$\| F \cdot G \|_{C^{0,\updelta_1}(B_2(z))}
		\leq
		2
		\| F \|_{C^{0,\updelta_1}(B_2(z))}
		\| G \|_{C^{0,\updelta_1}(B_2(z))}
$,
and the simple estimates (which are uniform in $z$)
$
\| \chi_z \|_{C^{0,\updelta_1}(B_1(z))}
\leq
\| \chi \|_{C^{0,\updelta_1}(\mathbb{R}^3)}
\lesssim 1
$
and
$
\| \partial \chi_z \|_{C^{0,\updelta_1}(B_1(z))}
\leq
\| \partial \chi \|_{C^{0,\updelta_1}(\mathbb{R}^3)}
\lesssim 1
$,
we obtain
\begin{align} \label{E:NEXTSTEPCOMPACTCASESCHAUDERESTIMATESFORDIVCURLSYSTEMS}
	\| \partial V \|_{C^{0,\updelta_1}(B_1(z))}
	& \lesssim 
		\| \dive (\chi_z V) \|_{C^{0,\updelta_1}(B_2(z))}
		+
		\| \curl (\chi_z V) \|_{C^{0,\updelta_1}(B_2(z))}
			\\
	& \lesssim 
		\| \dive V \|_{C^{0,\updelta_1}(B_2(z))}
		+
		\| \curl V \|_{C^{0,\updelta_1}(B_2(z))}
		+
		\| V \|_{C^{0,\updelta_1}(B_2(z))}.
		\notag
\end{align}

From \eqref{E:NEXTSTEPCOMPACTCASESCHAUDERESTIMATESFORDIVCURLSYSTEMS}
and the Sobolev embedding result $H^2(\mathbb{R}^3) \hookrightarrow C^{0,\updelta_1}(\mathbb{R}^3)$
(which is valid since $\updelta_1 < 1/2$),
we deduce that
\begin{align} \label{E:SMALLDISTANCESCHAUDERESTIMATE}
	\sup_{x,y \in B_1(z), 0 < |x-y|}
	\frac{|\partial V(x) - \partial V(y)|}{|x-y|^{\updelta_1}}
	& \lesssim
	\| \dive V \|_{C^{0,\updelta_1}(\mathbb{R}^3)}
	+
	\| \curl V \|_{C^{0,\updelta_1}(\mathbb{R}^3)}
	+
	\| V \|_{H^2(\mathbb{R}^3)}.
\end{align}
Moreover, since $\| \partial V \|_{L^2(B_1(z))} \leq \| V \|_{H^1(\mathbb{R}^3)}$
and since $B_1(z)$ has Euclidean volume greater than $1$,
there must be a point $p \in B_1(z)$ such that
$|\partial V(p)| \leq \| V \|_{H^1(\mathbb{R}^3)}$. From this simple fact
and \eqref{E:SMALLDISTANCESCHAUDERESTIMATE},
we conclude that
\begin{align} \label{E:SUPNORMBOUNDFORPARTIALVONABALL}
	\sup_{x \in B_1(z)}
	|\partial V(x)|
	& \lesssim 
		\| \dive V \|_{C^{0,\updelta_1}(\mathbb{R}^3)}
	+
	\| \curl V \|_{C^{0,\updelta_1}(\mathbb{R}^3)}
	+
	\| V \|_{H^2(\mathbb{R}^3)}.
\end{align}
Since $z$ is arbitrary in \eqref{E:SUPNORMBOUNDFORPARTIALVONABALL},
we conclude that
\begin{align} \label{E:SUPNORMBOUNDFORPARTIALV}
	\|
		\partial V
	\|_{L^{\infty}(\mathbb{R}^3)}
	& \lesssim 
		\| \dive V \|_{C^{0,\updelta_1}(\mathbb{R}^3)}
	+
	\| \curl V \|_{C^{0,\updelta_1}(\mathbb{R}^3)}
	+
	\| V \|_{H^2(\mathbb{R}^3)}.
\end{align}
From \eqref{E:SUPNORMBOUNDFORPARTIALV},
it easily follows that
\begin{align} \label{E:LARGEDISTANCEHOLDERSTIMATE}
\sup_{|x - y| \geq 1}
	\frac{|\partial V(x) - \partial V(y)|}{|x-y|^{\updelta_1}}
& \leq 
	2
	\|
		\partial V
	\|_{L^{\infty}(\mathbb{R}^3)}
	\lesssim 
	\| \dive V \|_{C^{0,\updelta_1}(\mathbb{R}^3)}
	+
	\| \curl V \|_{C^{0,\updelta_1}(\mathbb{R}^3)}
	+
	\| V \|_{H^2(\mathbb{R}^3)}.
\end{align}
Next, if $0 < |x-y| \leq 1$, then $y \in B_1(x)$,
which, in view of \eqref{E:NEXTSTEPCOMPACTCASESCHAUDERESTIMATESFORDIVCURLSYSTEMS} with $x$ in the role of $z$
and the Sobolev embedding result $H^2(\mathbb{R}^3) \hookrightarrow C^{0,\updelta_1}(\mathbb{R}^3)$,
implies that
\begin{align} \label{E:SMALLDISTANCEHOLDERSTIMATE}
\sup_{0 < |x - y| \leq 1}
	\frac{|\partial V(x) - \partial V(y)|}{|x-y|^{\updelta_1}}
	& \lesssim 
	\| \dive V \|_{C^{0,\updelta_1}(\mathbb{R}^3)}
	+
	\| \curl V \|_{C^{0,\updelta_1}(\mathbb{R}^3)}
	+
	\| V \|_{H^2(\mathbb{R}^3)}.
\end{align}
Finally, in view of definition \eqref{E:HOLDERNORMDEF},
we see that the desired estimate \eqref{E:SCHAUDERESTIMATESFORDIVCURLSYSTEMS}
follows from \eqref{E:SUPNORMBOUNDFORPARTIALV}, 
\eqref{E:LARGEDISTANCEHOLDERSTIMATE}, 
and
\eqref{E:SMALLDISTANCEHOLDERSTIMATE}.

\end{proof}

\subsection{Estimates for the flow map of the material derivative vectorfield}
\label{SS:ESTIMATESFORCHARACTERISTICSOFTRANSPORT}
Our proof of Theorem~\ref{T:MIXEDSPACETIMEHOLDERESTIMATESFORVORTICITYANDENTROPY} is through a Gr\"{o}nwall inequality estimate
that relies on having sufficient control of the flow map of the material derivative vectorfield $\Transport$.
In the next lemma, we derive the estimates for the flow map.

\begin{lemma}[Estimates for the flow map of the material derivative vectorfield]
	\label{L:ESTIMATESFORTRANSPORTCHARACTERISTICS}
	Let $\upgamma : [0,\Tboot] \times \mathbb{R}^3 \rightarrow [0,\Tboot] \times \mathbb{R}^3$
	be the flow map of $\Transport$, that is, the solution 
	to the following transport initial value problem
	for the Cartesian component functions $\upgamma^{\alpha}(t;x)$:
	\begin{subequations}
	\begin{align}
		\frac{d}{dt} \upgamma^{\alpha}(t;x)
		& = \Transport^{\alpha} \circ \upgamma(t;x),
		&&
		\label{E:TRANSPORTODESYSTEM} \\
		\upgamma^0(0;x)
		& = 0,
		&
		\upgamma^i(0;x)
		& = x^i.
		\label{E:DATAFORTRANSPORTODESYSTEM}
	\end{align}
	\end{subequations}
	Then under the bootstrap assumptions,
	for every fixed $x \in \mathbb{R}^3$, there exists a unique solution $t \rightarrow \upgamma(t;x)$ to the system 
	\eqref{E:TRANSPORTODESYSTEM}--\eqref{E:DATAFORTRANSPORTODESYSTEM}.
	Moreover, $\upgamma$ is a smooth function of $t$ and $x$.
	In addition, there exists a constant $C > 0$
	such that for $t \in [0,\Tboot]$ and all $x,y \in \mathbb{R}^3$,
	we have
	\begin{subequations}
	\begin{align}
		\upgamma^0(t;x)
		& = t,
			\label{E:GAMMA0IST} \\
		|
			\upgamma^i(t;x)
			- 
			x^i
		|
		& \leq C,
			\label{E:CHARACTERISTICNEARDATA} \\
		\sum_{i=1}^3	
		|
			\upgamma^i(t;x)
			-
			\upgamma^i(t;y)
		|
		& \approx |x-y|.
		\label{E:CHARACTERISTICSDIFFERENCEESTIMATE}
	\end{align}
	\end{subequations}
	In particular, for each fixed $t \in [0,\Tboot]$, the map
	$x \rightarrow \left(\upgamma^1(t,x),\upgamma^2(t,x),\upgamma^3(t,x)\right)$ is a smooth global diffeomorphism 
	from $\mathbb{R}^3$ to $\mathbb{R}^3$.
\end{lemma}

\begin{proof} 
	The identity \eqref{E:GAMMA0IST} follows easily from 
	considering the $0$ component of \eqref{E:TRANSPORTODESYSTEM}--\eqref{E:DATAFORTRANSPORTODESYSTEM}.
	
	Since the components $\Transport^{\alpha}$ are smooth on $[0,\Tboot] \times \mathbb{R}^3$
	and satisfy\footnote{Here, we are only using the \emph{qualitative} finiteness property
	$\sup_{t \in [0,\Tboot]} \| \mathbf{\partial} \Transport^{\alpha} \|_{L^{\infty}(\Sigma_t)} < \infty$
	to guarantee the existence and uniqueness of the solution to \eqref{E:TRANSPORTODESYSTEM}--\eqref{E:DATAFORTRANSPORTODESYSTEM}.
	In contrast, the constants in \eqref{E:GAMMA0IST}--\eqref{E:CHARACTERISTICSDIFFERENCEESTIMATE} are controlled
	by the bootstrap assumptions, such as \eqref{E:BOOTSTRICHARTZ}.} 
	$\sup_{t \in [0,\Tboot]} \| \mathbf{\partial}^{\leq 1} \Transport^{\alpha} \|_{L^{\infty}(\Sigma_t)} < \infty$,
	the existence and uniqueness of solutions $\upgamma(t;x)$
	to \eqref{E:TRANSPORTODESYSTEM}--\eqref{E:DATAFORTRANSPORTODESYSTEM}
	that depend smoothly on $t$ and $x$ is a standard result from ODE theory,
	as is the fact that	the map $x \rightarrow \left(\upgamma^1(t,x),\upgamma^2(t,x),\upgamma^3(t,x)\right)$ is a smooth global diffeomorphism 
	from $\mathbb{R}^3$ to $\mathbb{R}^3$ for each fixed $t \in [0,\Tboot]$.
	Next, we use the fundamental theorem of calculus 
	and the fact that $\Transport^i = v^i$ (see \eqref{E:MATERIALVECTORVIELDRELATIVETORECTANGULAR})
	to deduce
	\begin{align}  \label{E:FTCDIFFERENCECHARACTERISTICS}
		\upgamma^i(t;x)
		-
		\upgamma^i(t;y)
		& = x^i - y^i
			+
			\int_0^t
				\left\lbrace
					v^i \circ \upgamma(\uptau;x)
					-
					v^i \circ \upgamma(\uptau;y)
				\right\rbrace
			\, d\uptau.
	\end{align}
	Let $\underline{\upgamma}(t,x) := \left(\upgamma^1(t,x),\upgamma^2(t,x),\upgamma^3(t,x) \right)$.
	Since $\partial v$ and $\upgamma$ are smooth, we deduce 
	from \eqref{E:FTCDIFFERENCECHARACTERISTICS} and the mean value theorem that
	\begin{align}  \label{E:DIFFERENCECHARACTERISTICSFIRSTESTIMATE}
		|
			\left(\underline{\upgamma}(t;x)
			-
			\underline{\upgamma}(t;y)
			\right)
			-
			(x - y)
		|
		& \leq
			C
			\int_0^t
				\| \partial \vec{v} \|_{L^{\infty}(\Sigma_{\uptau})}
				|
					\underline{\upgamma}(\uptau;x)
					-
					\underline{\upgamma}(\uptau;y)
				|
			\, d\uptau.
	\end{align}
	From \eqref{E:DIFFERENCECHARACTERISTICSFIRSTESTIMATE} and Gr\"{o}nwall's inequality 
	(more precisely, a straightforward extension of the standard Gr\"{o}nwall inequality to yield upper and lower bounds),
	we deduce that
	\begin{align} \label{E:DIFFERENCECHARACTERISTICSGRONWALLEDUPPERANDLOWER}
		\exp \left(- C \int_0^t \| \partial \vec{v} \|_{L^{\infty}(\Sigma_{\uptau})} \, d \uptau \right)
		&
		\leq
		\frac{
		|
			\underline{\upgamma}(t;x)
			-
			\underline{\upgamma}(t;y)
		|}{|x - y|}
		\leq \exp \left(C \int_0^t \| \partial \vec{v} \|_{L^{\infty}(\Sigma_{\uptau})} \, d \uptau \right).
	\end{align}
	From \eqref{E:DIFFERENCECHARACTERISTICSGRONWALLEDUPPERANDLOWER} and the bootstrap assumption \eqref{E:BOOTSTRICHARTZ},
	we conclude the desired bounds \eqref{E:CHARACTERISTICSDIFFERENCEESTIMATE}.
	
	The estimate \eqref{E:CHARACTERISTICNEARDATA} follows from a similar argument
	based on the simple bound 
	$\| \vec{v} \|_{L^1([0,\Tboot])L_x^{\infty}} \lesssim 1$;
	we omit the details.
	
	\end{proof}

\subsection{Estimates for transport equations in H\"{o}lder spaces}
\label{SS:ESTIAMTESFORTRANSPORTINHOLDER}
With the help of Lemma~\ref{L:ESTIMATESFORTRANSPORTCHARACTERISTICS},
we now derive estimates for transport equations with
H\"{o}lder-class initial data and source terms.

	\begin{lemma}[Estimates for transport equations in H\"{o}lder spaces]
		\label{L:TRANSPORTESTIMATESINHOLDERSPACES}
		Let $\mathfrak{F}$ be a smooth function on $[0,\Tboot] \times \mathbb{R}^3$ 
		and let $\mathring{\varphi}$ be a smooth function on $\mathbb{R}^3$. 
		Let $\varphi$ be a smooth solution to the following inhomogeneous transport 
		equation initial value problem:
		\begin{subequations}
		\begin{align}
			\Transport^{\alpha} \partial_{\alpha} \varphi
			& = \mathfrak{F},
				\label{E:INHOMOGENEOUSTRANSPORT} \\
			\varphi|_{\Sigma_0} & = \mathring{\varphi}.
			\label{E:DATAINHOMOGENEOUSTRANSPORT}
		\end{align}
		\end{subequations}
		Then the following estimate holds for $t \in [0,\Tboot]$,
		where $\updelta_1 > 0$ is the parameter from \eqref{E:DELTA1DEFINITIONANDPOSITIVITY}:
		\begin{align} \label{E:HOLDERESTIMATEFORSOLUTIONTOINHOMOGENEOUSTRANSPORT}
			\| \varphi \|_{C^{0,\updelta_1}(\Sigma_t)}
			& \lesssim 
				\| \mathring{\varphi} \|_{C^{0,\updelta_1}(\Sigma_0)}
				+
				\int_0^t
					\| \mathfrak{F} \|_{C^{0,\updelta_1}(\Sigma_{\uptau})}
				\, d \uptau.
		\end{align}
	\end{lemma}	
	
	\begin{proof} 
		Let $\upgamma(t;x)$ be the flow map of $\Transport$, as in Lemma~\ref{L:ESTIMATESFORTRANSPORTCHARACTERISTICS}.
		Then equation \eqref{E:INHOMOGENEOUSTRANSPORT} can be rewritten as
		$
		\frac{d}{dt} \left(\varphi \circ \upgamma(t;x) \right)
		= \mathfrak{F} 
		$.
		Integrating in time and using \eqref{E:DATAFORTRANSPORTODESYSTEM},
		we find that
		\begin{align} \label{E:TRANSPORTSOLUTIONALONGCHARACTERISTICSFTC}
		\varphi \circ \upgamma(t;x)
		-
		\varphi \circ \upgamma(t;y)
		& 
		=
		\mathring{\varphi}(x)
		-
		\mathring{\varphi}(y)
		+
		\int_0^t
			\left\lbrace
				\mathfrak{F}(\uptau,x) 
				- 
				\mathfrak{F}(\uptau,y)
			\right\rbrace
		\, d \uptau,
		\end{align}
		from which it easily follows that 
		\begin{align} \label{E:TRANSPORTSOLUTIONALONGCHARACTERISTICSEASYESTIMATE}
			|
				\varphi \circ \upgamma(t;x)
				-
				\varphi \circ \upgamma(t;y)
			|	
			& \leq \| \mathring{\varphi} \|_{C_x^{0,\updelta_1}}|x-y|^{\updelta_1}
				+
				|x-y|^{\updelta_1}
				\int_0^t
					\| \mathfrak{F} \|_{C^{0,\updelta_1}(\Sigma_{\uptau})} 
				\, d \uptau.
		\end{align}
		From \eqref{E:CHARACTERISTICSDIFFERENCEESTIMATE} and \eqref{E:TRANSPORTSOLUTIONALONGCHARACTERISTICSEASYESTIMATE},
		we deduce that
		\begin{align}  \label{E:TRANSPORTSOLUTIONALONGCHARACTERISTICSESTIMATEBASEDONCHARACTERISTICSDIFFERENCEESTIMATE}
				|
				\varphi \circ \upgamma(t;x)
				-
				\varphi \circ \upgamma(t;y)
			|	
				&
				\lesssim
				\| \mathring{\varphi} \|_{C_x^{0,\updelta_1}}
				| 
				\upgamma(t;x)
				-
				\upgamma(t;y)
				|^{\updelta_1}
				\\
			& \ \
				+
				| 
				\upgamma(t;x)
				-
				\upgamma(t;y)
				|^{\updelta_1}
				\int_0^t
					\| \mathfrak{F} \|_{C^{0,\updelta_1}(\Sigma_{\uptau})} 
				\, d \uptau.
			\notag
		\end{align}
		Since Lemma~\ref{L:ESTIMATESFORTRANSPORTCHARACTERISTICS} guarantees that
		the map $x \rightarrow \left(\upgamma^1(t,x),\upgamma^2(t,x),\upgamma^3(t,x)\right)$ is a smooth global diffeomorphism 
		from $\mathbb{R}^3$ to $\mathbb{R}^3$
		for each fixed $t \in [0,\Tboot]$, we conclude from \eqref{E:TRANSPORTSOLUTIONALONGCHARACTERISTICSESTIMATEBASEDONCHARACTERISTICSDIFFERENCEESTIMATE}
		that 
		\begin{align} \label{E:LARGEDISTANCEHOLDERESTIMATEFORSOLUTIONTOINHOMOGENEOUSTRANSPORT}
		\sup_{0 < |x - y|}
		\frac{|\varphi(t,x) - \varphi(t,y)|}{|x-y|^{\updelta_1}}
		& \leq 
		\| \mathring{\varphi} \|_{C_x^{0,\updelta_1}}
				+
				\| \mathfrak{F} \|_{L^1([0,t]) C_x^{0,\updelta_1}}.
\end{align}
Using a similar but simpler argument, based on the fundamental theorem of calculus,
we find that
\begin{align}
\notag
\| \varphi \|_{L^{\infty}(\Sigma_t)}
\lesssim 
\| \mathring{\varphi} \|_{L^{\infty}(\Sigma_0)}
+
\| \mathfrak{F} \|_{L^1([0,t]) L_x^{\infty}}
\end{align}
which, in view of definition \eqref{E:HOLDERNORMDEF}
and \eqref{E:LARGEDISTANCEHOLDERESTIMATEFORSOLUTIONTOINHOMOGENEOUSTRANSPORT}, yields
\eqref{E:HOLDERESTIMATEFORSOLUTIONTOINHOMOGENEOUSTRANSPORT}.

\end{proof}
	
	\subsection{Proof of Theorem~\ref{T:MIXEDSPACETIMEHOLDERESTIMATESFORVORTICITYANDENTROPY}}
		\label{SS:PROOFOFPROPMIXEDSPACETIMEHOLDERESTIMATESFORVORTICITYANDENTROPY}
		From equations \eqref{E:DIVVORTICITY}--\eqref{E:CURLGRADENT},
		the bootstrap assumption \eqref{E:BOOTSOLUTIONDOESNOTESCAPEREGIMEOFHYPERBOLICITY}, 
		the energy-elliptic estimate \eqref{E:TOPORDERENERGYESTIMATES},
		the standard estimates
\begin{align}
\notag
		\| F \cdot G \|_{C^{0,\updelta_1}(\Sigma_t)}
		\lesssim
		\| F \|_{C^{0,\updelta_1}(\Sigma_t)}
		\| G \|_{C^{0,\updelta_1}(\Sigma_t)}
\mbox{ and }
		\| [\gensmoothfunction \circ \vec{\varphi}] \cdot G \|_{C^{0,\updelta_1}(\Sigma_t)}
		\lesssim 
		\| \vec{\varphi} \|_{C^{0,\updelta_1}(\Sigma_t)}
		\| G \|_{C^{0,\updelta_1}(\Sigma_t)}
\end{align}
		(where the latter estimate is valid for any fluid variable array $\vec{\varphi}$ comprised of elements of 
		$\lbrace \LogDensity,\Ent,\vec{v},\vec{\vortrenormalized},\vec{\GradEnt} \rbrace$ 
		and any function $\gensmoothfunction$ 
		that is smooth on the domain of $\vec{\varphi}$ values corresponding to the set $\mathfrak{K}$
		from \eqref{E:BOOTSOLUTIONDOESNOTESCAPEREGIMEOFHYPERBOLICITY}),
		the standard embedding result
		$
		H^2(\Sigma_t) \hookrightarrow C^{0,\updelta_1}(\Sigma_t)
		$
		(which is valid since $\updelta_1 < 1/2$),
		and Young's inequality,
		we deduce that 
		\begin{align} \label{E:POINTWISEESTIMATESFORINHOMOGTERMSFORMODIFIEDVARIABLESEVOLUTIONEQUATIONS} 
			\|
				\Transport \vec{\VortVort}
			\|_{C^{0,\updelta_1}(\Sigma_t)}
			+
			\|
				\Transport \DivGradEnt
			\|_{C^{0,\updelta_1}(\Sigma_t)}
			& \lesssim 
					\| \pmb{\partial} \vec{\Psi} \|_{C^{0,\updelta_1}(\Sigma_t)}^2
					+
					\| \pmb{\partial} \vec{\Psi} \|_{C^{0,\updelta_1}(\Sigma_t)}
					\| \partial (\vec{\vortrenormalized}, \vec{\GradEnt}) \|_{C^{0,\updelta_1}(\Sigma_t)}
					+
					\| \partial \vec{\vortrenormalized} \|_{C^{0,\updelta_1}(\Sigma_t)}
					+
					1,
						\\
			\|
				\dive \vortrenormalized
			\|_{C^{0,\updelta_1}(\Sigma_t)}
			+
			\|
				\curl \GradEnt
			\|_{C^{0,\updelta_1}(\Sigma_t)}
			& \lesssim
				\| \pmb{\partial} \vec{\Psi} \|_{C^{0,\updelta_1}(\Sigma_t)}.
				\label{E:POINTWISEESTIMATESFORINHOMOGTERMSFORMODIFIEDVARIABLESNONEVOLUTIONEQUATIONS}
		\end{align}
		Using Def.\,\ref{D:MODIFIEDFLUIDVARIABLES} to algebraically solve for $\curl \vortrenormalized$ and $\dive \GradEnt$
		and using a similar argument,
		we deduce that
		\begin{align} \label{E:POINTWISEBOUNDSFORCURLVORTICITYANDDIVENTROPYGRADIENTINTERMSOFMODIFIEDVARIABLES}
				\|
					\curl \vortrenormalized
				\|_{C^{0,\updelta_1}(\Sigma_t)}
				+
				\|
					\dive \GradEnt
				\|_{C^{0,\updelta_1}(\Sigma_t)}
				& \lesssim
					\| 
						(\vec{\VortVort}, \DivGradEnt)
					\|_{C^{0,\updelta_1}(\Sigma_t)}
					+
					\| \pmb{\partial} \vec{\Psi} \|_{C^{0,\updelta_1}(\Sigma_t)}.
		\end{align}
		Next, from \eqref{E:HOLDERESTIMATEFORSOLUTIONTOINHOMOGENEOUSTRANSPORT} with
		$(\vec{\VortVort}, \DivGradEnt)$ in the role of $\varphi$,
		the H\"{o}lder bounds
		\eqref{E:POINTWISEESTIMATESFORINHOMOGTERMSFORMODIFIEDVARIABLESEVOLUTIONEQUATIONS}--\eqref{E:POINTWISEESTIMATESFORINHOMOGTERMSFORMODIFIEDVARIABLESNONEVOLUTIONEQUATIONS},
		and the data-bound 
\begin{align}
\notag
		\| (\vec{\VortVort}, \DivGradEnt) \|_{C^{0,\updelta_1}(\Sigma_0)} \lesssim 1,
\end{align}
		(which follows from \eqref{E:DELTA1DEFINITIONANDPOSITIVITY} and \eqref{E:FINITETRANSPORTDATANORM}),
		we deduce
		\begin{align} \label{E:ELLIPTICHOLDERESTIMATEFORTANSPORTVARIABLESGRONWALLREADY}
			\| (\vec{\VortVort}, \DivGradEnt) \|_{C^{0,\updelta_1}(\Sigma_t)}
			& \lesssim
				1
				+
				\int_0^t
					\| \pmb{\partial} \vec{\Psi} \|_{C^{0,\updelta_1}(\Sigma_t)}^2
				\, d \uptau
				+
				\int_0^t
					\left\lbrace
						\| \pmb{\partial} \vec{\Psi} \|_{C^{0,\updelta_1}(\Sigma_t)}
						+
						1
					\right\rbrace
					\| \partial (\vec{\vortrenormalized}, \vec{\GradEnt}) \|_{C^{0,\updelta_1}(\Sigma_t)}
				\, d \uptau.
		\end{align}
		Next, using the elliptic estimate \eqref{E:SCHAUDERESTIMATESFORDIVCURLSYSTEMS} with $\vortrenormalized$ and $\GradEnt$ in the role of $V$,
		\eqref{E:POINTWISEESTIMATESFORINHOMOGTERMSFORMODIFIEDVARIABLESNONEVOLUTIONEQUATIONS}--\eqref{E:POINTWISEBOUNDSFORCURLVORTICITYANDDIVENTROPYGRADIENTINTERMSOFMODIFIEDVARIABLES},
	and the energy estimate \eqref{E:TOPORDERENERGYESTIMATES},
	we find that the following estimate holds for $t \in [0,\Tboot]$:
	\begin{align} \label{E:SCHAUDERESTIMATESFORVORTICITYANDENTOPYGRADIENTDIVCURLSYSTEMS}
	\| \partial (\vec{\vortrenormalized}, \vec{\GradEnt}) \|_{C^{0,\updelta_1}(\Sigma_t)}
	& \lesssim 
					\| (\vec{\VortVort}, \DivGradEnt) \|_{C^{0,\updelta_1}(\Sigma_t)}
					+
					\| \pmb{\partial} \vec{\Psi} \|_{C^{0,\updelta_1}(\Sigma_t)}
					+ 
					1.
\end{align}
Using \eqref{E:SCHAUDERESTIMATESFORVORTICITYANDENTOPYGRADIENTDIVCURLSYSTEMS} to bound the
factor $\| \partial (\vec{\vortrenormalized}, \vec{\GradEnt}) \|_{C^{0,\updelta_1}(\Sigma_{\uptau})}$
on RHS~\eqref{E:ELLIPTICHOLDERESTIMATEFORTANSPORTVARIABLESGRONWALLREADY},
applying Gr\"{o}nwall's inequality in the term $\| (\vec{\VortVort}, \DivGradEnt) \|_{C^{0,\updelta_1}(\Sigma_t)}$,
and using \eqref{E:HOLDERTYPESTRICHARTZESTIMATEFORWAVEVARIABLES},
we find that
\begin{align} \label{E:TRANSPORTVARIABLESGRONWALLED}
	\| (\vec{\VortVort}, \DivGradEnt) \|_{C^{0,\updelta_1}(\Sigma_t)}
		& \lesssim 1.
\end{align}
We have therefore proved \eqref{E:HOLDERREGULARITYOFMODIFIEDVARIABLESPROPAGATED}.
Then, using \eqref{E:TRANSPORTVARIABLESGRONWALLED} to bound the first term on RHS~\eqref{E:SCHAUDERESTIMATESFORVORTICITYANDENTOPYGRADIENTDIVCURLSYSTEMS},
squaring the resulting inequality and integrating it in time,
and using \eqref{E:HOLDERTYPESTRICHARTZESTIMATEFORWAVEVARIABLES},
we arrive at the desired estimate \eqref{E:MIXEDSPACETIMEHOLDERESTIMATESFORVORTICITYANDENTROPY}.

	\eqref{E:SUPPEDMIXEDSPACETIMEHOLDERESTIMATESFORVORTICITYANDENTROPY}
	then follows from \eqref{E:MIXEDSPACETIMEHOLDERESTIMATESFORVORTICITYANDENTROPY},
	the following well-known estimate (see, for example, \cite{MT1991}*{Equation~(A.1.2)} and the discussion surrounding it), 
	valid for scalar functions $f$:
	$
	\sup_{\upnu \geq 1}
	\upnu^{\updelta_1}
	\|
		P_{\upnu} f
	\|_{L^{\infty}(\Sigma_t)}
	\lesssim 
	\|
		f
	\|_{C^{0,\updelta_1}(\Sigma_t)}
	$,
	and the fact that the dyadic sum $\sum_{\upnu \geq 1} \upnu^{-\updelta_1}$ is finite.
	This completes the proof of Theorem~\ref{T:MIXEDSPACETIMEHOLDERESTIMATESFORVORTICITYANDENTROPY}.
	
	\hfill $\qed$

\section{The setup of the proof of Theorem~\ref{T:FREQUENCYLOCALIZEDSTRICHARTZ}: the rescaled solution and construction of the eikonal function}
\label{S:SETUPCONSTRUCTIONOFEIKONAL}
To complete our bootstrap argument and finish the proof of Theorem~\ref{T:MAINTHEOREMROUGHVERSION},
we have one remaining arduous task: proving Theorem~\ref{T:FREQUENCYLOCALIZEDSTRICHARTZ}.
We accomplish this in Sects.\,\ref{S:SETUPCONSTRUCTIONOFEIKONAL}--\ref{S:REDUCTIONSOFSTRICHARTZ}.
In this section, we set up the geometric and analytic framework
that we use in the rest of the paper.
As in the works \cites{sKiR2003,sKiR2005d,qW2017}, the
main ingredients are an appropriate rescaling of the solution,\footnote{In \cite{qW2017}, instead of rescaling the solution, the author
worked with rescaled coordinates. These two approaches are equivalent. \label{FN:RESCALING}}
an eikonal function $u$ with suitable initial conditions,
and a collection of geometric tensorfields constructed out of $u$.
Compared to previous works,
the main new contribution of the present section
is located in Subsubsect.\,\ref{SSS:PDESMODIFIEDACOUSTICALQUANTITIES},
where we derive various PDEs satisfied by the geometric tensorfields;
there, one explicitly sees how the source terms in these geometric PDEs 
depend on the vorticity and entropy, 
and some of the precise structures in these PDEs are crucial for our analysis.

\subsection{\texorpdfstring{The rescaled quantities and the radius $R$}{The rescaled quantities and the radius R}}
\label{SS:RESCALEDSOLUTION}

\subsubsection{The rescaled quantities}
\label{SSS:RESCALEDQUANTITIES}
Let $\lbrace [t_k,t_{k+1}] \rbrace_{k=1,2,\cdots}$ be the (finite collection of) time intervals introduced
in Subsect.\,\ref{SS:PARTITIONOFBOOTSTRAPINTERVAL}, and let $\Lambda_0 > 0$ be the large parameter introduced there.
For any fixed dyadic frequency $\uplambda \geq \Lambda_0$, let 
\begin{align} \label{E:RESCALEDTBOOT}
	\RescaledTboot & := \uplambda(t_{k+1} - t_k).
\end{align}
Note that since (by construction) $|t_{k+1} - t_k| \leq \uplambda^{- 8 \upepsilon_0} \Tboot$,
it follows that 
\begin{align} \label{E:RESCALEDBOOTBOUNDS}
	0 \leq \RescaledTboot \leq \uplambda^{1-8 \upepsilon_0} \Tboot.
\end{align}

We now define the ``rescaled'' solution variables that we will analyze in the rest of the paper.

\begin{definition}[Rescaled quantities]
	\label{D:RESCALEDRENORMALIZEDCURLOFSPECIFICVORTICITY}
We define the array of scalar functions
\begin{align}
\notag
\vec{\Psi}_{(\uplambda)} = (\LogDensity_{(\uplambda)},v_{(\uplambda)}^1,v_{(\uplambda)}^2,v_{(\uplambda)}^3,\Ent_{(\uplambda)})
\end{align}
and the Cartesian components of the $\Sigma_t$-tangent vectorfields
$\vortrenormalized_{(\uplambda)}$
and
$\GradEnt_{(\uplambda)}$
as follows, 
($i=1,2,3$):
\begin{align} \label{E:RESCALEDSOLUTIONVARIABLES}
	\vec{\Psi}_{(\uplambda)}(t,x) 
	& := \vec{\Psi}(t_k + \uplambda^{-1} t,\uplambda^{-1} x),
	&
	\vortrenormalized_{(\uplambda)}^i(t,x) 
	& := \vortrenormalized^i(t_k + \uplambda^{-1} t,\uplambda^{-1} x),
	&
	\\
	\notag
	\GradEnt_{(\uplambda)}^i(t,x) 
	& := \GradEnt^i(t_k + \uplambda^{-1} t,\uplambda^{-1} x).
\end{align}
	
	Similarly, we define the Cartesian components of the $\Sigma_t$-tangent vectorfield $\VortVort_{(\uplambda)}$ 
	and the scalar function $\DivGradEnt_{(\uplambda)}$ as follows:
	\begin{subequations}
	\begin{align} \label{E:RESCALEDRENORMALIZEDCURLOFSPECIFICVORTICITY}
		 \VortVort_{(\uplambda)}^i
		& :=
			\exp(-\LogDensity_{(\uplambda)}) (\Flatcurl \vortrenormalized_{(\uplambda)})^i
			+
			\exp(-3\LogDensity_{(\uplambda)}) \Speed^{-2}(\vec{\Psi}_{(\uplambda)}) \frac{p_{;\Ent}(\vec{\Psi}_{(\uplambda)})}{\bar{\varrho}}
			\GradEnt_{(\uplambda)}^a \partial_a v_{(\uplambda)}^i
				\\
		& \ \
			-
			\exp(-3\LogDensity_{(\uplambda)}) \Speed^{-2}(\vec{\Psi}_{(\uplambda)}) \frac{p_{;\Ent}(\vec{\Psi}_{(\uplambda)})}{\bar{\varrho}} 
			(\partial_a v_{(\uplambda)}^a) \GradEnt_{(\uplambda)}^i,
				\notag \\
		\DivGradEnt_{(\uplambda)}
		& := 
			\exp(-2 \LogDensity_{(\uplambda)}) \Flatdiv \GradEnt_{(\uplambda)} 
			-
			\exp(-2 \LogDensity_{(\uplambda)}) \GradEnt_{(\uplambda)}^a \partial_a \LogDensity_{(\uplambda)}.
			\label{E:RESCALEDRENORMALIZEDDIVOFENTROPY}
	\end{align}
	\end{subequations}
	
Finally, we let
$\gfour_{(\uplambda)}$,
$g_{(\uplambda)}$,
and
$\Transport_{(\uplambda)}$
be the ``rescaled'' tensorfields whose Cartesian components are as follows,
($\alpha, \beta = 0,1,2,3$ and $i,j=1,2,3$):
\begin{subequations}
\begin{align} \label{E:RESCALEDMETRICTENSORS}
(\gfour_{(\uplambda)})_{\alpha \beta}(t,x) 
&:= \gfour_{\alpha \beta}\left(\vec{\Psi}(t_k + \uplambda^{-1} t,\uplambda^{-1} x)\right),
&
(g_{(\uplambda)})_{ij}(t,x) 
&
:= g_{ij}\left(\vec{\Psi}(t_k + \uplambda^{-1} t,\uplambda^{-1} x)\right),
	\\
\Transport_{(\uplambda)}^{\alpha}(t,x) 
&
:= \Transport^{\alpha}\left(\vec{\Psi}(t_k + \uplambda^{-1} t,\uplambda^{-1} x)\right).
&&
\label{E:RESCALEDMATERIALVECTORFIELD}
\end{align} 
\end{subequations}
\end{definition}

\begin{remark}[Remarks on the rescaling]
\label{R:REMARKSONRESCALING}
Note that the slab $[0,\RescaledTboot] \times \mathbb{R}^3$ for $\vec{\Psi}_{(\uplambda)}(t,x)$
corresponds to the slab $[t_k,t_{k+1}] \times \mathbb{R}^3$ for $\vec{\Psi}(t,x)$.
The same remark applies for the other rescaled quantities.

Note also that when we are controlling the rescaled quantities such as
$\vec{\Psi}_{(\uplambda)}$,
the hypersurface that we denote by ``$\Sigma_t$'' in 
Sects.\,\ref{S:SETUPCONSTRUCTIONOFEIKONAL}--\ref{S:REDUCTIONSOFSTRICHARTZ}
corresponds to the hypersurface $\Sigma_{t_k + \uplambda^{-1} t}$
for the non-rescaled quantities,
which appear throughout
Sects.\,\ref{S:DATAANDBOOTSTRAPASSUMPTION}--\ref{S:ELLIPTICESTIMATESINHOLDERSPACES}.
\end{remark}

\begin{remark}
	Note that $\GradEnt_{(\uplambda)}^i \neq \partial_i \Ent_{(\uplambda)}$, but rather 
	$\GradEnt_{(\uplambda)}^i = \uplambda \partial_i \Ent_{(\uplambda)}$.
	This is merely a reflection of our choice of how to keep track of powers of $\uplambda$ in the equations and estimates. 
	Similarly, we have
	$
	\displaystyle
	\vortrenormalized_{(\uplambda)}^i = \uplambda \frac{(\Flatcurl v_{(\uplambda)})^i}{\exp \LogDensity_{(\uplambda)}}
	$.
	We clarify that although we use the
	relationships $\GradEnt_{(\uplambda)}^i = \uplambda \partial_i \Ent_{(\uplambda)}$
	and
	$
	\displaystyle
	\vortrenormalized_{(\uplambda)}^i = \uplambda \frac{(\Flatcurl v_{(\uplambda)})^i}{\exp \LogDensity_{(\uplambda)}}$
	to derive the equations of Prop.\,\ref{P:RESCALEDEULER},
	when we derive PDE estimates for solutions to these equations, 
	we generally do not need these relationships;
	that is, for estimates, 
	we generally treat
	$\GradEnt_{(\uplambda)}$,
	$\Ent_{(\uplambda)}$, 
	$\vortrenormalized_{(\uplambda)}$,
	$v_{(\uplambda)}$,
	and
	$\LogDensity_{(\uplambda)}$
	as if they were independent quantities.
\end{remark}

\subsubsection{The radius $R$}
For any $t \in [0,\RescaledTboot]$, $p \in \Sigma_t$, and $r > 0$,
let $B_r(p)$ denote the Euclidean ball of radius $r$ in $\Sigma_t$ centered at $p$
and let $B_{r;g_{(\uplambda)}(t,\cdot)}(p)$ denote the metric ball,
with respect to the rescaled Riemannian metric $g_{(\uplambda)}(t,\cdot)$,
of radius $r$ in $\Sigma_t$ centered at $p$.
The statement of Theorem~\ref{T:BOUNDEDNESSOFCONFORMALENERGY} refers 
to a Euclidean radius $R$, which we now define. Specifically, 
in the rest of the article, 
$R$ denotes a fixed number chosen such that 
\begin{subequations}
\begin{align} 
	0 & < R < 1, &&
		\label{E:RSMALLERTHANONE} \\
	B_R(p) 
	& \subset B_{1/2;g_{(\uplambda)}(t,\cdot)}(p),
	&&
	\forall
	t \in [0,\RescaledTboot] \mbox{ and } \forall p \in \Sigma_t.
	\label{E:EUCLIDEANBALLCONTAINEDINMETRICBALL}
\end{align}
\end{subequations}
The existence of such an $R$ (one that is independent of $\uplambda$) is guaranteed by the
formula \eqref{E:FIRSTFUNDAMENTALFORM}
(which in particular shows that $g_{(\uplambda)}(t,\cdot)$ is equal to $\Speed^{-2}$ times
the Euclidean metric on $\Sigma_t$, with $\Speed$ the speed of sound)
and the fact that,
by virtue of the bootstrap assumption \eqref{E:BOOTSOLUTIONDOESNOTESCAPEREGIMEOFHYPERBOLICITY},
$\Speed$ is uniformly bounded from above and below by positive constants.

\subsection{The rescaled compressible Euler equations}
\label{SS:COMPRESSIBLEEULERRESCALEDGEOMETRIC}
In the next proposition, we provide the equations verified by the rescaled quantities.
We omit the simple proof, which follows from scaling considerations.

\begin{proposition}[The rescaled geometric wave-transport formulation of the compressible Euler equations]
\label{P:RESCALEDEULER}
		For solutions to Prop.\,\ref{P:GEOMETRICWAVETRANSPORT},
		the rescaled quantities defined in Subsect.\,\ref{SS:RESCALEDSOLUTION}
		verify the following equations.
		
		\medskip
		
	\noindent \underline{\textbf{Wave equations}}:
For rescaled wave variables 
$\Psi_{(\uplambda)} \in \lbrace \LogDensity_{(\uplambda)}, v_{(\uplambda)}^1, v_{(\uplambda)}^2, v_{(\uplambda)}^3, \Ent_{(\uplambda)} \rbrace$, 
we have:
\begin{align} \label{E:RESCALEDCOVARIANTWAVE} 
\hat{\square}_{\gfour_{(\uplambda)}} \Psi_{(\uplambda)} 
& = 
	\uplambda^{-1}
	\linsmoothfunction(\vec{\Psi}_{(\uplambda)})[\vec{\VortVort}_{(\uplambda)},\DivGradEnt_{(\uplambda)}]
	+ 
	\quadsmoothfunction(\vec{\Psi}_{(\uplambda)})[\pmb{\partial} \vec{\Psi}_{(\uplambda)},\pmb{\partial} \vec{\Psi}_{(\uplambda)}].
\end{align}

\medskip

\noindent  \underline{\textbf{Transport equations}}:
\begin{align} \label{E:RESCALEDTRANSPORT}
\Transport_{(\uplambda)} \vortrenormalized_{(\uplambda)}^i 
& 
= \linsmoothfunction(\vec{\Psi}_{(\uplambda)},\vec{\vortrenormalized}_{(\uplambda)},\vec{\GradEnt}_{(\uplambda)})[\pmb{\partial} \vec{\Psi}_{(\uplambda)}],
&
\Transport_{(\uplambda)} \GradEnt_{(\uplambda)}^i 
& = \linsmoothfunction(\vec{\Psi}_{(\uplambda)},\vec{\GradEnt}_{(\uplambda)})[\pmb{\partial} \vec{\Psi}_{(\uplambda)}].
\end{align}

\medskip

\noindent \underline{\textbf{Transport div-curl system for the specific vorticity}}:
\begin{subequations}
\begin{align}
\dive \vortrenormalized_{(\uplambda)} 
	& = 
		\linsmoothfunction(\vec{\vortrenormalized}_{(\uplambda)})[\pmb{\partial} \vec{\Psi}_{(\uplambda)}],
\label{E:RESCALEDDIVVORTICITY}
\\
\Transport_{(\uplambda)} \VortVort_{(\uplambda)}^i 
	& = 
		\uplambda
		\quadsmoothfunction(\vec{\Psi}_{(\uplambda)})[\pmb{\partial} \vec{\Psi}_{(\uplambda)},\partial \vec{\vortrenormalized}_{(\uplambda)}]
		+
		\uplambda
		\quadsmoothfunction(\vec{\Psi}_{(\uplambda)})[\pmb{\partial} \vec{\Psi}_{(\uplambda)},\partial \vec{\GradEnt}_{(\uplambda)}]
		+
		\uplambda
		\quadsmoothfunction(\vec{\Psi}_{(\uplambda)},\vec{\GradEnt}_{(\uplambda)})[\pmb{\partial} \vec{\Psi}_{(\uplambda)},\pmb{\partial} \vec{\Psi}_{(\uplambda)}]
			\label{E:RESCALEDTRANSPORTVORTICITYVORTICITY}
			\\
	& \ \
		+
		\linsmoothfunction(\vec{\Psi}_{(\uplambda)},\vec{\vortrenormalized}_{(\uplambda)},\vec{\GradEnt}_{(\uplambda)})[\pmb{\partial} \vec{\Psi}_{(\uplambda)}].
	\notag
\end{align}
\end{subequations}

\medskip

\noindent \underline{\textbf{Transport div-curl system for the entropy gradient}}:
\begin{subequations}
\begin{align}
	\Transport_{(\uplambda)} \DivGradEnt_{(\uplambda)} & 
		=
		\uplambda
		\quadsmoothfunction(\vec{\Psi}_{(\uplambda)})[\pmb{\partial} \vec{\Psi}_{(\uplambda)},\partial \vec{\GradEnt}_{(\uplambda)}]
		+
		\uplambda
		\quadsmoothfunction(\vec{\Psi}_{(\uplambda)},\vec{\GradEnt}_{(\uplambda)})[\pmb{\partial} \vec{\Psi}_{(\uplambda)},\pmb{\partial} \vec{\Psi}_{(\uplambda)}]
		+
		\linsmoothfunction(\vec{\Psi}_{(\uplambda)},\vec{\GradEnt}_{(\uplambda)})[\partial \vec{\vortrenormalized}_{(\uplambda)}],
		\label{E:RESCALEDTRANSPORTDIVGRADENTROPY}
\\
(\curl \GradEnt_{(\uplambda)})^i & = 0.
\label{E:RESCALEDCURLGRADENT}
\end{align}
\end{subequations}
\end{proposition}

\subsection{Key notational remark and the mixed spacetime norm bootstrap assumptions for the rescaled quantities}
\label{SS:NOMORELAMBDA}
For notational convenience, in the remainder of the article, we drop the sub- and super-scripts ``$(\uplambda)$''
introduced in Subsect.\,\ref{SS:RESCALEDSOLUTION}, except for the rescaled time $T_{*,(\uplambda)}$. That is,
we write $\vec{\Psi}$ in place of $\vec{\Psi}_{(\uplambda)}$,
$\gfour$ in place of $\gfour_{(\uplambda)}$,
$\gfour_{\alpha \beta}(t,x)$ in place of 
$\gfour_{\alpha \beta}\left(\vec{\Psi}(t_k + \uplambda^{-1} t,\uplambda^{-1} x)\right)$,
etc. Nonetheless, our analysis
will properly take into account the explicit factors of $\uplambda$
on the RHSs of the equations of Prop.\,\ref{P:RESCALEDEULER}.

\subsection{\texorpdfstring{$\mathcal{M}$, the point $\bf{z}$, the eikonal function, and construction of the geometric coordinates}{mathcal M, the point bf z, the eikonal function, and construction of the geometric coordinates}}
\label{SS:EIKONAL}
Let $\mathcal{M} := [0,\RescaledTboot] \times \mathbb{R}^3 \subset \mathbb{R}^{1+3}$
denote the slab on which the rescaled quantities of Subsubsect.\,\ref{SSS:RESCALEDQUANTITIES}
are defined. In the rest of the paper, 
we will construct various geometric quantities and derive estimates on various subsets of $\mathcal{M}$.

The proof of Theorem~\ref{T:BOUNDEDNESSOFCONFORMALENERGY} fundamentally relies on the \emph{acoustic geometry},
that is, a solution $u$ to the eikonal equation 
(where under the conventions of Subsect.\,\ref{SS:NOMORELAMBDA}, ``$\gfour$'' denotes the rescaled metric):
\begin{align} \label{E:EIKONALEQUATIONFORRESCALEDMETRIC}
	(\gfour^{-1})^{\alpha \beta} 
	\partial_{\alpha} u 
	\partial_{\beta} u
	& = 0.
\end{align}
Following the setup used in \cite{qW2017}, we will construct
$u$ by patching an ``interior solution'' with an ``exterior solution.''
More precisely, the results of Subsubsects.\,\ref{SSS:EIKONALINTERIOR}--\ref{SSS:EIKONALEXTERIOR} 
will yield an eikonal function $u$
defined in subsets $\widetilde{\mathcal{M}} \subset \mathcal{M}$,
which we will define to be the union of an interior region and an exterior region:
$\widetilde{\mathcal{M}} := \widetilde{\mathcal{M}}^{(Int)} \cup \widetilde{\mathcal{M}}^{(Ext)}$. Moreover, an exercise in Taylor expansions, omitted here, yields that the solution $u$ is smooth in $\widetilde{\mathcal{M}}$ away from the cone-tip axis 
(which is a curve in $\widetilde{\mathcal{M}}^{(Int)}$ that we define in Subsubsect.\,\ref{SSS:EIKONALINTERIOR}). 

Throughout Sects.\,\ref{S:SETUPCONSTRUCTIONOFEIKONAL} and \ref{S:ESTIMATESFOREIKONALFUNCTION},
$\bf{z}$ denotes a fixed (but arbitrary) point in $\Sigma_0$ 
(where here, ``$\Sigma_0$'' corresponds to the hypersurface that we denoted by 
``$\Sigma_{t_k}$'' in Sects.\,\ref{S:DATAANDBOOTSTRAPASSUMPTION}--\ref{S:SETUPCONSTRUCTIONOFEIKONAL})
that forms the bottom tip of $\widetilde{\mathcal{M}}^{(Int)}$.  
The point $\bf{z}$ will vary when one carries out the partition of unity argument
that allows for a reduction of the proof of the desired Strichartz estimate 
(more precisely, the frequency localized estimates provided by Theorem~\ref{T:FREQUENCYLOCALIZEDSTRICHARTZ})
to that of Prop.\,\ref{P:SPATIALLYLOCALIZEDREDUCTIONOFPROOFOFTHEOREMDECAYESTIMATE}.
More precisely, the proof of Theorem~\ref{T:FREQUENCYLOCALIZEDSTRICHARTZ}
relies on partitioning the full slab $\mathcal{M}$ into
various ``localized'' subsets of type $\widetilde{\mathcal{M}}$
and proving dispersive decay estimates on subsets of the $\widetilde{\mathcal{M}}$
for solutions $\varphi$ to the linear wave equation 
$\square_{\gfour(\vec{\Psi})} \varphi = 0$.
The spatially localized dispersive decay estimates
(which correspond to a fixed $\widetilde{\mathcal{M}}$ and thus a fixed $\bf{z}$)
are provided by
Prop.\,\ref{P:SPATIALLYLOCALIZEDREDUCTIONOFPROOFOFTHEOREMDECAYESTIMATE}.
We refer readers to
Subsect.\,\ref{SS:SPATIALLYLOCALIZEDREDUCTIONOFPROOFOFTHEOREMDECAYESTIMATE} for further discussion
on the various standard reductions of the proof of the Strichartz estimates to spatially localized dispersive
estimates 
(and ultimately to the proof of control over the growth rate of a conformal energy,
provided by Theorem~\ref{T:BOUNDEDNESSOFCONFORMALENERGY}). 
We also remark that the
varying of $\bf{z}$ during the partition of unity argument
is a minor issue in the sense that
estimates that we derive in Sects.\,\ref{S:SETUPCONSTRUCTIONOFEIKONAL} and \ref{S:ESTIMATESFOREIKONALFUNCTION}
are independent of $\bf{z}$, and all of the constants and parameters in our analysis
can be chosen to be independent of $\bf{z}$.

We provide a figure, Fig.\,\ref{F:DOMAINS}, that exhibits many 
of the geometric objects that we will construct in Subsect.\,\ref{SS:EIKONAL}.
In the figure, for convenience, we have set $\bf{z}$ to be equal to the origin in $\Sigma_0$.

\begin{figure}[h!]
\centering
\includegraphics[width=6.5in,scale=.5]{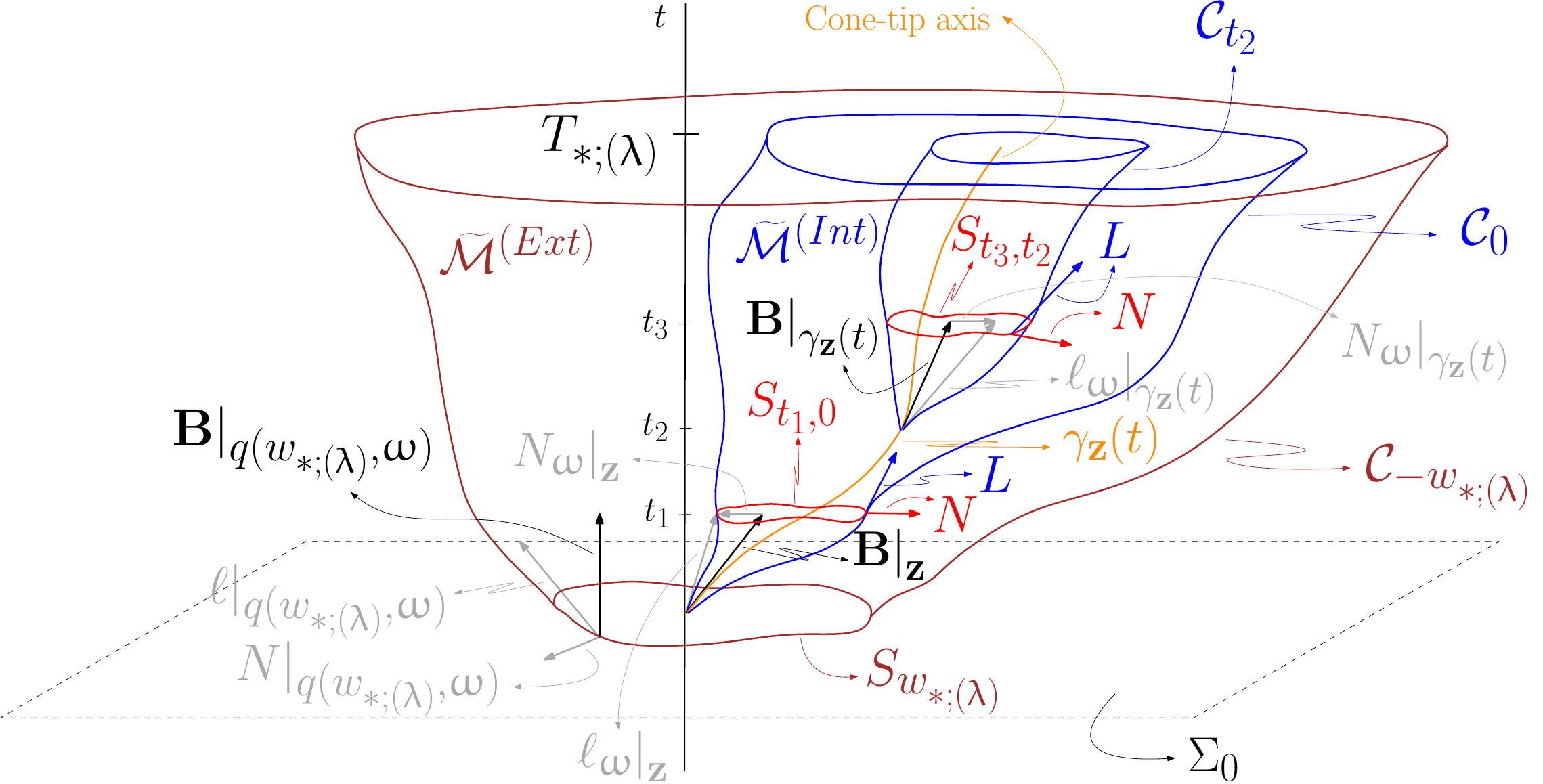}
\caption{The interior and exterior regions and related geometric constructions in the case ${\bf{z}} := {\bf{0}}$}
\label{F:DOMAINS}
\end{figure}

\subsubsection{The interior solution emanating from the cone-tip axis and the region $\widetilde{\mathcal{M}}^{(Int)}$}
\label{SSS:EIKONALINTERIOR}
We let $\Tranchar_{\bf{z}} = \Tranchar_{\bf{z}}(t)$ denote the future-directed integral curve of
the vectorfield\footnote{We again stress that
by the conventions of Subsect.\,\ref{SS:NOMORELAMBDA}, in the rest of the paper, we use the notation
$\Transport^{\alpha}(t,x)$ to denote $\Transport^{\alpha}\left(\vec{\Psi}(t_k + \uplambda^{-1} t,\uplambda^{-1} x)\right)$
and $\gfour_{\alpha \beta}(t,x)$ to denote $\gfour_{\alpha \beta}\left(\vec{\Psi}(t_k + \uplambda^{-1} t,\uplambda^{-1} x)\right)$.}  
$\Transport$ emanating from the point $\bf{z}$, i.e., $\Tranchar_{\bf{z}}(0) = {\bf{z}} \in \Sigma_0$. 
We refer to $\lbrace \Tranchar_{\bf{z}}(t) \rbrace_{t \in [0,\RescaledTboot]}$ as the \emph{cone-tip axis}.
Let $q = q(t) := \Tranchar_{\bf{z}}(t)$ be a point on the cone-tip axis. Let $\ell \in T_q \mathcal{M}$
be a null vector normalized by $\gfour|_q(\ell,\Transport|_q) = - 1$.
We denote the set of all of these normalized null vectors $\ell \in T_q \mathcal{M}$ by $\mathscr{N}_q$.
We now consider the case $q = {\bf{z}} \in \Sigma_0$. It is straightforward to see that $\mathscr{N}_{\bf{z}}$ is diffeomorphic to $\mathbb{S}^2$;
we therefore fix a diffeomorphism from $\mathbb{S}^2$ onto $\mathscr{N}_{\bf{z}}$.
For each $\upomega \in \mathbb{S}^2$, we let $\ell_{\upomega} \in \mathscr{N}_{\bf{z}}$ 
denote the corresponding (via the diffeomorphism) null vector. 
We will use parallel transport 
to construct a diffeomorphism from 
$\mathscr{N}_{\bf{z}}$ onto $\mathscr{N}_{\Tranchar_{\bf{z}}(t)}$.
Ultimately, this diffeomorphism will allow us, upon pre-composing it with the fixed diffeomorphism $\upomega \rightarrow \ell_{\upomega}$
from $\mathbb{S}^2$ onto $\mathscr{N}_{\bf{z}}$ and post-composing it with a null geodesic flow,\footnote{The null curves, whose Cartesian components are solutions to the ODE system \eqref{E:TIPNULLGEODESICS}, are not affine-parameterized. \label{FN:AFFINE}}
to construct angular coordinates $\upomega$ that are defined\footnote{More precisely,
the angular coordinate functions $(\upomega^1,\upomega^2)$ are uniquely defined away from the cone-tip axis, 
while each point on the cone-tip axis
is associated with an entire $\mathbb{S}^2$ manifold worth of angles
(i.e., the same degeneracy that occurs at the origin in $\mathbb{R}^3$ under the standard Euclidean spherical coordinates).}
in $\widetilde{\mathcal{M}}^{(Int)}$; 
see just below equation \eqref{E:INITIALCONDITIONSTIPNULLGEODESICS}.

To initiate the construction of the diffeomorphism from $\mathscr{N}_{\bf{z}}$ onto $\mathscr{N}_{\Tranchar_{\bf{z}}(t)}$, 
for each $\upomega \in \mathbb{S}^2$,
we define the vector $\spherenormal_{\upomega} \in T_{\bf{z}} \mathcal{M}$ as follows: 
$\spherenormal_{\upomega} := \ell_{\upomega} - \Transport|_{\bf{z}}$.
Considering the relations $\gfour|_{\bf{z}}(\ell_{\upomega},\Transport|_{\bf{z}}) = - 1$
and
$\gfour|_{\bf{z}}(\Transport|_{\bf{z}},\Transport|_{\bf{z}}) = - 1$,
we find that
$\gfour|_{\bf{z}}(\Transport|_{\bf{z}},\spherenormal_{\upomega}) = 0$.
Considering also that $\gfour|_{\bf{z}}(\ell_{\upomega},\ell_{\upomega}) = 0$, we
find that $\gfour|_{\bf{z}}(\spherenormal_{\upomega},\spherenormal_{\upomega}) = 1$.
Thus, $\spherenormal_{\upomega} \in UT_{\bf{z}} \Sigma_0$,
where $UT_{\bf{z}} \Sigma_0$ denotes the $g$-unit tangent bundle of $\Sigma_0$ at $\bf{z}$,
and $g$ is the rescaled first fundamental form of $\Sigma_0$.
It is straightforward to see that
the map $\ell_{\upomega} \rightarrow \ell_{\upomega} - \Transport|_{\bf{z}}$
defines a diffeomorphism from $\mathscr{N}_{\bf{z}}$ onto $UT_{\bf{z}} \Sigma_0$.
To propagate $\spherenormal_{\upomega}$ along the cone-tip axis, we solve the
parallel transport equation\footnote{This
is in fact parallel transport along geodesics since
$\Dfour_{\Transport} \Transport= 0$; this latter identity
is straightforward to derive using that $\gfour(\Transport,\Transport) = -1$ and
the fact that $[\Transport,Z]$ is $\Sigma_t$-tangent (hence $\gfour$-orthogonal to $\Transport$) whenever $Z$ is $\Sigma_t$-tangent.}
$\Dfour_{\Transport} \spherenormal_{\upomega} 
=0$,
where $\Dfour$ is the Levi-Civita connection of the rescaled spacetime metric $\gfour$.
In Cartesian coordinates,
for each $\spherenormal_{\upomega}^{\alpha}|_{\bf{z}} \in \mathscr{N}_{\bf{z}}$,
the parallel transport equation takes the form of the following transport equation system,
which is \emph{linear} in the scalar Cartesian component functions $\spherenormal_{\upomega}^{\alpha}$:
\begin{align} \label{E:FWTRANSPORT}
	\frac{d}{dt} \spherenormal_{\upomega}^{\alpha} 
	+
	\Chfour_{\kappa \ \lambda}^{\ \alpha}  
	\Transport^{\kappa} 
	\spherenormal_{\upomega}^{\lambda}
	& = 
		0,
\end{align}
where the initial conditions for \eqref{E:FWTRANSPORT} are $\spherenormal_{\upomega}^{\alpha}|_{\bf{z}}$,
$
	\Chfour_{\alpha \ \beta}^{\ \nu}
	=
	\frac{1}{2} (\gfour^{-1})^{\sigma \nu}
	\left\lbrace\partial_{\alpha} \gfour_{\sigma \beta} 
	+ \partial_{\beta} \gfour_{\alpha \sigma}
	-
	\partial_{\sigma} \gfour_{\alpha \beta}
	\right\rbrace
	$
are the Cartesian Christoffel symbols of the rescaled metric $\gfour$,
and it is understood that all quantities are evaluated along $\Tranchar_{\bf{z}}(t)$,
e.g., $\spherenormal_{\upomega}^{\alpha} = \spherenormal_{\upomega}^{\alpha} \circ \Tranchar_{\bf{z}}(t)$
and
$\Transport^{\kappa} = \Transport^{\kappa} \circ \vec{\Psi} \circ \Tranchar_{\bf{z}}(t)$,
with $\vec{\Psi}$ the rescaled solution. It is straightforward to show,
based on the normalization condition $\gfour|_{\Tranchar_{\bf{z}}(t)}(\Transport|_{\Tranchar_{\bf{z}}(t)},\Transport|_{\Tranchar_{\bf{z}}(t)}) = -1$,
\eqref{E:FWTRANSPORT}, and the initial conditions 
$\gfour|_{\bf{z}}(\Transport|_{\bf{z}},\spherenormal_{\upomega}|_{\bf{z}}) = 0$
and $\gfour|_{\bf{z}}(\spherenormal_{\upomega}|_{\bf{z}},\spherenormal_{\upomega}|_{\bf{z}}) = 1$,
that for $t \in [0, \RescaledTboot]$, the solution $\spherenormal_{\upomega}|_{\Tranchar_{\bf{z}}(t)}$ to equation
\eqref{E:FWTRANSPORT} is an element of $UT_{\Tranchar_{\bf{z}}(t)} \Sigma_t$, 
where $UT_{\Tranchar_{\bf{z}}(t)} \Sigma_t$ denotes the $g$-unit tangent bundle of $\Sigma_t$ at $\Tranchar_{\bf{z}}(t)$,
and $g$ is the rescaled first fundamental form of $\Sigma_t$.
That is, we have
$\gfour|_{\Tranchar_{\bf{z}}(t)}(\Transport|_{\Tranchar_{\bf{z}}(t)},\spherenormal_{\upomega}|_{\Tranchar_{\bf{z}}(t)}) = 0$
and $\gfour|_{\Tranchar_{\bf{z}}(t)}(\spherenormal_{\upomega}|_{\Tranchar_{\bf{z}}(t)},\spherenormal_{\upomega}|_{\Tranchar_{\bf{z}}(t)}) = 1$.
In particular, $\spherenormal_{\upomega}|_{\Tranchar_{\bf{z}}(t)}$ is tangent to $\Sigma_t$ at $\Tranchar_{\bf{z}}(t)$.
From these relations and arguments similar to the ones given above, 
we find that
$\ell_{\upomega}|_{\Tranchar_{\bf{z}}(t)} := \Transport|_{\Tranchar_{\bf{z}}(t)} + \spherenormal_{\upomega}|_{\Tranchar_{\bf{z}}(t)}
\in \mathscr{N}_{\Tranchar_{\bf{z}}(t)}
$.
Similar arguments that take into account standard ODE existence and uniqueness theory\footnote{Here we are using our qualitative assumption that the fluid solution
is smooth. \label{FN:SMOOTHFORODE}} 
for the equation \eqref{E:FWTRANSPORT}  
yield that the map 
$\spherenormal_{\upomega}|_{\bf{z}} \rightarrow \spherenormal_{\upomega}|_{\Tranchar_{\bf{z}}(t)}$
is a diffeomorphism from $UT_{{\bf{z}}} \Sigma_0$ onto $UT_{\Tranchar_{\bf{z}}(t)} \Sigma_t$.
Considering also that for each for $t \in [0, \RescaledTboot]$, 
the map 
$\spherenormal_{\upomega}|_{\Tranchar_{\bf{z}}(t)} \rightarrow \Transport|_{\Tranchar_{\bf{z}}(t)} + \spherenormal_{\upomega}|_{\Tranchar_{\bf{z}}(t)}$
(where $\spherenormal_{\upomega}|_{\Tranchar_{\bf{z}}(t)}$ is the solution to \eqref{E:FWTRANSPORT})
defines a diffeomorphism from $UT_{\Tranchar_{\bf{z}}(t)} \Sigma_t$ onto $\mathscr{N}_{\Tranchar_{\bf{z}}(t)}$,
we conclude that the map $\ell_{\upomega}|_{\bf{z}} \rightarrow \Transport|_{\Tranchar_{\bf{z}}(t)} + \spherenormal_{\upomega}|_{\Tranchar_{\bf{z}}(t)}$
is the desired diffeomorphism from 
$\mathscr{N}_{\bf{z}}$ onto $\mathscr{N}_{\Tranchar_{\bf{z}}(t)}$.

Next, for $u \in [0,\RescaledTboot]$, we let $q = q(u) := \Tranchar_{\bf{z}}(u)$ 
be the unique point\footnote{It is unique since $\Transport t = 1$.} 
on the cone-tip axis with Cartesian component $q^0 = u$.
Let $\upomega \in \mathbb{S}^2$,
and let $\ell_{\upomega} := \Transport|_{\Tranchar_{\bf{z}}(u)} + \spherenormal_{\upomega}|_{\Tranchar_{\bf{z}}(u)} \in \mathscr{N}_{\Tranchar_{\bf{z}}(u)}$
denote the corresponding null vector
that we constructed in the previous paragraph.
We now let $\Upsilon_{u;\upomega} = \Upsilon_{u;\upomega}(t)$ be the null geodesic curve 
emanating from $q(u)$ with initial velocity $\ell_{\upomega}$, parameterized by $t$ (see Footnote~\ref{FN:AFFINE}), 
that is, $\Upsilon_{u;\upomega}^0(t) = t$.
Introducing the notation $\dot{\Upsilon}_{u;\upomega}^{\alpha} := \frac{d}{dt} \Upsilon_{u;\upomega}^{\alpha}$ 
and
$\ddot{\Upsilon}_{u;\upomega}^{\alpha} := \frac{d^2}{dt^2} \Upsilon_{u;\upomega}^{\alpha}$,
we note that standard arguments\footnote{\eqref{E:TIPNULLGEODESICS} is 
equivalent to equation \eqref{E:DLUNITLUNIT} for $\Dfour_{\Lunit} \Lunit^{\alpha}$,
where $\dot{\Upsilon}_{u;\upomega}^{\alpha}$
can be identified with $\Lunit^{\alpha}$,
$\dot{\Upsilon}_{u;\upomega}^{\alpha} - \Transport^{\alpha}$
can be identified with $\spherenormal^{\alpha}$,
and
$
\frac{1}{2} [\Lie_{\Transport} \gfour]_{\kappa \lambda} (\dot{\Upsilon}_{u;\upomega}^{\kappa} - \Transport^{\kappa}) 
			(\dot{\Upsilon}_{u;\upomega}^{\lambda} - \Transport^{\lambda})
			\dot{\Upsilon}_{u;\upomega}^{\alpha}$
can be identified with $- k_{\spherenormal \spherenormal} \Lunit^{\alpha}$.}  
yield that the four scalar functions $\lbrace \Upsilon_{u;\upomega}^{\alpha}(t) \rbrace_{\alpha = 0,1,2,3}$ are the solution 
to the following ODE system initial value problem
(Footnote~\ref{FN:SMOOTHFORODE} also applies here)
with data given at $t=u$, where on RHS~\eqref{E:TIPNULLGEODESICS}, 
$\Lie_{\Transport}$ denotes Lie differentiation with respect to $\Transport$:
\begin{subequations}
\begin{align} \label{E:TIPNULLGEODESICS}
	\ddot{\Upsilon}_{u;\upomega}^{\alpha}(t)
	& = -
			\Chfour_{\kappa \ \lambda}^{\ \alpha}|_{\Upsilon_{u;\upomega}(t)}
			\dot{\Upsilon}_{u;\upomega}^{\kappa}(t)
			\dot{\Upsilon}_{u;\upomega}^{\lambda}(t)
				\\
	& \ \
			+
			\frac{1}{2} [\Lie_{\Transport} \gfour]_{\kappa \lambda}|_{\Upsilon_{u;\upomega}(t)} 
			\left(\dot{\Upsilon}_{u;\upomega}^{\kappa}(t) - \Transport^{\kappa}|_{\Upsilon_{u;\upomega}(t)} \right) 
			\left(\dot{\Upsilon}_{u;\upomega}^{\lambda}(t) - \Transport^{\lambda}|_{\Upsilon_{u;\upomega}(t)} \right)
			\dot{\Upsilon}_{u;\upomega}^{\alpha}(t),
				\notag \\
	\Upsilon_{u;\upomega}^{\alpha}(u)
	& = q^{\alpha}(u) = \Tranchar_{\bf{z}}^{\alpha}(u),
	\qquad
	\dot{\Upsilon}_{u;\upomega}^{\alpha}(u)
	= \ell_{\upomega}^{\alpha}.
	\label{E:INITIALCONDITIONSTIPNULLGEODESICS}
	\end{align}
	\end{subequations}
	
	We are now able to extend the angular coordinates by declaring that 
	$\upomega$ is constant along 
	the null geodesic curve $t \rightarrow \Upsilon_{u;\upomega}(t)$.
	Next, given any fixed $t \in [u,\RescaledTboot]$,
	we define the truncated cone
	\begin{align} \label{E:TIPCONE}
		\mathcal{C}_u^t
		& := \bigcup_{\uptau \in [u,t], \upomega \in \mathbb{S}^2} \Upsilon_{u;\upomega}(\uptau).
	\end{align}
	We then define a function $u$ by the requirement that its level sets are precisely the cones \eqref{E:TIPCONE},
	that is, along $\mathcal{C}_{u'}^{\RescaledTboot}$, the function $u$ takes the value $u'$. 
	We then set
	\begin{align} \label{E:INTERIORREGION}
		\widetilde{\mathcal{M}}^{(Int)} 
		& := \bigcup_{u \in [0,\RescaledTboot]} \mathcal{C}_u^{\RescaledTboot}.
	\end{align}
	At times, we will use the alternate notation 
	\begin{align} \label{E:ABBREVIATEDTRUNCATEDCONE}
		\mathcal{C}_u := \mathcal{C}_u^{\RescaledTboot}.
	\end{align}
	As is described, for example, in \cite{dCsK1993}, this construction provides a solution of
	\eqref{E:EIKONALEQUATIONFORRESCALEDMETRIC} in the region $\widetilde{\mathcal{M}}^{(Int)}$ depicted in Fig.\,\ref{F:DOMAINS}.
	Note that by construction, we have 
	\begin{align} \label{E:UISTALONGCONETIPAXIS}
		u(\Tranchar_{\bf{z}}(t)) & = t,
		&
		\Transport [u(\Tranchar_{\bf{z}}(t))]
		& = 1.
	\end{align}
	
	In total, we have constructed geometric coordinates $(t,u,\upomega)$ in $\widetilde{\mathcal{M}}^{(Int)}$.
	More precisely, standard ODE theory yields that the map 
	$(t,u,\upomega) \rightarrow 
	\left(\Upsilon_{u;\upomega}^0(t),\Upsilon_{u;\upomega}^1(t),\Upsilon_{u;\upomega}^2(t),\Upsilon_{u;\upomega}^3(t)\right)$
	is smooth on $\lbrace (t,u,\upomega) \ | \  u \in [0,\RescaledTboot], t \in [u,\RescaledTboot], \upomega \in \mathbb{S}^2 \rbrace$
	and locally injective away from points with $t=u$ (which correspond to the cone-tip axis);
	note that here we are identifying $\Upsilon_{u;\upomega}^{\alpha}(t)$ with the Cartesian coordinate $x^{\alpha}$.
	Moreover, the continuity argument mentioned in Subsect.\,\ref{SS:GEOMETRICSPACETIMESUBSETS} guarantees that in fact, 
	this map is a global diffeomorphism from 
	$\lbrace (t,u,\upomega) \ | \  u \in [0,\RescaledTboot], t \in [u,\RescaledTboot], \upomega \in \mathbb{S}^2 \rbrace 
	\backslash 
	\lbrace (u,u,\upomega) \ | \ u \in [0,\RescaledTboot], \upomega \in \mathbb{S}^2 \rbrace$
	onto its image, i.e., onto $\widetilde{\mathcal{M}}^{(Int)}$ minus the cone-tip axis 
	$\lbrace \Tranchar_{\bf{z}}(t) \rbrace_{t \in [0,\RescaledTboot]}$;
	see also Prop.\,\ref{P:CONTROLOFNULLGEODESICS} for a quantitative proof that the null curves $t \rightarrow \Upsilon_{u;\upomega}(t)$
	corresponding to distinct values of $u$ and $\upomega$
	remain separated.\footnote{By ``separated,'' in $\widetilde{\mathcal{M}}^{(Int)}$, we mean, of course, away from the cone-tip axis.
\label{FN:SEPARATEED}} 
	
\subsubsection{The exterior solution and the region $\widetilde{\mathcal{M}}^{(Ext)}$}
\label{SSS:EIKONALEXTERIOR}
Let $\bf{z}$ be the point in $\Sigma_0$ from Subsubsect.\,\ref{SSS:EIKONALINTERIOR}, i.e., the point 
$\Tranchar_{\bf{z}}(0)$, at which $t = u = 0$.
The same arguments leading to \cite{qW2017}*{Proposition 4.3}
guarantee that for $\Tboot$ sufficiently small, 
there is a neighborhood $\mathscr{O}$ in $\Sigma_0$ contained
in the metric ball $B_{\RescaledTboot}({\bf{z}},g)$ 
(with respect to the rescaled first fundamental form $g$ of $\Sigma_0$)
of radius $\RescaledTboot$ centered at $\bf{z}$
such that $\mathscr{O}$ can be foliated with the level sets of a function
$w$ on $\Sigma_0$, defined for $0 \leq w \leq \RescaledFoliationparameter := \frac{4}{5} \RescaledTboot$,
where, away from $\bf{z}$, $w$ is smooth and has level sets $S_w$ diffeomorphic to $\mathbb{S}^2$,
while $S_0 = \lbrace \bf{z} \rbrace$. To obtain suitable control of the geometry,
we require $w$ to have a variety of crucial properties, especially \eqref{E:EQNOFINITIALFOLIATION}; 
see Prop.\,\ref{P:INITIALFOLIATION} for the existence of a function $w$ with the desired properties.

Let $\upomega \in \mathbb{S}^2$ be as in Subsubsect.\,\ref{SSS:EIKONALINTERIOR},
let $\ell_{\upomega} \in T_{\bf{z}} \mathcal{M}$ be the corresponding null vector,
and let $\spherenormal_{\upomega} = \ell_{\upomega} - \Transport|_{\bf{z}}$ 
be the corresponding element of $UT_{\bf{z}} \Sigma_0$.
Let $\nabla$ denote the Levi-Civita connection of $g$ and let 
$a := |\nabla w|_g^{-1}$ 
denote the lapse,
where $|\nabla w|_g = \sqrt{(g^{-1})^{cd} \partial_c w \partial_d w}$.
In our forthcoming analysis, we will have $a({\bf{z}}) = 1$ and $a \approx 1$;
see Prop.\,\ref{P:INITIALFOLIATION}.
Let $\spherenormal$ be the outward $g$-unit normal to $S_w$ in $\Sigma_0$,
i.e., $\spherenormal^i := a (g^{-1})^{ic} \partial_c w$,
$\spherenormal^0 = 0$,
and $g_{cd}\spherenormal^c \spherenormal^d = 1$.
Each fixed integral curve of $\spherenormal$ can be extended\footnote{In particular, in the proof of Lemma~\ref{L:LUNITIALONGCONETIPAXISISC1INANGLEVARIABLES},
we show that along $\Sigma_0$, for $i=1,2,3$, 
	$\| \frac{\partial}{\partial u} \spherenormal^i \|_{L_u^2 L_{\upomega}^{\infty}} < \infty$,
	where this norm is defined in Subsect.\,\ref{SS:GEOMETRICNORMS}; 
	this implies the extendibility of each integral curve of $\spherenormal$ to $\bf{z}$,
	where $\bf{z}$ is the point at which $u=0$.} 
to a smooth curve emanating from $\bf{z}$.
More precisely, for each vector $\spherenormal_{\upomega} \in UT_{\bf{z}} \Sigma_0$, 
there is a unique integral curve $\Phi_{\upomega} : [0,\RescaledFoliationparameter] \rightarrow \Sigma_0$ 
of $a \spherenormal$ parameterized by $w$
 (i.e., $\dot{\Phi}_{\upomega}^i(w) = [a\spherenormal^i] \circ \Phi_{\upomega}(w)$, with $a$ the lapse, where 
	$\dot{\Phi}_{\upomega}^i(w) = \frac{\partial}{\partial w}\Phi_{\upomega}^i(w)$)
that emanates from $\bf{z}$ with $\Phi_{\upomega}(0) = \bf{z}$
and $\dot{\Phi}_{\upomega}(0) = \spherenormal_{\upomega}$
(here we have used that $a({\bf{z}}) = 1$).
This yields a diffeomorphism from $\mathbb{S}^2$ 
to each $S_w$ for $0 < w \leq \RescaledFoliationparameter$,
defined such that $\upomega$ is constant along the integral curve $w \rightarrow \Phi_{\upomega}(w)$.
In particular, if $\lbrace \upomega^A \rbrace_{A=1,2}$ are local angular coordinates on $\mathbb{S}^2$, 
then for each fixed $w$ with $0 < w \leq \RescaledFoliationparameter$, 
the map $\upomega \rightarrow \Phi_{\upomega}(w)$ yields angular coordinates $\lbrace \upomega^A \rbrace_{A=1,2}$
on $S_w$. It is straightforward to see that on $\cup_{0 < w \leq \RescaledFoliationparameter} S_w$, 
we have the vectorfield identity (where $\frac{\partial}{\partial w}$ denotes partial differentiation at fixed $\upomega$)
\begin{align} \label{E:PARTIALPARTIALWALONGSIGMA0}
	\frac{\partial}{\partial w}
	& = a \spherenormal,
\end{align}
and that the rescaled first fundamental form of $\Sigma_0$, denoted by $g$,
can be expressed relative to the coordinates $(w,\upomega)$
as follows:
\begin{align} \label{E:FORMOFFIRSTFUNDAMENTALFORMONSIGMA0}
	g
	& = a^2 dw \otimes dw 
		+ 
		\gsphere\left(\frac{\partial}{\partial \upomega^A},\frac{\partial}{\partial \upomega^B}\right) d \upomega^A \otimes d \upomega^B,
\end{align}
where $\gsphere$ is the Riemannian metric induced on $S_w$ by $g$.

In view of the constructions provided above,
to each point $q \in \cup_{0 < w \leq \RescaledFoliationparameter} S_w \subset \Sigma_0$,
we can associate the geometric coordinates $(0,w,\upomega)$ (where ``$0$'' is the time coordinate).
In particular, these points $q = q(w,\upomega)$ are parameterized by the coordinates $(w,\upomega) \in [0,\RescaledFoliationparameter] \times \mathbb{S}^2$.
We then define the vector $\ell_{q(w,\upomega)} := \Transport|_{q(w,\upomega)} + \spherenormal|_{q(w,\upomega)} \in T_{q(w,\upomega)} \mathcal{M}$.
Since $\gfour|_{q(w,\upomega)}(\Transport|_{q(w,\upomega)},\Transport|_{q(w,\upomega)}) = - 1$,
$\gfour|_{q(w,\upomega)}(\Transport|_{q(w,\upomega)},\spherenormal|_{q(w,\upomega)}) = 0$,
and
$\gfour|_{q(w,\upomega)}(\spherenormal|_{q(w,\upomega)},\spherenormal|_{q(w,\upomega)}) = 1$,
it follows that
$\gfour|_{q(w,\upomega)}(\ell_{q(w,\upomega)},\ell_{q(w,\upomega)}) = 0$,
i.e., $\ell_{q(w,\upomega)}$ is null.
Next, we construct the null geodesic $\Upsilon_{q(w,\upomega)} = \Upsilon_{q(w,\upomega)}(t)$ by solving the ODE \eqref{E:TIPNULLGEODESICS}
with initial conditions
$
\Upsilon_{q(w,\upomega)}^{\alpha}(0) = q^{\alpha}(w,\upomega)
$
and
$\dot{\Upsilon}_{q(w,\upomega)}^{\alpha}(0) = \ell_{q(w,\upomega)}^{\alpha}$.
For each fixed $w \in [0,\RescaledFoliationparameter]$,
the set $\lbrace \Upsilon_{q(w,\upomega)}(t) \ | \ (t,\upomega) \in [0,\RescaledTboot] \times \mathbb{S}^2 \rbrace$
is a portion of a $\gfour$-null cone. We define the function $u$ by declaring that along this null cone portion, 
it takes on the value $-w$. 
Thus, with $\mathcal{C}_u^t$ denoting the level set portion contained in $[0,t] \times \mathbb{R}^3$,
we have $\mathcal{C}_u^t = \lbrace \Upsilon_{q(-u,\upomega)}(\uptau) \ | \ (\uptau,\upomega) \in [0,t] \times \mathbb{S}^2 \rbrace$.
As we do in the interior region, we sometimes use the alternate notation $\mathcal{C}_u := \mathcal{C}_u^{\RescaledTboot}$.
We then set
	\begin{align} \label{E:EXTERIORREGION}
		\widetilde{\mathcal{M}}^{(Ext)} 
		& := \bigcup_{u \in [-\RescaledFoliationparameter,0]} \mathcal{C}_u^{\RescaledTboot}.
	\end{align}
This procedure yields a function $u$ defined in the region $\widetilde{\mathcal{M}}^{(Ext)}$ depicted in Fig.\,\ref{F:DOMAINS}.
It is a standard result that $u$ is a solution to the eikonal equation \eqref{E:EIKONALEQUATIONFORRESCALEDMETRIC}
in $\widetilde{\mathcal{M}}^{(Ext)}$. Finally, we extend the angular coordinates to $\widetilde{\mathcal{M}}^{(Ext)}$
by declaring that $\upomega$ is constant along 
the null geodesic curve $t \rightarrow \Upsilon_{q(w,\upomega)}(t)$.
In total, we have constructed geometric coordinates $(t,u,\upomega)$ in $\widetilde{\mathcal{M}}^{(Ext)}$.

\subsubsection{Acoustical metric and first fundamental forms}
\label{SSS:METRICANDFIRSTFUNDAMENTALFORM}
We refer to Subsubsect.\,\ref{SSS:ACOUSTICALMETRICANDWAVEOPERATORS} for discussion
of the acoustical metric $\gfour$ and the first fundamental form $g$ of $\Sigma_t$.
We now define $\gsphere$ to be the first fundamental form of $S_{t,u} := \mathcal{C}_u \cap \Sigma_t$,
that is, the Riemannian metric induced on $S_{t,u}$ by $\gfour$. 
We again clarify that we are working under the conventions of Subsect.\,\ref{SS:NOMORELAMBDA}.

\subsection{\texorpdfstring{Geometric subsets of spacetime and the containment $B_R(\Tranchar_{\bf{z}}(1)) \subset \mbox{\upshape Int} \widetilde{\Sigma}_1$}{Geometric subsets of spacetime and the containment B R(Tranchar bf z (1)) subset mbox upshape Int widetilde Sigma 1}}
\label{SS:GEOMETRICSPACETIMESUBSETS}
In the rest of the paper, we denote $\widetilde{\mathcal{M}} := \widetilde{\mathcal{M}}^{(Int)} \cup \widetilde{\mathcal{M}}^{(Ext)}$.
From the constructions in Subsubsects.\,\ref{SSS:EIKONALINTERIOR}--\ref{SSS:EIKONALEXTERIOR}, 
it follows that
\begin{align}  \label{E:INTERIORUNIONEXTERIOR}
	\widetilde{\mathcal{M}}
	& = \bigcup_{u \in [-\RescaledFoliationparameter,\RescaledTboot]} \mathcal{C}_u^{\RescaledTboot},
\end{align}
where $\RescaledFoliationparameter = \frac{4}{5} \RescaledTboot$.
We also define a truncated version of $\widetilde{\mathcal{M}}^{(Int)}$, namely $\widetilde{\mathcal{M}}_1^{(Int)}$, as follows:
\begin{align} \label{E:TRUNCATEDSOLIDCONE}
	\widetilde{\mathcal{M}}_1^{(Int)}
	& := \widetilde{\mathcal{M}}^{(Int)} \cap \left([1,\RescaledTboot] \times \mathbb{R}^3 \right).
\end{align}
We also define
\begin{align} \label{E:TRUNCATEDSIGMASANDCONES}
	\widetilde{\Sigma}_t 
	& := \Sigma_t \cap \widetilde{\mathcal{M}},
	&
	\widetilde{\Sigma}_t^{(Int)}
	& := \Sigma_t \cap \widetilde{\mathcal{M}}^{(Int)},
	&
	\widetilde{\mathcal{C}}_u 
	& := \mathcal{C}_u \cap \widetilde{\mathcal{M}}.
\end{align}
See Fig.\,\ref{F:SECTIONS} for a depiction of these sets.

\begin{figure}[h!]
\centering
\includegraphics[width=6.5in,scale=.5]{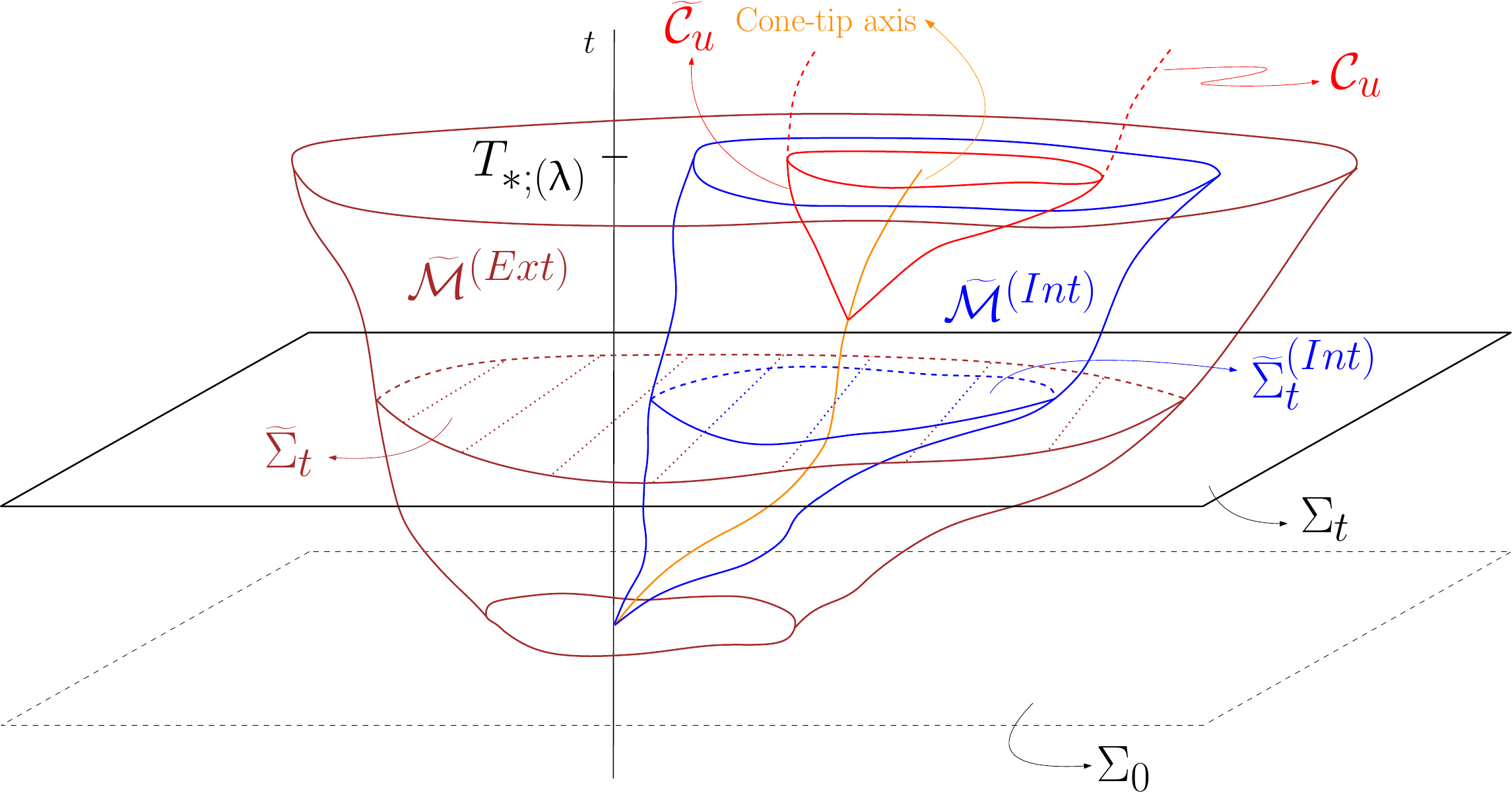}
\caption{Depiction of various subsets of spacetime in the case ${\bf{z}} := {\bf{0}}$}
\label{F:SECTIONS}
\end{figure}

For the same reasons given in \cite{qW2017}*{Section 4}, 
if $\Tboot$ is small\footnote{Note that, as is explained 
on \cite{qW2017}*{pg.~24},
it is only $\Tboot$, and not $\RescaledTboot$, that is required to be small;
once we have fixed $\Tboot$, we have $\RescaledTboot \to \infty$ as $\uplambda \to \infty$.} 
(where the required smallness is controlled by our bootstrap assumptions and our assumptions on the data), 
then the results of Subsubsects.\,\ref{SSS:EIKONALINTERIOR}--\ref{SSS:EIKONALEXTERIOR} 
yield a complete system of \emph{geometric coordinates} $(t,u,\upomega)$, which are
defined on $\widetilde{\mathcal{M}}$ and non-degenerate away from the cone-tip axis;
the proof is based on a continuity argument involving the bootstrap assumptions 
and the bounds
\eqref{E:TRCHILINFINITYESTIMATES}--\eqref{E:NULLLAPSECLOSETOUNITY} proved below; see also the proof of \cite{qW2012}*{Theorem~1.2} 
and \cites{sKiR2005b,sKiR2008} for additional details.
In particular, for $u \in [-\RescaledFoliationparameter,\RescaledTboot]$ and $t \in [[u]_+,\RescaledTboot]$
such that\footnote{Note that for $t \in [0,\RescaledTboot]$,
$S_{t,t}$ is a single point on the cone tip axis.} 
$t \neq u$
(where $[u]_+ := \max \lbrace 0,u \rbrace$),
the sets
\begin{align} \label{S:SPHERES}
	S_{t,u} 
	& := \mathcal{C}_u \cap \Sigma_t
\end{align}
are embedded submanifolds that are diffeomorphic to $\mathbb{S}^2$, equipped with the (local) angular coordinates $(\upomega^1,\upomega^2)$.
We also note that
\begin{subequations}
\begin{align} \label{E:SPACETIMEREGIONOFINTERESESTEXPRESSEASUNIONOFSPHERES}
	\widetilde{\mathcal{M}} 
	& = \bigcup_{u \in [-\RescaledFoliationparameter,\RescaledTboot],t \in [[u]_+,\RescaledTboot]} S_{t,u},
	&&
		\\
	\widetilde{\mathcal{M}}^{(Int)} 
	& = \bigcup_{u \in [0,\RescaledTboot],t \in [u,\RescaledTboot]} S_{t,u},
	&
	\widetilde{\mathcal{M}}^{(Ext)} 
	& = \bigcup_{u \in [-\RescaledFoliationparameter,0],t \in [0,\RescaledTboot]} S_{t,u},
	\label{E:INTERIORANDEXTERIOREGIONSAREUNIONSOFSPHERES}
		\\
	\widetilde{\Sigma}_t^{(Int)}
	& = \bigcup_{u \in [0,\RescaledTboot]} S_{t,u}.
	&&
	\label{E:SIGMATREGIONINTERIORISUNIONSOFSPHERES}
		\end{align}
\end{subequations}

For future use, we also note that for the same reasons given on \cite{qW2017}*{page~25},
based on \eqref{E:EUCLIDEANBALLCONTAINEDINMETRICBALL} and the estimate
\eqref{E:NULLLAPSECLOSETOUNITY} proved below,
we have the following containments,
where $B_R(\Tranchar_{\bf{z}}(1))$ denotes the Euclidean ball of radius $R$ 
centered at $\Tranchar_{\bf{z}}(1)$ in $\widetilde{\Sigma}_1$
(with $R$ is as in Subsect.\,\ref{SS:RESCALEDSOLUTION}),
and
$B_{1/2;g(1,\cdot)}(\Tranchar_{\bf{z}}(1))$ is the metric ball of radius $1/2$ centered at $\Tranchar_{\bf{z}}(1)$ 
in $\widetilde{\Sigma}_1$ corresponding
to the rescaled first fundamental form $g(1,\cdot)$:
\begin{align} \label{E:ALONGSIAMGA1EUCLIDEANBALLCONTAINEDINMETRICBALL}
	B_R(\Tranchar_{\bf{z}}(1)) 
	& \subset B_{1/2;g}(\Tranchar_{\bf{z}}(1))
	\subset
	\bigcup_{\frac{1}{3} \leq u \leq 1} S_{1,u}
	\subset
	\widetilde{\Sigma}_1^{(Int)}.
\end{align}

\subsection{Geometric quantities constructed out of the eikonal function}
\label{SS:GEOMETRICQUANTITIESCONSTRUCTEDFROMEIKONAL}
We now define a collection of geometric quantities constructed out of $u$.

\subsubsection{Geometric radial variable, null lapse, and the unit outward normal}
\label{SSS:GEOMETRICRADIALVARIABLEECT}
We define the \emph{geometric radial variable} $\rgeo$ as follows:
\begin{align} \label{E:GEOMETRICRADIAL}
	\rgeo
	= \rgeo(t,u)
	& := t - u.
\end{align}
Since in $\widetilde{\mathcal{M}}$
we have that $t \in [0,\RescaledTboot]$ and $u \in [-\RescaledFoliationparameter,t]$,
and since $\RescaledFoliationparameter := \frac{4}{5} \RescaledTboot$,
it follows from \eqref{E:RESCALEDBOOTBOUNDS} that
\begin{align} \label{E:RGEOANDUBOUNDS}
	0 & \leq \rgeo < 2 \RescaledTboot = 2 \uplambda^{1-8 \upepsilon_0} \Tboot,
	&
	-\frac{4}{5} \uplambda^{1-8 \upepsilon_0} \Tboot & \leq u \leq \uplambda^{1-8 \upepsilon_0} \Tboot.
\end{align}
Throughout the article, we will often silently use the inequalities in \eqref{E:RGEOANDUBOUNDS}.

We define the \emph{null lapse}
$\nulllapse$ to be the following scalar function,
where $|\nabla u|_g = \sqrt{(g^{-1})^{ab} \partial_a u \partial_b u}$:
\begin{align} \label{E:NULLLAPSE} 
\nulllapse
		& := \frac{1}{|\nabla u|_g}.
\end{align}
From \eqref{E:NULLLAPSE},
\eqref{E:FORMOFFIRSTFUNDAMENTALFORMONSIGMA0}, and the fact that $u = - w$ along $\Sigma_0$,
it follows that $\nulllapse = a$ along $\Sigma_0$.
Moreover, using
\eqref{E:NULLLAPSE}, 
\eqref{E:FIRSTFUNDAMENTALFORM},
\eqref{E:INVERSEACOUSTICALMETRIC}, and \eqref{E:EIKONALEQUATIONFORRESCALEDMETRIC},
we see that
\begin{align} \label{E:ALTERNATENULLLAPSEEXPRESSION}
\nulllapse = \frac{1}{\Transport u}. 
\end{align}
Considering also
\eqref{E:UISTALONGCONETIPAXIS}, we see that for $t \in [0,\RescaledTboot]$, we have
\begin{align} \label{E:NULLLAPSEISUNITYALONGCONETIPAXIS}
	\nulllapse|_{\Tranchar_{\bf{z}}(t)} = 1,
\end{align}
where the curve $t \rightarrow \Tranchar_{\bf{z}}(t)$ is the cone-tip axis introduced in Subsubsect.\,\ref{SSS:EIKONALINTERIOR}.

Let $\spherenormal$ denote the outward unit normal to $S_{t,u}$ in $\Sigma_t$, 
i.e., $\spherenormal$ is $\Sigma_t$-tangent, $g$-orthogonal to $S_{t,u}$,
outward pointing, and normalized by $g(\spherenormal,\spherenormal)=1$.
From \eqref{E:NULLLAPSE}, it follows that
\begin{align} \label{E:SPHEREOUTERNORMAL}
	\spherenormal^i
	& = - \nulllapse (g^{-1})^{ia} \partial_a u,
	&
	\spherenormal u
	& = -  \frac{1}{\nulllapse}.
\end{align}

\subsubsection{Null frame and basic geometric constructions}
\label{SSS:NULLFRAME}
We now define the following vectorfields:
\begin{align} \label{E:NULLVECTORFIELDS}
	\Lunit 
	& := \Transport + \spherenormal,
	&
	\uLunit
	& := \Transport - \spherenormal.
\end{align}
Since $\Transport^0 = 1$ and $\spherenormal^0 = 0$, it follows that
\begin{align} \label{E:NULLVECTORSAPPLIEDTOT}
	\Lunit t
	& = \uLunit t = 1.
\end{align}
Moreover, 
from 
\eqref{E:FIRSTFUNDAMENTALFORM},
\eqref{E:INVERSEACOUSTICALMETRIC},
\eqref{E:NULLLAPSE},
\eqref{E:ALTERNATENULLLAPSEEXPRESSION},
\eqref{E:SPHEREOUTERNORMAL},
and
\eqref{E:NULLVECTORFIELDS},
we see that
\begin{align} \label{E:LUNITISRESCALEDGRADIENTOFEIKONAL}
	\Lunit^{\alpha}
	& = - \nulllapse (\gfour^{-1})^{\alpha \beta} \partial_{\beta} u.
\end{align}
Since $\gfour(\Transport,\Transport) = -1$,
$\gfour(\spherenormal,\spherenormal) = 1$,
and $\gfour(\Transport,\spherenormal) = 0$,
it follows that
\begin{align} \label{E:NULLVECTORFIELDSINNERPRODUCT}
	\gfour(\Lunit,\Lunit) 
	= \gfour(\uLunit,\uLunit) 
	&= 0,
	&
	\gfour(\Lunit,\uLunit)
	& = -2.
\end{align}
In particular, 
\eqref{E:NULLVECTORSAPPLIEDTOT} and \eqref{E:NULLVECTORFIELDSINNERPRODUCT} imply that
$\Lunit$ and $\uLunit$ are future-directed and $\gfour$-null.
Let now $\lbrace e_A \rbrace_{A=1,2}$ be a (locally-defined)
$\gfour$-orthonormal frame on $S_{t,u}$, i.e.,
$\gsphere(e_A,e_B) = \updelta_{AB}$, where $\updelta_{AB}$ is the Kronecker delta.
We note that
since $\Transport$ and $\spherenormal$ are $\gfour$-orthogonal to $S_{t,u}$,
it follows from \eqref{E:NULLVECTORFIELDS} that
$\gfour(\Lunit,e_A) = \gfour(\uLunit,e_B) = 0$.
We refer to 
\begin{align} \label{E:NULLFRAME}
	\lbrace \Lunit, \uLunit, e_1, e_2 \rbrace
\end{align}
as a \emph{null frame}; see Fig.\,\ref{F:NULLFRAME}.

If $\pmb{\upxi}$ is a one-form, then 
$\pmb{\upxi}_{\Lunit} := \pmb{\upxi}_{\alpha} \Lunit^{\alpha}$,
$\pmb{\upxi}_{\uLunit} := \pmb{\upxi}_{\alpha} \uLunit^{\alpha}$,
and $\pmb{\upxi}_A := \pmb{\upxi}_{\alpha} e_A^{\alpha}$ denote contractions against the null frame elements.
Similarly, if $\mathbf{X}$ is a vectorfield, then 
$\mathbf{X}_{\Lunit} := \mathbf{X}_{\alpha} \Lunit^{\alpha}$,
$\mathbf{X}_{\uLunit} := \mathbf{X}_{\alpha} \uLunit^{\alpha}$,
and $\mathbf{X}_A := \mathbf{X}_{\alpha} e_A^{\alpha}$.
We use analogous notation to denote the components of higher order tensorfields
as well as contractions against $\spherenormal$, e.g., 
$\pmb{\upxi}_{A \spherenormal} := \pmb{\upxi}_{\alpha \beta} e_A^{\alpha} \spherenormal^{\beta}$.

It is straightforward to deduce from the above considerations that
\begin{align} \label{E:GINVERSERELATIVETONULLFRAME}
	(\gfour^{-1})^{\alpha \beta}
	& = - \frac{1}{2} \Lunit^{\alpha} \uLunit^{\beta}
		- \frac{1}{2} \uLunit^{\alpha} \Lunit^{\beta}
		+
		(\gsphere^{-1})^{\alpha \beta},
		&
		(\gsphere^{-1})^{\alpha \beta}
		& 
		= \sum_{A=1,2} e_A^{\alpha} e_A^{\beta}.
\end{align}

Next, we define the $\gfour$-orthogonal projection $\sphereproject$ onto $S_{t,u}$ 
and the $\gfour$-orthogonal projection $\Sigmatproject$ onto $\Sigma_t$
to be, respectively, the following type $\binom{1}{1}$ tensorfields,
where $\updelta_{\ \beta}^{\alpha}$ is the Kronecker delta:
\begin{align} \label{E:PROJECTS}
	\sphereproject_{\ \beta}^{\alpha}
	& := 
		\updelta_{\ \beta}^{\alpha}
		+
		\frac{1}{2} \Lunit^{\alpha} \uLunit_{\beta}
		+ 
		\frac{1}{2} \uLunit^{\alpha} \Lunit_{\beta},
	&
	\Sigmatproject_{\ \beta}^{\alpha}
	& := 
		\updelta_{\ \beta}^{\alpha}
		+
		\Transport^{\alpha} \Transport_{\beta}.
\end{align}
It is straightforward to check that
$\sphereproject_{\ \alpha}^0 = \Sigmatproject_{\ \alpha}^0 = 0$ for $\alpha = 0,1,2,3$;
we will silently use this simple fact throughout the article.

If $\pmb{\upxi}$ is a spacetime tensor, then $\sphereproject \pmb{\upxi}$ denotes its $\gfour$-orthogonal 
projection onto $S_{t,u}$, obtained by projecting every component of $\pmb{\upxi}$ onto $S_{t,u}$
via $\sphereproject$. For example, if $\mathbf{X}$ is a vectorfield, then 
$(\sphereproject \mathbf{X})^{\alpha} = \sphereproject_{\ \beta}^{\alpha} \mathbf{X}^{\beta}$,
and if $\pmb{\upxi}$ is a type $\binom{0}{2}$ tensorfield, 
then $(\sphereproject \pmb{\upxi})_{\alpha \beta} = \sphereproject_{\ \alpha}^{\gamma} \sphereproject_{\ \beta}^{\delta} \pmb{\upxi}_{\gamma \delta}$.
We say that a tensor $\pmb{\upxi}$ is $S_{t,u}$-tangent if $\sphereproject \pmb{\upxi} = \pmb{\upxi}$. 
We often denote $S_{t,u}$-tangent tensorfields in non-bold font, i.e., as $X$ or $\upxi$.
We use the notation $|\upxi|_{\gsphere}$ to denote the norm of the $S_{t,u}$-tangent tensorfield $\upxi$
with respect to the rescaled first fundamental form $\gsphere$.
For example, if $\upxi$ is a type $\binom{0}{2}$ $S_{t,u}$-tangent tensorfield, 
then $|\upxi|_{\gsphere} = \sqrt{(\gsphere^{-1})^{\alpha \gamma} (\gsphere^{-1})^{\beta \delta} 
\upxi_{\alpha \beta} \upxi_{\gamma \delta}} = \sqrt{\upxi_{AB} \upxi_{AB}}$,
where the last relation holds relative to the $S_{t,u}$-frame $\lbrace e_A \rbrace_{A=1,2}$.
If $\upxi$ is a symmetric type $\binom{0}{2}$ $S_{t,u}$-tangent tensorfield, then we define its $\gsphere$-trace to be the scalar 
$\mytr_{\gsphere} \upxi := (\gsphere^{-1})^{\alpha \beta} \upxi_{\alpha \beta} = \upxi_{AA}$,
where the last relation holds relative to the $S_{t,u}$-frame $\lbrace e_A \rbrace_{A=1,2}$.
We then define $\hat{\upxi} := \upxi - \frac{1}{2} (\mytr_{\gsphere} \upxi) \gsphere$ to be the trace-free part of $\upxi$.
Given a tensor whose components with respect to $\lbrace e_A \rbrace_{A=1,2}$ are known,
we can extend $\upxi$ to an $S_{t,u}$-tangent spacetime tensor $\upxi$ (i.e., one verifying $\sphereproject \upxi = \upxi$)
by declaring that all contractions of $\upxi$ against elements of $\lbrace \Lunit, \uLunit \rbrace$
vanish; throughout the paper, we will often implicitly assume such an extension.
Similarly, $\Sigmatproject \pmb{\upxi}$ denotes the $\gfour$-orthogonal 
projection of $\pmb{\upxi}$ onto $\Sigma_t$,
we say that
$\pmb{\upxi}$ is $\Sigma_t$-tangent if $\Sigmatproject \pmb{\upxi} = \pmb{\upxi}$,
and we can extend tensors $\upxi$ whose $\Sigma_t$ components are given 
to a $\Sigma_t$-tangent spacetime tensor by declaring that all contractions of $\upxi$ against $\Transport$ vanish. 
We also note that $\sphereproject \Lunit = \sphereproject \uLunit = 0$, and $\Sigmatproject \Transport = 0$. 

\begin{remark}
We remark that we do not attribute a tensorial structure to $\vec{\Psi}$ or $\pmb{\partial} 
\vec{\Psi}$. Therefore, 
whenever $\vec{\Psi}$ or $\pmb{\partial} \vec{\Psi}$ appears under the $|\cdot|_{\gsphere}$ norm,
it should be interpreted as the standard Euclidean norm of the array $\vec{\Psi}$
or $\pmb{\partial} 
\vec{\Psi}$. The only reason
why we occasionally have $\vec{\Psi}$ or $\pmb{\partial} 
\vec{\Psi}$ under $|\cdot|_{\gsphere}$ is because,
in our schematic notation, we sometimes group it with $S_{t,u}$-tangent tensorfields for which pointwise
norms are taken with respect to $|\cdot|_{\gsphere}$, such as, for example, in
\eqref{E:LAMBDAINVERSEQUADRATICTERMTIMEINTEGRALLTINFTTYLUINFTYLOMEGAP}.
\end{remark}

Throughout, if $\mathbf{V}$ is a spacetime vectorfield and $\pmb{\upxi}$ is a spacetime tensorfield,
then we define $\angLie_{\mathbf{V}} \pmb{\upxi} := \sphereproject \Lie_{\mathbf{V}} \pmb{\upxi}$
and
$\SigmatLie_{\mathbf{V}} \pmb{\upxi} := \Sigmatproject \Lie_{\mathbf{V}} \pmb{\upxi}$,
where $\Lie_{\mathbf{V}}$ denotes Lie differentiation with respect to $\mathbf{V}$.

We use the following notation to denote the arrays of the Cartesian components of $\Lunit$, $\uLunit$, $\spherenormal$:
\begin{align} \label{E:ARRAYSOFCARTESIANCOMPONENTSOFVECTORFIELDS}
	\vec{\Lunit}
	& := (1,\Lunit^1,\Lunit^2,\Lunit^3),
	&
	\vec{\uLunit}
	& := (1,\uLunit^1,\uLunit^2,\uLunit^3),
		&
	\vec{\spherenormal}
	& := (0,\spherenormal^1,\spherenormal^2,\spherenormal^3).
\end{align}
From \eqref{E:NULLVECTORFIELDS} and the fact that $\Transport^{\alpha}$ is a smooth function of $\vec{\Psi}$,
it follows that there exist smooth functions, denoted schematically by $\gensmoothfunction$, 
such that
both $\vec{\uLunit}$ and $\vec{\spherenormal}$ are of the form $\vec{\Lunit}-\gensmoothfunction(\vec{\Psi})$. 
In the rest of the paper, we will often use this fact without explicitly mentioning it.

\subsubsection{The metrics and volume forms relative to geometric coordinates, and the ratio $\volrat$}
\label{SSS:METRICSANDVOLUMEFORMSINGEOMETRICCOORDINATES}
From the above considerations, it is straightforward to deduce that
there exists an $S_{t,u}$-tangent vectorfield $Y$ such that
$\gfour$ and $g$ can be expressed as follows relative to the geometric coordinates 
(see \cite{jS2016b}*{Lemma~3.45} for further details):
\begin{subequations}
\begin{align} \label{E:SPACETIMEMETRICRELATIVETOGEOMETRICCOORDINATES}
	\gfour
	& = - 
			\nulllapse
			dt \otimes du
			- 
			\nulllapse
			du \otimes dt
			+
			\nulllapse^2 du \otimes du
			+
			\gsphere\left(\frac{\partial}{\partial \upomega^A},\frac{\partial}{\partial \upomega^B}\right) (d \upomega^A + Y^A du) \otimes (d \upomega^B + Y^B du),
				\\
	g
	& = \nulllapse^2 du \otimes du
			+
			\gsphere\left(\frac{\partial}{\partial \upomega^A},\frac{\partial}{\partial \upomega^B}\right) (d \upomega^A + Y^A du) \otimes (d \upomega^B + Y^B du).
			\label{E:FIRSTFUNDOFSIGMATRELATIVETOGEOMETRICCOORDINATES}
\end{align}
\end{subequations}

The volume form $d \spherevol$ induced on $S_{t,u}$ by $\gsphere$ can be expressed as follows relative to the geometric coordinates:
\begin{align}
	d \spherevolarg{t}{u}{\upomega}
	& = 
	\sqrt{\mbox{\upshape det} \gsphere} \, d \upomega^1 d \upomega^2.
\end{align}
In addition, 
the volume form $d \tvol$ induced on $\Sigma_t$ by $g$,
which in Cartesian coordinates takes the form 
$d \tvol = \Speed^{-3} dx^1 dx^2 dx^3$ (see \eqref{E:FIRSTFUNDAMENTALFORM}),
can be expressed as follows relative to the geometric coordinates:
\begin{align} \label{E:VOLUMEFORMFIRSTFUNDOFSIGMATRELATIVETOVARIOUSCOORDINATES}
	d \tvolarg{t}{u}{\upomega}
	& 
	= \nulllapse(t,u,\upomega) \, du d\spherevolarg{t}{u}{\upomega}.
\end{align}

Let $\stgsphere = \stgsphere(\upomega)$ be the standard round metric on the Euclidean unit sphere $\mathbb{S}^2$,
and let $d \flatspherevolarg{\upomega}$ denote the corresponding volume form.
The following ratio\footnote{Note that RHS~\eqref{E:RATIOOFVOLUMEFORMS} is invariant under arbitrary 
diffeomorphisms on $\mathbb{S}^2$, i.e., diffeomorphisms corresponding to  
the geometric angular coordinates. This ratio is determined by the diffeomorphism from $\mathbb{S}^2$ to $S_{t,u}$
that we constructed in Subsect.\,\ref{SS:EIKONAL} 
(which in particular determine the component functions 
$\gsphere(t,u,\upomega)\left(\frac{\partial}{\partial \upomega^A},\frac{\partial}{\partial \upomega^B}\right)$).
That is, the ratio is determined by our construction of the geometric coordinates.} 
of volume forms will play a role in the ensuing discussion:
\begin{align} \label{E:RATIOOFVOLUMEFORMS}
	\volrat(t,u,\upomega)
	& := 
	\frac{d \spherevolarg{t}{u}{\upomega}}{d \flatspherevolarg{\upomega}}
	=
	\frac{\sqrt{\mbox{\upshape det} \gsphere}(t,u,\upomega)}{\sqrt{\mbox{\upshape det} \stgsphere}(\upomega)}.
\end{align}

\subsubsection{Levi-Civita connections, angular divergence and curl operators, and curvatures}
\label{SSS:CONNECTIONSANDCURVATURE}
We let $\Dfour$ denote the Levi-Civita connection of the rescaled spacetime metric $\gfour$ and $\angD$ denote the Levi-Civita connection
of $\gsphere$. 
Our Christoffel symbol conventions for $\gfour$ are that 
$\Dfour_{\beta} \mathbf{X}^{\alpha} = \partial_{\beta} \mathbf{X}^{\alpha}
+
\Chfour_{\beta \ \gamma}^{\  \alpha} \mathbf{X}^{\gamma}
$,
where 
$
\Chfour_{\beta \ \gamma}^{\  \alpha} 
:=
(\gfour^{-1})^{\alpha \delta}
\Chfour_{\beta \delta \gamma}
$
and
$
\Chfour_{\beta \delta \gamma}
:=
\frac{1}{2}
\left\lbrace
	\partial_{\beta} \gfour_{\delta \gamma}
	+
	\partial_{\gamma} \gfour_{\beta \delta}
	-
	\partial_{\delta} \gfour_{\beta \gamma}
\right\rbrace
$.

If $\pmb{V}$ is a vectorfield and $\pmb{\upxi}$ is a spacetime tensorfield, 
then $\Dfour_{\pmb{V}} \pmb{\upxi} := \pmb{V}^{\alpha} \Dfour_{\alpha} \pmb{\upxi}$
and $\angprojDarg{\pmb{V}} \pmb{\upxi} := \sphereproject \Dfour_{\pmb{V}} \pmb{\upxi}$;
note that $\angprojDarg{V} \upxi := \angDarg{V} \upxi$
when both $V$ and $\upxi$ are $S_{t,u}$-tangent.

If $\upxi$ is an $S_{t,u}$-tangent one-form, then relative to an arbitrary $\gsphere$-orthonormal frame $\lbrace e_{(1)},e_{(2)} \rbrace$,
$\angdiv \upxi := \angDarg{A} \upxi_A$
and
$\angcurl \upxi := \upepsilon^{AB} \angDarg{A} \upxi_B$,
where repeated capital Latin indices are summed from $1$ to $2$
and
$\upepsilon^{AB}$ is fully antisymmetric and normalized by 
$\upepsilon^{12} = 1$. 
If $f$ is a scalar function defined on $S_{t,u}$, then $\angLap f := \angDsquaredarg{A}{A} f$ denotes its covariant angular Laplacian.
We clarify that above and in all of our subsequent formulas, 
frame contractions are taken after covariant differentiation.
For example, relative to arbitrary local coordinates $\lbrace y^1, y^2 \rbrace$ on $S_{t,u}$, we have
$\angDarg{A} \upxi_A := e_A^a e_A^b \angDarg{a} \upxi_{b}$ and
$\angDsquaredarg{A}{A} f := e_A^a e_A^b \angDarg{a} \angDarg{b} f$.
Similarly, if
$\upxi$ is a symmetric type $\binom{0}{2}$ $S_{t,u}$-tangent tensorfield, then 
$\angdiv \upxi_A := \angDarg{B} \upxi_{AB}$
and
$\angcurl \upxi_A := 
\upepsilon^{BC} \angDarg{B} \upxi_{CA}
=
\frac{1}{2}  
\upepsilon^{BC} 
\left\lbrace
	\angDarg{B} \upxi_{CA}
	-
	\angDarg{C} \upxi_{BA}
\right\rbrace
$.

We let $\Riemfour_{\alpha \beta \gamma \delta}$ denote the Riemann curvature of $\gfour$
and $\Ricfour_{\alpha \beta} := (\gfour^{-1})^{\gamma \delta} \Riemfour_{\alpha \gamma \beta \delta}$ denote its Ricci curvature.
We adopt the curvature sign convention
$\gfour(\Dfour_{\mathbf{X} \mathbf{Y}}^2 \mathbf{W} - \Dfour_{\mathbf{Y} \mathbf{X}}^2 \mathbf{W},\mathbf{Z})
	:= - \Riemfour(\mathbf{X},\mathbf{Y},\mathbf{W},\mathbf{Z}),
$
where $\mathbf{X}$, $\mathbf{Y}$, $\mathbf{W}$, and $\mathbf{Z}$ are arbitrary spacetime vectors,
and $\Dfour_{\mathbf{X} \mathbf{Y}}^2 \mathbf{W} := \mathbf{X}^{\alpha} \mathbf{Y}^{\beta} \Dfour_{\alpha} \Dfour_{\beta} \mathbf{W}$.

\subsubsection{Connection coefficients}
\label{SSS:CONNECTIONCOEFFICIENTS}
\begin{definition}[Connection coefficients]
\label{D:DEFSOFCONNECTIONCOEFFICIENTS}
We define the second fundamental form $k$ of $\Sigma_t$
to be the type $\binom{0}{2}$ $\Sigma_t$-tangent tensorfield
such that the following relation holds for all $\Sigma_t$-tangent 
vectorfields $X$ and $Y$:
\begin{align} 
		k(X,Y)
		& :=
		-
		\gfour(\Dfour_X \Transport, Y).
			\label{E:SECONDFUNDOFSIGMAT} 
	\end{align}

We define the second fundamental form $\spheresecondfund$ of $S_{t,u}$,
the null second fundamental form $\upchi$ of $S_{t,u}$, 
and $\underline{\upchi}$ 
to be the following type $\binom{0}{2}$ $S_{t,u}$-tangent tensorfields:
\begin{subequations}
	\begin{align}  \label{E:SECONDFUNDOFSPHERESINSIGMAT}
		\spheresecondfund_{AB}
		& := \gfour(\Dfour_A \spherenormal,e_B),
			&&
			\\
		\upchi_{AB}
		& := \gfour(\Dfour_A \Lunit, e_B),
		&
		\underline{\upchi}_{AB}
		& := \gfour(\Dfour_A \uLunit, e_B).
	\end{align}
\end{subequations}

We define the torsion $\upzeta$ and $\underline{\upzeta}$ to be the following
$S_{t,u}$-tangent one-forms:
\begin{align} \label{E:TORSION}
		\upzeta_A
		& := \frac{1}{2}
			\gfour(\Dfour_{\uLunit} \Lunit, e_A),
		&
		\underline{\upzeta}_A
		& := \frac{1}{2}
			\gfour(\Dfour_{\Lunit} \uLunit, e_A).
	\end{align}

\end{definition}

In the next lemma, we provide some standard decompositions and identities. 
We omit the simple proof and instead refer readers to \cite{sKiR2003} for details.

\begin{lemma}[Connection coefficients and relationships between various tensors]
	\label{L:CONNECTIONCOEFFICIENTS}
$k$, $\spheresecondfund$, $\upchi$, and $\underline{\upchi}$ are symmetric tensorfields.
Moreover, the following relations hold:
\begin{subequations}
\begin{align}
	k & = - \frac{1}{2} \SigmatLie_{\Transport} g
			= - \frac{1}{2} \SigmatLie_{\Transport} \gfour,
		\label{E:SECONDFORMLIEDIFFERENTIATIONDEF} \\
	\upchi & = \frac{1}{2} \angLie_{\Lunit} \gsphere
			= \frac{1}{2} \angLie_{\Lunit} \gfour,
	& 
	\underline{\upchi}
	& = \frac{1}{2} \angLie_{\uLunit} \gsphere
		= \frac{1}{2} \angLie_{\uLunit} \gfour,
		\label{E:NULLSECONDFORMSLIEDIFFERENTIATIONDEF}
\end{align}
\end{subequations}

\begin{align} \label{E:PROJECTEDNORMALDERIVATIVEOFSPHERENORMAL}
		\angprojDarg{\spherenormal} \spherenormal
		& = - \angD \ln \nulllapse,
		&
		\angprojDarg{A} \spherenormal_B
		& = \spheresecondfund_{AB},
\end{align}

\begin{subequations}
	\begin{align}
		\Dfour_A \Lunit
		& = \upchi_{AB} e_B
				-
				k_{A \spherenormal}
				\Lunit,
		&
		\Dfour_A \uLunit
		& = \underline{\upchi}_{AB} e_B
				+
				k_{A \spherenormal}
				\uLunit,
					\label{E:DALUNIT} \\
		\Dfour_{\Lunit} \Lunit
		& = 
		- k_{\spherenormal \spherenormal} 
			\Lunit,
		&
		\Dfour_{\Lunit} \uLunit
		& = 2 \underline{\upzeta}_A e_A 
				+
					k_{\spherenormal \spherenormal} 
					\uLunit,
		 \label{E:DLUNITLUNIT} \\
		\Dfour_{\uLunit} \Lunit
		& = 2 \upzeta_A e_A 
				+
				k_{\spherenormal \spherenormal} 
				\Lunit,
		&
		\Dfour_{\Lunit} e_A
		& = \angprojDarg{\Lunit} e_A
				+
				\underline{\upzeta}_A \Lunit,
					\label{E:DULUNITLUNIT} \\
		\Dfour_B e_A
		& = \angDarg{B} e_A
				+
				\frac{1}{2} \upchi_{AB} \uLunit
				+
				\frac{1}{2} \underline{\upchi}_{AB} \Lunit,
		&
		\Dfour_{\uLunit} \uLunit 
		& = - 2 (\angD_A \ln \nulllapse) e_A
				-
				k_{\spherenormal \spherenormal} 
				\uLunit,
	\end{align}
\end{subequations}

\begin{align} \label{E:CONNECTIONCOEFFICIENT}
	\upchi_{AB} 
	& =  \spheresecondfund_{AB}
			-
			k_{AB},
	&
	\underline{\upchi}_{AB} 
	& =  - 
			\spheresecondfund_{AB}
			-
			k_{AB},
	&
	\underline{\upzeta}_A
	& = - k_{A \spherenormal},
	&
	\upzeta_A
	& = 
		\angD_A \ln \nulllapse
		+
		k_{A \spherenormal}.
\end{align}

\end{lemma}

\subsection{Modified acoustical quantities}
\label{SS:MODIFIEDQUANTITIES}
As we explained at the end of Subsubsect.\,\ref{SSS:MODELWAVESTRICHARTZ},
to obtain suitable control of the acoustic geometry, we must work with modified quantities
and a metric equal to a conformal rescaling of $\gfour$. 
In this subsection, we define the relevant quantities.

\subsubsection{The conformal metric in $\widetilde{\mathcal{M}}^{(Int)}$}
\label{SSS:CONFORMALMETRIC}

\begin{definition}[The conformal factor and conformal metric in the interior region $\widetilde{\mathcal{M}}^{(Int)}$]
	\label{D:CONFORMALSTUFF}
	We define $\upsigma$ to be the solution to the following transport initial value problem
	(with data given on the cone-tip axis defined in Subsubsect.\,\ref{SSS:EIKONALINTERIOR}):
	\begin{subequations}
	\begin{align} \label{E:SIGMAEVOLUTION}
		\Lunit \upsigma(t,u,\upomega)
		& = \frac{1}{2} [\Chfour_{\Lunit}](t,u,\upomega),
		&&
		u \in [0,\RescaledTboot],
		\,
		t \in [u,\RescaledTboot],
			\,
		\upomega \in \mathbb{S}^2,
			\\
		\upsigma(u,u,\upomega)
		& = 0,
		&&
		u \in [0,\RescaledTboot],
			\,
		\upomega \in \mathbb{S}^2,
		\label{E:SIGMADATA}
	\end{align}
	\end{subequations}
	where $\Chfour_{\Lunit} := \Chfour_{\alpha} \Lunit^{\alpha}$
	and $\Chfour_{\alpha} := (\gfour^{-1})^{\kappa \lambda} \Chfour_{\kappa \alpha \lambda}$
	is a contracted (and lowered) Cartesian Christoffel symbol of $\gfour$.
	
	We define 
	\begin{align} \label{E:CONFORMALMETRIC}
	\widetilde{\gfour} & := e^{2\upsigma} \gfour,
	&
	\congsphere 
		&:= e^{2\upsigma} \gsphere
	\end{align}
	to be, respectively, the conformal spacetime metric 
	and the Riemannian metric that it induces on $S_{t,u}$.
\end{definition}

\begin{definition}[Null second fundamental forms of the conformal metric]
	\label{D:CONFORMALNULLSECONDFUNDAMENTAL}
	We define the null second fundamental forms of the conformal metric to be the following
	symmetric $S_{t,u}$-tangent tensorfields:
	\begin{align} \label{E:CONFORMALNULLSECONDFUND}
		\widetilde{\upchi}
		& := \frac{1}{2} \angLie_{\Lunit} \congsphere,
		&
		\underline{\widetilde{\upchi}}
		& := \frac{1}{2} \angLie_{\uLunit} \congsphere.
	\end{align}
\end{definition}

From straightforward computations, 
taking into consideration 
definition 
\eqref{E:CONFORMALNULLSECONDFUND}
and the PDE \eqref{E:SIGMAEVOLUTION},
we deduce the following relations:
\begin{subequations}
	\begin{align} \label{E:RELATIONBETWEENNULLSECONDFUNDANDCONFORMALNULLSECONDFUND}
		\widetilde{\upchi}
		& = e^{2 \upsigma}
				\left\lbrace
					\upchi + (\Lunit \upsigma) \gsphere)
				\right\rbrace,
		&
		\widetilde{\underline{\upchi}}
		& = e^{2 \upsigma}
				\left\lbrace
					\underline{\upchi} + (\uLunit \upsigma) \gsphere)
				\right\rbrace,
					\\
		\mytr_{\congsphere} \widetilde{\upchi}
		& = \mytr_{\gsphere} \upchi
				+
				2 \Lunit \upsigma
				=
				\mytr_{\gsphere} \upchi
				+ 
				\Chfour_{\Lunit},
		&
		\mytr_{\congsphere} \widetilde{\underline{\upchi}}
		& = \mytr_{\gsphere} \underline{\upchi}
				+
				2 \uLunit \upsigma,
				\label{E:TRACEDRELATIONBETWEENNULLSECONDFUNDANDCONFORMALNULLSECONDFUND}
					\\
	\upchi
	& = 
			\frac{1}{2}
			\left\lbrace 
				\mytr_{\congsphere} \widetilde{\upchi} 
				-
				\Chfour_{\Lunit}
			\right\rbrace
			\gsphere
			+
			\hat{\upchi},
	&
	\underline{\upchi}
	& = 
			\frac{1}{2}
			\left\lbrace 
				\mytr_{\congsphere} \widetilde{\underline{\upchi}}
				-
				2 \uLunit \upsigma
			\right\rbrace
			\gsphere
			+
			\hat{\underline{\upchi}}.
			\label{E:DETAILEDDECOMPOFNULLSECONDFUNDFORMSINTOTRACEANDTRACEFREE}
	\end{align}
\end{subequations}
Moreover, above and throughout, if $\upxi$ is a symmetric type $\binom{0}{2}$ $S_{t,u}$-tangent tensorfield, then
$\mytr_{\congsphere} \upxi := (\congsphere^{-1})^{\alpha \beta} \upxi_{\alpha \beta} 
= e^{-2\upsigma} (\gsphere^{-1})^{\alpha \beta} \upxi_{\alpha \beta}
= e^{-2\upsigma} \mytr_{\gsphere} \upxi$
denotes its trace with respect to $\congsphere$.

\subsubsection{Average values on $S_{t,u}$}
\label{SSS:AVGVALSONSPHERE}
Some of our forthcoming constructions refer to the average values of scalar functions $f$ on $S_{t,u}$.
Specifically, we define the average value of $f$, denoted by $\overline{f}$, 
as follows:
\begin{align} \label{E:AVERAGEVALUEOFSCALARFUNCTION}
	\overline{f}
	=
	\overline{f}(t,u)
	& := 
	\frac{1}{|S_{t,u}|_{\gsphere}}
	\int_{S_{t,u}}
		f
	\, d \spherevol,
	&
	|S_{t,u}|_{\gsphere}
	&
	:= \int_{S_{t,u}}
		1
	\, d \spherevol.
\end{align}

In the next lemma, we connect the evolution equation for $\overline{f}$ along integral curves of $\Lunit$ to that of $f$.
We omit the standard proof, which is based on the identity \eqref{E:EVOLUTIONVOLUMELEMENT} below.

\begin{lemma}[Evolution equation for the average value on $S_{t,u}$]
\label{L:EVOLUTIONEQUATIONFORAVERAGEVALUEONSTU}
For scalar functions $f$, we have
\begin{align} \label{E:EVOLUTIONEQUATIONFORAVERAGEVALUEONSTU}
	\Lunit \overline{f}
	+
	\mytr_{\gsphere} \upchi \overline{f}
	& =
		\left\lbrace
			\mytr_{\gsphere} \upchi
			-
			\overline{\mytr_{\gsphere} \upchi}
		\right\rbrace
		\overline{f}
		+
		\overline{
		\Lunit f
		+
		\mytr_{\gsphere} \upchi f}.
\end{align}	
	
\end{lemma}

\subsubsection{Definitions of the modified acoustical quantities}
\label{SSS:MODACOUSTICAL}

\begin{definition}[Modified acoustical quantities]
\label{D:MODIFIEDQUANTITIES}
In the interior region $\widetilde{\mathcal{M}}^{(Int)}$,
we define $\mytr_{\congsphere} \widetilde{\upchi}^{(Small)}$ to be\footnote{In \cite{qW2017}, $\mytr_{\congsphere} \widetilde{\upchi}^{(Small)}$ was denoted by ``$z$'' 
and $\mytr_{\congsphere} \widetilde{\upchi}$ was denoted by ``$\mytr \widetilde{\upchi}$.''} 
$- \frac{2}{\rgeo}$
plus the trace of
the $S_{t,u}$-tangent tensorfield $\widetilde{\upchi}$ defined in \eqref{E:CONFORMALNULLSECONDFUND}
with respect to the conformal metric $\congsphere$ defined in \eqref{E:CONFORMALMETRIC}.
That is, in view of \eqref{E:TRACEDRELATIONBETWEENNULLSECONDFUNDANDCONFORMALNULLSECONDFUND},
in $\widetilde{\mathcal{M}}^{(Int)}$, we have:
\begin{align} \label{E:MODTRICHISMALL}
		\mytr_{\congsphere} \widetilde{\upchi}^{(Small)}
		 & 	=
				\mytr_{\gsphere} \upchi
				+ 
				\Chfour_{\Lunit}
				-
				\frac{2}{\rgeo}
				=
				\mytr_{\congsphere} \widetilde{\upchi}
				-
				\frac{2}{\rgeo},
	\end{align}
	where $\Chfour_{\Lunit} := \Chfour_{\alpha} \Lunit^{\alpha}$,
	and $\Chfour_{\alpha} := (\gfour^{-1})^{\kappa \lambda} \Chfour_{\kappa \alpha \lambda}$
	is a contracted (and lowered) Cartesian Christoffel symbol of $\gfour$.
	We then extend the definition of $\mytr_{\congsphere} \widetilde{\upchi}^{(Small)}$ 
	to all of $\widetilde{\mathcal{M}}$ by declaring that
	the first equality in \eqref{E:MODTRICHISMALL} holds in all of $\widetilde{\mathcal{M}}$.

	In $\widetilde{\mathcal{M}}$, we define the mass aspect function
	$\upmu$ to be the following scalar function:
\begin{align} \label{E:MASSASPECT}
	\upmu
	& := 
			\uLunit \mytr_{\gsphere} \upchi
			+
			\frac{1}{2} \mytr_{\gsphere} \upchi \mytr_{\gsphere} \underline{\upchi}.
\end{align}

In $\widetilde{\mathcal{M}}^{(Int)}$, 
we define the modified mass aspect function\footnote{The 
	idea of working with quantities in the spirit of the mass aspect function and the modified mass aspect function
	originates in \cite{dCsK1993}. As in \cites{dCsK1993,qW2017}, 
	we use these quantities to avoid the loss of a derivative 
	when controlling the $\uLunit$ derivative of $\mytr_{\gsphere} \upchi$.}  
$\check{\upmu}$ to be the following scalar function:
\begin{align} \label{E:MODMASSASPECT}
	\check{\upmu}
	& := 2 \angLap \upsigma
			+
			\uLunit \mytr_{\gsphere} \upchi
			+
			\frac{1}{2} 
			\mytr_{\gsphere} \upchi 
			\mytr_{\gsphere} \underline{\upchi}
			-
			\mytr_{\gsphere} \upchi 
			k_{\spherenormal \spherenormal} 
			+
			\frac{1}{2} \mytr_{\gsphere} \upchi \Chfour_{\uLunit},
\end{align}
where $\Chfour_{\uLunit} := \Chfour_{\alpha} \uLunit^{\alpha}$.

In $\widetilde{\mathcal{M}}^{(Int)}$,
we define $\angupmu$ to be\footnote{Existence and uniqueness for the system \eqref{E:FURTHERMODOFMASSASPECT} is standard, given the smoothness of the source terms.} 
the $S_{t,u}$-tangent one-form that satisfies the following Hodge system on $S_{t,u}$:
\begin{align} \label{E:FURTHERMODOFMASSASPECT}
	\angdiv \angupmu
	& = \frac{1}{2}(\check{\upmu} - \overline{\check{\upmu}}),
	&
	\angcurl \angupmu
	& = 0.
\end{align}

In $\widetilde{\mathcal{M}}^{(Int)}$, 
we define the modified torsion $\widetilde{\upzeta}$ to be the following $S_{t,u}$-tangent one-form:
\begin{align} \label{E:MODTORSION}
	\widetilde{\upzeta}
	& := \upzeta 
			+
			\angD \upsigma.
\end{align}

\end{definition}

\subsection{PDEs verified by geometric quantities - a preliminary version}
\label{SS:NULLLAPSEANDCONNECTIONCOEFFICIENTPDES}
To control the acoustic geometry, we will derive estimates for the PDEs that various geometric quantities solve.
In the next lemma, we provide a first version of these PDEs. The results
are standard and are independent of the compressible Euler equations.
In Prop.\,\ref{P:PDESMODIFIEDACOUSTICALQUANTITIES}, we  
use the compressible Euler equations 
to re-express various terms in the PDEs,
which will lead to the form of the equations that
we use in our analysis.

\begin{lemma} \cite{sKiR2003}*{PDEs verified by the $S_{t,u}$ volume element ratio, null lapse, and connection coefficients, without
regard for the compressible Euler equations}
	\label{L:NULLLAPSEANDCONNECTIONCOEFFICIENTPDES}
		The following evolution equations hold\footnote{In \cite{qW2017}*{Equation~(5.28)}, the terms in braces on the last line of
		RHS~\eqref{E:ULUNITTRACECHI} were omitted. However, 
		equation \eqref{E:ULUNITTRACECHI} is needed only to derive the evolution equation 
		\eqref{E:MODIFIEDMASSASPECTEVOLUTIONEQUATION} for $\check{\upmu}$,
		and the omitted terms have the same schematic structure as
		other error terms that were bounded in \cite{qW2017}; i.e., the omitted terms are harmless.
		Moreover, in \cite{qW2017},
		the second term on LHS~\eqref{E:LDERIVATIVECHIHAT}
		was listed as
		$\frac{1}{2} (\mytr_{\gsphere} \upchi) \hat{\upchi}_{AB}$.
		Fortunately, correcting the coefficient from $\frac{1}{2}$ to $1$ does not lead to any changes in the estimates,
		as we further explain in the discussion surrounding equation \eqref{E:CHIHATTRANSPORTINEQUALITY}.
		In our statement of Lemma~\ref{L:NULLLAPSEANDCONNECTIONCOEFFICIENTPDES}, 
		we also corrected index-placement/sign errors in some curvature terms, specifically the term
	$\frac{1}{2} \Riemfour_{A \Lunit \Lunit \uLunit}$ on RHS~\eqref{E:LDERIVATIVETORSION},
	the term $\Riemfour_{A \uLunit \Lunit B}$
	on RHS~\eqref{E:ULUNITTRACECHI},
	and the term 
	$
	\frac{1}{2}
			\upepsilon^{AB}
			\Riemfour_{A \Lunit \uLunit B}
	$
	on RHS~\eqref{E:CURLTORSIONNOEULER}.
	These corrections are harmless in the sense that in practice, when deriving
	estimates, we only need to know the schematic structure of the first and third of these curvature terms,
	which is provided by \eqref{E:RICCIANDRIEMANNDECOMPSINVOLVINGVORTANDENT} and
	\eqref{E:RIEMALUNDERLINELBANTISYMMETRICCONTRACTEDDECOMPINVOLVINGVORTANDENT}
	and which is insensitive to signs. In particular, these corrections
	do not affect the schematic form of the equations of Prop.\,\ref{P:PDESMODIFIEDACOUSTICALQUANTITIES},
	which is what we use when deriving estimates for the acoustic geometry.
		\label{FN:CORRECTIONOFTYPOS}} 
		relative to a null frame:
	\begin{subequations}
	\begin{align}
		\Lunit \volrat
		 & = \volrat \mytr_{\gsphere} \upchi, 
			\label{E:EVOLUTIONVOLUMELEMENT} \\
		\Lunit \nulllapse
		& = 	
			- 
			\nulllapse
			k_{\spherenormal \spherenormal},
			\label{E:EVOLUTIONNULLAPSE} \\
		\Lunit \mytr_{\gsphere} \upchi
		+ 
		\frac{1}{2} (\mytr_{\gsphere} \upchi)^2
		& = - |\hat{\upchi}|_{\gsphere}^2
					- 
					k_{\spherenormal \spherenormal} 
					\mytr_{\gsphere} \upchi
					-
					\Ricfour_{\Lunit \Lunit},
						\label{E:RAYCHAUDHURI}  \\
	\angprojDarg{\Lunit} 
	\hat{\upchi}_{AB}
	+ 
	(\mytr_{\gsphere} \upchi)
	\hat{\upchi}_{AB}
	& = 
	- k_{\spherenormal \spherenormal} 
		\hat{\upchi}_{AB}
	-
	\left\lbrace
		\Riemfour_{\Lunit A \Lunit B}
		-
		\frac{1}{2}
		\Ricfour_{\Lunit \Lunit}
		\updelta_{AB}
	\right\rbrace,
		\label{E:LDERIVATIVECHIHAT} \\
	\angprojDarg{\Lunit}
	\upzeta_A
	+
	\frac{1}{2} 
	(\mytr_{\gsphere} \upchi)
	\upzeta_A
	& = - \left\lbrace
					k_{B \spherenormal} 
					+ 
					\upzeta_B 
				\right\rbrace \hat{\upchi}_{AB}
		-
		\frac{1}{2}
		\mytr_{\gsphere} \upchi
		k_{A \spherenormal}
		+
		\frac{1}{2} \Riemfour_{A \Lunit \Lunit \uLunit},
			\label{E:LDERIVATIVETORSION} \\
	\Lunit \mytr_{\gsphere} \underline{\upchi}
	+
	\frac{1}{2}
	(\mytr_{\gsphere} \upchi)
	\mytr_{\gsphere} \underline{\upchi}
	& = 2 \angdiv \underline{\upzeta}
		+
		k_{\spherenormal \spherenormal} 
		\mytr_{\gsphere} \underline{\upchi}
		-
		\hat{\upchi}_{AB} \hat{\underline{\upchi}}_{AB}
		+
		2
		|\underline{\upzeta}|_{\gsphere}^2
		+
		\Riemfour_{A \uLunit \Lunit A},
				\label{E:LUNITTRACEUCHI} \\
	\angprojDarg{\uLunit} 
	\hat{\upchi}_{AB}
	+ 
	\frac{1}{2} 
	(\mytr_{\gsphere} \underline{\upchi})
	\hat{\upchi}_{AB}
	& =
		-
	\frac{1}{2} 
	(\mytr_{\gsphere} \upchi)
	\underline{\hat{\upchi}}_{AB}	
	+
	2 \angDarg{A} \upzeta_B
	-
	\angdiv \upzeta
	\updelta_{AB}
	+
	k_{\spherenormal \spherenormal}  \hat{\upchi}_{AB}
	+
	\left\lbrace
	2 \upzeta_A \upzeta_B
	-
	|\upzeta|_{\gsphere}^2 \updelta_{AB}
	\right\rbrace
		\label{E:ULUNITTRACECHI} \\
	&   
	\ \
	-
	\left\lbrace
		\underline{\hat{\upchi}}_{AC} \hat{\upchi}_{CB}
		-
		\frac{1}{2}
		\underline{\hat{\upchi}}_{CD} \hat{\upchi}_{CD}
		\updelta_{AB}
	\right\rbrace
	+
	\Riemfour_{A \uLunit \Lunit B}
	-
	\frac{1}{2}
	\Riemfour_{C \Lunit \uLunit C}
	\updelta_{AB},
	\notag
	\end{align}
	\end{subequations}
	
	\begin{subequations}
	\begin{align}
		\angdiv \hat{\upchi}_A
		+
		\hat{\upchi}_{AB} 
		k_{B \spherenormal}
		& = 
		\frac{1}{2}
		\left\lbrace
			\angDarg{A} \mytr_{\gsphere} \upchi
			+
				k_{A \spherenormal}
				\mytr_{\gsphere} \upchi
		\right\rbrace
		+
		\Riemfour_{B \Lunit B A},
			\label{E:ANGDIVTRACEFREEPARTOFCHI} \\
		\angdiv \upzeta
		& = \frac{1}{2}
				\left\lbrace
					\upmu 
					- 
					k_{\spherenormal \spherenormal} \mytr_{\gsphere} \upchi
					- 
					2 |\upzeta|_{\gsphere}^2
					-
					|\hat{\upchi}|_{\gsphere}^2
					-
					2
					k_{AB}
					\hat{\upchi}_{AB}
				\right\rbrace
				-
				\frac{1}{2}
				\Riemfour_{A \uLunit \Lunit A},
				\label{E:DIVTORSIONNOEULER} \\
	\angcurl \upzeta
	& = 	\frac{1}{2}
				\upepsilon^{AB}
				\hat{\underline{\upchi}}_{AC}
				\hat{\upchi}_{BC}
			+
			\frac{1}{2}
			\upepsilon^{AB}
			\Riemfour_{A \Lunit \uLunit B}.
			\label{E:CURLTORSIONNOEULER}
	\end{align}
	\end{subequations}
	
\end{lemma}

\subsection{Main version of the PDEs verified by the acoustical quantities, including the modified ones}
\label{SS:MAINPDESFORACOUSTICALQUANT}
The main result of this subsection is Prop.\,\ref{P:PDESMODIFIEDACOUSTICALQUANTITIES},
in which we derive, with the help of the compressible Euler equations,
the main PDEs that we use to control the acoustic geometry.
The proposition in particular shows how the source terms in the compressible Euler equations
influence the evolution of the acoustic geometry.
Before proving the proposition, we first introduce some additional schematic notation
and, in Lemma~\ref{L:CURVATUREDECOMPOSITIONS}, provide some decompositions of various null components of
the acoustical curvature, that is, the curvature of $\gfour$.

\begin{remark}
Compared to previous works, what is new are the terms in Lemma~\ref{L:CURVATUREDECOMPOSITIONS} 
and Prop.\,\ref{P:PDESMODIFIEDACOUSTICALQUANTITIES}
that are multiplied by $\uplambda^{-1}$;
these terms capture, in particular, how the top-order derivatives of the vorticity and entropy
affect the acoustical curvature.
\end{remark}

\subsubsection{Additional schematic notation and a simple lemma}
\label{SSS:ADDITIONALSCHEMATIC}
Let $U$ and $\upxi$ be scalar functions or $S_{t,u}$-tangent tensorfields.
In the rest of the paper, we will use the schematic notation
\begin{align} \label{E:SCHEMATICNOTATIONCURVATURESECTION}
	U
	=
	\gensmoothfunction_{(\vec{\Lunit})}
	\cdot
	\upxi
\end{align}
to mean that the \emph{Cartesian} components of $U$ can be expressed
as linear combinations of products of the Cartesian components of $\upxi$
and scalar functions of type ``$\gensmoothfunction_{(\vec{\Lunit})}$,'' which by definition are 
linear combinations of
products of
\textbf{i)} smooth functions of $\vec{\Psi}$ 
and \textbf{ii)} the Cartesian components of vectorfields whose Cartesian components
are polynomials in the components of $\vec{\Lunit}$ with coefficients that are
smooth functions of $\vec{\Psi}$.
Expressions such as $\gensmoothfunction_{(\vec{\Lunit})} \cdot \upxi_{(1)} \cdot \upxi_{(2)}$ 
have the obvious analogous meaning. If $\upxi = (\upxi_{(1)},\cdots,\upxi_{(m)})$ is an array 
of scalar functions $S_{t,u}$-tangent tensorfields,
then $\gensmoothfunction_{(\vec{\Lunit})} \cdot \upxi$ means sums of terms of type 
$\gensmoothfunction_{(\vec{\Lunit})} \cdot \upxi_{(i)}$, $1 \leq i \leq m$.
As examples, we note (in view of the discussion below \eqref{E:ARRAYSOFCARTESIANCOMPONENTSOFVECTORFIELDS}) 
that $\spherenormal_a \VortVort^a = \gensmoothfunction_{(\vec{\Lunit})} \cdot \vec{\VortVort}$,
while
$\gensmoothfunction_{(\vec{\Lunit})} \cdot \pmb{\partial} \vec{\Psi} \cdot \pmb{\partial} \vec{\Psi}$
denotes a scalar function or an $S_{t,u}$-tangent tensorfield
whose Cartesian components are products 
of $\gensmoothfunction_{(\vec{\Lunit})}$
and a term that is quadratic in elements the array
$\pmb{\partial} \vec{\Psi}$.
As another example,
we note that \eqref{E:PROJECTS} and the discussion below \eqref{E:ARRAYSOFCARTESIANCOMPONENTSOFVECTORFIELDS} imply that
the $S_{t,u}$-tangent tensorfield 
$\sphereproject$ has Cartesian components of the form $\gensmoothfunction_{(\vec{\Lunit})}$, which we indicate by writing
$\sphereproject = \gensmoothfunction_{(\vec{\Lunit})}$.
Finally, we note that since (by \eqref{E:SECONDFORMLIEDIFFERENTIATIONDEF}) the Cartesian components
of the second fundamental form $k$ of $\Sigma_t$
verify $k_{ij} = \gensmoothfunction(\vec{\Psi}) \cdot \pmb{\partial} \vec{\Psi}$,
it follows that the $S_{t,u}$-tangent tensorfield
with components $k_{A \spherenormal} := k(e_A,\spherenormal)$, $(A=1,2)$,
is of the form
$k_{A \spherenormal} = \gensmoothfunction_{(\vec{\Lunit})} \cdot \pmb{\partial} \vec{\Psi}$.

We will use the following simple lemma in our proof of Prop.\,\ref{P:PDESMODIFIEDACOUSTICALQUANTITIES}.

\begin{lemma}[Identities for the derivatives of some scalar functions] 
	\label{L:ANGULARDERIVATIVESOFSOMESCALARFUNCTIONS}
	With $d f$ denoting the spacetime gradient of the scalar function $f$
	(and thus $\sphereproject \cdot d f = \angD f$), we have the following identities
	(where in \eqref{E:F1TYPETERMSANGULARDERIVATIVE} and \eqref{E:F1TYPEGENERICDERIVATIVEOFVECTORFIELDCOMPONENTS}, 
	the terms ``$\gensmoothfunction_{(\vec{\Lunit})}$''
	on the LHSs are not the same as the terms ``$\gensmoothfunction_{(\vec{\Lunit})}$'' on the RHSs):
	\begin{subequations}
	\begin{align} \label{E:ANGULARDERIVATIVEOFVECTORFIELDCOMPONENTS}
		\sphereproject \cdot d (\vec{\Lunit},\vec{\uLunit},\vec{\spherenormal})
		& =
		\gensmoothfunction_{(\vec{\Lunit})}
		\cdot 
		(\pmb{\partial} \vec{\Psi},\mytr_{\congsphere} \widetilde{\upchi}^{(Small)},\hat{\upchi},\rgeo^{-1}),
			\\
	\sphereproject \cdot d \gensmoothfunction_{(\vec{\Lunit})}
	& =
		\gensmoothfunction_{(\vec{\Lunit})}
		\cdot 
		(\pmb{\partial} \vec{\Psi},\mytr_{\congsphere} \widetilde{\upchi}^{(Small)},\hat{\upchi},\rgeo^{-1}).
		\label{E:F1TYPETERMSANGULARDERIVATIVE}
\end{align}
\end{subequations}	
	
	Moreover,
	\begin{subequations}
	\begin{align} \label{E:GENERICDERIVATIVEOFVECTORFIELDCOMPONENTS}
		d (\vec{\Lunit},\vec{\uLunit},\vec{\spherenormal})
		& =
		\gensmoothfunction_{(\vec{\Lunit})}
		\cdot 
		(\pmb{\partial} \vec{\Psi},\mytr_{\congsphere} \widetilde{\upchi}^{(Small)},\hat{\upchi},\upzeta,\rgeo^{-1}),
			\\
		d \gensmoothfunction_{(\vec{\Lunit})}
		& = 
		\gensmoothfunction_{(\vec{\Lunit})}
		\cdot 
		(\pmb{\partial} \vec{\Psi},\mytr_{\congsphere} \widetilde{\upchi}^{(Small)},\hat{\upchi},\upzeta,\rgeo^{-1}).
		\label{E:F1TYPEGENERICDERIVATIVEOFVECTORFIELDCOMPONENTS} 
\end{align}
\end{subequations}

\end{lemma}

\begin{proof}
	To prove \eqref{E:ANGULARDERIVATIVEOFVECTORFIELDCOMPONENTS},
	we first note the schematic relation
	$\Dfour \Lunit^{\alpha} = d \Lunit^{\alpha} + \Chfour \cdot \Lunit = d \Lunit^{\alpha} + \gensmoothfunction_{(\vec{\Lunit})} \cdot \pmb{\partial} \vec{\Psi}$, where $\Chfour$ denotes a Cartesian Christoffel symbol of the rescaled metric $\gfour$.
	Viewing $\Lunit^{\alpha}$ as a scalar function, we can interpret this relation 
	as an identity in which the term on the left and 
	the two terms on the right are one-forms. Projecting these one-forms onto $S_{t,u}$ with the tensorfield $\sphereproject$,
	and using the first identity in \eqref{E:DALUNIT},
	the fact that $k_{ij} = \gensmoothfunction(\vec{\Psi}) \cdot \pmb{\partial} \vec{\Psi}$,
	and the fact that 
	$\vec{\uLunit} = \gensmoothfunction_{(\vec{\Lunit})}$
	and
	$\vec{\spherenormal} = \gensmoothfunction_{(\vec{\Lunit})}$,
	we deduce that for $\alpha = 0,1,2,3$, we have the following schematic identity for the scalar function $\Lunit^{\alpha}$:
	$\sphereproject d \Lunit^{\alpha} 
	=
	\gensmoothfunction_{(\vec{\Lunit})} \cdot (\pmb{\partial} \vec{\Psi},\upchi)
	$.
	Considering also that 
	$\upchi 
	=
	\gensmoothfunction_{(\vec{\Lunit})}
	\cdot 
	(\pmb{\partial} \vec{\Psi},\mytr_{\congsphere} \widetilde{\upchi}^{(Small)},\hat{\upchi},\rgeo^{-1})$
	(as can be seen by decomposing $\upchi  = \hat{\upchi} + \frac{1}{2} (\mytr_{\gsphere} \upchi)\gsphere$ 
	and using \eqref{E:MODTRICHISMALL}),
	we conclude \eqref{E:ANGULARDERIVATIVEOFVECTORFIELDCOMPONENTS} for $\vec{\Lunit}$.
	In addition, taking into account that $\vec{\uLunit} = \gensmoothfunction_{(\vec{\Lunit})}$
	and
	$\vec{\spherenormal} = \gensmoothfunction_{(\vec{\Lunit})}$,
	and using the chain and product rules,
we also deduce the identity \eqref{E:ANGULARDERIVATIVEOFVECTORFIELDCOMPONENTS}
for $\vec{\uLunit}$ and $\vec{\spherenormal}$.
\eqref{E:F1TYPETERMSANGULARDERIVATIVE} follows from similar arguments, and we omit the details.

The identities \eqref{E:GENERICDERIVATIVEOFVECTORFIELDCOMPONENTS}--\eqref{E:F1TYPEGENERICDERIVATIVEOFVECTORFIELDCOMPONENTS} 
from from a similar argument, but we also take into account
\eqref{E:DLUNITLUNIT}
and
\eqref{E:DULUNITLUNIT}.
Note that the right-hand side of the identity \eqref{E:DULUNITLUNIT} for $\Dfour_{\uLunit} \Lunit$
leads to the presence of $\upzeta$ on 
RHSs~\eqref{E:GENERICDERIVATIVEOFVECTORFIELDCOMPONENTS}--\eqref{E:F1TYPEGENERICDERIVATIVEOFVECTORFIELDCOMPONENTS}.
\end{proof}

\subsubsection{Curvature component decompositions}
\label{SSS:CURVATURECOMPONENTDECOMPOSITIONS}
In the next lemma, we provide some expressions for various components 
of the curvatures of the acoustical metric $\gfour$.
These expressions will be important for controlling the acoustic geometry,
since curvature components appear as source terms in the PDEs that they satisfy; 
see Lemma~\ref{L:NULLLAPSEANDCONNECTIONCOEFFICIENTPDES}.
Moreover, some of the curvature components can be expressed with the help of the equations of Prop.\,\ref{P:RESCALEDEULER}, 
thus tying the evolution of the acoustic geometry to the fluid evolution; 
see Remark~\ref{R:SOMECURVATURECOMPONENTSINVOLVEEULER} and Prop.\,\ref{P:PDESMODIFIEDACOUSTICALQUANTITIES}.

\begin{lemma}[Curvature component decompositions]
\label{L:CURVATUREDECOMPOSITIONS}
Relative to the Cartesian coordinates, the following identity holds,
where on RHS~\eqref{E:RICCIIDENTITY}, the component $\gfour_{\alpha \beta}(\vec{\Psi})$ is treated as a scalar function
under covariant differentiation and $\Chfour_{\alpha} := (\gfour^{-1})^{\kappa \lambda} \gfour_{\alpha \beta} \Chfour_{\kappa \ \lambda}^{\ \beta}$
is treated as a one-form under covariant differentiation:
\begin{align} \label{E:RICCIIDENTITY}
	\Ricfour_{\alpha \beta}
	& = 
		- 
		\frac{1}{2} \square_{\gfour(\vec{\Psi})} \gfour_{\alpha \beta}(\vec{\Psi})
		+
		\frac{1}{2}
		\left\lbrace
			\Dfour_{\alpha} \Chfour_{\beta}
			+
			\Dfour_{\beta} \Chfour_{\alpha}
		\right\rbrace
		+
		\quadsmoothfunction(\vec{\Psi})[\pmb{\partial} \vec{\Psi},\pmb{\partial} \vec{\Psi}].
\end{align}

Moreover,
\begin{align} \label{E:RICLLPERFECTLDERIVATIVE}
	\Ricfour_{\Lunit \Lunit}
	& = \Lunit (\Chfour_{\Lunit})
		+
		k_{\spherenormal \spherenormal}
		\Chfour_{\Lunit}
		+
		\uplambda^{-1} \gensmoothfunction_{(\vec{\Lunit})} \cdot (\vec{\VortVort},\DivGradEnt)
		+ 
		\gensmoothfunction_{(\vec{\Lunit})}
		\cdot
		\pmb{\partial} \vec{\Psi}
		\cdot
		\pmb{\partial} \vec{\Psi}.
\end{align}

Finally, there exist scalar functions on $S_{t,u}$, 
$S_{t,u}$-tangent one-forms, 
and symmetric type $\binom{0}{2}$ $S_{t,u}$-tangent tensorfields,
all schematically denoted by
$
\upxi
$
and verifying 
$
\upxi = \gensmoothfunction_{(\vec{\Lunit})} \cdot \pmb{\partial} \vec{\Psi}
$
(in the sense of Subsubsect.\,\ref{SSS:ADDITIONALSCHEMATIC}),
such that
\begin{align} \label{E:RICCILLDECOMPSINVOLVINGVORTANDENT} 
		\Ricfour_{\Lunit \Lunit}
		-
		\Lunit (\Chfour_{\Lunit})
		& = 
		\uplambda^{-1} \gensmoothfunction_{(\vec{\Lunit})} \cdot (\vec{\VortVort},\DivGradEnt)
		+ 
		\gensmoothfunction_{(\vec{\Lunit})}
		\cdot
		\pmb{\partial} \vec{\Psi}
		\cdot
		\pmb{\partial} \vec{\Psi},
				\\
	\Ricfour_{\Lunit \uLunit}
		-
		\frac{1}{2}
		\left\lbrace
			\Lunit (\Chfour_{\uLunit})
			+
			\uLunit (\Chfour_{\Lunit})
		\right\rbrace
		& = 
		\uplambda^{-1} \gensmoothfunction_{(\vec{\Lunit})} \cdot (\vec{\VortVort},\DivGradEnt)
		+ 
		\gensmoothfunction_{(\vec{\Lunit})}
		\cdot
		(\pmb{\partial} \vec{\Psi},\upzeta)
		\cdot
		\pmb{\partial} \vec{\Psi},
			\label{E:RICCILULDECOMPSINVOLVINGVORTANDENT}
\end{align}

\begin{subequations}
\begin{align}
	\Ricfour_{\Lunit A},
		\,
	\Riemfour_{A \Lunit \Lunit \uLunit}
	& = (\angD,\angprojDarg{\Lunit}) \upxi
			+
			\uplambda^{-1} \gensmoothfunction_{(\vec{\Lunit})} \cdot (\vec{\VortVort},\DivGradEnt)
			+ 
			\gensmoothfunction_{(\vec{\Lunit})}
			\cdot
			(\pmb{\partial} \vec{\Psi},\mytr_{\congsphere} \widetilde{\upchi}^{(Small)},\hat{\upchi},\rgeo^{-1})
			\cdot
			\pmb{\partial} \vec{\Psi},
			\label{E:RICCIANDRIEMANNDECOMPSINVOLVINGVORTANDENT}
					\\
	\Riemfour_{A \Lunit \uLunit A}
	& = 	\angdiv \upxi
			+
			\uplambda^{-1} \gensmoothfunction_{(\vec{\Lunit})} \cdot (\vec{\VortVort},\DivGradEnt)
			+ 
			\gensmoothfunction_{(\vec{\Lunit})}
			\cdot
			(\pmb{\partial} \vec{\Psi},\mytr_{\congsphere} \widetilde{\upchi}^{(Small)},\hat{\upchi},\rgeo^{-1})
			\cdot
			\pmb{\partial} \vec{\Psi},
			\label{E:RIEMALUNDERLINELBCONTRACTEDDECOMPINVOLVINGVORTANDENT}
	\end{align}
\end{subequations}

\begin{subequations}
\begin{align}
		\upepsilon^{AB} \Riemfour_{A \Lunit \uLunit B}
	& = 	\angcurl \upxi
			+ 
			\gensmoothfunction_{(\vec{\Lunit})}
			\cdot
			(\pmb{\partial} \vec{\Psi},\mytr_{\congsphere} \widetilde{\upchi}^{(Small)},\hat{\upchi},\rgeo^{-1})
			\cdot
			\pmb{\partial} \vec{\Psi},
			\label{E:RIEMALUNDERLINELBANTISYMMETRICCONTRACTEDDECOMPINVOLVINGVORTANDENT} 
				\\
	\Riemfour_{\Lunit A \Lunit B}
	& = (\angD,\angprojDarg{\Lunit}) \upxi
			+ 
			\gensmoothfunction_{(\vec{\Lunit})}
			\cdot
			(\pmb{\partial} \vec{\Psi},\mytr_{\congsphere} \widetilde{\upchi}^{(Small)},\hat{\upchi},\rgeo^{-1})
			\cdot
			\pmb{\partial} \vec{\Psi},
			\label{E:RIEMLALBDECOMP}
				\\
	\Riemfour_{A B \Lunit B}
	& = \angdiv \upxi_A
			+
			\gensmoothfunction_{(\vec{\Lunit})}
			\cdot
			(\pmb{\partial} \vec{\Psi},\mytr_{\congsphere} \widetilde{\upchi}^{(Small)},\hat{\upchi},\rgeo^{-1})
			\cdot
			\pmb{\partial} \vec{\Psi},
			\label{E:RIEMCALBCONTRACTEDDECOMP} \\
	\Riemfour_{C A \Lunit B}
	& = \angD \upxi
			+
			\gensmoothfunction_{(\vec{\Lunit})}
			\cdot
			(\pmb{\partial} \vec{\Psi},\mytr_{\congsphere} \widetilde{\upchi}^{(Small)},\hat{\upchi},\rgeo^{-1})
			\cdot
			\pmb{\partial} \vec{\Psi},
				\label{E:RIEMCALBDECOMP}  \\
		\Riemfour_{ABAB}
		& = \angdiv \upxi
				+
				\gensmoothfunction_{(\vec{\Lunit})}
			\cdot
			(\pmb{\partial} \vec{\Psi},\mytr_{\congsphere} \widetilde{\upchi}^{(Small)},\hat{\upchi},\rgeo^{-1})
			\cdot
			\pmb{\partial} \vec{\Psi}.
					\label{E:RIEMABCDCONTRACTEDDECOMP} 
\end{align}
\end{subequations}

\end{lemma}

\begin{remark}
	The curvature identities of Lemma~\ref{L:CURVATUREDECOMPOSITIONS}
	are crucial for the proof of Prop.\,\ref{P:PDESMODIFIEDACOUSTICALQUANTITIES} below.
	In turn, the structure of the equations of
	Prop.\,\ref{P:PDESMODIFIEDACOUSTICALQUANTITIES}
	is crucial for our derivation of estimates for the acoustic geometry.
\end{remark}		

\begin{remark}
	\label{R:SOMECURVATURECOMPONENTSINVOLVEEULER}
	The proofs of the identities
	\eqref{E:RICLLPERFECTLDERIVATIVE}--\eqref{E:RIEMALUNDERLINELBCONTRACTEDDECOMPINVOLVINGVORTANDENT}
	rely on the compressible Euler equations, while the proofs of the remaining identities in Lemma~\ref{L:CURVATUREDECOMPOSITIONS} do not.
	This explains why the former identities feature $\uplambda^{-1}$-dependent source
	terms (which arise from RHS~\eqref{E:RESCALEDCOVARIANTWAVE}).
\end{remark}

\begin{proof}[Discussion of the proofs]
	The identities
	\eqref{E:RICCIIDENTITY}
	and
	\eqref{E:RIEMALUNDERLINELBANTISYMMETRICCONTRACTEDDECOMPINVOLVINGVORTANDENT}--\eqref{E:RIEMABCDCONTRACTEDDECOMP}
	are the same as in \cite{qW2017}*{Lemma~5.12}, 
	whose proofs can be found in \cite{sKiR2005c}. The identities 
	\eqref{E:RICLLPERFECTLDERIVATIVE}--\eqref{E:RIEMALUNDERLINELBCONTRACTEDDECOMPINVOLVINGVORTANDENT}
	also mirror those given in \cite{qW2017}*{Lemma~5.12},
	except here there are new terms 
	of type $\uplambda^{-1} \gensmoothfunction_{(\vec{\Lunit})} \cdot (\vec{\VortVort},\DivGradEnt)$,
	which arise when one uses equation \eqref{E:RESCALEDCOVARIANTWAVE}
	to substitute for the terms $\square_{\gfour(\vec{\Psi})} \Psi$
	that are generated by the term
	$- \frac{1}{2} \square_{\gfour(\vec{\Psi})} \gfour_{\alpha \beta}(\vec{\Psi})$
	on RHS~\eqref{E:RICCIIDENTITY}.
\end{proof}

\subsubsection{Main version of the PDEs verified by the acoustical quantities}
\label{SSS:PDESMODIFIEDACOUSTICALQUANTITIES}
We now provide the main result of Subsect.\,\ref{SS:MAINPDESFORACOUSTICALQUANT}.

\begin{proposition}[PDEs verified by the modified acoustical quantities, assuming a compressible Euler solution]
	\label{P:PDESMODIFIEDACOUSTICALQUANTITIES}
Assume that the Cartesian component functions 
$(\vec{\Psi},\vec{\vortrenormalized},\vec{\GradEnt},\vec{\VortVort},\DivGradEnt)$ 
are solutions to the rescaled compressible Euler equations of
Prop.\,\ref{P:RESCALEDEULER}
(under the conventions of Subsect.\,\ref{SS:NOMORELAMBDA}).
There exist $S_{t,u}$-tangent one-forms and 
symmetric type $\binom{0}{2}$ $S_{t,u}$-tangent tensorfields,
all schematically denoted by
$
\upxi
$
and verifying 
$
\upxi = \gensmoothfunction_{(\vec{\Lunit})} \cdot \pmb{\partial} \vec{\Psi}
$
(see Subsubsect.\,\ref{SSS:ADDITIONALSCHEMATIC} regarding the notation ``$\gensmoothfunction_{(\vec{\Lunit})} \cdot$''),
such that the following schematic identities hold,
where all terms on the left-hand sides are displayed exactly
and terms on the right-hand sides are displayed schematically 
(in particular, we have ignored numerical constants and minus signs on the right-hand sides).

\medskip

\noindent \underline{\textbf{Transport equations involving the Cartesian components} $\Lunit^i$ {\textbf and} $\spherenormal^i$}:
The following evolution equations hold in $\widetilde{\mathcal{M}}$:
\begin{align} \label{E:LUNITIANDNORMALITRANSPORT}
	\Lunit \Lunit^i
	& = \gensmoothfunction_{(\vec{\Lunit})}
			\cdot
			\pmb{\partial} \vec{\Psi},
	&
	\Lunit \spherenormal^i
	& = \gensmoothfunction_{(\vec{\Lunit})}
			\cdot
			\pmb{\partial} \vec{\Psi}.
\end{align}
Moreover, along $\Sigma_0$ (where $w = \rgeo = - u$ and $a = \nulllapse$),
we have
\begin{align} \label{E:LUNITIANDNORMALITRANSPORTALONGSIGMA0}
	\frac{\partial}{\partial w} \Lunit^i
	& = 
			a
			\cdot
			\gensmoothfunction_{(\vec{\Lunit})}
			\cdot
			\pmb{\partial} \vec{\Psi}
			+
			\angD a,
	&
	\frac{\partial}{\partial w} \spherenormal^i
	& = 
			a
			\cdot
			\gensmoothfunction_{(\vec{\Lunit})}
			\cdot
			\pmb{\partial} \vec{\Psi}
			+
			\angD a.
\end{align}

\medskip

	\noindent \underline{\textbf{Transport equations involving the Cartesian components} $\Theta_{(A)}^i$}:
	For $A=1,2$ and $i=1,2,3$, let $\left(\frac{\partial}{\partial \upomega^A} \right)^i$ denote a Cartesian component of 
	$\frac{\partial}{\partial \upomega^A}$ (i.e., $\left(\frac{\partial}{\partial \upomega^A} \right)^i = \frac{\partial}{\partial \upomega^A} x^i$),
	and let $\Theta_{(A)}$ be the $\Sigma_t$-tangent vectorfield with Cartesian components defined by
	\begin{align} \label{E:RESCALED}
		\Theta_{(A)}^i 
		& := \frac{1}{\rgeo} \left(\frac{\partial}{\partial \upomega^A} \right)^i.
	\end{align}
	Then the following evolution equation holds in $\widetilde{\mathcal{M}}$:
	\begin{align} \label{E:EVOLUTIONEQUATIONALONGINTEGRALCURVESOFLUNITFORCARTESIANCOMPONENTSOFTHETAA}
	\Lunit \Theta_{(A)}^i
	& 
	= 
	\gensmoothfunction_{(\vec{\Lunit})}
	\cdot
	(\pmb{\partial} \vec{\Psi},\mytr_{\congsphere} \widetilde{\upchi}^{(Small)},\hat{\upchi})
	\cdot
	\vec{\Theta}_{(A)}.
	\end{align}
	Moreover, along $\Sigma_0$ (where $w = \rgeo = - u$ and $a = \nulllapse$), the following
	evolution equation holds for $(w,\upomega) \in (0,\RescaledFoliationparameter] \times \mathbb{S}^2$:
	\begin{align} \label{E:EVOLUTIONEQUATIONALONGSIGMA0FORCARTESIANCOMPONENTSOFTHETAA}
	\frac{\partial}{\partial w} \Theta_{(A)}^i
	& = 
	a
	\cdot
	\gensmoothfunction_{(\vec{\Lunit})} 
	\cdot 
	(\pmb{\partial} \vec{\Psi},\hat{\upchi})
	\cdot
	\vec{\Theta}_{(A)}
	+
	\gensmoothfunction_{(\vec{\Lunit})}
	\cdot
	\angD a
	\cdot
	\vec{\Theta}_{(A)},
	\end{align}
	where $\vec{\Theta}_{(A)} := (\Theta_{(A)}^1,\Theta_{(A)}^2,\Theta_{(A)}^3)$.
	
\medskip

\noindent \underline{\textbf{Transport equations connected to the trace of} $\upchi$}:
	\begin{subequations}
	\begin{align}
		\Lunit \mytr_{\congsphere} \widetilde{\upchi}^{(Small)}
		+
		\frac{2}{\rgeo}
		\mytr_{\congsphere} \widetilde{\upchi}^{(Small)}
		& = 
			\uplambda^{-1} \gensmoothfunction_{(\vec{\Lunit})} \cdot (\vec{\VortVort},\DivGradEnt)
			\label{E:MODIFIEDRAYCHAUDHURI} \\
		& \ \
			+
			\gensmoothfunction_{(\vec{\Lunit})}
			\cdot
			(\pmb{\partial} \vec{\Psi},\mytr_{\congsphere} \widetilde{\upchi}^{(Small)},\rgeo^{-1})
			\cdot
			\pmb{\partial} \vec{\Psi}
			+
			|\hat{\upchi}|_{\gsphere}^2
			+
			\mytr_{\congsphere} \widetilde{\upchi}^{(Small)} 
			\cdot
			\mytr_{\congsphere} \widetilde{\upchi}^{(Small)},
			\notag
				\\
		\angprojDarg{\Lunit} \angD \mytr_{\congsphere} \widetilde{\upchi}^{(Small)}
		+
		\frac{3}{\rgeo}
		 \angD \mytr_{\gsphere} \widetilde{\upchi}^{(Small)}
		& = 
			\uplambda^{-1} \gensmoothfunction_{(\vec{\Lunit})} \cdot \angD (\vec{\VortVort},\DivGradEnt)
				\label{E:ANGDCOMMUTEDMODIFIEDRAYCHAUDHURI}
				\\
		& \ \
			+
			\uplambda^{-1}
			\gensmoothfunction_{(\vec{\Lunit})}
			\cdot
			(\vec{\GradEnt} \cdot \pmb{\partial} \vec{\Psi},
			\pmb{\partial} \vec{\Psi}, \partial \vec{\vortrenormalized}, \partial \vec{\GradEnt}) 
			\cdot 
			(\pmb{\partial} \vec{\Psi},\mytr_{\congsphere} \widetilde{\upchi}^{(Small)},\hat{\upchi},\rgeo^{-1})	
				\notag \\
		& \ \
			+
			\gensmoothfunction_{(\vec{\Lunit})}
			\cdot
			\angD \pmb{\partial} \vec{\Psi}
			\cdot
			(\pmb{\partial} \vec{\Psi},\mytr_{\congsphere} \widetilde{\upchi}^{(Small)}, \rgeo^{-1})
				\notag \\
		& \ \
			+
			\gensmoothfunction_{(\vec{\Lunit})}
			\cdot
			\angD \hat{\upchi} \cdot \hat{\upchi}
			+
			\gensmoothfunction_{(\vec{\Lunit})}
			\cdot
			\angD \mytr_{\congsphere} \widetilde{\upchi}^{(Small)} 
			\cdot 
			(\pmb{\partial} \vec{\Psi}, \mytr_{\congsphere} \widetilde{\upchi}^{(Small)},\hat{\upchi})
				\notag \\
		& \ \
			+
			\gensmoothfunction_{(\vec{\Lunit})}
			\cdot
			(\pmb{\partial} \vec{\Psi},\mytr_{\congsphere} \widetilde{\upchi}^{(Small)},\hat{\upchi},\rgeo^{-1})		
			\cdot
			(\pmb{\partial} \vec{\Psi},\mytr_{\congsphere} \widetilde{\upchi}^{(Small)},\rgeo^{-1})	
			\cdot
			\pmb{\partial} \vec{\Psi}.
			\notag
	\end{align}
	\end{subequations}
Above and throughout, we use $\vec{\GradEnt} \cdot \pmb{\partial} \vec{\Psi}$ to 
schematically denote terms of the form $\GradEnt^a \partial_{\alpha} \Psi_{\iota}$, 
where $a=1,2,3$, $\alpha=0,1,2,3$ and $\iota=0,1,2,3,4$. 

Moreover,
\begin{align}
		&
		\Lunit 
		\left\lbrace
			\frac{1}{2}
			\mytr_{\congsphere} \widetilde{\upchi}
			\volrat
		\right\rbrace
		-
		\frac{1}{4}
		\left(
			\mytr_{\gsphere} \upchi
		\right)^2
		\volrat
		+
		\frac{1}{2}
		\left\lbrace
			\Lunit \ln \nulllapse
		\right\rbrace
		\mytr_{\congsphere} \widetilde{\upchi}
		\volrat
		-
		|\angD \upsigma|_{\gsphere}^2
		\volrat
			\label{E:ANNOYINGERRORTERMALGEBRAICEXPRESSION} 
				\\
		& = 
			\uplambda^{-1} 
			\gensmoothfunction_{(\vec{\Lunit})} 
			\cdot 
			(\vec{\VortVort},\DivGradEnt)
			\cdot
			\volrat
			+
			\gensmoothfunction_{(\vec{\Lunit})}
			\cdot
			(\pmb{\partial} \vec{\Psi},\mytr_{\congsphere} \widetilde{\upchi}^{(Small)},\rgeo^{-1})
			\cdot
			\pmb{\partial} \vec{\Psi}
			\cdot
			\volrat
			+
			|\hat{\upchi}|_{\gsphere}^2
			\cdot
			\volrat
			+
			|\angD \upsigma|_{\gsphere}^2
			\cdot
			\volrat.
			\notag
\end{align}

\medskip
\noindent \underline{\textbf{PDEs involving} $\hat{\upchi}$}:

\begin{align} \label{E:DIVDIVTRFREECHISCHEMATIC}
		\angdiv \hat{\upchi}
		& = 
			\angD \mytr_{\congsphere} \widetilde{\upchi}^{(Small)}
			+
			\angdiv \upxi
			+
			\gensmoothfunction_{(\vec{\Lunit})}
			\cdot
			(\pmb{\partial} \vec{\Psi},\mytr_{\congsphere} \widetilde{\upchi}^{(Small)},\hat{\upchi},\rgeo^{-1})
			\cdot
			\pmb{\partial} \vec{\Psi},
	\end{align}
	
	\begin{align} \label{E:LDERIVATIVECHIHATAFTERUSINGEULER}
	\angprojDarg{\Lunit} 
	\hat{\upchi}
	+ 
	(\mytr_{\gsphere} \upchi)
	\hat{\upchi}
	& = 
		(\angD,\angprojDarg{\Lunit}) \upxi
		+
		\uplambda^{-1} \gensmoothfunction_{(\vec{\Lunit})} \cdot (\vec{\VortVort},\DivGradEnt)
		+
		\gensmoothfunction_{(\vec{\Lunit})}
		\cdot
		(\pmb{\partial} \vec{\Psi},\mytr_{\congsphere} \widetilde{\upchi}^{(Small)},\hat{\upchi},\rgeo^{-1})		
		\cdot
		\pmb{\partial} \vec{\Psi}.
\end{align}

\noindent \underline{\textbf{The transport equation for} $\upzeta$}:

\begin{align}
	\angprojDarg{\Lunit}
	\upzeta
	+
	\frac{1}{2} 
	(\mytr_{\gsphere} \upchi)
	\upzeta
	& = 
		(\angD,\angprojDarg{\Lunit}) \upxi
		+
		\uplambda^{-1} \gensmoothfunction_{(\vec{\Lunit})} \cdot (\vec{\VortVort},\DivGradEnt)
		+
		\gensmoothfunction_{(\vec{\Lunit})}
		\cdot
		(\pmb{\partial} \vec{\Psi},\mytr_{\congsphere} \widetilde{\upchi}^{(Small)},\hat{\upchi},\rgeo^{-1})		
		\cdot
		\pmb{\partial} \vec{\Psi}
		+
		\gensmoothfunction_{(\vec{\Lunit})}
		\cdot
		\upzeta 
		\cdot 
		\hat{\upchi}.
			\label{E:LDERIVATIVETORSIONAFTERUSINGEULER} 
\end{align}

\noindent \underline{\textbf{The transport equation for} $\nulllapse$}:
	
	\begin{align}
		\Lunit \nulllapse
		& = \nulllapse \cdot \gensmoothfunction_{(\vec{\Lunit})} \cdot \pmb{\partial} \vec{\Psi}.
			\label{E:EVOLUTIONNULLAPSEUSEEULER}
	\end{align}

\medskip
\noindent \underline{\textbf{Transport equation for} $\gsphere$}:
Along the integral curves of $\Lunit$, parameterized by $t$, we have, 
with $\stgsphere$ the standard round metric on the Euclidean unit sphere $\mathbb{S}^2$,
the following identity:
\begin{align} \label{E:EVOLUTIONEQUATIONFORANGULARCOORDINATECOMPONENTSOFGSPHEREMINUSEUCLIDEAN}
&
	\frac{d}{dt}
	\left\lbrace
		\rgeo^{-2}
		\gsphere\left(\frac{\partial}{\partial \upomega^A},\frac{\partial}{\partial \upomega^B} \right) 
			- 
			\stgsphere\left(\frac{\partial}{\partial \upomega^A},\frac{\partial}{\partial \upomega^B} \right)
	\right\rbrace
		\\
	& = \left\lbrace
				\mytr_{\congsphere} \widetilde{\upchi}^{(Small)}
				-
				\Chfour_{\Lunit}
			\right\rbrace
			\left\lbrace
				\rgeo^{-2}
				\gsphere\left(\frac{\partial}{\partial \upomega^A},\frac{\partial}{\partial \upomega^B} \right) 
				- 
				\stgsphere\left(\frac{\partial}{\partial \upomega^A},\frac{\partial}{\partial \upomega^B} \right)
			\right\rbrace
		\notag \\
	& \ \
		+
		\left\lbrace
				\mytr_{\congsphere} \widetilde{\upchi}^{(Small)}
				-
				\Chfour_{\Lunit}
			\right\rbrace
			\stgsphere\left(\frac{\partial}{\partial \upomega^A},\frac{\partial}{\partial \upomega^B} \right)
			+
			\frac{2}{\rgeo^2}
			\hat{\upchi}\left(\frac{\partial}{\partial \upomega^A},\frac{\partial}{\partial \upomega^B} \right).
		\notag
\end{align}

\noindent \underline{\textbf{Transport equations for} $\volrat$ \textbf{and} $\angD \volrat$}:

	\begin{subequations}
	\begin{align} \label{E:LUNITVOLUMEFORMRGEOTOMINUSTWORESCALED}
		\Lunit
		\ln \left(\rgeo^{-2} \volrat \right)
		& = 
			\mytr_{\gsphere} \upchi - \frac{2}{\rgeo}
			=
			\mytr_{\congsphere} \widetilde{\upchi}^{(Small)}
			-
			\Chfour_{\Lunit},
			\\
		\Lunit
		\angD 
		\ln \left(\rgeo^{-2} \volrat \right)
		+
		\frac{1}{2} 
		(\mytr_{\gsphere} \upchi)
		\angD
		\ln \left(\rgeo^{-2} \volrat \right)
		& = 
		\gensmoothfunction_{(\vec{\Lunit})}
		\cdot
		\hat{\upchi} 
		\cdot \angD \ln \left(\rgeo^{-2} \volrat \right)
		+
		\angD \mytr_{\congsphere} \widetilde{\upchi}^{(Small)}
		-
		\angD(\Chfour_{\Lunit}).
		\label{E:LUNITANGDCOMMUTEDVOLUMEFORMRGEOTOMINUSTWORESCALED}
	\end{align}
	\end{subequations}

\medskip
\noindent \underline{\textbf{An algebraic identity for} $\upmu$}:
	The mass aspect function $\upmu$ defined in \eqref{E:MASSASPECT} verifies the following identity:
	\begin{align} \label{E:MASSASPECTDECOMP}
		\upmu
		 & = 
				\uplambda^{-1} \gensmoothfunction_{(\vec{\Lunit})} \cdot (\vec{\VortVort},\DivGradEnt)
				+
				\angdiv \upxi
				+
				\gensmoothfunction_{(\vec{\Lunit})}
				\cdot
				\hat{\upchi} 
				\cdot
				\hat{\upchi}
				+
				\gensmoothfunction_{(\vec{\Lunit})}
				\cdot
				\angD \ln\left( \rgeo^{-2} \volrat \right)
				\cdot
				(\pmb{\partial} \vec{\Psi},\upzeta)
				\\
		& \ \
				+
				\gensmoothfunction_{(\vec{\Lunit})}
				\cdot
				(\pmb{\partial} \vec{\Psi},\mytr_{\congsphere} \widetilde{\upchi}^{(Small)},\hat{\upchi},\rgeo^{-1})
				\cdot
				\pmb{\partial} \vec{\Psi}.
				\notag
	\end{align}

\medskip	
\noindent \underline{\textbf{The transport equation for} $\check{\upmu}$}:
	The modified mass aspect function $\check{\upmu}$ defined by \eqref{E:MODMASSASPECT}
	verifies the following transport equation:
	\begin{align} \label{E:MODIFIEDMASSASPECTEVOLUTIONEQUATION}
		\Lunit \check{\upmu}
		+
		(\mytr_{\gsphere} \upchi) 
		\check{\upmu}
		&
		=	
			\mathfrak{I}_{(1)}
			+
			\mathfrak{I}_{(2)},
	\end{align}
	
	\begin{subequations}
	\begin{align}
	\mathfrak{I}_{(1)}
	& = \rgeo^{-1} \angdiv \upxi
			+
			\rgeo^{-2} \upxi,
		\label{E:FIRSTINHOMTERMMODIFIEDMASSASPECTEVOLUTIONEQUATION} 
			\\
	\mathfrak{I}_{(2)}
	& =
			\uplambda^{-1} \gensmoothfunction_{(\vec{\Lunit})} \cdot \pmb{\partial}(\vec{\VortVort},\DivGradEnt)
			+
			\uplambda^{-1}
			\gensmoothfunction_{(\vec{\Lunit})} 
			\cdot
			(\vec{\GradEnt} \cdot \pmb{\partial} \vec{\Psi},
			\pmb{\partial} \vec{\Psi},\pmb{\partial} \vec{\vortrenormalized},\pmb{\partial} \vec{\GradEnt}) 
			\cdot 
			(\pmb{\partial} \vec{\Psi},\mytr_{\congsphere} \widetilde{\upchi}^{(Small)},\hat{\upchi},\upzeta,\rgeo^{-1})	
				\label{E:SECONDINHOMTERMMODIFIEDMASSASPECTEVOLUTIONEQUATION} 
				\\
		& \ \
				+ 
				\gensmoothfunction_{(\vec{\Lunit})} 
				\cdot
				\angD \widetilde{\upzeta}
				\cdot
				\hat{\upchi} 
				+
				\gensmoothfunction_{(\vec{\Lunit})} 
				\cdot
				\angD \upsigma
				\cdot
				(\angD \pmb{\partial} \vec{\Psi},\angD \mytr_{\congsphere} \widetilde{\upchi}^{(Small)}) 
				+
				\gensmoothfunction_{(\vec{\Lunit})} 
				\cdot
				\angD \upsigma
				\cdot
				(\pmb{\partial} \vec{\Psi},\mytr_{\congsphere} \widetilde{\upchi}^{(Small)},\hat{\upchi},\rgeo^{-1}) 
				\cdot \pmb{\partial} \vec{\Psi}
			\notag \\
		& \ \
			+
			\gensmoothfunction_{(\vec{\Lunit})} 
			\cdot
			\angD \mytr_{\congsphere} \widetilde{\upchi}^{(Small)}
			\cdot
			(\pmb{\partial} \vec{\Psi},\upzeta)
			\notag \\
		& \ \
			+
			\gensmoothfunction_{(\vec{\Lunit})} 
			\cdot
			(\pmb{\partial} \vec{\Psi},\mytr_{\congsphere} \widetilde{\upchi}^{(Small)},\hat{\upchi},\upzeta,\rgeo^{-1})
			\cdot
			(\pmb{\partial} \vec{\Psi},\mytr_{\congsphere} \widetilde{\upchi}^{(Small)},\hat{\upchi},\upzeta)
			\cdot
			(\pmb{\partial} \vec{\Psi},\mytr_{\congsphere} \widetilde{\upchi}^{(Small)},\hat{\upchi})
				\notag \\
		& \ \
			+
			\gensmoothfunction_{(\vec{\Lunit})} 
			\cdot
			(\pmb{\partial} \vec{\Psi},\mytr_{\congsphere} \widetilde{\upchi}^{(Small)},\hat{\upchi},\upzeta) 
			\cdot 
			\pmb{\partial}^2 \vec{\Psi}.
			\notag
	\end{align}
	\end{subequations}

\medskip	
\noindent \underline{\textbf{The Hodge system for} $\upzeta$}:
	The torsion $\upzeta$ defined in \eqref{E:TORSION} satisfies the following Hodge system on $S_{t,u}$:
	\begin{subequations}
	\begin{align} \label{E:TORSIONDIV}
		\angdiv \upzeta
		& = 
				\uplambda^{-1} 
				\gensmoothfunction_{(\vec{\Lunit})} 
				\cdot
				(\vec{\VortVort},\DivGradEnt)
				+
				\angdiv \upxi
				+
				\gensmoothfunction_{(\vec{\Lunit})} 
				\cdot
				\upzeta 
				\cdot 
				\upzeta
				+
				\gensmoothfunction_{(\vec{\Lunit})} 
				\cdot	
				\hat{\upchi}
				\cdot
				\hat{\upchi}
				+
				\gensmoothfunction_{(\vec{\Lunit})} 
				\cdot
				(\pmb{\partial} \vec{\Psi},\mytr_{\congsphere} \widetilde{\upchi}^{(Small)},\hat{\upchi},\rgeo^{-1})
				\cdot
				\pmb{\partial} \vec{\Psi}
					\\
			& \ \
				+
				\gensmoothfunction_{(\vec{\Lunit})} 
				\cdot
				\angD \ln\left( \rgeo^{-2} \volrat \right)
				\cdot
				(\pmb{\partial} \vec{\Psi},\upzeta),
				\notag \\
		\angcurl \upzeta
		& =
			\angcurl \upxi
			+
			\gensmoothfunction_{(\vec{\Lunit})} 
			\cdot
			\hat{\upchi} 
			\cdot
			\hat{\upchi}
			+
			\gensmoothfunction_{(\vec{\Lunit})} 
			\cdot
			(\pmb{\partial} \vec{\Psi},\mytr_{\congsphere} \widetilde{\upchi}^{(Small)},\hat{\upchi},\rgeo^{-1})
			\cdot
			\pmb{\partial} \vec{\Psi}.
			\label{E:TORSIONCURL} 
	\end{align}
	\end{subequations}

\medskip	
\noindent \underline{\textbf{The Hodge system for} $\widetilde{\upzeta}$}:
	The modified torsion $\widetilde{\upzeta}$ defined by \eqref{E:MODTORSION}
	satisfies the following Hodge system on $S_{t,u}$:
	\begin{subequations}
	\begin{align} \label{E:MODIFIEDTORSIONDIV}
		\angdiv \widetilde{\upzeta}
		-
		\frac{1}{2} \check{\upmu}
		& = 
				\angdiv \upxi
				+
				\uplambda^{-1} 
				\gensmoothfunction_{(\vec{\Lunit})} 
				\cdot
				(\vec{\VortVort},\DivGradEnt)
				+
				\gensmoothfunction_{(\vec{\Lunit})} 
				\cdot
				\upzeta 
				\cdot 
				\upzeta
				+
				\gensmoothfunction_{(\vec{\Lunit})} 
				\cdot	
				\hat{\upchi}
				\cdot
				\hat{\upchi}
				+
				\gensmoothfunction_{(\vec{\Lunit})} 
				\cdot
				(\pmb{\partial} \vec{\Psi},\mytr_{\congsphere} \widetilde{\upchi}^{(Small)},\hat{\upchi},\rgeo^{-1})
				\cdot
				\pmb{\partial} \vec{\Psi},
			\\
		\angcurl \widetilde{\upzeta}
		& = 
			\angcurl \upxi
			+
			\gensmoothfunction_{(\vec{\Lunit})} 
			\cdot
			\hat{\upchi} 
			\cdot
			\hat{\upchi}
			+
			\gensmoothfunction_{(\vec{\Lunit})} 
			\cdot
			(\pmb{\partial} \vec{\Psi},\mytr_{\congsphere} \widetilde{\upchi}^{(Small)},\hat{\upchi},\rgeo^{-1})
			\cdot
			\pmb{\partial} \vec{\Psi}.
				\label{E:MODIFIEDTORSIONCURL}
	\end{align}
	\end{subequations}

\noindent \underline{\textbf{The Hodge system for} $\widetilde{\upzeta} - \angupmu$}:	
	The difference $\widetilde{\upzeta} - \angupmu$ 
	(where $\widetilde{\upzeta}$ is defined by \eqref{E:MODTORSION} and $\angupmu$ is defined by \eqref{E:FURTHERMODOFMASSASPECT})
	verifies the following Hodge system on $S_{t,u}$ 
	(see definition \eqref{E:AVERAGEVALUEOFSCALARFUNCTION} regarding ``overline'' notation):
	\begin{subequations}
	\begin{align} \label{E:COMBBINEDMODIFIEDTORSIONMINUSANGMODMASSASPECTDIV}
		\angdiv (\widetilde{\upzeta} - \angupmu)
		& = 
			\angdiv \upxi
				+
				\left\lbrace
				\uplambda^{-1} 
				\gensmoothfunction_{(\vec{\Lunit})} 
				\cdot
				(\vec{\VortVort},\DivGradEnt)
				-
				\uplambda^{-1} 
				\overline{
				\gensmoothfunction_{(\vec{\Lunit})} 
				\cdot
				(\vec{\VortVort},\DivGradEnt)}
				\right\rbrace
					\\
			& \ \
				+
				\left\lbrace
				\gensmoothfunction_{(\vec{\Lunit})} 
				\cdot
				\upzeta 
				\cdot 
				\upzeta
				-
				\overline{	\gensmoothfunction_{(\vec{\Lunit})} 
				\cdot
				\upzeta 
				\cdot 
				\upzeta}
				\right\rbrace
				+
				\left\lbrace
				\gensmoothfunction_{(\vec{\Lunit})} 
				\cdot	
				\hat{\upchi}
				\cdot
				\hat{\upchi}
				-
				\overline{\gensmoothfunction_{(\vec{\Lunit})} 
				\cdot	
				\hat{\upchi}
				\cdot
				\hat{\upchi}}
				\right\rbrace
					\notag \\
			& \ \
				+
				\left\lbrace
				\gensmoothfunction_{(\vec{\Lunit})} 
				\cdot
				(\pmb{\partial} \vec{\Psi},\mytr_{\congsphere} \widetilde{\upchi}^{(Small)},\hat{\upchi},\rgeo^{-1})
				\cdot
				\pmb{\partial} \vec{\Psi}
				-
				\overline{
				\gensmoothfunction_{(\vec{\Lunit})} 
				\cdot
				(\pmb{\partial} \vec{\Psi},\mytr_{\congsphere} \widetilde{\upchi}^{(Small)},\hat{\upchi},\rgeo^{-1})
				\cdot
				\pmb{\partial} \vec{\Psi}}
				\right\rbrace,
				\notag \\
		\angcurl (\widetilde{\upzeta} - \angupmu)
		& = 
		\angcurl \upxi
			+
			\gensmoothfunction_{(\vec{\Lunit})} 
			\cdot
			\hat{\upchi} 
			\cdot
			\hat{\upchi}
			+
			\gensmoothfunction_{(\vec{\Lunit})} 
			\cdot
			(\pmb{\partial} \vec{\Psi},\mytr_{\congsphere} \widetilde{\upchi}^{(Small)},\hat{\upchi},\rgeo^{-1})
			\cdot
			\pmb{\partial} \vec{\Psi}.
			\label{E:COMBBINEDMODIFIEDTORSIONMINUSANGMODMASSASPECTCURL}
	\end{align}
	\end{subequations}

	\noindent \underline{\textbf{A decomposition of} $\angupmu$ and \textbf{a Hodge-transport system for the constituent parts}}:	
	Let $\mathfrak{I}_{(1)}$ and $\mathfrak{I}_{(2)}$ be the inhomogeneous terms from 
	\eqref{E:FIRSTINHOMTERMMODIFIEDMASSASPECTEVOLUTIONEQUATION}--\eqref{E:SECONDINHOMTERMMODIFIEDMASSASPECTEVOLUTIONEQUATION}. 
	Then in $\widetilde{\mathcal{M}}^{(Int)}$ (see \eqref{E:INTERIORANDEXTERIOREGIONSAREUNIONSOFSPHERES}), 
	we can decompose the solution $\angupmu$ to \eqref{E:FURTHERMODOFMASSASPECT} as follows:
	\begin{align} \label{E:KEYANGMUALGEBRAICDECOMPOSITION}
	\angupmu
		& = 
		\angupmu_{(1)}
		+
		\angupmu_{(2)},
	\end{align}
	where $\angupmu_{(1)}$ and $\angupmu_{(2)}$
	verify the following Hodge-transport PDE systems:
	\begin{subequations}
	\begin{align} 
		\angdiv 
		\left\lbrace
			\angprojDarg{\Lunit} \angupmu_{(1)} 
			+
			\frac{1}{2} 
			(\mytr_{\gsphere} \upchi) 
			\angupmu_{(1)}
		\right\rbrace
		& =
				\mathfrak{I}_{(1)} - \overline{\mathfrak{I}_{(1)}},
				\label{E:ANGDIVLDERIVATIVEANGMU1} \\
		\angcurl 
		\left\lbrace
			\angprojDarg{\Lunit} \angupmu_{(1)} 
			+
			\frac{1}{2} 
			(\mytr_{\gsphere} \upchi) 
			\angupmu_{(1)}
		\right\rbrace
		& = 0,
		\label{E:ANGCURLLDERIVATIVEANGMU1}
	\end{align}
	\end{subequations}
	
	\begin{subequations}
	\begin{align} 
		\angdiv 
		\left\lbrace
			\angprojDarg{\Lunit} \angupmu_{(2)} 
			+
			\frac{1}{2} 
			(\mytr_{\gsphere} \upchi)
			\angupmu_{(2)}
		\right\rbrace
		& =
				\mathfrak{I}_{(2)} - \overline{\mathfrak{I}_{(2)}}
				+
				\hat{\upchi} 
				\cdot
				\angD \angupmu
				+
				(\angD \pmb{\partial} \vec{\Psi},\angD \mytr_{\congsphere} \widetilde{\upchi}^{(Small)})
				\cdot
				\angupmu
					\label{E:ANGDIVLDERIVATIVEANGMU2}\\
			& \ \
				+
				(\pmb{\partial} \vec{\Psi},\mytr_{\congsphere} \widetilde{\upchi}^{(Small)},\hat{\upchi},\rgeo^{-1})
				\cdot
				(\pmb{\partial} \vec{\Psi},\mytr_{\congsphere} \widetilde{\upchi}^{(Small)},\hat{\upchi})
				\cdot
				\angupmu
				+
				(\mytr_{\gsphere} \upchi - \overline{\mytr_{\gsphere} \upchi}) \overline{\check{\upmu}},
					\notag \\
		\angcurl 
		\left\lbrace
			\angprojDarg{\Lunit} \angupmu_{(2)} 
			+
			\frac{1}{2} 
			(\mytr_{\gsphere} \upchi) \angupmu_{(2)}
		\right\rbrace
		& = 
				\hat{\upchi} 
				\cdot
				\angD \angupmu
				+
				(\angD \pmb{\partial} \vec{\Psi},\angD \mytr_{\congsphere} \widetilde{\upchi}^{(Small)})
				\cdot
				\angupmu
					\label{E:ANGCURLLDERIVATIVEANGMU2} \\
			& \ \
				+
				(\pmb{\partial} \vec{\Psi},\mytr_{\congsphere} \widetilde{\upchi}^{(Small)},\hat{\upchi},\rgeo^{-1})
				\cdot
				(\pmb{\partial} \vec{\Psi},\mytr_{\congsphere} \widetilde{\upchi}^{(Small)},\hat{\upchi})
				\cdot
				\angupmu,
					\notag 
	\end{align}
	\end{subequations}
	subject to the following initial conditions along the cone-tip axis
	for $u \in [0,\RescaledTboot]$:
	\begin{align} \label{E:CONETIPINITIALCONDITIONSFORANGMUSLASHI1AND2}
		|
		\angupmu_{(1)}
		-
		\angupmu
		|_{\gsphere}
		(t,u,\upomega)
		&
		\rightarrow 0 \mbox{ as } t \downarrow u,
		&
		|
			\angupmu_{(2)}
		|_{\gsphere}(t,u,\upomega)
		\rightarrow 0 \mbox{ as } t \downarrow u.
	\end{align}
\end{proposition}

\begin{proof}[Proof sketch]
	Throughout, we will silently use the identities provided by Lemma~\ref{L:ANGULARDERIVATIVESOFSOMESCALARFUNCTIONS}.
	
	The equations in \eqref{E:LUNITIANDNORMALITRANSPORT} are a straightforward consequence of the first equation in \eqref{E:DLUNITLUNIT}
	and the relation $\Lunit^i = \spherenormal^i + \gensmoothfunction(\vec{\Psi})$.
	
	To prove \eqref{E:LUNITIANDNORMALITRANSPORTALONGSIGMA0}, we first note that along $\Sigma_0$,
	we have the vectorfield identity $\frac{\partial}{\partial w} = a \spherenormal$ (see \eqref{E:PARTIALPARTIALWALONGSIGMA0}).
	Also using the identity
	$
		\Lunit^i = \Transport^i + \spherenormal^i
	$
	and the fact that $\Transport^i = \gensmoothfunction(\vec{\Psi})$,
	we deduce that
	$
	\frac{\partial}{\partial w} \Lunit^i
	= 	
			a 
			\cdot
			\gensmoothfunction_{(\vec{\Lunit})}
			\cdot
			\pmb{\partial} \vec{\Psi}
			+
			\frac{\partial}{\partial w} \spherenormal^i
	$.
	Thus, to conclude both equations in \eqref{E:LUNITIANDNORMALITRANSPORTALONGSIGMA0},
	it suffices to derive the equation for $\frac{\partial}{\partial w} \spherenormal^i$ stated in \eqref{E:LUNITIANDNORMALITRANSPORTALONGSIGMA0}.
	The desired result is a straightforward consequence of
	the identity $\frac{\partial}{\partial w} = a \spherenormal$
	and the identity \eqref{E:PROJECTEDNORMALDERIVATIVEOFSPHERENORMAL}.
	
	\eqref{E:MODIFIEDRAYCHAUDHURI} is essentially proved as \cite{qW2017}*{Equation~(5.75)}. 
	The only difference is that in the present work, we have the $\uplambda^{-1}$-multiplied terms
	on RHS~\eqref{E:MODIFIEDRAYCHAUDHURI}, which arise when one uses
	equation \eqref{E:RICCILLDECOMPSINVOLVINGVORTANDENT} to algebraically substitute for the term
	$\Ricfour_{\Lunit \Lunit}$ on RHS~\eqref{E:RAYCHAUDHURI}.
	Similarly, \eqref{E:ANGDCOMMUTEDMODIFIEDRAYCHAUDHURI}
	was essentially proved as \cite{qW2017}*{Equation~(5.76)},
	the only difference being that we take into account Lemma~\ref{L:ANGULARDERIVATIVESOFSOMESCALARFUNCTIONS}
	and the expressions
	\eqref{E:RESCALEDRENORMALIZEDCURLOFSPECIFICVORTICITY} and \eqref{E:RESCALEDRENORMALIZEDDIVOFENTROPY} 
	for the rescaled $\VortVort^i$ and $\DivGradEnt$
	when computing $\angD$ applied to the 
	$\uplambda^{-1}$-multiplied terms
	on RHS~\eqref{E:MODIFIEDRAYCHAUDHURI}.
	
	The identity \eqref{E:ANNOYINGERRORTERMALGEBRAICEXPRESSION}
	follows from the same arguments used to prove \eqref{E:MODIFIEDRAYCHAUDHURI},
	based on
	\eqref{E:RICLLPERFECTLDERIVATIVE},
	\eqref{E:EVOLUTIONVOLUMELEMENT},
	\eqref{E:EVOLUTIONNULLAPSE},
	and
	\eqref{E:RAYCHAUDHURI};
	see the proof of \cite{qW2017}*{Proposition~7.22} for the analogous identity in the context of scalar wave equations.
	
	Based on \eqref{E:RIEMCALBCONTRACTEDDECOMP} and \eqref{E:ANGDIVTRACEFREEPARTOFCHI}
	(and the standard properties of $\Riemfour_{\alpha \beta \gamma \delta}$ under exchanges of indices),
	the identity \eqref{E:DIVDIVTRFREECHISCHEMATIC} was proved as \cite{qW2017}*{Equation~(5.77)}.
	
	The identity \eqref{E:LDERIVATIVECHIHATAFTERUSINGEULER}
	is essentially proved as \cite{qW2017}*{Equation~(5.68)}
	based on Lemma~\ref{L:CURVATUREDECOMPOSITIONS}
	and equation \eqref{E:LDERIVATIVECHIHAT}.
	The only difference (modulo Footnote~\ref{FN:CORRECTIONOFTYPOS}) is that in the present work, we have the $\uplambda^{-1}$-multiplied terms
	on RHS~\eqref{E:LDERIVATIVECHIHATAFTERUSINGEULER}, which arise when one uses
	equation \eqref{E:RICCILLDECOMPSINVOLVINGVORTANDENT} to algebraically substitute for the term
	$\Ricfour_{\Lunit \Lunit}$ on RHS~\eqref{E:LDERIVATIVECHIHAT}.
	Similar remarks apply to equation \eqref{E:LDERIVATIVETORSIONAFTERUSINGEULER},
	which follows from \eqref{E:LDERIVATIVETORSION} and \eqref{E:RICCIANDRIEMANNDECOMPSINVOLVINGVORTANDENT}.
	
	\eqref{E:EVOLUTIONNULLAPSEUSEEULER} follows from \eqref{E:EVOLUTIONNULLAPSE}.
	
	\eqref{E:EVOLUTIONEQUATIONFORANGULARCOORDINATECOMPONENTSOFGSPHEREMINUSEUCLIDEAN}
	was proved just below \cite{qW2017}*{Equation~(5.88)}.
	
	\eqref{E:LUNITVOLUMEFORMRGEOTOMINUSTWORESCALED}
	and
	\eqref{E:LUNITANGDCOMMUTEDVOLUMEFORMRGEOTOMINUSTWORESCALED}
	were derived in the proof of \cite{qW2017}*{Lemma~5.15},
	where $\ln\left(\rgeo^{-2} \volrat \right)$ was denoted by ``$\varphi$.''
	
	\eqref{E:MASSASPECTDECOMP} is essentially proved as \cite{qW2017}*{Equation~(5.92)},
	where $\rgeo^{-2} \volrat$ was denoted by ``$\varphi$.''
	The only difference is that in the present work, we have the $\uplambda^{-1}$-multiplied terms
	on RHS~\eqref{E:MASSASPECTDECOMP}, which arise when one uses
	equation \eqref{E:RIEMALUNDERLINELBCONTRACTEDDECOMPINVOLVINGVORTANDENT} to algebraically substitute for the term
	$\Riemfour_{A \uLunit \Lunit A}$
	on RHS~\eqref{E:LUNITTRACEUCHI}. We remark that equation \eqref{E:LUNITTRACEUCHI} is relevant for the proof
	since the argument relies on deriving an expression for $\Lunit \mytr_{\gsphere} \underline{\upchi} - \uLunit \mytr_{\gsphere} \upchi$.
	
	To prove \eqref{E:TORSIONDIV}--\eqref{E:TORSIONCURL}, we use
	\eqref{E:RIEMALUNDERLINELBCONTRACTEDDECOMPINVOLVINGVORTANDENT}--\eqref{E:RIEMALUNDERLINELBANTISYMMETRICCONTRACTEDDECOMPINVOLVINGVORTANDENT}
	to substitute for the curvature terms on
	RHSs~\eqref{E:DIVTORSIONNOEULER}--\eqref{E:CURLTORSIONNOEULER},
	and we use \eqref{E:MASSASPECTDECOMP} to substitute for the term $\upmu$ on RHS~\eqref{E:DIVTORSIONNOEULER}.
	Similarly, \eqref{E:MODIFIEDTORSIONDIV}--\eqref{E:MODIFIEDTORSIONCURL}
	follow from
	\eqref{E:DIVTORSIONNOEULER}--\eqref{E:CURLTORSIONNOEULER},
	the definitions of $\widetilde{\upzeta}$ and $\check{\upmu}$,
	and the curvature identities
	\eqref{E:RIEMALUNDERLINELBCONTRACTEDDECOMPINVOLVINGVORTANDENT}--\eqref{E:RIEMALUNDERLINELBANTISYMMETRICCONTRACTEDDECOMPINVOLVINGVORTANDENT}.
	\eqref{E:COMBBINEDMODIFIEDTORSIONMINUSANGMODMASSASPECTDIV}--\eqref{E:COMBBINEDMODIFIEDTORSIONMINUSANGMODMASSASPECTCURL}
	then follow easily from 
	\eqref{E:FURTHERMODOFMASSASPECT},
	\eqref{E:MODIFIEDTORSIONDIV}--\eqref{E:MODIFIEDTORSIONCURL},
	and the fact that $\angdiv$ of an $S_{t,u}$-tangent one-form must have vanishing average value on $S_{t,u}$
	(in the sense of \eqref{E:AVERAGEVALUEOFSCALARFUNCTION}).
	
	To prove \eqref{E:KEYANGMUALGEBRAICDECOMPOSITION}--\eqref{E:ANGCURLLDERIVATIVEANGMU2},
	one commutes equation \eqref{E:FURTHERMODOFMASSASPECT} with $\Lunit$
	and uses the same arguments used in the proof of \cite{qW2017}*{Equation~(6.34)},
	which in particular rely on Lemma~\ref{L:EVOLUTIONEQUATIONFORAVERAGEVALUEONSTU} as well as
	equation \eqref{E:MODIFIEDMASSASPECTEVOLUTIONEQUATION}, derived independently below.
	We clarify the following new feature of the present work: in \cite{qW2017}*{Equation~(6.34)},
	the author derived equations of the form
	$
	\angdiv
		\left\lbrace
			\angprojDarg{\Lunit} \angupmu 
			+
			\frac{1}{2} 
			(\mytr_{\gsphere} \upchi) 
			\angupmu
		\right\rbrace
	= \cdots
	$,
	$
	\angcurl 
		\left\lbrace
			\angprojDarg{\Lunit} \angupmu 
			+
			\frac{1}{2} 
			(\mytr_{\gsphere} \upchi) 
			\angupmu
		\right\rbrace
	= \cdots
	$,
	whereas for mathematical convenience, we have split
	these equations into similar equations for $\angupmu_{(1)}$ and $\angupmu_{(2)}$,
	the point being that later, we will use distinct arguments to control the $\angupmu_{(i)}$.
	The splitting is possible since equation \eqref{E:FURTHERMODOFMASSASPECT} is linear in $\angupmu$.
	
	To prove \eqref{E:CONETIPINITIALCONDITIONSFORANGMUSLASHI1AND2}, we first clarify that
	the $\angupmu_{(i)}$ are solved for by first solving their Hodge systems 
	\eqref{E:ANGDIVLDERIVATIVEANGMU1}--\eqref{E:ANGCURLLDERIVATIVEANGMU2}
	to obtain
	$\angprojDarg{\Lunit} \angupmu_{(i)} 
			+
			\frac{1}{2} 
			(\mytr_{\gsphere} \upchi) 
			\angupmu_{(i)}$
	and then integrating the corresponding inhomogeneous transport equations to obtain $\angupmu_{(i)}$. However,
	there is freedom in how we relate the ``initial conditions'' of $\angupmu$ 
	along the cone-tip axis to those of $\angupmu_{(1)}$ and $\angupmu_{(2)}$, 
	where the only constraint is that \eqref{E:KEYANGMUALGEBRAICDECOMPOSITION} must hold.
	Thus, \eqref{E:CONETIPINITIALCONDITIONSFORANGMUSLASHI1AND2}
	merely represents a choice of vanishing initial conditions for $\angupmu_{(2)}$.
	
	To prove \eqref{E:EVOLUTIONEQUATIONALONGSIGMA0FORCARTESIANCOMPONENTSOFTHETAA},
	we first note that since $\frac{\partial}{\partial w}|_{\Sigma_0} = [a \spherenormal]|_{\Sigma_0}$
	and since $\frac{\partial}{\partial w}$ commutes with $\frac{\partial}{\partial \upomega^A}$,
	we have the following evolution equation for the Cartesian components 
	$\left( \frac{\partial}{\partial \upomega^A} \right)^i$:
	$
	\frac{\partial}{\partial w} \left( \frac{\partial}{\partial \upomega^A} \right)^i
	=
	a \frac{\partial}{\partial \upomega^A} \spherenormal^i
	+
	\spherenormal^i
	\frac{\partial a}{\partial \upomega^A} 
	$.
	From this evolution equation and the second equation in \eqref{E:PROJECTEDNORMALDERIVATIVEOFSPHERENORMAL}, 
	we find, after splitting $\spheresecondfund$ into its trace and trace-free parts,
	that the evolution equation can be expressed in the following schematic form:
	\begin{align} \label{E:SCHEMATICEVOLUTIONFORUDERIVATIVEOFPARTIALOMEGAICARTESIANCOMPONENT}
	\frac{\partial}{\partial u} \left( \frac{\partial}{\partial \upomega^A} \right)^i
	& =
	a 
	\cdot
	\angprojDarg{\frac{\partial}{\partial \upomega^A}} \spherenormal^i
	+
	a
	\cdot
	\gensmoothfunction_{(\vec{\Lunit})} 
	\cdot 
	\partial \vec{\Psi}
	\cdot
\left\lbrace \left(\frac{\partial}{\partial \upomega^A}\right)^j \right\rbrace_{j=1,2,3}
	+
	\gensmoothfunction_{(\vec{\Lunit})} 
	\cdot
	\frac{\partial a}{\partial \upomega^A} 
		\\
&
	= 
	\frac{1}{2} a \mytr_{\gsphere} \spheresecondfund \left( \frac{\partial}{\partial \upomega^A} \right)^i
	+
	a
	\cdot
	\gensmoothfunction_{(\vec{\Lunit})} 
	\cdot 
	(\pmb{\partial} \vec{\Psi},\hat{\spheresecondfund})
	\cdot
\left\lbrace \left(\frac{\partial}{\partial \upomega^A}\right)^j \right\rbrace_{j=1,2,3}
		+
	\gensmoothfunction_{(\vec{\Lunit})} 
	\cdot
	\frac{\partial a}{\partial \upomega^A},
	\notag
	\end{align}
	where the first product on RHS~\eqref{E:SCHEMATICEVOLUTIONFORUDERIVATIVEOFPARTIALOMEGAICARTESIANCOMPONENT} 
	is precisely depicted and the last two are schematically depicted.
	Using \eqref{E:EQNOFINITIALFOLIATION} to substitute for the term $\mytr_{\gsphere} \spheresecondfund$
	and using \eqref{E:CONNECTIONCOEFFICIENT},
	we find that \eqref{E:SCHEMATICEVOLUTIONFORUDERIVATIVEOFPARTIALOMEGAICARTESIANCOMPONENT} 
	can be expressed as
	\begin{align} \label{E:SECONDVERSIONSCHEMATICEVOLUTIONFORUDERIVATIVEOFPARTIALOMEGAICARTESIANCOMPONENT}
	\frac{\partial}{\partial w} \left( \frac{\partial}{\partial \upomega^A} \right)^i
	& = 
	\frac{1}{w} \left( \frac{\partial}{\partial \upomega^A} \right)^i
	+
	a
	\cdot
	\gensmoothfunction_{(\vec{\Lunit})} 
	\cdot 
	(\pmb{\partial} \vec{\Psi},\hat{\upchi})
	\cdot
	\left\lbrace \left(\frac{\partial}{\partial \upomega^A}\right)^j \right\rbrace_{j=1,2,3}
	+
	\gensmoothfunction_{(\vec{\Lunit})} 
	\cdot
	\frac{\partial a}{\partial \upomega^A},
	\end{align}
	where the first product on RHS~\eqref{E:SECONDVERSIONSCHEMATICEVOLUTIONFORUDERIVATIVEOFPARTIALOMEGAICARTESIANCOMPONENT} 
	is precisely depicted and the last two are schematically depicted.
	From \eqref{E:SECONDVERSIONSCHEMATICEVOLUTIONFORUDERIVATIVEOFPARTIALOMEGAICARTESIANCOMPONENT} and the fact that
	$\frac{\partial}{\partial w} \rgeo = 1$ (because $\rgeo|_{\Sigma_0} = w$),
	we easily conclude the desired equation \eqref{E:EVOLUTIONEQUATIONALONGSIGMA0FORCARTESIANCOMPONENTSOFTHETAA}.
	
	To prove \eqref{E:EVOLUTIONEQUATIONALONGINTEGRALCURVESOFLUNITFORCARTESIANCOMPONENTSOFTHETAA},
	we first note that since $\frac{\partial}{\partial t} = \Lunit$ relative to the geometric coordinates,
	and since $\frac{\partial}{\partial t}$ commutes with $\frac{\partial}{\partial \upomega^A}$,
	we have the following evolution equation for the Cartesian components 
	$\left( \frac{\partial}{\partial \upomega^A} \right)^i$:
	$
	\frac{\partial}{\partial t} \left( \frac{\partial}{\partial \upomega^A} \right)^i
	=
	\frac{\partial}{\partial \upomega^A} \Lunit^i
	$.
	From this evolution equation and the first equation in \eqref{E:DALUNIT}, 
	we find, after splitting $\upchi$ into its trace and trace-free parts,
	that the evolution equation can be expressed in the following schematic form:
	\begin{align} \label{E:SCHEMATICEVOLUTIONFORTIMEDERIVATIVEOFPARTIALOMEGAICARTESIANCOMPONENT}
	\frac{\partial}{\partial t} \left( \frac{\partial}{\partial \upomega^A} \right)^i
	& =
	\angprojDarg{\frac{\partial}{\partial \upomega^A}} \Lunit^i
	+
	\gensmoothfunction_{(\vec{\Lunit})} 
	\cdot 
	\partial \vec{\Psi}
	\cdot
	\left\lbrace \left(\frac{\partial}{\partial \upomega^A} \right)^j \right\rbrace_{j=1,2,3}
		\\
	& 
	= 
	\frac{1}{2} \mytr_{\gsphere} \upchi \left( \frac{\partial}{\partial \upomega^A} \right)^i
	+
	\gensmoothfunction_{(\vec{\Lunit})} 
	\cdot 
	(\pmb{\partial} \vec{\Psi},\hat{\upchi})
	\cdot
	\left\lbrace \left(\frac{\partial}{\partial \upomega^A} \right)^j \right\rbrace_{j=1,2,3},
	\notag
	\end{align}
	where the first product on RHS~\eqref{E:SCHEMATICEVOLUTIONFORTIMEDERIVATIVEOFPARTIALOMEGAICARTESIANCOMPONENT}
	is precisely depicted and the second one is schematically depicted.
	Using \eqref{E:MODTRICHISMALL} to substitute for the term $\mytr_{\gsphere} \upchi$,
	we find that \eqref{E:SCHEMATICEVOLUTIONFORTIMEDERIVATIVEOFPARTIALOMEGAICARTESIANCOMPONENT} can be expressed as
	\begin{align} \label{E:SECONDVERSIONSCHEMATICEVOLUTIONFORTIMEDERIVATIVEOFPARTIALOMEGAICARTESIANCOMPONENT}
	\frac{\partial}{\partial t} \left( \frac{\partial}{\partial \upomega^A} \right)^i
	&= 
	\frac{1}{\rgeo} \left( \frac{\partial}{\partial \upomega^A} \right)^i
	+
	\gensmoothfunction_{(\vec{\Lunit})} 
	\cdot 
	(\pmb{\partial} \vec{\Psi},\mytr_{\congsphere} \widetilde{\upchi}^{(Small)},\hat{\upchi})
	\cdot
	\left\lbrace \left(\frac{\partial}{\partial \upomega^A} \right)^j \right\rbrace_{j=1,2,3},
	\end{align}
	where the first product on RHS~\eqref{E:SECONDVERSIONSCHEMATICEVOLUTIONFORTIMEDERIVATIVEOFPARTIALOMEGAICARTESIANCOMPONENT} 
	is precisely depicted and the last one is schematically depicted.
	From \eqref{E:SECONDVERSIONSCHEMATICEVOLUTIONFORTIMEDERIVATIVEOFPARTIALOMEGAICARTESIANCOMPONENT} 
	and the fact that
	$\frac{\partial}{\partial t} \rgeo = 1$,
	we easily conclude the desired equation 
	\eqref{E:EVOLUTIONEQUATIONALONGINTEGRALCURVESOFLUNITFORCARTESIANCOMPONENTSOFTHETAA}.
	
Finally, we provide the lengthy derivation of \eqref{E:MODIFIEDMASSASPECTEVOLUTIONEQUATION}.
Throughout the analysis, we will silently use the following identities,
valid for scalar functions $\varphi$:
\begin{align} \label{E:LLBARCOMMUTATORID}
	\Lunit \uLunit \varphi - \uLunit \Lunit \varphi
	& =
	2 (\underline{\upzeta}_A - \upzeta_A) \angDarg{A} \varphi
	+
	k_{\spherenormal \spherenormal} 
	\uLunit \varphi 
	- 	
	k_{\spherenormal \spherenormal} 
	\Lunit \varphi,
		\\
\Lunit \angLap \varphi 
- 
\angLap \Lunit \varphi
& =
	- 
	\mytr_{\gsphere} \upchi \angLap \varphi
	-
	2 \hat{\upchi}_{AB} \angDsquaredarg{A}{B} \varphi
	- 
	(\angdiv \upchi_A) \angDarg{A} \varphi
	+
	\left\lbrace
		\mytr_{\gsphere} \upchi \underline{\upzeta}_B
		-
		\upchi_{AB} \underline{\upzeta}_A
		-
		\Riemfour_{B C \Lunit C}
	\right\rbrace
		\angDarg{B} \varphi,
		\label{E:LANGLAPCOMMUTATORID}
			\\
	\square_{\gfour(\vec{\Psi})} \varphi
	& =
	-
	\Lunit \uLunit \varphi 
	+
	\angLap \varphi 
	-
	\frac{1}{2}\mytr_{\gsphere}{\upchi}\uLunit \varphi
	-
	\frac{1}{2}\mytr_{\gsphere}{\upchi} \Lunit \varphi
	+
	2\underline{\upzeta}_A \angDarg{A} \varphi
	+ 
	k_{\spherenormal \spherenormal}\uLunit \varphi.
	\label{E:WAVE_OPERATOR_RELATIVE_NULL_FRAME}
\end{align}
The identities \eqref{E:LLBARCOMMUTATORID}--\eqref{E:WAVE_OPERATOR_RELATIVE_NULL_FRAME}
follow from Lemma~\ref{L:CONNECTIONCOEFFICIENTS},
\eqref{E:GINVERSERELATIVETONULLFRAME},
and straightforward
calculations.
We will also often silently use the identity
(see \eqref{E:CONNECTIONCOEFFICIENT})
$
\underline{\upchi}_{AB}
 =  	- \upchi_{AB}
			-
			2 k_{AB}
$
to eliminate $\underline{\upchi}_{AB}$ from various equations.

We now apply $\Lunit$ to the definition \eqref{E:MASSASPECT} and 
use the evolution equations 
\eqref{E:RAYCHAUDHURI},
\eqref{E:LUNITTRACEUCHI},
and
\eqref{E:ULUNITTRACECHI},
and Lemma~\ref{L:CONNECTIONCOEFFICIENTS}
to deduce:
\begin{align}
	\Lunit \upmu 
	+
	\mytr_{\gsphere} \upchi \upmu 
	&=
	-
	\uLunit (\Ricfour_{\Lunit \Lunit})
	-
	\frac{1}{2}
	\Ricfour_{\Lunit \Lunit} \mytr_{\gsphere}\underline{\upchi}
	-
	(\uLunit k_{\spherenormal \spherenormal})
	\mytr_{\gsphere}\upchi 
	-
	(\Lunit \mytr_{\gsphere} \upchi) k_{\spherenormal \spherenormal}
	+ 2 
	(\underline{\upzeta}_A - \upzeta_A) \angDarg{A} \mytr_{\gsphere} \upchi
			\label{E:DERIVATION_MU_CHECK_EQUATION_INTERMEDIATE_1} \\
& \ \ 
	+
	\mytr_{\gsphere}\upchi(\angdiv\underline{\upzeta}
	+
	|\underline{\upzeta}|_{\gsphere}^2
	+
	\frac{1}{2}
	\Riemfour_{A \uLunit \Lunit A})
	+
	\frac{1}{2} 
	\left(\mytr_{\gsphere} \upchi \hat{\upchi}_{AB} \hat{\underline{\upchi}}_{AB}
		+
		\mytr_{\gsphere} \underline{\upchi} |\hat{\upchi}|_{\gsphere}^2 
	\right)
\notag \\
&  \ \  
	-2\hat{\upchi}_{AB}
		\left(
			2\angDarg{A} \upzeta_B
			+
			k_{\spherenormal \spherenormal}\hat{\upchi}_{AB}
			+
			2 \upzeta_A\upzeta_B
			-
			\underline{\hat{\upchi}}_{AC} \hat{\upchi}_{CB}
			+
			\Riemfour_{A \Lunit \uLunit B}
		\right).
	\notag
\end{align}
We will now re-express the factor $\uLunit (k_{\spherenormal \spherenormal})$ that appears on 
RHS~\eqref{E:DERIVATION_MU_CHECK_EQUATION_INTERMEDIATE_1}. To this end,
we set $X=Y:=\spherenormal$ in \eqref{E:SECONDFUNDOFSIGMAT}, 
apply $\Dfour_{\Transport}$ to both sides (so that the LHS of the resulting identity is the scalar function 
$\Transport (k_{\spherenormal \spherenormal})$),
commute $\Dfour_{\Transport}$ with $\Dfour_{\spherenormal}$ on the RHS of the resulting identity using
the definition of curvature,
use the relation $\Transport = \frac{1}{2}(\Lunit + \uLunit)$ (see \eqref{E:NULLVECTORFIELDS}),
use the relation $\Dfour_{\Transport} \Transport= 0$ 
(which is straightforward to derive using that $\gfour(\Transport,\Transport) = -1$ and
the fact that $[\Transport,Z]$ is $\Sigma_t$-tangent 
-- hence $\gfour$-orthogonal to $\Transport$ --
whenever $Z$ is $\Sigma_t$-tangent),
and use Lemma~\ref{L:CONNECTIONCOEFFICIENTS} to derive the following ``second variation'' identity:
\begin{align} \label{E:SECONDVARIATIONIDENTITY}
	\uLunit (k_{\spherenormal \spherenormal}) 
	&= 
	-
	\Lunit (k_{\spherenormal \spherenormal})
	+ 
	2 k_{A \spherenormal} k_{A \spherenormal} 
	- 
	2 (k_{\spherenormal \spherenormal})^2 
	+ 
	4 k_{A\spherenormal} \upzeta_A
	+
	\frac{1}{2} \Riemfour_{\Lunit \uLunit \Lunit \uLunit}.
\end{align} 
Since
\begin{align}
	\Ricfour_{\uLunit \Lunit}
	&=
	\uLunit^{\alpha} \Lunit^{\beta} (\gfour^{-1})^{\mu \nu} \Riemfour_{\alpha\mu\beta\nu},
\notag \\
	\Riemfour_{A \uLunit A \Lunit} 
	&= 
	\uLunit^{\alpha} \Lunit^{\beta} (\gsphere^{-1})^{\mu\nu}\Riemfour_{\alpha\mu\beta\nu},
\notag
\end{align}
we have, in view of \eqref{E:GINVERSERELATIVETONULLFRAME},
\begin{align}
	\Ricfour_{\uLunit \Lunit}
	-
	\Riemfour_{A \uLunit A \Lunit}
	&=
	\uLunit^{\alpha} \Lunit^{\beta} \left[(\gfour^{-1})^{\mu \nu}-(\gsphere^{-1})^{\mu\nu} \right]\Riemfour_{\alpha\mu\beta\nu}
\notag \\
	&=
	-\frac{1}{2}\uLunit^{\alpha} \Lunit^{\beta} \Lunit^\mu \uLunit ^\nu \Riemfour_{\alpha\mu\beta\nu}
	-\frac{1}{2}\uLunit^{\alpha} \Lunit^{\beta} \uLunit^\mu \Lunit^\nu \Riemfour_{\alpha\mu\beta\nu}
\notag \\
	&=
	-\frac{1}{2}(\Riemfour_{\uLunit \Lunit \Lunit \uLunit}
	+\Riemfour_{\uLunit \uLunit \Lunit \Lunit})
	=
	\frac{1}{2}\Riemfour_{\Lunit \uLunit \Lunit \uLunit}.
\notag
\end{align}
From this identity, 
the symmetries of the Riemann curvature tensor,
and
\eqref{E:RIEMALUNDERLINELBCONTRACTEDDECOMPINVOLVINGVORTANDENT}, 
we find that
\begin{align}
	\frac{1}{2} \Riemfour_{\Lunit \uLunit \Lunit \uLunit} 
	&= \Ricfour_{\uLunit \Lunit}-\updelta^{AB}\Riemfour_{A\uLunit B \Lunit}
\notag 
\\
	& = \Ricfour_{\uLunit \Lunit} 
	+ \angdiv \upxi
	+\uplambda^{-1} \gensmoothfunction_{(\vec{\Lunit})} \cdot (\vec{\VortVort},\DivGradEnt)
	+ \gensmoothfunction_{(\vec{\Lunit})}
	\cdot
	(
		\pmb{\partial} \vec{\Psi},\mytr_{\congsphere} \widetilde{\upchi}^{(Small)},\hat{\upchi},\rgeo^{-1}
	)
	\cdot
	\pmb{\partial} \vec{\Psi}.
\notag 
\end{align}
Combining the above calculations,
we can rewrite \eqref{E:DERIVATION_MU_CHECK_EQUATION_INTERMEDIATE_1} as follows:
\begin{align}
	\Lunit \upmu +\mytr_{\gsphere} \upchi\upmu 
	&=
	-\uLunit (\Ricfour_{\Lunit \Lunit})
	-\frac{1}{2}\Ricfour_{\Lunit \Lunit} \mytr_{\gsphere}\underline{\upchi}
\notag \\
& \ \ 
	+ \mytr_{\gsphere}\upchi 
	\left\lbrace
		\Lunit (k_{\spherenormal \spherenormal}) 
		- 
		\Ricfour_{\uLunit \Lunit} 
		-
		\angdiv \upxi
		- 
		2 k_{A \spherenormal} k_{A \spherenormal} 
		+ 
		2 (k_{\spherenormal \spherenormal})^2 
		- 
		4 k_{A\spherenormal} \upzeta_A
	\right\rbrace
\notag \\
& \ \ 
	- \mytr_{\gsphere}\upchi 
	\left\lbrace
		\uplambda^{-1} \gensmoothfunction_{(\vec{\Lunit})} \cdot (\vec{\VortVort},\DivGradEnt)
		+ \gensmoothfunction_{(\vec{\Lunit})}
		\cdot
		(
			\pmb{\partial} \vec{\Psi},\mytr_{\congsphere} \widetilde{\upchi}^{(Small)},\hat{\upchi},\rgeo^{-1}
		)
		\cdot
		\pmb{\partial} \vec{\Psi}
	\right\rbrace
\notag \\
& \ \ 
	- (\Lunit\mytr_{\gsphere}  \upchi) k_{\spherenormal \spherenormal}
	+2(
		\underline{\upzeta}_A
		-\upzeta_A
	)
	\angDarg{A} \mytr_{\gsphere} \upchi
\notag \\
& \ \ 
	+ 
	\mytr_{\gsphere}\upchi\left(\angdiv\underline{\upzeta}
	+
	|\underline{\upzeta}|_{\gsphere}^2
	+
	\frac{1}{2}\Riemfour_{A\uLunit  \Lunit A}\right)
	+
	\frac{1}{2}
	\left(
		\mytr_{\gsphere} \upchi \hat{\upchi}_{AB}\hat{\underline{\upchi}}_{AB}
		+\mytr_{\gsphere}\underline{\upchi} |\hat{\upchi}|_{\gsphere}^2
	\right)
\notag \\
&  \ \  
	-2\hat{\upchi}_{AB}
	\left(
		2
			\angDarg{A} \upzeta_B
			+
			k_{\spherenormal \spherenormal} \hat{\upchi}_{AB}
			+
			2 \upzeta_A \upzeta_B
			-
			\underline{\hat{\upchi}}_{AC} \hat{\upchi}_{CB}
			+
			\Riemfour_{A \Lunit \uLunit B}
	\right).
\notag
\end{align}
With the help of \eqref{E:CONNECTIONCOEFFICIENT}, 
we can rearrange the RHS to rewrite this identity as follows:
\begin{align}
	\Lunit \upmu 
	+
	\mytr_{\gsphere} \upchi\upmu 
	&=
	-
	\uLunit (\Ricfour_{\Lunit \Lunit}) 
	-
	\frac{1}{2}\Ricfour_{\Lunit \Lunit}\mytr_{\gsphere}\underline{\upchi}
	-
	\mytr_{\gsphere}\upchi \Ricfour_{\uLunit \Lunit} 
\label{E:DERIVATION_MU_CHECK_EQUATION_INTERMEDIATE_2} \\
& \ \  
	-
	(\Lunit\mytr_{\gsphere} \upchi) k_{\spherenormal \spherenormal}
	+
	2(
	\underline{\upzeta}_A
	-
	\upzeta_A)
	\angDarg{A} \mytr_{\gsphere} \upchi
	+
	\frac{1}{2}
	\left(
		\mytr_{\gsphere} \upchi \hat{\upchi}_{AB}\hat{\underline{\upchi}}_{AB}
		+
		\mytr_{\gsphere}\underline{\upchi} |\hat{\upchi}|_{\gsphere}^2
	\right)
\notag \\
& \ \ 
	+ \mytr_{\gsphere}\upchi 
	\left\lbrace
		\angdiv\underline{\upzeta}
		-
		\angdiv \upxi
		+
		\frac{1}{2}\Riemfour_{A\uLunit \Lunit A}	
		+
		\Lunit (k_{\spherenormal \spherenormal}) 
		- 
		|\underline{\upzeta}|_{\gsphere}^2
		+ 
		2 (k_{\spherenormal \spherenormal})^2
		+ 
		4 \underline{\upzeta}_A \upzeta_A
	\right\rbrace
\notag \\
&  \ \ 
	-2\hat{\upchi}_{AB}
	\left(
		2\angDarg{A} \upzeta_B
		+
		k_{\spherenormal \spherenormal}\hat{\upchi}_{AB}
		+
		2\upzeta_A\upzeta_B
		-
		\underline{\hat{\upchi}}_{AC} \hat{\upchi}_{CB}
		+
		\Riemfour_{A \Lunit \uLunit B}
	\right)
\notag \\
& \ \ 
	- \mytr_{\gsphere}\upchi
	\left\lbrace
		\uplambda^{-1} \gensmoothfunction_{(\vec{\Lunit})} \cdot (\vec{\VortVort},\DivGradEnt)
		+ 
		\gensmoothfunction_{(\vec{\Lunit})}
		\cdot
		(
			\pmb{\partial} \vec{\Psi},\mytr_{\congsphere} \widetilde{\upchi}^{(Small)},\hat{\upchi},\rgeo^{-1}
		)
		\cdot
		\pmb{\partial} \vec{\Psi}
	\right\rbrace.
	\notag
\end{align}

We will now uncover the structure of the terms on RHS~\eqref{E:DERIVATION_MU_CHECK_EQUATION_INTERMEDIATE_2}.
To help the reader navigate the calculations in the remainder of the proof of \eqref{E:MODIFIEDMASSASPECTEVOLUTIONEQUATION}, 
we also recall that we treat
$\Chfour_{\alpha} := (\gfour^{-1})^{\kappa \lambda} \gfour_{\alpha \beta} \Chfour_{\kappa \ \lambda}^{\ \beta}$
as a one-form under covariant differentiation (as in Lemma~\ref{L:CURVATUREDECOMPOSITIONS}),
that $\Chfour_{\uLunit} := \uLunit^{\alpha} \Chfour_{\alpha}$, and that
$\Chfour_A := e_A^{\alpha} \Chfour_{\alpha}$.
First, invoking \eqref{E:RICLLPERFECTLDERIVATIVE} and 
\eqref{E:RICCILULDECOMPSINVOLVINGVORTANDENT}, we find that
\begin{align}
	&
	-\uLunit (\Ricfour_{\Lunit \Lunit}) 
	-\frac{1}{2}\Ricfour_{\Lunit \Lunit}\mytr_{\gsphere}\underline{\upchi}
	-\mytr_{\gsphere}\upchi 	\Ricfour_{\uLunit \Lunit}  
		\label{E:DERIVATION_MU_CHECK_EQUATION_INTERMEDIATE_3} \\
	 & = 
	-\uLunit \Lunit (\Chfour_{\Lunit})
	-\frac{1}{2}\mytr_{\gsphere}\underline{\upchi} \Lunit (\Chfour_{\Lunit})
	-\frac{1}{2} \mytr_{\gsphere}\upchi \Lunit (\Chfour_{\uLunit})
	-\frac{1}{2} \mytr_{\gsphere}\upchi \uLunit (\Chfour_{\Lunit})
		\notag 
	\\
	& \ \
	- (\uLunit (k_{\spherenormal \spherenormal})) \Chfour_{\Lunit} 
	-  k_{\spherenormal \spherenormal} \uLunit (\Chfour_{\Lunit})
	- \frac{1}{2}\mytr_{\gsphere}\underline{\upchi} k_{\spherenormal \spherenormal} \Chfour_{\Lunit} 
	- \uLunit
	 \left\lbrace
	 	\uplambda^{-1} \gensmoothfunction_{(\vec{\Lunit})} \cdot (\vec{\VortVort},\DivGradEnt)
		+ 
		\gensmoothfunction_{(\vec{\Lunit})}
		\cdot
		\pmb{\partial} \vec{\Psi}
		\cdot
		\pmb{\partial} \vec{\Psi} 
	\right\rbrace
		\notag \\
	& \ \
	-
	\mytr_{\gsphere}\underline{\upchi}		 	 
	\left\lbrace
		\uplambda^{-1} \gensmoothfunction_{(\vec{\Lunit})} \cdot (\vec{\VortVort},\DivGradEnt)
		+ \gensmoothfunction_{(\vec{\Lunit})}
		\cdot
		\pmb{\partial} \vec{\Psi}
		\cdot
		\pmb{\partial} \vec{\Psi} 
	\right\rbrace
	-\mytr_{\gsphere}\upchi 
	\left\lbrace
		\uplambda^{-1} \gensmoothfunction_{(\vec{\Lunit})} \cdot (\vec{\VortVort},\DivGradEnt)
		+ \gensmoothfunction_{(\vec{\Lunit})}
		\cdot
		(\pmb{\partial} \vec{\Psi},\upzeta)
		\cdot
		\pmb{\partial} \vec{\Psi}
	\right\rbrace.
	\notag	
\end{align}
A key observation is that the first, second, and fourth terms on 
RHS~\eqref{E:DERIVATION_MU_CHECK_EQUATION_INTERMEDIATE_3}
produce $\square_{\gfour(\vec{\Psi})} (\Chfour_{\Lunit})$  
(up to lower-order terms) 
when added to\footnote{We recall that $\angLap f := \angDsquaredarg{A}{A} f$; 
see the discussion in Section \ref{SSS:CONNECTIONSANDCURVATURE}.}
$\angLap (\Chfour_{\Lunit})$;
this can be seen from the expression \eqref{E:WAVE_OPERATOR_RELATIVE_NULL_FRAME}.
Next, we apply the operator $2 \angLap$ to \eqref{E:SIGMAEVOLUTION} and use 
\eqref{E:LANGLAPCOMMUTATORID}
to commute it through the operator $\Lunit$.
Also using the identity
$	\angdiv \upchi_A
		= 
		-
		\hat{\upchi}_{AB} 
		k_{B \spherenormal}
		+
		\angDarg{A} \mytr_{\gsphere} \upchi
		+
		\frac{1}{2}		
		k_{A \spherenormal}
		\mytr_{\gsphere} \upchi
		+
		\Riemfour_{B \Lunit B A}
	$
(see \eqref{E:ANGDIVTRACEFREEPARTOFCHI})
and \eqref{E:CONNECTIONCOEFFICIENT},
we obtain the following identity:
\begin{align}
	2\Lunit \angLap \upsigma 
	+
	2
	\mytr_{\gsphere}\upchi \angLap \upsigma 
	& = 
	\angLap (\Chfour_{\Lunit})
	-
	4 \hat{\upchi}_{AB} \angDsquaredarg{A}{B} \upsigma 
	-
	2 (\angDarg{A} \mytr_{\gsphere} \upchi) \angDarg{A} \upsigma
	\label{E:DERIVATION_MU_CHECK_EQUATION_INTERMEDIATE_4}	\\
& \ \ 
	-
	4 \Riemfour_{A B \Lunit B} \angDarg{A} \upsigma 
	+ 
	\mytr_{\gsphere}\upchi
	\underline{\upzeta}_A \angDarg{A} \upsigma
	-
	4 \hat{\upchi}_{AB} \underline{\upzeta}_A \angDarg{B} \upsigma.
	\notag
\end{align}
Adding \eqref{E:DERIVATION_MU_CHECK_EQUATION_INTERMEDIATE_2}	
and \eqref{E:DERIVATION_MU_CHECK_EQUATION_INTERMEDIATE_4}, 
using 
\eqref{E:DERIVATION_MU_CHECK_EQUATION_INTERMEDIATE_3} and
\eqref{E:WAVE_OPERATOR_RELATIVE_NULL_FRAME},
and rearranging the terms,
we deduce that
\begin{align}
	&
	\Lunit (\upmu +	2 \angLap \upsigma) 
	+
	\mytr_{\gsphere} \upchi (\upmu + 	2\angLap \upsigma)
		\label{E:DERIVATION_MU_CHECK_EQUATION_INTERMEDIATE_5} 
		\\
	&=
	\square_{\gfour(\vec{\Psi})} (\Chfour_{\Lunit})
	-
	2\underline{\upzeta}_A \angDarg{A} (\Chfour_{\Lunit})
	-
	2 k_{\spherenormal \spherenormal} \uLunit (\Chfour_{\Lunit})
	- 
	(\uLunit (k_{\spherenormal \spherenormal})) \Chfour_{\Lunit} 
		\notag \\
& \ \  
	-
	4 \hat{\upchi}_{AB} \angDsquaredarg{A}{B} \upsigma
	-
	4\hat{\upchi}_{AB}
	\angDarg{A} \upzeta_B
		\notag \\
& \ \
	-
	2 (\angDarg{A} \mytr_{\gsphere}\upchi) \angDarg{A} \upsigma
	+
	2
	(
	\underline{\upzeta}_A
	-
	\upzeta_A)
	\angDarg{A} \mytr_{\gsphere} \upchi
	\notag \\
& \ \
	-
	4 \Riemfour_{A B \Lunit B} \angDarg{A} \upsigma 
\notag \\
& \ \ 
	-(
		k_{A \spherenormal}\mytr_{\gsphere}\upchi
		-2\hat{\upchi}_{AB}k_{B \spherenormal}
	)\angDarg{A} \upsigma
	-2\upchi_{AB}\underline{\upzeta}_A\angDarg{B}\upsigma
\notag \\	
& \ \
	-
	(\Lunit\mytr_{\gsphere} \upchi) k_{\spherenormal \spherenormal}
		\notag \\
& \ \ 
	+ 
	\mytr_{\gsphere}\upchi 
	\left\lbrace
		\angdiv \underline{\upzeta}
		-
		\angdiv \upxi
	\right\rbrace
	+
	\frac{1}{2}
	\mytr_{\gsphere} \upchi 
	\Riemfour_{A \uLunit \Lunit A}
			\notag \\
& \ \
	+
	\mytr_{\gsphere}\upchi \Lunit (k_{\spherenormal \spherenormal}) 
	-
	\frac{1}{2} 
	\mytr_{\gsphere}\upchi
	\Lunit (\Chfour_{\uLunit})
		\notag \\
& \ \
	+
	\mytr_{\gsphere}\upchi 
	\left\lbrace
		- 
		|\underline{\upzeta}|_{\gsphere}^2 
		+ 
		2 (k_{\spherenormal \spherenormal})^2 
		+
		4 \underline{\upzeta}_A \upzeta_A
	\right\rbrace
\notag \\
& \ \ 
	+ 
	\frac{1}{2} \mytr_{\gsphere} \upchi \hat{\upchi}_{AB}\hat{\underline{\upchi}}_{AB}
	+
	\mytr_{\gsphere} \underline{\upchi} |\hat{\upchi}|_{\gsphere}^2
	-
	\frac{1}{2}\mytr_{\gsphere} \underline{\upchi} k_{\spherenormal \spherenormal} \Chfour_{\Lunit} 
\notag \\
&  \ \ 
	- 
	2 \hat{\upchi}_{AB} \Riemfour_{A\Lunit \uLunit B}
	-
	2\hat{\upchi}_{AB}
	\left\lbrace
		k_{\spherenormal \spherenormal} \hat{\upchi}_{AB}
		+	
		2\upzeta_A\upzeta_B
	\right\rbrace
\notag 
	\\
& \ \
	- 
	\mytr_{\gsphere}\upchi \gensmoothfunction_{(\vec{\Lunit})}
	\cdot
	(
		\pmb{\partial} \vec{\Psi},\mytr_{\congsphere} \widetilde{\upchi}^{(Small)},\hat{\upchi},\rgeo^{-1}
	)
	\cdot
	\pmb{\partial} \vec{\Psi}
\notag \\
& \ \ 
	-\uLunit 
	\left\lbrace
		\uplambda^{-1} \gensmoothfunction_{(\vec{\Lunit})} \cdot (\vec{\VortVort},\DivGradEnt)
		+ \gensmoothfunction_{(\vec{\Lunit})}
		\cdot
		\pmb{\partial} \vec{\Psi}
		\cdot
		\pmb{\partial} \vec{\Psi} 
	\right\rbrace
\notag \\
& \ \ 
	-\mytr_{\gsphere}\underline{\upchi}		 	 
	 \left\lbrace
	 	\uplambda^{-1} \gensmoothfunction_{(\vec{\Lunit})} \cdot (\vec{\VortVort},\DivGradEnt)
		+ \gensmoothfunction_{(\vec{\Lunit})}
		\cdot
		\pmb{\partial} \vec{\Psi}
		\cdot
		\pmb{\partial} \vec{\Psi}
	\right\rbrace
\notag \\
& \ \ 
	-\mytr_{\gsphere}\upchi 	
	\left\lbrace
		\uplambda^{-1} \gensmoothfunction_{(\vec{\Lunit})} \cdot (\vec{\VortVort},\DivGradEnt)
		+ \gensmoothfunction_{(\vec{\Lunit})}
		\cdot
		(
			\pmb{\partial} \vec{\Psi},\upzeta
		)
		\cdot
		\pmb{\partial} \vec{\Psi}
	\right\rbrace.
	\notag
\end{align}
We next manipulate \eqref{E:DERIVATION_MU_CHECK_EQUATION_INTERMEDIATE_5} as follows:
we move the terms $	\mytr_{\gsphere}\upchi \Lunit (k_{\spherenormal \spherenormal}) 
	-
	\frac{1}{2} 
	\mytr_{\gsphere}\upchi
	\Lunit (\Chfour_{\uLunit})$
from the RHS to the LHS; subtract 
$(\Lunit \mytr_{\gsphere}\upchi)  k_{\spherenormal \spherenormal}$ from both sides; 
add $\frac{1}{2} (\Lunit  \mytr_{\gsphere} \upchi) \Chfour_{\uLunit}$ to both sides; 
and finally, add 
$\frac{1}{2} (\mytr_{\gsphere}\upchi)^2 \Chfour_{\uLunit}
-
(\mytr_{\gsphere}\upchi)^2 k_{\spherenormal \spherenormal}$ to both sides. 
After these steps, 
in view of definition \eqref{E:MODMASSASPECT}, 
we see that the LHS becomes
\begin{align}
		\Lunit \check{\upmu} +\mytr_{\gsphere} \upchi \check{\upmu}.
		\notag
\end{align}
With the help of definition \eqref{E:MODTORSION},
we now rearrange some other terms on RHS~\eqref{E:DERIVATION_MU_CHECK_EQUATION_INTERMEDIATE_5}
as follows:
\begin{align}
		-4
		\hat{\upchi}_{AB} \angDsquaredarg{A}{B} \upsigma 
		-
		4
		\hat{\upchi}_{AB}\angDarg{A} \upzeta_B 
		& =
		-4\hat{\upchi}_{AB} \angDarg{A} \widetilde{\upzeta}_B,
		\notag
			\\
		-2
		(\angDarg{A} \mytr_{\gsphere} \upchi) \angDarg{A}\upsigma
		+2
		(\underline{\upzeta}_A	
		-\upzeta_A)
		\angDarg{A} \mytr_{\gsphere} \upchi 
		& 
		= 
			2(\underline{\upzeta}_A	- \widetilde{\upzeta}_A)\angDarg{A} \mytr_{\gsphere} \upchi.
			\notag
\end{align}

Combining the above calculations, 
we have thus far obtained the following equation:
\begin{align}
	&
	\Lunit \check{\upmu} 
	+
\mytr_{\gsphere} \upchi \check{\upmu}
		\label{E:DERIVATION_MU_CHECK_EQUATION_INTERMEDIATE_6} 
			\\
	&=
	\square_{\gfour(\vec{\Psi})} (\Chfour_{\Lunit})
	-
	2 \underline{\upzeta}_A \angDarg{A} (\Chfour_{\Lunit})
	-
	2 k_{\spherenormal \spherenormal} \uLunit (\Chfour_{\Lunit})
	- 
	(\uLunit (k_{\spherenormal \spherenormal})) \Chfour_{\Lunit} 
		\notag \\
& \ \
	+
	\frac{1}{2} (\mytr_{\gsphere}\upchi)^2 \Chfour_{\uLunit}
	- 
	(\mytr_{\gsphere} \upchi)^2 k_{\spherenormal \spherenormal}	
		\notag \\
& \ \  
	-4\hat{\upchi}_{AB} \angDarg{A} \widetilde{\upzeta}_B
	+
	2(\underline{\upzeta}_A	- \widetilde{\upzeta}_A)\angDarg{A} \mytr_{\gsphere} \upchi
		\notag \\
& \ \
	-
	4 \Riemfour_{A B \Lunit B} \angDarg{A} \upsigma 
\notag \\
& \ \ 
	-(
		k_{A \spherenormal}\mytr_{\gsphere}\upchi
		-2\hat{\upchi}_{AB}k_{B \spherenormal}
	)\angDarg{A} \upsigma
	-2\upchi_{AB}\underline{\upzeta}_A\angDarg{B}\upsigma
\notag \\	
& \ \
	-
	(2 \Lunit\mytr_{\gsphere} \upchi) k_{\spherenormal \spherenormal}
	+
	\frac{1}{2} (\Lunit\mytr_{\gsphere} \upchi) 
	\Chfour_{\uLunit}
		\notag \\
& \ \ 
	+ 
	\mytr_{\gsphere}\upchi 
	\left\lbrace
		\angdiv \underline{\upzeta}
		-
		\angdiv \upxi
	\right\rbrace
	+
	\frac{1}{2}
	\mytr_{\gsphere} \upchi 
	\Riemfour_{A \uLunit \Lunit A}
			\notag \\
& \ \
	+
	\mytr_{\gsphere}\upchi 
	\left\lbrace
		- 
		|\underline{\upzeta}|_{\gsphere}^2 
		+ 
		2 (k_{\spherenormal \spherenormal})^2 
		+
		4 \underline{\upzeta}_A \upzeta_A
	\right\rbrace
\notag \\
& \ \ 
	+ 
	\frac{1}{2} \mytr_{\gsphere} \upchi \hat{\upchi}_{AB}\hat{\underline{\upchi}}_{AB}
	+
	\mytr_{\gsphere} \underline{\upchi} |\hat{\upchi}|_{\gsphere}^2
	-
	\frac{1}{2}\mytr_{\gsphere} \underline{\upchi} k_{\spherenormal \spherenormal} \Chfour_{\Lunit} 
\notag \\
&  \ \ 
	- 
	2 \hat{\upchi}_{AB} \Riemfour_{A\Lunit \uLunit B}
	-
	2\hat{\upchi}_{AB}
	\left\lbrace
		k_{\spherenormal \spherenormal} \hat{\upchi}_{AB}
		+	
		2\upzeta_A\upzeta_B
	\right\rbrace
\notag 
	\\
& \ \
	- 
	\mytr_{\gsphere}\upchi \gensmoothfunction_{(\vec{\Lunit})}
	\cdot
	(
		\pmb{\partial} \vec{\Psi},\mytr_{\congsphere} \widetilde{\upchi}^{(Small)},\hat{\upchi},\rgeo^{-1}
	)
	\cdot
	\pmb{\partial} \vec{\Psi}
\notag \\
& \ \ 
	-\uLunit 
	\left\lbrace
		\uplambda^{-1} \gensmoothfunction_{(\vec{\Lunit})} \cdot (\vec{\VortVort},\DivGradEnt)
		+ \gensmoothfunction_{(\vec{\Lunit})}
		\cdot
		\pmb{\partial} \vec{\Psi}
		\cdot
		\pmb{\partial} \vec{\Psi} 
	\right\rbrace
\notag \\
& \ \ 
	-\mytr_{\gsphere}\underline{\upchi}		 	 
	 \left\lbrace
	 	\uplambda^{-1} \gensmoothfunction_{(\vec{\Lunit})} \cdot (\vec{\VortVort},\DivGradEnt)
		+ \gensmoothfunction_{(\vec{\Lunit})}
		\cdot
		\pmb{\partial} \vec{\Psi}
		\cdot
		\pmb{\partial} \vec{\Psi}
	\right\rbrace
\notag \\
& \ \ 
	-\mytr_{\gsphere}\upchi 	
	\left\lbrace
		\uplambda^{-1} \gensmoothfunction_{(\vec{\Lunit})} \cdot (\vec{\VortVort},\DivGradEnt)
		+ \gensmoothfunction_{(\vec{\Lunit})}
		\cdot
		(
			\pmb{\partial} \vec{\Psi},\upzeta
		)
		\cdot
		\pmb{\partial} \vec{\Psi}
	\right\rbrace.
	\notag
\end{align}

We proceed to expand the term $\square_{\gfour(\vec{\Psi})} (\Chfour_{\Lunit})$ 
on RHS~\eqref{E:DERIVATION_MU_CHECK_EQUATION_INTERMEDIATE_6}, where we recall
from Def.\,\ref{D:CONFORMALSTUFF} that $\Chfour_{\Lunit} := \Lunit^{\alpha} \Chfour_{\alpha}$.
We therefore compute that\footnote{Since
$\square_{\gfour(\vec{\Psi})}  (\Chfour_{\Lunit}) 
= 
\square_{\gfour(\vec{\Psi})} (\Lunit^{\alpha} \Chfour_\alpha)
= (\gfour^{-1})^{\alpha \beta} \Dfour_{\alpha} \Dfour_{\beta} (\Lunit^{\alpha} \Chfour_\alpha)
$,
to obtain \eqref{E:DERIVATION_MU_CHECK_EQUATION_BOX_GAMMA_L_PRODUCT_RULE}, we have expanded this expression using
the Leibniz rule, where we treat $\Lunit^{\alpha}$ as a vectorfield under covariant differentiation and
we treat $\Chfour_\alpha$ as a one-form under covariant differentiation.}
\begin{align}
	\square_{\gfour(\vec{\Psi})} (\Chfour_{\Lunit})
	&=
	\Chfour_{\alpha} \square_{\gfour(\vec{\Psi})} \Lunit^{\alpha} 
	+
	\Lunit^{\alpha}  \square_{\gfour(\vec{\Psi})} \Chfour_{\alpha}
	+
	2 (\Dfour_\mu \Lunit^{\alpha}) \Dfour^\mu \Chfour_{\alpha}.
	\label{E:DERIVATION_MU_CHECK_EQUATION_BOX_GAMMA_L_PRODUCT_RULE}
\end{align}
We first handle the term $\square_{\gfour(\vec{\Psi})} \Lunit^{\alpha}$ in
\eqref{E:DERIVATION_MU_CHECK_EQUATION_BOX_GAMMA_L_PRODUCT_RULE}. 
Using the decomposition of $\gfour^{-1}$ relative to a null frame
(i.e., \eqref{E:GINVERSERELATIVETONULLFRAME})
and Lemma~\ref{L:CONNECTIONCOEFFICIENTS}, 
we compute that
\begin{align} \label{E:JAREDSTEP}
	\square_{\gfour(\vec{\Psi})} \Lunit^{\alpha} 
	& =
		- 
		\Dfour_{\uLunit} (\Dfour_\Lunit \Lunit^{\alpha})
		+ 
		\Dfour_{e_A} (\Dfour_{e_A} \Lunit^{\alpha})
		+ 
		\Dfour_{\Dfour_{\uLunit} \Lunit} \Lunit^{\alpha}
		- 
		\Dfour_{\Dfour_{e_A} e_A} \Lunit^{\alpha}
		-
		\frac{1}{2} \Riemfour^\alpha_{\  \Lunit \Lunit \uLunit}
			\\
	& = \angdiv \upchi_A e_A^{\alpha}
			+
			(\angdiv \underline{\upzeta}) \Lunit^{\alpha}
			+
			(\uLunit (k_{\spherenormal \spherenormal})) \Lunit^{\alpha}
			+ 
			\frac{1}{2} |\upchi|_{\gsphere}^2 \uLunit^{\alpha}
			+
			\frac{1}{2} \upchi_{AB} \underline{\upchi}_{AB} \Lunit^{\alpha}
				\notag \\
	& \ \ 
			-
			\mytr_{\gsphere} \upchi \upzeta_A e_A^{\alpha}
			-
			\frac{1}{2} \mytr_{\gsphere} \upchi k_{\spherenormal \spherenormal} \Lunit^{\alpha}
			+
			\frac{1}{2} \mytr_{\gsphere} \underline{\upchi} k_{\spherenormal \spherenormal}	\Lunit^{\alpha}
			+
			2 \upzeta_A \upchi_{AB} e_B^{\alpha}
			+
			2 k_{\spherenormal \spherenormal} \upzeta_A e_A^{\alpha}
			+ 
			\underline{\upzeta}_A \upchi_{AB} e_B^{\alpha}
			+
			|\underline{\upzeta}|_{\gsphere}^2 \Lunit^{\alpha}
				\notag \\
		& \ \
			-
			\frac{1}{2} \Riemfour^{\alpha}_{\  \Lunit \Lunit \uLunit}.
		\notag
\end{align}
Contracting \eqref{E:JAREDSTEP} against $\Chfour_{\alpha}$, 
we find that
\begin{align} \label{E:DERIVATION_MU_CHECK_EQUATION_BOX_L}
	\Chfour_{\alpha} \square_{\gfour(\vec{\Psi})} \Lunit^{\alpha} 
	&=
		(\angdiv \upchi_A) \Chfour_A
		+
		(\angdiv \underline{\upzeta}) \Chfour_{\Lunit}
		+
		(\uLunit (k_{\spherenormal \spherenormal})) \Chfour_{\Lunit}
		+ 
		\frac{1}{2} |\upchi|_{\gsphere}^2 \Chfour_{\uLunit}
		+
		\frac{1}{2} \upchi_{AB} \underline{\upchi}_{AB} \Chfour_{\Lunit}
				\\
	& \ \ 
			-
			\mytr_{\gsphere} \upchi \upzeta_A \Chfour_A
			-
			\frac{1}{2} \mytr_{\gsphere} \upchi k_{\spherenormal \spherenormal} \Chfour_{\Lunit}
			+
			\frac{1}{2} \mytr_{\gsphere} \underline{\upchi} k_{\spherenormal \spherenormal}	\Chfour_{\Lunit}
			+
			2 \upzeta_A \upchi_{AB} \Chfour_B
			+
			2 k_{\spherenormal \spherenormal} \upzeta_A \Chfour_A
			+ 
			\underline{\upzeta}_A \upchi_{AB} \Chfour_B
			+
			|\underline{\upzeta}|_{\gsphere}^2 \Chfour_{\Lunit}
				\notag \\
		& \ \
			-
			\frac{1}{2} \Chfour_{\alpha} \Riemfour^{\alpha}_{\  \Lunit \Lunit \uLunit}.
			\notag
\end{align}
Next, we again use \eqref{E:GINVERSERELATIVETONULLFRAME} and Lemma~\ref{L:CONNECTIONCOEFFICIENTS} 
to compute the last product in \eqref{E:DERIVATION_MU_CHECK_EQUATION_BOX_GAMMA_L_PRODUCT_RULE}:
\begin{align} \label{E:DERIVATION_MU_CHECK_EQUATION_CROSS_TERM_BOX}
	2 (\Dfour_\mu \Lunit^{\alpha}) \Dfour^\mu \Chfour_{\alpha}
	&=
	k_{\spherenormal \spherenormal} \Lunit^{\alpha} \Dfour_\Lunit \Chfour_\alpha
	- 
	2\upzeta_A e_A^\alpha \Dfour_\Lunit \Chfour_\alpha
	- 
	k_{\spherenormal \spherenormal} \Lunit^{\alpha} \Dfour_{\uLunit} \Chfour_\alpha
	-
	2 k_{A \spherenormal} \Lunit^{\alpha} \Dfour_A \Chfour_{\alpha}
	+
	2 \upchi_{AB} e_B^\alpha \Dfour_A \Chfour_{\alpha}.
\end{align}

Next we use the decomposition $\upchi_{AB} = \hat{\upchi}_{AB} + \frac{1}{2} \mytr_{\gsphere}\upchi \gsphere_{AB}$,
\eqref{E:MODTRICHISMALL}, 
and \eqref{E:ANGDIVTRACEFREEPARTOFCHI}
to rewrite the first product on RHS~\eqref{E:DERIVATION_MU_CHECK_EQUATION_BOX_L}
as follows:
\begin{align}
	(\angdiv \upchi_A) \Chfour_A
	&
	=	
	(\angDarg{A} \mytr_{\congsphere} \widetilde{\upchi}^{(Small)}) \Chfour_{A} 
	-
	(\angDarg{A} (\Chfour_{\Lunit})) \Chfour_{A} 
	-
	\hat{\upchi}_{AB} k_{B \spherenormal} \Chfour_{A}
	\label{E:DERIVATION_MU_CHECK_EQUATION_INTERMEDIATE_9} \\
& \ \
	+
	\frac{1}{2} \mytr_{\gsphere}\upchi k_{A \spherenormal} \Chfour_{A} 
	+ 
	\Riemfour_{A\Lunit A B} \Chfour_{B}.
	\notag
\end{align}

Moreover, we use the decomposition $\upchi_{AB} = \hat{\upchi}_{AB} + \frac{1}{2} \mytr_{\gsphere}\upchi \gsphere_{AB}$ 
and \eqref{E:MODTRICHISMALL} to rewrite the last product 
on RHS~\eqref{E:DERIVATION_MU_CHECK_EQUATION_CROSS_TERM_BOX} as follows,
where $\upxi$ denotes the $\gfour$-orthogonal projection onto $S_{t,u}$
of the one-form with Cartesian components $2\Chfour_{\alpha}$:
\begin{align}
	2 \upchi_{AB} e_B^\alpha \Dfour_A \Chfour_{\alpha}
	&=
	2 \hat{\upchi}_{AB} e_B^\alpha \angprojD_A \Chfour_{\alpha} 
	+\mytr_{\gsphere}\upchi  e_A^\alpha \angprojD_A \Chfour_{\alpha}
	\label{E:DERIVATION_MU_CHECK_EQUATION_INTERMEDIATE_8} \\
	& =
	\frac{2}{\rgeo} e_A^\alpha \angprojD_A \Chfour_{\alpha}
	+ \mytr_{\congsphere} \widetilde{\upchi}^{(Small)} e_A^\alpha \angprojD_A \Chfour_{\alpha}
	- \Chfour_{\Lunit}  e_A^\alpha \angprojD_A \Chfour_{\alpha}
	+2\hat{\upchi}_{AB}  e_B^\alpha \angprojD_A \Chfour_{\alpha}
\notag \\
	&=
	\frac{1}{\rgeo} \angdiv \upxi
	+ \mytr_{\congsphere} \widetilde{\upchi}^{(Small)} e_A^\alpha \angprojD_A \Chfour_{\alpha}
	- \Chfour_{\Lunit}  e_A^\alpha \angprojD_A \Chfour_{\alpha}
	+2\hat{\upchi}_{AB}  e_B^\alpha \angprojD_A \Chfour_{\alpha}.
	\notag
\end{align}

Moreover, using \eqref{E:RAYCHAUDHURI}, we derive the following
identity for the two $\Lunit \mytr_{\gsphere} \upchi$-involving
products on RHS~\eqref{E:DERIVATION_MU_CHECK_EQUATION_INTERMEDIATE_6}:
\begin{align} \label{E:JAREDRAY1}
	-2 (\Lunit \mytr_{\gsphere} \upchi) k_{\spherenormal \spherenormal}
	& = 
		(\mytr_{\gsphere} \upchi)^2
		k_{\spherenormal \spherenormal}
		+
		2 |\hat{\upchi}|_{\gsphere}^2 
		k_{\spherenormal \spherenormal}
		+ 2 
		\mytr_{\gsphere} \upchi
		(k_{\spherenormal \spherenormal})^2 
		+ 2
		\Ricfour_{\Lunit \Lunit} k_{\spherenormal \spherenormal},
			\\
	\frac{1}{2} (\Lunit \mytr_{\gsphere} \upchi) \Chfour_{\uLunit}
	& = 
		- 
		\frac{1}{4} (\mytr_{\gsphere} \upchi)^2 \Chfour_{\uLunit}
		-
		\frac{1}{2} |\hat{\upchi}|_{\gsphere}^2 \Chfour_{\uLunit}
		-
		\frac{1}{2} 
			\mytr_{\gsphere} \upchi
			k_{\spherenormal \spherenormal} 
			\Chfour_{\uLunit}
			-
			\frac{1}{2} \Ricfour_{\Lunit \Lunit} \Chfour_{\uLunit}.
		\label{E:JAREDRAY2}
\end{align}

We now use \eqref{E:JAREDRAY1}--\eqref{E:JAREDRAY2} to substitute for the relevant products on 
RHS~\eqref{E:DERIVATION_MU_CHECK_EQUATION_INTERMEDIATE_6},
we use \eqref{E:DERIVATION_MU_CHECK_EQUATION_BOX_GAMMA_L_PRODUCT_RULE}
to substitute for the first term
$\square_{\gfour(\vec{\Psi})} (\Chfour_{\Lunit})$ on RHS~\eqref{E:DERIVATION_MU_CHECK_EQUATION_INTERMEDIATE_6},
we use
\eqref{E:DERIVATION_MU_CHECK_EQUATION_BOX_L}--\eqref{E:DERIVATION_MU_CHECK_EQUATION_CROSS_TERM_BOX}
to substitute for the first and third products
on RHS~\eqref{E:DERIVATION_MU_CHECK_EQUATION_BOX_GAMMA_L_PRODUCT_RULE}
(specifically,
$
\Chfour_{\alpha} \square_{\gfour(\vec{\Psi})} \Lunit^{\alpha}$
and
$
	2 (\Dfour_\mu \Lunit^{\alpha}) \Dfour^\mu \Chfour_{\alpha}
$),
and we use
\eqref{E:DERIVATION_MU_CHECK_EQUATION_INTERMEDIATE_9}--\eqref{E:DERIVATION_MU_CHECK_EQUATION_INTERMEDIATE_8}
to substitute for the relevant products on
RHSs~\eqref{E:DERIVATION_MU_CHECK_EQUATION_BOX_L}--\eqref{E:DERIVATION_MU_CHECK_EQUATION_CROSS_TERM_BOX}.
Also using \eqref{E:CONNECTIONCOEFFICIENT}, in total, we compute that the following equation holds:
\begin{align}
	\Lunit \check{\upmu} 
	+
	\mytr_{\gsphere} \upchi \check{\upmu} 
	&
	=
	\Lunit^{\alpha} \square_{\gfour(\vec{\Psi})} \Chfour_{\alpha}
	-
	4 \Riemfour_{A B \Lunit B} \angDarg{A} \upsigma
	-
	2 \hat{\upchi}_{AB}\Riemfour_{A \Lunit \uLunit B} 
	+ 
	\Riemfour_{A \Lunit A B} \Chfour_{B} 
	-
	\frac{1}{2} \Chfour_{\alpha} \Riemfour^\alpha_{\  \Lunit \Lunit \uLunit}
		\label{E:DERIVATION_MU_CHECK_EQUATION_FINAL} \\
& \ \ 
	+ 
	\frac{1}{2}\mytr_{\gsphere}\upchi \Riemfour_{A\uLunit  \Lunit A}	
	+
	2
	\Ricfour_{\Lunit \Lunit} k_{\spherenormal \spherenormal}
	-
	\frac{1}{2} \Ricfour_{\Lunit \Lunit} \Chfour_{\uLunit}
	+	
	\check{\mathsf{Err}},
	\notag
\end{align}
where
\begin{align}
	\check{\mathsf{Err}}
	&=
	\frac{1}{\rgeo} \angdiv \upxi
	+ 
	\mytr_{\gsphere}\upchi 
	\left\lbrace
		\angdiv \underline{\upzeta}
		-
		\angdiv \upxi
	\right\rbrace
	+
	\frac{1}{\rgeo^2} \upxi
	+
	\frac{1}{4} 
	(\mytr_{\gsphere}\upchi)^2
	\Chfour_{\uLunit}
	+ 
	\frac{1}{2} |\upchi|_{\gsphere}^2 \Chfour_{\uLunit}
	+
	\frac{1}{2} \upchi_{AB} \underline{\upchi}_{AB} \Chfour_{\Lunit}
		\label{E:DERIVATION_MU_CHECK_EQUATION_REMAINDER} 
		\\
& \ \
	+	
	(\angDarg{A} \mytr_{\congsphere} \widetilde{\upchi}^{(Small)}) \Chfour_{A} 
	+
	(\angdiv \underline{\upzeta}) \Chfour_{\Lunit}
		\notag \\
	& \ \
	+
	k_{\spherenormal \spherenormal} \Lunit^{\alpha} \Dfour_\Lunit \Chfour_\alpha
	- 
	2\upzeta_A e_A^\alpha \Dfour_\Lunit \Chfour_\alpha
	- 
	k_{\spherenormal \spherenormal} \Lunit^{\alpha} \Dfour_{\uLunit} \Chfour_\alpha
	-
	2 k_{\spherenormal \spherenormal} \uLunit (\Chfour_{\Lunit})
		\notag \\
& \ \
	-
	(\angDarg{A} (\Chfour_{\Lunit})) \Chfour_{A} 
	-
	2 k_{A \spherenormal} \Lunit^{\alpha} \Dfour_A \Chfour_{\alpha}
	+ 
	\mytr_{\congsphere} \widetilde{\upchi}^{(Small)} e_A^\alpha \angprojD_A \Chfour_{\alpha}
	- 
	\Chfour_{\Lunit}  e_A^\alpha \angprojD_A \Chfour_{\alpha}
	+
	2\hat{\upchi}_{AB}  e_B^\alpha \angprojD_A \Chfour_{\alpha}
	-
	2 \underline{\upzeta}_A \angDarg{A} (\Chfour_{\Lunit})
		\notag \\
& \ \  
	-
	4 \hat{\upchi}_{AB} \angDarg{A} \widetilde{\upzeta}_B
	+
	2(\underline{\upzeta}_A	- \widetilde{\upzeta}_A)\angDarg{A} \mytr_{\gsphere} \upchi
	-
	(
		k_{A \spherenormal}\mytr_{\gsphere}\upchi
		-
		2\hat{\upchi}_{AB}k_{B \spherenormal}
	)
	\angDarg{A} \upsigma
	-
	2\upchi_{AB}\underline{\upzeta}_A \angDarg{B}\upsigma
\notag \\	
& \ \ 
			+
			2 \upchi_{AB} \upzeta_A  \Chfour_B
			+ 
			\upchi_{AB} \underline{\upzeta}_A \Chfour_B
				\notag \\
& \ \
	+
	\mytr_{\gsphere}\upchi 
	\left\lbrace
		- 
		|\underline{\upzeta}|_{\gsphere}^2 
		+ 
		4 (k_{\spherenormal \spherenormal})^2 
		+
		4 \underline{\upzeta}_A \upzeta_A
		-
		\frac{1}{2} k_{\spherenormal \spherenormal} \Chfour_{\Lunit}
		+
		\frac{1}{2} k_{A \spherenormal} \Chfour_{A} 
		-
		 \upzeta_A \Chfour_A
		-
		\frac{1}{2} 
		k_{\spherenormal \spherenormal} 
		\Chfour_{\uLunit}
	\right\rbrace
\notag \\
& \ \ 
	+ 
	\frac{1}{2} \mytr_{\gsphere} \upchi \hat{\upchi}_{AB}\hat{\underline{\upchi}}_{AB}
	+
	\mytr_{\gsphere} \underline{\upchi} |\hat{\upchi}|_{\gsphere}^2
\notag \\
&  \ \ 
	+
	\hat{\upchi}_{AB}
	\left\lbrace
		-
		4 \upzeta_A\upzeta_B
		+
		\underline{\upzeta}_A \Chfour_{B}
	\right\rbrace
	-
	\frac{1}{2} 
	|\hat{\upchi}|_{\gsphere}^2 \Chfour_{\uLunit}
\notag 
	\\
& \ \
			+
			2 k_{\spherenormal \spherenormal} \upzeta_A \Chfour_A
			+
			|\underline{\upzeta}|_{\gsphere}^2 \Chfour_{\Lunit}
			\notag \\
& \ \
	- 
	\mytr_{\gsphere} \upchi \gensmoothfunction_{(\vec{\Lunit})}
	\cdot
	(
		\pmb{\partial} \vec{\Psi},\mytr_{\congsphere} \widetilde{\upchi}^{(Small)},\hat{\upchi},\rgeo^{-1}
	)
	\cdot
	\pmb{\partial} \vec{\Psi}
\notag \\
& \ \ 
	-\uLunit 
	\left\lbrace
		\uplambda^{-1} \gensmoothfunction_{(\vec{\Lunit})} \cdot (\vec{\VortVort},\DivGradEnt)
		+ \gensmoothfunction_{(\vec{\Lunit})}
		\cdot
		\pmb{\partial} \vec{\Psi}
		\cdot
		\pmb{\partial} \vec{\Psi} 
	\right\rbrace
\notag \\
& \ \ 
	-\mytr_{\gsphere}\underline{\upchi}		 	 
	 \left\lbrace
	 	\uplambda^{-1} \gensmoothfunction_{(\vec{\Lunit})} \cdot (\vec{\VortVort},\DivGradEnt)
		+ \gensmoothfunction_{(\vec{\Lunit})}
		\cdot
		\pmb{\partial} \vec{\Psi}
		\cdot
		\pmb{\partial} \vec{\Psi}
	\right\rbrace
\notag \\
& \ \ 
	-\mytr_{\gsphere}\upchi 	
	\left\lbrace
		\uplambda^{-1} \gensmoothfunction_{(\vec{\Lunit})} \cdot (\vec{\VortVort},\DivGradEnt)
		+ \gensmoothfunction_{(\vec{\Lunit})}
		\cdot
		(
			\pmb{\partial} \vec{\Psi},\upzeta
		)
		\cdot
		\pmb{\partial} \vec{\Psi}
	\right\rbrace.
	\notag
\end{align}

With the help of 
the decomposition 
$\upchi_{AB} = \hat{\upchi}_{AB} + \frac{1}{2} \mytr_{\gsphere}\upchi \gsphere_{AB}$,
definition \eqref{E:MODTRICHISMALL}
(which implies that schematically, we have
$ \mytr_{\gsphere} \upchi
=
\gensmoothfunction_{(\vec{\Lunit})}
\cdot
		(
			\pmb{\partial} \vec{\Psi},\mytr_{\congsphere} \widetilde{\upchi}^{(Small)},\rgeo^{-1}
		)
$), 
definition \eqref{E:MODTORSION},
the identity $\upchi_{AB} + \underline{\upchi}_{AB} = -2 k_{AB}$
(see \eqref{E:CONNECTIONCOEFFICIENT}),
and Lemma~\ref{L:ANGULARDERIVATIVESOFSOMESCALARFUNCTIONS},
we verify by direct inspection that all terms on RHS~\eqref{E:DERIVATION_MU_CHECK_EQUATION_REMAINDER} 
can be accommodated into 
RHS~\eqref{E:SECONDINHOMTERMMODIFIEDMASSASPECTEVOLUTIONEQUATION},
aside from the terms on the first line of RHS~\eqref{E:DERIVATION_MU_CHECK_EQUATION_REMAINDER},
which split into terms of type 
RHS~\eqref{E:FIRSTINHOMTERMMODIFIEDMASSASPECTEVOLUTIONEQUATION}
and of
type RHS~\eqref{E:SECONDINHOMTERMMODIFIEDMASSASPECTEVOLUTIONEQUATION}.

To finish the proof of \eqref{E:MODIFIEDMASSASPECTEVOLUTIONEQUATION}, it remains
only for us to verify that the remaining terms on RHS~\eqref{E:DERIVATION_MU_CHECK_EQUATION_FINAL}
have the form of terms on 
either RHS~\eqref{E:FIRSTINHOMTERMMODIFIEDMASSASPECTEVOLUTIONEQUATION} or
RHS~\eqref{E:SECONDINHOMTERMMODIFIEDMASSASPECTEVOLUTIONEQUATION}.
First, using \eqref{E:RIEMCALBCONTRACTEDDECOMP}, we see that the term 
$-4 \Riemfour_{A B \Lunit B} \angDarg{A} \upsigma$
can be accommodated into the terms on RHS~\eqref{E:SECONDINHOMTERMMODIFIEDMASSASPECTEVOLUTIONEQUATION}
featuring a factor of $\angD \upsigma$.
Next, to handle the term 
$-
	2 \hat{\upchi}_{AB}\Riemfour_{A \Lunit \uLunit B}$,
we first note that since the Cartesian components $\gfour_{\alpha \beta}$ are of the schematic form $\gfour_{\alpha \beta} = \gensmoothfunction(\vec{\Psi})$,
the standard expression for the components of $\Riemfour$
in terms of the Christoffel symbols of $\gfour$ 
and their first derivatives yields that relative to the Cartesian coordinates, 
we have  
$\Riemfour_{\alpha \beta \gamma \delta} 
= 
\gensmoothfunction(\vec{\Psi}) \cdot \pmb{\partial}^2 \vec{\Psi}
+
\gensmoothfunction(\vec{\Psi}) \cdot (\pmb{\partial} \vec{\Psi})^2$.
It follows that, schematically, we have
$
-
	2 \hat{\upchi}_{AB}\Riemfour_{A \Lunit \uLunit B}
= 	\gensmoothfunction_{(\vec{\Lunit})}  \cdot \hat{\upchi} \cdot \pmb{\partial}^2 \vec{\Psi}
+
\gensmoothfunction_{(\vec{\Lunit})}  \cdot \hat{\upchi} \cdot (\pmb{\partial} \vec{\Psi})^2
$,
which is of the form of the next-to-last and last products 
on RHS~\eqref{E:SECONDINHOMTERMMODIFIEDMASSASPECTEVOLUTIONEQUATION}.
Using the schematic relations 
$\Chfour_{\alpha} = \gensmoothfunction(\vec{\Psi}) \cdot \pmb{\partial} \vec{\Psi}$,
$\Chfour_{\uLunit} = \gensmoothfunction_{(\vec{\Lunit})} \cdot \pmb{\partial} \vec{\Psi}$,
and
$
k_{\spherenormal \spherenormal} = \gensmoothfunction_{(\vec{\Lunit})} \cdot \pmb{\partial} \vec{\Psi}$,
we can handle the terms
$
\Riemfour_{A \Lunit A B} \Chfour_{B}$,
$ -
	\frac{1}{2} \Chfour_{\alpha} \Riemfour^\alpha_{\  \Lunit \Lunit \uLunit}$,
$	
2
	\Ricfour_{\Lunit \Lunit} k_{\spherenormal \spherenormal}
$,
and
$-
	\frac{1}{2} \Ricfour_{\Lunit \Lunit} \Chfour_{\uLunit}
$
using a similar argument,
which allows us to incorporate these error terms into
the next-to-last and last products 
on RHS~\eqref{E:SECONDINHOMTERMMODIFIEDMASSASPECTEVOLUTIONEQUATION}.
To handle
$
\frac{1}{2}\mytr_{\gsphere} \upchi \Riemfour_{A \uLunit  \Lunit A}	
$,
we use \eqref{E:RIEMALUNDERLINELBCONTRACTEDDECOMPINVOLVINGVORTANDENT} 
to substitute for $\Riemfour_{A \uLunit  \Lunit A}$
and \eqref{E:MODTRICHISMALL};
this leads to terms of the form
RHSs~\eqref{E:FIRSTINHOMTERMMODIFIEDMASSASPECTEVOLUTIONEQUATION}--\eqref{E:SECONDINHOMTERMMODIFIEDMASSASPECTEVOLUTIONEQUATION}.
To treat the remaining term $\Lunit^{\alpha} \square_{\gfour} \Chfour_{\alpha}$ on 
RHS~\eqref{E:DERIVATION_MU_CHECK_EQUATION_FINAL}, 
we first recall that $\Chfour_{\alpha} = \gensmoothfunction(\vec{\Psi}) \cdot \pmb{\partial} \vec{\Psi}$.
Thus,
we can commute equation \eqref{E:RESCALEDCOVARIANTWAVE}
with
$
\gensmoothfunction(\vec{\Psi}) \cdot \pmb{\partial}
$
(recall that we have dropped the ``$\uplambda$'' subscripts featured in \eqref{E:RESCALEDCOVARIANTWAVE})
to 
conclude that
$\Lunit^{\alpha} \square_{\gfour} \Chfour_{\alpha}$
can be accommodated into the terms on RHS~\eqref{E:SECONDINHOMTERMMODIFIEDMASSASPECTEVOLUTIONEQUATION}
as desired. We clarify that when one commutes equation \eqref{E:RESCALEDCOVARIANTWAVE}
with
$
\gensmoothfunction(\vec{\Psi}) \cdot \pmb{\partial}
$, a source term appears from the RHS  \eqref{E:RESCALEDCOVARIANTWAVE} that is of the form $\uplambda^{-1} \gensmoothfunction(\vec{\Psi}) \cdot \pmb{\partial}\vec{\Psi} \cdot (\vec{\VortVort},\DivGradEnt)$. One then uses equations \eqref{E:RESCALEDRENORMALIZEDCURLOFSPECIFICVORTICITY} and \eqref{E:RESCALEDRENORMALIZEDDIVOFENTROPY} to express $\vec{\VortVort} = \gensmoothfunction(\vec{\Psi}) \cdot \pmb{\partial} \vec{\vortrenormalized} + \gensmoothfunction(\vec{\Psi}) \cdot \vec{\GradEnt} \cdot \pmb{\partial} \vec{\Psi}$ and $\DivGradEnt = \gensmoothfunction(\vec{\Psi}) \pmb{\partial} \vec{S} + \gensmoothfunction(\vec{\Psi}) \cdot  \vec{S} \cdot \pmb{\partial} \vec{\Psi}$. In particular, 
this leads to the presence of the terms of type $\uplambda^{-1} \gensmoothfunction(\vec{\Psi}) \cdot \pmb{\partial} \vec{\Psi} \cdot \vec{S} \cdot \pmb{\partial} \vec{\Psi}$.
This finishes the proof of \eqref{E:MODIFIEDMASSASPECTEVOLUTIONEQUATION}
and completes our proof sketch of the proposition.

\end{proof}

\subsection{Norms}
\label{SS:GEOMETRICNORMS}
In this subsection, we define the norms that we will use to control the acoustic geometry.
These norms are stated in terms of the volume forms defined in
Subsubsect.\,\ref{SSS:METRICSANDVOLUMEFORMSINGEOMETRICCOORDINATES}.

\begin{definition}[Norms]
	\label{D:GEOMETRICNORMS}
	For $S_{t,u}$-tangent tensorfields $\upxi$ and $q \in [1,\infty)$, we define
	\begin{subequations}
	\begin{align} \label{E:GEOMETRICSTULQNORM}
		\| \upxi \|_{L_{\gsphere}^q(S_{t,u})}
		&
		:=
		\left\lbrace
		\int_{\upomega \in \mathbb{S}^2}
					|\upxi(t,u,\upomega)|_{\gsphere}^q
				\, d \spherevolarg{t}{u}{\upomega}
		\right\rbrace^{1/q}, 
		&&		\\
		\| \upxi \|_{L_{\upomega}^q(S_{t,u})}
		& := 
				\left\lbrace
				\int_{\upomega \in \mathbb{S}^2}
					|\upxi(t,u,\upomega)|_{\gsphere}^q
				\, d \flatspherevolarg{\upomega}
				\right\rbrace^{1/q},
		&
		\| \upxi \|_{L_{\upomega}^{\infty}(S_{t,u})}
		& := \mbox{ess sup}_{\upomega \in \mathbb{S}^2} |\upxi(t,u,\upomega)|_{\gsphere}.
	\end{align}
	\end{subequations}
	
	Moreover, if $q_1 \in [1,\infty)$ and $q_2 \in [1,\infty]$, then 
	with $[u]_+ := \max \lbrace 0,u \rbrace$ denoting the minimum value of $t$
	along $\widetilde{\mathcal{C}}_u$, 
	we define
	\begin{align} \label{E:MIXEDNORMCU}
		\| \upxi \|_{L_t^{q_1} L_{\upomega}^{q_2}(\widetilde{\mathcal{C}}_u)}
		& 
		:=
		\left\lbrace
		\int_{[u]_+}^{\RescaledTboot}
			\| \upxi \|_{L_{\upomega}^{q_2}(S_{\uptau,u})}^{q_1}
		\, d \uptau
		\right\rbrace^{1/q_1}, 
		&
		\| \upxi \|_{L_t^{\infty} L_{\upomega}^{q_2}(\widetilde{\mathcal{C}}_u)} 
		& := 
		\mbox{ess sup}_{\uptau \in [[u]_+,\RescaledTboot]}
		\| \upxi \|_{L_{\upomega}^{q_2}(S_{t,u})}.
	\end{align}
	
	Moreover, if $q_1 \in [1,\infty)$ and $q_2 \in [1,\infty]$, then
	noting that
	$-\frac{4}{5} \RescaledTboot \leq u \leq t$ along $\widetilde{\Sigma}_t$, we define
	\begin{align}\label{E:MIXEDNORMSIGMAT}
		\| \upxi \|_{L_u^{q_1} L_{\upomega}^{q_2}(\widetilde{\Sigma}_t)}
		& :=
		\left\lbrace
		\int_{-\frac{4}{5} \RescaledTboot}^t
			\| \upxi \|_{L_{\upomega}^{q_2}(S_{t,u})}^{q_1}
		\, d u
		\right\rbrace^{1/q_1},
		&
		\| \upxi \|_{L_u^{\infty} L_{\upomega}^{q_2}(\widetilde{\Sigma}_t)} 
		& := 
		\mbox{ess sup}_{u \in [-\frac{4}{5} \RescaledTboot,t]}
		\| \upxi \|_{L_{\upomega}^{q_2}(S_{\uptau,u})}.
		\end{align}
		Similarly, if $q_1, q_2 \in [1,\infty)$, then
		\begin{align}
		\| \upxi \|_{L_u^{q_1} L_{\gsphere}^{q_2}(\widetilde{\Sigma}_t)}
		& :=
		\left\lbrace
		\int_{-\frac{4}{5} \RescaledTboot}^t
			\| \upxi \|_{L_{\gsphere}^{q_2}(S_{t,u})}^{q_1}
		\, d u
		\right\rbrace^{1/q_1},
		&
		\| \upxi \|_{L_u^{\infty} L_{\gsphere}^{q_2}(\widetilde{\Sigma}_t)}
		& :=
		\mbox{ess sup}_{u \in [-\frac{4}{5} \RescaledTboot,t]}
		\| \upxi \|_{L_{\gsphere}^{q_2}(S_{t,u})}^{q_1}.
	\end{align}
		
	Similarly, if $q_1, q_2, q_3 \in [1,\infty)$, 
	then we define
	\begin{subequations}
	\begin{align}
		\| \upxi \|_{L_t^{q_1} L_u^{q_2} L_{\upomega}^{q_3}(\widetilde{\mathcal{M}})}
		& :=
		\left\lbrace
		\int_0^{\RescaledTboot}
			\| \upxi \|_{L_u^{q_2} L_{\upomega}^{q_3}(\widetilde{\Sigma}_{\uptau})}^{q_1}
		\, d \uptau
		\right\rbrace^{1/q_1},
			\label{E:UFIRSTSPACETIMEMIXEDNORMS} \\
	\| \upxi \|_{L_u^{q_1} L_t^{q_2} L_{\upomega}^{q_3}(\widetilde{\mathcal{M}})}
		& :=
		\left\lbrace
		\int_{-\frac{4}{5} \RescaledTboot}^{\RescaledTboot}
			\| \upxi \|_{L_t^{q_2} L_{\upomega}^{q_3}(\widetilde{\mathcal{C}}_u)}^{q_1}
		\, d u
		\right\rbrace^{1/q_1},
			\label{E:TFIRSTSPACETIMEMIXEDNORMS} \\
	\| \upxi \|_{L_t^q L_x^{\infty}(\widetilde{\mathcal{M}})}
	& := 
		\left\lbrace
		\int_0^{\RescaledTboot}
			\| \upxi \|_{L^{\infty}(\widetilde{\Sigma}_{\uptau})}^q
		\, d \uptau
		\right\rbrace^{1/q},
			\label{E:TSECONDSPACETIMEMIXEDNORMS} \\
	\| \upxi \|_{L^{\infty}(\widetilde{\mathcal{M}})}
	& := \mbox{ess sup}_{t \in [0,\RescaledTboot], \, u \in [-\frac{4}{5} \RescaledTboot,t], \, \upomega \in \mathbb{S}^2} |\upxi(t,u,\upomega)|_{\gsphere}.
	\label{E:LINFINITYNORMONWHOLEDOMAIN}
	\end{align}
	\end{subequations}		
	We also extend the definitions \eqref{E:UFIRSTSPACETIMEMIXEDNORMS}--\eqref{E:TFIRSTSPACETIMEMIXEDNORMS} to allow  
	$q_1, q_2, q_3 \in [1,\infty]$ by making the obvious modifications.
	We also define, by making the obvious modifications in \eqref{E:UFIRSTSPACETIMEMIXEDNORMS}--\eqref{E:LINFINITYNORMONWHOLEDOMAIN},
	norms in which the set $\widetilde{\mathcal{M}}$ 
	is replaced with the set $\widetilde{\mathcal{M}}^{(Int)}$ (see \eqref{E:INTERIORANDEXTERIOREGIONSAREUNIONSOFSPHERES}).
	For example, if $q_1, q_2, q_3 \in [1,\infty)$, then 
	$\| \upxi \|_{L_t^{q_1} L_u^{q_2} L_{\upomega}^{q_3}(\widetilde{\mathcal{M}}^{(Int)})}
		:=
		\left\lbrace
		\int_0^{\RescaledTboot}
				\left\lbrace
				\int_0^{\uptau}
					\left\lbrace
					\int_{\upomega \in \mathbb{S}^2}
						|\upxi(t,u,\upomega)|_{\gsphere}^{q_3}
					\, d \flatspherevolarg{\upomega}
					\right\rbrace^{q_2/q_3}
					\, du
				\right\rbrace^{q_1/q_2}
		\, d \uptau
		\right\rbrace^{1/q_1}
	$.
	
	Next, for $q \in [1,\infty)$, we define the following norms, where 
	\eqref{E:WEIGHTEDLTFINITYFIRSTNORM} and \eqref{E:WEIGHTEDLUFINITYFIRSTNORM} involve $\volrat$ 
	(see definition \eqref{E:RATIOOFVOLUMEFORMS}),
	such that an $L^{\infty}$ norm in $t$ or $u$ acts first:
	\begin{subequations}
	\begin{align}
		\| \upxi \|_{L_{\gsphere}^q L_t^{\infty}(\widetilde{\mathcal{C}}_u)}
		& := 
				\left\lbrace
				\int_{\mathbb{S}^2}
					\mbox{ess sup}_{t \in [[u]_+,\RescaledTboot]} 
						\left(
							\volrat(t,u,\upomega)
							|\upxi (t,u,\upomega)|_{\gsphere}^q  
						\right)
				\, d \flatspherevolarg{\upomega}
				\right\rbrace^{1/q},
					\label{E:WEIGHTEDLTFINITYFIRSTNORM} \\
		\| \upxi \|_{L_{\upomega}^q L_t^{\infty}(\widetilde{\mathcal{C}}_u)}
		& := 
				\left\lbrace
				\int_{\mathbb{S}^2}
					\mbox{ess sup}_{t \in [[u]_+,\RescaledTboot]}  |\upxi (t,u,\upomega)|_{\gsphere}^q
				\, d \flatspherevolarg{\upomega}
				\right\rbrace^{1/q},
					\\
		\| \upxi \|_{L_{\gsphere}^q L_u^{\infty}(\widetilde{\Sigma}_t)}^q
		& := 
					\left\lbrace
					\int_{\upomega \in \mathbb{S}^2}
					\mbox{ess sup}_{u \in [-\frac{4}{5} \RescaledTboot,t]}
					\left(
						\volrat(t,u,\upomega)
						|\upxi(t,u,\upomega)|_{\gsphere}^q
				\right)
				\, d \flatspherevolarg{\upomega}
				\right\rbrace^{1/q}.
				\label{E:WEIGHTEDLUFINITYFIRSTNORM}
	\end{align}
	\end{subequations}

\end{definition}

\subsection{The fixed number $p$}
In the rest of the article, $p > 2$ denotes a fixed number with 
\begin{align} \label{E:BOUNDSONLEBESGUEEXPONENTP}
	 0 & < \updelta_0 < 1 - \frac{2}{p} < \Sob - 2,
\end{align}
where $\updelta_0$ is the parameter that we fixed in \eqref{E:DELTA0DEF}.
$p$ will appear in many of our ensuing estimates.

\subsection{H\"{o}lder norms in the geometric angular variables}
\label{SS:HOLDERNORMSINGEOMETRICANGULARVARIABLES}
	Some of our elliptic estimates for $\hat{\upchi}$ involve H\"{o}lder norms in the geometric angular variables,
	which we define in this subsection. We remind the reader that
	$\stgsphere$ denotes the standard round metric on the Euclidean unit sphere $\mathbb{S}^2$.
	In the rest of the paper, for points $\upomega_{(1)}, \upomega_{(2)} \in \mathbb{S}^2$,
	we denote their distance with respect to $\stgsphere$
	by $d_{\stgsphere}(\upomega_{(1)},\upomega_{(2)})$
	In particular, $d_{\stgsphere}(\upomega_{(1)},\upomega_{(2)}) \leq \pi$.
	
	To proceed, for each pair of points $\upomega_{(1)}, \upomega_{(2)} \in \mathbb{S}^2$
	with $d_{\stgsphere}(\upomega_{(1)},\upomega_{(2)}) < \pi$
	and for each pair $m,n$ of non-negative integers,
	let $\Phi_n^m(\upomega_{(1)};\upomega_{(2)}) : (T_n^m)_{\upomega_{(1)}}(\mathbb{S}^2) \rightarrow (T_n^m)_{\upomega_{(2)}}(\mathbb{S}^2)$,
	$\upxi \rightarrow \Phi_n^m(\upomega_{(1)};\upomega_{(2)})[\upxi]$,
	denote the parallel transport operator with respect to $\stgsphere$, 
	where $(T_n^m)_{\upomega}(\mathbb{S}^2)$ denotes the vector space of type $\binom{m}{n}$ tensors at $\upomega \in \mathbb{S}^2$.
	Note that
	$\Phi_n^m(\upomega_{(1)};\upomega_{(2)})$
	provides a linear isomorphism between type $\binom{m}{n}$ tensors $\upxi$ at $\upomega_{(1)}$ 
	and type $\binom{m}{n}$ tensors at $\upomega_{(2)}$ by parallel transport
	along the unique $\stgsphere$-geodesic connecting $\upomega_{(1)}$ and $\upomega_{(2)}$.
	From the basic properties of parallel transport, it follows that $\Phi_n^m$ respects tensor products and contractions.
	That is, if $\upxi_{(1)} \cdot \upxi_{(2)}$ schematically denotes the tensor product of $\upxi_{(1)}$ and $\upxi_{(2)}$
	possibly followed by some contractions, 
	then
	$\Phi_n^m(\upomega_{(1)};\upomega_{(2)})[\upxi_{(1)} \cdot \upxi_{(2)}]  
	= 
	\Phi_n^m(\upomega_{(1)};\upomega_{(2)})[\upxi_{(1)}] \cdot 
	\Phi_n^m(\upomega_{(1)};\upomega_{(2)})[\upxi_{(2)}]$.
	If $\upxi = \upxi(\upomega)$ is a type $\binom{m}{n}$ tensorfield on $\mathbb{S}^2$ and $d_{\stgsphere}(\upomega_{(1)};\upomega_{(2)}) < \pi$, 
	then we define\footnote{For example, if $\upxi = \upxi(\upomega)$ is a scalar function,
	then $\upxi^{\parallel}(\upomega_{(1)};\upomega_{(2)}) = \upxi(\upomega_{(1)})$.
	As a second example, if $\upxi = \upxi(\upomega)$ is a one-form,
	then in a local angular coordinate chart containing the point $\upomega_{(2)}$, 
	we have, for $\upomega_{(1)}$ close to $\upomega_{(2)}$:
	$\upxi^{\parallel}(\upomega_{(1)};\upomega_{(2)})(\frac{\partial}{\partial \upomega^A}|_{\upomega_{(2)}})
	=
	M_A^B(\upomega_{(1)};\upomega_{(2)}) \upxi(\upomega_{(1)})(\frac{\partial}{\partial \upomega^B}|_{\upomega_{(1)}})$,
	where the $M_A^B(\upomega_{(1)};\upomega_{(2)})$ are smooth functions of $\upomega_{(1)}$ and $\upomega_{(2)}$
	such that for $A,B,C = 1,2$, we have $M_A^B(\upomega_{(C)},\upomega_{(C)}) = \updelta_A^B$,
	where $\updelta_A^B$ is the Kronecker delta. That is, for $C=1,2$, $\upxi^{\parallel}(\upomega_{(C)};\upomega_{(C)}) = \upxi(\upomega_{(C)})$.}  
	$\upxi^{\parallel}(\upomega_{(1)};\upomega_{(2)}) 
	:= \Phi_n^m(\upomega_{(1)};\upomega_{(2)})[\upxi(\upomega_{(1)})] \in (T_n^m)_{\upomega_{(2)}}(\mathbb{S}^2)$.
	Note that $(\upxi_{(1)} \cdot \upxi_{(2)})^{\parallel}(\upomega_{(1)};\upomega_{(2)}) 
	=
	\upxi_{(1)}^{\parallel}(\upomega_{(1)};\upomega_{(2)})
	\cdot
	\upxi_{(2)}^{\parallel}(\upomega_{(1)};\upomega_{(2)})
	$.

\begin{definition}[H\"{o}lder norms in the geometric angular variables]
	\label{D:HOLDERNORMSINGEOMETRICANGULARVARIABLES}
	For constants $\upbeta \in (0,1)$, we define
	\begin{subequations}
	\begin{align} \label{E:HOMOGENEOUSHOLDERNORMSINGEOMETRICANGULARVARIABLES}
	\| \upxi \|_{\dot{C}_{\upomega}^{0,\upbeta}(S_{t,u})}
	& := 
			\sup_{0 < d_{\stgsphere}(\upomega_{(2)},\upomega_{(1)}) < \frac{\pi}{2}}	
			\frac{\rgeo^{(m-n)} \left|\upxi(t,u,\upomega_{(1)}) - \upxi^{\parallel}(t,u,\upomega_{(2)};\upomega_{(1)}) \right|_{\stgsphere(\upomega_{(1)})}}{
			d_{\stgsphere}^{\upbeta}(\upomega_{(1)};\upomega_{(2)})},
				\\
	\| \upxi \|_{C_{\upomega}^{0,\upbeta}(S_{t,u})}
	& := 
			\| \upxi \|_{L_{\upomega}^{\infty}(S_{t,u})}
			+
			\| \upxi \|_{\dot{C}_{\upomega}^{0,\upbeta}(S_{t,u})}.
			\label{E:HOLDERNORMSINGEOMETRICANGULARVARIABLES}
\end{align}
\end{subequations}
\end{definition}

Note that our bootstrap assumption \eqref{E:BOOTSTRAPMETRICAPPROXIMATELYROUND} below implies that
if $\upxi$ is type $\binom{m}{n}$, then the denominator on RHS~\eqref{E:HOMOGENEOUSHOLDERNORMSINGEOMETRICANGULARVARIABLES} satisfies
\begin{align} \label{E:STNORMVSGSPHERENORMCOMPARISONWITHPOWERSOFR}
\rgeo^{(m-n)}
\left|\upxi(t,u,\upomega_{(1)}) - \upxi^{\parallel}(t,u,\upomega_{(2)};\upomega_{(1)}) \right|_{\stgsphere(\upomega_{(1)})}
\approx
\left|\upxi(t,u,\upomega_{(1)}) - \upxi^{\parallel}(t,u,\upomega_{(2)};\upomega_{(1)}) \right|_{\gsphere(t,u,\upomega_{(1)})}.
\end{align}

In Sect.\,\ref{S:ESTIMATESFOREIKONALFUNCTION}, 
we will also use mixed norms that are defined
by replacing the $L_{\upomega}^{\infty}$ norm from Subsect.\,\ref{SS:GEOMETRICNORMS} 
with the $C_{\upomega}^{0,\updelta_0}$ norm.
For example, for $q \in [1,\infty)$, we define
\begin{align} \label{E:UFIRSTHOLDERANGULARSPACETIMEMIXEDNORMS}
\| 
	\upxi
\|_{L_t^q L_u^{\infty} C_{\upomega}^{0,\updelta_0}(\widetilde{\mathcal{M}}^{(Int)})}
& :=
		\left\lbrace
		\int_0^{\RescaledTboot}
			\mbox{ess sup}_{u \in [0,\uptau]}
			\| \upxi \|_{C_{\upomega}^{0,\updelta_0}(S_{\uptau,u})}^q
		\, d \uptau
		\right\rbrace^{1/q},
\end{align}
and we extend definition \eqref{E:UFIRSTHOLDERANGULARSPACETIMEMIXEDNORMS} to the case $q=\infty$ by making the obvious modification.

\subsection{\texorpdfstring{The initial foliation on $\Sigma_0$}{The initial foliation on Sigma 0}}
\label{SS:INITIALFOLIATION}
In this subsection, we state Proposition~\ref{P:INITIALFOLIATION}, which 
yields the existence of an initial condition
for the eikonal function $u$ (see Subsect.\,\ref{SS:EIKONAL}) 
featuring a variety of properties that we exploit in our analysis.
More precisely, as we mentioned in Subsubsect.\,\ref{SSS:EIKONALEXTERIOR},
we set $u|_{\Sigma_0} := -w$, where $w$ is the function yielded by the proposition.
The proof of the proposition is the same as in \cite{qW2017} and we therefore omit it. 
In particular, the key equation \eqref{E:EQNOFINITIALFOLIATION} stated below is exactly the same as in \cite{qW2017}. 
The proof of Proposition~\ref{P:INITIALFOLIATION} relies on the regularity of the Ricci curvature of the spatial metric induced on
$\Sigma_0$, and the regularity is exactly the same as in \cite{qW2017}. 
More precisely, 
since the spatial metric $g$ satisfies $g_{ij} = g_{ij}(\vec{\Psi})$, 
the regularity of the spatial Ricci curvature on
$\Sigma_0$ is controlled by the energy estimates\footnote{As we highlighted in Remark~\ref{R:REMARKSONRESCALING}, 
the hypersurface that we denote by ``$\Sigma_0$''
here corresponds to the hypersurface that we denoted by 
``$\Sigma_{t_k}$'' in Sects.\,\ref{S:DATAANDBOOTSTRAPASSUMPTION}--\ref{S:ELLIPTICESTIMATESINHOLDERSPACES}.
Hence, to control the appropriate Sobolev norms of $\vec{\Psi}$ along these hypersurfaces,
we need the energy estimates.
We also point out that Props.\,\ref{P:PRELIMINARYENERGYANDELLIPTICESTIMATES} and \ref{P:TOPORDERENERGYESTIMATES}
yield energy estimates for the non-rescaled solution variables,
while in the expression ``$g_{ij}(\vec{\Psi})$'' in the present section,
$\vec{\Psi}$ denotes the rescaled solution (see Subsect.\,\ref{SS:NOMORELAMBDA}).
Hence, one needs to account for the rescaling
when controlling the size of the $L^2$ norms of the derivatives of 
$g_{ij}(\vec{\Psi})$
(such bounds are needed to prove Proposition~\ref{P:INITIALFOLIATION}
using the arguments given in \cite{arxivqW2014}*{Appendix~C}).
\label{FN:FIRSTVERSIONINITIALFOLIATIONCAREFULABOUTHYPERSURFACES}}    
we already derived in 
Props.\,\ref{P:PRELIMINARYENERGYANDELLIPTICESTIMATES} and \ref{P:TOPORDERENERGYESTIMATES}, 
and we stress that our energy estimates for $\vec{\Psi}$ are the same as the energy estimates derived in \cite{qW2017}.
Proposition~\ref{P:INITIALFOLIATION} provides, in particular, initial conditions
for various tensorfields constructed out of the eikonal function
that are relevant for the study of $\widetilde{\mathcal{M}}^{(Ext)}$.
We emphasize how important the proposition is for the viability of our approach. For example,
if we had instead chosen the ``simpler'' initial condition  $u|_{\Sigma_0} := - r$, where $r$ 
is the standard Euclidean radial variable,
then given the limited regularity of the fluid solution,
the null mean curvature of the spheres $\lbrace r = \mbox{\upshape const} \rbrace$
with respect to the metric induced on them by the acoustical metric $\gfour(\vec{\Psi})$
\emph{would not generally have enjoyed any useful quantitative pointwise boundedness properties}.
This could have led to the instantaneous formation of null focal points\footnote{More precisely, 
this would have led to the possibility that $\| \mytr_{\congsphere} \widetilde{\upchi}^{(Small)} \|_{L^1([0,T])L_x^{\infty}}$
is infinite no matter how small $T$ is; see, for example, the proofs of 
\eqref{E:QUANTITATIVECONTROLOFCARTESIANCOMPONENTSRESCALEDGEOMETRICCOORDINATEANGULARDERIVATIVEVECTORVIELDINALLREGIONS}
and \eqref{E:INTEGRALCURVESOFLUNITSEPARATEDESTIMATEFORSMALLANGLES}
for clarification on the connection between having quantitative control of time integrals of 
$\| \mytr_{\congsphere} \widetilde{\upchi}^{(Small)} \|_{L_x^{\infty}(\Sigma_t)}$
and obtaining control over the local separation of the integral curves of $\Lunit$.
\label{FN:NULLFOCAL}}
and the breakdown of our geometric coordinate system.
In contrast, \eqref{E:EQUIVALENTEQNOFINITIALFOLIATION} and the estimates
of the proposition imply, for example, that 
$\| \rgeo^{1/2} \mytr_{\congsphere} \widetilde{\upchi}^{(Small)} \|_{L^{\infty}(\Sigma_0)} \lesssim \uplambda^{-1/2}$.
Initial condition bounds of this type play a crucial role in the proof of Prop.\,\ref{P:MAINESTIMATESFOREIKONALFUNCTIONQUANTITIES},
which provides the main estimates for the acoustic geometry.

\begin{proposition}[Existence and properties of the initial foliation]
	\label{P:INITIALFOLIATION}
	On the hypersurface\footnote{As we highlighted in Remark~\ref{R:REMARKSONRESCALING}, 
	the hypersurface that we denote by ``$\Sigma_0$''
in this proposition
corresponds to the hypersurface that we denoted by 
``$\Sigma_{t_k}$'' in Sects.\,\ref{S:DATAANDBOOTSTRAPASSUMPTION}--\ref{S:ELLIPTICESTIMATESINHOLDERSPACES}.
\label{FN:INITIALFOLIATIONCAREFULABOUTHYPERSURFACES}} $\Sigma_0$, there exists a function
	$w = w(x)$ on the domain 
	implicitly defined by 
	$0 \leq w \leq \RescaledFoliationparameter := \frac{4}{5} \RescaledTboot$,
	such that $w(\bf{z}) = 0$ (where $\bf{z}$ is the point in $\Sigma_0$ mentioned in Subsect.\,\ref{SS:EIKONAL}), 
	such that $w$ is smooth away from $\bf{z}$,
	such that its levels sets $S_w$ are diffeomorphic to $\mathbb{S}^2$ for $0 < w \leq \RescaledFoliationparameter$,
	such that $\mathscr{O} := \cup_{0 \leq w < \RescaledFoliationparameter} S_w$ is a neighborhood of $\bf{z}$
	contained in the metric ball $B_{\RescaledTboot}({\bf{z}},g)$ 
	(with respect to the rescaled first fundamental form $g$ of $\Sigma_0$)
	of radius $\RescaledTboot$ centered at $\bf{z}$,
	and such that the following relations hold,
	where $\displaystyle a = \frac{1}{\sqrt{(g^{-1})^{cd} \partial_c w \partial_d w}}$ is the lapse,
	$\mytr_g k := (g^{-1})^{cd} k_{cd}$,
	and $\Chfour_{\Lunit} := \Chfour_{\alpha} \Lunit^{\alpha}$
	is a contracted (and lowered) Cartesian Christoffel symbol of the rescaled spacetime metric $\gfour$:
	\begin{align} \label{E:EQNOFINITIALFOLIATION}
		\mytr_{\gsphere} \spheresecondfund
		+
		k_{\spherenormal \spherenormal}
		& = \frac{2}{a w}
			+ 
			\mytr_g k
			-
			\Chfour_{\Lunit},
		&
		a(\bf{z})
		& = 1.
	\end{align}
	Note that by \eqref{E:CONNECTIONCOEFFICIENT},
	\eqref{E:MODTRICHISMALL}, and the relation $\rgeo(0,- u) = w$ (for $- \RescaledFoliationparameter \leq u \leq 0$),
	the first equation in \eqref{E:EQNOFINITIALFOLIATION} is equivalent to
	\begin{align} \label{E:EQUIVALENTEQNOFINITIALFOLIATION}
		\mytr_{\congsphere} \widetilde{\upchi}^{(Small)}|_{\Sigma_0}
		& = \frac{2(1-a)}{a w},
		&&
		\mbox{for } 0 \leq w \leq \RescaledFoliationparameter.
	\end{align}
	
	Let $q_*$ satisfy $0 < 1 - \frac{2}{q_*} < \Sob - 2$
	and let $\stgsphere = \stgsphere(\upomega)$ be the standard round metric on the Euclidean unit sphere $\mathbb{S}^2$,
	where the angular coordinates $\lbrace \upomega^A \rbrace_{A=1,2}$ are as in Subsubsect.\,\ref{SSS:EIKONALEXTERIOR}.
	Then if $q_*$ is sufficiently close to $2$,
	the following estimates hold\footnote{In \cite{qW2017}*{Proposition~4.3}, 
	the author stated the weaker estimate $\| w^{-1/2}(a-1) \|_{L^{\infty}(\Sigma_0^{\RescaledFoliationparameter})}
	\lesssim \uplambda^{-1/2}$
	in place of the stronger estimate
	$\| w^{-1/2}(a-1) \|_{L_w^{\infty}C_{\upomega}^{0,1 - \frac{2}{q_*}}(\Sigma_0^{\RescaledFoliationparameter})}
	\lesssim \uplambda^{-1/2}$ 
	appearing in \eqref{E:INTIALLAPSEANDSPHEREVOLUMEELEMENTESTIMATE}.
	However, the desired stronger estimate follows from the Morrey-type estimate \eqref{E:EUCLIDEANFORMMORREYONSTU}
	and the analysis given just above \cite{arxivqW2014}*{Equation~(10.113)}.
	\label{FN:UPGRADELAPSETOHOLDERESTIMATE}}
	on $\Sigma_0^{\RescaledFoliationparameter} := \cup_{0 \leq w \leq \RescaledFoliationparameter} S_w$,
	where $\upepsilon_0$ is as in Subsect.\,\ref{SS:PARAMETERS}, where the role of $q$ is played by $q_*$:
	\begin{subequations}
	\begin{align}
			|a - 1|
			\lesssim \uplambda^{-4 \upepsilon_0}
			& \leq \frac{1}{4},
			&
			\| w^{-1/2}(a-1) \|_{L_w^{\infty}C_{\upomega}^{0,1 - \frac{2}{q_*}}(\Sigma_0^{\RescaledFoliationparameter})}
			& \lesssim \uplambda^{-1/2},
			&
			\volrat(w,\upomega)
			:=
			\frac{\sqrt{\mbox{\upshape det} \gsphere}(w,\upomega)}{\sqrt{\mbox{\upshape det} \stgsphere}(\upomega)}
			& \approx w^2,
				\label{E:INTIALLAPSEANDSPHEREVOLUMEELEMENTESTIMATE} 
		\end{align}
		
		\begin{align}
			\| w^{\frac{1}{2} - \frac{2}{q_*}} (\hat{\spheresecondfund},\angD \ln a) \|_{L_w^{\infty} L_{\gsphere}^{q_*}(\Sigma_0^{\RescaledFoliationparameter})}
			& \lesssim \uplambda^{-1/2},
			&
			\| \angD \ln a \|_{L_w^2 L_{\upomega}^{\infty}(\Sigma_0^{\RescaledFoliationparameter})},
				\, 
			\| \hat{\upchi} \|_{L_w^2 L_{\upomega}^{\infty}(\Sigma_0^{\RescaledFoliationparameter})}
			& \lesssim \uplambda^{-1/2},
			&&
				\label{E:INTIALDERIVATIVESOFLAPSEANDTRACEFREECHIEST} 
	\end{align}
	\begin{align}
			\max_{A,B=1,2}
			\left\| 
				w^{-2}
				\gsphere\left(\frac{\partial}{\partial \upomega^A},\frac{\partial}{\partial \upomega^B} \right) 
				- 
				\stgsphere\left(\frac{\partial}{\partial \upomega^A},\frac{\partial}{\partial \upomega^B} \right)
				\right\|_{L^{\infty}(\Sigma_0^{\RescaledFoliationparameter})}
			& \lesssim \uplambda^{-4 \upepsilon_0},
				\label{E:INITIALSPHEREMETRICESTIMATE}
				&&
					\\
			\max_{A,B,C=1,2}
			\left\| 
			\frac{\partial}{\partial \upomega^A}
			\left\lbrace 
				w^{-2} 
				\gsphere\left(\frac{\partial}{\partial \upomega^B},\frac{\partial}{\partial \upomega^C} \right) 
				- 
				\stgsphere\left(\frac{\partial}{\partial \upomega^B},\frac{\partial}{\partial \upomega^C} \right)
			\right\rbrace
			\right\|_{L_w^{\infty} L_{\upomega}^{q_*}(\Sigma_0^{\RescaledFoliationparameter})}
			& \lesssim \uplambda^{-4 \upepsilon_0},
			&&
				\label{E:INITIALSPHEREMETRICFIRSTCOORDINATEANGULARDERIVATIVESESTIMATE} \\
			\| w^{\frac{1}{2} - \frac{2}{q_*}} \angD \ln\left( \rgeo^{-2} \volrat \right) \|_{L_w^{\infty} L_{\gsphere}^{q_*}(\Sigma_0^{\RescaledFoliationparameter})}
			& \lesssim \uplambda^{-1/2}.
			&&
			\label{E:INTEGRATEDINTIALSPHERVOLUMEELEMENTESTIMATE}
	\end{align}	
	\end{subequations}
	
	Finally, $\Sigma_0^{\RescaledFoliationparameter}$ is contained in the Euclidean ball of radius $\RescaledTboot$ in $\Sigma_0$
	centered at $\bf{z}$.
	
\end{proposition}

\begin{proof}[Discussion of the proof]
Based on the energy estimates we derived in Props.\,\ref{P:PRELIMINARYENERGYANDELLIPTICESTIMATES} 
and \ref{P:TOPORDERENERGYESTIMATES} (which are estimates for the non-rescaled solution variables),
the proof is the same as the proof of \cite{qW2017}*{Proposition~4.3},
which is given in \cite{arxivqW2014}*{Appendix~C}.
\end{proof}

\subsection{Initial conditions on the cone-tip axis tied to the eikonal function}
\label{SS:INITIALCONDITIONSTIEDTOEIKONAL}
The next lemma complements Prop.\,\ref{P:INITIALFOLIATION} by providing the initial conditions on the cone-tip axis
for various tensorfields tied to the eikonal function, i.e., initial conditions relevant for the study of $\widetilde{\mathcal{M}}^{(Int)}$.

\begin{lemma}[Initial conditions on the cone-tip axis tied to the eikonal function]
	\label{L:INITIALCONDITIONSTIEDTOEIKONAL}
	The following estimates hold on any acoustic null cone $\mathcal{C}_u$ emanating from a point
	on the cone-tip axis with $0 \leq u = t \leq \RescaledTboot$,
	where ``$\upxi = \mathcal{O}(\rgeo) \mbox{as } t \downarrow u$''
	means that\footnote{On RHS~\eqref{E:CONNECTIONCOEFFICIENTS0LIMITSALONGTIP}, 
	the implicit constants are allowed to depend on the $L^{\infty}$ norm of 
	the higher derivatives of the fluid solution. However, these
	constants never enter into our estimates since, in our subsequent analysis, 
	\eqref{E:CONNECTIONCOEFFICIENTS0LIMITSALONGTIP} will be used only to conclude
	that LHS~\eqref{E:CONNECTIONCOEFFICIENTS0LIMITSALONGTIP} is $0$ 
	along the cone-tip axis.} 
	$|\upxi|_{\gsphere} \lesssim (t - u)$
	$\mbox{as } t \downarrow u$:
	\begin{subequations}
		\begin{align} \label{E:CONNECTIONCOEFFICIENTS0LIMITSALONGTIP}
			&
			\mytr_{\gsphere} \upchi - \frac{2}{\rgeo},
				\,
			\rgeo \mytr_{\congsphere} \widetilde{\upchi}^{(Small)},
				\,
			|\hat{\upchi}|_{\gsphere},
				\,
			|\rgeo \sphereproject_j^a \partial_a \Lunit^i - \sphereproject_j^i|,
				\,
			\nulllapse - 1,
				\,
			|\upzeta|_{\gsphere},
				\,
			\upsigma,
				\\
		& \rgeo |\angD \mytr_{\gsphere} \upchi|_{\gsphere},
				\,
			\rgeo^2 |\angD \mytr_{\congsphere} \widetilde{\upchi}^{(Small)}|_{\gsphere},
				\,
			\rgeo |\angD \hat{\upchi}|_{\gsphere},
				\,
			\rgeo |\angD \nulllapse|_{\gsphere},
				\,
			\rgeo |\angD \upzeta|_{\gsphere},
				\,
			\rgeo |\angD \upsigma|_{\gsphere},
				\notag 
				\\
		&  
			\rgeo^2 \angLap \nulllapse,
				\,
			\rgeo^2 \angLap \upsigma,
				\,
			\rgeo^2 \upmu,
					\,
			\rgeo^2 \check{\upmu}
				\notag \\
		& =
			\mathcal{O}(\rgeo) \mbox{as } t \downarrow u,
				\notag \\
			\lim_{t \downarrow u}
			\| 
				(
					\underline{\upzeta},
					k
					)
			\|_{L_{\upomega}^{\infty}(S_{t,u})}
			& < \infty.
			\label{E:CONNECTIONCOEFFICIENTSFINITELIMITSALONGTIP}
		\end{align}
	\end{subequations}
	
	Moreover, with $\stgsphere$ denoting the standard round metric on the Euclidean unit sphere $\mathbb{S}^2$,
	we have
	\begin{subequations}
	\begin{align} \label{E:SPHEREFINITELIMITSALONGTIP}
		\lim_{t \downarrow u}
		\left\lbrace
		\rgeo^{-2}(t,u)
		\gsphere(t,u,\upomega) \left(\frac{\partial}{\partial \upomega^B},\frac{\partial}{\partial \upomega^C} \right) 
		\right\rbrace
		& = \stgsphere(\upomega)\left(\frac{\partial}{\partial \upomega^B},\frac{\partial}{\partial \upomega^C} \right),
			\\
		\lim_{t \downarrow u}
		\left\lbrace
		\rgeo^{-2}(t,u)
		\frac{\partial}{\partial \upomega^C}
		\gsphere(t,u,\upomega)\left(\frac{\partial}{\partial \upomega^B},\frac{\partial}{\partial \upomega^C} \right) 
		\right\rbrace
		& = 
		\frac{\partial}{\partial \upomega^C}
		\left\lbrace
			\stgsphere(\upomega) \left(\frac{\partial}{\partial \upomega^B},\frac{\partial}{\partial \upomega^C} \right)
		\right\rbrace.
		\label{E:ANGULARDERIVATIVESSPHEREFINITELIMITSALONGTIP}
	\end{align}
	\end{subequations}
	
	Moreover, with $\RescaledFoliationparameter := \frac{4}{5} \RescaledTboot$ (as in Prop.\,\ref{P:INITIALFOLIATION}),
	on $\Sigma_0^{\RescaledFoliationparameter} := \cup_{w \in (0,\RescaledFoliationparameter]} S_w$, we have
	(recalling that $w = - u|_{\Sigma_0} \geq 0$):
	\begin{align} \label{E:MODTRCHIESTIAMTESALONGINITIALFOLIATION}
		& \| w \mytr_{\congsphere} \widetilde{\upchi}^{(Small)} \|_{L^{\infty}(\Sigma_0^{\RescaledFoliationparameter})}
		 \lesssim \uplambda^{-4 \upepsilon_0},
		\\
		\notag
		&
		\| w^{3/2} \angD \mytr_{\congsphere} \widetilde{\upchi}^{(Small)} \|_{L_w^{\infty}L_{\upomega}^p(\Sigma_0^{\RescaledFoliationparameter})},
			\,
		\| w^{1/2} \mytr_{\congsphere} \widetilde{\upchi}^{(Small)} \|_{L_w^{\infty}C_{\upomega}^{0,1 - \frac{2}{p}}(\Sigma_0^{\RescaledFoliationparameter})}
		\lesssim \uplambda^{-1/2}.
	\end{align}

	
	Finally, with $\spherenormal$ denoting the unit outward normal to $S_w$ in $\Sigma_0$ and
	$\sphereproject$ denoting the $g$-orthogonal projection tensorfield onto $S_w$
	(where $g$ is the rescaled metric on $\Sigma_0$), we have
	\begin{align} \label{E:ANGULARDERIVATIVEOFSPHERENORMALATINITIALCONETIPAXISPOINT}
		\sum_{i,j=1,2,3}
		|w \sphereproject_j^c \partial_c \spherenormal^i - \sphereproject_j^i|
		& = \mathcal{O}(w) \mbox{as } w \downarrow 0.
	\end{align}
	
\end{lemma}

\begin{proof}[Discussion of the proof]
The lemma follows from the same arguments, based on Taylor expansions, that are
found in \cites{eP2004,qW2006}, and we therefore omit the details.
We refer to \cite{qW2017}*{Lemma~5.1} and \cite{arxivqW2014}*{Appendix~C}
for the analogous results in the context of quasilinear wave equations.
We also remark that there are simpler, alternative proofs available in \cite{Graf2021}*{Appendix~B} and \cite{Shao2010}*{Sect.~3}.
 We further clarify that in \cites{qW2006,eP2004}, the expansions along null cones
were derived not in terms of $\rgeo$, but rather 
in terms of the affine parameter $\aff = \aff(t,u,\upomega)$ of the geodesic null vectorfield $\nulllapse^{-1} \Lunit$ 
(i.e., $\Lunit \aff = \nulllapse$, where $\nulllapse$ is defined in \eqref{E:NULLLAPSE}),
normalized by $\aff(u,u,\upomega) = 0$. However, the same asymptotic expansions hold with $\rgeo$ in place of $\aff$,
thanks in part to the asymptotic relation
$
\lim_{t \downarrow u}
\frac{\aff(t,u,\upomega)}{\rgeo(t,u)}
= 1
$,
which follows from the identities $\Lunit \aff = \nulllapse$
and $\Lunit \rgeo = \Lunit t = 1$,
and the following fact, which can be independently established with the help of \eqref{E:NULLLAPSEISUNITYALONGCONETIPAXIS}:
$\lim_{t \downarrow u}
	\left\lbrace
	\nulllapse(t,u,\upomega) - 1
	\right\rbrace
	= 0
$.
We also clarify that the estimate 
$
\| w^{1/2} \mytr_{\congsphere} \widetilde{\upchi}^{(Small)} \|_{L_w^{\infty}C_{\upomega}^{0,1 - \frac{2}{p}}(\Sigma_0^{\RescaledFoliationparameter})}
\lesssim \uplambda^{-1/2}
$
in \eqref{E:MODTRCHIESTIAMTESALONGINITIALFOLIATION}
is stronger than the analogous estimate 
$
\| w^{1/2} \mytr_{\congsphere} \widetilde{\upchi}^{(Small)} \|_{L^{\infty}(\Sigma_0^{\RescaledFoliationparameter})}
\lesssim \uplambda^{-1/2}
$
stated \cite{qW2017}*{Lemma~5.1}; the desired stronger estimate is a simple consequence of \eqref{E:EQUIVALENTEQNOFINITIALFOLIATION}
and the first and second estimates in \eqref{E:INTIALLAPSEANDSPHEREVOLUMEELEMENTESTIMATE}.


\end{proof}

\section{Estimates for quantities constructed out of the eikonal function}
\label{S:ESTIMATESFOREIKONALFUNCTION}
Our main goal in this section is to prove Prop.\,\ref{P:MAINESTIMATESFOREIKONALFUNCTIONQUANTITIES},
which provides estimates for the acoustic geometry.
As we explain in Sect.\,\ref{S:REDUCTIONSOFSTRICHARTZ}, these estimates are the last new ingredient 
needed to prove the frequency-localized Strichartz estimate of Theorem~\ref{T:FREQUENCYLOCALIZEDSTRICHARTZ}.
The proof of Prop.\,\ref{P:MAINESTIMATESFOREIKONALFUNCTIONQUANTITIES} is based on a bootstrap argument 
and is located in Subsect.\,\ref{SS:PROOFOFPROPMAINESTIMATESFOREIKONALFUNCTIONQUANTITIES}.
Before proving the proposition, we first introduce the bootstrap assumptions (see Subsect.\,\ref{SS:ASSUMPTIONS}) 
and provide a series of preliminary inequalities and estimates.
Many of these preliminary results have been derived in prior works, and we typically do not repeat
the proofs. In Lemma~\ref{L:NEWESTIMATESFORPROPMAINESTIMATESFOREIKONALFUNCTIONQUANTITIES},
we isolate the new estimates that are not found in earlier works;
the results of Lemma~\ref{L:NEWESTIMATESFORPROPMAINESTIMATESFOREIKONALFUNCTIONQUANTITIES} in particular
quantify the effect of the high order derivatives of the vorticity and entropy 
on the evolution of the acoustic geometry; this will become clear during the proof of Prop.\,\ref{P:MAINESTIMATESFOREIKONALFUNCTIONQUANTITIES}.

\begin{remark}
	We remind the reader that in Sect.\,\ref{S:ESTIMATESFOREIKONALFUNCTION},
	we are operating under the conventions of Subsect.\,\ref{SS:NOMORELAMBDA}.
\end{remark}

\subsection{The main estimates for the eikonal function quantities}
\label{SS:MAINESTIMATESFOREIKONALFUNCTIONQUANTITIES}
Recall that $p > 2$ denotes the fixed number satisfying \eqref{E:BOUNDSONLEBESGUEEXPONENTP},
where $\updelta_0$ is the parameter that we fixed in \eqref{E:DELTA0DEF}.
We now state the main result of Sect.\,\ref{S:ESTIMATESFOREIKONALFUNCTION}; 
see Subsect.\,\ref{SS:PROOFOFPROPMAINESTIMATESFOREIKONALFUNCTIONQUANTITIES} for the proof.

\begin{proposition}[The main estimates for the eikonal function quantities]
\label{P:MAINESTIMATESFOREIKONALFUNCTIONQUANTITIES}
Let $p$ be as in \eqref{E:BOUNDSONLEBESGUEEXPONENTP},
assume that $q > 2$ is sufficiently close to $2$,
and recall that we fixed several small parameters, including $\upepsilon_0$, in Subsect.\,\ref{SS:PARAMETERS}.
There exists a large constant $\Lambda_0 > 0$ such that 
under the bootstrap assumptions of Subsect.\,\ref{SS:ASSUMPTIONS},
if $\uplambda \geq \Lambda_0$,
then the following estimates hold on $\widetilde{\mathcal{M}} \subset [0,\RescaledTboot] \times \mathbb{R}^3$,
where the norms referred to below are defined in Subsect.\,\ref{SS:GEOMETRICNORMS},
and the corresponding spacetime regions such as $\widetilde{\mathcal{C}}_u \subset \widetilde{\mathcal{M}}$
are defined in Subsect.\,\ref{SS:GEOMETRICSPACETIMESUBSETS}.

\medskip

\noindent \underline{\textbf{Estimates for connection coefficients}}:
The connection coefficients from Subsubsects.\,\ref{SSS:CONNECTIONCOEFFICIENTS},
\ref{SSS:CONFORMALMETRIC}, and \ref{SSS:MODACOUSTICAL}
verify the following estimates:
\begin{subequations}
\begin{align}
	\|
		(\mytr_{\congsphere} \widetilde{\upchi}^{(Small)},
					\hat{\upchi}, 
					\upzeta
		)
	\|_{L_t^2 L_{\upomega}^p(\widetilde{\mathcal{C}}_u)},
	\,
	\|
		\rgeo 
		\angprojDarg{\Lunit} 
		(\mytr_{\congsphere} \widetilde{\upchi}^{(Small)},
					\hat{\upchi}, 
					\upzeta
		)
	\|_{L_t^2 L_{\upomega}^p(\widetilde{\mathcal{C}}_u)}
	& \lesssim \uplambda^{-1/2},
		\label{E:ACOUSTICALLT2LOMEGAPANDLDERIVATIVESALONGCONES} \\
	\|
		\rgeo^{1/2}
		(\mytr_{\congsphere} \widetilde{\upchi}^{(Small)},
					\hat{\upchi}, 
					\upzeta 
		)
	\|_{L_t^{\infty}L_{\upomega}^p(\widetilde{\mathcal{C}}_u)}
	& \lesssim \uplambda^{-1/2},
		\label{E:ACOUSTICALLTINFTYLOMEGAPALONGCONES} 
			\\
	\|
		\rgeo
		(
					\mytr_{\congsphere} \widetilde{\upchi}^{(Small)},
					\hat{\upchi}, 
					\upzeta
		)
	\|_{L_t^{\infty}L_{\upomega}^p(\widetilde{\mathcal{C}}_u)}
	& \lesssim \uplambda^{- 4 \upepsilon_0},
		\label{E:ASECONDACOUSTICALLTINFTYLOMEGAPALONGCONES} 
\end{align}
\end{subequations}

\begin{subequations}
\begin{align}
	\rgeo \mytr_{\congsphere} \widetilde{\upchi}
	& \approx 1,
		\label{E:TRCHILINFINITYESTIMATES}	 
		\\
	\| 
		\rgeo^{1/2} \mytr_{\congsphere} \widetilde{\upchi}^{(Small)} 
	\|_{L^{\infty}(\widetilde{\mathcal{M}})}
	& \lesssim \uplambda^{-1/2},
	\label{E:TRCHIMODLINFINITYESTIMATES}	
		\\
	\| 
		\rgeo^{3/2} \angD \mytr_{\congsphere} \widetilde{\upchi}^{(Small)} 
	\|_{L_t^{\infty} L_u^{\infty} L_{\upomega}^p (\widetilde{\mathcal{M}})}
	& \lesssim \uplambda^{-1/2}, 
		\label{E:ONEANGULARDERIVATIVEOFTRCHIMODLINFINITYESTIMATES}
		\\
	\| 
		\rgeo (\angD \mytr_{\congsphere} \widetilde{\upchi}^{(Small)}, \angD \hat{\upchi}) 
	\|_{L_t^2 L_{\upomega}^p (\widetilde{\mathcal{C}}_u)}
	& \lesssim \uplambda^{-1/2},
		\label{E:ONEANGULARDERIVATIVEOFTRCHIMODANDTRFREECHIL2INTIMEESTIMATES}
		\\
	\|
		(\mytr_{\congsphere} \widetilde{\upchi}^{(Small)}, \hat{\upchi},\upzeta)
	\|_{L_t^2 C_{\upomega}^{0,\updelta_0}(\widetilde{\mathcal{C}}_u)}
	& \lesssim \uplambda^{-1/2}.
	\label{E:TRICHIMODSMALLTRFREECHIANDTORSIONL2INTIMEALONGSOUNDCONES} 
\end{align}
\end{subequations}

In addition, the null lapse $\nulllapse$ defined in \eqref{E:NULLLAPSE}
verifies the following estimates:
\begin{align} \label{E:L2INTIMEESTIMATESFORNULLLAPSEALONGCONES}
	\left\|
		\frac{\nulllapse^{-1} - 1}{\rgeo}
	\right\|_{L_t^2 L_x^{\infty}(\widetilde{\mathcal{M}})},
		\,
	\left\|
		\frac{\nulllapse^{-1} - 1}{\rgeo^{1/2}}
	\right\|_{L_t^{\infty} L_u^{\infty} L_{\upomega}^{2p}(\widetilde{\mathcal{M}})},
		\,
	\left\|
		\rgeo (\angprojDarg{\Lunit},\angD)
		\left(
			\frac{\nulllapse^{-1} - 1}{\rgeo}
		\right)
	\right\|_{L_t^2 L_{\upomega}^p(\widetilde{\mathcal{C}}_u)}
	& \lesssim \uplambda^{-1/2}.
\end{align}

Furthermore, for any $u \in [-\frac{4}{5}\RescaledTboot,\RescaledTboot]$, 
$t \in [[u]_+,\RescaledTboot]$, and $\upomega \in \mathbb{S}^2$,
the Cartesian spatial components $\Lunit^i$ verify the following estimate:
\begin{align} \label{E:LINFINITYESTIMATESFORRECTANGULARSPATIALCOMPONENTSOFL}
	|
		\Lunit^i(t,u,\upomega)
		-
		\Lunit^i(0,0,\upomega)
	|
	& \lesssim \uplambda^{- 4\upepsilon_0}.
\end{align}

Moreover, for any smooth scalar-valued function of the type described in Subsubsect.\,\ref{SSS:ADDITIONALSCHEMATIC},
we have:
\begin{align} \label{E:HOLDERESTIMATESFORF1TYPEFACTOR}
	\| 
		\gensmoothfunction_{(\vec{\Lunit})}
	\|_{L_t^{\infty} L_u^{\infty} C_{\upomega}^{0,\updelta_0}(\widetilde{\mathcal{M}})}
	& \lesssim 1. 
\end{align}

Furthermore,
\begin{align} \label{E:CONNECTIONCOEFFICIENTESTIMATESNEEDEDTODERIVESPATIALLYLOCALIZEDDECAYFROMCONFORMALENERGYESTIMATE}
\left\| 
	\left(
	\mytr_{\congsphere} \widetilde{\upchi}^{(Small)}, \hat{\upchi}, \mytr_{\gsphere} \upchi - \frac{2}{\rgeo}
	\right)
	\right\|_{L_t^{\frac{q}{2}} L_u^{\infty} C_{\upomega}^{0,\updelta_0}(\widetilde{\mathcal{M}})}
	& \lesssim \uplambda^{\frac{2}{q} - 1 - 4 \upepsilon_0(\frac{4}{q} - 1)},
	&
	\| 
		\upzeta
	\|_{L_t^{\frac{q}{2}} L_x^{\infty}(\widetilde{\mathcal{M}})}
	& \lesssim \uplambda^{\frac{2}{q} - 1 - 4 \upepsilon_0(\frac{4}{q} - 1)}.
\end{align}

\medskip
\noindent \underline{\textbf{Improved estimates in the interior region}}:
We have the following improved\footnote{The most important improvement afforded by
\eqref{E:IMPROVEDININTERIORL2INTIMELINFINITYINSPACECONNECTIONCOFFICIENTS}
is that on the LHSs of the estimates, 
the $L_t^2$ norms are taken \emph{after} a spatial norm along constant-time hypersurfaces.
This is crucial for the proof of Theorem~\ref{T:BOUNDEDNESSOFCONFORMALENERGY} and
contrasts with, for example, 
the estimate \eqref{E:TRICHIMODSMALLTRFREECHIANDTORSIONL2INTIMEALONGSOUNDCONES}, 
in which \emph{only the angular} $C_{\upomega}^{0,\updelta_0}$ norm
is taken before the $L_t^2$ norm.} estimates\footnote{Our estimate \eqref{E:IMPROVEDININTERIORL2INTIMELINFINITYINSPACECONNECTIONCOFFICIENTS} involves H\"{o}lder norms
in the angular variables, while the analogous estimates in \cite{qW2017} involved weaker $L^{\infty}$-norms.
The reason for the discrepancy is that $L_{\upomega}^{\infty}(S_{t,u})$ bound for $\hat{\upchi}$ proved just below 
\cite{qW2017}*{Equation~(5.87)} relies on the invalid Calderon--Zygmund estimate 
$
\| \upxi \|_{L_{\upomega}^{\infty}(S_{t,u})}
	\lesssim
	\sum_{i=1,2}
	\| \mathfrak{F}_{(i)} \|_{L_{\upomega}^{\infty}(S_{t,u})}
	\ln\left(2 +
		\| \rgeo^{3/2} \angD \mathfrak{F}_{(i)} \|_{L_{\upomega}^{\leb}(S_{t,u})}
		\right)
		+
		\| \rgeo \mathfrak{G} \|_{L_{\upomega}^{\leb}(S_{t,u})}
$
for solutions to the elliptic PDE \eqref{E:ANGDIVSYMMETRICTRACEFREESTUTENSORFIELD}.
Unfortunately, this estimate cannot be correct because 
the power of $\rgeo^{3/2}$ on the RHS is not compatible with the natural
scaling of
\eqref{E:ANGDIVSYMMETRICTRACEFREESTUTENSORFIELD}
on Euclidean round spheres of radius $\rgeo$;
the natural scaling coefficient would be $\rgeo$, not $\rgeo^{3/2}$,
and the distinction is especially crucial near $\rgeo = 0$.
In particular, since the correct power is $\rgeo$, one cannot combine the 
correct Calderon--Zygmund estimate with the $\rgeo^{3/2}$-involving bound
\eqref{E:ONEANGULARDERIVATIVEOFTRCHIMODLINFINITYESTIMATES}
to obtain the estimate for $\hat{\upchi}$ stated in \cite{qW2017}*{Equation~(5.11)}.
For this reason, we use an alternate approach in deriving
some of the estimates for $\hat{\upchi}$, one that involves
H\"{o}lder norms in the angular variables and the corresponding Calderon--Zygmund estimate 
\eqref{E:ANGULARHOLDERHODGEESTIMATENODERIVATIVESONLHSANGDIVSYMMETRICTRACEFREESTUTENSORFIELD}.} 
in the interior region:
\begin{align} \label{E:IMPROVEDINTERIORL2INTIMEESTIMATESFORNULLLAPSEALONGCONES}
	\left\|
		\frac{\nulllapse^{-1} - 1}{\rgeo}
	\right\|_{L_t^2 L_x^{\infty}(\widetilde{\mathcal{M}}^{(Int)})}
	& \lesssim \uplambda^{-1/2 - 4 \upepsilon_0},
\end{align}

\begin{align} \label{E:ACOUSTICALLOMEGA2PLTINFTYALONGCONES}
\|
		\rgeo^{1/2}
		(	\mytr_{\congsphere} \widetilde{\upchi}^{(Small)},
					\hat{\upchi}, 
					\upzeta
	)
	\|_{L_{\upomega}^{2p} L_t^{\infty} (\widetilde{\mathcal{C}}_u)}
	& \lesssim \uplambda^{-1/2},
	&&
	\mbox{if } \widetilde{\mathcal{C}}_u \in \widetilde{\mathcal{M}}^{(Int)},
\end{align}

\begin{align} \label{E:IMPROVEDININTERIORL2INTIMELINFINITYINSPACECONNECTIONCOFFICIENTS}
	\| 
		(
			\mytr_{\congsphere} \widetilde{\upchi}^{(Small)}, \mytr_{\gsphere} \upchi - \frac{2}{\rgeo}, \hat{\upchi}
		)
	\|_{L_t^2 L_u^{\infty} C_{\upomega}^{0,\updelta_0}(\widetilde{\mathcal{M}}^{(Int)})}
	& \lesssim \uplambda^{- 1/2 - 3 \upepsilon_0},
	&
	\| 
		\upzeta
	\|_{L_t^2 L_x^{\infty}(\widetilde{\mathcal{M}}^{(Int)})}
	& \lesssim \uplambda^{- 1/2 - 3 \upepsilon_0}.
\end{align}

\medskip
\noindent \underline{\textbf{Estimates for the geometric angular coordinate components of $\gsphere$}}:
With $\stgsphere$ denoting the standard round metric on the Euclidean unit sphere $\mathbb{S}^2$,
we have
\begin{subequations}
\begin{align}
	\max_{A,B=1,2}
		\left\| 
			\left\lbrace\rgeo^{-2} 
				\gsphere\left(\frac{\partial}{\partial \upomega^A},\frac{\partial}{\partial \upomega^B} \right) 
				- 
				\stgsphere\left(\frac{\partial}{\partial \upomega^A},\frac{\partial}{\partial \upomega^B} \right)
			\right\rbrace
			\right\|_{L^{\infty}(\widetilde{\mathcal{M}})}
	& \lesssim \uplambda^{- 4 \upepsilon_0},	
		\label{E:COMPARISONWITHROUNDMETRICLINFINITYBOUNDS} 
			\\
	\max_{A,B,C=1,2}
			\left\| 
			\frac{\partial}{\partial \upomega^A}
			\left\lbrace\rgeo^{-2} 
				\gsphere\left(\frac{\partial}{\partial \upomega^B},\frac{\partial}{\partial \upomega^C} \right) 
				- 
				\stgsphere\left(\frac{\partial}{\partial \upomega^B},\frac{\partial}{\partial \upomega^C} \right)
			\right\rbrace
			\right\|_{L_{\upomega}^p L_t^{\infty}(\widetilde{\mathcal{C}}_u)}
	& \lesssim \uplambda^{- 4 \upepsilon_0}.
	\label{E:ONEANGULARDERIVATIVECOMPARISONWITHROUNDMETRICLPINANGLESLINFINITYINTIMEBOUNDS}
\end{align}
\end{subequations}

\medskip
\noindent \underline{\textbf{Estimates for $\volrat$ and $\nulllapse$}}:
The following estimates hold\footnote{We point out that we prove \eqref{E:STUVOLUMEFORMCOMPARISONWITHUNITROUNDMETRICVOLUMEFORM}--\eqref{E:NULLLAPSECLOSETOUNITY}
	independently in the proof of Prop.\,\ref{P:NORMCOMPARISONTRACEINEQUALITIESANDTRACEINEQUALITIES},
	which in turn plays a role in the proofs of the remaining estimates of Prop.\,\ref{P:MAINESTIMATESFOREIKONALFUNCTIONQUANTITIES}.
	It is only for convenience that we have restated \eqref{E:STUVOLUMEFORMCOMPARISONWITHUNITROUNDMETRICVOLUMEFORM}--\eqref{E:NULLLAPSECLOSETOUNITY} in
	Prop.\,\ref{P:MAINESTIMATESFOREIKONALFUNCTIONQUANTITIES}.} 
for the volume form ratio $\volrat$ defined in \eqref{E:RATIOOFVOLUMEFORMS} and the null lapse $\nulllapse$ defined in \eqref{E:NULLLAPSE}:
\begin{subequations}
\begin{align} 
		\volrat
		:= \frac{\sqrt{\mbox{\upshape det} \gsphere}}{\sqrt{\mbox{\upshape det} \stgsphere}}
		& 
		\approx 
		\rgeo^2,
		\label{E:STUVOLUMEFORMCOMPARISONWITHUNITROUNDMETRICVOLUMEFORM} 
		\\
	\|\nulllapse - 1 \|_{L^{\infty}(\widetilde{\mathcal{M}})} 
	& \lesssim \uplambda^{- 4 \upepsilon_0} 
	<
	\frac{1}{4}.
	\label{E:NULLLAPSECLOSETOUNITY}
\end{align}
\end{subequations}

Furthermore,
\begin{align} \label{E:CONESRATIOOFSPHEREVOLUMEFORMTORGEOSQUAREDTIMESROUNDVOLUMEFORML2INTIMELPINOMEGA}
	\|
		\rgeo^{\frac{1}{2}} \angD \ln\left(\rgeo^{-2} \volrat \right)
	\|_{L_t^{\infty} L_u^{\infty} L_{\upomega}^p(\widetilde{\mathcal{M}})},
		\,
	\|
		\angD \ln\left(\rgeo^{-2} \volrat \right)
	\|_{L_t^2 L_{\upomega}^p(\widetilde{\mathcal{C}}_u)},
		\,
	\|
		\rgeo \Lunit \angD \ln\left(\rgeo^{-2} \volrat \right)
	\|_{L_t^2 L_{\upomega}^p(\widetilde{\mathcal{C}}_u)}
	& \lesssim \uplambda^{-1/2}.
\end{align}

\medskip
\noindent \underline{\textbf{Estimates for $\upmu$ and $\angD \upzeta$}}:
The torsion defined in \eqref{E:TORSION} and the mass aspect function $\upmu$ defined in \eqref{E:MASSASPECT}
verify the following estimates:
\begin{align} \label{E:CONESRWEIGHTEDMASSASPECTANDANGULARDERIVATIVESOFTORSIONL2INTIMELPINOMEGA}
	\| (\rgeo \upmu, \rgeo \angD \upzeta) \|_{L_t^2 L_{\upomega}^p(\widetilde{\mathcal{C}}_u)}
	& \lesssim \uplambda^{-1/2}.
\end{align}
	
\medskip
\noindent \underline{\textbf{Interior region estimates for $\upsigma$}}:
	The conformal factor $\upsigma$ from Def.\,\ref{D:CONFORMALSTUFF}
	verifies the following estimates in the interior region:
	\begin{subequations}
	\begin{align}
		\| \rgeo^{\frac{1}{2}} \Lunit \upsigma \|_{L_t^{\infty} L_{\upomega}^{2p}(\widetilde{\mathcal{C}}_u)},
			\,
		\| \rgeo^{\frac{1}{2} - \frac{2}{p}} \angD \upsigma \|_{L_{\gsphere}^p L_t^{\infty}(\widetilde{\mathcal{C}}_u)},
			\,
		\| \rgeo^{\frac{1}{2}} \angD \upsigma \|_{L_{\upomega}^p L_t^{\infty}(\widetilde{\mathcal{C}}_u)},
			\,
		\| \angD \upsigma \|_{L_t^2 L_{\upomega}^p(\widetilde{\mathcal{C}}_u)}
		& \lesssim \uplambda^{-1/2},
			&&
			\mbox{if } \widetilde{\mathcal{C}}_u \subset \widetilde{\mathcal{M}}^{(Int)},
			\label{E:CONEFIRSTDERIVATIVEBOUNDSFORCONFORMALFACTOR} \\
		\| \upsigma \|_{L^{\infty}(\widetilde{\mathcal{M}}^{(Int)})}
		& \lesssim \uplambda^{-8 \upepsilon_0},
			 &&
			\label{E:LINFTYBOUNDSFORCONFORMALFACTOR}
				\\
		\| \rgeo^{-1/2} \upsigma \|_{L^{\infty}(\widetilde{\mathcal{M}}^{(Int)})}
		& \lesssim \uplambda^{-\frac{1}{2}-4 \upepsilon_0}.
			\label{E:RWEIGHTEDLINFTYBOUNDSFORCONFORMALFACTOR}
	\end{align}
	\end{subequations}
	

	\medskip
	\noindent \underline{\textbf{Interior region estimates for $\upsigma$, $\check{\upmu}$, $\widetilde{\upzeta}$, and $\angupmu$}}:
	The conformal factor $\upsigma$ from Def.\,\ref{D:CONFORMALSTUFF},
	the modified mass aspect function $\check{\upmu}$ defined in \eqref{E:MODMASSASPECT},
	and the modified torsion $\widetilde{\upzeta}$ defined in \eqref{E:MODTORSION}
	verify the following estimates in the interior region:
	\begin{subequations}
	\begin{align}
		\| \angD \upsigma \|_{L_u^2 L_t^2 C_{\upomega}^{0,\updelta_0} (\widetilde{\mathcal{M}}^{(Int)})},	
			\,
		\| (\rgeo \check{\upmu}, \rgeo \angD \widetilde{\upzeta}) \|_{L_u^2 L_t^2 L_{\upomega}^p (\widetilde{\mathcal{M}}^{(Int)})}
		& \lesssim \uplambda^{- 4 \upepsilon_0},
			\label{E:SPACETIMEL2INUANDTLPINOMEGAFORANGDSIGMAANDRGEOWEIGHTEDMODIFIEDMASSASPECTANDANGDMODIFIEDTORSION} 
				\\
		\| \rgeo^{\frac{3}{2}} \check{\upmu} \|_{L_u^2 L_t^{\infty} L_{\upomega}^p (\widetilde{\mathcal{M}}^{(Int)})}
		& \lesssim \uplambda^{- 4 \upepsilon_0}.
			\label{E:SPACETIMEL2INULINFINTYINTLPINOMEGAFORRGEOTHREEHAVESWEIGHTEDMODIFIEDMASSASPECT} 
	\end{align}
	\end{subequations}
	
	In addition, the one-form $\angupmu$, which satisfies the Hodge system \eqref{E:FURTHERMODOFMASSASPECT},
	verifies the following estimates:
	\begin{align} \label{E:PRELIMINARYANGUPMUSPACETIMEBOUNDS}
		\| (\rgeo \angD \angupmu, \angupmu) \|_{L_t^2 L_u^2 L_{\upomega}^p(\widetilde{\mathcal{M}}^{(Int)})},
			\,
		\| \angupmu \|_{L_t^2 L_u^2 L_{\upomega}^{\infty}(\widetilde{\mathcal{M}}^{(Int)})}
		& \lesssim 
			\uplambda^{- 4 \upepsilon_0}.
	\end{align}
	
	\medskip
	\noindent \underline{\textbf{Delicate decomposition of $\angD \upsigma$ and corresponding estimates in the interior region}}:
	Finally, in \allowbreak $\widetilde{\mathcal{M}}^{(Int)}$, we can decompose
	$\angD \upsigma$ into $S_{t,u}$-tangent one-forms as follows:
	\begin{align} \label{E:KEYANGCONFORMALFACTORALGEBRAICDECOMPOSITION}
		\angD \upsigma
		& = - \upzeta
			+
			(\widetilde{\upzeta} - \angupmu)
			+
			\angupmu_{(1)}
			+
			\angupmu_{(2)}.
	\end{align}
	In \eqref{E:KEYANGCONFORMALFACTORALGEBRAICDECOMPOSITION},
	$\upzeta$ is the torsion from \eqref{E:TORSION}, 
	$\widetilde{\upzeta}$ and $\angupmu$ are as in Def.\,\ref{D:MODIFIEDQUANTITIES},
	and $\angupmu_{(1)}$ and $\angupmu_{(2)}$ are as in \eqref{E:KEYANGMUALGEBRAICDECOMPOSITION}
	and are respectively solutions to the Hodge-transport systems
	\eqref{E:ANGDIVLDERIVATIVEANGMU1}--\eqref{E:ANGCURLLDERIVATIVEANGMU1}
	and
	\eqref{E:ANGDIVLDERIVATIVEANGMU2}--\eqref{E:ANGCURLLDERIVATIVEANGMU2} on $S_{t,u}$
	that satisfy the following asymptotic conditions near the cone-tip axis:
	\begin{align} \label{E:ANGUPMU1AND2CONETIPCONDITIONS}
		\rgeo \angupmu_{(1)}(t,u,\upomega),
			\,
		\rgeo \angupmu_{(2)}(t,u,\upomega)
		=
		\mathcal{O}(\rgeo) \mbox{as } t \downarrow u.
	\end{align}
	Moreover, the following bounds hold:
	\begin{subequations}
	\begin{align}
		\| \widetilde{\upzeta} - \angupmu \|_{L_t^2 L_x^{\infty}(\widetilde{\mathcal{M}}^{(Int)})},
			\,
		\| \angupmu_{(1)} \|_{L_t^2 L_x^{\infty}(\widetilde{\mathcal{M}}^{(Int)})}
		& \lesssim \uplambda^{- \frac{1}{2} - 3 \upepsilon_0},
			\label{E:ANGDUPSIGMAL2TLINFINITYSPACEPARTESTIMATE} 
				\\
		\| \angupmu_{(2)} \|_{L_u^2 L_t^{\infty} L_{\upomega}^{\infty}(\widetilde{\mathcal{M}}^{(Int)})}
		& \lesssim \uplambda^{- \frac{1}{2} - 4 \upepsilon_0}.
			\label{E:ANGDUPSIGMAL2ULINFINITYCONEESTIMATE}
	\end{align}
	\end{subequations}
	
\end{proposition}

\subsection{Assumptions, including bootstrap assumptions for the eikonal function quantities}
\label{SS:ASSUMPTIONS}
In this subsection, we recall some important results proved in previous sections
and state some bootstrap assumptions that will play a role in our proof of Prop.\,\ref{P:MAINESTIMATESFOREIKONALFUNCTIONQUANTITIES}.

\subsubsection{Restatement of assumptions and results from prior sections}
\label{SSS:RESTATMENTOFPRIOR}
From scaling considerations, it is straightforward to see that
\eqref{E:PARTITIONEDBOOTSTRICHARTZ}--\eqref{E:PARTITIONEDBOOTL2LINFINITYFIRSTDERIVATIVESOFVORTICITYBOOTANDNENTROPYGRADIENT}
imply that the \emph{rescaled} solution variables (as defined in Subsect.\,\ref{SS:RESCALEDSOLUTION} and under the conventions of Subsect.\,\ref{SS:NOMORELAMBDA}) 
verify the following bootstrap assumptions
(where $\updelta_0$, $\upepsilon_0$, and the other parameters in our analysis are defined in Subsect.\,\ref{SS:PARAMETERS}):
\begin{subequations}
\begin{align} \label{E:RESCALEDSTRICHARTZ}
		\|
			\pmb{\partial} \vec{\Psi}
		\|_{L_t^2 L_x^{\infty}(\widetilde{\mathcal{M}})}
		+
		\uplambda^{\updelta_0}
		\sqrt{
		\sum_{\upnu \geq 2}
		\upnu^{2 \updelta_0}
		\|
			P_{\upnu} \pmb{\partial} \vec{\Psi}
		\|_{L_t^2 L_x^{\infty}(\widetilde{\mathcal{M}})}^2
		}
	& \leq \uplambda^{-1/2 - 4 \upepsilon_0},
		\\
	\|
		\partial (\vec{\vortrenormalized}, \vec{\GradEnt}) 
	\|_{L_t^2 L_x^{\infty}(\widetilde{\mathcal{M}})}
	+
	\uplambda^{\updelta_0}
	\sqrt{
		\sum_{\upnu \geq 2}
		\upnu^{2 \updelta_0}
		\|
			P_{\upnu} \partial (\vec{\vortrenormalized}, \vec{\GradEnt})
		\|_{L_t^2 L_x^{\infty}(\widetilde{\mathcal{M}})}^2
		}
	& \leq \uplambda^{-1/2 - 4 \upepsilon_0}.
	\label{E:RESCALEDBOOTL2LINFINITYFIRSTDERIVATIVESOFVORTICITYBOOTANDNENTROPYGRADIENT}
\end{align}
\end{subequations}
We will use \eqref{E:RESCALEDSTRICHARTZ}--\eqref{E:RESCALEDBOOTL2LINFINITYFIRSTDERIVATIVESOFVORTICITYBOOTANDNENTROPYGRADIENT}
throughout the rest of Sect.\,\ref{S:ESTIMATESFOREIKONALFUNCTION}.
We will also use the bootstrap assumption \eqref{E:BOOTSOLUTIONDOESNOTESCAPEREGIMEOFHYPERBOLICITY}.
We clarify that, although the bootstrap assumption \eqref{E:BOOTSOLUTIONDOESNOTESCAPEREGIMEOFHYPERBOLICITY}
refers to the non-rescaled solution, it also implies that the rescaled solution
is contained in $\mathfrak{K}$ on the spacetime domain $\widetilde{\mathcal{M}}$. 
Moreover, we recall that we will assume that $\uplambda$ is sufficiently large;
that is, there exists a (non-explicit) $\Lambda_0 > 0$ such that all of our estimates hold
whenever $\uplambda \geq \Lambda_0$.
Moreover, throughout Sect.\,\ref{S:ESTIMATESFOREIKONALFUNCTION},
we will use the top-order energy estimates of
Prop.\,\ref{P:TOPORDERENERGYESTIMATES} along constant-time hypersurfaces 
and the energy estimates of Prop.\,\ref{P:ENERGYESTIMATESALONGNULLHYPERSURFACES} along acoustic null hypersurfaces
(both of which concern estimates for the non-rescaled solution variables, from which
estimates for the rescaled variables immediately follow via scaling considerations).

Next, for use throughout the rest of the article, 
we use \eqref{E:RESCALEDSTRICHARTZ}--\eqref{E:RESCALEDBOOTL2LINFINITYFIRSTDERIVATIVESOFVORTICITYBOOTANDNENTROPYGRADIENT},
the product estimate \eqref{E:FREQUENCYPROJECTEDLINFINITYSMOOTHFUNCTIONPRODUCTESTIMATE},
the energy estimates of Prop.\,\ref{P:TOPORDERENERGYESTIMATES},
and the harmonic analysis results mentioned in the proof discussion of Cor.\,\ref{C:HOLDERTYPESTRICHARTZESTIMATEFORWAVEVARIABLES}
to deduce the following estimates for the rescaled solution, 
valid for any smooth function $\gensmoothfunction$:
\begin{subequations}
\begin{align} \label{E:FIRSTDERIVATIVESOFMETRICSIMPLECONSEQUENCEOFRESCALEDSTRICHARTZ}
		\|
			\pmb{\partial} \gfour(\vec{\Psi})
		\|_{L_t^2 L_x^{\infty}(\widetilde{\mathcal{M}})}
	& \lesssim \uplambda^{-1/2 - 4 \upepsilon_0},
\end{align}
\begin{align}
	&
	\|	
		(\pmb{\partial} \vec{\Psi},\pmb{\partial} \vec{\vortrenormalized},\pmb{\partial} \vec{\GradEnt},\vec{\VortVort},\DivGradEnt) 
	\|_{L_t^2 L_x^{\infty}(\widetilde{\mathcal{M}})}
		\label{E:FREQUENCYSQUARESUMMEDSMOOTHFUNCTIONTIMESBASICVARIABLESSIMPLECONSEQUENCEOFRESCALEDSTRICHARTZ} 
		\\
		&
		\ \
		+
		\uplambda^{\updelta_0}
		\sqrt{
		\sum_{\upnu \geq 2}
		\upnu^{2 \updelta_0}
		\left\|
			P_{\upnu} 
			\left\lbrace
				\gensmoothfunction(\vec{\Psi},\vec{\vortrenormalized},\vec{\GradEnt})
				(\pmb{\partial} \vec{\Psi},\pmb{\partial} \vec{\vortrenormalized},\pmb{\partial} \vec{\GradEnt},\vec{\VortVort},\DivGradEnt)
			\right\rbrace
		\right\|_{L_t^2 L_x^{\infty}(\widetilde{\mathcal{M}})}^2
		}
		\lesssim \uplambda^{-1/2 - 4 \upepsilon_0},
	\notag
\end{align}
\begin{align}
	\left\|	
		\gensmoothfunction(\vec{\Psi},\vec{\vortrenormalized},\vec{\GradEnt})
		(\pmb{\partial} \vec{\Psi},\pmb{\partial} \vec{\vortrenormalized},\pmb{\partial} \vec{\GradEnt},\vec{\VortVort},\DivGradEnt) 
	\right\|_{L_t^2 C_x^{0,\updelta_0}(\widetilde{\mathcal{M}})}
	& \lesssim \uplambda^{-1/2 - 4 \upepsilon_0}.
	\label{E:RESCALEDSOLUTIONHOLDERESTIMATE}
	\end{align}
	\end{subequations}
We clarify that to obtain the bounds in \eqref{E:FREQUENCYSQUARESUMMEDSMOOTHFUNCTIONTIMESBASICVARIABLESSIMPLECONSEQUENCEOFRESCALEDSTRICHARTZ}
involving
$
\partial_t (\vec{\vortrenormalized}, \vec{\GradEnt})
$,
we use
\eqref{E:RESCALEDTRANSPORT}
to algebraically solve for 
$
\partial_t (\vec{\vortrenormalized}, \vec{\GradEnt})
$.
Moreover, to obtain the bounds in \eqref{E:FREQUENCYSQUARESUMMEDSMOOTHFUNCTIONTIMESBASICVARIABLESSIMPLECONSEQUENCEOFRESCALEDSTRICHARTZ}
involving $\vec{\VortVort}$ and $\DivGradEnt$,
we use \eqref{E:RESCALEDRENORMALIZEDCURLOFSPECIFICVORTICITY}--\eqref{E:RESCALEDRENORMALIZEDDIVOFENTROPY} to express
$(\vec{\VortVort},\DivGradEnt) = \gensmoothfunction(\vec{\Psi},\vec{\vortrenormalized},\vec{\GradEnt}) 
\cdot 
\partial 
(\vec{\Psi},\vec{\vortrenormalized},\vec{\GradEnt})$,
where $\gensmoothfunction$ is a schematically depicted smooth function.

\subsubsection{Bootstrap assumptions for the eikonal function quantities}
\label{SSS:BOOTSTRAPFOREIKONAL}
Recall that $p$ denotes the number we fixed in \eqref{E:BOUNDSONLEBESGUEEXPONENTP}.
We assume that
\begin{subequations}
\begin{align}
	\max_{A,B=1,2}
		\left\| 
			\left\lbrace\rgeo^{-2} 
				\gsphere\left(\frac{\partial}{\partial \upomega^A},\frac{\partial}{\partial \upomega^B} \right) 
				- 
				\stgsphere\left(\frac{\partial}{\partial \upomega^A},\frac{\partial}{\partial \upomega^B} \right)
			\right\rbrace
			\right\|_{L^{\infty}(\widetilde{\mathcal{M}})}
	& \leq \uplambda^{- \upepsilon_0},	
		\label{E:BOOTSTRAPMETRICAPPROXIMATELYROUND} \\
	\max_{A,B,C=1,2}
			\left\| 
			\frac{\partial}{\partial \upomega^A}
			\left\lbrace\rgeo^{-2} 
				\gsphere\left(\frac{\partial}{\partial \upomega^B},\frac{\partial}{\partial \upomega^C} \right) 
				- 
				\stgsphere\left(\frac{\partial}{\partial \upomega^B},\frac{\partial}{\partial \upomega^C} \right)
			\right\rbrace
			\right\|_{L_t^{\infty} L_{\upomega}^p(\widetilde{\mathcal{C}}_u)}
	& \leq \uplambda^{- \upepsilon_0}.
	\label{E:DERIVATIVESBOOTSTRAPMETRICAPPROXIMATELYROUND}
\end{align}
\end{subequations}

We also assume that for any $\widetilde{\mathcal{C}}_u \subset \widetilde{\mathcal{M}}$, we have
	\begin{align} 	\label{E:BOOTSTRAPCHIANDTORSIONALONGANYCONE} 
	\| (\mytr_{\congsphere} \widetilde{\upchi}^{(Small)}, \hat{\upchi}, \upzeta) \|_{L_t^2 C_{\upomega}^{0,\updelta_0}(\widetilde{\mathcal{C}}_u)}
	& \leq \uplambda^{-1/2 + 2 \upepsilon_0}.
	\end{align}

Moreover, we assume that for any $S_{t,u} \subset \widetilde{\mathcal{M}}$, we have
\begin{subequations}
\begin{align}
	\| \rgeo(\hat{\upchi},\mytr_{\congsphere} \widetilde{\upchi}^{(Small)},\upzeta) \|_{L_{\upomega}^p(S_{t,u})}
	& \leq 1,
	\label{E:BOOTSTRAPCHIBOOTSTRAPANYSPHERE} \\
	\|\nulllapse - 1 \|_{L_{\upomega}^{\infty}(S_{t,u})} & \leq \frac{1}{2}.
	\label{E:NULLLAPSEBOOSTRAP}
\end{align}
\end{subequations}

In addition, we assume that for every 
$u \in [-\frac{4}{5}\RescaledTboot,\RescaledTboot]$, 
$t \in [[u]_+,\RescaledTboot]$, and $\upomega \in \mathbb{S}^2$,
we have
\begin{align} \label{E:BOOTSTRAPFORRECTANGULARSPATIALCOMPONENTSOFL}
	|
		\Lunit^i(t,u,\upomega)
		-
		\Lunit^i(0,0,\upomega)
	|
	& \leq 1.
\end{align}

Finally, we assume that the following estimates hold in the interior region:
\begin{align} \label{E:LT2LINFINITYBOOTSTRAPCHIANDZETAININTERIORREGION} 
	\| (\hat{\upchi},\mytr_{\congsphere} \widetilde{\upchi}^{(Small)}) \|_{L_t^2 L_u^{\infty} C_{\upomega}^{0,\updelta_0}(\widetilde{\mathcal{M}}^{(Int)})}
	& \leq \uplambda^{-1/2},
	&
	\| \upzeta \|_{L_t^2 L_x^{\infty}(\widetilde{\mathcal{M}}^{(Int)})}
	& \leq \uplambda^{-1/2}.
\end{align}

\begin{remark}
Our bootstrap assumptions are similar to the ones in \cite{qW2017}*{Section~5},
except that for convenience, we have strengthened a few and included a few additional ones.
We also note that we derive a strict improvement of \eqref{E:BOOTSTRAPMETRICAPPROXIMATELYROUND} in \eqref{E:COMPARISONWITHROUNDMETRICLINFINITYBOUNDS},
of \eqref{E:DERIVATIVESBOOTSTRAPMETRICAPPROXIMATELYROUND} in \eqref{E:ONEANGULARDERIVATIVECOMPARISONWITHROUNDMETRICLPINANGLESLINFINITYINTIMEBOUNDS},
of \eqref{E:BOOTSTRAPCHIANDTORSIONALONGANYCONE} in \eqref{E:TRICHIMODSMALLTRFREECHIANDTORSIONL2INTIMEALONGSOUNDCONES},
of \eqref{E:BOOTSTRAPCHIBOOTSTRAPANYSPHERE} in \eqref{E:ASECONDACOUSTICALLTINFTYLOMEGAPALONGCONES},
of \eqref{E:NULLLAPSEBOOSTRAP} in \eqref{E:NULLLAPSECLOSETOUNITY},
of \eqref{E:BOOTSTRAPFORRECTANGULARSPATIALCOMPONENTSOFL} in \eqref{E:LINFINITYESTIMATESFORRECTANGULARSPATIALCOMPONENTSOFL},
and of
\eqref{E:LT2LINFINITYBOOTSTRAPCHIANDZETAININTERIORREGION} in \eqref{E:IMPROVEDININTERIORL2INTIMELINFINITYINSPACECONNECTIONCOFFICIENTS}.
\end{remark}

\subsection{Analytic tools}
\label{SS:ANALYTICTOOLS}
In this subsection, we record some inequalities that will play a role in the forthcoming analysis.
All of the results are the same as or simple consequences of results from \cite{qW2017}*{Section~5}.

\subsubsection{Norm comparisons, trace inequalities, and Sobolev inequalities}
\label{SSS:ANALYTICTOOLSANDGEOMETRICLITTLEWOODPALEY}

\begin{proposition}[Norm comparisons, trace inequalities, and Sobolev inequalities]
	\label{P:NORMCOMPARISONTRACEINEQUALITIESANDTRACEINEQUALITIES}
	Under the assumptions of Subsect.\,\ref{SS:ASSUMPTIONS}, the following estimates hold
	(see Subsect.\,\ref{SS:GEOMETRICNORMS} for the definitions of the norms).
	
	\medskip
	
	\noindent \underline{\textbf{Comparison of $S_{t,u}$-norms with different volume forms}}: 
	If $1 \leq \leb < \infty$,
	then for any $S_{t,u}$-tangent tensorfield $\upxi$, we have
	\begin{align} \label{E:STUNORMCOMPARISONDIFFERENTVOLUMEFORMS}
		\| \upxi \|_{L_{\gsphere}^{\leb}(S_{t,u})}
		& \approx 
		\| \rgeo^{\frac{2}{\leb}} \upxi \|_{L_{\upomega}^{\leb}(S_{t,u})}.
	\end{align}
	
	\medskip
	
	\noindent \underline{\textbf{Trace inequalities}}:
	For any $S_{t,u}$-tangent tensorfield $\upxi$, we have
	\begin{align} \label{E:L2SOBOLEVONSTU}
		\| \rgeo^{-1/2} \upxi \|_{L_{\gsphere}^2(S_{t,u})}
		+
		\| \upxi \|_{L_{\gsphere}^4(S_{t,u})}
		& \lesssim 
		\| \upxi \|_{H^1(\widetilde{\Sigma}_t)}.
	\end{align}
	
	\medskip
	
	\noindent \underline{\textbf{Sobolev and Morrey-type inequalities}}:
	For any $S_{t,u}$-tangent tensorfield $\upxi$, we have
	\begin{align} 
		\| \upxi \|_{L_u^2 L_{\upomega}^2(\widetilde{\Sigma}_t)}
		& \lesssim 
		\| \upxi \|_{H^1(\widetilde{\Sigma}_t)},
		\label{E:EUCLIDEANFORML2SOBOLEVONSIGMAT}
					\\
		\| \rgeo^{1/2} \upxi \|_{L_{\upomega}^{2p} L_t^{\infty}(\widetilde{\mathcal{C}}_u)}^2
		& \lesssim 
		\left\lbrace
			\| \rgeo \angprojDarg{\Lunit} \upxi \|_{L_{\upomega}^p L_t^2(\widetilde{\mathcal{C}}_u)}
			+
			\| \upxi \|_{L_{\upomega}^p L_t^2(\widetilde{\mathcal{C}}_u)}
		\right\rbrace
		\| \upxi \|_{L_{\upomega}^{\infty} L_t^2(\widetilde{\mathcal{C}}_u)}.
		\label{E:LOMEGA2PLTINFTYTRACEINEQUALITYNEEDEDFORCHIHAT}
	\end{align}
	
	Furthermore, if $2 < \leb < \infty$,
	then for any $S_{t,u}$-tangent tensorfield $\upxi$,
	we have
	\begin{align} \label{E:EUCLIDEANFORMLQSOBOLEVONSTU}
		\| \upxi \|_{L_{\upomega}^{\leb}(S_{t,u})}
		& \lesssim 
		\| \rgeo \angD \upxi \|_{L_{\upomega}^2(S_{t,u})}^{1 - \frac{2}{\leb}}
		\| \upxi \|_{L_{\upomega}^2(S_{t,u})}^{\frac{2}{\leb}}
		+
		\| \upxi \|_{L_{\upomega}^2(S_{t,u})}.
	\end{align}
	
	Moreover, if $2 < \leb \leq p$ (where $p$ is as in Subsect.\,\ref{SS:MAINESTIMATESFOREIKONALFUNCTIONQUANTITIES}), 
	then for any $S_{t,u}$-tangent tensorfield $\upxi$,
	we have
	\begin{align}
		\| \upxi \|_{C_{\upomega}^{0,1 - \frac{2}{\leb}}(S_{t,u})}
		& \lesssim 
		\| \rgeo \angD \upxi \|_{L_{\upomega}^{\leb}(S_{t,u})}
		+
		\| \upxi \|_{L_{\upomega}^2(S_{t,u})}.
		\label{E:EUCLIDEANFORMMORREYONSTU}
	\end{align}

	In addition, if $2 \leq \leb$,
	then for any $S_{t,u}$-tangent tensorfield $\upxi$, we have
	\begin{align} \label{E:L2QSOBOLEVONSIGMAT}
		\| \rgeo^{\frac{1}{2} - \frac{1}{\leb}} \upxi \|_{L_{\gsphere}^{2 \leb}L_u^{\infty}(\widetilde{\Sigma}_t)}^2
		& \lesssim 
		\left\lbrace
			\| \rgeo (\angprojDarg{\spherenormal},\angD) \upxi \|_{L_{\upomega}^{\leb} L_u^2(\widetilde{\Sigma}_t)}
			+
			\| \upxi \|_{L_{\upomega}^{\leb} L_u^2(\widetilde{\Sigma}_t)}
		\right\rbrace
		\| \upxi \|_{L_{\upomega}^{\infty}L_u^2(\widetilde{\Sigma}_t)}.
	\end{align}
	
	Finally, if $0 < 1 - \frac{2}{\leb} < \Sob - 2$,
	then for any scalar function $f$, we have
	\begin{align} \label{E:EUCLIDEANFORMSCALARFUNCTIONSIGMATSOBOLEV}
		\| \rgeo f \|_{L_u^2 L_{\upomega}^{\leb}(\widetilde{\Sigma}_t)}
		& \lesssim 
			\| f \|_{H^{\Sob-2}(\widetilde{\Sigma}_t)}.
	\end{align}

\end{proposition}

\begin{remark}[Silent use of \eqref{E:STUNORMCOMPARISONDIFFERENTVOLUMEFORMS}]
	Following the proof of the proposition,
	in the rest of the article, we will often use the estimate \eqref{E:STUNORMCOMPARISONDIFFERENTVOLUMEFORMS}
	without explicitly mentioning it. For example, 
	when deriving \eqref{E:FIRSTSTEPFORTORSIONESTIMATEL2INTIMELINFINITYINSPACECONNECTIONCOFFICIENTS},
	we silently use \eqref{E:STUNORMCOMPARISONDIFFERENTVOLUMEFORMS}
	when controlling the term $\| \rgeo^{1 - \frac{2}{\leb}} \mathfrak{G} \|_{L_{\gsphere}^{\leb}(S_{t,u})}$
	on the right-hand side of the Calderon--Zygmund estimate \eqref{E:LINFTYHODGEESTIMATENODERIVATIVESONLHSINVOLVINGCARTESIANCOMPONENTSONRHS}.
\end{remark}

\begin{proof}[Discussion of proof]
		To obtain the desired estimates, 
		we first note that the following bounds hold:
		$\volrat \approx \rgeo^2$
		and 
		$
		\|\nulllapse - 1 \|_{L^{\infty}(\widetilde{\mathcal{M}})} 
		\lesssim \uplambda^{- 4 \upepsilon_0} 
	\leq
	\frac{1}{4}
	$.
		These bounds follow from the proof of \cite{qW2017}*{Lemma~5.4},
		based on the transport equations 
		\eqref{E:EVOLUTIONNULLAPSEUSEEULER} and \eqref{E:LUNITVOLUMEFORMRGEOTOMINUSTWORESCALED},
		the initial conditions 
		\eqref{E:INTIALLAPSEANDSPHEREVOLUMEELEMENTESTIMATE}, 
		\eqref{E:CONNECTIONCOEFFICIENTS0LIMITSALONGTIP},
		and \eqref{E:SPHEREFINITELIMITSALONGTIP}
		(recall that $\nulllapse|_{\Sigma_0} = a$ and that $u|_{\Sigma_0} = -w$),
		and the bootstrap assumptions.
		The estimates in the proposition can be proved using only on these estimates for $\volrat$ and $\nulllapse - 1$
		and the bootstrap assumptions, especially 
		\eqref{E:BOOTSTRAPMETRICAPPROXIMATELYROUND}--\eqref{E:DERIVATIVESBOOTSTRAPMETRICAPPROXIMATELYROUND},
		which capture the fact that $\rgeo^{-2} \gsphere$ is close, in appropriate norms, to the standard round Euclidean metric.
		
		The desired bound \eqref{E:STUNORMCOMPARISONDIFFERENTVOLUMEFORMS} follows 
		from the estimate $\volrat \approx \rgeo^2$ and the definitions of the norms on the left- and right-hand sides.
		All of the remaining estimates follow from proofs given in other works,
		thanks to the bounds for $\volrat$ and $\nulllapse$ mentioned in the previous paragraph and the bootstrap assumptions; 
		for the reader's convenience, we now provide references. \eqref{E:L2SOBOLEVONSTU} follows from straightforward adaptations of the proofs of
		\cite{qW2012}*{Lemma~7.4} and \cite{qW2012}*{Equation~(7.4)}.
		\eqref{E:EUCLIDEANFORML2SOBOLEVONSIGMAT} follows from a standard adaptation of the proof of
		\cite{qW2012}*{Proposition~7.5}, together with \eqref{E:STUNORMCOMPARISONDIFFERENTVOLUMEFORMS} and \eqref{E:L2SOBOLEVONSTU}.
		The estimate \eqref{E:LOMEGA2PLTINFTYTRACEINEQUALITYNEEDEDFORCHIHAT} follows 
		from a straightforward adaptation of the proof of \cite{qW2012}*{Equation~(8.17)},
		where one uses $\rgeo^2$ in the role of $\volrat$;
		see also \cite{qW2014}*{Lemma~2.13}, in which an estimate 
		equivalent (taking into account \eqref{E:STUNORMCOMPARISONDIFFERENTVOLUMEFORMS})
		to \eqref{E:LOMEGA2PLTINFTYTRACEINEQUALITYNEEDEDFORCHIHAT} is stated.
		
		\eqref{E:EUCLIDEANFORMLQSOBOLEVONSTU}
		and \eqref{E:EUCLIDEANFORMMORREYONSTU}
		can be proved by first noting that the same estimates hold for 
		the round metric $\stgsphere$ on the Euclidean-unit sphere (with $\rgeo$ replaced by unity and $\angD$ replaced by the connection of $\stgsphere$),
		and then using the bootstrap assumptions
		\eqref{E:BOOTSTRAPMETRICAPPROXIMATELYROUND}--\eqref{E:DERIVATIVESBOOTSTRAPMETRICAPPROXIMATELYROUND}
		to conclude the desired estimates as ``perturbations'' of the corresponding ones for the round metric.
		We will give the details for \eqref{E:EUCLIDEANFORMMORREYONSTU}
		and omit the argument for \eqref{E:EUCLIDEANFORMLQSOBOLEVONSTU}, 
		which can be proved using similar arguments.
		Let $\angD$ denote the Levi-Civita connection of $\gsphere$, and let
		${^{(0)}\angD}$ denote the Levi-Civita connection of $\stgsphere$.
		Let $\Gamma$ schematically denote the Christoffel symbols of $\gsphere$ relative to the
		geometric angular coordinates, and let ${^{(0)}\Gamma}$ schematically denote the 
		corresponding Christoffel symbols of $\stgsphere$, 
		i.e., schematically, we have
		$\Gamma = (\gsphere^{-1})^{AB} 
		\frac{\partial}{\partial \upomega^C}
		\gsphere\left(\frac{\partial}{\partial \upomega^D},\frac{\partial}{\partial \upomega^E} \right)$
		and
		${^{(0)}\Gamma} = (\stgsphere^{-1})^{AB} 
		\frac{\partial}{\partial \upomega^C}
		\stgsphere\left(\frac{\partial}{\partial \upomega^D},\frac{\partial}{\partial \upomega^E} \right)$.
		Then schematically, relative to the geometric angular coordinates, 
		we have $\angD \upxi = {^{(0)} \angD} \upxi + (\Gamma - {^{(0)}\Gamma}) \cdot \upxi$.
		In view of Def.\,\ref{D:HOLDERNORMSINGEOMETRICANGULARVARIABLES}, we see that 
		the standard Morrey inequality on the round sphere for type $\binom{m}{n}$
		tensorfields $\upxi$ yields:
		$
		\| \rgeo^{n-m} \upxi \|_{C_{\upomega}^{0,1 - \frac{2}{\leb}}(S_{t,u})}
		\lesssim 
		\| |{^{(0)} \angD} \upxi|_{\stgsphere} \|_{L_{\upomega}^{\leb}(S_{t,u})}
		+
		\| |\upxi|_{\stgsphere} \|_{L_{\upomega}^2(S_{t,u})}
		$.
		Hence, multiplying both sides of this inequality by $\rgeo^{m-n}$ and
		using \eqref{E:BOOTSTRAPMETRICAPPROXIMATELYROUND},
		we find that
		$
		\|  \upxi \|_{C_{\upomega}^{0,1 - \frac{2}{\leb}}(S_{t,u})}
		\lesssim 
		\| \rgeo {^{(0)} \angD} \upxi \|_{L_{\upomega}^{\leb}(S_{t,u})}
		+
		\| \upxi \|_{L_{\upomega}^2(S_{t,u})}
		$
		and thus
		\begin{align} \label{E:PROOFSTEPEUCLIDEANFORMMORREYONSTU}
		\| \upxi \|_{C_{\upomega}^{0,1 - \frac{2}{\leb}}(S_{t,u})}
		&
		\lesssim 
		\| \rgeo \angD \upxi \|_{L_{\upomega}^{\leb}(S_{t,u})}
		+
		\| \upxi \|_{L_{\upomega}^2(S_{t,u})}
		+
		\| \rgeo^{m-n} | (\Gamma - {^{(0)}\Gamma}) \cdot \upxi|_{\stgsphere} \|_{L_{\upomega}^{\leb}(S_{t,u})}.
		\end{align}
		The bootstrap assumptions
		\eqref{E:BOOTSTRAPMETRICAPPROXIMATELYROUND}--\eqref{E:DERIVATIVESBOOTSTRAPMETRICAPPROXIMATELYROUND}
		imply that the last term on RHS~\eqref{E:PROOFSTEPEUCLIDEANFORMMORREYONSTU} satisfies the estimate
		\begin{align} \label{E:SECONDPROOFSTEPEUCLIDEANFORMMORREYONSTU}
		&
		\| \rgeo^{m-n} | (\Gamma - {^{(0)}\Gamma}) \upxi|_{\stgsphere} \|_{L_{\upomega}^{\leb}(S_{t,u})}
			\\
		& \lesssim 
		\sum_{A,B,C=1,2}
		\left\| 
			\Gamma_{A \ B}^{\ C} - {^{(0)}\Gamma}_{A \ B}^{\ C} 
		\right\|_{L_{\upomega}^{\leb}(S_{t,u})}
		\| \upxi \|_{L_{\upomega}^{\infty}(S_{t,u})}
			\notag \\
		& \lesssim 
		\sum_{A,B,C,D,E=1,2}
		\left\| (\gsphere^{-1})^{AB} 
		\frac{\partial}{\partial \upomega^C}
		\gsphere\left(\frac{\partial}{\partial \upomega^D},\frac{\partial}{\partial \upomega^E} \right) 
		- 
		(\stgsphere^{-1})^{AB} 
		\frac{\partial}{\partial \upomega^C}
		\stgsphere\left(\frac{\partial}{\partial \upomega^D},\frac{\partial}{\partial \upomega^E} \right)
		\right\|_{L_{\upomega}^{\leb}(S_{t,u})}
		\| \upxi \|_{L_{\upomega}^{\infty}(S_{t,u})}
		 \notag \\
		& \lesssim 
		\uplambda^{-\upepsilon_0}
		\| \upxi \|_{L_{\upomega}^{\infty}(S_{t,u})}.
			\notag
		\end{align}
		From \eqref{E:SECONDPROOFSTEPEUCLIDEANFORMMORREYONSTU}, 
		in view of Def.\,\ref{D:HOLDERNORMSINGEOMETRICANGULARVARIABLES},
		we see that if $\uplambda$ is sufficiently large,
		then we can absorb the last term on RHS~\eqref{E:PROOFSTEPEUCLIDEANFORMMORREYONSTU}
		back into the left, at the expense of doubling the (implicit) constants on the RHS.
		We have therefore proved \eqref{E:EUCLIDEANFORMMORREYONSTU}.
		
		The estimate \eqref{E:L2QSOBOLEVONSIGMAT}
		follows from a straightforward adaptation of the proof of \cite{qW2012}*{Equation~(8.17)},
		where one uses the geometric coordinate partial derivative vectorfield $\frac{\partial}{\partial u}$
		in the role of the vectorfield $\frac{\partial}{\partial t}$ and
		$\rgeo^2$ in the role of $\volrat$ (note also that $|\angprojDarg{\frac{\partial}{\partial u}} \upxi|_{\gsphere} 
		\lesssim |(\angprojDarg{\spherenormal},\angD) \upxi|_{\gsphere}$).
		Finally, we note that the estimate \eqref{E:EUCLIDEANFORMSCALARFUNCTIONSIGMATSOBOLEV}
		is proved as \cite{qW2017}*{Equation~(5.39)}
		as a consequence of \eqref{E:EUCLIDEANFORMLQSOBOLEVONSTU}--\eqref{E:EUCLIDEANFORMMORREYONSTU}.
		\end{proof}

\subsubsection{Hardy--Littlewood maximal function}
\label{SSS:MAXIMALFUNCTION}
If $f = f(t)$ is a scalar function defined on the interval $I$, 
then we define the corresponding Hardy--Littlewood maximal function $\mathcal{M}(f) = \mathcal{M}(f)(t)$ 
to be the following scalar function on $I$:
\begin{align} \label{E:MAXFUNCTIONDEF}
	\mathcal{M}(f)(t)
	& :=
	\sup_{t' \in I \cap (-\infty,t)}
	\frac{1}{|t - t'|}
	\int_{t'}^t
		f(\uptau)
	\, d \uptau.
\end{align}

We will use the following well-known estimate, valid for $1 < \leb \leq \infty$:
\begin{align} \label{E:STANDARDMAXIMALFUNCTIONLQESTIMATE}
	\| \mathcal{M}(f) \|_{L^{\leb}(I)}
	& \lesssim \| f \|_{L^{\leb}(I)}.
\end{align}

\subsubsection{Transport lemma}
\label{SSS:TRANSPORTLEMMA}
Many of the geometric quantities that we must estimate satisfy transport
equations along the integral curves of $\Lunit$. Our starting point
for the analysis of such quantities will often be based on the following standard 
``transport lemma.''

\begin{lemma}[Transport lemma]
\label{L:TRANSPORT}
Let $m$ be a constant, and let $\upxi$ and $\mathfrak{F}$ be $S_{t,u}$-tangent tensorfields
such that the following transport equation holds
along the null cone portion $\widetilde{\mathcal{C}}_u \subset \widetilde{\mathcal{M}}$:
\begin{align} \label{E:TRANSPORTEQUATIONINTRANSPORTLEMMA}
	\angprojDarg{\Lunit} \upxi
	+
	m \mytr_{\gsphere} \upchi \upxi
	& = \mathfrak{F}.
\end{align}
Then we have the following identities, where $\rgeo$ and $\volrat$ are defined in Subsubsect.\,\ref{SSS:METRICSANDVOLUMEFORMSINGEOMETRICCOORDINATES},
and we recall that $[u]_+ := \max \lbrace u,0 \rbrace$ (and thus $[u]_+$ denotes the minimum value of $t$ along $\widetilde{\mathcal{C}}_u$):
\begin{subequations}
\begin{align} \label{E:MAIDENTITYTRANSPORTLEMMA}
	\left[\volrat^m \upxi \right](t,u,\upomega)
	& 
	= 
	\lim_{\uptau \downarrow [u]_+}
	\left[\volrat^m \upxi \right](\uptau,u,\upomega)
	+
	\int_{[u]_+}^t
		\left[\volrat^m \mathfrak{F} \right](\uptau,u,\upomega)
	\, d \uptau,
		\\
	\left[\rgeo^{2m} \upxi \right](t,u,\upomega)
	& 
	= 
	\lim_{\uptau \downarrow [u]_+}
	\left[\rgeo^{2m} \upxi \right](\uptau,u,\upomega)
	+
	\int_{[u]_+}^t
		\left\lbrace
			\left[\rgeo^{2m} \mathfrak{F} \right](\uptau,u,\upomega)
			+
			m
			\left[\rgeo^m \left(\frac{2}{\rgeo} - \mytr_{\gsphere} \upchi \right) \upxi \right](\uptau,u,\upomega)
		\right\rbrace
	\, d \uptau.
	\label{E:SECONDMAIDENTITYTRANSPORTLEMMA}
\end{align}
\end{subequations}

Similarly, if $\upxi$, $\mathfrak{F}$, and $\mathfrak{G}$ are $S_{t,u}$-tangent tensorfields
such that the following transport equation holds:
\begin{align} \label{E:RGEOTRANSPORTEQUATIONINTRANSPORTLEMMA}
	\angprojDarg{\Lunit} \upxi
	+
	\frac{2m}{\rgeo} \upxi
	& =
		\mathfrak{G} \cdot \upxi
		+
		\mathfrak{F},
\end{align}
and if 
\begin{align} \label{E:TRANSPORTLEMMAINHOMOFACTORASSUMEDBOUND}
	\| \mathfrak{G} \|_{L_{\upomega}^{\infty} L_t^1(\widetilde{\mathcal{C}}_u)}
	& \leq C,
\end{align}
then under the assumptions of Subsect.\,\ref{SS:ASSUMPTIONS}, 
the following estimate holds 
(where the implicit constants in \eqref{E:MAIDRGEOINEQUALITYTRANSPORT} 
depend on the constant $C$ on RHS~\eqref{E:TRANSPORTLEMMAINHOMOFACTORASSUMEDBOUND}):
\begin{align} \label{E:MAIDRGEOINEQUALITYTRANSPORT}
	|\rgeo^{2m} \upxi|_{\gsphere}(t,u,\upomega)
	& 
	\lesssim
	\lim_{\uptau \downarrow [u]_+}
		|\rgeo^{2m} \upxi|_{\gsphere}(\uptau,u,\upomega)
	+
	\int_{[u]_+}^t
			|\rgeo^{2m} \mathfrak{F}|_{\gsphere}(\uptau,u,\upomega)
	\, d \uptau.
\end{align}

\end{lemma}

\begin{proof}[Discussion of proof]
The results are restatements of \cite{qW2017}*{Lemma~5.11}
and can be proved using the same arguments,
based on equation \eqref{E:EVOLUTIONVOLUMELEMENT}
and the estimate $\volrat \approx \rgeo^2$
noted in the proof of Prop.\,\ref{P:NORMCOMPARISONTRACEINEQUALITIESANDTRACEINEQUALITIES}.

\end{proof}

\subsection{Estimates for the fluid variables}
\label{SS:ESTIMATESFORFLUIDVARIABLES}
Recall that Prop.\,\ref{P:PDESMODIFIEDACOUSTICALQUANTITIES} provides the PDEs verified by the
geometric quantities under study and that some source terms in those PDEs depend on the
fluid variables. In Prop.\,\ref{P:ESTIMATESFORFLUIDVARIABLES}, we provide some estimates
that are useful for controlling the fluid variable source terms.
In particular, we use the estimates of Prop.\,\ref{P:ESTIMATESFORFLUIDVARIABLES}
in our proof of Lemma~\ref{L:NEWESTIMATESFORPROPMAINESTIMATESFOREIKONALFUNCTIONQUANTITIES},
which provides the main new estimates needed to prove Prop.\,\ref{P:MAINESTIMATESFOREIKONALFUNCTIONQUANTITIES}.

\begin{proposition}[Estimates for the fluid variables]
\label{P:ESTIMATESFORFLUIDVARIABLES}
Under the assumptions of Subsect.\,\ref{SS:ASSUMPTIONS},
for any $2 \leq \leb \leq p$ (where $p$ is as in \eqref{E:BOUNDSONLEBESGUEEXPONENTP}), 
the following estimates hold on $\widetilde{\mathcal{M}}$:
\begin{subequations}
\begin{align}
	\| 
		\pmb{\partial} (\vec{\Psi},\vec{\vortrenormalized},\vec{\GradEnt})
	\|_{L_u^2 L_{\upomega}^p(\widetilde{\Sigma}_t)},
		\,
	\| 
		\rgeo^{1/2} \pmb{\partial} (\vec{\Psi},\vec{\vortrenormalized},\vec{\GradEnt}) 
	\|_{L_u^{\infty} L_{\upomega}^{2p}(\widetilde{\Sigma}_t)}
	& \lesssim \uplambda^{-1/2},
		\label{E:FLUIDONEDERIVATIVESIGMAT} \\
	\| 
		\rgeo^{1 - \frac{2}{\leb}} \pmb{\partial}^2 (\vec{\Psi},\vec{\vortrenormalized},\vec{\GradEnt}) 
	\|_{L_u^2 L_{\gsphere}^{\leb}(\widetilde{\Sigma}_t)}
	& \lesssim \uplambda^{-1/2},
		\label{E:FLUIDTWODERIVATIVESSIGMAT} 
		\\
	\| 
		\pmb{\partial} (\vec{\Psi},\vec{\vortrenormalized},\vec{\GradEnt}) 
	\|_{L_t^2 L_{\upomega}^{\infty}(\widetilde{\mathcal{C}}_u)}
	& \lesssim \uplambda^{-1/2 - 4 \upepsilon_0}, 
	  \label{E:FLUIDONEDERIVATIVECONE} \\
	\| 
		\pmb{\partial} (\vec{\Psi},\vec{\vortrenormalized},\vec{\GradEnt}) 
	\|_{L_t^2 L_{\upomega}^p(\widetilde{\mathcal{C}}_u)}
	& \lesssim \uplambda^{-1/2 - 4 \upepsilon_0}, 
	  \label{E:LT2LPOMEGAFLUIDONEDERIVATIVECONE} \\
	\| 
		\rgeo \pmb{\partial} (\vec{\Psi},\vec{\vortrenormalized},\vec{\GradEnt}) 
	\|_{L_t^2 L_{\upomega}^p(\widetilde{\mathcal{C}}_u)}
	& \lesssim \uplambda^{1/2 - 12 \upepsilon_0}, 
	  \label{E:FULLRWEIGHTFLUIDONEDERIVATIVECONE} \\	
	\| 
		(\angD,\angprojDarg{\Lunit}) \pmb{\partial} \vec{\Psi} 
	\|_{L^2(\widetilde{\mathcal{C}}_u)},
		\,
	\| 
		\rgeo^{1 - \frac{2}{p}} (\angD,\angprojDarg{\Lunit})  \pmb{\partial} \vec{\Psi} 
	\|_{L_t^2 L_{\gsphere}^p(\widetilde{\mathcal{C}}_u)}
	& \lesssim \uplambda^{-1/2},
	\label{E:WAVEVARIABLESTANGENTIALSECONDDERIVATIVESCONE}
\end{align}
\end{subequations}

\begin{subequations}
\begin{align}
	\| 
		(\vec{\VortVort},\DivGradEnt) 
	\|_{L_u^2 L_{\upomega}^p(\widetilde{\Sigma}_t)},
		\,
	\| 
		\rgeo^{1/2} (\vec{\VortVort},\DivGradEnt) 
	\|_{L_u^{\infty} L_{\upomega}^{2p}(\widetilde{\Sigma}_t)}
	& \lesssim \uplambda^{-1/2},
		\label{E:MODFLUIDSIGMAT} \\
	\| 
		\rgeo^{1 - \frac{2}{\leb}} \pmb{\partial} (\vec{\VortVort},\DivGradEnt) 
	\|_{L_u^2 L_{\gsphere}^{\leb}(\widetilde{\Sigma}_t)}
	& \lesssim \uplambda^{-1/2},
		\label{E:MODFLUIDONEDERIVAIVESIGMAT} \\
	\| 
		(\vec{\VortVort},\DivGradEnt) 
	\|_{L_t^2 L_{\upomega}^{\infty}(\widetilde{\mathcal{C}}_u)}
	& \lesssim \uplambda^{-1/2 - 4 \upepsilon_0},
		\label{E:MODFLUIDCONE} \\
	\| 
		(\vec{\VortVort},\DivGradEnt) 
	\|_{L_t^2 L_{\upomega}^p(\widetilde{\mathcal{C}}_u)}
	& \lesssim \uplambda^{-1/2 - 4 \upepsilon_0},
		\label{E:LT2LOMEGAPMODFLUIDCONE} \\
	\| 
		\rgeo (\vec{\VortVort},\DivGradEnt) 
	\|_{L_t^2 L_{\upomega}^{\infty}(\widetilde{\mathcal{C}}_u)}
	& \lesssim \uplambda^{1/2 - 12 \upepsilon_0},
		\label{E:FULLRWEIGHTMODFLUIDCONE} \\
	\| 
		\pmb{\partial} (\vec{\VortVort},\DivGradEnt) 
	\|_{L^2(\widetilde{\mathcal{C}}_u)},
		\,
	\| 
		\rgeo^{1 - \frac{2}{p}} (\angD,\angprojDarg{\Lunit}) (\vec{\VortVort},\DivGradEnt)
	\|_{L_t^2 L_{\gsphere}^p(\widetilde{\mathcal{C}}_u)}
	& \lesssim \uplambda^{-1/2}.
	\label{E:MODFLUIDONEDERIVATIVECONE}
\end{align}
\end{subequations}

Moreover, for any smooth function $\gensmoothfunction$, we have
\begin{subequations}
\begin{align}
	\| 
		\pmb{\partial} (\vec{\Psi},\vec{\vortrenormalized},\vec{\GradEnt})
	\|_{L_u^2 L_{\upomega}^{\leb}(\widetilde{\Sigma}_t)},
		\,
	\| 
		\rgeo^{1/2} 
		\pmb{\partial} (\vec{\Psi},\vec{\vortrenormalized},\vec{\GradEnt}) 
	\|_{L_t^{\infty} L_u^{\infty} L_{\upomega}^{2 \leb}(\widetilde{\mathcal{M}})}
	& \lesssim \uplambda^{-1/2},
		\label{E:SMOOTHFUNCTIONTIMESFLUIDONEDERIVATIVESIGMAT} \\
	\left\| 
		\rgeo
		(\angD,\angprojDarg{\Lunit})
		\left\lbrace
			\gensmoothfunction(\vec{\Psi},\vec{\vortrenormalized},\vec{\GradEnt},\vec{\Lunit}) 
			\pmb{\partial} \vec{\Psi} 
		\right\rbrace
	\right\|_{L_t^2 L_{\upomega}^{\leb}(\widetilde{\mathcal{C}}_u)}
	& \lesssim \uplambda^{-1/2},
		\label{E:SMOOTHFUNCTIONTIMESWAVEVARIABLESTANGENTIALSECONDDERIVATIVESCONE} 
			\\
	\left\| 
		\rgeo
		\pmb{\partial}
		\left\lbrace
			\gensmoothfunction(\vec{\Psi},\vec{\vortrenormalized},\vec{\GradEnt}) 
			\pmb{\partial} (\vec{\Psi},\vec{\vortrenormalized},\vec{\GradEnt}) 
		\right\rbrace
	\right\|_{L_u^2 L_{\upomega}^{\leb}(\widetilde{\Sigma}_t)}
	& \lesssim \uplambda^{-1/2},
		\label{E:ONEDERIVATIVESMOOTHFUNCTIONTIMESONEDERIVATIVETIMESFLUIDSIGMAT}
\end{align}
\end{subequations}

\begin{subequations}
\begin{align}
	\| 
		(\vec{\VortVort},\DivGradEnt) 
	\|_{L_u^2 L_{\upomega}^{\leb}(\widetilde{\Sigma}_t)},
		\,
	\| 
		\rgeo^{1/2} 
		(\vec{\VortVort},\DivGradEnt) 
	\|_{L_t^{\infty} L_u^{\infty} L_{\upomega}^{2 \leb}(\widetilde{\mathcal{M}})}
	& \lesssim \uplambda^{-1/2},
		\label{E:SMOOTHFUNCTIONTIMESMODVARIABLESCONE} 
			\\
	\left\| 
		\rgeo
		(\angD,\angprojDarg{\Lunit})
		\left\lbrace
			\gensmoothfunction(\vec{\Psi},\vec{\vortrenormalized},\vec{\GradEnt},\vec{\Lunit}) 
			(\vec{\VortVort},\DivGradEnt) 
		\right\rbrace
	\right\|_{L_t^2 L_{\upomega}^{\leb}(\widetilde{\mathcal{C}}_u)}
	& \lesssim \uplambda^{-1/2},
		 \label{E:ONETANGENTIALDERIVATIVESMOOTHFUNCTIONTIMESMODVARIABLESCONE}  \\
	\left\| 
		\rgeo
		\pmb{\partial}
		\left\lbrace
			\gensmoothfunction(\vec{\Psi},\vec{\vortrenormalized},\vec{\GradEnt}) 
			(\vec{\VortVort},\DivGradEnt) 
		\right\rbrace
	\right\|_{L_u^2 L_{\upomega}^{\leb}(\widetilde{\Sigma}_t)}
	& \lesssim \uplambda^{-1/2},
	 \label{E:ONEDERIVATIVESMOOTHFUNCTIONTIMESMODVARIABLESSIGMAT}
\end{align}
\end{subequations}

\begin{align}
	\left\| 
		\rgeo^{1/2} \gensmoothfunction(\vec{\Psi},\vec{\Lunit}) \pmb{\partial} (\vec{\Psi},\vec{\vortrenormalized},\vec{\GradEnt}) 
	\right\|_{L_u^2 L_t^{\infty} L_{\upomega}^p(\widetilde{\mathcal{M}})}
	& \lesssim \uplambda^{- 4 \upepsilon_0}.
		\label{E:FLUIDONEDERIVATIVESPACETIME}
\end{align}

\end{proposition}

\begin{proof}[Discussion of the proof]
	Thanks to the assumptions of Subsect.\,\ref{SS:ASSUMPTIONS},
	the availability of the energy-elliptic estimates of Prop.\,\ref{P:TOPORDERENERGYESTIMATES},
	the estimates of Prop.\,\ref{P:ENERGYESTIMATESALONGNULLHYPERSURFACES} along null hypersurfaces
	(with $\widetilde{\mathcal{C}}_u$ in the role of $\mathcal{N}$ in Prop.\,\ref{P:ENERGYESTIMATESALONGNULLHYPERSURFACES}), 
	and Prop.\,\ref{P:NORMCOMPARISONTRACEINEQUALITIESANDTRACEINEQUALITIES},
	all estimates except for \eqref{E:FLUIDONEDERIVATIVESPACETIME}
	follow from the same arguments given in \cite{qW2017}*{Lemma~5.5}, 
	\cite{qW2017}*{Proposition~5.6},
	and
	\cite{qW2017}*{Lemma~5.7}.
	We clarify that, in view of definition \eqref{E:MIXEDNORMCU},
		\eqref{E:FLUIDONEDERIVATIVECONE}
		follows from the bootstrap assumptions
		\eqref{E:RESCALEDSTRICHARTZ}--\eqref{E:RESCALEDBOOTL2LINFINITYFIRSTDERIVATIVESOFVORTICITYBOOTANDNENTROPYGRADIENT}
		and the bound
		$
		\| 
			\pmb{\partial} (\vec{\Psi},\vec{\vortrenormalized},\vec{\GradEnt}) 
		\|_{L_{\upomega}^{\infty}(S_{t,u})}
		\leq
		\| 
			\pmb{\partial} (\vec{\Psi},\vec{\vortrenormalized},\vec{\GradEnt}) 
		\|_{L^{\infty}(\Sigma_t)}
		$,
		which implies that
		$
		\| 
			\pmb{\partial} (\vec{\Psi},\vec{\vortrenormalized},\vec{\GradEnt}) 
		\|_{L_t^2 L_{\upomega}^{\infty}(\widetilde{\mathcal{C}}_u)}
		\leq
		\| 
			\pmb{\partial} (\vec{\Psi},\vec{\vortrenormalized},\vec{\GradEnt}) 
		\|_{L_t^2 L_x^{\infty}(\widetilde{\mathcal{M}})}
		$.
		Similar remarks apply to \eqref{E:MODFLUIDCONE},
		where we take into account definitions
		\eqref{E:RESCALEDRENORMALIZEDCURLOFSPECIFICVORTICITY}--\eqref{E:RESCALEDRENORMALIZEDDIVOFENTROPY}
		and the remarks of Subsect.\,\ref{SS:NOMORELAMBDA}.
		Similar remarks apply \eqref{E:FULLRWEIGHTMODFLUIDCONE},
		where we take into account the bound \eqref{E:RGEOANDUBOUNDS} for $\rgeo$.
	We also refer the readers to the proof of \cite{qW2014}*{Proposition~2.6}
	for further details on the role that the energy-elliptic estimates
	and the estimates along null hypersurface play in the proof of Prop.\,\ref{P:ESTIMATESFORFLUIDVARIABLES}.
	To prove the remaining estimate \eqref{E:FLUIDONEDERIVATIVESPACETIME},
	we use
	\eqref{E:RGEOANDUBOUNDS}
	and
	\eqref{E:FLUIDONEDERIVATIVESIGMAT}
	to conclude that
	$
	\| 
		\rgeo^{1/2} \pmb{\partial} (\vec{\Psi},\vec{\vortrenormalized},\vec{\GradEnt}) 
	\|_{L_u^2 L_t^{\infty} L_{\upomega}^p(\widetilde{\mathcal{M}})}
	\lesssim
	\uplambda^{1/2 - 4 \upepsilon_0}
	\| 
		\rgeo^{1/2} \pmb{\partial} (\vec{\Psi},\vec{\vortrenormalized},\vec{\GradEnt}) 
	\|_{L_u^{\infty} L_t^{\infty} L_{\upomega}^p(\widetilde{\mathcal{M}})}
	\lesssim
	\uplambda^{- 4 \upepsilon_0}
	$
	as desired.
\end{proof}

\subsection{\texorpdfstring{The new estimates needed to prove Proposition~\ref{P:MAINESTIMATESFOREIKONALFUNCTIONQUANTITIES}}{The new estimates needed to prove 
Proposition ref P:MAINESTIMATESFOREIKONALFUNCTIONQUANTITIES}}
\label{SS:NEWESTIMATESFORPROPMAINESTIMATESFOREIKONALFUNCTIONQUANTITIES}
The following lemma provides
the main new estimates needed to prove Prop.\,\ref{P:MAINESTIMATESFOREIKONALFUNCTIONQUANTITIES};
the other estimates needed to prove Prop.\,\ref{P:MAINESTIMATESFOREIKONALFUNCTIONQUANTITIES}
were essentially derived in \cite{qW2017}.

\begin{lemma}[The new estimates needed to prove Prop.\,\ref{P:MAINESTIMATESFOREIKONALFUNCTIONQUANTITIES}]
	\label{L:NEWESTIMATESFORPROPMAINESTIMATESFOREIKONALFUNCTIONQUANTITIES}
	Under the assumptions of Prop.\,\ref{P:MAINESTIMATESFOREIKONALFUNCTIONQUANTITIES},
	the following estimates hold whenever\footnote{The estimates \eqref{E:LAMBDAINVERSELINEARTERMLTQLUINFITYLUPOMEAGINFTY} and 
	\eqref{E:RGEOLAMBDAINVERSELINEARTERMLTQOVER2LUINFTYLOMEGAP}
	in fact hold whenever $q \geq 2$, but in proving  Props.\,\ref{P:MAINESTIMATESFOREIKONALFUNCTIONQUANTITIES} and 
	\ref{P:SPATIALLYLOCALIZEDREDUCTIONOFPROOFOFTHEOREMDECAYESTIMATE}, 
	we need these estimates only when $q > 2$ is sufficiently close to $2$.} 
	$q > 2$ is sufficiently close to $2$,
	where $p$ is defined in \eqref{E:BOUNDSONLEBESGUEEXPONENTP}, 
	and we recall that $[u]_+ := \max \lbrace u,0 \rbrace$ (and thus $[u]_+$ denotes the minimum value of $t$ along $\widetilde{\mathcal{C}}_u$).
	
	\medskip
	
	\noindent \underline{\textbf{Estimates for time-integrated terms}}:
	\begin{subequations}
	\begin{align}
	\uplambda^{-1} 
	\left\|
		\frac{1}{\rgeo(t,u)}
		\int_{[u]_+}^t
			|
				\rgeo (\vec{\VortVort},\DivGradEnt)
			|(\uptau,u,\upomega)
	\, d \uptau
\right\|_{L_t^2 L_x^{\infty}(\widetilde{\mathcal{M}})}
& \lesssim \uplambda^{-1/2 - 12 \upepsilon_0},
		\label{E:LAMBDAINVERSELINEARTERMTIMEINTEGRALLT2LUINFTIYLOMEGAINFTY} 
			\\
	\uplambda^{-1} 
\left\|
	\frac{1}{\rgeo^2(t,u)}
	\int_{[u]_+}^t
			|
				\rgeo^2 (\vec{\VortVort},\DivGradEnt)
			|(\uptau,u,\upomega)
	\, d \uptau
\right\|_{L_t^2 L_x^{\infty}(\widetilde{\mathcal{M}})}
& \lesssim
	\uplambda^{-1/2 - 12 \upepsilon_0},
	\label{E:ANOTHERLAMBDAINVERSELINEARTERMLT2LUPOMEAGINFTY} 
		\\
		\uplambda^{-1} 
\left\|
	\frac{1}{\rgeo^2(t,u)}
	\int_{[u]_+}^t
			|
				\rgeo^2 (\vec{\VortVort},\DivGradEnt)
			|(\uptau,u,\upomega)
	\, d \uptau
\right\|_{L_t^{\frac{q}{2}} L_x^{\infty}(\widetilde{\mathcal{M}})}
& \lesssim
	\uplambda^{\frac{2}{q} - 1 - 4 \upepsilon_0(\frac{4}{q} + 2)},
	\label{E:LAMBDAINVERSELINEARTERMLTQLUINFITYLUPOMEAGINFTY} 
		\\
\uplambda^{-1}
	\left\|	
	\frac{1}{\rgeo^{1/2}(t,u,\upomega)}
	\int_{[u]_+}^t
			|
				\rgeo (\vec{\VortVort},\DivGradEnt)
			|(\uptau,u,\upomega)
	\, d \uptau
\right\|_{L_t^{\infty} L_u^{\infty} L_{\upomega}^p(\widetilde{\mathcal{M}})}
& \lesssim \uplambda^{-1/2 - 12 \upepsilon_0},
	\label{E:LAMBDAINVERSELINEARTERMTIMEINTEGRALLUINFTIYLTINFTYLOMEGAP}
	\\
\uplambda^{-1}
	\left\|	
	\frac{1}{\rgeo^{3/2}(t,u,\upomega)}
	\int_{[u]_+}^t
			|
				\rgeo^2 (\vec{\VortVort},\DivGradEnt)
			|(\uptau,u,\upomega)
	\, d \uptau
\right\|_{L_t^{\infty} L_u^{\infty} L_{\upomega}^p(\widetilde{\mathcal{M}})}
& \lesssim \uplambda^{-1/2 - 12 \upepsilon_0},
	\label{E:SECONDLAMBDAINVERSELINEARTERMTIMEINTEGRALLUINFTIYLTINFTYLOMEGAP}
	\\
\uplambda^{-1} 
\left\|
	\frac{1}{\rgeo(t,u)}
	\int_{[u]_+}^t
			|
				\rgeo^2 (\vec{\VortVort},\DivGradEnt)
			|(\uptau,u,\upomega)
	\, d \uptau
\right\|_{L^{\infty}(\widetilde{\mathcal{M}})}
& \lesssim
	\uplambda^{-16 \upepsilon_0},
	\label{E:LAMBDAINVERSELINEARTERMLINFTYWHOLESPACETIME} 
	\\
\uplambda^{-1} 
\left\|
	\frac{1}{\rgeo^{3/2}(t,u)}
	\int_{[u]_+}^t
			|
				\rgeo^2 (\vec{\VortVort},\DivGradEnt)
			|(\uptau,u,\upomega)
	\, d \uptau
\right\|_{L^{\infty}(\widetilde{\mathcal{M}})}
& \lesssim
	\uplambda^{-1/2-12 \upepsilon_0},
	\label{E:ASECONDLAMBDAINVERSELINEARTERMLINFTYWHOLESPACETIME} 
\\
\uplambda^{-1} 
\left\|
	\frac{1}{\rgeo^2(t,u)}
	\int_{[u]_+}^t
			|
				\rgeo^2 (\vec{\VortVort},\DivGradEnt)
			|(\uptau,u,\upomega)
	\, d \uptau
\right\|_{L^{\infty}(\widetilde{\mathcal{M}})}
& \lesssim
	\uplambda^{-1 -8 \upepsilon_0},
	\label{E:ATHIRDLAMBDAINVERSELINEARTERMLINFTYWHOLESPACETIME} 
\end{align}
\end{subequations}

\begin{subequations}
\begin{align}
\uplambda^{-1} 
\left\|	
	\frac{1}{\rgeo^{3/2}(t,u)}
	\int_{[u]_+}^t
			|
				\rgeo^3 \angD(\vec{\VortVort},\DivGradEnt)
			|_{\gsphere}(\uptau,u,\upomega)
	\, d \uptau
\right\|_{L_t^{\infty} L_u^{\infty} L_{\upomega}^p(\widetilde{\mathcal{M}})}
& \lesssim
\uplambda^{-1/2 - 8 \upepsilon_0},	
		\label{E:LAMBDAINVERSESECONDANGULARDERIVATIVELINEARTERMTIMEINTEGRALLTINFTTYLUINFTYLOMEGAP}
	\\
\uplambda^{-1} 
\left\|	
	\frac{1}{\rgeo^2(t,u)}
	\int_{[u]_+}^t
			|
				\rgeo^3 \angD(\vec{\VortVort},\DivGradEnt)
			|_{\gsphere}(\uptau,u,\upomega)
	\, d \uptau
\right\|_{L_t^2 L_u^{\infty} L_{\upomega}^p(\widetilde{\mathcal{M}})}
& \lesssim
\uplambda^{- 1/2 - 8 \upepsilon_0},	
		\label{E:ASECONDLAMBDAINVERSESECONDANGULARDERIVATIVELINEARTERMTIMEINTEGRALLTINFTTYLUINFTYLOMEGAP}
\end{align}
\end{subequations}

\begin{subequations}
\begin{align}
	&
	\uplambda^{-1}
	\left\|
		\frac{1}{\rgeo^{3/2}(t,u)}
		\int_{[u]_+}^t
			\left|
				\rgeo^3
				(\vec{\GradEnt} \cdot \pmb{\partial} \vec{\Psi},\pmb{\partial} \vec{\Psi},\pmb{\partial} \vec{\vortrenormalized},\pmb{\partial} \vec{\GradEnt}) 
				\cdot 
				(\pmb{\partial} \vec{\Psi},\mytr_{\congsphere} \widetilde{\upchi}^{(Small)},\hat{\upchi},\upzeta,\rgeo^{-1})
			\right|_{\gsphere}(\uptau,u,\upomega)
	\, d \uptau
	\right\|_{L_t^{\infty} L_u^{\infty} L_{\upomega}^p(\widetilde{\mathcal{M}})}
		\label{E:LAMBDAINVERSEQUADRATICTERMTIMEINTEGRALLTINFTTYLUINFTYLOMEGAP}
		\\
	& \lesssim \uplambda^{-1/2 - 12 \upepsilon_0},
	\notag
		\\
	&
	\uplambda^{-1}
	\left\|
		\frac{1}{\rgeo^2(t,u)}
		\int_{[u]_+}^t
			\left|
				\rgeo^3
				(\vec{\GradEnt} \cdot \pmb{\partial} \vec{\Psi}, \pmb{\partial} \vec{\Psi},\pmb{\partial} \vec{\vortrenormalized},\pmb{\partial} \vec{\GradEnt}) 
				\cdot 
				(\pmb{\partial} \vec{\Psi},\mytr_{\congsphere} \widetilde{\upchi}^{(Small)},\hat{\upchi},\upzeta,\rgeo^{-1})
			\right|_{\gsphere}(\uptau,u,\upomega)
	\, d \uptau
	\right\|_{L_t^2 L_u^{\infty} L_{\upomega}^p(\widetilde{\mathcal{M}})}
		\label{E:ASECONDLAMBDAINVERSEQUADRATICTERMTIMEINTEGRALLTINFTTYLUINFTYLOMEGAP} 
		\\
	& \lesssim \uplambda^{-1/2 - 12 \upepsilon_0},
	\notag
\end{align}
\end{subequations}

\begin{subequations}
\begin{align}
&
\uplambda^{-1} 
\left\|	
	\frac{1}{\rgeo(t,u)}
	\int_{[u]_+}^t
			\left|
			\rgeo^2
			(\vec{\GradEnt} \cdot \pmb{\partial} \vec{\Psi},\pmb{\partial} \vec{\Psi},\pmb{\partial} \vec{\vortrenormalized},\pmb{\partial} \vec{\GradEnt}) 
			\cdot 
			(\pmb{\partial} \vec{\Psi},\mytr_{\congsphere} \widetilde{\upchi}^{(Small)},\hat{\upchi},\upzeta,\rgeo^{-1})	
		\right|_{\gsphere}(\uptau,u,\upomega)
	\, d \uptau
\right\|_{L_u^2 L_t^2 L_{\upomega}^p (\widetilde{\mathcal{M}})}
	\label{E:LAMBDAINVERSEQUADRATICTERMTIMEINTEGRALLT2LU2LOMEGAP}
	\\
& \lesssim
\uplambda^{- 16 \upepsilon_0},
		\notag \\
&
\uplambda^{-1} 
\left\|	
	\frac{1}{\rgeo^{1/2}(t,u)}
	\int_{[u]_+}^t
			\left|
			\rgeo^2
			(\vec{\GradEnt} \cdot \pmb{\partial} \vec{\Psi},\pmb{\partial} \vec{\Psi},\pmb{\partial} \vec{\vortrenormalized},\pmb{\partial} \vec{\GradEnt}) 
			\cdot 
			(\pmb{\partial} \vec{\Psi},\mytr_{\congsphere} \widetilde{\upchi}^{(Small)},\hat{\upchi},\upzeta,\rgeo^{-1})	
			\right|_{\gsphere}(\uptau,u,\upomega)
	\, d \uptau
\right\|_{L_u^2 L_t^{\infty} L_{\upomega}^p (\widetilde{\mathcal{M}})}
	\label{E:LAMBDAINVERSEQUADRATICTERMTIMEINTEGRALLU2LTINFTYLOMEGAP}
	\\
& \lesssim
\uplambda^{- 16 \upepsilon_0},
	\notag
\end{align}
\end{subequations}	

\begin{subequations}
\begin{align}
\uplambda^{-1} 
\left\|	
	\frac{1}{\rgeo(t,u)}
	\int_{[u]_+}^t
			|
				\rgeo^2 (\pmb{\partial}\vec{\VortVort},\pmb{\partial}\DivGradEnt)
			|(\uptau,u,\upomega)
	\, d \uptau
\right\|_{L_u^2 L_t^2 L_{\upomega}^p (\widetilde{\mathcal{M}})}
& \lesssim
\uplambda^{- 12 \upepsilon_0},	
	\label{E:LAMBDAINVERSETIMEINTEGRALLT2LU2LOMEGAP}
	\\
\uplambda^{-1} 
\left\|	
	\frac{1}{\rgeo^{1/2}(t,u)}
	\int_{[u]_+}^t
			|
				\rgeo^2 (\pmb{\partial}\vec{\VortVort},\pmb{\partial}\DivGradEnt)
			|(\uptau,u,\upomega)
	\, d \uptau
\right\|_{L_u^2 L_t^{\infty} L_{\upomega}^p (\widetilde{\mathcal{M}})}
& \lesssim
\uplambda^{-12 \upepsilon_0}.
	\label{E:LAMBDAINVERSETIMEINTEGRALLU2LTINFTYLOMEGAP}
\end{align}
\end{subequations}

\medskip

\noindent \underline{\textbf{Standard spacetime norm estimates}}:
	\begin{subequations}
	\begin{align}
		\uplambda^{-1} 
		\|
			\rgeo (\vec{\VortVort},\DivGradEnt)
		\|_{L_t^2 L_u^{\infty} L_{\upomega}^p(\widetilde{\mathcal{M}})}
		& \lesssim \uplambda^{-1/2 - 8 \upepsilon_0},
		\label{E:RGEOLAMBDAINVERSELINEARTERMLT2LUINFTYLOMEGAP}
			\\
		\uplambda^{-1} 
		\|
			\rgeo (\vec{\VortVort},\DivGradEnt)
		\|_{L_t^{\frac{q}{2}} L_u^{\infty} L_{\upomega}^p(\widetilde{\mathcal{M}})}
		& \lesssim \uplambda^{\frac{2}{q} - 1 - 4 \upepsilon_0(\frac{4}{q} + 1)},
		\label{E:RGEOLAMBDAINVERSELINEARTERMLTQOVER2LUINFTYLOMEGAP}
			\\
		\uplambda^{-1} 
		\|
			\rgeo (\vec{\VortVort},\DivGradEnt)
		\|_{ L_u^2 L_t^2 L_{\upomega}^p(\widetilde{\mathcal{M}})}
		& \lesssim \uplambda^{- 12 \upepsilon_0},
		\label{E:RGEOLAMBDAINVERSELINEARTERMLU2LT2LOMEGAP}
			\\
		\uplambda^{-1}
		\|
			\rgeo \pmb{\partial}(\vec{\VortVort},\DivGradEnt) 
		\|_{L_u^2 L_t^1 L_{\upomega}^p(\widetilde{\mathcal{M}})}
		& \lesssim
		\uplambda^{-\frac{1}{2} - 8 \upepsilon_0},
		\label{E:RGEOLAMBDAINVERSEONEDERIVATIVEOFMODFLUIDLU2LT1LOMEGAP} \\
		\uplambda^{-1}
		\|
			\rgeo 
			(\vec{\GradEnt} \cdot \pmb{\partial} \vec{\Psi},\pmb{\partial} \vec{\Psi},\pmb{\partial} \vec{\vortrenormalized},\pmb{\partial} \vec{\GradEnt}) 
			\cdot 
			(\pmb{\partial} \vec{\Psi},\mytr_{\congsphere} \widetilde{\upchi}^{(Small)},\hat{\upchi},\upzeta,\rgeo^{-1})	
		\|_{L_u^2 L_t^1 L_{\upomega}^p(\widetilde{\mathcal{M}})}
		& \lesssim \uplambda^{-\frac{1}{2} - 10 \upepsilon_0}.
		\label{E:RGEOLAMBDAINVERSEQUADTERMLU2LT1LOMEGAP}
	\end{align}
	\end{subequations}
	
	\end{lemma}

\begin{proof}
Throughout the proof, we silently use the simple bound $\rgeo(\uptau,u)/\rgeo(t,u) \lesssim 1$ for $\uptau \leq t$.

To prove \eqref{E:LAMBDAINVERSELINEARTERMTIMEINTEGRALLT2LUINFTIYLOMEGAINFTY},
we first use \eqref{E:RESCALEDBOOTBOUNDS} 
and the bound
$
		\| 
			(\vec{\VortVort},\DivGradEnt) 
		\|_{L_{\upomega}^{\infty}(S_{t,u})}
		\leq
		\| 
			(\vec{\VortVort},\DivGradEnt)
		\|_{L^{\infty}(\Sigma_t)}
		$
to deduce that
\begin{align}
\notag
\mbox{LHS~\eqref{E:LAMBDAINVERSELINEARTERMTIMEINTEGRALLT2LUINFTIYLOMEGAINFTY}}
& \lesssim
\uplambda^{-1} 
\left\|	
	\int_{[u]_+}^t
			|
				(\vec{\VortVort},\DivGradEnt)
			|(\uptau,u,\upomega)
	\, d \uptau
\right\|_{L_t^2 L_u^{\infty} L_{\upomega}^{\infty}(\widetilde{\mathcal{M}})}
\\
\notag
&
\lesssim \uplambda^{-1/2 - 4 \upepsilon_0}
\left\|	
	\int_{[u]_+}^t
			\|
				(\vec{\VortVort},\DivGradEnt)
			\|_{L_{\upomega}^{\infty}(S_{t,u})}
	\, d \uptau
\right\|_{L_t^{\infty} L_u^{\infty}}
\\
\notag
&
\lesssim \uplambda^{-1/2 - 4 \upepsilon_0}
\|
	(\vec{\VortVort},\DivGradEnt)
\|_{L_t^1 L_x^{\infty}(\widetilde{\mathcal{M}})}.
\end{align}
Using 
\eqref{E:RESCALEDBOOTBOUNDS}
and \eqref{E:FREQUENCYSQUARESUMMEDSMOOTHFUNCTIONTIMESBASICVARIABLESSIMPLECONSEQUENCEOFRESCALEDSTRICHARTZ},
we bound the RHS of the previous expression by
\begin{align}
\notag
\lesssim 
\uplambda^{- 8 \upepsilon_0}
\|	
	(\vec{\VortVort},\DivGradEnt)
\|_{L_t^2 L_x^{\infty}(\widetilde{\mathcal{M}})}
\lesssim 
\uplambda^{-1/2 - 12 \upepsilon_0}
\end{align}
as desired.

The estimate \eqref{E:ANOTHERLAMBDAINVERSELINEARTERMLT2LUPOMEAGINFTY}
can be proved using an argument that is nearly identical to the one we used to prove 
\eqref{E:LAMBDAINVERSELINEARTERMTIMEINTEGRALLT2LUINFTIYLOMEGAINFTY},
and we therefore omit the details.

To prove \eqref{E:LAMBDAINVERSELINEARTERMLTQLUINFITYLUPOMEAGINFTY},
we argue as above to deduce that
\begin{align}
\notag
\mbox{LHS~\eqref{E:LAMBDAINVERSELINEARTERMLTQLUINFITYLUPOMEAGINFTY}}
& \lesssim
\uplambda^{-1} 
\left\|	
	\int_{[u]_+}^t
			|
				(\vec{\VortVort},\DivGradEnt)
			|(\uptau,u,\upomega)
	\, d \uptau
\right\|_{L_t^{\frac{q}{2}} L_u^{\infty} L_{\upomega}^{\infty}(\widetilde{\mathcal{M}})}
\\
\notag
&
\lesssim
\uplambda^{-1}
(\uplambda^{1-8 \upepsilon_0})^{\frac{2}{q}}
\left\|	
	\int_{[u]_+}^t
			\|
				(\vec{\VortVort},\DivGradEnt)
			\|_{L_{\upomega}^{\infty}(S_{t,u})}
	\, d \uptau
\right\|_{L_t^{\infty} L_u^{\infty}}
\\
\notag
&
\lesssim 
\uplambda^{-1 + \frac{2}{q} - \frac{16}{q} \upepsilon_0}
\|	
	(\vec{\VortVort},\DivGradEnt)
\|_{L_t^1 L_x^{\infty}(\widetilde{\mathcal{M}})}.
\end{align}
Using 
\eqref{E:RESCALEDBOOTBOUNDS}
and \eqref{E:FREQUENCYSQUARESUMMEDSMOOTHFUNCTIONTIMESBASICVARIABLESSIMPLECONSEQUENCEOFRESCALEDSTRICHARTZ},
we bound the RHS of the previous expression by
\begin{align}
\notag
\lesssim 
\uplambda^{-1/2 - 4 \upepsilon_0 + \frac{2}{q} - \frac{16}{q} \upepsilon_0}
\|	
	(\vec{\VortVort},\DivGradEnt)
\|_{L_t^2 L_x^{\infty}(\widetilde{\mathcal{M}})}
\lesssim 
\uplambda^{\frac{2}{q} -1 - \frac{16}{q} \upepsilon_0 - 8 \upepsilon_0}
=
\uplambda^{\frac{2}{q} - 1 - 4 \upepsilon_0(\frac{4}{q} + 2)}
\end{align}
as desired.

The estimates \eqref{E:LAMBDAINVERSELINEARTERMTIMEINTEGRALLUINFTIYLTINFTYLOMEGAP}--\eqref{E:ATHIRDLAMBDAINVERSELINEARTERMLINFTYWHOLESPACETIME} 
can be proved using similar arguments that also take into account the bound \eqref{E:RGEOANDUBOUNDS} for $\rgeo$,
and we therefore omit the straightforward details.

To prove \eqref{E:LAMBDAINVERSESECONDANGULARDERIVATIVELINEARTERMTIMEINTEGRALLTINFTTYLUINFTYLOMEGAP},
we first observe (switching the order of $L_u^{\infty}$ and $L_t^{\infty}$) that it suffices to prove
that for each fixed $u \in [-\frac{4}{5} \uplambda^{1-8 \upepsilon_0} \Tboot, \uplambda^{1-8 \upepsilon_0} \Tboot]$,
we have
\begin{align}
\notag
\uplambda^{-1} 
\left\|	
	\int_{[u]_+}^t
			|
				\rgeo^{3/2} \angD(\vec{\VortVort},\DivGradEnt)
			|_{\gsphere}(\uptau,u,\upomega)
	\, d \uptau
\right\|_{L_t^{\infty} L_{\upomega}^p (\widetilde{\mathcal{C}}_u)}
\lesssim
\uplambda^{-1/2-8 \upepsilon_0}.
\end{align}
Using 
\eqref{E:RESCALEDBOOTBOUNDS},
\eqref{E:RGEOANDUBOUNDS},
and \eqref{E:ONETANGENTIALDERIVATIVESMOOTHFUNCTIONTIMESMODVARIABLESCONE},
we conclude that the LHS of the previous expression is
\begin{align}
\notag
\lesssim
\uplambda^{-1}
\|	
	\rgeo^{3/2} \angD(\vec{\VortVort},\DivGradEnt)
\|_{L_t^1 L_{\upomega}^p (\widetilde{\mathcal{C}}_u)}
	\lesssim
\uplambda^{- 8 \upepsilon_0} 
\|	
	\rgeo \angD (\vec{\VortVort},\DivGradEnt)
\|_{L_t^2 L_{\upomega}^p (\widetilde{\mathcal{C}}_u)}
\lesssim
\uplambda^{-1/2- 8 \upepsilon_0}
\end{align}
as desired.

The estimate \eqref{E:ASECONDLAMBDAINVERSESECONDANGULARDERIVATIVELINEARTERMTIMEINTEGRALLTINFTTYLUINFTYLOMEGAP}
can be proved using a similar argument, and we omit the details.

To prove \eqref{E:LAMBDAINVERSEQUADRATICTERMTIMEINTEGRALLTINFTTYLUINFTYLOMEGAP},
we first observe (switching the order of $L_u^{\infty}$ and $L_t^{\infty}$
and using that $|\vec{\GradEnt}| \lesssim 1$) that it suffices to prove
that for each fixed $u \in [-\frac{4}{5} \uplambda^{1-8 \upepsilon_0} \Tboot, \uplambda^{1-8 \upepsilon_0} \Tboot]$,
we have
\[
\uplambda^{-1}
	\left\|
		\int_{[u]_+}^t
			\left|
				\rgeo^{3/2}
				\pmb{\partial} (\vec{\Psi},\vec{\vortrenormalized},\vec{\GradEnt}) 
				\cdot 
				(\pmb{\partial} \vec{\Psi},\mytr_{\congsphere} \widetilde{\upchi}^{(Small)},\hat{\upchi},\upzeta,\rgeo^{-1})
			\right|_{\gsphere}(\uptau,u,\upomega)
	\, d \uptau
	\right\|_{L_t^{\infty} L_{\upomega}^p (\widetilde{\mathcal{C}}_u)}
\lesssim
\uplambda^{-1/2 - 12 \upepsilon_0}.
\]
Using
\eqref{E:RESCALEDBOOTBOUNDS},
\eqref{E:RGEOANDUBOUNDS},
\eqref{E:FREQUENCYSQUARESUMMEDSMOOTHFUNCTIONTIMESBASICVARIABLESSIMPLECONSEQUENCEOFRESCALEDSTRICHARTZ},
and \eqref{E:BOOTSTRAPCHIANDTORSIONALONGANYCONE},
we deduce that the LHS of the previous expression is
\begin{align}
	& \lesssim
	\uplambda^{-8 \upepsilon_0}
	\|
		\pmb{\partial} (\vec{\Psi},\vec{\vortrenormalized},\vec{\GradEnt}) 
	\|_{L_t^2 L_{\upomega}^{\infty} (\widetilde{\mathcal{C}}_u)}
	+
	\uplambda^{1/2 - 12 \upepsilon_0}
	\|
		\pmb{\partial} (\vec{\Psi},\vec{\vortrenormalized},\vec{\GradEnt}) 
	\|_{L_t^2 L_{\upomega}^{\infty} (\widetilde{\mathcal{C}}_u)}
	\left\|
		(\pmb{\partial} \vec{\Psi},\mytr_{\congsphere} \widetilde{\upchi}^{(Small)},\hat{\upchi},\upzeta)
	\right\|_{L_t^2 L_{\upomega}^p (\widetilde{\mathcal{C}}_u)}
		\label{E:MAINSTEPLAMBDAINVERSEQUADRATICTERMTIMEINTEGRALLTINFTTYLUINFTYLOMEGAP} \\
	& \lesssim
		\uplambda^{-1/2 - 12 \upepsilon_0}
		\notag
\end{align}
as desired.

The estimates \eqref{E:ASECONDLAMBDAINVERSEQUADRATICTERMTIMEINTEGRALLTINFTTYLUINFTYLOMEGAP},
\eqref{E:LAMBDAINVERSEQUADRATICTERMTIMEINTEGRALLT2LU2LOMEGAP},
and
\eqref{E:LAMBDAINVERSEQUADRATICTERMTIMEINTEGRALLU2LTINFTYLOMEGAP}
can be proved using similar arguments, and we omit the details.

To prove \eqref{E:LAMBDAINVERSETIMEINTEGRALLT2LU2LOMEGAP},
we first use
\eqref{E:RESCALEDBOOTBOUNDS}
to deduce 
(switching the order of $L_u^2$ and $L_t^2$) 
that
$$
\mbox{LHS~\eqref{E:LAMBDAINVERSETIMEINTEGRALLT2LU2LOMEGAP}}
\lesssim
\uplambda^{-1/2 - 4 \upepsilon_0} 
\left\|	
	\int_0^t
			\|
				\rgeo (\pmb{\partial}\vec{\VortVort},\pmb{\partial}\DivGradEnt)
			\|_{L_u^2 L_{\upomega}^p(\widetilde{\Sigma}_{\uptau})}
	\, d \uptau
\right\|_{L_t^{\infty}}.
$$
Using 
\eqref{E:RESCALEDBOOTBOUNDS}
and
\eqref{E:ONEDERIVATIVESMOOTHFUNCTIONTIMESMODVARIABLESSIGMAT} with $\leb := p$,
we deduce that the RHS of the previous expression is
$$
\lesssim 
\uplambda^{-1/2 - 4 \upepsilon_0}
\|
	\rgeo (\pmb{\partial}\vec{\VortVort},\pmb{\partial}\DivGradEnt)
\|_{L_t^1 L_u^2 L_{\upomega}^p (\widetilde{\mathcal{M}})}
\lesssim
\uplambda^{1/2 - 12 \upepsilon_0}
\|
	\rgeo (\pmb{\partial}\vec{\VortVort},\pmb{\partial}\DivGradEnt)
\|_{L_t^{\infty} L_u^2 L_{\upomega}^p (\widetilde{\mathcal{M}})}
\lesssim
\uplambda^{- 12 \upepsilon_0}
$$
as desired.

The estimate \eqref{E:LAMBDAINVERSETIMEINTEGRALLU2LTINFTYLOMEGAP}
can be proved using similar arguments that also take into account the bound \eqref{E:RGEOANDUBOUNDS} for $\rgeo$,
and we therefore omit the straightforward details.

To prove \eqref{E:RGEOLAMBDAINVERSELINEARTERMLT2LUINFTYLOMEGAP},
we use 
\eqref{E:RESCALEDBOOTBOUNDS},
\eqref{E:RGEOANDUBOUNDS},
and
\eqref{E:MODFLUIDSIGMAT}
to conclude that
\begin{align}
\mbox{LHS~\eqref{E:RGEOLAMBDAINVERSELINEARTERMLT2LUINFTYLOMEGAP}}
& \lesssim
\uplambda^{-1/2 - 4 \upepsilon_0}
\| \rgeo^{1/2} \|_{L^{\infty}(\widetilde{\mathcal{M}})}
\|
	\rgeo^{1/2} (\vec{\VortVort},\DivGradEnt)
\|_{L_t^{\infty} L_u^{\infty} L_{\upomega}^p(\widetilde{\mathcal{M}})}
\lesssim
\uplambda^{-1/2 - 8 \upepsilon_0}
\end{align}
as desired.

To prove \eqref{E:RGEOLAMBDAINVERSELINEARTERMLTQOVER2LUINFTYLOMEGAP},
we use a similar argument
to conclude that
\begin{align}
\uplambda^{-1} 
	\|			
		\rgeo (\vec{\VortVort},\DivGradEnt)
	\|_{L_t^{\frac{q}{2}} L_u^{\infty} L_{\upomega}^p(\widetilde{\mathcal{M}})}
& \lesssim
\uplambda^{-1} 
\uplambda^{(1 - 8 \upepsilon_0) \frac{2}{q}}
\| \rgeo^{1/2} \|_{L^{\infty}(\widetilde{\mathcal{M}})}
\|
	\rgeo^{1/2} (\vec{\VortVort},\DivGradEnt)
\|_{L_t^{\infty} L_u^{\infty} L_{\upomega}^p(\widetilde{\mathcal{M}})}
	\\
& 
\lesssim
\uplambda^{\frac{2}{q} - 1 - 4 \upepsilon_0 - \frac{16}{q} \upepsilon_0}
=
\uplambda^{\frac{2}{q} - 1 - 4 \upepsilon_0(\frac{4}{q} + 1)}
\notag
\end{align}
as desired.

The estimate \eqref{E:RGEOLAMBDAINVERSELINEARTERMLU2LT2LOMEGAP}
follows easily from \eqref{E:RGEOLAMBDAINVERSELINEARTERMLT2LUINFTYLOMEGAP}
and the bounds \eqref{E:RGEOANDUBOUNDS} for $u$.

To prove \eqref{E:RGEOLAMBDAINVERSEONEDERIVATIVEOFMODFLUIDLU2LT1LOMEGAP},
we use
\eqref{E:RESCALEDBOOTBOUNDS}
and
\eqref{E:ONEDERIVATIVESMOOTHFUNCTIONTIMESMODVARIABLESSIGMAT}
with $\leb := p$
to deduce that
\begin{align}
\notag
\uplambda^{-1} \| \rgeo \pmb{\partial}(\vec{\VortVort},\DivGradEnt) \|_{L_u^2 L_t^1 L_{\upomega}^p(\widetilde{\mathcal{M}})}
\lesssim
\uplambda^{-1} \| \rgeo \pmb{\partial}(\vec{\VortVort},\DivGradEnt) \|_{L_t^1 L_u^2 L_{\upomega}^p(\widetilde{\mathcal{M}})}
\lesssim
\uplambda^{-8 \upepsilon_0} \| \rgeo \pmb{\partial}(\vec{\VortVort},\DivGradEnt) \|_{L_t^{\infty} L_u^2 L_{\upomega}^p(\widetilde{\mathcal{M}})}
\lesssim 
\uplambda^{-\frac{1}{2} - 8 \upepsilon_0}
\end{align}
as desired.

To prove \eqref{E:RGEOLAMBDAINVERSEQUADTERMLU2LT1LOMEGAP},
we first use \eqref{E:RESCALEDBOOTBOUNDS}, \eqref{E:RGEOANDUBOUNDS},
and the fact that $|\vec{\GradEnt}| \lesssim 1$
to deduce that
\begin{align} \label{E:FIRSTSTEPRGEOLAMBDAINVERSEQUADTERMLU2LT1LOMEGAP}
\mbox{LHS~\eqref{E:RGEOLAMBDAINVERSEQUADTERMLU2LT1LOMEGAP}}
&
\lesssim
\uplambda^{-1/2 - 4 \upepsilon_0}
\left\|
	\rgeo 
	(\pmb{\partial} \vec{\Psi},\pmb{\partial} \vec{\vortrenormalized},\pmb{\partial} \vec{\GradEnt}) 
	\cdot 
	(\pmb{\partial} \vec{\Psi},\mytr_{\congsphere} \widetilde{\upchi}^{(Small)},\hat{\upchi},\upzeta,\rgeo^{-1})	
\right\|_{L_u^{\infty} L_t^1 L_{\upomega}^p(\widetilde{\mathcal{M}})}
	\\
&
\notag
\lesssim
\uplambda^{1/2 - 8 \upepsilon_0} 
\| 
	(\pmb{\partial} \vec{\Psi},\pmb{\partial} \vec{\vortrenormalized},\pmb{\partial} \vec{\GradEnt}) 
\|_{L_u^{\infty} L_t^2 L_{\upomega}^{\infty}(\widetilde{\mathcal{M}})}
\| 
 (\pmb{\partial} \vec{\Psi},\mytr_{\congsphere} \widetilde{\upchi}^{(Small)},\hat{\upchi},\upzeta)
\|_{L_u^{\infty} L_t^2 L_{\upomega}^p(\widetilde{\mathcal{M}})}
\\
\notag
 &
\ \ 
+
\uplambda^{- 8 \upepsilon_0} 
\| 
	(\pmb{\partial} \vec{\Psi},\pmb{\partial} \vec{\vortrenormalized},\pmb{\partial} \vec{\GradEnt}) 
\|_{L_u^{\infty} L_t^2 L_{\upomega}^p(\widetilde{\mathcal{M}})}.
\end{align}
Using \eqref{E:FREQUENCYSQUARESUMMEDSMOOTHFUNCTIONTIMESBASICVARIABLESSIMPLECONSEQUENCEOFRESCALEDSTRICHARTZ}
and the bootstrap assumptions \eqref{E:BOOTSTRAPCHIANDTORSIONALONGANYCONE},
we conclude that 
$\mbox{RHS~\eqref{E:FIRSTSTEPRGEOLAMBDAINVERSEQUADTERMLU2LT1LOMEGAP}} 
\lesssim 
\uplambda^{-\frac{1}{2} - 10 \upepsilon_0}$
as desired.	

\end{proof}

\subsection{\texorpdfstring{Control of the integral curves of $\Lunit$}{Control of the integral curves of Lunit}}
\label{SS:CONTROLOFNULLGEODESICS}
The main results of this subsection are Prop.\,\ref{P:CONTROLOFNULLGEODESICS}
and Cor.\,\ref{C:ESTIMATEFORANGLEHOLDERCONTINUITYOFPARTIALPSIPARTIALVORTANDPARTIALGRADENT}.
The proposition yields quantitative estimates showing that at fixed $u$,
the distinct integral curves of $\Lunit$ remain separated (see Footnote~\ref{FN:SEPARATEED}).
The corollary is a simple consequence of the proposition and the bootstrap assumptions.
It provides $L_t^2 L_u^{\infty} C_{\upomega}^{0,\updelta_0}$ estimates for the fluid variables.
Later, we will combine these estimates with the Schauder-type estimate
\eqref{E:ANGULARHOLDERHODGEESTIMATENODERIVATIVESONLHSANGDIVSYMMETRICTRACEFREESTUTENSORFIELD}
to obtain $L_t^q L_u^{\infty} C_{\upomega}^{0,\updelta_0}$-control of $\hat{\upchi}$ for several values of $q$;
see the proofs of 
\eqref{E:CONNECTIONCOEFFICIENTESTIMATESNEEDEDTODERIVESPATIALLYLOCALIZEDDECAYFROMCONFORMALENERGYESTIMATE}
and
\eqref{E:IMPROVEDININTERIORL2INTIMELINFINITYINSPACECONNECTIONCOFFICIENTS} for $\hat{\upchi}$.

We start with some preliminary estimates, provided by the following lemma.

\begin{lemma}[Preliminary results for controlling the integral curves of $\Lunit$]
	\label{L:LUNITIALONGCONETIPAXISISC1INANGLEVARIABLES}
	Under the assumptions of Subsect.\,\ref{SS:ASSUMPTIONS}, 
	if $\uplambda$ is sufficiently large, then the following results hold.
	
	\medskip
	
	\noindent \underline{\textbf{Results along $\Sigma_0$}}:
	For $A=1,2$ and $i=1,2,3$,
	let $\left(\frac{\partial}{\partial \upomega^A} \right)^i$ denote 
	the Cartesian components of $\frac{\partial}{\partial \upomega^A}$,
	and let $\Theta_{(A)}$ be the $S_{t,u}$-tangent vectorfield with Cartesian components
	$\Theta_{(A)}^i := \frac{1}{\rgeo} \left(\frac{\partial}{\partial \upomega^A} \right)^i$,
	as in \eqref{E:RESCALED}.
	Along $\Sigma_0$ (where $\rgeo = w = - u$), for $0 < w \leq \RescaledFoliationparameter := \frac{4}{5}\RescaledTboot$ 
	and $\upomega \in \mathbb{S}^2$,
	we view $\Theta_{(A)}^i = \Theta_{(A)}^i(0,w,\upomega)$,
	and similarly for the Cartesian spatial components $\spherenormal^i$ and $\Lunit^i$.
	Then for each $\upomega \in \mathbb{S}^2$,
		$\lim_{w \downarrow 0} \spherenormal^i(0,w,\upomega)$,
	$\lim_{w \downarrow 0} \Lunit^i(0,w,\upomega)$,
	and $\lim_{w \downarrow 0} \Theta_{(A)}^i(0,w,\upomega)$ exist,
	and we respectively denote the limits by 
	$\spherenormal^i(0,0,\upomega)$,
	$\Lunit^i(0,0,\upomega)$,
	and $\Theta_{(A)}^i(0,0,\upomega)$.
		Furthermore, for each $\upomega \in \mathbb{S}^2$,
	we have that
\begin{align}
\notag
    g_{cd}(0,0,\upomega) \Theta_{(A)}^c(0,0,\upomega)\Theta_{(B)}^d(0,0,\upomega)
	=
	\stgsphere(\upomega)\left(\frac{\partial}{\partial \upomega^A},\frac{\partial}{\partial \upomega^B}\right)
	+
	\mathcal{O}(\uplambda^{- 4 \upepsilon_0}).
\end{align}	
	
	In addition, the following estimates hold for
	$(w,\upomega) \in [0,\frac{4}{5}\RescaledTboot] \times \mathbb{S}^2$,
	where $x^i(0,w,\upomega)$ are the Cartesian spatial coordinates viewed as a function of $w,\upomega$ along $\Sigma_0$,
	and ${\bf{z}}^i$ are the Cartesian spatial coordinates of the point ${\bf{z}} \in \Sigma_0$ (see Subsect.\,\ref{SS:EIKONAL}):
	\begin{subequations}
	\begin{align}
			x^i(0,w,\upomega)
		& = {\bf{z}}^i
				+
				w \left\lbrace
					\spherenormal^i(0,0,\upomega)
					+
					\mathcal{O}(\uplambda^{-4 \upepsilon_0})
				\right\rbrace,
					\label{E:ESTIMATEFORCARTESIANCOORDINATESALONGSIGMA0}  
					\\
		\spherenormal^i(0,w,\upomega)
		& =
		\spherenormal^i(0,0,\upomega)
		+
		\mathcal{O}(\uplambda^{-4 \upepsilon_0}),
			\label{E:SPHERENORMALIALONGSIGMA0STAYSCLOSETOORIGINVALUES}  
			\\
		\Lunit^i(0,w,\upomega)
		& =
		\Lunit^i(0,0,\upomega)
		+
		\mathcal{O}(\uplambda^{- 4\upepsilon_0}),
		\label{E:LUNITIALONGSIGMA0STAYSCLOSETOORIGINVALUES}  
			\\
	\Theta_{(A)}^i(0,w,\upomega)
	& =
	\Theta_{(A)}^i(0,0,\upomega)
		+
		\mathcal{O}(\uplambda^{-4 \upepsilon_0}).
		\label{E:RESCALEDGEOMETRICCOORDINATEVECTORFIELDRELATEDTOAMEQUANTITITYATZ}
\end{align}
\end{subequations}
	
	Moreover, the following identity holds:
	\begin{align} \label{E:ANGULARDERIVATIVEOFNORMALVECTORATORIGININSIGMA0}
		\frac{\partial}{\partial \upomega^A} \spherenormal^i(0,0,\upomega) 
			= \frac{\partial}{\partial \upomega^A} \Lunit^i(0,0,\upomega) 
			& = \Theta_{(A)}^i(0,0,\upomega).
	\end{align}

	In addition, with $d_{\stgsphere}(\upomega_{(1)},\upomega_{(2)})$ denoting the distance between the points 
	$\upomega_{(1)},\upomega_{(2)} \in \mathbb{S}^2$
	with respect to the standard Euclidean round metric $\stgsphere$ on $\mathbb{S}^2$,
	we have the following estimate:
	\begin{align} \label{E:DISTANCEBETWEENNORMALCARTESIANCOMPONENTSISCOMPARABLETOS2DISTANCEBETWEENANGLES}
	\sum_{i=1}^3
	|
		\spherenormal^i(0,0,\upomega_{(1)})
		-
		\spherenormal^i(0,0,\upomega_{(2)})
	|
	=
	\sum_{i=1}^3
	|
		\Lunit^i(0,0,\upomega_{(1)})
		-
		\Lunit^i(0,0,\upomega_{(2)})
	|
	\approx
	d_{\stgsphere}(\upomega_{(1)},\upomega_{(2)}).
	\end{align}
	
	Finally, we have the following estimate, ($\alpha = 0,1,2,3$):
	\begin{align} \label{E:DATAESTIMATEFORHOLDERNORMOFLUNITI}
		\|
			\Lunit^{\alpha}
		\|_{L_u^{\infty} C_{\upomega}^{0,\updelta_0}(\widetilde{\Sigma}_0)}
		& \lesssim 1.
	\end{align}
	
	\medskip
	
	\noindent \underline{\textbf{Results along the cone-tip axis}}:
	In $\widetilde{\mathcal{M}}^{(Int)}$, let us view $\Theta_{(A)}^i = \Theta_{(A)}^i(t,u,\upomega)$,
	and similarly for the Cartesian spatial components $\spherenormal^i$ and $\Lunit^i$.
	Then for each $(u,\upomega) \in [0,\RescaledTboot] \times \mathbb{S}^2$,
	$\lim_{t \downarrow u} \Theta_{(A)}^i(t,u,\upomega)$ exists, 
	and we denote the limit by $\Theta_{(A)}^i(u,u,\upomega)$.
	Furthermore, the following estimate holds for $(t,\upomega) \in [0,\RescaledTboot] \times \mathbb{S}^2$:
	$g_{ab}(t,t,\upomega) \Theta_{(A)}^a(t,t,\upomega) \Theta_{(B)}^b(t,t,\upomega)
	=
	\stgsphere(\upomega)\left(\frac{\partial}{\partial \upomega^A},\frac{\partial}{\partial \upomega^B}\right)
	+ \mathcal{O}(\uplambda^{-\upepsilon_0})
	$,
	and within each coordinate chart on $\mathbb{S}^2$,
	for each $\upomega$ in the domain of the chart, 
	$\lbrace \Theta_{(1)}(t,t,\upomega),\Theta_{(2)}(t,t,\upomega) \rbrace$ is a linearly independent set of vectors in $\mathbb{R}^3$.
	
	Moreover, along the cone-tip axis, that is, for $t \in [0,\RescaledTboot]$ we have:
	\begin{subequations}
	\begin{align}
		\spherenormal^i(t,t,\upomega)
		& =
		\spherenormal^i(0,0,\upomega)
		+
		\mathcal{O}(\uplambda^{-8 \upepsilon_0}),
			\label{E:SPHERENORMALIALONGCONETIPAXISSTAYSCLOSETOORIGINVALUES}  
			\\
		\Lunit^i(t,t,\upomega)
		& =
		\Lunit^i(0,0,\upomega)
		+
		\mathcal{O}(\uplambda^{-8 \upepsilon_0}),
		\label{E:LUNITIALONGCONETIPAXISSTAYSCLOSETOORIGINVALUES}  
			\\
	\Theta_{(A)}^i(t,t,\upomega)
	& =
	\Theta_{(A)}^i(0,0,\upomega)
	+
	\mathcal{O}(\uplambda^{-4 \upepsilon_0}).
		\label{E:RESCALEDGEOMETRICCOORDINATEVECTORFIELDALONGCONETIPAXISSTAYSCLOSETOORIGINVALUES}
	\end{align}
	\end{subequations}
	
	In addition, for $(t,\upomega) \in [0,\RescaledTboot] \times \mathbb{S}^2$, 
	the following relations hold along the cone-tip axis, that is, for $t \in [0,\RescaledTboot]$:
	\begin{align} \label{E:ANGULARDERIVATIVEOFNORMALVECTORATORIGIN}
		\frac{\partial}{\partial \upomega^A} \spherenormal^i(t,t,\upomega) 
			= \frac{\partial}{\partial \upomega^A} \Lunit^i(t,t,\upomega) 
			& = \Theta_{(A)}^i(t,t,\upomega).
	\end{align}
	
	
	\medskip
	
	\noindent \underline{\textbf{Results in $\widetilde{\mathcal{M}}$}}:
	For $u \in [-\frac{4}{5} \RescaledTboot,\RescaledTboot]$, $t \in [[u]_+,\RescaledTboot]$, and $\upomega \in \mathbb{S}^2$,
	we have
	\begin{subequations}
	\begin{align} 
			\Lunit^i(t,u,\upomega)
			& = 
			\Lunit^i(0,0,\upomega)
			+
			\mathcal{O}(\uplambda^{-4 \upepsilon_0}),
				\label{E:QUANTITATIVECONTROLOFLUNITIALLREGIONS} 
						\\
			\Theta_{(A)}^i(t,u,\upomega)
			& = 
				\Theta_{(A)}^i(0,0,\upomega)
				+
				\mathcal{O}(\uplambda^{-4 \upepsilon_0}).
				\label{E:QUANTITATIVECONTROLOFCARTESIANCOMPONENTSRESCALEDGEOMETRICCOORDINATEANGULARDERIVATIVEVECTORVIELDINALLREGIONS}
	\end{align}
	\end{subequations}
	
\end{lemma}

\begin{proof} \ \\
	\noindent \textbf{Proof of the results along $\Sigma_0$}.
	We start by showing that $\lim_{w \downarrow 0} \Theta_{(A)}^i(0,w,\upomega) := \Theta_{(A)}^i(0,0,\upomega)$ exists,
	and we exhibit the desired properties of the limit.
	We will use the evolution equation \eqref{E:EVOLUTIONEQUATIONALONGSIGMA0FORCARTESIANCOMPONENTSOFTHETAA}.
	From the bootstrap assumptions,
	the simple bound $\| \gensmoothfunction_{(\vec{\Lunit})} \|_{L^{\infty}(\widetilde{\mathcal{M}})} \lesssim 1$ implied by them,
	the estimates of Prop.\,\ref{P:INITIALFOLIATION},
	\eqref{E:EUCLIDEANFORMMORREYONSTU} with $\leb := p$,
	\eqref{E:SMOOTHFUNCTIONTIMESFLUIDONEDERIVATIVESIGMAT} and \eqref{E:ONEDERIVATIVESMOOTHFUNCTIONTIMESONEDERIVATIVETIMESFLUIDSIGMAT} along $\Sigma_0$ with $\leb := p$,
	the bound \eqref{E:RGEOANDUBOUNDS} for $\rgeo|_{\Sigma_0} = w$,
	and the estimate $\| \Theta_{(A)}^i\|_{L^{\infty}(\widetilde{\Sigma}_0)} \lesssim 1$
	implied by \eqref{E:INITIALSPHEREMETRICESTIMATE},
	we find that the first term on RHS~\eqref{E:EVOLUTIONEQUATIONALONGSIGMA0FORCARTESIANCOMPONENTSOFTHETAA}
	verifies
	\begin{align*}
	&
	\|
	a
	\cdot
	\gensmoothfunction_{(\vec{\Lunit})} 
	\cdot 
	(\pmb{\partial} \vec{\Psi},\hat{\upchi})
	\cdot
	\vec{\Theta}_{(A)}
	\|_{L_w^1 L_{\upomega}^{\infty}(\widetilde{\Sigma}_0)} 
		\\
	& \lesssim
	\uplambda^{1/2 - 4 \upepsilon_0}
	\| \rgeo \angD \pmb{\partial} \vec{\Psi} \|_{L_w^2 L_{\upomega}^p(\widetilde{\Sigma}_0)}
	+
	\uplambda^{1/2 - 4 \upepsilon_0}
	\| \pmb{\partial} \vec{\Psi} \|_{L_w^2 L_{\upomega}^2(\widetilde{\Sigma}_0)}
	+
	\uplambda^{1/2 - 4 \upepsilon_0}
	\|
		\hat{\upchi}
	\|_{L_w^2 L_{\upomega}^{\infty}(\widetilde{\Sigma}_0)} 
		\\
	& \lesssim 
	\uplambda^{-4 \upepsilon_0},
	\end{align*}
	and that the last term on RHS~\eqref{E:EVOLUTIONEQUATIONALONGSIGMA0FORCARTESIANCOMPONENTSOFTHETAA}
	verifies the same bound:
	$$
	\|
	\gensmoothfunction_{(\vec{\Lunit})}
	\cdot
	\angD a
	\cdot
	\vec{\Theta}_{(A)}
	\|_{L_w^1 L_{\upomega}^{\infty}(\widetilde{\Sigma}_0)} 
	\lesssim
	\uplambda^{1/2 - 4 \upepsilon_0}
	\|
	\angD a
	\|_{L_w^2 L_{\upomega}^{\infty}(\widetilde{\Sigma}_0)} 
	\lesssim \uplambda^{-4 \upepsilon_0}.
	$$
	
	We now integrate equation 
	\eqref{E:EVOLUTIONEQUATIONALONGSIGMA0FORCARTESIANCOMPONENTSOFTHETAA} with respect to $w$
	and use these estimates
	and the initial condition for $\vec{\Theta}_{(A)}$ 
	at the convenient value $w=1$ (which, by \eqref{E:INITIALSPHEREMETRICESTIMATE}, 
	is a value at which the vectors $\vec{\Theta}_{(1)}$ and $\vec{\Theta}_{(2)}$ are known to be finite and linearly independent)
	thereby concluding that if $\uplambda$ is sufficiently large, then
	$\lim_{w \downarrow 0} \Theta_{(A)}^i(0,w,\upomega)$ exists,
	that for $0 \leq w \leq \frac{4}{5} \RescaledTboot$ and $\upomega \in \mathbb{S}^2$ we have
	$$
	\Theta_{(A)}^i(0,w,\upomega)
	= 
	\Theta_{(A)}^i(0,1,\upomega)
	+
	\mathcal{O}(\uplambda^{-4 \upepsilon_0}),
	$$
	and that
	$$g_{cd}(0,0,\upomega) \Theta_{(A)}^c(0,0,\upomega) \Theta_{(B)}^d(0,0,\upomega)
	\approx
	\stgsphere(\upomega) \left(\frac{\partial}{\partial \upomega^A},\frac{\partial}{\partial \upomega^B}\right).$$
	Except for \eqref{E:ANGULARDERIVATIVEOFNORMALVECTORATORIGININSIGMA0}, 
	these arguments yield all desired results for $\Theta_{(A)}^i$ along $\Sigma_0$,
	including \eqref{E:RESCALEDGEOMETRICCOORDINATEVECTORFIELDRELATEDTOAMEQUANTITITYATZ}.
	To prove \eqref{E:ANGULARDERIVATIVEOFNORMALVECTORATORIGININSIGMA0},
	we contract the estimate \eqref{E:ANGULARDERIVATIVEOFSPHERENORMALATINITIALCONETIPAXISPOINT} against
	$\Theta_{(A)}^j(0,w,\upomega)$,
	use the identities 
	$w \Theta_{(A)}^j \sphereproject_j^c \partial_c \spherenormal^i = \frac{\partial}{\partial \upomega^A} \spherenormal^i$,
	$\Theta_{(A)}^j \sphereproject_j^i = \Theta_{(A)}^i$,
	and $\Lunit^i = \Transport^i + \spherenormal^i$,
	use that $\Transport^i(0,0,\upomega) = \Transport^i|_{{\bf{z}}}$ is independent of $\upomega$,
	and use the previous results proved in this paragraph.

	The results for $\Lunit^i$ and $\spherenormal^i$ along $\Sigma_0$ stated in the lemma,
	including \eqref{E:SPHERENORMALIALONGSIGMA0STAYSCLOSETOORIGINVALUES} and \eqref{E:LUNITIALONGSIGMA0STAYSCLOSETOORIGINVALUES},
	can be obtained from similar reasoning
	based on the evolution equations in \eqref{E:LUNITIANDNORMALITRANSPORTALONGSIGMA0},
	and we omit the details.
	
	Next, we consider the map $\mathfrak{N}(\upomega) 
	:= \left(\spherenormal^1(0,0,\upomega),\spherenormal^2(0,0,\upomega),\spherenormal^3(0,0,\upomega)\right)$
	from the domain $\mathbb{S}^2$ to the target
	$$UT_{\bf{z}} \Sigma_0 := \lbrace V \in T_{\bf{z}} \Sigma_0 \ | \ g_{cd}|_{\bf{z}} V^c V^d = 1 \rbrace \simeq \mathbb{S}^2.$$
	The results from the first paragraph of this proof, including \eqref{E:ANGULARDERIVATIVEOFNORMALVECTORATORIGININSIGMA0},
	yield that the differential of $\mathfrak{N}$ with respect to $\upomega$ is injective.
	Thus, $\mathfrak{N}$ is a differentiable open map from $\mathbb{S}^2$ to $\mathbb{S}^2$,
	and it is a standard result of differential topology that $\mathfrak{N}$ must be a covering map (in particular, it is onto).
	Thus, taking into account the quantitative bounds for the differential of $\mathfrak{N}$ with respect to $\upomega$ proved 
	above, we conclude that there exists a uniform constant $0 < \upbeta < \pi$
	such that if $\uplambda$ is sufficiently large, then
	\eqref{E:DISTANCEBETWEENNORMALCARTESIANCOMPONENTSISCOMPARABLETOS2DISTANCEBETWEENANGLES}
	holds (with bounded implicit constants)
	for all pairs $\upomega_{(1)},\upomega_{(2)} \in \mathbb{S}^2$
	such that
	$d_{\stgsphere}(\upomega_{(1)},\upomega_{(2)}) < \upbeta$.
	Moreover, since the domain $\mathbb{S}^2$ is path-connected and the target $UT_{\bf{z}}\Sigma_0 \simeq \mathbb{S}^2$ 
	is simply connected, 
	it is a standard result in algebraic topology that
	$\mathfrak{N}$ is in fact a diffeomorphism 
	(see \cite{jM2000}*{Theorem 54.4} and note that
	$UT_{\bf{z}} \Sigma_0 \simeq \mathbb{S}^2$ has a trivial fundamental group since it is simply connected).
	In particular, $\mathfrak{N}$ is globally injective.
	This fact yields \eqref{E:DISTANCEBETWEENNORMALCARTESIANCOMPONENTSISCOMPARABLETOS2DISTANCEBETWEENANGLES}
	(again, with bounded implicit constants)
	for all $\upomega_{(1)},\upomega_{(2)} \in \mathbb{S}^2$ 
	with $\upbeta \leq d_{\stgsphere}(\upomega_{(1)},\upomega_{(2)}) \leq \pi$.
	
	To prove \eqref{E:ESTIMATEFORCARTESIANCOORDINATESALONGSIGMA0}, we first use 
	\eqref{E:PARTIALPARTIALWALONGSIGMA0} to deduce
	$
	\frac{\partial}{\partial w} x^i(0,w,\upomega) = [a \spherenormal^i](0,w,\upomega)
	$.
	Also using \eqref{E:INTIALLAPSEANDSPHEREVOLUMEELEMENTESTIMATE} and \eqref{E:SPHERENORMALIALONGSIGMA0STAYSCLOSETOORIGINVALUES},
	we see that $\frac{\partial}{\partial w} x^i(0,w,\upomega) = \spherenormal^i(0,0,\upomega) + \mathcal{O}(\uplambda^{-4 \upepsilon_0})$.
	Integrating this estimate with respect to $w$ starting from the value $w=0$, and using the initial condition
	$x^i(0,w,\upomega) = {\bf{z}^i}$, we conclude \eqref{E:ESTIMATEFORCARTESIANCOORDINATESALONGSIGMA0}.
	
	We now show that for each $(u,\upomega) \in [-\frac{4}{5} \RescaledTboot,0) \times \mathbb{S}^2$,
\begin{align}
\notag
\lim_{t \downarrow 0} \Theta_{(A)}^i(t,u,\upomega) = \Theta_{(A)}^i(0,u,\upomega)
\end{align}
	and 
\begin{align}
\notag
g_{cd}(0,u,\upomega) \Theta_{(A)}^c(0,u,\upomega)\Theta_{(B)}^d(0,u,\upomega)
	=
	\stgsphere(\upomega)\left(\frac{\partial}{\partial \upomega^A},\frac{\partial}{\partial \upomega^B}\right)
	+
	\mathcal{O}(\uplambda^{-\upepsilon_0}).
\end{align}
	The desired results can be obtained by using arguments similar to the ones given in the first paragraph of this proof,
	based on the evolution equation \eqref{E:EVOLUTIONEQUATIONALONGINTEGRALCURVESOFLUNITFORCARTESIANCOMPONENTSOFTHETAA}
	and the bootstrap assumptions, 
	including
	\eqref{E:RESCALEDBOOTBOUNDS},
	\eqref{E:RESCALEDSTRICHARTZ},
	\eqref{E:BOOTSTRAPMETRICAPPROXIMATELYROUND},
	and
	\eqref{E:LT2LINFINITYBOOTSTRAPCHIANDZETAININTERIORREGION}.
	
	Finally, we prove \eqref{E:DATAESTIMATEFORHOLDERNORMOFLUNITI}. The result is trivial for $\Lunit^0$ since this component is constantly unity.
	Next, we note the schematic identity $\angD \Lunit^i = \gensmoothfunction_{(\vec{\Lunit})} \upchi + \gensmoothfunction_{(\vec{\Lunit})} \cdot \pmb{\partial} \vec{\Psi}$,
	where on the LHS, 
	we are viewing $\angD \Lunit^i$ to be the angular gradient of the scalar function $\Lunit^i$.
	Hence, applying \eqref{E:EUCLIDEANFORMMORREYONSTU} with $\leb := p$
	and with the scalar function $\Lunit^i$ in the role of $\upxi$, 
	and using the simple bound $\| \gensmoothfunction_{(\vec{\Lunit})} \|_{L^{\infty}(\widetilde{\mathcal{M}})} \lesssim 1$
	implied by the bootstrap assumptions, we find that for $u \in [-\frac{4}{5} \RescaledTboot,0]$, we have
	$$\| \Lunit^i \|_{C_{\upomega}^{0,1 - \frac{2}{p}}(S_{0,u})}
		\lesssim 
		\| \rgeo \upchi \|_{L_{\upomega}^p(S_{0,u})}
		+
		\| \rgeo \pmb{\partial} \vec{\Psi} \|_{L_{\upomega}^p(S_{0,u})}
		+
		1.
		$$
		Also using the first identities in
		\eqref{E:CONNECTIONCOEFFICIENT}
		and
		\eqref{E:DETAILEDDECOMPOFNULLSECONDFUNDFORMSINTOTRACEANDTRACEFREE},
		\eqref{E:MODTRICHISMALL},
		the schematic identity $k_{AB} = \gensmoothfunction_{(\vec{\Lunit})} \cdot \pmb{\partial} \vec{\Psi}$,
		and the parameter relation \eqref{E:BOUNDSONLEBESGUEEXPONENTP},
		we find that
			$$\| \Lunit^i \|_{L_u^{\infty} C_{\upomega}^{0,\updelta_0}(\widetilde{\Sigma}_0)}
		\lesssim 
		\| \rgeo \mytr_{\congsphere} \widetilde{\upchi}^{(Small)} \|_{L_u^{\infty}L_{\upomega}^p(\widetilde{\Sigma}_0)}
		+
		\| \rgeo \hat{\spheresecondfund} \|_{L_u^{\infty}L_{\upomega}^p(\widetilde{\Sigma}_0)}
		+
		\| \rgeo \pmb{\partial} \vec{\Psi} \|_{L_u^{\infty}L_{\upomega}^p(\widetilde{\Sigma}_0)}
		+
		1.
		$$
		From \eqref{E:INTIALDERIVATIVESOFLAPSEANDTRACEFREECHIEST} with $q_* := p$,
		\eqref{E:MODTRCHIESTIAMTESALONGINITIALFOLIATION},
		\eqref{E:SMOOTHFUNCTIONTIMESFLUIDONEDERIVATIVESIGMAT},
		and \eqref{E:RGEOANDUBOUNDS} for $u|_{\Sigma_0} = - w$,
		we conclude that the RHS of the previous estimate is $\lesssim 1$,
		which yields \eqref{E:DATAESTIMATEFORHOLDERNORMOFLUNITI}.
		
	\medskip
	
	\noindent \textbf{Proof of the results along the cone-tip axis}.
	The ODE \eqref{E:FWTRANSPORT} can be expressed in the schematic form
	$
		\frac{d}{dt} \vec{\spherenormal}_{\upomega}
		=
		\gensmoothfunction(\vec{\Psi}) \cdot \pmb{\partial} \vec{\Psi}
		\cdot
		\vec{\spherenormal}_{\upomega}
	$,
	Here, $\vec{\spherenormal}_{\upomega} = \vec{\spherenormal}_{\upomega}(t)$
	denotes the array of Cartesian spatial components of the unit outward normal vector $\spherenormal$ 
	(corresponding to the parameter $\upomega \in \mathbb{S}^2$)
	along the cone-tip axis $\Tranchar_{\bf{z}}(t)$. That is, if $\vec{\spherenormal}(t,u,\upomega)$ denotes the
	array of Cartesian spatial components of $\spherenormal$ viewed as a function of the geometric coordinates $(t,u,\upomega)$,
	then $\vec{\spherenormal}_{\upomega}(t) := \vec{\spherenormal}(t,t,\upomega)$.
	Moreover, in the previous expressions, we have abbreviated
	$\vec{\Psi} = \vec{\Psi} \circ \Tranchar_{\bf{z}}(t)$
	and $\pmb{\partial} \vec{\Psi} = [\pmb{\partial} \vec{\Psi}] \circ \Tranchar_{\bf{z}}(t)$.
	Integrating the ODE in time
	and using the bootstrap assumptions,
	we deduce that
	$
	|
		\vec{\spherenormal}_{\upomega}(t)
		-
		\vec{\spherenormal}_{\upomega}(0)
	|
	\lesssim
	\int_0^t
		\| \pmb{\partial} \vec{\Psi} \|_{L^{\infty}(\Sigma_{\uptau})}
	\, d \uptau
	$.
	From this estimate, \eqref{E:RESCALEDBOOTBOUNDS}, and \eqref{E:RESCALEDSTRICHARTZ}, 
	we arrive at the desired bound
	\eqref{E:SPHERENORMALIALONGCONETIPAXISSTAYSCLOSETOORIGINVALUES}.
	The desired bound for \eqref{E:LUNITIALONGCONETIPAXISSTAYSCLOSETOORIGINVALUES} 
	follows from \eqref{E:SPHERENORMALIALONGCONETIPAXISSTAYSCLOSETOORIGINVALUES},
	the identity $\Lunit = \Transport + \spherenormal$,
	and the estimate 
	$
	|
		\Transport^{\alpha}(t,t,\upomega)
		-
		\Transport^{\alpha}(0,0,\upomega)
	|
	\lesssim \uplambda^{-8 \upepsilon_0}
	$, 
	which follows from integrating the estimate 
	$|\Transport \Transport^{\alpha}|(\uptau,\uptau,\upomega) \lesssim \| \pmb{\partial} \vec{\Psi} \|_{L^{\infty}(\Sigma_{\uptau})}$
	(valid since $\Transport^{\alpha} = \Transport^{\alpha}(\vec{\Psi})$) with respect to $\uptau$
	and using \eqref{E:RESCALEDBOOTBOUNDS} and \eqref{E:RESCALEDSTRICHARTZ}.
	
	We now show that for each $(u,\upomega) \in [0,\RescaledTboot] \times \mathbb{S}^2$,
	$\lim_{t \downarrow u} \Theta_{(A)}^i(t,u,\upomega) := \Theta_{(A)}^i(u,u,\upomega)$ exists
	and that
	$$g_{cd}(u,u,\upomega) \Theta_{(A)}^c(u,u,\upomega)\Theta_{(B)}^d(u,u,\upomega)
	=
	\stgsphere(\upomega)\left(\frac{\partial}{\partial \upomega^A},\frac{\partial}{\partial \upomega^B}\right)
	+
	\mathcal{O}(\uplambda^{- \upepsilon_0}).
	$$
	The desired results can be obtained by using arguments similar to the ones given in the first paragraph of this proof,
	based on the evolution equation \eqref{E:EVOLUTIONEQUATIONALONGINTEGRALCURVESOFLUNITFORCARTESIANCOMPONENTSOFTHETAA}
	and the bootstrap assumptions, 
	including
	\eqref{E:RESCALEDBOOTBOUNDS},
	\eqref{E:RESCALEDSTRICHARTZ},
	\eqref{E:BOOTSTRAPMETRICAPPROXIMATELYROUND},
	and
	\eqref{E:LT2LINFINITYBOOTSTRAPCHIANDZETAININTERIORREGION};
	we omit the details.
	
	We now prove \eqref{E:ANGULARDERIVATIVEOFNORMALVECTORATORIGIN}.
	From the identity $\Lunit = \Transport + \spherenormal$
	and the fact that $\frac{\partial}{\partial \upomega^A} \Transport^{\alpha}(t,t,\upomega)
	= 
	\frac{\partial}{\partial \upomega^A}
	[\Transport^{\alpha}
	\circ
	\vec{\Psi} 
	\circ \Tranchar_{\bf{z}}(t) 
	]
	= 0$,
	we find that
	$\frac{\partial}{\partial \upomega^A} \spherenormal^i(t,t,\upomega) 
		= \frac{\partial}{\partial \upomega^A} \Lunit^i(t,t,\upomega)$,
	as is stated in \eqref{E:ANGULARDERIVATIVEOFNORMALVECTORATORIGIN}.
	From the fact that $\lim_{t \downarrow u} \Theta_{(A)}^i(t,u,\upomega) = \Theta_{(A)}^i(u,u,\upomega)$
	and the asymptotic initial condition \eqref{E:CONNECTIONCOEFFICIENTS0LIMITSALONGTIP}
	for $|\rgeo \sphereproject_j^a \partial_a \Lunit^i - \sphereproject_j^i|$,
	we find that 
	$\frac{\partial}{\partial \upomega^A} \spherenormal^i(t,t,\upomega) 
		= \Theta_{(A)}^i(t,t,\upomega)$,
	which finishes the proof of \eqref{E:ANGULARDERIVATIVEOFNORMALVECTORATORIGIN}.
	
	We now prove \eqref{E:RESCALEDGEOMETRICCOORDINATEVECTORFIELDALONGCONETIPAXISSTAYSCLOSETOORIGINVALUES}.
	We differentiate the ODE \eqref{E:FWTRANSPORT}
	with respect to the parameter $\upomega^A$ (that is, with the operator $\frac{\partial}{\partial \upomega^A}$),
	use the fact that
	$
	\frac{\partial}{\partial \upomega^A} [\vec{\Psi} \circ \Tranchar_{\bf{z}}(t)]
	=
	\frac{\partial}{\partial \upomega^A} \left([\pmb{\partial} \Psi] \circ \Tranchar_{\bf{z}}(t) \right)
	=
	0
	$,
	integrate the resulting ODE in time,
	and use the bootstrap assumptions,
	thereby deducing that
	\begin{align} \label{E:FIRSTSTEPRESCALEDGEOMETRICCOORDINATEVECTORFIELDALONGCONETIPAXISSTAYSCLOSETOORIGINVALUES}
	\left|
		\frac{\partial}{\partial \upomega^A} \vec{\spherenormal}_{\upomega}(t)
		-
		\frac{\partial}{\partial \upomega^A} \vec{\spherenormal}_{\upomega}(0)
	\right|
	& \lesssim
	\int_0^t
		\| \pmb{\partial} \vec{\Psi} \|_{L^{\infty}(\Sigma_{\uptau})}
		\left|
			\frac{\partial}{\partial \upomega^A} \vec{\spherenormal}_{\upomega}(t)
			-
			\frac{\partial}{\partial \upomega^A} \vec{\spherenormal}_{\upomega}(0)
		\right|
	\, d \uptau
		\\
	& \ \
	+
	\int_0^t
		\| \pmb{\partial} \vec{\Psi} \|_{L^{\infty}(\Sigma_{\uptau})}
		\left|
			\frac{\partial}{\partial \upomega^A} \vec{\spherenormal}_{\upomega}(0)
		\right|
	\, d \uptau.
	\notag
	\end{align}
	From \eqref{E:FIRSTSTEPRESCALEDGEOMETRICCOORDINATEVECTORFIELDALONGCONETIPAXISSTAYSCLOSETOORIGINVALUES}, 
	\eqref{E:RESCALEDBOOTBOUNDS}, \eqref{E:RESCALEDSTRICHARTZ}, 
	\eqref{E:ANGULARDERIVATIVEOFSPHERENORMALATINITIALCONETIPAXISPOINT} 
	(which, in view of \eqref{E:INITIALSPHEREMETRICESTIMATE},
	implies that $\left|
		\frac{\partial}{\partial \upomega^A} \vec{\spherenormal}_{\upomega}(0)
	\right| \lesssim 1$),
	and Gr\"{o}nwall's inequality,
	we find that $\frac{\partial}{\partial \upomega^A} \spherenormal^i(t,t,\upomega) = \frac{\partial}{\partial \upomega^A} \spherenormal^i(0,0,\upomega)
	+ \mathcal{O}(\uplambda^{-8 \upepsilon_0})
	$.
	From this estimate and \eqref{E:ANGULARDERIVATIVEOFNORMALVECTORATORIGININSIGMA0}, we 
	deduce that
	$
	\frac{\partial}{\partial \upomega^A} \spherenormal^i(t,t,\upomega) 
	= \Theta_{(A)}^i(0,0,\upomega)
	+
	\mathcal{O}(\uplambda^{-4 \upepsilon_0})
	$.
	Finally, from this bound and \eqref{E:ANGULARDERIVATIVEOFNORMALVECTORATORIGIN}, 
	we conclude \eqref{E:RESCALEDGEOMETRICCOORDINATEVECTORFIELDALONGCONETIPAXISSTAYSCLOSETOORIGINVALUES}.
	
	
	
	\medskip
	\noindent \textbf{Proof of the results in $\widetilde{\mathcal{M}}$}.
	We now show that \eqref{E:QUANTITATIVECONTROLOFCARTESIANCOMPONENTSRESCALEDGEOMETRICCOORDINATEANGULARDERIVATIVEVECTORVIELDINALLREGIONS} holds.
	This estimate can be obtained by using arguments similar to the ones given in the first paragraph of this proof,
	based on the evolution equation \eqref{E:EVOLUTIONEQUATIONALONGINTEGRALCURVESOFLUNITFORCARTESIANCOMPONENTSOFTHETAA}
	and the bootstrap assumptions, 
	including
	\eqref{E:RESCALEDBOOTBOUNDS},
	\eqref{E:RESCALEDSTRICHARTZ},
	\eqref{E:BOOTSTRAPMETRICAPPROXIMATELYROUND},
	and
	\eqref{E:LT2LINFINITYBOOTSTRAPCHIANDZETAININTERIORREGION}.
	The initial conditions for $\Theta_{(A)}^i$ on $\Sigma_0$ (which are relevant for the region $\widetilde{\mathcal{M}}^{(Ext)}$)
	can be related back to $\Theta_{(A)}^i(0,0,\upomega)$ via the already proven estimate \eqref{E:RESCALEDGEOMETRICCOORDINATEVECTORFIELDRELATEDTOAMEQUANTITITYATZ},
	while the initial conditions for $\Theta_{(A)}^i$ on the cone-tip axis (which are relevant for the region $\widetilde{\mathcal{M}}^{(Ext)}$)
	can be related back to $\Theta_{(A)}^i(0,0,\upomega)$
	via \eqref{E:RESCALEDGEOMETRICCOORDINATEVECTORFIELDALONGCONETIPAXISSTAYSCLOSETOORIGINVALUES};
	we omit the details.
	
	The estimate \eqref{E:QUANTITATIVECONTROLOFLUNITIALLREGIONS} can be obtained in a similar fashion
	based on the evolution equation for $\Lunit^i$ stated in \eqref{E:LUNITIANDNORMALITRANSPORT},
	the bootstrap assumptions,
	\eqref{E:RESCALEDBOOTBOUNDS},
	\eqref{E:RESCALEDSTRICHARTZ},
	and the already proven estimates
	\eqref{E:LUNITIALONGSIGMA0STAYSCLOSETOORIGINVALUES}
	and
	\eqref{E:LUNITIALONGCONETIPAXISSTAYSCLOSETOORIGINVALUES}; we omit the details.

\end{proof}

We now derive quantitative control of the integral curves of $\Lunit$ in $\widetilde{\mathcal{M}}$.

\begin{proposition}[Control of the integral curves of $\Lunit$ in $\widetilde{\mathcal{M}}$]
\label{P:CONTROLOFNULLGEODESICS}
Let $\Upsilon_{u;\upomega}(t)$ be the family of null geodesic curves
from Subsubsects.\,\ref{SSS:EIKONALINTERIOR} and \ref{SSS:EIKONALEXTERIOR},
which depend on the parameters $(u,\upomega) \in [-\frac{4}{5} \RescaledTboot,\RescaledTboot] \times \mathbb{S}^2$
and are parameterized by $t \in [[u]_+,\RescaledTboot]$
and normalized by $\Upsilon_{u;\upomega}^0(t) = t$.
Let $\upomega_{(1)},\upomega_{(2)} \in \mathbb{S}^2$,
and let $d_{\stgsphere}(\upomega_{(1)},\upomega_{(2)})$ denote their distance
with respect to the standard Euclidean round metric $\stgsphere$.
Under the assumptions of Subsect.\,\ref{SS:ASSUMPTIONS},
the following estimate for the Cartesian components $\Upsilon_{u;\upomega}^{\alpha}(t)$
(which can be identified with the Cartesian coordinate functions $x^{\alpha}$, viewed as a function of $(t,u,\upomega)$)
holds for $u \in  [-\frac{4}{5} \RescaledTboot,\RescaledTboot]$ and $t \in [[u]_+,\RescaledTboot]$:
\begin{align} \label{E:CONTROLOFNULLGEODESICS}
	\sum_{\alpha=0}^3
	|
	\Upsilon_{u;\upomega_{(1)}}^{\alpha}(t)
	-
	\Upsilon_{u;\upomega_{(2)}}^{\alpha}(t)
	|
	\approx
	\rgeo d_{\stgsphere}(\upomega_{(1)},\upomega_{(2)}).
\end{align}
\end{proposition}

\begin{proof}
At the end of the proof, we will show that the following two estimates hold for 
$u \in [-\frac{4}{5} \RescaledTboot,\RescaledTboot]$,
$t \in [[u]_+,\RescaledTboot]$,
and
$\upomega \in \mathbb{S}^2$, 
($A=1,2$ and $i=1,2,3)$:
\begin{align} 
\Upsilon_{u;\upomega}^i(t)
	& 
	=
	\Upsilon_{u;\upomega}^i([u]_+)
	+
	(t - [u]_+)
	\left\lbrace
		\Lunit^i(0,0,\upomega)
		+
		\mathcal{O}(\uplambda^{-4 \upepsilon_0})
	\right\rbrace,
	\label{E:INTEGRALCURVESOFLUNITSEPARATEDESTIMATEFORLARGEANGLES}
		\\
	\frac{\partial}{\partial \upomega^A} \Upsilon_{u;\upomega}^i(t)
	& = \rgeo
			\left\lbrace
				\Theta_{(A)}^i(0,0,\upomega)
				+
				\mathcal{O}(\uplambda^{-4 \upepsilon_0})
			\right\rbrace.
\label{E:INTEGRALCURVESOFLUNITSEPARATEDESTIMATEFORSMALLANGLES}
\end{align}
From \eqref{E:INTEGRALCURVESOFLUNITSEPARATEDESTIMATEFORSMALLANGLES}
and the properties of the (linearly independent) set
$\lbrace \Theta_{(1)}(0,0,\upomega), \Theta_{(2)}(0,0,\upomega) \rbrace$
shown in Lemma\allowbreak~\ref{L:LUNITIALONGCONETIPAXISISC1INANGLEVARIABLES},
it follows that the map 
$\upomega \rightarrow \left(\Upsilon_{u;\upomega}^1(t),\Upsilon_{u;\upomega}^2(t),\Upsilon_{u;\upomega}^3(t)\right)$
has an injective differential and, in particular, there exists $0 < \upbeta < \pi$ such that
if $\uplambda$ is sufficiently large, then
\eqref{E:CONTROLOFNULLGEODESICS} holds whenever 
$d_{\stgsphere}(\upomega_{(1)},\upomega_{(2)}) < \upbeta$.
From \eqref{E:DISTANCEBETWEENNORMALCARTESIANCOMPONENTSISCOMPARABLETOS2DISTANCEBETWEENANGLES},
\eqref{E:INTEGRALCURVESOFLUNITSEPARATEDESTIMATEFORLARGEANGLES},
and the fact that $\Upsilon_{u;\upomega}^i(u)$ is independent of $\upomega$ when $u \in [0,\RescaledTboot]$,
it follows that for this fixed value of $\upbeta$, 
if $\uplambda > 0$ is sufficiently large,
then \eqref{E:CONTROLOFNULLGEODESICS}
holds whenever $\upbeta \leq d_{\stgsphere}(\upomega_{(1)},\upomega_{(2)}) \leq \pi$,
$u \in [0,\RescaledTboot]$, and $t \in [u,\RescaledTboot]$.
\eqref{E:CONTROLOFNULLGEODESICS} can be proved in 
the remaining case, in which
$\upbeta \leq d_{\stgsphere}(\upomega_{(1)},\upomega_{(2)}) \leq \pi$,
$u \in [-\frac{4}{5} \RescaledTboot,0]$, and $t \in [0,\RescaledTboot]$,
via a similar argument that also takes into account
the estimate \eqref{E:ESTIMATEFORCARTESIANCOORDINATESALONGSIGMA0}, as we now explain. 
\eqref{E:ESTIMATEFORCARTESIANCOORDINATESALONGSIGMA0} is
relevant in that 
the identity $\Lunit^i = \Transport^i + \spherenormal^i$,
the fact that $\Transport^i(0,0,\upomega)$ is independent of $\upomega$, 
and the estimates
\eqref{E:ESTIMATEFORCARTESIANCOORDINATESALONGSIGMA0} and \eqref{E:INTEGRALCURVESOFLUNITSEPARATEDESTIMATEFORLARGEANGLES}
collectively 
imply that for $u \in [-\frac{4}{5} \RescaledTboot,0]$,
$t \in [0,\RescaledTboot]$,
and $\upomega_{(1)},\upomega_{(2)} \in \mathbb{S}^2$, 
we have
$$
\sum_{\alpha=0}^3
	|
	\Upsilon_{u;\upomega_{(1)}}^{\alpha}(t)
	-
	\Upsilon_{u;\upomega_{(2)}}^{\alpha}(t)
	|
	=
	(|u| + t)
	\left\lbrace
	\sum_{i=1}^3
		|
		\Lunit^i(0,0,\upomega_{(1)})
		-
		\Lunit^i(0,0,\upomega_{(2)})
		|
		+
	\mathcal{O}(\uplambda^{-4 \upepsilon_0})
	\right\rbrace.
$$ In view of \eqref{E:DISTANCEBETWEENNORMALCARTESIANCOMPONENTSISCOMPARABLETOS2DISTANCEBETWEENANGLES}
and the assumption $\upbeta \leq d_{\stgsphere}(\upomega_{(1)},\upomega_{(2)})$,
we see that for $\uplambda$ sufficiently large, the
$\mathcal{O}(\uplambda^{-4 \upepsilon_0})$ term is negligible.
Since $\rgeo = |u| + t$ when $u \leq 0$,
we have completed the proof of \eqref{E:CONTROLOFNULLGEODESICS}.

It remains for us to prove \eqref{E:INTEGRALCURVESOFLUNITSEPARATEDESTIMATEFORLARGEANGLES}--\eqref{E:INTEGRALCURVESOFLUNITSEPARATEDESTIMATEFORSMALLANGLES}.
The estimate \eqref{E:INTEGRALCURVESOFLUNITSEPARATEDESTIMATEFORSMALLANGLES} 
follows directly from multiplying
\eqref{E:QUANTITATIVECONTROLOFCARTESIANCOMPONENTSRESCALEDGEOMETRICCOORDINATEANGULARDERIVATIVEVECTORVIELDINALLREGIONS} by $\rgeo$
and considering the definitions of $\Theta_{(A)}^i$ and $\Upsilon_{u;\upomega}^i(t)$.
To derive \eqref{E:INTEGRALCURVESOFLUNITSEPARATEDESTIMATEFORLARGEANGLES}, 
we first use \eqref{E:QUANTITATIVECONTROLOFLUNITIALLREGIONS}  to deduce that
$\frac{\partial}{\partial t} \Upsilon_{u;\upomega}^i(t) = \Lunit^i(t,u,\upomega) = \Lunit^i(0,0,\upomega) 
+ \mathcal{O}(\uplambda^{-4 \upepsilon_0})$.
Integrating this estimate with respect to time starting from the time value $[u]_+$,
we conclude \eqref{E:INTEGRALCURVESOFLUNITSEPARATEDESTIMATEFORLARGEANGLES}.
\end{proof}

We now derive the main consequence of Prop.\,\ref{P:CONTROLOFNULLGEODESICS}:
a corollary that yields
$L_t^2 L_u^{\infty} C_{\upomega}^{0,\updelta_0}(\widetilde{\mathcal{M}})$ estimates
for various fluid variables.

\begin{corollary}[$L_t^2 L_u^{\infty} C_{\upomega}^{0,\updelta_0}(\widetilde{\mathcal{M}})$ estimates]
	\label{C:ESTIMATEFORANGLEHOLDERCONTINUITYOFPARTIALPSIPARTIALVORTANDPARTIALGRADENT}
	Under the assumptions of Subsect.\,\ref{SS:ASSUMPTIONS},
	we have the following estimates:
	\begin{align} \label{E:ESTIMATEFORANGLEHOLDERCONTINUITYOFPARTIALPSIPARTIALVORTANDPARTIALGRADENT}
		\|
			\pmb{\partial} \vec{\Psi}
		\|_{L_t^2 L_u^{\infty} C_{\upomega}^{0,\updelta_0}(\widetilde{\mathcal{M}})},
			\,
		\|
			(\partial \vec{\vortrenormalized},\partial \vec{\GradEnt}) 
		\|_{L_t^2 L_u^{\infty} C_{\upomega}^{0,\updelta_0}(\widetilde{\mathcal{M}})},
			\,
		\|
			 (\vec{\VortVort},\DivGradEnt) 
		\|_{L_t^2 L_u^{\infty} C_{\upomega}^{0,\updelta_0}(\widetilde{\mathcal{M}})}
		& \lesssim 
			\uplambda^{-1/2 - 3 \upepsilon_0}.
\end{align}
Moreover,
\begin{align} \label{E:ESTIMATEFORANGLEHOLDERCONTINUITYOFPSIVORTANDGRADENT}
	\|
		(\vec{\Psi},\vec{\vortrenormalized},\vec{\GradEnt}) 
	\|_{L_t^{\infty} L_u^{\infty} C_{\upomega}^{0,\updelta_0}(\widetilde{\mathcal{M}})}
	& \lesssim 1.
\end{align}
\end{corollary}

\begin{proof}
	We prove \eqref{E:ESTIMATEFORANGLEHOLDERCONTINUITYOFPARTIALPSIPARTIALVORTANDPARTIALGRADENT} only for
	the first term on the LHS; the remaining terms on LHS~\eqref{E:ESTIMATEFORANGLEHOLDERCONTINUITYOFPARTIALPSIPARTIALVORTANDPARTIALGRADENT} 
	can be bounded using the same arguments.
	To proceed, we first use \eqref{E:CONTROLOFNULLGEODESICS}
	to deduce that 
	$$
	\frac{|\pmb{\partial} \vec{\Psi}(t,u,\upomega_{(1)}) - \pmb{\partial} \vec{\Psi}(t,u,\upomega_{(2)})|}
			{[\rgeo d_{\stgsphere}(\upomega_{(1)},\upomega_{(2)})]^{\updelta_0}}
	\lesssim 
	\|
		\pmb{\partial} \vec{\Psi}
	\|_{C_x^{0,\updelta_0}(\widetilde{\Sigma}_t)}.
$$
From this bound, the estimate \eqref{E:RGEOANDUBOUNDS} for $\rgeo$,
	and the inequality $\uplambda^{(1 - 8 \upepsilon_0) \updelta_0} 
	\leq \uplambda^{\updelta_0} \leq \uplambda^{\upepsilon_0}$
	(see \eqref{E:EPSILON0INEQUALITY}--\eqref{E:DELTA0DEF}),
we find, considering separately the cases $0 \leq \rgeo \leq 1$ and $1 \leq \rgeo$, 
that
$
\|
	\pmb{\partial} \vec{\Psi}
\|_{C_{\upomega}^{0,\updelta_0}(S_{t,u})}
\lesssim 
\uplambda^{\upepsilon_0}
\|
	\pmb{\partial} \vec{\Psi}
\|_{C_x^{0,\updelta_0}(\widetilde{\Sigma}_t)}
$.
From this bound and \eqref{E:RESCALEDSOLUTIONHOLDERESTIMATE},
we conclude the desired estimate \eqref{E:ESTIMATEFORANGLEHOLDERCONTINUITYOFPARTIALPSIPARTIALVORTANDPARTIALGRADENT}.

	To prove \eqref{E:ESTIMATEFORANGLEHOLDERCONTINUITYOFPSIVORTANDGRADENT},
	we note that Prop.\,\ref{P:TOPORDERENERGYESTIMATES} and Sobolev embedding $H^{\Sob}(\Sigma_t) \hookrightarrow C_x^{0,\updelta_0}(\widetilde{\Sigma}_t)$
	imply that the \emph{non-rescaled} solution variables $(\vec{\Psi},\vec{\vortrenormalized},\vec{\GradEnt})$
	are bounded in the norm
	$\| \cdot \|_{L_t^{\infty} C_x^{0,\updelta_0}(\widetilde{\mathcal{M}})}$ by $\lesssim 1$. It follows that the 
	rescaled solution variables on LHS~\eqref{E:ESTIMATEFORANGLEHOLDERCONTINUITYOFPSIVORTANDGRADENT}
	(as defined in Subsect.\,\ref{SS:RESCALEDSOLUTION} and under the conventions of Subsect.\,\ref{SS:NOMORELAMBDA})
	are bounded in the norm
	$\| \cdot \|_{L^{\infty}(\widetilde{\mathcal{M}})}$ by $\lesssim 1$
	and in the norm
	$\| \cdot \|_{L_t^{\infty} \dot{C}_x^{0,\updelta_0}(\widetilde{\mathcal{M}})}$ by $\lesssim \uplambda^{- \updelta_0}$.
	From these estimates and arguments similar to the ones given in the previous paragraph,
	we conclude \eqref{E:ESTIMATEFORANGLEHOLDERCONTINUITYOFPSIVORTANDGRADENT}.
\end{proof}

\subsection{\texorpdfstring{Estimates for transport equations along the integral curves of $\Lunit$ in H\"{o}lder spaces in the angular variables $\upomega$}{Estimates for transport equations along the integral curves of Lunit in Holder spaces in the angular variables upomega}}
\label{SS:ANGULARESTIAMTESFORTRANSPORTINHOLDER}
We now derive estimates for transport equations along the integral curves of $\Lunit$
with initial data and source terms that are H\"{o}lder-class in the geometric angular
variables $\upomega$.

	\begin{lemma}[Estimates for transport equations along the integral curves of $\Lunit$ in H\"{o}lder spaces with respect to $\upomega$]
		\label{L:TRANSPORTESTIMATESINANGULARHOLDERSPACES}
		Let $\widetilde{\mathcal{C}}_u \subset \widetilde{\mathcal{M}}$.
		Let $\mathfrak{F}$ be a smooth scalar-valued function on $\widetilde{\mathcal{C}}_u$ 
		and let $\mathring{\varphi}$ be a smooth scalar-valued function on $S_{[u]_+,u}$. 
		For $(t,\upomega) \in [[u]_+,\RescaledTboot] \times \mathbb{S}^2$, 
		let the scalar-valued function $\varphi$ be a smooth solution to the following inhomogeneous transport 
		equation with data given on $S_{[u]_+,u}$:
		\begin{subequations}
		\begin{align}
			\Lunit \varphi(t,u,\upomega)
			& = \mathfrak{F}(t,u,\upomega),
				\label{E:ANGULARINHOMOGENEOUSTRANSPORT} \\
			\varphi([u]_+,u,\upomega) & = \mathring{\varphi}(\upomega).
			\label{E:ANGULARDATAINHOMOGENEOUSTRANSPORT}
		\end{align}
		\end{subequations}
		Under the assumptions of Subsect.\,\ref{SS:ASSUMPTIONS}, the following estimate holds for $t \in [[u]_+,\RescaledTboot]$:
		\begin{align} \label{E:ANGULARHOLDERESTIMATEFORSOLUTIONTOINHOMOGENEOUSTRANSPORT}
			\| \varphi \|_{C_{\upomega}^{0,\updelta_0}(S_{t,u})}
			& \lesssim 
				\| \mathring{\varphi} \|_{C_{\upomega}^{0,\updelta_0}(S_{[u]_+,u})}
				+
				\int_{[u]_+}^t
					\| \mathfrak{F} \|_{C_{\upomega}^{0,\updelta_0}(S_{\uptau,u})}
				\, d \uptau.
		\end{align}
		Moreover,
		\begin{align} \label{E:ZERODATAANGULARHOLDERESTIMATEFORSOLUTIONTOINHOMOGENEOUSTRANSPORT}
			\left\| \int_{[u]_+}^t \mathfrak{F}(\uptau,u,\upomega) \, d \uptau \right\|_{C_{\upomega}^{0,\updelta_0}(S_{t,u})}
			& \lesssim 
				\int_{[u]_+}^t
					\| \mathfrak{F} \|_{C_{\upomega}^{0,\updelta_0}(S_{\uptau,u})}
				\, d \uptau.
		\end{align}
	\end{lemma}	
	
	\begin{proof}
		The lemma is a straightforward consequence of the fundamental theorem of calculus
		and the fact that the angular geometric coordinate functions $\lbrace \upomega^A \rbrace_{A=1,2}$
		are constant along the integral curves of $\Lunit = \frac{\partial}{\partial t}$.
	\end{proof}

\subsection{\texorpdfstring{Calderon--Zygmund- and Schauder-type Hodge estimates on $S_{t,u}$}{Calderon--Zygmund- and Schauder-type Hodge estimates on S t,u}}
\label{SS:HODGESTU}
Some of the tensorfields under study are solutions to Hodge systems on $S_{t,u}$.
To control them, we will use the Calderon--Zygmund and Schauder-type estimates provided by the 
following lemma.

\begin{lemma}[Calderon--Zygmund- and Schauder-type Hodge estimates on $S_{t,u}$]
\label{L:CALDERONZYGMUNDHODGESTU}
Under the assumptions of Subsect.\,\ref{SS:ASSUMPTIONS} and the estimates of Prop.\,\ref{P:ESTIMATESFORFLUIDVARIABLES},
if $\upxi$ is an $S_{t,u}$-tangent one-form and $2 \leq \leb \leq p$ (where $p$ is as in \eqref{E:BOUNDSONLEBESGUEEXPONENTP}),
then
\begin{align} \label{E:ONEFORMSTUCALDERONZYGMUNDHODGEESTIMATES}
	\| \angD \upxi \|_{L_{\gsphere}^{\leb}(S_{t,u})}
	+
	\| \rgeo^{-1} \upxi \|_{L_{\gsphere}^{\leb}(S_{t,u})}
	& \lesssim
	\| \angdiv \upxi \|_{L_{\gsphere}^{\leb}(S_{t,u})}
	+
	\| \angcurl \upxi \|_{L_{\gsphere}^{\leb}(S_{t,u})}.
\end{align}
Similarly, if $\upxi$ is an $S_{t,u}$-tangent type $\binom{0}{2}$ symmetric trace-free tensorfield, 
then
\begin{align} \label{E:SYMMETRICTRACEFREESTUCALDERONZYGMUNDHODGEESTIMATES}
	\| \angD \upxi \|_{L_{\gsphere}^{\leb}(S_{t,u})}
	+
	\| \rgeo^{-1} \upxi \|_{L_{\gsphere}^{\leb}(S_{t,u})}
	& \lesssim
	\| \angdiv \upxi \|_{L_{\gsphere}^{\leb}(S_{t,u})}.
\end{align}

Moreover, let $\upxi$ be an $S_{t,u}$-tangent type $\binom{0}{2}$ symmetric trace-free tensorfield,
let $\mathfrak{F}_{(1)}$ be a scalar function,
let $\mathfrak{F}_{(2)}$ be an $S_{t,u}$-tangent type $\binom{0}{2}$ symmetric tensorfield,
and let $\mathfrak{G}$ be an $S_{t,u}$-tangent one-form.
Assume that\footnote{On RHS~\eqref{E:ANGDIVSYMMETRICTRACEFREESTUTENSORFIELD}, 
we made a minor change compared to \cite{qW2017}*{Proposition~5.9}:
we allowed for the presence of the $\mathfrak{F}_{(2)}$ term,
in particular so that we can handle the second term on RHS~\eqref{E:DIVDIVTRFREECHISCHEMATIC}. 
We will now explain why the estimate \eqref{E:HODGEESTIMATENODERIVATIVESONLHSANGDIVSYMMETRICTRACEFREESTUTENSORFIELD}
holds in the presence of this new term.
First, we can split the
$S_{t,u}$-tangent type $\binom{0}{2}$ symmetric tensorfield
$\mathfrak{F}_{(2)}$ into its trace-free and pure-trace parts.
We then bring the trace-free part over to the left-hand side of the equation
(so that the new LHS is of the form $\angdiv (\upxi - \hat{\mathfrak{F}}_{(2)})$),
while we absorb the pure-trace part of $\mathfrak{F}_{(2)}$ into the
$\mathfrak{F}_{(1)}$ term. This allows one to reduce the
proof of \eqref{E:HODGEESTIMATENODERIVATIVESONLHSANGDIVSYMMETRICTRACEFREESTUTENSORFIELD}
to the case in which the $\mathfrak{F}_{(2)}$ term on RHS~\eqref{E:ANGDIVSYMMETRICTRACEFREESTUTENSORFIELD}
is absent, as was assumed in \cite{qW2017}*{Proposition~5.9}.
\label{FN:CZERRORFIX}}
\begin{align} \label{E:ANGDIVSYMMETRICTRACEFREESTUTENSORFIELD}
	\angdiv \upxi
	& = 
		\angD \mathfrak{F}_{(1)}
		+
		\angdiv \mathfrak{F}_{(2)}
		+
		\mathfrak{G}.
\end{align}
Let $2 < \leb < \infty$, and let $\leb'$ be defined by $\frac{1}{2} + \frac{1}{\leb} = \frac{1}{\leb'}$.
Then the following estimate holds:
\begin{align} \label{E:HODGEESTIMATENODERIVATIVESONLHSANGDIVSYMMETRICTRACEFREESTUTENSORFIELD}
	\| \upxi \|_{L_{\gsphere}^{\leb}(S_{t,u})}
	& \lesssim
		\sum_{i=1,2}
		\| \mathfrak{F}_{(i)} \|_{L_{\gsphere}^{\leb}(S_{t,u})}
		+
		\| \mathfrak{G} \|_{L_{\gsphere}^{\leb'}(S_{t,u})}.
\end{align}
In addition, if $2 < \leb \leq p$ (where $p$ is as in Subsect.\,\ref{SS:MAINESTIMATESFOREIKONALFUNCTIONQUANTITIES}),
then
\begin{align}	 \label{E:ANGULARHOLDERHODGEESTIMATENODERIVATIVESONLHSANGDIVSYMMETRICTRACEFREESTUTENSORFIELD}
	\| \upxi \|_{C_{\upomega}^{0,1 - \frac{2}{\leb}}(S_{t,u})}
	& \lesssim
		\sum_{i=1,2}
		\| \mathfrak{F}_{(i)} \|_{C_{\upomega}^{0,1 - \frac{2}{\leb}}}
		+
		\| \rgeo \mathfrak{G} \|_{L_{\upomega}^{\leb}(S_{t,u})}.
\end{align}

Similarly, assume that $\upxi$, $\mathfrak{F}_{(1)}$, and $\mathfrak{F}_{(2)}$ are $S_{t,u}$-tangent one-forms
and $\mathfrak{G}_{(1)}$, and $\mathfrak{G}_{(2)}$
are scalar functions such that $\upxi$ satisfies the following Hodge system:
\begin{subequations}
\begin{align} \label{E:ANGDIVSTUONEFORM}
	\angdiv \upxi
	& = \angdiv \mathfrak{F}_{(1)}
		+
		\mathfrak{G}_{(1)},
			\\
	\angcurl \upxi
	& = \angcurl \mathfrak{F}_{(2)}
		+
		\mathfrak{G}_{(2)}.
		\label{E:ANGCURLSTUONEFORM}
\end{align}
\end{subequations}
Then under the same assumptions on $\leb$ and $\leb'$ stated in the previous paragraph,
$\upxi$ satisfies the estimates 
\eqref{E:HODGEESTIMATENODERIVATIVESONLHSANGDIVSYMMETRICTRACEFREESTUTENSORFIELD}--\eqref{E:ANGULARHOLDERHODGEESTIMATENODERIVATIVESONLHSANGDIVSYMMETRICTRACEFREESTUTENSORFIELD}
with $\mathfrak{G} := (\mathfrak{G}_{(1)}, \mathfrak{G}_{(2)})$.

Finally, assume that 
$\upxi$,
$\mathfrak{F}=(\mathfrak{F}_{(1)}, \mathfrak{F}_{(2)})$,
and
$\mathfrak{G}$
are $S_{t,u}$ tensorfields of the type from the previous two paragraphs
(in particular satisfying \eqref{E:ANGDIVSYMMETRICTRACEFREESTUTENSORFIELD} 
or
\eqref{E:ANGDIVSTUONEFORM}--\eqref{E:ANGCURLSTUONEFORM}).
Assume that $\mathfrak{F}$ is the $S_{t,u}$-projection of a spacetime tensorfield $\widetilde{\mathfrak{F}}$ 
or is a contraction of a spacetime tensorfield $\widetilde{\mathfrak{F}}$ against $\Lunit$, $\uLunit$, or $\spherenormal$.
If $\leb > 2$, 
$1 \leq m < \infty$,
and $\updelta' > 0$ is sufficiently small, 
then the following estimates hold,
where $\vec{\widetilde{\mathfrak{F}}}$ denotes the array of (scalar) Cartesian component functions
of $\widetilde{\mathfrak{F}}$:
\begin{align} \label{E:LINFTYHODGEESTIMATENODERIVATIVESONLHSINVOLVINGCARTESIANCOMPONENTSONRHS}
	\| \upxi \|_{L_{\upomega}^{\infty}(S_{t,u})}
	& \lesssim
		\left\| 
			\upnu^{\updelta'} P_{\upnu} \vec{\widetilde{\mathfrak{F}}} 
		\right\|_{\ell_{\upnu}^m L_{\upomega}^{\infty}(S_{t,u})}
		+
		\left\| 
			\vec{\widetilde{\mathfrak{F}}} 
		\right\|_{L_{\upomega}^{\infty}(S_{t,u})}
		+
		\| \rgeo^{1 - \frac{2}{\leb}} \mathfrak{G} \|_{L_{\gsphere}^{\leb}(S_{t,u})}.
\end{align}

\end{lemma}

\begin{proof}[Discussion of proof]
		Aside from \eqref{E:ANGULARHOLDERHODGEESTIMATENODERIVATIVESONLHSANGDIVSYMMETRICTRACEFREESTUTENSORFIELD},
		these estimates are a restatement of 
		\cite{qW2017}*{Lemma~5.8},
		\cite{qW2017}*{Proposition~5.9},
		and
		\cite{qW2017}*{Proposition~5.10}.
		Thanks to the bootstrap assumptions and the estimates of Prop.\,\ref{P:ESTIMATESFORFLUIDVARIABLES}, 
		the estimates can be proved using the same arguments given in
		\cite{qW2014}*{Lemma~2.18},
		\cite{sKiR2005c}*{Proposition~6.20},
		and
		\cite{qW2014}*{Proposition~3.5}.
		
		The elliptic Schauder-type estimate 
		\eqref{E:ANGULARHOLDERHODGEESTIMATENODERIVATIVESONLHSANGDIVSYMMETRICTRACEFREESTUTENSORFIELD}
		for Hodge systems
		can be proved using a perturbative argument, 
		that is, using the (standard) fact that
		it holds on $\mathbb{S}^2$ equipped with the standard round metric $\stgsphere$,
		and then obtaining the desired estimate perturbatively,
		with the help of the bootstrap assumptions
		\eqref{E:BOOTSTRAPMETRICAPPROXIMATELYROUND}--\eqref{E:DERIVATIVESBOOTSTRAPMETRICAPPROXIMATELYROUND}
		(which imply that $\rgeo^{-2} \gsphere$ is close to $\stgsphere$)
		and the Morrey-type estimate \eqref{E:EUCLIDEANFORMMORREYONSTU}.
		Here we will give a detailed proof of 
		\eqref{E:ANGULARHOLDERHODGEESTIMATENODERIVATIVESONLHSANGDIVSYMMETRICTRACEFREESTUTENSORFIELD}
		for one-forms $\upxi$ that solve the system
		\eqref{E:ANGDIVSTUONEFORM}--\eqref{E:ANGCURLSTUONEFORM}.
		The estimate \eqref{E:ANGULARHOLDERHODGEESTIMATENODERIVATIVESONLHSANGDIVSYMMETRICTRACEFREESTUTENSORFIELD}
		for $S_{t,u}$-tangent type $\binom{0}{2}$ symmetric trace-free tensorfields $\upxi$
		that solve \eqref{E:ANGDIVSYMMETRICTRACEFREESTUTENSORFIELD}
		can be proved using similar arguments, and we omit those details.

		To proceed, we let $\angD$, $\gsphere$, $\Gamma$, 
		${^{(0)}\angD}$, $\stgsphere$, and ${^{(0)}\Gamma}$
		be as in our proof of \eqref{E:EUCLIDEANFORMMORREYONSTU}.
		Let $\angdiv \upxi$ denote the divergence of $\upxi$
		with respect to $\gsphere$, and let
		${^{(0)}\mbox{\upshape{div} $\mkern-17mu /$\,} \upxi}$
		denote the divergence of $\upxi$ with respect to $\stgsphere$.
		Let $\upepsilon \mkern-9mu /\,$
		denote the type $\binom{0}{2}$ volume form of $\gsphere$,
		let ${^{(0)}\upepsilon \mkern-9mu /\,}$
		denote the type $\binom{0}{2}$ volume form of $\stgsphere$,
		and let ${\upepsilon \mkern-9mu /\,}^{\# \#}$
		denote the type $\binom{2}{0}$ volume form of $\gsphere$,
		i.e., the dual of 
		${\upepsilon \mkern-9mu /\,}$ with respect to $\gsphere$.
		Then by \eqref{E:ANGDIVSTUONEFORM}--\eqref{E:ANGCURLSTUONEFORM},
		$\upxi$ satisfies the following equations, 
		schematically depicted relative to the geometric
		angular coordinates, where
		$\mbox{Id}$ denotes the type $\binom{1}{1}$ identity
		tensorfield,
		$[\stgsphere \cdot (\rgeo^2 \gsphere^{-1})]_{\ B}^A := \stgsphere_{BC} (\rgeo^2 \gsphere^{-1})^{AC}$,
		and
		$[{^{(0)}\upepsilon \mkern-9mu /\,}
			\cdot
			(\rgeo^2 {\upepsilon \mkern-9mu /\,}^{\# \#})]_{\ B}^A 
			:= 
			{^{(0)}\upepsilon \mkern-9mu /\,}_{BC} (\rgeo^2 {\upepsilon \mkern-9mu /\,}^{\# \#})^{AC}$
		(and note that if $\gsphere = \rgeo^2 \stgsphere$, 
		then
		$\stgsphere \cdot (\rgeo^2 \gsphere^{-1}) = \mbox{Id}$
		and
		${^{(0)}\upepsilon \mkern-9mu /\,}
			\cdot
			(\rgeo^2 {\upepsilon \mkern-9mu /\,}^{\# \#})
			=
			- \mbox{Id}
		$):
		\begin{align} \label{E:ROUNDSPHEREANGDIVSTUONEFORM}
		{^{(0)}\mbox{\upshape{div} $\mkern-17mu /$\,} \upxi}
		& = 
		{^{(0)}\mbox{\upshape{div} $\mkern-17mu /$\,}}
		\left\lbrace
			[\mbox{Id} - \stgsphere \cdot (\rgeo^2 \gsphere^{-1})]
			\cdot
			\upxi
		\right\rbrace
		+
		{^{(0)}\mbox{\upshape{div} $\mkern-17mu /$\,}}
		\left\lbrace
			\stgsphere 
			\cdot 
			(\rgeo^2 \gsphere^{-1})
			\cdot
			\mathfrak{F}_{(1)}
		\right\rbrace
				\\
		& \ \
			+
			\rgeo^2 \mathfrak{G}_{(1)}
			+
			({^{(0)}\Gamma} - \Gamma)
			\cdot
			(\rgeo^2 \gsphere^{-1})
			\cdot
			\mathfrak{F}_{(1)}
			+
			({^{(0)}\Gamma} - \Gamma)
			\cdot
			(\rgeo^2 \gsphere^{-1})
			\cdot
			\upxi,
			\notag 
			\\
	{^{(0)}\mbox{\upshape{curl} $\mkern-17mu /$\,} \upxi}
	& = 
		{^{(0)}\mbox{\upshape{curl} $\mkern-17mu /$\,}}
		\left\lbrace
			[\mbox{Id} 
			+ 
			{ ^{(0)}\upepsilon \mkern-9mu /\,}
			\cdot
			(\rgeo^2 {\upepsilon \mkern-9mu /\,}^{\# \#})]
			\cdot
			\upxi
		\right\rbrace
			-
			{^{(0)}\mbox{\upshape{curl} $\mkern-17mu /$\,}}
			\left\lbrace
			{ ^{(0)} \upepsilon \mkern-9mu /\,}
			\cdot
			(\rgeo^2 {\upepsilon \mkern-9mu /\,}^{\# \#})
			\cdot
			\mathfrak{F}_{(2)}
		\right\rbrace
		\label{E:ROUNDSPHEREANGCURLSTUONEFORM} \\
	& \ \
		+
		\rgeo^2 \mathfrak{G}_{(2)}
			+
			({^{(0)}\Gamma} - \Gamma)
			\cdot
			(\rgeo^2  {\upepsilon \mkern-9mu /\,}^{\# \#})
			\cdot
			\mathfrak{F}_{(2)}
			+
			({^{(0)}\Gamma} - \Gamma)
			\cdot
			(\rgeo^2 {\upepsilon \mkern-9mu /\,}^{\# \#})
			\cdot
			\upxi.
		\notag
\end{align}
	We view \eqref{E:ROUNDSPHEREANGDIVSTUONEFORM}--\eqref{E:ROUNDSPHEREANGCURLSTUONEFORM}
	as a div-curl system on the standard round sphere.
	To control the solutions, we will use the following simple product-type estimate, 
which can easily be seen to be valid for $S_{t,u}$-tangent tensorfields $\upxi_{(1)}$ and $\upxi_{(2)}$,
where ``$\cdot$'' schematically denotes tensor products and natural contractions:
\begin{align} \label{E:SIMPLEHOLDERPRODUCESTIMATEONSPHERES}
\| \upxi_{(1)} \cdot \upxi_{(2)}\|_{C_{\upomega}^{0,\updelta_0}(S_{t,u})}
& \lesssim
\| \upxi_{(1)} \|_{L_{\upomega}^{\infty}(S_{t,u})}
\| \upxi_{(2)} \|_{C_{\upomega}^{0,\updelta_0}(S_{t,u})}
+
\| \upxi_{(2)} \|_{L_{\upomega}^{\infty}(S_{t,u})}
\| \upxi_{(1)} \|_{C_{\upomega}^{0,\updelta_0}(S_{t,u})}.
\end{align}
	From \eqref{E:ROUNDSPHEREANGDIVSTUONEFORM}--\eqref{E:ROUNDSPHEREANGCURLSTUONEFORM}
	and the fact that
		the analog of \eqref{E:ANGULARHOLDERHODGEESTIMATENODERIVATIVESONLHSANGDIVSYMMETRICTRACEFREESTUTENSORFIELD}
		holds on the standard round sphere,
		we have, in view of Def.\,\ref{D:HOLDERNORMSINGEOMETRICANGULARVARIABLES} and \eqref{E:STNORMVSGSPHERENORMCOMPARISONWITHPOWERSOFR},
		\eqref{E:SIMPLEHOLDERPRODUCESTIMATEONSPHERES},
		the Morrey estimate \eqref{E:EUCLIDEANFORMMORREYONSTU} for the standard round sphere,
		and \eqref{E:BOOTSTRAPMETRICAPPROXIMATELYROUND},
		the following estimate:
		\begin{align}	 \label{E:ROUNDSPHEREANGULARHOLDERHODGEESTIMATENODERIVATIVESONLHSANGDIVSYMMETRICTRACEFREESTUTENSORFIELD}
			\| \rgeo \upxi \|_{C_{\upomega}^{0,1 - \frac{2}{\leb}}(S_{t,u})}
			& \lesssim
				\sum_{A,B,C=1,2}
				\left\| 
				\frac{\partial}{\partial \upomega^C}
					\left[\mbox{Id}_{\ B}^A - \stgsphere_{BD} \cdot (\rgeo^2 \gsphere^{-1})^{AD} \right]
				\right\|_{L_{\upomega}^{\leb}(S_{t,u})}
				\| \rgeo \upxi \|_{C_{\upomega}^{0,1 - \frac{2}{\leb}}(S_{t,u})}
				\\
		& \ \
				+
				\sum_{A,B=1,2}
				\left\| 
					\mbox{Id}_{\ B}^A - \stgsphere_{BC} \cdot (\rgeo^2 \gsphere^{-1})^{AC} 
				\right\|_{L_{\upomega}^{\infty}(S_{t,u})}
				\| \rgeo \upxi \|_{C_{\upomega}^{0,1 - \frac{2}{\leb}}(S_{t,u})}
				\notag \\
		& \ \
				+
				\sum_{A,B,C=1,2}
				\left\| 
				\frac{\partial}{\partial \upomega^C}
					\left[
						\mbox{Id}_{\ B}^A 
					+ 
					{ ^{(0)}\upepsilon \mkern-9mu /\,}_{BD} 
					(\rgeo^2 {\upepsilon \mkern-9mu /\,}^{\# \#})^{AD} 
					\right]
				\right\|_{L_{\upomega}^{\leb}(S_{t,u})}
				\| \rgeo \upxi \|_{C_{\upomega}^{0,1 - \frac{2}{\leb}}(S_{t,u})}
				\notag \\
		& \ \
				+
				\sum_{A,B=1,2}
				\left\| 
					\mbox{Id}_{\ B}^A 
					+ 
					{ ^{(0)}\upepsilon \mkern-9mu /\,}_{BC} 
					(\rgeo^2 {\upepsilon \mkern-9mu /\,}^{\# \#})^{AC} 
				\right\|_{L_{\upomega}^{\infty}(S_{t,u})}
				\| \rgeo \upxi \|_{C_{\upomega}^{0,1 - \frac{2}{\leb}}(S_{t,u})}
				\notag \\
		& \ \
		+
		\sum_{A,B,C=1,2}
		\left\| 
			{^{(0)}\Gamma}_{A \ B}^{\ C}
			-
			\Gamma_{A \ B}^{\ C} 
		\right\|_{L_{\upomega}^{\leb}(S_{t,u})}
		\| |\rgeo^2 \gsphere^{-1}|_{\stgsphere}  \|_{L_{\upomega}^{\infty}(S_{t,u})}
		\| \rgeo \upxi \|_{C_{\upomega}^{0,1 - \frac{2}{\leb}}(S_{t,u})}
			\notag \\
		& \ \
		+
		\sum_{A,B,C=1,2}
		\left\| 
			{^{(0)}\Gamma}_{A \ B}^{\ C}
			-
			\Gamma_{A \ B}^{\ C} 
		\right\|_{L_{\upomega}^{\leb}(S_{t,u})}
		\| |\rgeo^2 {\upepsilon \mkern-9mu /\,}^{\# \#}|_{\stgsphere}  \|_{L_{\upomega}^{\infty}(S_{t,u})}
		\| \rgeo \upxi \|_{C_{\upomega}^{0,1 - \frac{2}{\leb}}(S_{t,u})}
		\notag \\
		& \ \
			+
			\| |\rgeo^2 \gsphere^{-1}|_{\stgsphere}  \|_{L_{\upomega}^{\infty}(S_{t,u})}
			\| \rgeo \mathfrak{F}_{(1)} \|_{C_{\upomega}^{0,1 - \frac{2}{\leb}}(S_{t,u})}
			+
			\| |\rgeo^2 {\upepsilon \mkern-9mu /\,}^{\# \#}|_{\stgsphere}  \|_{L_{\upomega}^{\infty}(S_{t,u})}
			\| \rgeo \mathfrak{F}_{(2)} \|_{C_{\upomega}^{0,1 - \frac{2}{\leb}}(S_{t,u})}
				\notag \\
		& \ \
		+
		\sum_{A,B,C=1,2}
		\left\| 
			{^{(0)}\Gamma}_{A \ B}^{\ C}
			-
			\Gamma_{A \ B}^{\ C}  
		\right\|_{L_{\upomega}^{\leb}(S_{t,u})}
		\| |\rgeo^2 \gsphere^{-1}|_{\stgsphere}  \|_{L_{\upomega}^{\infty}(S_{t,u})}
		\| \rgeo \mathfrak{F}_{(1)} \|_{C_{\upomega}^{0,1 - \frac{2}{\leb}}(S_{t,u})}
			\notag \\
		& \ \
		+
		\sum_{A,B,C=1,2}
		\left\| 
			{^{(0)}\Gamma}_{A \ B}^{\ C}
			-
			\Gamma_{A \ B}^{\ C} 
		\right\|_{L_{\upomega}^{\leb}(S_{t,u})}
		\| |\rgeo^2 {\upepsilon \mkern-9mu /\,}^{\# \#}|_{\stgsphere}  \|_{L_{\upomega}^{\infty}(S_{t,u})}
		\| \rgeo \mathfrak{F}_{(2)} \|_{C_{\upomega}^{0,1 - \frac{2}{\leb}}(S_{t,u})}
		\notag
			\\
		& \ \
			+
			\sum_{i=1,2}
			\| \rgeo^2 \mathfrak{G}_{(i)} \|_{L_{\upomega}^{\leb}(S_{t,u})}.
			\notag
		\end{align}
		Using 
		\eqref{E:ROUNDSPHEREANGULARHOLDERHODGEESTIMATENODERIVATIVESONLHSANGDIVSYMMETRICTRACEFREESTUTENSORFIELD}
		and
		\eqref{E:BOOTSTRAPMETRICAPPROXIMATELYROUND}--\eqref{E:DERIVATIVESBOOTSTRAPMETRICAPPROXIMATELYROUND},
		we deduce that
		\begin{align}	 
		\label{E:SECONDROUNDSPHEREANGULARHOLDERHODGEESTIMATENODERIVATIVESONLHSANGDIVSYMMETRICTRACEFREESTUTENSORFIELD}
			\| \rgeo \upxi \|_{C_{\upomega}^{0,1 - \frac{2}{\leb}}(S_{t,u})}
			& \lesssim
				\uplambda^{- \upepsilon_0}
				\| \rgeo \upxi \|_{C_{\upomega}^{0,1 - \frac{2}{\leb}}(S_{t,u})}
			+
			\sum_{i=1,2}
			\| \rgeo \mathfrak{F}_{(i)} \|_{C_{\upomega}^{0,1 - \frac{2}{\leb}}(S_{t,u})}
				+
			\sum_{i=1,2}
			\| \rgeo^2 \mathfrak{G}_{(i)} \|_{L_{\upomega}^{\leb}(S_{t,u})}.
		\end{align}
		From \eqref{E:SECONDROUNDSPHEREANGULARHOLDERHODGEESTIMATENODERIVATIVESONLHSANGDIVSYMMETRICTRACEFREESTUTENSORFIELD}, 
		we see that if $\uplambda$ is sufficiently large,
		then we can absorb the first term on 
		RHS~\eqref{E:SECONDROUNDSPHEREANGULARHOLDERHODGEESTIMATENODERIVATIVESONLHSANGDIVSYMMETRICTRACEFREESTUTENSORFIELD}
		back into the left, at the expense of doubling the (implicit) constants on the RHS.
		We have therefore proved \eqref{E:ANGULARHOLDERHODGEESTIMATENODERIVATIVESONLHSANGDIVSYMMETRICTRACEFREESTUTENSORFIELD}
		for one-forms $\upxi$ that solve the system
		\eqref{E:ANGDIVSTUONEFORM}--\eqref{E:ANGCURLSTUONEFORM}.

		\end{proof}

\subsection{\texorpdfstring{Proof of Proposition~\ref{P:MAINESTIMATESFOREIKONALFUNCTIONQUANTITIES}}{Proof of Proposition ref P:MAINESTIMATESFOREIKONALFUNCTIONQUANTITIES}}
\label{SS:PROOFOFPROPMAINESTIMATESFOREIKONALFUNCTIONQUANTITIES}
Armed with the previous results of Sect.\,\ref{S:ESTIMATESFOREIKONALFUNCTION},
we are now ready to prove Prop.\,\ref{P:MAINESTIMATESFOREIKONALFUNCTIONQUANTITIES}.
Let us make some preliminary remarks.
We mainly focus on estimating the terms that are new compared to \cite{qW2017},
typically referring the reader to the relevant spots in \cite{qW2017}
for terms that have already been handled. When we refer to \cite{qW2017} for proof details,
we implicitly mean that those details can involve the
results of Prop.\,\ref{P:INITIALFOLIATION},
Lemma~\ref{L:INITIALCONDITIONSTIEDTOEIKONAL},
the inequalities proved in Subsect.\,\ref{SS:ANALYTICTOOLS}, 
and Prop.\,\ref{P:ESTIMATESFORFLUIDVARIABLES},
which subsume results derived in \cite{qW2017}.
The arguments given in \cite{qW2017} often also involve
the bootstrap assumptions of Subsect.\,\ref{SS:ASSUMPTIONS},
which subsume the bootstrap assumptions made in \cite{qW2017}.
We sometimes silently use the results of Prop.\,\ref{P:INITIALFOLIATION}
and
Lemma~\ref{L:INITIALCONDITIONSTIEDTOEIKONAL},
which concern estimates for the initial data of various quantities.
We also stress that the order in which we derive the estimates is important, 
though we do not always make this explicit.
Moreover, throughout the proof, we silently use the simple bound $\rgeo(\uptau,u)/\rgeo(t,u) \lesssim 1$ for $\uptau \leq t$.
Finally, we highlight that the factors of $\gensmoothfunction_{(\vec{\Lunit})}$ appearing on the RHSs
of the equations of Prop.\,\ref{P:PDESMODIFIEDACOUSTICALQUANTITIES} are, 
by virtue of the bootstrap assumptions, bounded in magnitude by $\lesssim 1$.
Therefore, these factors of $\gensmoothfunction_{(\vec{\Lunit})}$ are not important for the overwhelming majority of our estimates, 
and we typically do not even mention them in our discussion below.

\begin{remark}
	In the PDEs that we estimate below,
	all of the terms that are new compared to \cite{qW2017}
	are easy to identify: they all are multiplied by $\uplambda^{-1}$.
\end{remark}

\subsubsection{Proof of \eqref{E:STUVOLUMEFORMCOMPARISONWITHUNITROUNDMETRICVOLUMEFORM}--\eqref{E:NULLLAPSECLOSETOUNITY}}
Based on the transport equations \eqref{E:EVOLUTIONNULLAPSEUSEEULER} and \eqref{E:LUNITVOLUMEFORMRGEOTOMINUSTWORESCALED}
and Lemma~\ref{L:TRANSPORT}, the proof of \cite{qW2017}*{Lemma~5.4} goes through verbatim.

\subsubsection{Proof of \eqref{E:LINFINITYESTIMATESFORRECTANGULARSPATIALCOMPONENTSOFL} and \eqref{E:HOLDERESTIMATESFORF1TYPEFACTOR}}
\label{SSS:INTERIORREGEIONHOLDERESTIMATESFORF1TYPEFACTOR}
Throughout, we will use the simple product-type estimate
\eqref{E:SIMPLEHOLDERPRODUCESTIMATEONSPHERES}.
We will also use the simple estimate
$\| \gensmoothfunction \circ \vec{\varphi} \|_{C_{\upomega}^{0,\updelta_0}(S_{t,u})}
\lesssim
1
+
\| \vec{\varphi} \|_{C_{\upomega}^{0,\updelta_0}(S_{t,u})}
$,
which is valid for scalar functions $\gensmoothfunction$ of array-valued functions $\vec{\varphi}$ on $S_{t,u}$
whenever $\gensmoothfunction$ is smooth on an open set containing
the image set $\vec{\varphi}(S_{t,u})$.

We first prove \eqref{E:HOLDERESTIMATESFORF1TYPEFACTOR}. To proceed,
we note that from the bootstrap assumptions, it easily follows that 
$\| \gensmoothfunction_{(\vec{\Lunit})} \|_{L^{\infty}(\widetilde{\mathcal{M}})} \lesssim 1$,
$\| \vec{\Psi} \|_{L^{\infty}(\widetilde{\mathcal{M}})} \lesssim 1$,
and $\| \vec{\Lunit} \|_{L^{\infty}(\widetilde{\mathcal{M}})} \lesssim 1$.
From these bounds, the estimates mentioned in the previous paragraph,
and \eqref{E:ESTIMATEFORANGLEHOLDERCONTINUITYOFPSIVORTANDGRADENT},
we see that
$\| \gensmoothfunction_{(\vec{\Lunit})} \|_{C_{\upomega}^{0,\updelta_0}(S_{t,u})}
\lesssim
1
+
\| \vec{\Psi} \|_{C_{\upomega}^{0,\updelta_0}(S_{t,u})}
+
\| \vec{\Lunit} \|_{C_{\upomega}^{0,\updelta_0}(S_{t,u})}
\lesssim
1
+
\| \vec{\Lunit} \|_{C_{\upomega}^{0,\updelta_0}(S_{t,u})}
$.
Thus, to prove \eqref{E:HOLDERESTIMATESFORF1TYPEFACTOR},
it suffices to show that for $u \in [-\frac{4}{5}\RescaledTboot,\RescaledTboot]$ and $t \in [[u]_+,\RescaledTboot]$, 
we have $\| \vec{\Lunit} \|_{C_{\upomega}^{0,\updelta_0}(S_{t,u})} \lesssim 1$.
To this end, we first note that Lemma~\ref{L:LUNITIALONGCONETIPAXISISC1INANGLEVARIABLES}
and the bootstrap assumptions imply that for $u \in [-\frac{4}{5}\RescaledTboot,\RescaledTboot]$, 
we have the following estimate:
$
\|
	\vec{\Lunit}
\|_{C_{\upomega}^{0,\updelta_0}(S_{[u]_+,u})}
\lesssim 1
$
(in fact, \eqref{E:RESCALEDGEOMETRICCOORDINATEVECTORFIELDALONGCONETIPAXISSTAYSCLOSETOORIGINVALUES}--\eqref{E:ANGULARDERIVATIVEOFNORMALVECTORATORIGIN} 
imply the stronger bound 
$
\|
	\vec{\Lunit}
\|_{C_{\upomega}^{0,1}(S_{u,u})}
\lesssim 1
$ for $u \in [0,\RescaledTboot]$, whose full strength we do not need here).
From this ``initial data bound,'' 
the first transport equation in \eqref{E:LUNITIANDNORMALITRANSPORT},
the estimates mentioned in the previous paragraph,
the estimate $\| \gensmoothfunction_{(\vec{\Lunit})} \|_{C_{\upomega}^{0,\updelta_0}(S_{t,u})} 
\lesssim
1
+
\| \vec{\Lunit} \|_{C_{\upomega}^{0,\updelta_0}(S_{t,u})}$,
the estimate
$
\int_{[u]_+}^t
	\|
		\pmb{\partial} \vec{\Psi}
	\|_{C_{\upomega}^{0,\updelta_0}(S_{\uptau,u})}
\, d \uptau
\lesssim \uplambda^{-7 \upepsilon_0}
$
(which follows from \eqref{E:RESCALEDBOOTBOUNDS} and \eqref{E:ESTIMATEFORANGLEHOLDERCONTINUITYOFPARTIALPSIPARTIALVORTANDPARTIALGRADENT}),
and inequality \eqref{E:ANGULARHOLDERESTIMATEFORSOLUTIONTOINHOMOGENEOUSTRANSPORT},
we deduce that the following bound holds for $u \in [-\frac{4}{5}\RescaledTboot,\RescaledTboot]$ and $t \in [[u]_+,\RescaledTboot]$,
where we recall that $[u]_+ = \max \lbrace 0,u \rbrace$ is the minimum value of $t$
along $\widetilde{\mathcal{C}}_u$:
\begin{align} \label{E:GRONWALLREADYINTERIORREGEIONHOLDERESTIMATESFORF1TYPEFACTOR}
\|
	\vec{\Lunit}
\|_{C_{\upomega}^{0,\updelta_0}(S_{t,u})}
& \lesssim	
\|
	\vec{\Lunit}
\|_{C_{\upomega}^{0,\updelta_0}(S_{[u]_+,u})}
+
\int_{[u]_+}^t
	\|
		\gensmoothfunction_{(\vec{\Lunit})} \cdot \pmb{\partial} \vec{\Psi}
	\|_{C_{\upomega}^{0,\updelta_0}(S_{\uptau,u})}
\, d \uptau
	\\
&
\lesssim	
1
+
\|
	\vec{\Lunit}
\|_{C_{\upomega}^{0,\updelta_0}(S_{[u]_+,u})}
+
\int_{[u]_+}^t
	\|
		\pmb{\partial} \vec{\Psi}
	\|_{C_{\upomega}^{0,\updelta_0}(S_{\uptau,u})}
\, d \uptau
+
\int_{[u]_+}^t
	\|
		\pmb{\partial} \vec{\Psi}
	\|_{C_{\upomega}^{0,\updelta_0}(S_{\uptau,u})}
	\|
		\vec{\Lunit}
	\|_{C_{\upomega}^{0,\updelta_0}(S_{\uptau,u})}
\, d \uptau
	\notag
		\\
&
\lesssim	
1
+
\int_{[u]_+}^t
	\|
		\pmb{\partial} \vec{\Psi}
	\|_{C_{\upomega}^{0,\updelta_0}(S_{\uptau,u})}
	\|
		\vec{\Lunit}
	\|_{C_{\upomega}^{0,\updelta_0}(S_{\uptau,u})}
\, d \uptau.
	\notag
\end{align}
From \eqref{E:GRONWALLREADYINTERIORREGEIONHOLDERESTIMATESFORF1TYPEFACTOR},
the estimate
$
\int_{[u]_+}^t
	\|
		\pmb{\partial} \vec{\Psi}
	\|_{C_{\upomega}^{0,\updelta_0}(S_{\uptau,u})}
\, d \uptau
\lesssim \uplambda^{-7 \upepsilon_0}
$
noted above,
and Gr\"{o}nwall's inequality, we deduce that
$
\|
	\vec{\Lunit}
\|_{C_{\upomega}^{0,\updelta_0}(S_{t,u})}
\lesssim 1
$,
thereby completing the proof of \eqref{E:HOLDERESTIMATESFORF1TYPEFACTOR}.

We will now prove \eqref{E:LINFINITYESTIMATESFORRECTANGULARSPATIALCOMPONENTSOFL}.
To proceed, we again use the first transport equation in \eqref{E:LUNITIANDNORMALITRANSPORT},
the fundamental theorem of calculus,
and the bootstrap assumptions and argue as above to deduce
$
|
	\Lunit^i(t,u,\upomega)
		-
	\Lunit^i([u]_+,u,\upomega)
|
\lesssim
\int_{[u]_+}^t
	\|
		\pmb{\partial} \vec{\Psi}
	\|_{L^{\infty}(\widetilde{\Sigma}_{\uptau})}
\, d \uptau
\lesssim \uplambda^{-7 \upepsilon_0}
$.
From this estimate and the data bounds
\eqref{E:LUNITIALONGSIGMA0STAYSCLOSETOORIGINVALUES}
and
\eqref{E:LUNITIALONGCONETIPAXISSTAYSCLOSETOORIGINVALUES},
we conclude the desired estimate \eqref{E:LINFINITYESTIMATESFORRECTANGULARSPATIALCOMPONENTSOFL}.

\subsubsection{Proof of \eqref{E:ACOUSTICALLT2LOMEGAPANDLDERIVATIVESALONGCONES}--\eqref{E:ACOUSTICALLTINFTYLOMEGAPALONGCONES}
for $\hat{\upchi}$, $\angprojDarg{\Lunit} \hat{\upchi}$, $\upzeta$, and $\angprojDarg{\Lunit} \upzeta$}
We first prove \eqref{E:ACOUSTICALLT2LOMEGAPANDLDERIVATIVESALONGCONES} for
$\| \hat{\upchi} \|_{L_t^2 L_{\upomega}^p(\widetilde{\mathcal{C}}_u)}$.
From the transport equation \eqref{E:LDERIVATIVECHIHATAFTERUSINGEULER},
\eqref{E:MAIDENTITYTRANSPORTLEMMA},
and \eqref{E:STUVOLUMEFORMCOMPARISONWITHUNITROUNDMETRICVOLUMEFORM},
we deduce
\begin{align} \label{E:CHIHATTRANSPORTINEQUALITY}
	|
		\rgeo^2 \hat{\upchi}
	|_{\gsphere}(t,u,\upomega)
	& 
	\lesssim
	\lim_{\uptau \downarrow [u]_+}
		|
			\rgeo^2 \hat{\upchi}
		|_{\gsphere}(\uptau,u,\upomega)
	+
	\uplambda^{-1} 
	\int_{[u]_+}^t
			|
				\rgeo^2 (\vec{\VortVort},\DivGradEnt)
			|(\uptau,u,\upomega)
	\, d \uptau
		\\
	& \ \
			+
		\int_{[u]_+}^t
			|
			\rgeo^2
			(\angD,\angprojDarg{\Lunit}) \upxi
			|_{\gsphere}(\uptau,u,\upomega)
	\, d \uptau
	+
	\int_{[u]_+}^t
		\left|
				\rgeo^2 
				(\pmb{\partial} \vec{\Psi},\mytr_{\congsphere} \widetilde{\upchi}^{(Small)},\hat{\upchi},\rgeo^{-1})		
				\cdot
				\pmb{\partial} \vec{\Psi}
	\right|_{\gsphere}(\uptau,u,\upomega)
	\, d \uptau,
\notag
\end{align}
where the correction mentioned in Footnote~\ref{FN:CORRECTIONOFTYPOS}
leads to $m=1$ in \eqref{E:MAIDENTITYTRANSPORTLEMMA},
thus correcting the value $m=\frac{1}{2}$ appearing \cite{qW2017}*{Equation~(5.69)}.
We now divide \eqref{E:CHIHATTRANSPORTINEQUALITY} by $\rgeo^2(t,u)$ and take the norm
$\| \cdot \|_{L_t^2 L_{\upomega}^p(\widetilde{\mathcal{C}}_u)}$.
The arguments given just below \cite{qW2017}*{Equation~(5.68)}
yield that the norms of all terms on RHS~\eqref{E:CHIHATTRANSPORTINEQUALITY} are $\lesssim \uplambda^{-1/2}$ 
(the correction of the value of $m$ mentioned above does not substantially affect the arguments given there),
except the term multiplied by $\uplambda^{-1}$ was not present in \cite{qW2017}.
To handle the remaining term, we use \eqref{E:ANOTHERLAMBDAINVERSELINEARTERMLT2LUPOMEAGINFTY}.
We clarify that to handle the case in which $u \leq 0$,
this argument relies on the initial data bound
$\| w^{1/2} \hat{\upchi} \|_{L_w^{\infty}L_{\upomega}^p(\Sigma_0^{\RescaledFoliationparameter})}
\lesssim 
\uplambda^{-1/2}
$,
which follows from 
\textbf{i)}
using the
first equation in \eqref{E:CONNECTIONCOEFFICIENT}
to express $\hat{\upchi}$
in terms of $\hat{\spheresecondfund}$
and $\hat{k}$;
\textbf{ii)} bounding $\hat{\spheresecondfund}$
in the norm 
$\| w^{1/2} \cdot \|_{L_w^{\infty}L_{\upomega}^p(\Sigma_0^{\RescaledFoliationparameter})}$
by using the estimate 
\eqref{E:INTIALDERIVATIVESOFLAPSEANDTRACEFREECHIEST};
and \textbf{iii)} bounding  $\hat{k}$ in the norm
$\| w^{1/2} \cdot \|_{L_w^{\infty}L_{\upomega}^p(\Sigma_0^{\RescaledFoliationparameter})}$
by using the schematic identity $\hat{k}_{AB} = \gensmoothfunction_{(\vec{\Lunit})} \cdot \pmb{\partial} \vec{\Psi}$,
the estimate \eqref{E:HOLDERESTIMATESFORF1TYPEFACTOR},
and the estimate \eqref{E:FLUIDONEDERIVATIVESIGMAT} for $\rgeo^{1/2} \pmb{\partial} \vec{\Psi}$;
in total, this allows one to deduce (recalling that $w = - u|_{\Sigma_0} \geq 0$ 
and that $\rgeo(\uptau,u) = \uptau-u$) the estimate
$
\| \frac{u^2}{\rgeo^2} \hat{\upchi}(0,u,\upomega) \|_{L_t^2 L_{\upomega}^p(\widetilde{\mathcal{C}}_u)}
\lesssim
\| |u|^{1/2} \hat{\upchi}(0,u,\upomega) \|_{L_u^{\infty} L_{\upomega}^p}
=
\| w^{1/2} \hat{\upchi} \|_{L_w^{\infty}L_{\upomega}^p(\Sigma_0^{\RescaledFoliationparameter})}
\lesssim 
\uplambda^{-1/2}
$,
which is needed to control
the term generated by the first term on RHS~\eqref{E:CHIHATTRANSPORTINEQUALITY} when $u \leq 0$.

To prove the estimate for \eqref{E:ACOUSTICALLTINFTYLOMEGAPALONGCONES} for
$\| \rgeo^{1/2} \hat{\upchi} \|_{L_t^{\infty} L_{\upomega}^p(\widetilde{\mathcal{C}}_u)}$,
we note that all terms on RHS~\eqref{E:CHIHATTRANSPORTINEQUALITY}
can, after being divided by $\rgeo^{3/2}$, be handled using similar arguments 
(see just below \cite{qW2017}*{Equation~(5.73)}, where we again note that 
the correction of the powers of $\rgeo$ mentioned above does not substantially affect the arguments),
but the term multiplied by $\uplambda^{-1}$ was not present in \cite{qW2017}.
To handle this remaining term, we use \eqref{E:SECONDLAMBDAINVERSELINEARTERMTIMEINTEGRALLUINFTIYLTINFTYLOMEGAP}. 

We now prove \eqref{E:ACOUSTICALLT2LOMEGAPANDLDERIVATIVESALONGCONES} for
$\| \rgeo \angprojDarg{\Lunit} \hat{\upchi} \|_{L_t^2 L_{\upomega}^p(\widetilde{\mathcal{C}}_u)}$.
We use the transport equation \eqref{E:LDERIVATIVECHIHATAFTERUSINGEULER} to solve for $\angprojDarg{\Lunit} \hat{\upchi}$,
multiply the resulting identity by $\rgeo$, and then take the norm $\| \cdot \|_{L_t^2 L_{\upomega}^p(\widetilde{\mathcal{C}}_u)}$.
Thanks to the already proven bound \eqref{E:ACOUSTICALLT2LOMEGAPANDLDERIVATIVESALONGCONES} for
$\| \hat{\upchi} \|_{L_t^2 L_{\upomega}^p(\widetilde{\mathcal{C}}_u)}$,
the same arguments given in the paragraph below \cite{qW2017}*{Equation~(5.73)}
imply that all terms satisfy the desired estimate
(where the correction mentioned in Footnote~\ref{FN:CORRECTIONOFTYPOS} is not important for this argument), 
except the following term was not present there:
$
\uplambda^{-1} \rgeo \gensmoothfunction_{(\vec{\Lunit})} \cdot (\vec{\VortVort},\DivGradEnt)
$.
To handle this remaining term, we use \eqref{E:RGEOLAMBDAINVERSELINEARTERMLT2LUINFTYLOMEGAP}.

The estimates \eqref{E:ACOUSTICALLT2LOMEGAPANDLDERIVATIVESALONGCONES} and 
\eqref{E:ACOUSTICALLTINFTYLOMEGAPALONGCONES} for
$\upzeta$ and $\angprojDarg{\Lunit} \upzeta$
follow from a similar argument since, by \eqref{E:LDERIVATIVETORSIONAFTERUSINGEULER},
$\upzeta$ satisfies a transport equation that is schematically similar to the one that $\hat{\upchi}$ satisfies,
except it features the additional source term
$\upzeta \cdot \hat{\upchi}$,
which can be handled with the bootstrap assumptions \eqref{E:BOOTSTRAPCHIANDTORSIONALONGANYCONE};
we omit the details.
We clarify that, in view of the second term on LHS~\eqref{E:LDERIVATIVETORSIONAFTERUSINGEULER},
the correct power of $\rgeo$ in the analog of inequality \eqref{E:CHIHATTRANSPORTINEQUALITY} for $\upzeta$
is $\rgeo$. Thus, to handle the $\uplambda^{-1}$-multiplied terms,
we use the estimates
\eqref{E:LAMBDAINVERSELINEARTERMTIMEINTEGRALLT2LUINFTIYLOMEGAINFTY}
and
\eqref{E:LAMBDAINVERSELINEARTERMTIMEINTEGRALLUINFTIYLTINFTYLOMEGAP}
in place of the estimates
\eqref{E:ANOTHERLAMBDAINVERSELINEARTERMLT2LUPOMEAGINFTY}
and
\eqref{E:SECONDLAMBDAINVERSELINEARTERMTIMEINTEGRALLUINFTIYLTINFTYLOMEGAP} we used to handle $\hat{\upchi}$.

\subsubsection{Proof of \eqref{E:ACOUSTICALLOMEGA2PLTINFTYALONGCONES} for $\hat{\upchi}$ and $\upzeta$}
These estimates follow from \eqref{E:LOMEGA2PLTINFTYTRACEINEQUALITYNEEDEDFORCHIHAT} with $(\hat{\upchi},\upzeta)$ in the role of $\upxi$,
the already proven estimates 
\eqref{E:ACOUSTICALLT2LOMEGAPANDLDERIVATIVESALONGCONES} for $(\hat{\upchi},\upzeta)$ and $(\angprojDarg{\Lunit} \hat{\upchi},\angprojDarg{\Lunit} \upzeta)$,
and the bootstrap assumptions \eqref{E:LT2LINFINITYBOOTSTRAPCHIANDZETAININTERIORREGION} for $(\hat{\upchi},\upzeta)$.

\subsubsection{Proof of 
\eqref{E:ACOUSTICALLT2LOMEGAPANDLDERIVATIVESALONGCONES},
\eqref{E:ACOUSTICALLTINFTYLOMEGAPALONGCONES},
\eqref{E:ACOUSTICALLOMEGA2PLTINFTYALONGCONES},
\eqref{E:TRCHILINFINITYESTIMATES}, 
\eqref{E:TRCHIMODLINFINITYESTIMATES},
and
\eqref{E:TRICHIMODSMALLTRFREECHIANDTORSIONL2INTIMEALONGSOUNDCONES}
for 
$\mytr_{\congsphere} \widetilde{\upchi}$,
$\mytr_{\congsphere} \widetilde{\upchi}^{(Small)}$,
and $\angprojDarg{\Lunit} \mytr_{\congsphere} \widetilde{\upchi}^{(Small)}$}
To prove \eqref{E:TRCHILINFINITYESTIMATES},
we note that the definition \eqref{E:MODTRICHISMALL}
of $\mytr_{\congsphere} \widetilde{\upchi}^{(Small)}$ implies that
it suffices to prove the pointwise bound
$|\rgeo \mytr_{\congsphere} \widetilde{\upchi}^{(Small)}| \lesssim \lambda^{-4 \upepsilon_0}$.
To this end, we first use the transport equation \eqref{E:MODIFIEDRAYCHAUDHURI}
and \eqref{E:MAIDRGEOINEQUALITYTRANSPORT} with $\mathfrak{G}:=0$
to deduce
\begin{align} \label{E:TRCHISMALLTRANSPORTINEQUALITY}
	|
		\rgeo^2 \mytr_{\congsphere} \widetilde{\upchi}^{(Small)}
	|(t,u,\upomega)
	& 
	\lesssim
	\lim_{\uptau \downarrow [u]_+}
		|
			\rgeo^2 \mytr_{\congsphere} \widetilde{\upchi}^{(Small)}
		|(\uptau,u,\upomega)
	+
	\uplambda^{-1} 
	\int_{[u]_+}^t
			|
				\rgeo^2 (\vec{\VortVort},\DivGradEnt)
			|(\uptau,u,\upomega)
	\, d \uptau
		\\
	& \ \
			+
		\int_{[u]_+}^t
			\left|
			\rgeo^2
			(\pmb{\partial} \vec{\Psi},\mytr_{\congsphere} \widetilde{\upchi}^{(Small)},\rgeo^{-1})
			\cdot
			\pmb{\partial} \vec{\Psi}
			\right|_{\gsphere}(\uptau,u,\upomega)
	\, d \uptau
	+
	\int_{[u]_+}^t
			|
				\rgeo^2 \hat{\upchi} \cdot \hat{\upchi}
			|_{\gsphere}(\uptau,u,\upomega)
	\, d \uptau
		\notag \\
& \ \
	+
	\int_{[u]_+}^t
			|
				\rgeo^2 
				\mytr_{\congsphere} \widetilde{\upchi}^{(Small)} 
				\cdot
				\mytr_{\congsphere} \widetilde{\upchi}^{(Small)}
			|(\uptau,u,\upomega)
	\, d \uptau.
	\notag
\end{align}
We now divide \eqref{E:TRCHISMALLTRANSPORTINEQUALITY} by $\rgeo(t,u)$.
To handle the term on RHS~\eqref{E:TRCHISMALLTRANSPORTINEQUALITY}
that is multiplied by $\uplambda^{-1}$, we use \eqref{E:LAMBDAINVERSELINEARTERMLINFTYWHOLESPACETIME}.
The remaining terms were suitably bounded in the arguments given just below \cite{qW2017}*{Equation~(5.78)}.
We have thus proved \eqref{E:TRCHILINFINITYESTIMATES}.
The estimate \eqref{E:TRCHIMODLINFINITYESTIMATES} follows from nearly identical arguments,
where one uses \eqref{E:ASECONDLAMBDAINVERSELINEARTERMLINFTYWHOLESPACETIME} to handle the
$\uplambda^{-1}$-multiplied term; we omit the details.

The estimates \eqref{E:ACOUSTICALLTINFTYLOMEGAPALONGCONES} and \eqref{E:ACOUSTICALLOMEGA2PLTINFTYALONGCONES}
for $\mytr_{\congsphere} \widetilde{\upchi}^{(Small)}$ then follow as straightforward consequences of \eqref{E:TRCHIMODLINFINITYESTIMATES}.
 
We now prove the estimate \eqref{E:TRICHIMODSMALLTRFREECHIANDTORSIONL2INTIMEALONGSOUNDCONES} for $\mytr_{\congsphere} \widetilde{\upchi}^{(Small)}$.
First, using the transport equation \eqref{E:MODIFIEDRAYCHAUDHURI},
we deduce that
\begin{align} \label{E:REWRITEENRSQUAREDTRANSPORTTRICHIMOD}
\Lunit (\rgeo^2 \mytr_{\congsphere} \widetilde{\upchi}^{(Small)})
= 
\mathfrak{F}
& :=
			\uplambda^{-1} \rgeo^2 \gensmoothfunction_{(\vec{\Lunit})} \cdot (\vec{\VortVort},\DivGradEnt)
			+
			\rgeo^2
			\gensmoothfunction_{(\vec{\Lunit})}
			\cdot
			(\pmb{\partial} \vec{\Psi},\mytr_{\congsphere} \widetilde{\upchi}^{(Small)},\rgeo^{-1})
			\cdot
			\pmb{\partial} \vec{\Psi}
				\\
	& \ \
			+
			\rgeo^2
			\gensmoothfunction_{(\vec{\Lunit})}
			\hat{\upchi}
			\cdot
			\hat{\upchi}
			+
			\rgeo^2
			\mytr_{\congsphere} \widetilde{\upchi}^{(Small)} 
			\cdot
			\mytr_{\congsphere} \widetilde{\upchi}^{(Small)}.
			\notag
\end{align}
From \eqref{E:REWRITEENRSQUAREDTRANSPORTTRICHIMOD}
and the vanishing initial condition (along the cone-tip axis) for
$\rgeo^2 \mytr_{\congsphere} \widetilde{\upchi}^{(Small)}$ guaranteed by \eqref{E:CONNECTIONCOEFFICIENTS0LIMITSALONGTIP},
we find, with $[u]_- := |\min \lbrace u,0 \rbrace|$ and $[u]_+ := \max \lbrace u,0 \rbrace$, that
\begin{align} \label{E:INTEGRALEQUATIONFORTRCHISMALL}
\mytr_{\congsphere} \widetilde{\upchi}^{(Small)}(t,u,\upomega)
= 
\frac{[u]_-^2}{(t + [u]_-)^2}
\mytr_{\congsphere} \widetilde{\upchi}^{(Small)}([u]_+,u,\upomega)
+
\frac{1}{\rgeo^2(t,u)}
\int_{[u]_+}^t
	\mathfrak{F}(\uptau,u,\upomega)
\, d \uptau.
\end{align}
Using \eqref{E:INTEGRALEQUATIONFORTRCHISMALL},
applying the product-type estimate 
\eqref{E:SIMPLEHOLDERPRODUCESTIMATEONSPHERES}
to the term $\mathfrak{F}$ in \eqref{E:REWRITEENRSQUAREDTRANSPORTTRICHIMOD},
using the already proven estimate \eqref{E:HOLDERESTIMATESFORF1TYPEFACTOR} for $\gensmoothfunction_{(\vec{\Lunit})}$,
and using \eqref{E:ZERODATAANGULARHOLDERESTIMATEFORSOLUTIONTOINHOMOGENEOUSTRANSPORT},
we find,
in view of the definition \eqref{E:MAXFUNCTIONDEF} of the Hardy--Littlewood maximal function,  
that for $u \in [-\frac{4}{5}\RescaledTboot,\RescaledTboot]$ and $t \in [[u]_+,\RescaledTboot]$, we have
\begin{align} \label{E:FIRSTSTEPINCONEONLYTRCHISMALLCOMEGADELTAESTIMATE}
	\|
		\mytr_{\congsphere} \widetilde{\upchi}^{(Small)}
	\|_{C_{\upomega}^{0,\updelta_0}(S_{t,u})}
	& \lesssim
	\frac{[u]_-^{3/2}}{(t + [u]_-)^2}
	\left\| 
		|u|^{1/2} \mytr_{\congsphere} \widetilde{\upchi}^{(Small)} 
	\right \|_{L_u^{\infty} C_{\upomega}^{0,\updelta_0}(\widetilde{\Sigma}_0)}
		\\
	&  \ \
	+
	\uplambda^{-1}
	\| (\vec{\VortVort},\DivGradEnt) \|_{L_t^1 C_{\upomega}^{0,\updelta_0}(\widetilde{\mathcal{C}}_u)}
	+
	\| (\pmb{\partial} \vec{\Psi},\mytr_{\congsphere} \widetilde{\upchi}^{(Small)},\hat{\upchi}) \|_{L_t^2 C_{\upomega}^{0,\updelta_0}(\widetilde{\mathcal{C}}_u)}^2
	\notag	\\
& \ \
	+
	\mathcal{M}\left(
		\| \pmb{\partial} \vec{\Psi} \|_{L_u^{\infty} C_{\upomega}^{0,\updelta_0}(\widetilde{\Sigma}_t)}
	\right).
	\notag
\end{align}
From \eqref{E:FIRSTSTEPINCONEONLYTRCHISMALLCOMEGADELTAESTIMATE},
the last estimate in \eqref{E:MODTRCHIESTIAMTESALONGINITIALFOLIATION},
the parameter relation \eqref{E:BOUNDSONLEBESGUEEXPONENTP},
\eqref{E:RESCALEDBOOTBOUNDS},
\eqref{E:BOOTSTRAPCHIANDTORSIONALONGANYCONE},
and
\eqref{E:ESTIMATEFORANGLEHOLDERCONTINUITYOFPARTIALPSIPARTIALVORTANDPARTIALGRADENT},
we find that 
$
\|
	\mytr_{\congsphere} \widetilde{\upchi}^{(Small)}
\|_{C_{\upomega}^{0,\updelta_0}(S_{t,u})}
\lesssim
\frac{[u]_-^{3/2}}{(t + [u]_-)^2}
\uplambda^{-1/2}
+
\uplambda^{-1 + 4 \upepsilon_0}
+
	\mathcal{M}\left(
		\| \pmb{\partial} \vec{\Psi} \|_{L_u^{\infty} C_{\upomega}^{0,\updelta_0}(\widetilde{\Sigma}_t)}
	\right)
$.
Taking the norm $\| \cdot \|_{L_t^2([[u]_+,\RescaledTboot])}$ of this inequality
and using \eqref{E:RESCALEDBOOTBOUNDS},
\eqref{E:STANDARDMAXIMALFUNCTIONLQESTIMATE} with $\leb := 2$,
and
\eqref{E:ESTIMATEFORANGLEHOLDERCONTINUITYOFPARTIALPSIPARTIALVORTANDPARTIALGRADENT},
we conclude the desired bound
\eqref{E:TRICHIMODSMALLTRFREECHIANDTORSIONL2INTIMEALONGSOUNDCONES} for $\mytr_{\congsphere} \widetilde{\upchi}^{(Small)}$.

The estimate \eqref{E:ACOUSTICALLT2LOMEGAPANDLDERIVATIVESALONGCONES} for
$\| \mytr_{\congsphere} \widetilde{\upchi}^{(Small)} \|_{L_t^2 L_{\upomega}^p(\widetilde{\mathcal{C}}_u)}$
then follows as a straightforward consequence of the estimate 
\eqref{E:TRICHIMODSMALLTRFREECHIANDTORSIONL2INTIMEALONGSOUNDCONES}
for $\mytr_{\congsphere} \widetilde{\upchi}^{(Small)}$.

We now prove the estimate \eqref{E:ACOUSTICALLT2LOMEGAPANDLDERIVATIVESALONGCONES} for
$\| \rgeo \angprojDarg{\Lunit} \mytr_{\congsphere} \widetilde{\upchi}^{(Small)} \|_{L_t^2 L_{\upomega}^p(\widetilde{\mathcal{C}}_u)}$
by using the transport equation \eqref{E:MODIFIEDRAYCHAUDHURI} to algebraically solve for $\rgeo \angprojDarg{\Lunit} \mytr_{\congsphere} \widetilde{\upchi}^{(Small)}$.
Thanks to the bound \eqref{E:RGEOANDUBOUNDS} for $\rgeo$,
the bootstrap assumptions,
and the already proven bounds
\eqref{E:ACOUSTICALLT2LOMEGAPANDLDERIVATIVESALONGCONES}
and
\eqref{E:ACOUSTICALLTINFTYLOMEGAPALONGCONES}
for
$\hat{\upchi}$ and $\mytr_{\congsphere} \widetilde{\upchi}^{(Small)}$,
the same arguments given just below \cite{qW2017}*{Equation~(5.80)}
imply that all terms on RHS~\eqref{E:MODIFIEDRAYCHAUDHURI}
and the term
$
\frac{2}{\rgeo}
		\mytr_{\congsphere} \widetilde{\upchi}^{(Small)}
$
on LHS~\eqref{E:MODIFIEDRAYCHAUDHURI}
satisfy (upon being multiplied by $\rgeo$) 
the desired estimate, except the following term was not present in \cite{qW2017}:
$
\uplambda^{-1} \rgeo \gensmoothfunction_{(\vec{\Lunit})} \cdot (\vec{\VortVort},\DivGradEnt)
$.
To bound this remaining term, we use \eqref{E:RGEOLAMBDAINVERSELINEARTERMLT2LUINFTYLOMEGAP}.

\subsubsection{Proof of \eqref{E:ASECONDACOUSTICALLTINFTYLOMEGAPALONGCONES}}
\eqref{E:ASECONDACOUSTICALLTINFTYLOMEGAPALONGCONES} follows from the already proven estimate
\eqref{E:ACOUSTICALLTINFTYLOMEGAPALONGCONES} and the bound \eqref{E:RGEOANDUBOUNDS} for $\rgeo$.

\subsubsection{Proof of \eqref{E:ONEANGULARDERIVATIVEOFTRCHIMODLINFINITYESTIMATES} 
and \eqref{E:ONEANGULARDERIVATIVEOFTRCHIMODANDTRFREECHIL2INTIMEESTIMATES}}
To prove \eqref{E:ONEANGULARDERIVATIVEOFTRCHIMODANDTRFREECHIL2INTIMEESTIMATES},
we first note the following bound for some factors in the next-to-last product on RHS~\eqref{E:ANGDCOMMUTEDMODIFIEDRAYCHAUDHURI},
which follows from  
\eqref{E:RESCALEDBOOTBOUNDS},
\eqref{E:RESCALEDSTRICHARTZ},
and \eqref{E:BOOTSTRAPCHIANDTORSIONALONGANYCONE}:
$
\|\gensmoothfunction_{(\vec{\Lunit})}\cdot(\pmb{\partial} \vec{\Psi},\mytr_{\congsphere} \widetilde{\upchi}^{(Small)},\hat{\upchi}) \|_{L_{\upomega}^{\infty} L_t^1(\widetilde{\mathcal{C}}_u)}
\lesssim 
\uplambda^{-2 \upepsilon_0}
\leq 1
$.
From this bound,
the transport equation \eqref{E:ANGDCOMMUTEDMODIFIEDRAYCHAUDHURI},
and \eqref{E:MAIDRGEOINEQUALITYTRANSPORT}
with 
$
\mathfrak{G} := \gensmoothfunction_{(\vec{\Lunit})} \cdot (\pmb{\partial} \vec{\Psi},\mytr_{\congsphere} \widetilde{\upchi}^{(Small)},\hat{\upchi})
$,
we deduce
\begin{align} \label{E:ANGDMODTRCHISMALLTRANSPORTINEQUALITY}
	|
		\rgeo^3 \angD \mytr_{\congsphere} \widetilde{\upchi}^{(Small)}
	|_{\gsphere}(t,u,\upomega)
	& 
	\lesssim
	\lim_{\uptau \downarrow [u]_+}
		|
			\rgeo^3 \angD \mytr_{\congsphere} \widetilde{\upchi}^{(Small)}
		|_{\gsphere}(\uptau,u,\upomega)
	+
	\uplambda^{-1} 
	\int_{[u]_+}^t
			|
				\rgeo^3 \angD (\vec{\VortVort},\DivGradEnt)
			|_{\gsphere}(\uptau,u,\upomega)
	\, d \uptau
		\\
	& \ \
			+
		\uplambda^{-1}
		\int_{[u]_+}^t
				\left|
				\rgeo^3
				(\vec{\GradEnt} \cdot \pmb{\partial} \vec{\Psi},\pmb{\partial} \vec{\Psi},\pmb{\partial} \vec{\vortrenormalized},\pmb{\partial} \vec{\GradEnt}) 
				\cdot 
				(\pmb{\partial} \vec{\Psi},\mytr_{\congsphere} \widetilde{\upchi}^{(Small)},\hat{\upchi},\rgeo^{-1})	
			\right|_{\gsphere}(\uptau,u,\upomega)
	\, d \uptau
		\notag \\
	& \ \
	+
	\int_{[u]_+}^t
			\left|
				\rgeo^3
				\angD \pmb{\partial} \vec{\Psi}
				\cdot
				(\pmb{\partial} \vec{\Psi},\mytr_{\congsphere} \widetilde{\upchi}^{(Small)}, \rgeo^{-1})
			\right|_{\gsphere}(\uptau,u,\upomega)
	\, d \uptau
\notag
\\
&
\ \ 
	+
	\int_{[u]_+}^t
			|
				\rgeo^3
				\angD \hat{\upchi} \cdot \hat{\upchi}
			|_{\gsphere}(\uptau,u,\upomega)
	\, d \uptau
\notag
	\\
	& \ \
		+
		\int_{[u]_+}^t
			\left|
				\rgeo^3
				(\pmb{\partial} \vec{\Psi},\mytr_{\congsphere} \widetilde{\upchi}^{(Small)},\hat{\upchi},\rgeo^{-1})		
				\cdot
				(\pmb{\partial} \vec{\Psi},\mytr_{\congsphere} \widetilde{\upchi}^{(Small)},\rgeo^{-1})	
				\cdot
				\pmb{\partial} \vec{\Psi}
			\right|_{\gsphere}(\uptau,u,\upomega)
	\, d \uptau.
	\notag
\end{align}
In the arguments given in the paragraph below \cite{qW2017}*{Equation~(5.81)},
with the help of the bootstrap assumptions,
all terms on RHS~\eqref{E:ANGDMODTRCHISMALLTRANSPORTINEQUALITY} 
were shown, after dividing by $\rgeo^2(t,u)$,
to be bounded in the norm $\| \cdot \|_{L_t^2 L_{\upomega}^p (\widetilde{\mathcal{C}}_u)}$
by $\lesssim \uplambda^{-1/2} + \uplambda^{-2 \upepsilon_0} 
\| 
		\rgeo \angD \hat{\upchi} 
	\|_{L_t^2 L_{\upomega}^p (\widetilde{\mathcal{C}}_u)}$, except that the
two terms multiplied by $\uplambda^{-1}$ were not present there.
To handle these remaining terms,
we use \eqref{E:ASECONDLAMBDAINVERSESECONDANGULARDERIVATIVELINEARTERMTIMEINTEGRALLTINFTTYLUINFTYLOMEGAP}
and \eqref{E:ASECONDLAMBDAINVERSEQUADRATICTERMTIMEINTEGRALLTINFTTYLUINFTYLOMEGAP},
which in total yields
\begin{align} \label{E:TRANSPORTSTEPONEANGULARDERIVATIVEOFTRCHIMODANDTRFREECHIL2INTIMEESTIMATES}
	\| 
		\rgeo \angD \mytr_{\congsphere} \widetilde{\upchi}^{(Small)} 
	\|_{L_t^2 L_{\upomega}^p (\widetilde{\mathcal{C}}_u)}
	&
	\lesssim \uplambda^{-1/2} 
		+ 
	\uplambda^{-2 \upepsilon_0} 
	\| 
		\rgeo \angD \hat{\upchi} 
	\|_{L_t^2 L_{\upomega}^p (\widetilde{\mathcal{C}}_u)}.
\end{align}
Next, we note that the divergence equation \eqref{E:DIVDIVTRFREECHISCHEMATIC},
the Hodge estimate \eqref{E:SYMMETRICTRACEFREESTUCALDERONZYGMUNDHODGEESTIMATES} with $\leb :=p$,
and the same arguments given in the paragraph below \cite{qW2017}*{Equation~(5.82)}
yield
\begin{align} \label{E:HODGEESTIMATESFORPROOFOFONEANGULARDERIVATIVEOFTRCHIMODANDTRFREECHIL2INTIMEESTIMATES}
	\| 
		\rgeo \angD \hat{\upchi} 
	\|_{L_t^2 L_{\upomega}^p (\widetilde{\mathcal{C}}_u)}
	& \lesssim
		\| 
			\rgeo \angD \mytr_{\congsphere} \widetilde{\upchi}^{(Small)} 
		\|_{L_t^2 L_{\upomega}^p (\widetilde{\mathcal{C}}_u)}
		+
		\uplambda^{-1/2}.
\end{align}
From \eqref{E:TRANSPORTSTEPONEANGULARDERIVATIVEOFTRCHIMODANDTRFREECHIL2INTIMEESTIMATES} 
and \eqref{E:HODGEESTIMATESFORPROOFOFONEANGULARDERIVATIVEOFTRCHIMODANDTRFREECHIL2INTIMEESTIMATES}, 
we conclude 
(when $\uplambda$ is sufficiently large)
the desired bounds in \eqref{E:ONEANGULARDERIVATIVEOFTRCHIMODANDTRFREECHIL2INTIMEESTIMATES}.

As is noted just below \cite{qW2017}*{Equation~(5.84)}, 
the estimate \eqref{E:ONEANGULARDERIVATIVEOFTRCHIMODLINFINITYESTIMATES} can be proved using a similar argument,
based on dividing \eqref{E:ANGDMODTRCHISMALLTRANSPORTINEQUALITY} by $\rgeo^{3/2}(t,u)$,
where we use
\eqref{E:LAMBDAINVERSESECONDANGULARDERIVATIVELINEARTERMTIMEINTEGRALLTINFTTYLUINFTYLOMEGAP}
and \eqref{E:LAMBDAINVERSEQUADRATICTERMTIMEINTEGRALLTINFTTYLUINFTYLOMEGAP}
to handle the two $\uplambda^{-1}$-multiplied terms on RHS~\eqref{E:ANGDMODTRCHISMALLTRANSPORTINEQUALITY}.

\subsubsection{Proof of \eqref{E:IMPROVEDININTERIORL2INTIMELINFINITYINSPACECONNECTIONCOFFICIENTS} for $\mytr_{\congsphere} \widetilde{\upchi}^{(Small)}$
and $\mytr_{\gsphere} \upchi - \frac{2}{\rgeo}$}
\label{SSS:ANNOYINGTRCHITRANSPORTHOLDERESTIMATE}
We first prove \eqref{E:IMPROVEDININTERIORL2INTIMELINFINITYINSPACECONNECTIONCOFFICIENTS} for $\mytr_{\congsphere} \widetilde{\upchi}^{(Small)}$.
A slight modification of the proof of \eqref{E:FIRSTSTEPINCONEONLYTRCHISMALLCOMEGADELTAESTIMATE} yields the following bound:
\begin{align} \label{E:SECONDSTEPINTRCHISMALLCOMEGADELTAESTIMATE}
	\|
		\mytr_{\congsphere} \widetilde{\upchi}^{(Small)}
	\|_{L_u^{\infty} C_{\upomega}^{0,\updelta_0}(\widetilde{\Sigma}_t^{(Int)})}
	& \lesssim
	\uplambda^{-1}
	\|  (\vec{\VortVort},\DivGradEnt) \|_{L_t^1 L_u^{\infty} C_{\upomega}^{0,\updelta_0}(\widetilde{\mathcal{M}}^{(Int)})}
	+
	\mathcal{M}\left(
		\| \pmb{\partial} \vec{\Psi} \|_{L_u^{\infty} C_{\upomega}^{0,\updelta_0}(\widetilde{\Sigma}_t^{(Int)})}
	\right)
		\\
& \ \
	+
	\| (\pmb{\partial} \vec{\Psi},\mytr_{\congsphere} \widetilde{\upchi}^{(Small)},\hat{\upchi}) \|_{L_t^2 L_u^{\infty} C_{\upomega}^{0,\updelta_0}(\widetilde{\mathcal{M}}^{(Int)})}^2.
	\notag
\end{align}
From \eqref{E:SECONDSTEPINTRCHISMALLCOMEGADELTAESTIMATE},
\eqref{E:RESCALEDBOOTBOUNDS},
\eqref{E:LT2LINFINITYBOOTSTRAPCHIANDZETAININTERIORREGION},
and
\eqref{E:ESTIMATEFORANGLEHOLDERCONTINUITYOFPARTIALPSIPARTIALVORTANDPARTIALGRADENT},
we deduce 
\begin{align} \label{E:THIRDSTEPINTRCHISMALLCOMEGADELTAESTIMATE}
\|
		\mytr_{\congsphere} \widetilde{\upchi}^{(Small)}
\|_{L_u^{\infty} C_{\upomega}^{0,\updelta_0}(\widetilde{\Sigma}_t^{(Int)})}
\lesssim
\uplambda^{-1}
+
	\mathcal{M}\left(
		\| \pmb{\partial} \vec{\Psi} \|_{L_u^{\infty} C_{\upomega}^{0,\updelta_0}(\widetilde{\Sigma}_t^{(Int)})}
	\right).
\end{align}
Taking the norm $\| \cdot \|_{L_t^2([0,\RescaledTboot])}$ of \eqref{E:THIRDSTEPINTRCHISMALLCOMEGADELTAESTIMATE}
and using \eqref{E:RESCALEDBOOTBOUNDS},
\eqref{E:STANDARDMAXIMALFUNCTIONLQESTIMATE} with $\leb := 2$,
and
\eqref{E:ESTIMATEFORANGLEHOLDERCONTINUITYOFPARTIALPSIPARTIALVORTANDPARTIALGRADENT},
we conclude the desired bound
\eqref{E:IMPROVEDININTERIORL2INTIMELINFINITYINSPACECONNECTIONCOFFICIENTS} for $\mytr_{\congsphere} \widetilde{\upchi}^{(Small)}$.

The estimate \eqref{E:IMPROVEDININTERIORL2INTIMELINFINITYINSPACECONNECTIONCOFFICIENTS} 
for $\mytr_{\gsphere} \upchi - \frac{2}{\rgeo}$
then follows from the identity $\mytr_{\gsphere} \upchi - \frac{2}{\rgeo} = \mytr_{\congsphere} \widetilde{\upchi}^{(Small)} - \Chfour_{\Lunit}$,
the schematic relation 
$\Chfour_{\Lunit} = \gensmoothfunction_{(\vec{\Lunit})} \cdot \pmb{\partial} \vec{\Psi}$,
the already proven estimate \eqref{E:IMPROVEDININTERIORL2INTIMELINFINITYINSPACECONNECTIONCOFFICIENTS} for $\mytr_{\congsphere} \widetilde{\upchi}^{(Small)}$,
the product-type estimate
\eqref{E:SIMPLEHOLDERPRODUCESTIMATEONSPHERES},
\eqref{E:HOLDERESTIMATESFORF1TYPEFACTOR}, 
and \eqref{E:ESTIMATEFORANGLEHOLDERCONTINUITYOFPARTIALPSIPARTIALVORTANDPARTIALGRADENT}.

\subsubsection{Proof of \eqref{E:IMPROVEDININTERIORL2INTIMELINFINITYINSPACECONNECTIONCOFFICIENTS} for $\hat{\upchi}$}
We first use equation \eqref{E:DIVDIVTRFREECHISCHEMATIC},
the estimate \eqref{E:ANGULARHOLDERHODGEESTIMATENODERIVATIVESONLHSANGDIVSYMMETRICTRACEFREESTUTENSORFIELD} with $\leb := p$,
the parameter relation \eqref{E:BOUNDSONLEBESGUEEXPONENTP},
the product-type estimate
\eqref{E:SIMPLEHOLDERPRODUCESTIMATEONSPHERES},
\eqref{E:HOLDERESTIMATESFORF1TYPEFACTOR},
and H\"{o}lder's inequality
to deduce that 
\begin{align} \label{E:FIRSTSTEPINHATCHIIMPROVEDININTERIORL2INTIMELUINFINITYHOLDERINANGLESCHIHAT}
		\|
			\hat{\upchi}
		\|_{L_t^2 L_u^{\infty} C_{\upomega}^{0,\updelta_0}(\widetilde{\mathcal{M}}^{(Int)})}
		& \lesssim 
			\|
				\mytr_{\congsphere} \widetilde{\upchi}^{(Small)}
			\|_{L_t^2 L_u^{\infty} C_{\upomega}^{0,\updelta_0}(\widetilde{\mathcal{M}}^{(Int)})}
			+
			\|
				\pmb{\partial} \vec{\Psi}
			\|_{L_t^2 L_u^{\infty} C_{\upomega}^{0,\updelta_0}(\widetilde{\mathcal{M}}^{(Int)})}
			\\
		& \ \
			+
			\| \rgeo^{1/2} \|_{L^{\infty}(\widetilde{\mathcal{M}})}
			\| 
				\pmb{\partial} \vec{\Psi} 
			\|_{L_t^2 L_x^{\infty}(\widetilde{\mathcal{M}})}
			\left\|	
				\rgeo^{1/2}(\pmb{\partial} \vec{\Psi},\mytr_{\congsphere} \widetilde{\upchi}^{(Small)},\hat{\upchi})
			\right\|_{L_t^{\infty} L_u^{\infty} L_{\upomega}^p(\widetilde{\mathcal{M}})}.
			\notag
	\end{align}
	Using 
	the bound \eqref{E:RGEOANDUBOUNDS} for $\rgeo$,
	the estimate \eqref{E:ACOUSTICALLTINFTYLOMEGAPALONGCONES} for $\mytr_{\congsphere} \widetilde{\upchi}^{(Small)}$ and $\hat{\upchi}$,
	the estimate \eqref{E:FLUIDONEDERIVATIVESIGMAT} for $\pmb{\partial} \vec{\Psi}$,
	the already proven estimate \eqref{E:IMPROVEDININTERIORL2INTIMELINFINITYINSPACECONNECTIONCOFFICIENTS} for $\mytr_{\congsphere} \widetilde{\upchi}^{(Small)}$,
	\eqref{E:RESCALEDSTRICHARTZ} for $\pmb{\partial} \vec{\Psi}$,
	and \eqref{E:ESTIMATEFORANGLEHOLDERCONTINUITYOFPARTIALPSIPARTIALVORTANDPARTIALGRADENT} for $\pmb{\partial} \vec{\Psi}$,
	we conclude that 
	$\mbox{RHS~\eqref{E:FIRSTSTEPINHATCHIIMPROVEDININTERIORL2INTIMELUINFINITYHOLDERINANGLESCHIHAT}} \lesssim \uplambda^{-1/2 - 3 \upepsilon_0}$
	as desired.

\subsubsection{Proof of \eqref{E:CONNECTIONCOEFFICIENTESTIMATESNEEDEDTODERIVESPATIALLYLOCALIZEDDECAYFROMCONFORMALENERGYESTIMATE} for
$\mytr_{\congsphere} \widetilde{\upchi}^{(Small)}$ and $\mytr_{\gsphere} \upchi - \frac{2}{\rgeo}$}
We first bound
$\| \mytr_{\congsphere} \widetilde{\upchi}^{(Small)} \|_{L_t^{\frac{q}{2}} L_u^{\infty} C_{\upomega}^{0,\updelta_0}(\widetilde{\mathcal{M}})}$.
We start by noting the following estimate, which is a simple consequence of the estimate 
proved just below \eqref{E:FIRSTSTEPINCONEONLYTRCHISMALLCOMEGADELTAESTIMATE},
and which holds for $t \in [0,\RescaledTboot]$:
\begin{align} \label{E:ASECONDSTEPCONNECTIONCOEFFICIENTESTIMATESNEEDEDTODERIVESPATIALLYLOCALIZEDDECAYFROMCONFORMALENERGYESTIMATE}
\|
	\mytr_{\congsphere} \widetilde{\upchi}^{(Small)}
\|_{L_u^{\infty} C_{\upomega}^{0,\updelta_0}(\widetilde{\Sigma}_t)}
\lesssim
t^{-1/2}
\uplambda^{-1/2}
+
\uplambda^{-1 + 4 \upepsilon_0}
+
	\mathcal{M}\left(
		\| \pmb{\partial} \vec{\Psi} \|_{L_u^{\infty} C_{\upomega}^{0,\updelta_0}(\widetilde{\Sigma}_t)}
	\right).
\end{align}
Taking the norm $\| \cdot \|_{L_t^{\frac{q}{2}} ([0,\RescaledTboot])}$ of 
\eqref{E:ASECONDSTEPCONNECTIONCOEFFICIENTESTIMATESNEEDEDTODERIVESPATIALLYLOCALIZEDDECAYFROMCONFORMALENERGYESTIMATE}
and using \eqref{E:RESCALEDBOOTBOUNDS},
\eqref{E:STANDARDMAXIMALFUNCTIONLQESTIMATE} with $\leb := \frac{q}{2}$,
and
\eqref{E:ESTIMATEFORANGLEHOLDERCONTINUITYOFPARTIALPSIPARTIALVORTANDPARTIALGRADENT},
we conclude that if $q > 2$ is sufficiently close to $2$, then the desired estimate
\eqref{E:CONNECTIONCOEFFICIENTESTIMATESNEEDEDTODERIVESPATIALLYLOCALIZEDDECAYFROMCONFORMALENERGYESTIMATE} for $\mytr_{\congsphere} \widetilde{\upchi}^{(Small)}$
holds.

To prove \eqref{E:CONNECTIONCOEFFICIENTESTIMATESNEEDEDTODERIVESPATIALLYLOCALIZEDDECAYFROMCONFORMALENERGYESTIMATE} for $\mytr_{\gsphere} \upchi - \frac{2}{\rgeo}$,
we use \eqref{E:MODTRICHISMALL},
the schematic relation $\Chfour_{\Lunit} = \gensmoothfunction_{(\vec{\Lunit})} \cdot \pmb{\partial} \vec{\Psi}$,
the product-type estimate 
\eqref{E:SIMPLEHOLDERPRODUCESTIMATEONSPHERES},
\eqref{E:RESCALEDBOOTBOUNDS},
\eqref{E:ESTIMATEFORANGLEHOLDERCONTINUITYOFPARTIALPSIPARTIALVORTANDPARTIALGRADENT},
\eqref{E:HOLDERESTIMATESFORF1TYPEFACTOR},
the already proven estimate
\eqref{E:CONNECTIONCOEFFICIENTESTIMATESNEEDEDTODERIVESPATIALLYLOCALIZEDDECAYFROMCONFORMALENERGYESTIMATE} 
for $\| \mytr_{\congsphere} \widetilde{\upchi}^{(Small)} \|_{L_t^{\frac{q}{2}} L_u^{\infty} C_{\upomega}^{0,\updelta_0}(\widetilde{\mathcal{M}})}$,
to conclude that if $q > 2$ is sufficiently close to $2$, then
\begin{align*}
\| 
	\mytr_{\gsphere} \upchi - \frac{2}{\rgeo}
\|_{L_t^{\frac{q}{2}} L_u^{\infty} C_{\upomega}^{0,\updelta_0}(\widetilde{\mathcal{M}})}
& \lesssim
\| \mytr_{\congsphere} \widetilde{\upchi}^{(Small)} \|_{L_t^{\frac{q}{2}} L_u^{\infty} C_{\upomega}^{0,\updelta_0}(\widetilde{\mathcal{M}})}
+
\| \pmb{\partial} \vec{\Psi} \|_{L_t^{\frac{q}{2}} L_u^{\infty} C_{\upomega}^{0,\updelta_0}(\widetilde{\mathcal{M}})}
	\\
&
\lesssim \uplambda^{\frac{2}{q} - 1 - 4 \upepsilon_0(\frac{4}{q} - 1)}
+
(\uplambda^{1 - 8 \upepsilon_0})^{(\frac{2}{q} - \frac{1}{2})}
\cdot
\uplambda^{-1/2 - 3 \upepsilon_0}
\lesssim
\uplambda^{\frac{2}{q} - 1 - 4 \upepsilon_0 (\frac{4}{q} - 1)}
\end{align*}
as desired.

\subsubsection{Proof of \eqref{E:CONNECTIONCOEFFICIENTESTIMATESNEEDEDTODERIVESPATIALLYLOCALIZEDDECAYFROMCONFORMALENERGYESTIMATE} for
$\hat{\upchi}$}
A slight modification of the proof of \eqref{E:FIRSTSTEPINHATCHIIMPROVEDININTERIORL2INTIMELUINFINITYHOLDERINANGLESCHIHAT}
yields that
\begin{align} \label{E:FIRSTSTEPCONNECTIONCOEFFICIENTESTIMATESNEEDEDTODERIVESPATIALLYLOCALIZEDDECAYFROMCONFORMALENERGYESTIMATE}
		\|
			\hat{\upchi}
		\|_{L_t^{\frac{q}{2}} L_u^{\infty} C_{\upomega}^{0,\updelta_0}(\widetilde{\mathcal{M}})}
		& \lesssim 
			\|
				\mytr_{\congsphere} \widetilde{\upchi}^{(Small)}
			\|_{L_t^{\frac{q}{2}} L_u^{\infty} C_{\upomega}^{0,\updelta_0}(\widetilde{\mathcal{M}})}
			+
			\|
				\pmb{\partial} \vec{\Psi}
			\|_{L_t^{\frac{q}{2}} L_u^{\infty} C_{\upomega}^{0,\updelta_0}(\widetilde{\mathcal{M}})}
			\\
		& \ \
			+
			\| \rgeo^{1/2} \|_{L^{\infty}(\widetilde{\mathcal{M}})}
			\| 
				\pmb{\partial} \vec{\Psi} 
			\|_{L_t^{\frac{q}{2}} L_x^{\infty}(\widetilde{\mathcal{M}})}
			\left\|	
				\rgeo^{1/2}(\pmb{\partial} \vec{\Psi},\mytr_{\congsphere} \widetilde{\upchi}^{(Small)},\hat{\upchi})
			\right\|_{L_t^{\infty} L_u^{\infty} L_{\upomega}^p(\widetilde{\mathcal{M}})}.
			\notag
	\end{align}
	From 
	\eqref{E:FIRSTSTEPCONNECTIONCOEFFICIENTESTIMATESNEEDEDTODERIVESPATIALLYLOCALIZEDDECAYFROMCONFORMALENERGYESTIMATE},
	\eqref{E:RESCALEDBOOTBOUNDS},
	\eqref{E:RGEOANDUBOUNDS},
	\eqref{E:ACOUSTICALLTINFTYLOMEGAPALONGCONES}, the already proven bound 
	\eqref{E:CONNECTIONCOEFFICIENTESTIMATESNEEDEDTODERIVESPATIALLYLOCALIZEDDECAYFROMCONFORMALENERGYESTIMATE} 
	for $\mytr_{\congsphere} \widetilde{\upchi}^{(Small)}$,
	\eqref{E:SMOOTHFUNCTIONTIMESFLUIDONEDERIVATIVESIGMAT},
	and \eqref{E:ESTIMATEFORANGLEHOLDERCONTINUITYOFPARTIALPSIPARTIALVORTANDPARTIALGRADENT},
	we conclude that if $q > 2$ is sufficiently close to $2$, 
	then
	$
	\|
			\hat{\upchi}
		\|_{L_t^{\frac{q}{2}} L_u^{\infty} C_{\upomega}^{0,\updelta_0}(\widetilde{\mathcal{M}})}
	\lesssim
	\uplambda^{\frac{2}{q} - 1 - 4 \upepsilon_0(\frac{4}{q} - 1)}
	$
	as desired.

\subsubsection{Proof of \eqref{E:TRICHIMODSMALLTRFREECHIANDTORSIONL2INTIMEALONGSOUNDCONES} for $\hat{\upchi}$}
Using \eqref{E:EUCLIDEANFORMMORREYONSTU} with $\leb := p$
and taking into account \eqref{E:BOUNDSONLEBESGUEEXPONENTP},
we find that
$\| \hat{\upchi} \|_{L_t^2 C_{\upomega}^{0,\updelta_0}(\widetilde{\mathcal{C}}_u)}
\lesssim
\| \rgeo \angD \hat{\upchi} \|_{L_t^2 L_{\upomega}^p(\widetilde{\mathcal{C}}_u)}
+
\| \upchi \|_{L_t^2 L_{\upomega}^2(\widetilde{\mathcal{C}}_u)}
$.
Using the already proven estimates \eqref{E:ONEANGULARDERIVATIVEOFTRCHIMODANDTRFREECHIL2INTIMEESTIMATES} for
$\| \rgeo \angD \hat{\upchi} \|_{L_t^2 L_{\upomega}^p (\widetilde{\mathcal{C}}_u)}$
and \eqref{E:ACOUSTICALLT2LOMEGAPANDLDERIVATIVESALONGCONES}
for $\| \upchi \|_{L_t^2 L_{\upomega}^p(\widetilde{\mathcal{C}}_u)}$,
we conclude that the RHS of the previous expression is $\lesssim \uplambda^{-\frac{1}{2}}$ as desired.

\subsubsection{Proof of \eqref{E:COMPARISONWITHROUNDMETRICLINFINITYBOUNDS}--\eqref{E:ONEANGULARDERIVATIVECOMPARISONWITHROUNDMETRICLPINANGLESLINFINITYINTIMEBOUNDS}}
Based on the transport equation \eqref{E:EVOLUTIONEQUATIONFORANGULARCOORDINATECOMPONENTSOFGSPHEREMINUSEUCLIDEAN},
Lemma~\ref{L:TRANSPORT},
\eqref{E:SPHEREFINITELIMITSALONGTIP}--\eqref{E:ANGULARDERIVATIVESSPHEREFINITELIMITSALONGTIP},
\eqref{E:INITIALSPHEREMETRICESTIMATE}--\eqref{E:INITIALSPHEREMETRICFIRSTCOORDINATEANGULARDERIVATIVESESTIMATE},
\eqref{E:RESCALEDBOOTBOUNDS},
the bootstrap assumptions,
Prop.\,\ref{P:ESTIMATESFORFLUIDVARIABLES},
and the previously proven estimates 
\eqref{E:ACOUSTICALLT2LOMEGAPANDLDERIVATIVESALONGCONES},
\eqref{E:ONEANGULARDERIVATIVEOFTRCHIMODANDTRFREECHIL2INTIMEESTIMATES},
\eqref{E:TRICHIMODSMALLTRFREECHIANDTORSIONL2INTIMEALONGSOUNDCONES},
and
\eqref{E:IMPROVEDININTERIORL2INTIMELINFINITYINSPACECONNECTIONCOFFICIENTS},
the proof given in \cite{qW2017}*{Subsubsection~5.2.2} goes through verbatim.

\subsubsection{Proof of \eqref{E:L2INTIMEESTIMATESFORNULLLAPSEALONGCONES} and \eqref{E:IMPROVEDINTERIORL2INTIMEESTIMATESFORNULLLAPSEALONGCONES}}
Based on the transport equation \eqref{E:EVOLUTIONNULLAPSEUSEEULER},
the bootstrap assumptions,
and the previously proven estimates \eqref{E:NULLLAPSECLOSETOUNITY} and \eqref{E:ACOUSTICALLT2LOMEGAPANDLDERIVATIVESALONGCONES},
the arguments given in the discussion surrounding \cite{qW2017}*{Equation~(5.90)} go through verbatim.

\subsubsection{Proof of \eqref{E:CONESRATIOOFSPHEREVOLUMEFORMTORGEOSQUAREDTIMESROUNDVOLUMEFORML2INTIMELPINOMEGA}}
Based on the evolution equations \eqref{E:LUNITVOLUMEFORMRGEOTOMINUSTWORESCALED}--\eqref{E:LUNITANGDCOMMUTEDVOLUMEFORMRGEOTOMINUSTWORESCALED},
Lemma~\ref{L:TRANSPORT},
the bootstrap assumptions,
and the previously proven estimate \eqref{E:ONEANGULARDERIVATIVEOFTRCHIMODANDTRFREECHIL2INTIMEESTIMATES},
 the proof of \cite{qW2017}*{Lemma~5.15} 
(in which $\ln\left(\rgeo^{-2} \volrat \right)$ was denoted by ``$\varphi$'')
goes through verbatim.

\subsubsection{Proof of \eqref{E:TRICHIMODSMALLTRFREECHIANDTORSIONL2INTIMEALONGSOUNDCONES} 
for $\| \upzeta \|_{L_t^2 C_{\upomega}^{0,\updelta_0}(\widetilde{\mathcal{C}}_u)}$
and \eqref{E:CONESRWEIGHTEDMASSASPECTANDANGULARDERIVATIVESOFTORSIONL2INTIMELPINOMEGA}}
We first simultaneously prove 
\eqref{E:TRICHIMODSMALLTRFREECHIANDTORSIONL2INTIMEALONGSOUNDCONES}
for $\| \upzeta \|_{L_t^2 C_{\upomega}^{0,\updelta_0}(\widetilde{\mathcal{C}}_u)}$
and
\eqref{E:CONESRWEIGHTEDMASSASPECTANDANGULARDERIVATIVESOFTORSIONL2INTIMELPINOMEGA} for
$\| \rgeo \angD \upzeta \|_{L_t^2 L_{\upomega}^p(\widetilde{\mathcal{C}}_u)}$
with the help of the Hodge system \eqref{E:TORSIONDIV}--\eqref{E:TORSIONCURL}.
We now define the following two scalar functions:
$\mathfrak{F} := \mbox{RHS}~\eqref{E:TORSIONDIV}$,
$\mathfrak{G} := \mbox{RHS}~\eqref{E:TORSIONCURL}$.
From the Calderon--Zygmund estimate \eqref{E:ONEFORMSTUCALDERONZYGMUNDHODGEESTIMATES}
with $\leb := p$,
we deduce that 
for each fixed $u \in [-\frac{4}{5} \RescaledTboot, \RescaledTboot]$, 
we have
$\| \rgeo \angD \upzeta \|_{L_t^2 L_{\upomega}^p(\widetilde{\mathcal{C}}_u)}
\lesssim 
\| \rgeo (\mathfrak{F},\mathfrak{G}) \|_{L_t^2 L_{\upomega}^p(\widetilde{\mathcal{C}}_u)}
$.
In the arguments given just below \cite{qW2017}*{Equation~(5.97)},
based on the bootstrap assumptions,
\eqref{E:RESCALEDBOOTBOUNDS}, 
\eqref{E:RGEOANDUBOUNDS},
and the previously proven estimates
\eqref{E:ACOUSTICALLTINFTYLOMEGAPALONGCONES},
\eqref{E:TRICHIMODSMALLTRFREECHIANDTORSIONL2INTIMEALONGSOUNDCONES} for $\mytr_{\congsphere} \widetilde{\upchi}^{(Small)}$ and $\hat{\upchi}$,
and
\eqref{E:CONESRATIOOFSPHEREVOLUMEFORMTORGEOSQUAREDTIMESROUNDVOLUMEFORML2INTIMELPINOMEGA},
all terms on RHSs~\eqref{E:TORSIONDIV}--\eqref{E:TORSIONCURL} were shown to be bounded in the norm
$\| \rgeo \cdot \|_{L_t^2 L_{\upomega}^p(\widetilde{\mathcal{C}}_u)}$
by $\lesssim \uplambda^{-1/2} + \uplambda^{-4 \upepsilon_0}  \| \upzeta \|_{L_t^2 L_{\upomega}^{\infty}(\widetilde{\mathcal{C}}_u)}$,
except that the terms on RHS~\eqref{E:TORSIONDIV}
that are multiplied by $\uplambda^{-1}$ were not present there.
To handle these remaining terms, we use \eqref{E:RGEOLAMBDAINVERSELINEARTERMLT2LUINFTYLOMEGAP}.
We have thus shown that
$\| \rgeo \angD \upzeta \|_{L_t^2 L_{\upomega}^p(\widetilde{\mathcal{C}}_u)} 
\lesssim \uplambda^{-1/2} + \uplambda^{-4 \upepsilon_0}  \| \upzeta \|_{L_t^2 L_{\upomega}^{\infty}(\widetilde{\mathcal{C}}_u)}$.
Moreover, using \eqref{E:EUCLIDEANFORMMORREYONSTU} with $\leb := p$,
the parameter relation \eqref{E:BOUNDSONLEBESGUEEXPONENTP}, 
and \eqref{E:ACOUSTICALLT2LOMEGAPANDLDERIVATIVESALONGCONES}
(which implies that $\| \upzeta \|_{L_t^2 L_{\upomega}^2(\widetilde{\mathcal{C}}_u)} 
\lesssim 
\| \upzeta \|_{L_t^2 L_{\upomega}^p(\widetilde{\mathcal{C}}_u)}
\lesssim \uplambda^{-1/2}
$),
we find that
$\| \upzeta \|_{L_t^2 C_{\upomega}^{0,\updelta_0}(\widetilde{\mathcal{C}}_u)}
\lesssim
\| \rgeo \angD \upzeta \|_{L_t^2 L_{\upomega}^p(\widetilde{\mathcal{C}}_u)}
+
\| \upzeta \|_{L_t^2 L_{\upomega}^2(\widetilde{\mathcal{C}}_u)}
\lesssim
\| \rgeo \angD \upzeta \|_{L_t^2 L_{\upomega}^p(\widetilde{\mathcal{C}}_u)}
+
\uplambda^{-1/2}
$.
Combining the above estimates, we
find that
$\| \rgeo \angD \upzeta \|_{L_t^2 L_{\upomega}^p(\widetilde{\mathcal{C}}_u)} 
\lesssim \uplambda^{-1/2} + 
\uplambda^{-4 \upepsilon_0} \| \rgeo \angD \upzeta \|_{L_t^2 L_{\upomega}^p(\widetilde{\mathcal{C}}_u)}
$,
from which we readily conclude 
(when $\uplambda$ is sufficiently large)
the desired bound
\eqref{E:CONESRWEIGHTEDMASSASPECTANDANGULARDERIVATIVESOFTORSIONL2INTIMELPINOMEGA} for
$\| \rgeo \angD \upzeta \|_{L_t^2 L_{\upomega}^p(\widetilde{\mathcal{C}}_u)}$
and the desired bound \eqref{E:TRICHIMODSMALLTRFREECHIANDTORSIONL2INTIMEALONGSOUNDCONES}
for $\| \upzeta \|_{L_t^2 C_{\upomega}^{0,\updelta_0}(\widetilde{\mathcal{C}}_u)}$.

To prove \eqref{E:CONESRWEIGHTEDMASSASPECTANDANGULARDERIVATIVESOFTORSIONL2INTIMELPINOMEGA} for
$\| \rgeo \upmu  \|_{L_t^2 L_{\upomega}^p(\widetilde{\mathcal{C}}_u)}$,
we must show that 
$\| \rgeo \times \mbox{RHS~\eqref{E:MASSASPECTDECOMP}}  \|_{L_t^2 L_{\upomega}^p(\widetilde{\mathcal{C}}_u)}
\lesssim \uplambda^{-1/2}
$.
In the arguments given just below \cite{qW2017}*{Equation~(5.101)},
based on the bootstrap assumptions,
\eqref{E:RESCALEDBOOTBOUNDS},
\eqref{E:RGEOANDUBOUNDS}, 
and the previously proven estimates
\eqref{E:ACOUSTICALLTINFTYLOMEGAPALONGCONES},
\eqref{E:TRICHIMODSMALLTRFREECHIANDTORSIONL2INTIMEALONGSOUNDCONES},
and
\eqref{E:CONESRATIOOFSPHEREVOLUMEFORMTORGEOSQUAREDTIMESROUNDVOLUMEFORML2INTIMELPINOMEGA},
all terms on RHS~\eqref{E:MASSASPECTDECOMP} were shown to satisfy the desired bound,
except the term on RHS~\eqref{E:MASSASPECTDECOMP} that is multiplied by $\uplambda^{-1}$ was not present there.
To handle this remaining term, we use \eqref{E:RGEOLAMBDAINVERSELINEARTERMLT2LUINFTYLOMEGAP}.

\subsubsection{Proof of \eqref{E:IMPROVEDININTERIORL2INTIMELINFINITYINSPACECONNECTIONCOFFICIENTS} for $\upzeta$ and 
\eqref{E:CONNECTIONCOEFFICIENTESTIMATESNEEDEDTODERIVESPATIALLYLOCALIZEDDECAYFROMCONFORMALENERGYESTIMATE}
for $\upzeta$}
To prove \eqref{E:IMPROVEDININTERIORL2INTIMELINFINITYINSPACECONNECTIONCOFFICIENTS} for $\upzeta$,
we will use the Hodge system \eqref{E:TORSIONDIV}--\eqref{E:TORSIONCURL}. 
From these equations and 
the Calderon--Zygmund estimate
\eqref{E:LINFTYHODGEESTIMATENODERIVATIVESONLHSINVOLVINGCARTESIANCOMPONENTSONRHS} 
with $\upzeta$ in the role of $\upxi$, 
with $\leb :=p$ and $m := 2$, 
with $\gensmoothfunction_{(\vec{\Lunit})} \cdot \pmb{\partial} \vec{\Psi}$ in the role of $\mathfrak{F}$
(where $\mathfrak{F}$ represents the second terms on RHSs \eqref{E:TORSIONDIV} and \eqref{E:TORSIONCURL}),
with $\gensmoothfunction(\vec{\Psi}) \cdot \pmb{\partial} \vec{\Psi}$ in the role of $\vec{\widetilde{\mathfrak{F}}}$
and with $\updelta' > 0$ chosen to be sufficiently small,
we find (where the implicit constants can depend on $\updelta'$) that
\begin{align} \label{E:FIRSTSTEPFORTORSIONESTIMATEL2INTIMELINFINITYINSPACECONNECTIONCOFFICIENTS}
	\|
		\upzeta
	\|_{L_t^2 L_x^{\infty}(\widetilde{\mathcal{M}}^{(Int)})}
	& 
	\lesssim
	\left\| 
		\rgeo 
		(\pmb{\partial} \vec{\Psi},\mytr_{\congsphere} \widetilde{\upchi}^{(Small)},\hat{\upchi},\upzeta,\rgeo^{-1})
		\cdot
		(\pmb{\partial} \vec{\Psi},\hat{\upchi},\upzeta)
	\right\|_{L_t^2 L_u^{\infty} L_{\upomega}^p(\widetilde{\mathcal{M}}^{(Int)})}
	 \\
& \ \
+
\uplambda^{-1} 
\|
	\rgeo (\vec{\VortVort},\DivGradEnt)
\|_{L_t^2 L_u^{\infty} L_{\upomega}^p(\widetilde{\mathcal{M}})}
	+
	\left\|
		\rgeo
			\angD \ln\left( \rgeo^{-2} \volrat \right)
				\cdot
				(\pmb{\partial} \vec{\Psi},\upzeta)
	\right\|_{L_t^2 L_u^{\infty} L_{\upomega}^p(\widetilde{\mathcal{M}})}
	\notag \\
& \ \
	+
	\left\| 
		\upnu^{\updelta'} P_{\upnu} 
		\left(
			\gensmoothfunction_{(\vec{\Lunit})} \cdot \pmb{\partial} \vec{\Psi}
		\right)
	\right\|_{L_t^2 \ell_{\upnu}^2 L_x^{\infty}(\widetilde{\mathcal{M}})}
	+
	\| \pmb{\partial} \vec{\Psi}\|_{L_t^2 L_x^{\infty}(\widetilde{\mathcal{M}})}.
	\notag
\end{align}
Assuming that $\updelta' > 0$ is chosen to be sufficiently small
(in particular, at least as small as the parameter $\updelta_0$ in \eqref{E:RESCALEDBOOTL2LINFINITYFIRSTDERIVATIVESOFVORTICITYBOOTANDNENTROPYGRADIENT}),
the arguments given on \cite{qW2017}*{page~52}
show that, 
thanks to the bootstrap assumptions,
\eqref{E:RESCALEDBOOTBOUNDS},
\eqref{E:RGEOANDUBOUNDS}, 
and the already proven estimates
\eqref{E:ACOUSTICALLTINFTYLOMEGAPALONGCONES},
\eqref{E:TRICHIMODSMALLTRFREECHIANDTORSIONL2INTIMEALONGSOUNDCONES},
\eqref{E:IMPROVEDININTERIORL2INTIMELINFINITYINSPACECONNECTIONCOFFICIENTS}
for $\mytr_{\congsphere} \widetilde{\upchi}^{(Small)}$ and $\hat{\upchi}$,
and
\eqref{E:CONESRATIOOFSPHEREVOLUMEFORMTORGEOSQUAREDTIMESROUNDVOLUMEFORML2INTIMELPINOMEGA},
all terms on RHS~\eqref{E:FIRSTSTEPFORTORSIONESTIMATEL2INTIMELINFINITYINSPACECONNECTIONCOFFICIENTS}
are
$\lesssim \uplambda^{-\frac{1}{2} - 3 \upepsilon_0} 
+ 
 \uplambda^{- 4 \upepsilon_0} 
\|
		\upzeta
\|_{L_t^2 L_x^{\infty}(\widetilde{\mathcal{M}}^{(Int)})}
$,
except that the term on RHS~\eqref{E:FIRSTSTEPFORTORSIONESTIMATEL2INTIMELINFINITYINSPACECONNECTIONCOFFICIENTS} that
is multiplied by $\uplambda^{-1}$ was not present in \cite{qW2017}. 
To handle this remaining term,
we use \eqref{E:RGEOLAMBDAINVERSELINEARTERMLT2LUINFTYLOMEGAP}.
This shows that
$
\|
	\upzeta
\|_{L_t^2 L_x^{\infty}(\widetilde{\mathcal{M}}^{(Int)})}
\lesssim \uplambda^{-\frac{1}{2} - 3 \upepsilon_0} 
+ 
 \uplambda^{- 4 \upepsilon_0} 
\|
	\upzeta
\|_{L_t^2 L_x^{\infty}(\widetilde{\mathcal{M}}^{(Int)})}
$
which, when $\uplambda$ is sufficiently large,
yields the desired bound \eqref{E:IMPROVEDININTERIORL2INTIMELINFINITYINSPACECONNECTIONCOFFICIENTS} for $\upzeta$.

Similarly, 
based on the Hodge system \eqref{E:TORSIONDIV}--\eqref{E:TORSIONCURL}, 
the Calderon--Zygmund estimate
\eqref{E:LINFTYHODGEESTIMATENODERIVATIVESONLHSINVOLVINGCARTESIANCOMPONENTSONRHS} with $\leb :=p$, 
the bootstrap assumptions,
\eqref{E:RESCALEDBOOTBOUNDS},
\eqref{E:RGEOANDUBOUNDS}, 
the already proven estimates
\eqref{E:ACOUSTICALLTINFTYLOMEGAPALONGCONES}
and
\eqref{E:CONESRATIOOFSPHEREVOLUMEFORMTORGEOSQUAREDTIMESROUNDVOLUMEFORML2INTIMELPINOMEGA}
and the already proven estimate
\eqref{E:CONNECTIONCOEFFICIENTESTIMATESNEEDEDTODERIVESPATIALLYLOCALIZEDDECAYFROMCONFORMALENERGYESTIMATE}
for $\mytr_{\congsphere} \widetilde{\upchi}^{(Small)}$ and $\hat{\upchi}$,
the arguments given on \cite{qW2017}*{page~52} yield the desired estimate
\eqref{E:CONNECTIONCOEFFICIENTESTIMATESNEEDEDTODERIVESPATIALLYLOCALIZEDDECAYFROMCONFORMALENERGYESTIMATE}
for $\upzeta$, where we use \eqref{E:RGEOLAMBDAINVERSELINEARTERMLTQOVER2LUINFTYLOMEGAP} to handle the
$\uplambda^{-1}$-multiplied terms on RHSs~\eqref{E:TORSIONDIV}--\eqref{E:TORSIONCURL}.

\subsubsection{Proof of \eqref{E:CONEFIRSTDERIVATIVEBOUNDSFORCONFORMALFACTOR}--\eqref{E:RWEIGHTEDLINFTYBOUNDSFORCONFORMALFACTOR}}
Based on \eqref{E:SIGMAEVOLUTION}--\eqref{E:SIGMADATA},
Lemma~\ref{L:TRANSPORT}, 
the bootstrap assumptions,
and the previously proven estimate \eqref{E:IMPROVEDININTERIORL2INTIMELINFINITYINSPACECONNECTIONCOFFICIENTS},
the proof of these estimates given in \cite{qW2017}*{Lemma~6.1} goes through verbatim,
except for the estimate \eqref{E:CONEFIRSTDERIVATIVEBOUNDSFORCONFORMALFACTOR}
for $\| \rgeo^{\frac{1}{2}} \angD \upsigma \|_{L_{\upomega}^p L_t^{\infty}(\widetilde{\mathcal{C}}_u)}$.
To bound this remaining term,
we first use \eqref{E:STUVOLUMEFORMCOMPARISONWITHUNITROUNDMETRICVOLUMEFORM}
to deduce
(noting that $u \geq 0$ since, by assumption, we have $\widetilde{\mathcal{C}}_u \subset \widetilde{\mathcal{M}}^{(Int)}$)
\begin{align} \label{E:FIRSTSTEPINPROOFOFEXTRATERMBOUNDCONEFIRSTDERIVATIVEBOUNDSFORCONFORMALFACTOR}
\| \rgeo^{\frac{1}{2}} \angD \upsigma \|_{L_{\upomega}^p L_t^{\infty}(\widetilde{\mathcal{C}}_u)}^p
& 
\lesssim
\int_{\mathbb{S}^2}
\mbox{ess sup}_{t \in [u,\RescaledTboot]} 
|\rgeo^{\frac{1}{2}} \angD \upsigma |_{\gsphere}^p(t,u,\upomega)
\, d \flatspherevolarg{\upomega}	
	\\
& 
\lesssim
\int_{\mathbb{S}^2}
\mbox{ess sup}_{t \in [u,\RescaledTboot]} 
\left\lbrace
\volrat(t,u,\upomega)
| 
	\rgeo^{\frac{1}{2} - \frac{2}{p}} \angD \upsigma 
|_{\gsphere}^p(t,u,\upomega)
\right\rbrace
\, d \flatspherevolarg{\upomega}
	\notag
		\\
& :=
	\| \rgeo^{\frac{1}{2} - \frac{2}{p}} \angD \upsigma \|_{L_{\gsphere}^p L_t^{\infty}(\widetilde{\mathcal{C}}_u)}^p.
\notag
\end{align}
From \eqref{E:FIRSTSTEPINPROOFOFEXTRATERMBOUNDCONEFIRSTDERIVATIVEBOUNDSFORCONFORMALFACTOR} and the already proven bound
\eqref{E:CONEFIRSTDERIVATIVEBOUNDSFORCONFORMALFACTOR} for 
$\| \rgeo^{\frac{1}{2} - \frac{2}{p}} \angD \upsigma \|_{L_{\gsphere}^p L_t^{\infty}(\widetilde{\mathcal{C}}_u)}$,
we conclude that 
$\mbox{RHS}~\eqref{E:FIRSTSTEPINPROOFOFEXTRATERMBOUNDCONEFIRSTDERIVATIVEBOUNDSFORCONFORMALFACTOR}
\lesssim
\uplambda^{-p/2}
$ 
as desired.

\subsubsection{Proof of \eqref{E:SPACETIMEL2INUANDTLPINOMEGAFORANGDSIGMAANDRGEOWEIGHTEDMODIFIEDMASSASPECTANDANGDMODIFIEDTORSION}--\eqref{E:SPACETIMEL2INULINFINTYINTLPINOMEGAFORRGEOTHREEHAVESWEIGHTEDMODIFIEDMASSASPECT}}
We make the bootstrap assumption 
$\| \angD \upsigma \|_{L_u^2 L_t^2 L_{\upomega}^{\infty} (\widetilde{\mathcal{M}}^{(Int)})} \leq 1$;
this is viable because \eqref{E:SPACETIMEL2INUANDTLPINOMEGAFORANGDSIGMAANDRGEOWEIGHTEDMODIFIEDMASSASPECTANDANGDMODIFIEDTORSION}
yields an improvement of this bootstrap assumption.

We start by deriving a preliminary estimate for 
$\| \rgeo \angD \widetilde{\upzeta} \|_{L_u^2 L_t^2 L_{\upomega}^p (\widetilde{\mathcal{M}}^{(Int)})}$
using Hodge system \eqref{E:MODIFIEDTORSIONDIV}--\eqref{E:MODIFIEDTORSIONCURL}.
To proceed, we define the following two scalar functions:
$\mathfrak{F} := \mbox{RHS}~\eqref{E:MODIFIEDTORSIONDIV}$,
$\mathfrak{G} := \mbox{RHS}~\eqref{E:MODIFIEDTORSIONCURL}$.
From these equations and the Calderon--Zygmund estimate \eqref{E:ONEFORMSTUCALDERONZYGMUNDHODGEESTIMATES},
we deduce
\begin{align}
\notag
\| \rgeo \angD \widetilde{\upzeta} \|_{L_u^2 L_t^2 L_{\upomega}^p (\widetilde{\mathcal{M}}^{(Int)})}
\lesssim
\| \rgeo \mathfrak{F} \|_{L_u^2 L_t^2 L_{\upomega}^p (\widetilde{\mathcal{M}}^{(Int)})}
+
\| \rgeo \mathfrak{G} \|_{L_u^2 L_t^2 L_{\upomega}^p (\widetilde{\mathcal{M}}^{(Int)})}
+
\| \rgeo \check{\upmu} \|_{L_u^2 L_t^2 L_{\upomega}^p (\widetilde{\mathcal{M}}^{(Int)})}.
\end{align}
In the last two paragraphs of the proof of \cite{qW2017}*{Proposition~6.3},
based on the bootstrap assumptions,
\eqref{E:RESCALEDBOOTBOUNDS},
\eqref{E:RGEOANDUBOUNDS}, 
and the previously proven estimates
\eqref{E:ACOUSTICALLT2LOMEGAPANDLDERIVATIVESALONGCONES}
and
\eqref{E:IMPROVEDININTERIORL2INTIMELINFINITYINSPACECONNECTIONCOFFICIENTS},
the author showed that all terms on RHSs~\eqref{E:MODIFIEDTORSIONDIV}--\eqref{E:MODIFIEDTORSIONCURL}
are bounded in the norm
$\| \rgeo \cdot \|_{L_u^2 L_t^2 L_{\upomega}^p (\widetilde{\mathcal{M}}^{(Int)})}$
by $\lesssim \uplambda^{- 4 \upepsilon_0}$,
except that the terms on RHS~\eqref{E:MODIFIEDTORSIONDIV}
that are multiplied by $\uplambda^{-1}$ were not present there. To handle these remaining terms, 
we use \eqref{E:RGEOLAMBDAINVERSELINEARTERMLU2LT2LOMEGAP},
which in total yields the desired preliminary estimate
$
\| \rgeo \angD \widetilde{\upzeta} \|_{L_u^2 L_t^2 L_{\upomega}^p (\widetilde{\mathcal{M}}^{(Int)})}
\lesssim \uplambda^{- 4 \upepsilon_0}
+
\| \rgeo \check{\upmu} \|_{L_u^2 L_t^2 L_{\upomega}^p (\widetilde{\mathcal{M}}^{(Int)})}$.

We now derive estimates for $\check{\upmu}$.
Using the transport equation \eqref{E:MODIFIEDMASSASPECTEVOLUTIONEQUATION},
the identity \eqref{E:MAIDENTITYTRANSPORTLEMMA},
the vanishing initial conditions for $\rgeo^2 \check{\upmu}$
along the cone-tip axis guaranteed by \eqref{E:CONNECTIONCOEFFICIENTS0LIMITSALONGTIP},
and \eqref{E:STUVOLUMEFORMCOMPARISONWITHUNITROUNDMETRICVOLUMEFORM},
we see that in $\widetilde{\mathcal{M}}^{(Int)}$, we have
\begin{align} \label{E:CHECKMUTRANSPORTINTEGRATED}
	| 
		\rgeo^2 \check{\upmu}
	|(t,u,\upomega)
	& \lesssim
		\int_u^t
			\rgeo^2
			\left\lbrace
				|
				\mathfrak{I}_{(1)}
				+
				\mathfrak{I}_{(2)}
			|
			\right\rbrace
			(\uptau,u,\upomega)
		\, d \uptau,
\end{align}
where $\mathfrak{I}_{(1)}$ and $\mathfrak{I}_{(2)}$ 
are defined in
\eqref{E:FIRSTINHOMTERMMODIFIEDMASSASPECTEVOLUTIONEQUATION}--\eqref{E:SECONDINHOMTERMMODIFIEDMASSASPECTEVOLUTIONEQUATION}.
We now divide \eqref{E:CHECKMUTRANSPORTINTEGRATED} by $\rgeo(t,u)$ and take the norm
$\| \cdot \|_{L_u^2 L_t^2 L_{\upomega}^p (\widetilde{\mathcal{M}}^{(Int)})}$.
In the proof of \cite{qW2017}*{Proposition~6.3}, 
the author derived estimates for the terms on RHS~\eqref{E:CHECKMUTRANSPORTINTEGRATED} that imply,
based on the bootstrap assumptions,
\eqref{E:RESCALEDBOOTBOUNDS},
\eqref{E:RGEOANDUBOUNDS}, 
and the previously proven estimates
\eqref{E:ACOUSTICALLT2LOMEGAPANDLDERIVATIVESALONGCONES},
\eqref{E:ASECONDACOUSTICALLTINFTYLOMEGAPALONGCONES},
\eqref{E:ONEANGULARDERIVATIVEOFTRCHIMODLINFINITYESTIMATES},
\eqref{E:ONEANGULARDERIVATIVEOFTRCHIMODANDTRFREECHIL2INTIMEESTIMATES},
and
\eqref{E:IMPROVEDININTERIORL2INTIMELINFINITYINSPACECONNECTIONCOFFICIENTS},
that
$\| \rgeo \check{\upmu} \|_{L_u^2 L_t^2 L_{\upomega}^p (\widetilde{\mathcal{M}}^{(Int)})}
\lesssim 
\uplambda^{- 4 \upepsilon_0}
+
\uplambda^{- 8 \upepsilon_0}
\| \rgeo \check{\upmu} \|_{L_u^2 L_t^2 L_{\upomega}^p (\widetilde{\mathcal{M}}^{(Int)})}
+
\uplambda^{- 4 \upepsilon_0}
\| \rgeo \angD \widetilde{\upzeta} \|_{L_u^2 L_t^2 L_{\upomega}^p (\widetilde{\mathcal{M}}^{(Int)})}
$,
except that the terms on RHS~\eqref{E:SECONDINHOMTERMMODIFIEDMASSASPECTEVOLUTIONEQUATION} 
that are multiplied by $\uplambda^{-1}$ were not present there. To handle these remaining terms, we use
\eqref{E:LAMBDAINVERSEQUADRATICTERMTIMEINTEGRALLT2LU2LOMEGAP}
and
\eqref{E:LAMBDAINVERSETIMEINTEGRALLT2LU2LOMEGAP}.
Considering also the preliminary estimate for
$
\| \rgeo \angD \widetilde{\upzeta} \|_{L_u^2 L_t^2 L_{\upomega}^p (\widetilde{\mathcal{M}}^{(Int)})}
$ derived in the previous paragraph,
we deduce
$\| \rgeo \check{\upmu} \|_{L_u^2 L_t^2 L_{\upomega}^p (\widetilde{\mathcal{M}}^{(Int)})}
\lesssim \uplambda^{- 4 \upepsilon_0}
+
\uplambda^{- 8 \upepsilon_0}
\| \rgeo \check{\upmu} \|_{L_u^2 L_t^2 L_{\upomega}^p (\widetilde{\mathcal{M}}^{(Int)})}
$.
Thus, when $\uplambda$ is sufficiently large, we conclude the desired bound
\eqref{E:SPACETIMEL2INUANDTLPINOMEGAFORANGDSIGMAANDRGEOWEIGHTEDMODIFIEDMASSASPECTANDANGDMODIFIEDTORSION} for
$\| \rgeo \check{\upmu} \|_{L_u^2 L_t^2 L_{\upomega}^p (\widetilde{\mathcal{M}}^{(Int)})}$.
Inserting this bound into the preliminary estimate for
$
\| \rgeo \angD \widetilde{\upzeta} \|_{L_u^2 L_t^2 L_{\upomega}^p (\widetilde{\mathcal{M}}^{(Int)})}
$ derived in the previous paragraph,
we also conclude the desired bound \eqref{E:SPACETIMEL2INUANDTLPINOMEGAFORANGDSIGMAANDRGEOWEIGHTEDMODIFIEDMASSASPECTANDANGDMODIFIEDTORSION} for
$\| \rgeo \angD \widetilde{\upzeta} \|_{L_u^2 L_t^2 L_{\upomega}^p (\widetilde{\mathcal{M}}^{(Int)})}$.

A similar argument yields \eqref{E:SPACETIMEL2INULINFINTYINTLPINOMEGAFORRGEOTHREEHAVESWEIGHTEDMODIFIEDMASSASPECT}, 
where we divide \eqref{E:CHECKMUTRANSPORTINTEGRATED} by $\rgeo^{1/2}(t,u)$,
and to handle the terms on RHS~\eqref{E:SECONDINHOMTERMMODIFIEDMASSASPECTEVOLUTIONEQUATION} 
that are multiplied by $\uplambda^{-1}$,
we use 
\eqref{E:LAMBDAINVERSEQUADRATICTERMTIMEINTEGRALLU2LTINFTYLOMEGAP}
and
\eqref{E:LAMBDAINVERSETIMEINTEGRALLU2LTINFTYLOMEGAP}; we omit the details.

It remains for us to prove the estimate \eqref{E:SPACETIMEL2INUANDTLPINOMEGAFORANGDSIGMAANDRGEOWEIGHTEDMODIFIEDMASSASPECTANDANGDMODIFIEDTORSION} for 
$\| \angD \upsigma \|_{L_u^2 L_t^2 C_{\upomega}^{0,\updelta_0} (\widetilde{\mathcal{M}}^{(Int)})}$.
First, using definition \eqref{E:MODTORSION},
\eqref{E:EUCLIDEANFORMMORREYONSTU} with $\leb := p$,
and the parameter relation \eqref{E:BOUNDSONLEBESGUEEXPONENTP}, 
we see that 
\begin{align}
\notag
\| \angD \upsigma \|_{L_u^2 L_t^2 C_{\upomega}^{0,\updelta_0} (\widetilde{\mathcal{M}}^{(Int)})}
		\lesssim 
		\| \rgeo \angD (\widetilde{\upzeta},\upzeta) \|_{L_u^2 L_t^2 L_{\upomega}^p (\widetilde{\mathcal{M}}^{(Int)})}
		+
		\| \angD \upsigma \|_{L_u^2 L_t^2 L_{\upomega}^2 (\widetilde{\mathcal{M}}^{(Int)})}.
\end{align}
We have already shown that $\| \rgeo \angD \widetilde{\upzeta} \|_{L_u^2 L_t^2 L_{\upomega}^p (\widetilde{\mathcal{M}}^{(Int)})} \lesssim \uplambda^{- 4 \upepsilon_0}$.
To bound $\| \rgeo \angD \upzeta \|_{L_u^2 L_t^2 L_{\upomega}^p (\widetilde{\mathcal{M}}^{(Int)})}$
by $\lesssim \uplambda^{- 4 \upepsilon_0}$,
we square the already proven estimate \eqref{E:CONESRWEIGHTEDMASSASPECTANDANGULARDERIVATIVESOFTORSIONL2INTIMELPINOMEGA}
for $\| \rgeo \angD \upzeta  \|_{L_t^2 L_{\upomega}^p(\widetilde{\mathcal{C}}_u)}$,
integrate with respect to $u$ over $u \in [0,\RescaledTboot]$, and use the bound \eqref{E:RGEOANDUBOUNDS} for $u$.
Finally, to obtain the bound
$\| \angD \upsigma \|_{L_u^2 L_t^2 L_{\upomega}^2 (\widetilde{\mathcal{M}}^{(Int)})} \lesssim \uplambda^{- 4 \upepsilon_0}$,
we square the already proven estimate \eqref{E:CONEFIRSTDERIVATIVEBOUNDSFORCONFORMALFACTOR}
for $\| \angD \upsigma \|_{L_t^2 L_{\upomega}^p(\widetilde{\mathcal{C}}_u)}$,
integrate with respect to $u$ over $u \in [0,\RescaledTboot]$, 
and use the bound \eqref{E:RGEOANDUBOUNDS} for $u$. Combining these estimates,
we conclude that $\| \angD \upsigma \|_{L_u^2 L_t^2 C_{\upomega}^{0,\updelta_0} (\widetilde{\mathcal{M}}^{(Int)})} \lesssim \uplambda^{- 4 \upepsilon_0}$ 
as desired.

\begin{remark}
Throughout the rest of the proof of Prop.\,\ref{P:MAINESTIMATESFOREIKONALFUNCTIONQUANTITIES},
we silently use the following estimates, 
valid for $1 \leq \leb \leq \infty$,
which are simple consequences of \eqref{E:STUVOLUMEFORMCOMPARISONWITHUNITROUNDMETRICVOLUMEFORM}:
$|\overline{f}(t,u)| \lesssim \| f \|_{L_{\upomega}^{\leb}(S_{t,u})}$
and
$\| \overline{f} \|_{L_{\upomega}^{\leb}(S_{t,u})} \lesssim \| f \|_{L_{\upomega}^{\leb}(S_{t,u})}$
(see \eqref{E:AVERAGEVALUEOFSCALARFUNCTION} regarding the ``overline'' notation).
\end{remark}

\subsubsection{Proof of \eqref{E:PRELIMINARYANGUPMUSPACETIMEBOUNDS}}
Using the Hodge system \eqref{E:FURTHERMODOFMASSASPECT},
\eqref{E:EUCLIDEANFORMMORREYONSTU} with $\leb := p$,
and
\eqref{E:ONEFORMSTUCALDERONZYGMUNDHODGEESTIMATES} with $\leb := p$,
we find that
\begin{align} \label{E:FIRSTSTEPPRELIMINARYANGUPMUSPACETIMEBOUNDS}
		\| (\rgeo \angD \angupmu, \angupmu) \|_{L_{\upomega}^p(S_{t,u})},
			\,
		\| \angupmu\|_{L_{\upomega}^{\infty}(S_{t,u})}
		& \lesssim
		\| \rgeo(\check{\upmu} - \overline{\check{\upmu}}) \|_{L_{\upomega}^p(S_{t,u})}.
\end{align}
Taking the norm $\| \cdot \|_{L_t^2 L_u^2}$ of \eqref{E:FIRSTSTEPPRELIMINARYANGUPMUSPACETIMEBOUNDS}
over the range of $(t,u)$-values corresponding to $\widetilde{\mathcal{M}}^{(Int)}$
and using the already proven estimate \eqref{E:SPACETIMEL2INUANDTLPINOMEGAFORANGDSIGMAANDRGEOWEIGHTEDMODIFIEDMASSASPECTANDANGDMODIFIEDTORSION}
for $\| \rgeo \check{\upmu} \|_{L_u^2 L_t^2 L_{\upomega}^p (\widetilde{\mathcal{M}}^{(Int)})}$,
we arrive at the desired bound \eqref{E:PRELIMINARYANGUPMUSPACETIMEBOUNDS}.

\subsubsection{Proof of \eqref{E:KEYANGCONFORMALFACTORALGEBRAICDECOMPOSITION}--\eqref{E:ANGDUPSIGMAL2ULINFINITYCONEESTIMATE}}
We first note that the decomposition \eqref{E:KEYANGCONFORMALFACTORALGEBRAICDECOMPOSITION} follows from the definitions of the quantities involved.

Throughout the rest of proof, $\mathscr{D}^{-1}(\mathfrak{F},\mathfrak{G})$ will denote the solution $\upxi$
the following Hodge system on $S_{t,u}$: $\angdiv \upxi = \mathfrak{F}$, $\angcurl \upxi = \mathfrak{G}$.
In our applications, $\upxi$ will be a one-form or a symmetric trace-free type $\binom{0}{2}$ tensor
(where in the latter case, one can show that the one-forms $\mathfrak{F}$ and $\mathfrak{G}$ are constrained by
the relation $\mathfrak{G}_A = \upepsilon_{AB} \mathfrak{F}_A$, where $\upepsilon_{AB}$ is the antisymmetric symbol 
with $\upepsilon_{12} = 1$ relative to a $\gsphere$-orthonormal frame on $S_{t,u}$).

We start by proving \eqref{E:ANGDUPSIGMAL2TLINFINITYSPACEPARTESTIMATE}
for the term $\| \widetilde{\upzeta} - \angupmu \|_{L_t^2 L_x^{\infty}(\widetilde{\mathcal{M}}^{(Int)})}$
on the LHS. We will use the Hodge system 
\eqref{E:COMBBINEDMODIFIEDTORSIONMINUSANGMODMASSASPECTDIV}--\eqref{E:COMBBINEDMODIFIEDTORSIONMINUSANGMODMASSASPECTCURL}.
Note that we can split $\widetilde{\upzeta} - \angupmu = \mathscr{D}^{-1}(\angdiv \upxi,\angcurl \upxi)
+
\mathscr{D}^{-1}(\cdots,\cdots)
$,
where $\angdiv \upxi$ is the first term on RHS~\eqref{E:COMBBINEDMODIFIEDTORSIONMINUSANGMODMASSASPECTDIV},
$\angcurl \upxi$ is the first term on RHS~\eqref{E:COMBBINEDMODIFIEDTORSIONMINUSANGMODMASSASPECTCURL},
and $(\cdots,\cdots)$ denotes the remaining terms on
RHSs~\eqref{E:COMBBINEDMODIFIEDTORSIONMINUSANGMODMASSASPECTDIV}--\eqref{E:COMBBINEDMODIFIEDTORSIONMINUSANGMODMASSASPECTCURL}.
Recall that the $S_{t,u}$-tangent tensorfields denoted here by $\upxi$ have Cartesian component functions 
of the form $\gensmoothfunction_{(\vec{\Lunit})} \cdot \pmb{\partial} \vec{\Psi}$
(and thus $\upxi$ satisfies the hypotheses needed to apply the estimate \eqref{E:LINFTYHODGEESTIMATENODERIVATIVESONLHSINVOLVINGCARTESIANCOMPONENTSONRHS}
with $\gensmoothfunction_{(\vec{\Lunit})} \cdot \pmb{\partial} \vec{\Psi}$ in the role of $\mathfrak{F}$ 
and $\gensmoothfunction(\vec{\Psi}) \cdot \pmb{\partial} \vec{\Psi}$ in the role of $\vec{\widetilde{\mathfrak{F}}}$).
Therefore, using \eqref{E:LINFTYHODGEESTIMATENODERIVATIVESONLHSINVOLVINGCARTESIANCOMPONENTSONRHS}
with $\updelta' > 0$ chosen to be sufficiently small 
(at least as small as the parameter $\updelta_0 > 0$ in \eqref{E:RESCALEDBOOTL2LINFINITYFIRSTDERIVATIVESOFVORTICITYBOOTANDNENTROPYGRADIENT}) and $m := 2$
to handle the term 
$\mathscr{D}^{-1}(\angdiv \upxi,\angcurl \upxi)$,
and
\eqref{E:EUCLIDEANFORMMORREYONSTU} 
and
\eqref{E:ONEFORMSTUCALDERONZYGMUNDHODGEESTIMATES} with $\leb := p$
to handle the term $\mathscr{D}^{-1}(\cdots,\cdots)$, 
we deduce that
\begin{align}
\| \widetilde{\upzeta} - \angupmu \|_{L_t^2 L_x^{\infty}(\widetilde{\mathcal{M}}^{(Int)})}
& 
\lesssim
\| 
	\rgeo 
	(\pmb{\partial} \vec{\Psi},\mytr_{\congsphere} \widetilde{\upchi}^{(Small)},\hat{\upchi},\upzeta,\rgeo^{-1})
	\cdot
	(\pmb{\partial} \vec{\Psi},\hat{\upchi},\upzeta)
\|_{L_t^2 L_u^{\infty} L_{\upomega}^p(\widetilde{\mathcal{M}}^{(Int)})}
	\label{E:MODTORSIONMINUSUPSIGMAKEYSPACETIMEBOUND} \\
& \ \
+
\uplambda^{-1} 
\|
	\rgeo (\vec{\VortVort},\DivGradEnt)
\|_{L_t^2 L_u^{\infty} L_{\upomega}^p(\widetilde{\mathcal{M}})}
	\notag \\
& \ \
	+
	\left\| 
		\upnu^{\updelta'} P_{\upnu} 
		\left(
			\gensmoothfunction(\vec{\Psi}) \cdot \pmb{\partial} \vec{\Psi} 
		\right)
	\right\|_{L_t^2 \ell_{\upnu}^2 L_x^{\infty}(\widetilde{\mathcal{M}})}
	+
	\| \pmb{\partial} \vec{\Psi}\|_{L_t^2 L_x^{\infty}(\widetilde{\mathcal{M}})}.
	\notag
\end{align}
At the very end of the proof of \cite{qW2017}*{Proposition~6.4} 
(in which the author derived bounds for the second piece of a quantity denoted by ``$\textbf{A}^{\dagger}$,'' which was split into two pieces there),
the author gave arguments showing that,
thanks to the bootstrap assumptions,
\eqref{E:RGEOANDUBOUNDS}, 
and the previously proven estimates
\eqref{E:ASECONDACOUSTICALLTINFTYLOMEGAPALONGCONES}
and
\eqref{E:IMPROVEDININTERIORL2INTIMELINFINITYINSPACECONNECTIONCOFFICIENTS},
all terms on RHS~\eqref{E:MODTORSIONMINUSUPSIGMAKEYSPACETIMEBOUND}
are $\lesssim \uplambda^{-\frac{1}{2} - 3 \upepsilon_0}$,
except that the term
$\uplambda^{-1} 
\|
	\rgeo (\vec{\VortVort},\DivGradEnt)
\|_{L_t^2 L_u^{\infty} L_{\upomega}^p(\widetilde{\mathcal{M}})}$
was not present in \cite{qW2017}.
To handle this remaining term, we use \eqref{E:RGEOLAMBDAINVERSELINEARTERMLT2LUINFTYLOMEGAP}.
We have therefore proved \eqref{E:ANGDUPSIGMAL2TLINFINITYSPACEPARTESTIMATE}
for $\| \widetilde{\upzeta} - \angupmu \|_{L_t^2 L_x^{\infty}(\widetilde{\mathcal{M}}^{(Int)})}$.

To prove \eqref{E:ANGUPMU1AND2CONETIPCONDITIONS}, we first note that in view of \eqref{E:CONETIPINITIALCONDITIONSFORANGMUSLASHI1AND2},
it suffices to show that
$
\rgeo \angupmu(t,u,\upomega)
=
\mathcal{O}(\rgeo) \mbox{as } t \downarrow u
$.
The desired bound follows from 
applying the Calderon--Zygmund estimate \eqref{E:LINFTYHODGEESTIMATENODERIVATIVESONLHSINVOLVINGCARTESIANCOMPONENTSONRHS}
with $\mathfrak{F} = 0$ to the Hodge system \eqref{E:FURTHERMODOFMASSASPECT}
and using the asymptotic estimate \eqref{E:CONNECTIONCOEFFICIENTS0LIMITSALONGTIP}
for $\check{\upmu}$.

We now prove the estimate \eqref{E:ANGDUPSIGMAL2TLINFINITYSPACEPARTESTIMATE}
for the remaining term $\| \angupmu_{(1)} \|_{L_t^2 L_x^{\infty}(\widetilde{\mathcal{M}}^{(Int)})}$
on the LHS. Note that $\angupmu_{(1)}$ solves the Hodge-transport system \eqref{E:ANGDIVLDERIVATIVEANGMU1}--\eqref{E:ANGCURLLDERIVATIVEANGMU1},
where the inhomogeneous term $\mathfrak{I}_{(1)} - \overline{\mathfrak{I}_{(1)}}$
is defined by \eqref{E:FIRSTINHOMTERMMODIFIEDMASSASPECTEVOLUTIONEQUATION}.
From 
\eqref{E:STUVOLUMEFORMCOMPARISONWITHUNITROUNDMETRICVOLUMEFORM},
\eqref{E:MAIDENTITYTRANSPORTLEMMA}, 
and the initial condition \eqref{E:ANGUPMU1AND2CONETIPCONDITIONS}, 
we deduce the pointwise identity
\begin{align} \label{E:ANGUPMU1TRANSPORTID}
	\angupmu_{(1)}(t,u,\upomega)
	& 
	= 
	\volrat^{-\frac{1}{2}}(t,u,\upomega)
	\int_u^t
		\left[\volrat^{\frac{1}{2}} \mathscr{D}^{-1}(\mathfrak{I}_{(1)} - \overline{\mathfrak{I}_{(1)}},0) \right](\uptau,u,\upomega)
	\, d \uptau.
\end{align}
The term $\mathfrak{I}_{(1)} - \overline{\mathfrak{I}_{(1)}}$ on RHS~\eqref{E:ANGUPMU1TRANSPORTID}
is the same term appearing in \cite{qW2017}.
At the start of the last paragraph in the proof of \cite{qW2017}*{Proposition~6.4} 
(in which the author derived bounds for the first piece of quantity denoted by $\textbf{A}^{\dagger}$, which was split into two pieces),
the author derived estimates for RHS~\eqref{E:ANGUPMU1TRANSPORTID} showing
that $\| \angupmu_{(1)} \|_{L_t^2 L_x^{\infty}(\widetilde{\mathcal{M}}^{(Int)})} \lesssim \uplambda^{- \frac{1}{2} - 4 \upepsilon_0}$,
which is in fact slightly better than the bound stated in \eqref{E:ANGDUPSIGMAL2TLINFINITYSPACEPARTESTIMATE}.

Finally, we prove the estimate 
\eqref{E:ANGDUPSIGMAL2ULINFINITYCONEESTIMATE} for $\angupmu_{(2)}$ 
using the Hodge-transport system \eqref{E:ANGDIVLDERIVATIVEANGMU2}--\eqref{E:ANGCURLLDERIVATIVEANGMU2}.
We first define the following two scalar functions:
$\mathfrak{F} := \mbox{RHS}~\eqref{E:ANGDIVLDERIVATIVEANGMU2}$,
$\mathfrak{G} := \mbox{RHS}~\eqref{E:ANGCURLLDERIVATIVEANGMU2}$.
From 
\eqref{E:ANGDIVLDERIVATIVEANGMU2}--\eqref{E:ANGCURLLDERIVATIVEANGMU2},
\eqref{E:STUVOLUMEFORMCOMPARISONWITHUNITROUNDMETRICVOLUMEFORM},
\eqref{E:MAIDENTITYTRANSPORTLEMMA}, 
and the initial condition \eqref{E:ANGUPMU1AND2CONETIPCONDITIONS},
we deduce the pointwise identity
\begin{align} \label{E:ANGUPMU2TRANSPORTID}
	\angupmu_{(2)}(t,u,\upomega)
	& 
	= 
	\volrat^{-\frac{1}{2}}(t,u,\upomega)
	\int_u^t
		\left[\volrat^{\frac{1}{2}} \mathscr{D}^{-1}(\mathfrak{F},\mathfrak{G}) \right](\uptau,u,\upomega)
	\, d \uptau.
\end{align}
From \eqref{E:EUCLIDEANFORMMORREYONSTU} with $\leb := p$
and
\eqref{E:ONEFORMSTUCALDERONZYGMUNDHODGEESTIMATES},
we find that
$\| \mathscr{D}^{-1}(\mathfrak{F},\mathfrak{G}) \|_{L_{\upomega}^{\infty}(S_{t,u})}
\lesssim 
\| \rgeo (\mathfrak{F},\mathfrak{G}) \|_{L_{\upomega}^p(S_{t,u})}
$.
From this estimate, 
\eqref{E:ANGUPMU2TRANSPORTID}, 
\eqref{E:STUVOLUMEFORMCOMPARISONWITHUNITROUNDMETRICVOLUMEFORM},
and the simple bound $\rgeo(\uptau,u)/\rgeo(t,u) \lesssim 1$ for $\uptau \leq t$,
we deduce that
\begin{align}
\notag
\| \angupmu_{(2)} \|_{L_u^2 L_t^{\infty} L_{\upomega}^{\infty}(\widetilde{\mathcal{M}}^{(Int)})}
\lesssim 
\| \rgeo (\mathfrak{F},\mathfrak{G}) \|_{L_u^2 L_t^1 L_{\upomega}^p(\widetilde{\mathcal{M}}^{(Int)})}.
\end{align}
In \cite{qW2017}*{Equation~(6.37)} and the discussion below that equation, 
based on 
the bootstrap assumptions,
\eqref{E:RESCALEDBOOTBOUNDS},
\eqref{E:RGEOANDUBOUNDS},
and the already proven estimates 
\eqref{E:ACOUSTICALLT2LOMEGAPANDLDERIVATIVESALONGCONES},
\eqref{E:ASECONDACOUSTICALLTINFTYLOMEGAPALONGCONES},
\eqref{E:ONEANGULARDERIVATIVEOFTRCHIMODANDTRFREECHIL2INTIMEESTIMATES},
\eqref{E:IMPROVEDININTERIORL2INTIMELINFINITYINSPACECONNECTIONCOFFICIENTS},
\eqref{E:SPACETIMEL2INUANDTLPINOMEGAFORANGDSIGMAANDRGEOWEIGHTEDMODIFIEDMASSASPECTANDANGDMODIFIEDTORSION},
and \eqref{E:PRELIMINARYANGUPMUSPACETIMEBOUNDS},
the author gave arguments that imply that
$\| \rgeo (\mathfrak{F},\mathfrak{G}) \|_{L_u^2 L_t^1 L_{\upomega}^p(\widetilde{\mathcal{M}}^{(Int)})} \lesssim \uplambda^{-\frac{1}{2} - 4 \upepsilon_0}$
as desired,
except that the terms in $\mathfrak{I}_{(2)} - \overline{\mathfrak{I}_{(2)}}$
(i.e., the first term RHS~\eqref{E:ANGDIVLDERIVATIVEANGMU2})
generated by the two terms on RHS~\eqref{E:SECONDINHOMTERMMODIFIEDMASSASPECTEVOLUTIONEQUATION}
with the coefficient $\uplambda^{-1}$ were not present in \cite{qW2017}.
To handle these new terms, we use
\eqref{E:RGEOLAMBDAINVERSEONEDERIVATIVEOFMODFLUIDLU2LT1LOMEGAP}--\eqref{E:RGEOLAMBDAINVERSEQUADTERMLU2LT1LOMEGAP}.
We have therefore proved \eqref{E:ANGDUPSIGMAL2ULINFINITYCONEESTIMATE},
which completes the proof of Prop.\,\ref{P:MAINESTIMATESFOREIKONALFUNCTIONQUANTITIES}.

\section{Summary of the reductions of the proof of the Strichartz estimate of Theorem~\ref{T:FREQUENCYLOCALIZEDSTRICHARTZ}}
\label{S:REDUCTIONSOFSTRICHARTZ}
In this section, we outline how the Strichartz estimate of Theorem~\ref{T:FREQUENCYLOCALIZEDSTRICHARTZ}
follows as a consequence of the estimates for the eikonal function that we derived in Sect.\,\ref{S:ESTIMATESFOREIKONALFUNCTION}.
We only sketch the arguments since,
given the estimates that we derived in Sect.\,\ref{S:ESTIMATESFOREIKONALFUNCTION}, 
the proof of Theorem~\ref{T:FREQUENCYLOCALIZEDSTRICHARTZ} 
follows from the same arguments given in \cite{qW2017}.
For the reader's convenience, we note that the 
flow of the logic can be summarized as follows, although in 
Subsects.\,\ref{SS:RESCALEDVERSIONOFTHEOREMFREQUENCYLOCALIZEDSTRICHARTZ}--\ref{SS:DISCUSSIONOFPROOFOFPROPSPATIALLYLOCALIZEDREDUCTIONOFPROOFOFTHEOREMDECAYESTIMATE}, 
we will discuss the steps in the reverse order:
\begin{enumerate}
	\item Estimates for the eikonal function, connection coefficients, and conformal factor $\upsigma$ obtained in Sect.\,\ref{S:ESTIMATESFOREIKONALFUNCTION}
	\item $\implies$ Estimates for a conformal energy for solutions $\varphi$ to the linear wave equation $\square_{\gfour(\vec{\Psi})} \varphi = 0$ 
	\item $\implies$ Dispersive-type decay estimate for the linear wave equation solution $\varphi$
	\item $\implies$ Rescaled version of the desired Strichartz estimates 
	\item $\implies$ Theorem~\ref{T:FREQUENCYLOCALIZEDSTRICHARTZ}.
\end{enumerate}

We remind the reader that the completion of the proof of Theorem~\ref{T:FREQUENCYLOCALIZEDSTRICHARTZ}
closes the bootstrap argument initiated in Subsect.\,\ref{SS:BOOT},
thereby justifying the estimate \eqref{E:MIXEDSPACETIMEESTIMATENEEDEDFORMAINTHEOREM}
and completing the proof of Theorem~\ref{T:MAINTHEOREMROUGHVERSION}.

\subsection{Rescaled version of Theorem~\ref{T:FREQUENCYLOCALIZEDSTRICHARTZ}}
\label{SS:RESCALEDVERSIONOFTHEOREMFREQUENCYLOCALIZEDSTRICHARTZ}
From standard scaling considerations, one can easily show that Theorem~\ref{T:FREQUENCYLOCALIZEDSTRICHARTZ}
(where in \eqref{E:LINEARWAVEFORFREQUENCYLOCALIZEDSTRICHARTZ}, $\vec{\Psi}$ denotes the non-rescaled wave variables)
would follow\footnote{More precisely, the analog of Theorem~\ref{T:FREQUENCYLOCALIZEDSTRICHARTZ} in \cite{qW2017},
namely \cite{qW2017}*{Theorem~3.2}, was stated only in the special case $\uptau = t_k$, where $\uptau$ is
as in the statement of Theorem~\ref{T:FREQUENCYLOCALIZEDSTRICHARTZ}. However, the case of
a general $\uptau \in [t_k,t_{k+1}]$ follows from the same arguments.} 
from a rescaled version of it, which we state as Theorem~\ref{T:RESCALEDVERSIONOFTHEOREMFREQUENCYLOCALIZEDSTRICHARTZ}.
Here we do not provide the simple proof that Theorem~\ref{T:FREQUENCYLOCALIZEDSTRICHARTZ} follows from
Theorem~\ref{T:RESCALEDVERSIONOFTHEOREMFREQUENCYLOCALIZEDSTRICHARTZ}; we refer readers to
\cite{qW2017}*{Section~3.1} and \cite{qW2017}*{Theorem~3.3} for further discussion.

\begin{theorem}[Rescaled version of Theorem~\ref{T:FREQUENCYLOCALIZEDSTRICHARTZ}]
	\label{T:RESCALEDVERSIONOFTHEOREMFREQUENCYLOCALIZEDSTRICHARTZ}
	Let $P$ denote the Littlewood--Paley projection onto frequencies $\xi$ with $\frac{1}{2} \leq |\xi| \leq 2$.
	Under the assumptions of Subsect.\,\ref{SS:ASSUMPTIONS},
	there is a $\Lambda_0 > 0$ such that for every $\uplambda \geq \Lambda_0$,
	every $q > 2$ that is sufficiently close to $2$,
	and every solution $\varphi$ to the homogeneous linear wave equation 
	\begin{align} \label{E:BOUNDEDNESSOFCONFORMALENERGY}
	\square_{\gfour(\vec{\Psi})} 
	\varphi & = 0
\end{align}
	on the slab $[0,\RescaledTboot] \times \mathbb{R}^3$,
	the following mixed space-time estimate holds:
	\begin{align} \label{E:RESCALEDVERSIONOFTHEOREMFREQUENCYLOCALIZEDSTRICHARTZ}
		\| P \pmb{\partial} \varphi \|_{L^q([0,\RescaledTboot]) L_x^{\infty}}
		& \lesssim \| \pmb{\partial} \varphi \|_{L^2(\Sigma_0)}.
	\end{align}
	We clarify that in \eqref{E:BOUNDEDNESSOFCONFORMALENERGY}, 
	the argument ``$\vec{\Psi}$'' in $\gfour(\vec{\Psi})$ denotes the rescaled solution, as
	in Subsect.\,\ref{SS:RESCALEDSOLUTION} and Subsect.\,\ref{SS:NOMORELAMBDA}.
\end{theorem}

\subsection{Dispersive-type decay estimate}
\label{SS:DECAYESTIMATE}
As we discussed in Subsect.\,\ref{SS:RESCALEDVERSIONOFTHEOREMFREQUENCYLOCALIZEDSTRICHARTZ},
to prove Theorem~\ref{T:FREQUENCYLOCALIZEDSTRICHARTZ},
it suffices to prove Theorem~\ref{T:RESCALEDVERSIONOFTHEOREMFREQUENCYLOCALIZEDSTRICHARTZ}.
Theorem~\ref{T:RESCALEDVERSIONOFTHEOREMFREQUENCYLOCALIZEDSTRICHARTZ}
can be shown, via a technical-but-by-now-standard $\mathcal{T} \mathcal{T}^*$ argument,
to follow as a consequence of the dispersive-type decay estimate
provided by Theorem~\ref{T:DECAYESTIMATE}. 
See \cite{qW2017}*{Appendix~B}
for a proof that
Theorem~\ref{T:RESCALEDVERSIONOFTHEOREMFREQUENCYLOCALIZEDSTRICHARTZ}
follows from
Theorem~\ref{T:DECAYESTIMATE}.
We remark that the proof given in
\cite{qW2017}*{Appendix~B} goes through almost verbatim,
with only minor changes needed to handle the fact that
the future-directed unit normal to $\Sigma_t$ is $\Transport$ in the present article
(and thus the $\Transport$-differentiation occurs on LHS~\eqref{E:DECAYESTIMATE}),
while in \cite{qW2017}, the future-directed unit normal to $\Sigma_t$ is $\partial_t$.

We now state Theorem~\ref{T:DECAYESTIMATE}. 
In Subsect.\,\ref{SS:SPATIALLYLOCALIZEDREDUCTIONOFPROOFOFTHEOREMDECAYESTIMATE},
we will discuss its proof.\footnote{The presence of up to three derivatives of $\varphi$ on 
RHS~\eqref{E:DECAYESTIMATE} is not problematic because in practice, 
the estimate \eqref{E:DECAYESTIMATE} is only used on functions
supported near unit frequencies in Fourier space (and thus the functions' 
derivatives can be controlled in terms of the function itself, by Bernstein's inequality).
See \cite{qW2017}*{Appendix B}, especially the first estimate
on page 105.}

\begin{theorem}[Dispersive-type decay estimate]
	\label{T:DECAYESTIMATE}
	Let $P$ denote the Littlewood-Paley projection onto frequencies $\xi$ with $\frac{1}{2} \leq |\xi| \leq 2$.
	Under the conventions of Subsect.\,\ref{SS:NOMORELAMBDA}
	and the assumptions of Subsect.\,\ref{SS:ASSUMPTIONS},
	there exists a large $\Lambda_0 > 0$ 
	and a function $d(t) \geq 0$
	such that if $\uplambda \geq \Lambda_0$
	and if $q > 2$ is sufficiently close to $2$,
	then 
	\begin{align} \label{E:LEBESGUENORMOFDECAYERRORTERM}
		\| d \|_{L^{\frac{q}{2}}([0,\RescaledTboot])}
		& \lesssim 1,
	\end{align}
	and for every solution $\varphi$ to the homogeneous linear wave equation \eqref{E:BOUNDEDNESSOFCONFORMALENERGY}
	on the slab $[0,\RescaledTboot] \times \mathbb{R}^3$,
	the following decay estimate holds for $t \in [0,\RescaledTboot]$:
	\begin{align} \label{E:DECAYESTIMATE}
		\| P \Transport \varphi \|_{L^{\infty}(\Sigma_t)}
		& \lesssim 
		\left\lbrace
			\frac{1}{(1 + t)^{\frac{2}{q}}}
			+
			d(t)
		\right\rbrace
		\left\lbrace
			\sum_{m=0}^3 \| \partial^m \varphi \|_{L^1(\Sigma_0)}
			+
			\sum_{m=0}^2 \| \partial^m \partial_t \varphi \|_{L^1(\Sigma_0)}
		\right\rbrace.
	\end{align}
\end{theorem}

\subsection{Reduction of the proof of Theorem~\ref{T:DECAYESTIMATE} to the case of compactly supported data}
\label{SS:SPATIALLYLOCALIZEDREDUCTIONOFPROOFOFTHEOREMDECAYESTIMATE}
It is convenient to reduce the proof of Theorem~\ref{T:DECAYESTIMATE}
to a spatially localized version in which the $L^1$ norms on the RHSs 
of the estimates are replaced with terms involving $L^2$ norms, which 
are more natural (in view of their connection to energy estimates).
More precisely, the same arguments given in \cite{qW2017}*{Section~4}
yield that Theorem~\ref{T:DECAYESTIMATE} follows as a consequence of Prop.\,\ref{P:SPATIALLYLOCALIZEDREDUCTIONOFPROOFOFTHEOREMDECAYESTIMATE},
which is an analog of \cite{qW2017}*{Proposition~4.1},
and Lemma~\ref{L:STANDARDENERGYESTIMATEFORLINEARWAVEQUATIONWITHRESCALEDBACKGROUND}, 
which is an analog of \cite{qW2017}*{Lemma~4.2}.
We will discuss the proof of Prop.\,\ref{P:SPATIALLYLOCALIZEDREDUCTIONOFPROOFOFTHEOREMDECAYESTIMATE}
in Subsect.\,\ref{SS:DISCUSSIONOFPROOFOFPROPSPATIALLYLOCALIZEDREDUCTIONOFPROOFOFTHEOREMDECAYESTIMATE},
while we provide the simple proof of Lemma~\ref{L:STANDARDENERGYESTIMATEFORLINEARWAVEQUATIONWITHRESCALEDBACKGROUND}
in this subsection.

\begin{proposition}[Spatially localized version of Theorem~\ref{T:DECAYESTIMATE}]
\label{P:SPATIALLYLOCALIZEDREDUCTIONOFPROOFOFTHEOREMDECAYESTIMATE}
Let $R > 0$ be as in Subsect.\,\ref{SS:RESCALEDSOLUTION},
fix any\footnote{As we highlighted in Remark~\ref{R:REMARKSONRESCALING}, the hypersurface that we denote by ``$\Sigma_0$''
in this proposition
corresponds to the hypersurface that we denoted by 
``$\Sigma_{t_k}$'' in Sects.\,\ref{S:DATAANDBOOTSTRAPASSUMPTION}--\ref{S:ELLIPTICESTIMATESINHOLDERSPACES}.
Similar remarks apply for the other constant-time hypersurfaces appearing in this proposition.
\label{FN:CAREFULABOUTHYPERSURFACES}} 
${\bf{z}} \in \Sigma_0$,
and let $\Tranchar_{\bf{z}}(1)$ be the unique point on the cone-tip axis in $\Sigma_1$ 
(see Subsubsect.\,\ref{SSS:EIKONALINTERIOR} for the definition of the cone-tip axis).
Let $P$ denote the Littlewood--Paley projection onto frequencies $\xi$ with $\frac{1}{2} \leq |\xi| \leq 2$.
	Under the assumptions of Subsect.\,\ref{SS:ASSUMPTIONS},
	there exists a large $\Lambda_0 > 0$ 
	and a function $d(t) \geq 0$
	such that if $\uplambda \geq \Lambda_0$
	and if $q > 2$ is sufficiently close to $2$,
	then 
	\begin{align} \label{E:SPATIALLYLOCALIZEDLEBESGUENORMOFDECAYERRORTERM}
		\| d \|_{L^{\frac{q}{2}}([0,\RescaledTboot])}
		& \lesssim 1,
	\end{align}
	and for every solution $\varphi$ to the homogeneous linear wave equation \eqref{E:BOUNDEDNESSOFCONFORMALENERGY}
	on the slab $[0,\RescaledTboot] \times \mathbb{R}^3$
	whose data on $\Sigma_1$
	are supported in the Euclidean ball $B_R(\Tranchar_{\bf{z}}(1))$ of radius $R$ centered at $\Tranchar_{\bf{z}}(1)$,
	the following decay estimate holds for $t \in [1,\RescaledTboot]$:
	\begin{align} \label{E:SPATIALLYLOCALIZEDDECAYESTIMATE}
		\| P \Transport \varphi \|_{L^{\infty}(\Sigma_t)}
		& \lesssim 
		\left\lbrace
			\frac{1}{(1 + |t-1|)^{\frac{2}{q}}}
			+
			d(t)
		\right\rbrace
		\left\lbrace
			\| \pmb{\partial} \varphi \|_{L^2(\Sigma_1)}
			+
			\| \varphi \|_{L^2(\Sigma_1)}
		\right\rbrace.
	\end{align}

\end{proposition}

\begin{remark}[$\varphi$ vanishes in {$\left([1,\RescaledTboot] \times \mathbb{R}^3 \right) \backslash \widetilde{\mathcal{M}}_1^{(Int)}$}]
	\label{R:PHIVANISHESINM1INTCOMPLEMENT}
	From the definition \eqref{E:TRUNCATEDSOLIDCONE} of $\widetilde{\mathcal{M}}_1^{(Int)}$,
	\eqref{E:INTERIORANDEXTERIOREGIONSAREUNIONSOFSPHERES},
	\eqref{E:SIGMATREGIONINTERIORISUNIONSOFSPHERES},
	\eqref{E:ALONGSIAMGA1EUCLIDEANBALLCONTAINEDINMETRICBALL},
	and standard domain of dependence considerations,
	it follows that the solution $\varphi$ from Prop.\,\ref{P:SPATIALLYLOCALIZEDREDUCTIONOFPROOFOFTHEOREMDECAYESTIMATE}
	satisfies $\varphi \equiv 0$ in 
	$\left([1,\RescaledTboot] \times \mathbb{R}^3 \right) \backslash \widetilde{\mathcal{M}}_1^{(Int)}$.
\end{remark}

\begin{lemma}[Standard energy estimate for the wave equation]
\label{L:STANDARDENERGYESTIMATEFORLINEARWAVEQUATIONWITHRESCALEDBACKGROUND}
Under the bootstrap assumption \eqref{E:RESCALEDSTRICHARTZ} for the first term on the LHS,
there exists a large $\Lambda_0 > 0$ 
such that if $\uplambda \geq \Lambda_0$,
then solutions $\varphi$ to the homogeneous linear wave equation \eqref{E:BOUNDEDNESSOFCONFORMALENERGY}
(where in \eqref{E:BOUNDEDNESSOFCONFORMALENERGY}, $\gfour = \gfour(\vec{\Psi})$, with $\vec{\Psi}$ the rescaled solution)
verify the following estimate for $t \in [0,\RescaledTboot]$:
\begin{align} \label{E:STANDARDENERGYESTIMATEFORLINEARWAVEQUATIONWITHRESCALEDBACKGROUND}
	\| \pmb{\partial} \varphi \|_{L^2(\Sigma_t)}
	& \lesssim \| \pmb{\partial} \varphi \|_{L^2(\Sigma_0)}.
\end{align}
Moreover, for $0 \leq t \leq 1$, we have
\begin{align} \label{E:L2NORMSTANDARDENERGYESTIMATEFORLINEARWAVEQUATIONWITHRESCALEDBACKGROUND}
	\| \varphi \|_{L^2(\Sigma_t)}
	& \lesssim \| \pmb{\partial} \varphi \|_{L^2(\Sigma_0)}
		+
		\| \varphi \|_{L^2(\Sigma_0)}.
\end{align}

\end{lemma}

\begin{proof}
	Reasoning as in our proof of \eqref{E:BASICENERGYINEQUALITYFORWAVEEQUATIONS},
	but omitting the $\varphi^2$ term in the analog of the 
	energy \eqref{E:BASICENERGYDEF}
	and the energy identity \eqref{E:WAVEEQUATIONENERGYIDUSEDINPROOF},
	we find that
	\[
	\| \pmb{\partial} \varphi \|_{L^2(\Sigma_t)}^2
	\lesssim
	\| \pmb{\partial} \varphi \|_{L^2(\Sigma_0)}^2
	+
				\int_0^t
						\|
							\pmb{\partial} \vec{\Psi} 
						\|_{L^{\infty}(\Sigma_{\uptau})}
						\| \pmb{\partial} \varphi \|_{L^2(\Sigma_{\uptau})}^2
				\, d \uptau.
	\]
	\eqref{E:STANDARDENERGYESTIMATEFORLINEARWAVEQUATIONWITHRESCALEDBACKGROUND} now follows from this estimate,
	Gr\"{o}nwall's inequality, 
	and the estimate
	$\| \pmb{\partial} \vec{\Psi} \|_{L^1([0,\RescaledTboot]) L_x^{\infty}} \lesssim \uplambda^{-8 \upepsilon_0} \leq 1$,
	which is a simple consequence of
	\eqref{E:RESCALEDBOOTBOUNDS}
	and
	\eqref{E:RESCALEDSTRICHARTZ}.
	
	\eqref{E:L2NORMSTANDARDENERGYESTIMATEFORLINEARWAVEQUATIONWITHRESCALEDBACKGROUND}
	then follows from
	\eqref{E:STANDARDENERGYESTIMATEFORLINEARWAVEQUATIONWITHRESCALEDBACKGROUND} and the fundamental theorem of calculus.
\end{proof}

\subsection{Mild growth rate for a conformal energy}
\label{SS:MILDGROWTHCONFORMALENERGY}
The proof of Prop.~\ref{P:SPATIALLYLOCALIZEDREDUCTIONOFPROOFOFTHEOREMDECAYESTIMATE}
fundamentally relies on deriving estimates for a conformal energy,
which we define in this subsection.
We stress that our definition coincides with the definition of the conformal energy
given in \cite{qW2017}*{Definition~4.4}.

\subsubsection{Definition of the conformal energy}
\label{SSS:DEFINITIONOFCONFORMAL}
We start by fixing two smooth, non-negative cut-off functions of $(t,u)$, denoted by $\weight$ and $\uweight$
and satisfying $0 \leq \weight(t,u) \leq 1$, $0 \leq \uweight(t,u) \leq 1$, such that the following
properties hold for $t > 0$:
\begin{subequations}
\begin{align} \label{E:CONFORMALENERGYCUTOFFS}
	\weight(t,u)
	& = 
	\begin{cases}
		1 & \mbox{if } \frac{u}{t} \in [0,1/2],
			\\
		0 & \mbox{if } \frac{u}{t} \in (-\infty,-1/4] \cup [3/4,1], 
	\end{cases}
	& 
	\uweight(t,u)
	& = 
	\begin{cases}
		1 & \mbox{if } \frac{u}{t} \in [0,1],
			\\
		0 & \mbox{if } \frac{u}{t} \in (-\infty,-1/4], 
	\end{cases}
		\\
	\weight(t,u)
	& = \uweight(t,u)
	&& 
	\mbox{if }
	t \in [1,\RescaledTboot] \mbox{ and } \frac{u}{t} \in [-1/4,0].
	\label{E:AGREEMENTCONFORMALENERGYCUTOFFS}
\end{align}
\end{subequations}
See Fig.\,\ref{F:REGIONSCUTOFF} for a schematic depiction of the regions in the case ${\bf{z}} := 0$,
where for convenience, we have suppressed the ``quasilinear nature'' of the geometry by
depicting it as flat.

\begin{figure}[h!]
\centering
\includegraphics[scale=.4]{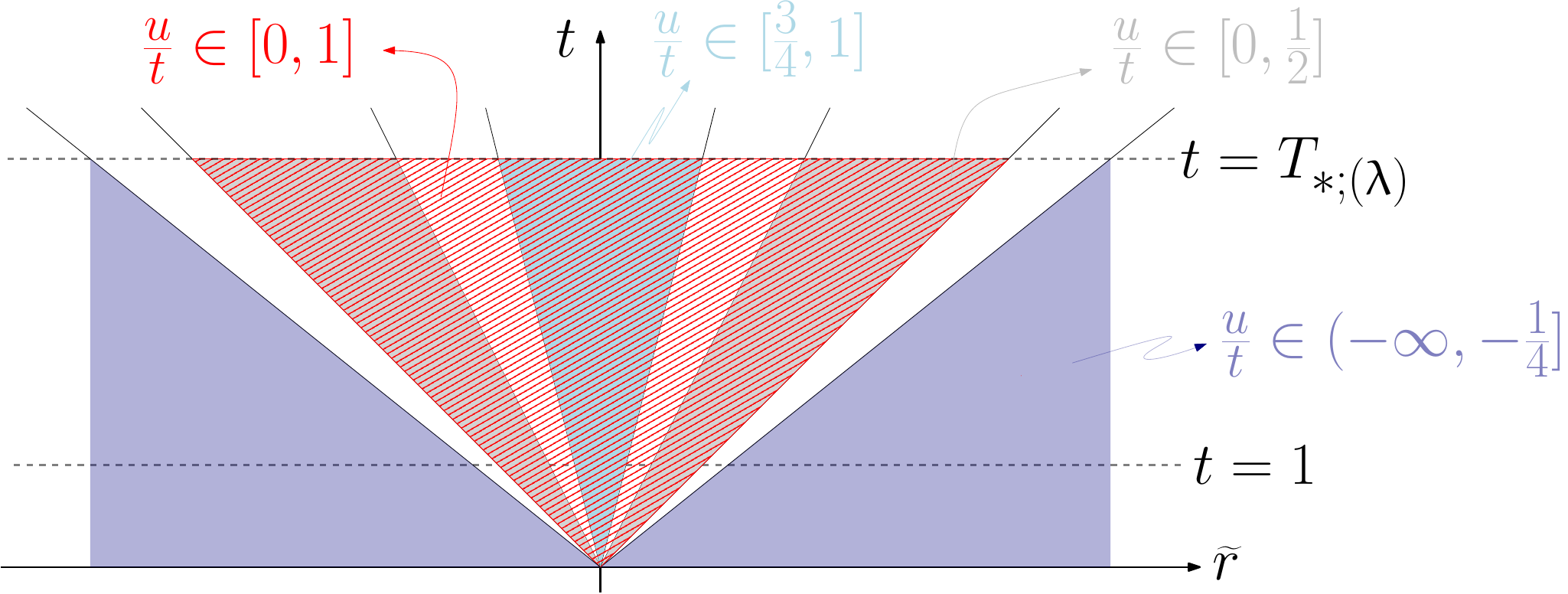}
\caption{Schematic illustration of the regions appearing in definition \eqref{E:CONFORMALENERGYCUTOFFS} in the case ${\bf{z}} := 0$.}
\label{F:REGIONSCUTOFF}
\end{figure}

\begin{definition}[Conformal energy]
\label{D:CONFORMALENERGY}
For scalar functions $\varphi$ 
that vanish outside of $\widetilde{\mathcal{M}}_1^{(Int)}$ 
(see definition \eqref{E:TRUNCATEDSOLIDCONE} and Remark~\ref{R:PHIVANISHESINM1INTCOMPLEMENT}),
we define the conformal energy $\mathscr{C}[\varphi]$ as follows:
\begin{align} \label{E:CONFORMALENERGY}
	\mathscr{C}[\varphi](t)
	& := 
		\int_{\widetilde{\Sigma}_t^{(Int)}}
				(\uweight - \weight) t^2
				\left\lbrace
					|\Dfour \varphi|^2
					+
					|\rgeo^{-1} \varphi|^2
				\right\rbrace
			\, d \varpi_g
			+
			\int_{\widetilde{\Sigma}_t^{(Int)}}
				\weight
				\left\lbrace
					|\rgeo \Dfour_{\Lunit} \varphi|^2
					+
					|\rgeo \angD \varphi|_{\gsphere}^2
					+
					|\varphi|^2
				\right\rbrace
			\, d \varpi_g.
\end{align}
\end{definition}

\subsubsection{The precise eikonal function and connection coefficient estimates needed for the proof of the conformal energy estimate}
\label{SSS:EIKONAFUNCTIONESTIMATESFORPROOFOFCONFORMALENERGYESTIMATE}
The following corollary is a routine consequence of Prop.\,\ref{P:MAINESTIMATESFOREIKONALFUNCTIONQUANTITIES}.
It provides all of the estimates for the eikonal function and connection coefficients
that are needed to prove Theorem~\ref{T:BOUNDEDNESSOFCONFORMALENERGY},
which in turn
provides the main estimates needed to prove Prop.\,\ref{P:SPATIALLYLOCALIZEDREDUCTIONOFPROOFOFTHEOREMDECAYESTIMATE}.
Some statements in the corollary are redundant in the sense that they already appeared in
Prop.\,\ref{P:MAINESTIMATESFOREIKONALFUNCTIONQUANTITIES}. For the reader's convenience,
we have allowed for redundancies; having all needed estimates
in the same corollary will facilitate our discussion of the proof of
Theorem~\ref{T:BOUNDEDNESSOFCONFORMALENERGY}.

\begin{corollary}[The precise estimates needed for the proof of the conformal energy estimate]
	\label{C:EIKONAFUNCTIONESTIMATESFORPROOFOFCONFORMALENERGYESTIMATE}
Let
\begin{align} \label{E:LOTSOFEIKONALFUNCTIONQUANTITIES} 
	\mathbf{A}
	& :=
	\gensmoothfunction_{(\vec{\Lunit})} \cdot
	\left(
			\mytr_{\congsphere} \widetilde{\upchi}^{(Small)}, 
			\hat{\upchi}, 
			\mytr_{\gsphere} \upchi - \frac{2}{\rgeo}, 
			\upzeta, 
			 \pmb{\partial} \vec{\Psi},
			\vec{\VortVort},
			\DivGradEnt,
			\frac{\nulllapse^{-1} - 1}{\rgeo},
			k,
			\Lunit \ln \nulllapse,
			\hat{\spheresecondfund},
			\spheresecondfund - \frac{2}{\rgeo}
		\right),
\end{align}
where $\gensmoothfunction_{(\vec{\Lunit})}$ is any smooth function of the type described in Subsubsect.\,\ref{SSS:ADDITIONALSCHEMATIC}.

Under the assumptions of Subsect.\,\ref{SS:ASSUMPTIONS}, the following estimates hold:
\begin{subequations}
\begin{align}
	\RescaledTboot & \leq \uplambda^{1-8 \upepsilon_0} \Tboot,
	&
	0 \leq \rgeo 
	& < 2 \RescaledTboot,
	&&
	\label{E:RESTATEMENTTANDRBOUNDS}	\\
	\|
		\nulllapse - 1
	\|_{L^{\infty}(\widetilde{\mathcal{M}})}
	& \lesssim \uplambda^{- \upepsilon_0} 
	\leq \frac{1}{4},
	&
	\volrat
	&
	\approx \rgeo^2,
	&
	\rgeo \mytr_{\congsphere} \widetilde{\upchi}
	&
	\approx 1.
	\label{E:RESTATEMENTNULLLAPSEANDVOLFORMBOUNDS} 
\end{align}
\end{subequations}

Moreover, we have the following estimates,\footnote{Our estimates 
\eqref{E:AIMPROVEDINTERIORREGEIONLT2LXINFINITYBONDS}
and
\eqref{E:SECONDCRUCIALPRODUCTESTIMATE}
feature the power $- 1/2 - 3 \upepsilon_0$ on the RHS, as opposed to the power
$- 1/2 - 4 \upepsilon_0$ that appeared in the analogous estimates of \cite{qW2017}.
This minor change has no substantial effect on the main results.
\label{FN:SMALLCORRECTIONTOLAMBDAPOWER}}
where the norms are defined in Subsects.\,\ref{SS:GEOMETRICNORMS} and \ref{SS:HOLDERNORMSINGEOMETRICANGULARVARIABLES},
the corresponding spacetime regions such as $\widetilde{\mathcal{C}}_u \subset \widetilde{\mathcal{M}}$
are defined in Subsect.\,\ref{SS:GEOMETRICSPACETIMESUBSETS} (see especially \eqref{E:TRUNCATEDSIGMASANDCONES}),
and $p$ is as in \eqref{E:BOUNDSONLEBESGUEEXPONENTP}:
\begin{subequations}
\begin{align}
	\| 
		\mathbf{A}
	\|_{L_t^2 L_x^{\infty}(\widetilde{\mathcal{M}}^{(Int)})}
	& \lesssim \uplambda^{- 1/2 - 3 \upepsilon_0},
		\label{E:AIMPROVEDINTERIORREGEIONLT2LXINFINITYBONDS} 
\end{align}
\begin{align}
	\|
		\rgeo
		(\angD,\angprojDarg{\Lunit}) \mathbf{A}
	\|_{L_t^2 L_{\upomega}^p(\widetilde{\mathcal{C}}_u)},
		\,
	\|
		\mathbf{A}
	\|_{L_t^2 L_{\upomega}^p(\widetilde{\mathcal{C}}_u)},
		\,
	\|
		\rgeo^{1/2} \mathbf{A} 
	\|_{L_t^{\infty} L_{\upomega}^{2p}(\widetilde{\mathcal{C}}_u)}
	\lesssim \uplambda^{-1/2},
	\label{E:LAMBDATOMINUSONEHALFAANDUNDERLINEABOUNDS} 
\end{align}
\end{subequations}

	\begin{subequations}
	\begin{align}
		\| \rgeo^{\frac{1}{2}} \Lunit \upsigma \|_{L_t^{\infty} L_{\upomega}^{2p}(\widetilde{\mathcal{C}}_u)},
			\,
		\| \rgeo^{\frac{1}{2} - \frac{2}{p}} \angD \upsigma \|_{L_{\gsphere}^p L_t^{\infty}(\widetilde{\mathcal{C}}_u)},
			\,
		\| \rgeo^{\frac{1}{2}} \angD \upsigma \|_{L_{\upomega}^p L_t^{\infty}(\widetilde{\mathcal{C}}_u)},
			\,
		\| \angD \upsigma \|_{L_t^2 L_{\upomega}^p(\widetilde{\mathcal{C}}_u)}
		& \lesssim \uplambda^{-1/2},
			&&
			\mbox{if } \widetilde{\mathcal{C}}_u \subset \widetilde{\mathcal{M}}^{(Int)}, 
			\label{E:RESTATEMENTCONEFIRSTDERIVATIVEBOUNDSFORCONFORMALFACTOR} \\
		\| \upsigma \|_{L^{\infty}(\widetilde{\mathcal{M}}^{(Int)})}
		& \lesssim \uplambda^{-8 \upepsilon_0},
			 &&
			\label{E:RESTATEMENTLINFTYBOUNDSFORCONFORMALFACTOR}
				\\
		\| \rgeo^{-1/2} \upsigma \|_{L^{\infty}(\widetilde{\mathcal{M}}^{(Int)})}
		& \lesssim \uplambda^{-\frac{1}{2}-4 \upepsilon_0},
			\label{E:RESTATEMENTRWEIGHTEDLINFTYBOUNDSFORCONFORMALFACTOR}
	\end{align}
	\end{subequations}

\begin{align} \label{E:ADDITIONALCONNECTIONCOEFFICIENTSLAMBDATOMINUSONEHALFAANDUNDERLINEABOUNDS}
	\left\|
		\rgeo^{1/2}
		\left(
		\frac{\nulllapse^{-1} - 1}{\rgeo},\mytr_{\gsphere} \upchi - \frac{2}{\rgeo},\mytr_{\gsphere} \underline{\upchi} 
		+ 
		\frac{2}{\rgeo}, k_{\spherenormal \spherenormal}
		\right) 
	\right\|_{L_t^{\infty} L_{\upomega}^2(\widetilde{\mathcal{C}}_u)} 
	& \lesssim \uplambda^{-1/2},
\end{align}

\begin{subequations}
\begin{align}
\| \angupmu_{(2)} \|_{L_u^2 L_t^{\infty} L_{\upomega}^{\infty}(\widetilde{\mathcal{M}}^{(Int)})}
\cdot
\|
	\rgeo^{1/2} \angD \upsigma
\|_{L_u^{\infty} L_t^{\infty} L_{\upomega}^p(\widetilde{\mathcal{M}}^{(Int)})}
& 
\lesssim \uplambda^{- 1 - 4 \upepsilon_0},
	\label{E:FIRSTCRUCIALPRODUCTESTIMATE} \\
\| (\upzeta,\widetilde{\upzeta} - \angupmu,\angupmu_{(1)}) \|_{L_t^2 L_x^{\infty}(\widetilde{\mathcal{M}}^{(Int)})}
\cdot
\|
	\rgeo^{1/2} \angD \upsigma
\|_{L_t^{\infty} L_u^{\infty} L_{\upomega}^p(\widetilde{\mathcal{M}}^{(Int)})}
& 
\lesssim \uplambda^{- 1 - 3 \upepsilon_0},
	\label{E:SECONDCRUCIALPRODUCTESTIMATE} \\
\|
	\rgeo^{3/2} 
	(\check{\upmu},
	\mytr_{\gsphere} \upchi \cdot \Chfour_{\uLunit}
	)
\|_{L_u^2 L_t^{\infty} L_{\upomega}^p(\widetilde{\mathcal{M}}^{(Int)})}
& \lesssim
	\uplambda^{- 4 \upepsilon_0},
		\label{E:CRUCIALCHECKMUANDRELATEDTERMSESTIMATE} \\
\|
	\rgeo \angD \upsigma
\|_{L_u^{\infty} L_t^{\infty} L_{\upomega}^p(\widetilde{\mathcal{M}}^{(Int)})}
& \lesssim 
	\uplambda^{- 4 \upepsilon_0},
		\label{E:CRUCIALRGEOANGDUPSIGMAESTIMATE} 
			\\
\left\|
	\rgeo^{-1/2}
	\left\lbrace
		\Lunit 
		\left(
			\frac{1}{2}
			\mytr_{\congsphere} \widetilde{\upchi}
			\volrat
		\right)
		-
		\frac{1}{4}
		\left(
			\mytr_{\gsphere} \upchi
		\right)^2
		\volrat
		+
		\frac{1}{2}
		\left\lbrace
			\Lunit \ln \nulllapse
		\right\rbrace
		\mytr_{\congsphere} \widetilde{\upchi}
		\volrat
		-
		|\angD \upsigma|_{\gsphere}^2
		\volrat
	\right\rbrace
\right\|_{L_u^{\infty} L_t^{\infty} L_{\upomega}^{\frac{p}{2}}(\widetilde{\mathcal{M}}^{(Int)})}
& \lesssim
\uplambda^{- \frac{1}{2}}.
\label{E:CRUCIALREORMALIZEDTRCHIESTIMATEAPPEARINGINCONFORMALENERGY}
\end{align}
\end{subequations}

\end{corollary}

\begin{proof}
The bootstrap assumptions imply that $\| \gensmoothfunction_{(\vec{\Lunit})} \|_{L^{\infty}(\widetilde{\mathcal{M}}^{(Int)})} \lesssim 1$;
thus, we can ignore $\gensmoothfunction_{(\vec{\Lunit})}$ throughout the rest of this proof.
The estimates \eqref{E:RESTATEMENTTANDRBOUNDS}, 
\eqref{E:RESTATEMENTNULLLAPSEANDVOLFORMBOUNDS},
\eqref{E:AIMPROVEDINTERIORREGEIONLT2LXINFINITYBONDS},
\eqref{E:LAMBDATOMINUSONEHALFAANDUNDERLINEABOUNDS},
\eqref{E:RESTATEMENTCONEFIRSTDERIVATIVEBOUNDSFORCONFORMALFACTOR},
\eqref{E:RESTATEMENTLINFTYBOUNDSFORCONFORMALFACTOR},
and \eqref{E:RESTATEMENTRWEIGHTEDLINFTYBOUNDSFORCONFORMALFACTOR}
are restatements of 
\eqref{E:RESCALEDBOOTBOUNDS},
\eqref{E:RGEOANDUBOUNDS},
\eqref{E:FREQUENCYSQUARESUMMEDSMOOTHFUNCTIONTIMESBASICVARIABLESSIMPLECONSEQUENCEOFRESCALEDSTRICHARTZ},
and of estimates derived in Props.\,\ref{P:MAINESTIMATESFOREIKONALFUNCTIONQUANTITIES} and \ref{P:ESTIMATESFORFLUIDVARIABLES},
combined with the schematic relations
$\Lunit \ln \nulllapse = \gensmoothfunction_{(\vec{\Lunit})} \cdot \pmb{\partial} \vec{\Psi}$,
$\Lunit \upsigma = \gensmoothfunction_{(\vec{\Lunit})} \cdot \pmb{\partial} \vec{\Psi}$,
$k = \gensmoothfunction(\vec{\Psi}) \cdot \pmb{\partial} \vec{\Psi}$,
$\hat{\spheresecondfund} = \hat{\upchi} + \gensmoothfunction_{(\vec{\Lunit})} \cdot \pmb{\partial} \vec{\Psi}$,
and
$\spheresecondfund - \frac{2}{\rgeo} 
= 
\mytr_{\congsphere} \widetilde{\upchi}^{(Small)} 
+ \gensmoothfunction_{(\vec{\Lunit})} \cdot \pmb{\partial} \vec{\Psi}$
(see \eqref{E:SECONDFORMLIEDIFFERENTIATIONDEF},
\eqref{E:SIGMAEVOLUTION},
\eqref{E:EVOLUTIONNULLAPSEUSEEULER},
\eqref{E:CONNECTIONCOEFFICIENT},
and \eqref{E:MODTRICHISMALL}).
Here we clarify that although the proof of the estimate
\eqref{E:IMPROVEDININTERIORL2INTIMELINFINITYINSPACECONNECTIONCOFFICIENTS}
relied on Schauder-type estimates for $\hat{\upchi}$ that forced us to obtain control of
$	\| 
		(
			\mytr_{\congsphere} \widetilde{\upchi}^{(Small)}, \mytr_{\gsphere} \upchi - \frac{2}{\rgeo}, \hat{\upchi}
		)
	\|_{L_t^2 L_u^{\infty} C_{\upomega}^{0,\updelta_0}(\widetilde{\mathcal{M}}^{(Int)})}
$,
we have stated the estimate \eqref{E:AIMPROVEDINTERIORREGEIONLT2LXINFINITYBONDS} in terms of the weaker
norm $\| \cdot \|_{L_t^2 L_x^{\infty}(\widetilde{\mathcal{M}}^{(Int)})}$;
control of this weaker norm is sufficient for the proof of Theorem~\ref{T:BOUNDEDNESSOFCONFORMALENERGY}.

\eqref{E:ADDITIONALCONNECTIONCOEFFICIENTSLAMBDATOMINUSONEHALFAANDUNDERLINEABOUNDS}
follows from
\eqref{E:LAMBDATOMINUSONEHALFAANDUNDERLINEABOUNDS}
and the schematic relations
$\mytr_{\gsphere} \upchi - \frac{2}{\rgeo}
=
\mytr_{\congsphere} \widetilde{\upchi}^{(Small)} 
+ \gensmoothfunction_{(\vec{\Lunit})} \cdot \pmb{\partial} \vec{\Psi}$,
$\mytr_{\gsphere} \underline{\upchi} + \frac{2}{\rgeo}
=
- \mytr_{\congsphere} \widetilde{\upchi}^{(Small)} 
+ \gensmoothfunction_{(\vec{\Lunit})} \cdot \pmb{\partial} \vec{\Psi}$, 
and
$k_{\spherenormal \spherenormal} = \gensmoothfunction_{(\vec{\Lunit})} \cdot \pmb{\partial} \vec{\Psi}$
(see \eqref{E:SECONDFORMLIEDIFFERENTIATIONDEF},
\eqref{E:CONNECTIONCOEFFICIENT},
and \eqref{E:MODTRICHISMALL}).

\eqref{E:FIRSTCRUCIALPRODUCTESTIMATE}
follows from
\eqref{E:CONEFIRSTDERIVATIVEBOUNDSFORCONFORMALFACTOR}
and
\eqref{E:ANGDUPSIGMAL2ULINFINITYCONEESTIMATE}.

\eqref{E:SECONDCRUCIALPRODUCTESTIMATE}
follows from
\eqref{E:CONEFIRSTDERIVATIVEBOUNDSFORCONFORMALFACTOR},
\eqref{E:ANGDUPSIGMAL2TLINFINITYSPACEPARTESTIMATE},
and
\eqref{E:AIMPROVEDINTERIORREGEIONLT2LXINFINITYBONDS} for $\upzeta$.

\eqref{E:CRUCIALCHECKMUANDRELATEDTERMSESTIMATE}
follows from
\eqref{E:TRCHILINFINITYESTIMATES},
\eqref{E:SPACETIMEL2INULINFINTYINTLPINOMEGAFORRGEOTHREEHAVESWEIGHTEDMODIFIEDMASSASPECT},
\eqref{E:FLUIDONEDERIVATIVESPACETIME},
the estimate $\| \rgeo^{1/2} \|_{L^{\infty}(\widetilde{\mathcal{M}})} \lesssim \uplambda^{1/2 - 4 \upepsilon_0}$
guaranteed by \eqref{E:RESTATEMENTTANDRBOUNDS},
\eqref{E:SMOOTHFUNCTIONTIMESFLUIDONEDERIVATIVESIGMAT} for the second term on the LHS,
and the schematic relations 
$\Chfour_{\uLunit} = \gensmoothfunction_{(\vec{\Lunit})} \cdot \pmb{\partial} \vec{\Psi}$
and
$\mytr_{\congsphere} \widetilde{\upchi}
= \mytr_{\gsphere} \upchi
+
\gensmoothfunction_{(\vec{\Lunit})} \cdot \pmb{\partial} \vec{\Psi}
$.

\eqref{E:CRUCIALRGEOANGDUPSIGMAESTIMATE} follows from the bound \eqref{E:RESTATEMENTCONEFIRSTDERIVATIVEBOUNDSFORCONFORMALFACTOR} for
$\| \rgeo^{\frac{1}{2}} \angD \upsigma \|_{L_{\upomega}^p L_t^{\infty}(\widetilde{\mathcal{C}}_u)}$
and the estimate $\| \rgeo^{1/2} \|_{L^{\infty}(\widetilde{\mathcal{M}})} \lesssim \uplambda^{1/2 - 4 \upepsilon_0}$
guaranteed by \eqref{E:RESTATEMENTTANDRBOUNDS}.

To obtain \eqref{E:CRUCIALREORMALIZEDTRCHIESTIMATEAPPEARINGINCONFORMALENERGY},
we first use \eqref{E:ANNOYINGERRORTERMALGEBRAICEXPRESSION},
the estimate \eqref{E:RESTATEMENTNULLLAPSEANDVOLFORMBOUNDS} for $\volrat$,
and the aforementioned estimate
\begin{align} \notag
\| \rgeo^{1/2} \|_{L^{\infty}(\widetilde{\mathcal{M}})} \lesssim \uplambda^{1/2 - 4 \upepsilon_0}
\end{align}
to deduce that
\begin{align} \label{E:FIRSTSTEPCRUCIALREORMALIZEDTRCHIESTIMATEAPPEARINGINCONFORMALENERGY}
	\mbox{LHS~\eqref{E:CRUCIALREORMALIZEDTRCHIESTIMATEAPPEARINGINCONFORMALENERGY}}
	& \lesssim
		\|
			\rgeo^{1/2} (\pmb{\partial} \vec{\Psi},\vec{\VortVort},\DivGradEnt)
		\|_{L_t^{\infty} L_u^{\infty} L_{\upomega}^p(\widetilde{\mathcal{M}})}
			\\
	& \ \
		+
		\uplambda^{1/2 -  4 \upepsilon_0}
		\left\|
			\rgeo^{1/2} (\pmb{\partial} \vec{\Psi},\mytr_{\congsphere} \widetilde{\upchi}^{(Small)},\hat{\upchi},\angD \upsigma)
		\right\|_{L_u^{\infty} L_t^{\infty} L_{\upomega}^p(\widetilde{\mathcal{M}}^{(Int)})}^2.
		\notag
\end{align}
From the estimate \eqref{E:MODFLUIDSIGMAT},
the estimate \eqref{E:LAMBDATOMINUSONEHALFAANDUNDERLINEABOUNDS} for 
$
\|
		\rgeo^{1/2}
		\mathbf{A}
\|_{L_t^{\infty} L_{\upomega}^p(\widetilde{\mathcal{C}}_u)}
$,
and the estimate
\eqref{E:RESTATEMENTCONEFIRSTDERIVATIVEBOUNDSFORCONFORMALFACTOR}
for
$\| \rgeo^{\frac{1}{2}} \angD \upsigma \|_{L_{\upomega}^p L_t^{\infty}(\widetilde{\mathcal{C}}_u)}$,
we conclude 
that
$\mbox{RHS}~\eqref{E:FIRSTSTEPCRUCIALREORMALIZEDTRCHIESTIMATEAPPEARINGINCONFORMALENERGY} \lesssim \uplambda^{-1/2}$
as desired.

\end{proof}

\subsubsection{Mild growth estimate for the conformal energy}
\label{SS:CONFORMALENERGYGROWTHESTIMATE}
The main estimate needed to prove Prop.~\ref{P:SPATIALLYLOCALIZEDREDUCTIONOFPROOFOFTHEOREMDECAYESTIMATE}
is provided by the following theorem. The proof of the theorem is fundamentally based on the
estimates for the acoustic geometry provided by Cor.\,\ref{C:EIKONAFUNCTIONESTIMATESFORPROOFOFCONFORMALENERGYESTIMATE}. 

\begin{theorem}[Mild growth estimate for the conformal energy]
\label{T:BOUNDEDNESSOFCONFORMALENERGY}
Let $R > 0$ be as in Subsect.\,\ref{SS:RESCALEDSOLUTION}
and let $\Tranchar_{\bf{z}}(1)$ be the unique point $\Tranchar_{\bf{z}}(1)$ on the cone-tip axis in $\Sigma_1$ (see Subsubsect.\,\ref{SSS:EIKONALINTERIOR}).
Let $\varphi$ be any solution to the covariant linear wave equation
\eqref{E:BOUNDEDNESSOFCONFORMALENERGY} on the slab $[0,\RescaledTboot] \times \mathbb{R}^3$ such that 
$(\varphi|_{\Sigma_1},\partial_t \varphi|_{\Sigma_1})$ 
is supported in the Euclidean ball of radius $R$ centered at the point $\Tranchar_{\bf{z}}(1)$ in $\Sigma_1$
(and thus Remark~\ref{R:PHIVANISHESINM1INTCOMPLEMENT} applies).

Then under the assumptions of Subsect.\,\ref{SS:ASSUMPTIONS},
for any $\varepsilon > 0$, there exists a constant $C_{\varepsilon} > 0$ (which can blow up as $\varepsilon \downarrow 0$)
such that the conformal energy of $\varphi$ (which is defined in \eqref{E:CONFORMALENERGY})
satisfies the following estimate for $t \in [1,\RescaledTboot]$:
\begin{align} \label{E:CONFORMALENERGYBOUND}
	\mathscr{C}[\varphi](t)
	& \leq C_{\varepsilon} (1 + t)^{2 \varepsilon}
		\left\lbrace
			\| \pmb{\partial} \varphi \|_{L^2(\Sigma_1)}^2
			+
			\| \varphi \|_{L^2(\Sigma_1)}^2
		\right\rbrace.
\end{align}

\end{theorem}

\begin{proof}[Discussion of proof]
Given the estimates that we have already derived,
the proof of Theorem~\ref{T:BOUNDEDNESSOFCONFORMALENERGY}
is the same as the proof of \cite{qW2017}*{Theorem~4.5}
given in \cite{qW2017}*{Section~7.6}.
Thus, here we only clarify which estimates are needed to apply the preliminary arguments
given in \cite{qW2017}*{Section~7}, which are used in \cite{qW2017}*{Section~7.6} to prove
Theorem~\ref{T:BOUNDEDNESSOFCONFORMALENERGY}. 

The proof of \eqref{E:CONFORMALENERGYBOUND} given in \cite{qW2017}*{Section~7}
is carried out via a bootstrap argument, wherein one needs to establish
\cite{qW2017}*{Equations (7.61)--(7.63)} to close the bootstrap;
see \cite{qW2017}*{Section~7.4.1}.
For the reader's convenience, we first list the main steps given in \cite{qW2017}*{Section~7},
which lead to the proof of Theorem~\ref{T:BOUNDEDNESSOFCONFORMALENERGY}.
They are a collection estimates for the linear solution $\varphi$ in the statement of the theorem:
\begin{enumerate}
	\item The most basic ingredient in the proof is that one needs a uniform bound, in terms of the data, 
			for a standard non-weighted energy of $\varphi$ along a portion of
			the constant-time hypersurfaces $\Sigma_t$ and null cones $\widetilde{\mathcal{C}}_u$; see \cite{qW2017}*{Lemma~7.1}.
	\item A Morawetz-type energy estimate, which, when combined with Step 1,
			yields preliminary control 
			of a coercive spacetime integral 
			of $|\pmb{\partial} \varphi|^2$ and $\varphi^2$
			near the cone-tip axis. The integral involves
			weights with negative powers of $\rgeo$,
			and it is bounded by the data plus some 
			error terms that are controlled later in the argument.
	\item In this step, one makes preliminary progress in controlling the yet-to-be-controlled error terms
		mentioned above in Step 2. Specifically, one derives estimates showing that $\rgeo$-weighted versions 
		of $\varphi$ can be controlled in $L^2$ along a portion of $\Sigma_t$ 
		in terms of a weighted spacetime integral involving the square of its outgoing null derivative
		and an integral of $\varphi^2$ along a portion of a sound cone.
	\item Comparison results for various norms and energies, some of which
				involve the conformal metric $\widetilde{\gfour}$ from Subsubsect.\,\ref{SSS:CONFORMALMETRIC}
				and a corresponding conformally rescaled solution variable $\widetilde{\varphi} := e^{-\upsigma} \varphi$.
	\item Weighted energy estimates for the wave equation $\square_{\widetilde{\gfour}} \widetilde{\varphi} = \cdots$,
			where the energies control the $\Lunit$ and $\angD$ derivatives of $\widetilde{\varphi}$ along portions of $\Sigma_t$
			with weights involving $\volrat$ (see \eqref{E:RATIOOFVOLUMEFORMS}) and positive powers of $\rgeo$.
			These are obtained by multiplying the wave equation $\square_{\widetilde{\gfour}} \widetilde{\varphi} = \cdots$ by
			$(\Lunit \widetilde{\varphi} + \frac{1}{2} \mytr_{\congsphere} \widetilde{\upchi}) \rgeo^m$ for appropriate
			choices of $m \geq 0$, and integrating by parts. 
				Ultimately, when combined with the results from the previous steps,
				this allows one to bound the conformal energy (i.e., the terms on on RHS~\eqref{E:CONFORMALENERGY})
				in the region $\lbrace u \leq \frac{3t}{4} \rbrace \cap \widetilde{\mathcal{M}}^{(Int)}$;
				see \cite{qW2017}*{Section~7.6}, 
				in particular \cite{qW2017}*{Equation~(7.94)} and \cite{qW2017}*{Equation~(7.95)}.
	\item A decay estimate for the standard non-weighted energy along $\Sigma_t$, showing in particular that it decays like
				$(1+t)^{-2}$; see \cite{qW2017}*{Equation~(7.93)}.
				Ultimately, when combined with the preliminary estimates for $\varphi$ provided by Step 3, 
				this yields the desired control of the conformal energy (i.e., the terms on on RHS~\eqref{E:CONFORMALENERGY})
				in the region $\lbrace u \geq \frac{t}{2} \rbrace \cap \widetilde{\mathcal{M}}^{(Int)}$;
				see \cite{qW2017}*{Section~7.6}.
\end{enumerate}

We now discuss precisely which of the estimates we have already derived are needed to repeat the arguments of \cite{qW2017}*{Section~7} 
and to carry out the above steps. We will not fully describe all of the analysis in \cite{qW2017}*{Section~7}; 
rather, we will describe only the part of the analysis that relies on the estimates we have derived.
Specifically, we focus primarily on arguments that rely on estimates for the acoustic
geometry. We again emphasize that although we have derived the same 
estimates for the acoustic geometry as in \cite{qW2017},
our proof of the estimates (derived in Sect.\,\ref{S:ESTIMATESFOREIKONALFUNCTION})
required substantial additional arguments 
because we had to control new source terms coming from the entropy and vorticity.
The remaining arguments, not discussed here, needed
to close the bootstrap -- and hence establish
Theorem~\ref{T:BOUNDEDNESSOFCONFORMALENERGY} -- are the same as in 
\cite{qW2017}*{Section~7}, to which we refer the reader for more details.
We start by noting that the basic estimates
\eqref{E:RESTATEMENTTANDRBOUNDS},
\eqref{E:RESTATEMENTNULLLAPSEANDVOLFORMBOUNDS},
and \eqref{E:RESTATEMENTLINFTYBOUNDSFORCONFORMALFACTOR}
are used throughout \cite{qW2017}*{Section~7}. 
We also refer readers to Footnote~\ref{FN:SMALLCORRECTIONTOLAMBDAPOWER}
regarding a minor discrepancy between the estimates we derived here
and corresponding estimates in \cite{qW2017};
we will not comment further on these issues.

Step 1 (see \cite{qW2017}*{Lemma~7.1})
is essentially equivalent to the basic energy estimates for the wave
equations derived in the proofs of 
Props.\,\ref{P:PRELIMINARYENERGYANDELLIPTICESTIMATES} and \ref{P:ENERGYESTIMATESALONGNULLHYPERSURFACES},
differing only in that the needed estimates are spatially localized.
For the proof,
one needs only the bound
\begin{align} \label{E:SIMPLEL1TLINFTYXDEFORMATIONTENSORCOMPONENTBOUND}
\|
	\deformarg{\Transport}{\alpha}{\beta}
\|_{L_t^1 L_x^{\infty}}
& \lesssim
\uplambda^{-8 \upepsilon_0},
\end{align}
where $\deformarg{\Transport}{\alpha}{\beta}$ are the Cartesian components of the deformation tensor of
$\Transport$. Recalling that each Cartesian component $\Transport^{\alpha}$
satisfies $\Transport^{\alpha} = \gensmoothfunction(\vec{\Psi})$ 
for some smooth function $\gensmoothfunction$ (where $\vec{\Psi}$ is the rescaled solution),
we see that the Cartesian components $\deformarg{\Transport}{\alpha}{\beta}$
satisfy
$\deformarg{\Transport}{\alpha}{\beta} = \gensmoothfunction(\vec{\Psi}) \cdot \pmb{\partial} \Psi$
(for some other smooth function $\gensmoothfunction$). Hence,
the desired bound \eqref{E:SIMPLEL1TLINFTYXDEFORMATIONTENSORCOMPONENTBOUND} follows from \eqref{E:RESCALEDBOOTBOUNDS},
\eqref{E:RESCALEDSOLUTIONHOLDERESTIMATE}, 
and H\"{o}lder's inequality.

The Morawetz estimate from Step 2 is provided in \cite{qW2017}*{Lemma~7.4} and \cite{qW2017}*{Lemma~7.5}.
The proof relies on applying the divergence theorem (the geometric version, with respect to the rescaled metric $\gfour$)
on an appropriate spacetime region 
to the vectorfield
$
\Jenarg{\mathbf{X}}{\alpha}[\varphi]
	:= \enmomem^{\alpha \beta}[\varphi] 
		\mathbf{X}_{\beta}
		-
		\frac{1}{2}
		\left\lbrace
			(\gfour^{-1})^{\alpha \beta} \partial_{\beta} \Theta
		\right\rbrace
		\varphi^2
		+
		\frac{1}{2}
		\Theta
		(\gfour^{-1})^{\alpha \beta} \partial_{\beta} 
		(\varphi^2)
$,
where $\enmomem^{\alpha \beta}[\varphi]$ is defined in \eqref{E:ENMOMENTUMTENSOR},
$\mathbf{X} := f \spherenormal$,
$\spherenormal$ is the outward $g$-unit normal to $S_{t,u}$ in $\Sigma_t$ (see \eqref{E:SPHEREOUTERNORMAL}),
$f := \upepsilon_0^{-1} - \frac{\upepsilon_0^{-1}}{(1 + \rgeo)^{2 \upepsilon_0}}$,
and $\Theta := \rgeo^{-1} f$. The error terms involve various geometric derivatives of $\spherenormal$
that can be expressed in terms of connection coefficients of the null frame and their first derivatives.
For the proof of \cite{qW2017}*{Lemma~7.4} and \cite{qW2017}*{Lemma~7.5} to go through verbatim,
one needs only the estimates 
\eqref{E:AIMPROVEDINTERIORREGEIONLT2LXINFINITYBONDS}
and 
\eqref{E:LAMBDATOMINUSONEHALFAANDUNDERLINEABOUNDS};
see just below \cite{qW2017}*{Equation~(7.18)}.

In obtaining estimates for $\rgeo$-weighted versions $\varphi^2$ in Step 3,
in the sub-step provided by \cite{qW2017}*{Lemma~7.6},
one needs the estimate \eqref{E:AIMPROVEDINTERIORREGEIONLT2LXINFINITYBONDS};
see below \cite{qW2017}*{Equation~(7.34)}.

For the comparison results from Step 4, 
in the sub-step provided by \cite{qW2017}*{Proposition~7.10},
one needs the estimates
\eqref{E:LAMBDATOMINUSONEHALFAANDUNDERLINEABOUNDS}
and
\eqref{E:RESTATEMENTCONEFIRSTDERIVATIVEBOUNDSFORCONFORMALFACTOR};
see below 
\cite{qW2017}*{Equation~(7.43)}
and \cite{qW2017}*{Equation~(7.45)}.
In the sub-step provided by \cite{qW2017}*{Lemma~7.11},
one needs the estimates
\eqref{E:LAMBDATOMINUSONEHALFAANDUNDERLINEABOUNDS}
and \eqref{E:RESTATEMENTCONEFIRSTDERIVATIVEBOUNDSFORCONFORMALFACTOR}.

In deriving the weighted energy estimate from Step 5,
in the sub-step provided by \cite{qW2017}*{Lemma~7.15}, one needs the estimate \eqref{E:AIMPROVEDINTERIORREGEIONLT2LXINFINITYBONDS};
see the first line of the proof. Then, in the same proof,
to bound the error terms denoted on \cite{qW2017}*{page~87} by ``$\mathcal{A}_i$'', $(i=1,2,3)$, 
one needs, respectively, the estimates
\eqref{E:FIRSTCRUCIALPRODUCTESTIMATE},
\eqref{E:SECONDCRUCIALPRODUCTESTIMATE},
and \eqref{E:CRUCIALCHECKMUANDRELATEDTERMSESTIMATE};
see the analysis just below \cite{qW2017}*{Equation~(7.72)}.

For the energy-decay estimate provided by Step 6,
in the sub-step provided by \cite{qW2017}*{Proposition~7.22},
the estimates \eqref{E:CRUCIALRGEOANGDUPSIGMAESTIMATE}--\eqref{E:CRUCIALREORMALIZEDTRCHIESTIMATEAPPEARINGINCONFORMALENERGY}
are ingredients needed to help bound the term denoted by
``$\mathcal{I}$'' on \cite{qW2017}*{page~94};
see \cite{qW2017}*{page~95} for the role that \eqref{E:CRUCIALRGEOANGDUPSIGMAESTIMATE}--\eqref{E:CRUCIALREORMALIZEDTRCHIESTIMATEAPPEARINGINCONFORMALENERGY} play.
One also needs 
\eqref{E:RESTATEMENTCONEFIRSTDERIVATIVEBOUNDSFORCONFORMALFACTOR}
(see the bottom of \cite{qW2017}*{page~95})
and \eqref{E:AIMPROVEDINTERIORREGEIONLT2LXINFINITYBONDS} 
(see the top of \cite{qW2017}*{page~96}).

\end{proof}

\subsection{Discussion of the proof of Proposition~\ref{P:SPATIALLYLOCALIZEDREDUCTIONOFPROOFOFTHEOREMDECAYESTIMATE}}
\label{SS:DISCUSSIONOFPROOFOFPROPSPATIALLYLOCALIZEDREDUCTIONOFPROOFOFTHEOREMDECAYESTIMATE}
Thanks to the assumptions of Subsect.\,\ref{SS:ASSUMPTIONS}
and the estimates for the acoustic geometry that we obtained in 
\eqref{E:CONNECTIONCOEFFICIENTESTIMATESNEEDEDTODERIVESPATIALLYLOCALIZEDDECAYFROMCONFORMALENERGYESTIMATE},
Prop.\,\ref{P:SPATIALLYLOCALIZEDREDUCTIONOFPROOFOFTHEOREMDECAYESTIMATE}
follows as a consequence of Theorem~\ref{T:BOUNDEDNESSOFCONFORMALENERGY}
and the same arguments given in \cite{qW2017}*{Section~4.1}
(see in particular \cite{qW2017}*{Proposition~4.1})
and Lemma~\ref{L:STANDARDENERGYESTIMATEFORLINEARWAVEQUATIONWITHRESCALEDBACKGROUND}.

\section*{Acknowledgments}
We are grateful to Qian Wang and the anonymous referees for offering enlightening comments and insights,
and for suggestions that have helped improve the exposition.

\bibliographystyle{amsalpha}
\bibliography{JBib} 

\end{document}